\documentclass{article} 

\usepackage[utf8]{inputenc} 

\usepackage[margin=1in]{geometry} 
\usepackage{graphicx} 

\usepackage{booktabs} 
\usepackage{array} 
\usepackage{paralist} 
\usepackage{verbatim} 
\usepackage{subfig} 
\usepackage{mathrsfs}
\usepackage{amssymb}
\usepackage{amsthm}
\usepackage{amsmath,amsfonts,amssymb,esint}
\usepackage{graphics}
\usepackage{enumerate}
\usepackage{mathtools}
\usepackage{xfrac}
\usepackage{fancyhdr} 
\usepackage[nottoc,notlof,notlot]{tocbibind} 
\usepackage[titles,subfigure]{tocloft} 

\numberwithin{equation}{section}

\def\Xint#1{\mathchoice
{\XXint\displaystyle\textstyle{#1}}%
{\XXint\textstyle\scriptstyle{#1}}%
{\XXint\scriptstyle\scriptscriptstyle{#1}}%
{\XXint\scriptscriptstyle\scriptscriptstyle{#1}}%
\!\int}
\def\XXint#1#2#3{{\setbox0=\hbox{$#1{#2#3}{\int}$ }
\vcenter{\hbox{$#2#3$ }}\kern-.6\wd0}}

\def\dashint{\Xint-}

\newtheorem{theorem}{Theorem}[section]
\newtheorem{corollary}[theorem]{Corollary}
\newtheorem{proposition}[theorem]{Proposition}
\newtheorem{lemma}[theorem]{Lemma}
\newtheorem{definition}[theorem]{Definition}
\theoremstyle{definition}
\newtheorem{remark}[theorem]{Remark}

\newcommand{\norm}[1]{\left\|#1\right\|}
\newcommand{\abs}[1]{\left|#1\right|}

\newcommand*{\supp}{\ensuremath{\mathrm{supp\,}}}
\newcommand*{\dist}{\ensuremath{\mathrm{dist\,}}}
\newcommand*{\Id}{\ensuremath{\mathrm{Id}}}
\renewcommand*{\div}{\ensuremath{\mathrm{div\,}}}

\newcommand*{\N}{\ensuremath{\mathbb{N}}}
\newcommand*{\T}{\ensuremath{\mathbb{T}}}
\newcommand*{\Tthreexi}{\ensuremath{\mathbb{T}_\xi^3}}
\newcommand*{\Z}{\ensuremath{\mathbb{Z}}}

\newcommand*{\R}{\ensuremath{\mathbb{R}}}
\newcommand*{\RSZ}{\ensuremath{\mathcal{R}}}

\newcommand*{\tr}{\ensuremath{\mathrm{tr\,}}}
\newcommand{\eps}{\varepsilon}

\newcommand{\PP}{\mathcal P}
\newcommand{\RR}{\mathring R}
\newcommand{\HH}{\mathring H}
\newcommand{\nn}{{\tilde{n}}}
\newcommand{\dpot}{{\mathsf d}}

\renewcommand*{\tilde}{\widetilde}

\newcommand*{\curl}{\ensuremath{\mathrm{curl\,}}}
\newcommand*{\ad}{\ensuremath{\mathrm{ad\,}}}

\newcommand{\cstar}{ \mathsf{c_0} }
\newcommand{\cstarn}{ \mathsf{c}_{\textnormal{n}} }
\newcommand{\cstarnprime}{ \mathsf{c}_{\textnormal{n}'} }
\newcommand{\cstarnn}{ \mathsf{c}_{\tilde{\textnormal{n}}} }
\newcommand{\cstarnmax}{ \mathsf{c}_{\textnormal{n}_{\textnormal{max}}} }

\newcommand{\shaq}{{ -\mathsf{C_R}} }

\newcommand{\shaqqplusone}{ \Gamma_{q+1}^{-\mathsf{C_R}} }

\newcommand{\les}{\lesssim}
\newcommand{\imax}{{i_{\rm max}}}
\newcommand{\jmax}{{j_{\rm max}}}
\newcommand{\nmax}{{n_{\rm max}}}
\newcommand{\Nmax}{{N_{\rm max}}}
\newcommand{\Npipe}{\mathsf{N}_{\rm pipe}}
\newcommand{\twopi}{2\pi}
\newcommand{\pmax}{{p_{\rm max}}}
\newcommand{\Ncut}{\mathsf{N}_{\rm cut}}

\newcommand{\NcutSmall}{\mathsf{N}_{\rm cut,t}}
\newcommand{\NcutLarge}{\mathsf{N}_{\rm cut,x}}
\newcommand{\NindSmall}{\mathsf{N}_{\textnormal{ind,t}}}
\newcommand{\NindLarge}{\mathsf{N}_{\textnormal{ind,v}}}
\newcommand{\NindRt}{\mathsf{N}_{\textnormal{ind,t}}}

\newcommand{\Nindvt}{\mathsf{N}_{\textnormal{ind,t}}}
\newcommand{\Nindt}{\mathsf{N}_{\textnormal{ind,t}}}

\newcommand{\Nindv}{\mathsf{N}_{\textnormal{ind,v}}}
\newcommand{\Nfin}{\mathsf{N}_{\rm fin}}
\newcommand{\Nfn}{\mathsf{N}_{\rm fin, n}}
\newcommand{\Nfnn}{\mathsf{N}_{\rm{fin,}\nn}}

\newcommand{\NN}[1]{\mathsf{N}_{#1}}
\newcommand{\CLebesgue}{{\mathsf{C}_{b}}}

\newcommand{\Ndec}{{\mathsf{N}_{\rm dec}}}
\newcommand{\WW}{\ensuremath{\mathbb{W}}}

\newcommand{\UU}{\ensuremath{\mathbb{U}}}
\newcommand{\Proj}{\ensuremath{\mathbb{P}}}
\newcommand{\MM}[1]{\ensuremath{\mathcal{M}}\left(#1\right)}
\newcommand{\rqnperp}{r_{q+1,n}}
\newcommand{\rqnperptilde}{r_{q+1,\nn}}

\newcommand{\vlq}{v_{\ell_q}}
\newcommand{\vlqprime}{v_{\ell_{q'}}}
\newcommand{\vlqminus}{v_{\ell_{q-1}}}

\newcommand{\Pqx}{\mathcal{P}_{q,x}}

\newcommand{\Pqt}{\mathcal{P}_{q,t}}
\newcommand{\Pqxt}{\mathcal{P}_{q,x,t}}
\newcommand{\divH}{\mathcal{H}}
\newcommand{\divR}{\mathcal{R}^*}
\newcommand{\Dtq}{D_{t,q}}

\newcommand{\istar}{{i^*}}
\newcommand{\jstar}{{j^*}}
\newcommand{\kstar}{{k^*}}
\newcommand{\xistar}{{\xi^*}}
\newcommand{\pstar}{{p^*}}
\newcommand{\nstar}{{n^*}}

\newcommand{\lstar}{l^*}
\newcommand{\wstar}{w^*}
\newcommand{\hstar}{h^*}

\newcommand{\LPqnp}{\mathbb{P}_{[q,n,p]}}

\newcommand{\Phiik}{\Phi_{(i,k)}}
\newcommand{\ijk}{{(i,j,k)}}
\newcommand{\ijkstar}{{(i^*,j^*,k^*)}}
\newcommand{\ijkstarpp}{{(i^*,j^*,k^*,\pp)}}

\newcommand{\ijkpp}{{(i,j,k,\pp)}}
\newcommand{\ijkpstar}{{(i^*,j^*,k^*,p^*)}}
\newcommand{\minus}{{(-)}}
\newcommand{\sgn}{\text{sgn}}
\newcommand{\pp}{{\tilde{p}}}

\newcommand{\const}{\mathcal{C}}

\newcommand{\qn}{_{q,n}}
\newcommand{\qnp}{_{q,n,p}}
\newcommand{\qnpminus}{_{q,n,p-1}}

\newcommand{\qplusnp}{_{q+1,n,p}}
\newcommand{\qnnpp}{_{q,\nn,\pp}}
\newcommand{\qplusnnpp}{_{q+1,\nn,\pp}}
\newcommand{\ff}{\left(\frac{4}{5}\right)}
\newcommand{\symring}{\, \mathring \otimes_{\rm s}\,}

\newcommand{\qnn}{_{q,\nn}}
\newcommand{\qplusnn}{_{q+1,\nn}}

\newcommand{\qnnone}{_{q,\nn,1}}

\newcommand{\lessg}{\lesssim}

\newcommand{\LPqnpmax}{\mathbb{P}_{\left[q,\nmax,\pmax+1\right]}}
\newcommand{\Psum}{\sum_{n=\nn+1}^{\nmax}\sum_{p=1}^{\pmax} \LPqnp + \LPqnpmax}
\newcommand{\Psumchop}{\sum_{n=\nn+1}^{\nmax}\sum_{p=1}^{\pmax} \LPqnp}

\newcommand{\Nsharp}{N^{\sharp}}
\newcommand{\Onpnp}{\mathcal{O}_{\nn,\pp,n,p}}
\newcommand{\Pqnn}{\mathbb{P}_{\geq\lambda_{q,\nn}}}

\newcommand{\ijklwhstarzero}{{(\istar,\jstar,\kstar,0,\lstar,\wstar,\hstar)}}

\newcommand{\ijklwhzero}{{(i,j,k,0,l,w,h)}}
\newcommand{\Phiikstar}{\Phi_{(\istar,\kstar)}}
\newcommand{\vecl}{\vec{l}}
\newcommand{\veclstar}{\vecl^*}

\usepackage[usenames,dvipsnames]{color}
\usepackage[colorlinks=true, pdfstartview=FitV, linkcolor=blue, citecolor=blue, urlcolor=blue,pdfencoding=auto]{hyperref}

\title{\bf Non-conservative $H^{\sfrac 12-}$ weak solutions of the incompressible 3D Euler equations}

\author{
{   {\bf Tristan Buckmaster}}\thanks{\footnotesize Department of Mathematics, 
Princeton University, Princeton, NJ 08544,
 \href{tjb4@math.princeton.edu}{tjb4@math.princeton.edu}.}
 \and 
 {\!\!  {\bf Nader Masmoudi}}\thanks{\footnotesize NYUAD Research Institute, New York University Abu Dhabi, PO Box 129188, Abu Dhabi, UAE \& Courant Institute of Mathematical Sciences, New York University, New York, NY 10012, \href{masmoudi@cims.nyu.edu}{masmoudi@cims.nyu.edu}.}
 \and 
{\!\!   {\bf Matthew Novack}}\thanks{\footnotesize Courant Institute of Mathematical Sciences, New York University, New York, NY 10012,  \href{mdn7@cims.nyu.edu}{mdn7@cims.nyu.edu}.}
\and 
{\!\!  {\bf Vlad Vicol}}\thanks{\footnotesize Courant Institute of Mathematical Sciences, New York University, New York, NY 10012, \href{vicol@cims.nyu.edu}{vicol@cims.nyu.edu}.}
}
 
\begin{document}

\maketitle

\begin{abstract}
For any positive regularity parameter $\beta < \sfrac 12$, we construct non-conservative weak solutions of the 3D incompressible Euler equations which lie in $H^{\beta}$ uniformly in time. In particular, we construct solutions which have an $L^2$-based regularity index \emph{strictly larger} than $\sfrac 13$, thus deviating from the $H^{\sfrac{1}{3}}$-regularity corresponding to  the Kolmogorov-Obhukov $\sfrac 53$ power spectrum in the inertial range.
\end{abstract}

\setcounter{tocdepth}{2}
\tableofcontents

\allowdisplaybreaks

\section{Introduction}
\label{sec:intro}

We consider the homogeneous incompressible Euler equations
\begin{subequations}
\label{e:eulereq}
\begin{align}
\partial_t v + \div(v \otimes v) +\nabla p &=0\\
\div v &= 0
\end{align}
\end{subequations}
for the unknown velocity vector field $v$ and scalar pressure field $p$, posed on the the three dimensional box $\T^3=[-\pi,\pi]^3$ with periodic boundary conditions. We consider weak solutions of \eqref{e:eulereq}, which may be defined in the usual way for $v\in L^2_t L^2_{x}$. 

We show that within the class of weak  solutions of regularity $C^0_t H^{\sfrac 12-}_x$, the 3D Euler system \eqref{e:eulereq} is {\em flexible}.\footnote{Loosely speaking, we consider a system of partial differential equations of physical origin to be {\em flexible} in a certain regularity class, if at this regularity level the PDEs are not anymore predictive: there exist infinitely many solutions, which behave in a non-physical way, in stark contrast to the behavior of the PDE in the smooth category. We refer the interested reader to the discussion in the surveys of De Lellis and Sz\'ekelyhidi Jr.~\cite{DLSZ12,DLSZ17} which draw the analogy with the flexibility in Gromov's $h$-principle~\cite{Gromov86}.} An example of this flexibility is provided by:

\begin{theorem}[\textbf{Main result}]
\label{thm:main}
Fix $\beta\in (0,\sfrac{1}{2})$. For any divergence-free $v_{\mathrm{start}}, v_{\mathrm{end}} \in L^2(\T^3)$ which have the same mean, any $T>0$, and any $\epsilon >0$, there exists a weak solution $v\in C([0,T];H^\beta(\mathbb{T}^3))$ to the 3D Euler equations~\eqref{e:eulereq} such that $\norm{v(\cdot,0) - v_{\mathrm{start}}}_{L^2(\T^3)} \leq \epsilon$ and $\norm{v(\cdot,T)-v_{\mathrm{end}}}_{L^2(\T^3)} \leq \epsilon$.
\end{theorem}

Since the vector field $v_{\mathrm{end}}$ may be chosen to have a much higher (or much lower) kinetic energy than the vector field $v_{\mathrm{start}}$, the above result shows the existence of infinitely many {\em non-conservative} weak solutions of 3D Euler in the regularity class $C^0_t H^{\sfrac 12-}_x$. Theorem~\ref{thm:main} further shows that the set of  so-called {\em wild initial data} is dense in the space of $L^2$ periodic functions of given mean. The novelty of this result is that these weak solutions have {\em more than $\sfrac 13$ regularity}, when measured on a $L^2_x$-based Banach scale.

\begin{remark}[\textbf{Corollaries of the proof}]
\label{rem:flexibility}
We have chosen to state the flexibility of the 3D Euler equations as in Theorem~\ref{thm:main} because it is a simple way to exhibit weak solutions which are non-conservative, leaving the entire emphasis of the proof on the {\em regularity class} in which the weak solutions lie. Using by now standard approaches encountered in convex integration constructions for the Euler equations, we may alternatively establish the following variants of flexibility for \eqref{e:eulereq} within the class of $C^0_t H^{\sfrac 12-}_x$ weak solutions:

\begin{enumerate}[(a)]
\item The proof of Theorem~\ref{thm:main} also shows that: {\em given any $\beta < \sfrac 12$, $T>0$, and $E>0$, there exists a weak solution $v\in C(\R,H^{\beta}(T^3))$ of the 3D Euler equations such that: $\supp_t v \subset [-T,T]$, and $\norm{v(\cdot,0)}_{L^2} \geq E$}. Such weak solutions are nontrivial and have compact support in time, thereby implying the {\em non-uniqueness} of weak solutions to \eqref{e:eulereq} in the regularity class $C^0_t H^{\sfrac 12-}_x$. The argument is sketched in Remark~\ref{rem:FU:P} below. 

\item The proof of Theorem~\ref{thm:main} may be modified to show that: {\em given any $\beta \in (0,\sfrac 12)$, and any $C^\infty$ smooth  function $e \colon [0,T] \to (0,\infty)$, there exists a weak solution $v\in C^0([0,T];H^{\beta}(\T^3))$ of the 3D Euler equations, such that $v(\cdot,t)$ has kinetic energy $e(t)$, for all $t\in [0,T]$.} In particular, the flexibility of 3D Euler in $C^0_t H^{\sfrac 12-}_x$ may be shown to also hold within the class of {\em dissipative} weak solutions, by choosing $e$ to be a non-increasing function of time. This is further discussed in Remark~\ref{rem:FU:L} below.
\end{enumerate}
\end{remark}

\subsection{Context and motivation}
\label{sec:intro:prev}

Classical solutions of the Cauchy problem for the 3D Euler equations \eqref{e:eulereq} are known to exist, locally in time, for initial velocities which lie in $C^{1,\alpha}$ for some $\alpha>0$ (see e.g.\ Lichtenstein~\cite{Lichtenstein25}). These solutions are unique, and they conserve (in time) the kinetic energy ${\mathcal E}(t) = \frac 12 \int_{\T^3} |v(x,t)|^2 dx$, giving two manifestations of {\em rigidity} of the Euler equations within the class of smooth solutions.  

Motivated by hydrodynamic turbulence, it is natural to consider a much broader class of solutions to the 3D Euler system; these are the {\em distributional} or {\em weak} solutions of \eqref{e:eulereq}, which may be defined in the natural way as soon as $v\in L^2_t L^2_x$, since \eqref{e:eulereq} is in divergence form. Indeed, one of the fundamental assumptions of Kolmogorov's '41 theory of turbulence~\cite{Kolmogorov41a} is that in the infinite Reynolds number limit, turbulent solutions of the 3D Navier-Stokes equations exhibit anomalous dissipation of kinetic energy; by now, this is considered to be an experimental fact, see e.g.\ the book of Frisch~\cite{Frisch95} for a detailed account. In particular, this anomalous dissipation of energy necessitates that the family of Navier-Stokes solutions does not remain uniformly bounded in the topology of $L^3_t B^{\alpha}_{3,\infty,x}$ for any $\alpha > \sfrac 13$, as the Reynolds number diverges, as was alluded to in the work of Onsager~\cite{Onsager49}.\footnote{Onsager did not use the Besov norm $\norm{v}_{B^{\alpha}_{p,\infty}} = \norm{v}_{L^p} + \sup_{|z|>0} |z|^{-\alpha} \norm{v(\cdot+z)-v(\cdot)}_{L^p}$; here we use this modern notation and the sharp version of this conclusion, cf.~Constantin, E, and Titi~\cite{ConstantinETiti94}, Duchon and Robert~\cite{DuchonRobert00}, Drivas and Eyink~\cite{DrivasEyink19}.} Thus, in the infinite Reynolds number limit for turbulent solutions of 3D Navier-Stokes, one expects the convergence to \emph{weak} solutions of 3D Euler, not classical ones. 

It turns out that even in the context of weak solutions, the 3D Euler equations enjoy some conditional variants of rigidity. An example is the classical {\em weak-strong} uniqueness property.\footnote{If $v$ is a strong solution of the Cauchy problem for \eqref{e:eulereq} with initial datum $v_0\in L^2$, and $w \in L^\infty_t L^2_x$ is merely a weak solution of the Cauchy problem for \eqref{e:eulereq}, which has the additional property that it its kinetic energy ${\mathcal E}(t)$ is less than the kinetic energy of $v_0$, for a.e. $t>0$, then in fact $v \equiv w$. See e.g.\ the review~\cite{Wiedemann17} for a detailed account.}
Another example is the question of whether weak solutions of the 3D Euler equation conserve kinetic energy. This  is the subject of the Onsager conjecture~\cite{Onsager49}, one of the most celebrated connections between phenomenological theories in turbulence and the rigorous mathematical analysis of the PDEs of  fluid dynamics. For a detailed account we refer the reader to the reviews~\cite{EyinkSreeniviasan06,Constantin07,Shvydkoy10,DLSZ12,Szekelyhidi12,DLSZ17,DLSZ19,BV_EMS19,BV20} and mention here only a few of the results in the Onsager program for 3D Euler. 

Constantin, E, and Titi~\cite{ConstantinETiti94} established the rigid side of the Onsager conjecture, which states that if a weak solution $v$ of \eqref{e:eulereq} lies in $L^3_t B^{\beta}_{3,\infty,x}$ for some $\beta > \sfrac 13$, then $v$  conserves its kinetic energy. The endpoint case $\beta = \sfrac 13$ was addressed by Cheskidov, Constantin, Friedlander, and Shvydkoy~\cite{CCFS08}, who established a criterion which is known to be sharp in the context of 1D Burgers. By using the Bernstein inequality to transfer information from $L^2_x$ into $L^3_x$ , the authors of~\cite{CCFS08} also prove energy-rigidity for weak solutions based on a regularity condition for an $L^2_x$ based scale: if $v \in L^3_t H^\beta_x$ with $\beta > \sfrac 56$, then $v$ conserves kinetic energy (see also the work of Sulem and Frisch~\cite{SulemFrisch75}). We emphasize the discrepancy between the energy-rigidity threshold exponents $\sfrac 56$ for the $L^2$-based Sobolev scale, and $\sfrac 13$ for $L^p$-based regularity scales with $p\geq 3$. 

The first flexibility results were obtained by Scheffer~\cite{Scheffer93}, who constructed non-trivial weak solutions of the 2D Euler system, which lie in $L^2_t L^2_x$ and have compact support in space and time.
The existence of infinitely many dissipative weak solutions to the Euler equations was first proven by Shnirelman in~\cite{Shnirelman00}, in the regularity class $L^\infty_t L^2_x$.  Inspired by the work~\cite{MullerSverak03} of M\"uller and \v{S}verak for Lipschitz differential inclusions,  in~\cite{DeLellisSzekelyhidi09} De Lellis and Sz\'ekelyhidi Jr. have  constructed infinitely many dissipative weak solutions of \eqref{e:eulereq} in the regularity class $L^\infty_t L^\infty_x$ and have developed a systematic program towards the resolution of the flexible of the Onsager conjecture, using the technique of {\em convex integration}. Inspired by Nash's paradoxical constructions for the isometric embedding problem~\cite{Nash54}, the first proof of flexibility of the 3D Euler system in a H\"older space was given by 
De Lellis and Sz\'ekelyhidi Jr.\ in the  work~\cite{DeLellisSzekelyhidi13}. This breakthrough or crossing of the $L^\infty_x$ to $C^0_x$ barrier in convex integration for 3D Euler~\cite{DeLellisSzekelyhidi13} has in turn spurred a number of results~\cite{BDLISZ15,Buckmaster15,BDLSZ16,DaneriSzekelyhidi17} which have used finer properties of the Euler equations to increase the regularity of the wild weak solutions being constructed. The flexible part of the Onsager conjecture was finally resolved by Isett~\cite{Isett2018,Isett17} in the context of weak solutions with compact support in time (see also the subsequent work by the first and last authors with De Lellis and Sz\'ekelyhidi Jr.~\cite{BDLSV17} for dissipative weak solutions), by showing that for any regularity parameter $\beta < \sfrac 13$, the 3D Euler system~\eqref{e:eulereq} is flexible in the class of $C^\beta_{t,x}$ weak solutions. We refer the reader to the review papers~\cite{DLSZ12,Szekelyhidi12,DLSZ17,DLSZ19,BV_EMS19,BV20} for more details concerning convex integration constructions in fluid dynamics, and for open problems in this area.

Since the aforementioned convex integration constructions are spatially homogenous, they yield  weak solutions whose H\"older regularity index cannot be taken to be larger than $\sfrac 13$ (recall that weak solutions in $L^3_t C^\beta_{x}$ with $\beta>\sfrac 13$ must conserve kinetic energy). However, {\em the  exponent $\sfrac 13$ is not expected to be a sharp threshold for energy-rigidity/flexibility if the weak solutions' regularity is measured on an $L^p_x$-based Banach scale with $p<3$}. This expectation stems from the measured intermittent nature of turbulent flows, see e.g.\ Frisch~\cite[Figure 8.8, page 132]{Frisch95}. In broad terms, intermittency is characterized as a  deviation from the Kolmogorov '41 scaling laws, which were derived under the assumptions of homogeneity and isotropy (for a rigorous way to measure this deviation, see~Cheskidov and Shvydkoy~\cite{CS2014}).  A common signature of intermittency is that for $p\neq 3$, the $p^{th}$ order structure function\footnote{In analogy with $L^p$-based Besov spaces, absolute $p^{th}$ order structure functions are typically defined as $S_p(\ell) = \fint_{0}^T \fint_{\T^3} \fint_{{\mathbb S}^2} |v(x+\ell z,t) - v(x,t)|^p dz dx dt$. The structure function exponents in Kolmogorov's '41 theory are then given by  $\zeta_p = \limsup_{\ell \to 0^+} \frac{\log S_p(\ell)}{\log (\epsilon \ell)}$, where $\epsilon > 0$ is the postulated anomalous dissipation rate in the infinite Reynolds number limit. Of course, for any non-conservative weak solution we may define a positive number $\epsilon = \fint_0^T |\frac{d}{dt} {\mathcal E}(t)| dt$ as a substitute for Kolmogorov's $\epsilon$, which allows one to define $\zeta_p$ accordingly.} 
exponents $\zeta_p$ deviate from the Kolmogorov-predicted values of $\sfrac{p}{3}$. We note that the regularity statement $v \in C^0_t B^{s}_{p,\infty}$ corresponds to a structure function exponent $\zeta_p = sp$; that is, Kolmogorov '41 predicts that $s=\sfrac{1}{3}$ for all $p$. The exponent $p=2$ plays a special role, as it allows one to measure the intermittent nature of turbulent flows on the Fourier side as a power-law decay of the energy spectrum. Throughout the last five decades, the experimentally measured values of $\zeta_2$ (in the inertial range, for viscous flows at very high Reynolds numbers) have been consistently observed to {\em exceed} the Kolmogorov-predicted value of $\sfrac 23$~\cite{AnselmetEtAl84,MeneveauSreenivasan91,SreenivasanEtAl96,KanedaEtAl03,ChenEtAl05,IshiharaEtAl09,NguyenEtAl19}, thus showing a steeper decay rate in the inertial range power spectrum than the one predicted by the Kolmogorov-Obhukov $5/3$ law.
Moreover, in the mathematical literature, Constantin and Fefferman~\cite{ConstantinFefferman94} and Constantin, Nie, and Tanveer~\cite{ConstantinNieTanveer99} have used the 3D Navier-Stokes equations to show that the Kolmogorov '41 prediction $\zeta_2 = \sfrac 23$ is only consistent with a  lower bound for $\zeta_2$, instead of an exact equality.

Prior to this work, it was not known whether the 3D Euler equation can sustain weak solutions which have kinetic energy that is uniformly bounded in time but not conserved, and which have spatial regularity equal to or exceeding $H^{\sfrac 13}_x$, corresponding to $\zeta_2 \geq \sfrac 23$; see~\cite[Open Problem~5]{BV_EMS19} and~\cite[Conjecture 2.6]{BV20}. Theorem~\ref{thm:main} gives the first such existence result. The solutions in Theorem~\ref{thm:main} may be constructed to have second order structure function exponent $\zeta_2$ an arbitrary number in $(0,1)$, showing that \eqref{e:eulereq} exhibits weak solutions which severely deviate from the Kolmogorov-Obhukov $5/3$ power spectrum. 

We note that in a recent work~\cite{cheskidov2020sharp}, Cheskidov and Luo established the sharpness of the $L^2_t L^\infty_x$ endpoint of the Prodi-Serrin criteria for the 3D Navier-Stokes equations, by constructing non-unique weak (mild) solutions of these equations in $L^p_t L^\infty_x$, for any $p<2$.\footnote{See also~\cite{cheskidov2021} for a proof that the space $C^0_t L^p_x$ is critical for uniqueness at $p=2$, in two space dimensions.} As noted in~\cite[Theorem 1.10]{cheskidov2020sharp}, their approach also applies to the 3D Euler equations, yielding weak solutions that lie in $L^1_t C^\beta_x$ for any $\beta < 1$, and thus these weak solutions also have more than $\sfrac 13$ regularity. The drawback is that the solutions constructed in~\cite{cheskidov2020sharp} do not have bounded (in time) kinetic energy, in contrast to Theorem~\ref{thm:main}, which yields weak solutions with kinetic energy that is continuous in time.

Theorem~\ref{thm:main} is proven by using an intermittent convex integration scheme, which is necessary in order to reach beyond the $\sfrac 13$ regularity exponent, uniformly in time. Intermittent convex integration schemes have been introduced by the first and last authors in~\cite{BV19} in order to prove the non-uniqueness of weak (mild) solutions of the 3D Navier-Stokes equations with bounded kinetic energy, and then refined in collaboration with Colombo~\cite{BCV18} to construct solutions which have partial regularity in time. Recently, intermittent convex integration techniques have been used successfully to construct non-unique weak solutions for the transport equation (cf.~Modena and Sz\'ekelyhidi Jr.~\cite{ModenaSZ17,ModenaSZ18}, Bru\'e, Colombo, and De Lellis~\cite{brue2020positive}, and Cheskidov and Luo~\cite{cheskidov2020nonuniqueness}), the 2D Euler equations with vorticity in a Lorentz space (cf.~\cite{BrueColombo21}), the stationary 4D Navier-Stokes equations (cf.~Luo~\cite{Luo18}), the $\alpha$-Euler equations (cf.~\cite{RajMatt}), in the context of the MHD equations (cf.~Dai~\cite{Dai18}, the first and last authors with Beekie~\cite{BBV19}), and the effect of temporal intermittency has recently been studied by Cheskidov and Luo~\cite{cheskidov2020sharp}. We refer the reader to the reviews~\cite{BV_EMS19,BV20} for further references, and for a comparison between intermittent and homogenous convex integration.

When applied to three-dimensional nonlinear problems, intermittent convex integration has insofar only been successful at producing weak solutions with negligible spatial regularity indices, uniformly in time. As we explain in Section~\ref{sec:intro:ideas}, there is a fundamental obstruction to achieving high regularity: in physical space, intermittency causes concentrations that results in the formation of intermittent peaks, and to handle these peaks the existing techniques have used an extremely large separation between the frequencies in consecutive steps of the convex integration scheme.\footnote{This becomes less of an issue when one considers the equations of fluid dynamics in very high space dimensions, cf.~Tao~\cite{tao19}.} This paper is the first to successfully implement a high-regularity (in $L^2$), spatially-intermittent, temporally-homogenous, convex integration scheme in three space dimensions, and shows that for the 3D Euler system  any regularity exponent $\beta < \sfrac 12$ may be achieved.\footnote{It was known within the community (see Section~\ref{sss:onehalf} for a detailed explanation) that 
there is a key obstruction to reaching a regularity index in $L^2$ for a solution to the Euler equations larger than $\sfrac 12$ via convex integration.} In fact, the techniques developed in this paper are the backbone for the recent work~\cite{MattVlad} of the last two authors, which gives an alternative, intermittent, proof of the Onsager conjecture.

\subsection{Ideas and difficulties}
\label{sec:intro:ideas}

As alluded to in the previous paragraph, the main difficulty in reaching a high regularity exponent for weak solutions of \eqref{e:eulereq} is that the existing intermittent convex integration schemes do not allow for consecutive frequency parameters $\lambda_q$ and $\lambda_{q+1}$ to be close to each other. In essence, this is because intermittency smears out the set of active frequencies in the approximate solutions to the Euler system (instead of concentric spheres, they are more akin to thick concentric annuli), and several of the key estimates in the scheme require frequency separation to achieve $L^p$-decoupling (see Section~\ref{sss:onehalf}). Indeed, high regularity exponents necessitate an almost geometric growth of  frequencies ($\lambda_q = \lambda_0^q$), or at least a barely super-exponential growth rate $\lambda_{q+1} = \lambda_q^b$ with $0 < b - 1\ll 1$ (in comparison, the schemes in~\cite{BV19,BCV18} require $b \approx 10^{3}$).  Essentially every new idea in this paper is aimed either directly or indirectly at rectifying this issue: how does one take advantage of intermittency, and at the same time keep the frequency separation to be nearly geometric? 

The building blocks used in the convex integration scheme are intermittent pipe flows,\footnote{The moniker used in \cite{DaneriSzekelyhidi17} and the rest of the literature for these stationary solutions has been ``Mikado flows".  However, we rely rather heavily on the geometric properties of these solutions, such as orientation and concentration around axes, and so to emphasize the tube-like nature of these objects, we will frequently use the name ``pipe flows".} which we describe in Section~\ref{ss:concentratedpipes}. Due to their spatial concentration and their periodization rate, quadratic interactions of these building blocks produce both the helpful low frequency term which is used to cancel the previous Reynolds stress $\RR_q$, and also a number of other errors which live at intermediate frequencies.  These errors are spread throughout the frequency annulus with inner radius $\lambda_q$ and outer radius $\lambda_{q+1}$, and may have size only slightly less than that of $\RR_q$. If left untreated, these errors only allow for a very small regularity parameter $\beta$. In order to increase the regularity index of our weak solutions, we need to take full advantage of the frequency separation between the slow frequency $\lambda_q$ and the fast frequency $\lambda_{q+1}$. As such, the intermediate-frequency errors need to be further corrected via velocity increments designed to push these residual stresses towards the frequency sphere of radius $\lambda_{q+1}$.  The quadratic interactions among these higher-order velocity corrections themselves, and in principle also with the old velocity increments, in turn create {\em higher order Reynolds stresses}, which live again at intermediate frequencies (slightly higher than before), but whose amplitude is slightly smaller than before. This process of adding {\em higher order velocity perturbations} designed to cancel intermediate-frequency higher order stresses has to be repeated many times until all the resulting errors are either small, or they live at frequency $\approx \lambda_{q+1}$, and thus are also small upon inverting the divergence. See Sections~\ref{ss:higherorderstresses} and~\ref{ss:perturbation} for a more thorough account of this iteration. 

Throughout the process described in the above paragraph, we need to keep adding velocity increments, while at the same time keeping the high-high-high frequency interactions under control. The fundamental obstacle here is that when composing the intermittent pipe flows with the Lagrangian flow of the slow velocity field, the resulting deformations are not spatiotemporally homogenous. In essence, the intermittent nature of the approximate velocity fields implies that a sharp global control on their Lipschitz norm is unavailable, thus precluding us from implementing a gluing technique as in~\cite{Isett17,BDLSV17}. Additionally, we are faced with the issue that pipe flows which were added at different stages of the higher order correction process have different periodization rates and different spatial concentration rates, and may a-priori overlap. Our main idea here is to implement a {\em placement technique}  which uses the {\em relative intermittency} of pipe  flows from previous or same generations, in conjunction with a sharp bound on their local Lagrangian deformation rate, to determine suitable spatial shifts for the placement of new pipe  flows so that they dodge all other bent pipes which live in a restricted space-time region. This geometric placement technique is discussed in Section~\ref{ss:checkerboard:heuristics}.

A rigorous mathematical implementation of the heuristic ideas described in the previous two paragraphs, which crucially allows us to slow down the frequency growth to be almost geometric, requires extremely sharp information on all higher order errors and their associated velocity increments. For instance, in order to take advantage of the transport nature of the linearized Euler system while mitigating the loss of derivatives issue which is characteristic of convex integration schemes, we need to keep track of essentially \emph{infinitely many sharp material derivative estimates} for all velocity increments and stresses. Such estimates are naturally only attainable on a {\em local inverse Lipschitz timescale}, which in turn necessitates keeping track of the precise location in space of the peaks in the densities of the  pipe flows, and performing a frequency localization with respect to both the Eulerian and the Lagrangian coordinates. In order to achieve this, we introduce carefully designed {\em cutoff functions}, which are defined recursively for the velocity increments (in order to keep track of overlapping pipe flows from different stages of the iteration), and iteratively for the Reynolds stresses (in order to keep track of the correct amplitude of the perturbation which needs to be added to correct these stresses); see Section~\ref{ss:cutoffs}. The cutoff functions we construct effectively play the role of a joint Eulerian-and-Lagrangian Littlewood-Paley frequency decomposition, which in addition keeps track of both the position in space and the amplitude of various objects (more akin to a wavelet decomposition). The analysis of these cutoff functions requires estimating very high order commutators between Lagrangian and Eulerian derivatives which in great part are responsible for the length of this paper (see Section~\ref{sec:cutoff} and Appendix~\ref{sec:appendix}). Lastly, we mention an additional technical complication: since the sharp control of the Lipschitz norm of the approximate velocities in our scheme is local in space and time, we need to work with an inverse divergence operator (e.g.\ for computing higher order stresses) which, up to much lower order error terms, maintains the spatial support of the vector fields that it is applied to.  Additionally, we need to be able to estimate an essentially infinite number of  material derivatives applied to the output of this inverse divergence operator. This issue is addressed in Section~\ref{sec:inverse:divergence}.

The rest of the paper is organized as follows. Section~\ref{sec:outline} contains an outline of the convex integration scheme, in which we replace some of the actual estimates and definitions appearing in the proof with heuristic ones in order to highlight the new ideas at an intuitive level. The proof of Theorem~\ref{thm:main} is given in Section~\ref{section:inductive:assumptions}, assuming that a number of estimates hold true inductively for the solutions of the Euler-Reynolds system at every step of the convex integration iteration.  The remainder of the paper is dedicated to showing that the inductive bounds stated in Section~\ref{sec:inductive:estimates} may indeed be propagated from step $q$ to step $q+1$. Section~\ref{sec:building:blocks} contains the construction of the intermittent pipe  flows used in this paper and describes the careful placement required to show that these pipe  flows do not overlap on a suitable space-time set. The mollification step of the proof is performed in Section~\ref{sec:mollification:stuff}.  Section~\ref{sec:cutoff} contains the definitions of the cutoff functions used in the proof and establishes their properties.  Section~\ref{section:statements} breaks down the the main inductive bounds from Section~\ref{sec:inductive:estimates} into components which take into account the higher order stresses and perturbations.  Section~\ref{s:stress:estimates} then proves the constituent parts of the inductive bounds outlined in the previous section.  Section~\ref{sec:parameters} carefully defines the many parameters in the scheme, states the precise order in which they are chosen, and lists a few consequences of their definitions.   Finally, Appendix~\ref{sec:appendix} contains the analytical toolshed to which we appeal throughout the paper.

\subsection*{Acknowledgements}
T.B.~was supported by the NSF grant DMS-1900149 and a Simons Foundation Mathematical and Physical Sciences Collaborative Grant. 
N.M.~was supported by  the NSF grant DMS-1716466 and by Tamkeen under the
NYU Abu Dhabi Research Institute grant of the center SITE.
V.V.~was supported by the NSF grant CAREER DMS-1911413.

\section{Outline of the convex integration scheme}
\label{sec:outline}

\subsection{A guide to the parameters}
In order to make sharp estimates throughout the scheme, we will require numerous parameters.  For the reader's convenience, we have collected in this section the {\em heuristic definitions}
of all the parameters introduced in the following sections of the outline.  The parameters are listed in Section~\ref{sss:definitions} in the order corresponding to their first appearance in the outline.  We give as well brief descriptions of the significance of each parameter.

\subsubsection{Definitions}\label{sss:definitions}
\begin{definition}[\textbf{Parameters Introduced in Section~\ref{sec:intro}}]\
\begin{enumerate}[(1)]
    \item $\beta$ - The regularity exponent corresponding to a final solution $v\in C\left(\mathbb{R};H^\beta(\mathbb{T}^3)\right)$.
\end{enumerate}
\end{definition}
\begin{definition}[\textbf{Parameters Introduced in Section~\ref{ss:inductive:assumptions}}]\label{d:parameters:2.2}\ 
\begin{enumerate}[(1)]
    \item $q$ - The integer which represents the primary stages of the iterative convex integration scheme.
    \item ${\lambda_q = a^{(b^q)}}$ - The primary parameter used to quantify frequencies.  $a$ and $b$ will be chosen later, with $a\in\mathbb{R}_+$ being a sufficiently large positive number and $b\in\mathbb{R}$ a real number slightly larger than $1$.
    \item $\delta_q = \lambda_q^{-2\beta}$ - The primary parameter used to quantify amplitudes of stresses and perturbations.  
    \item $\tau_q= (\delta_q^{\sfrac{1}{2}}\lambda_q )^{-1}$ - The primary parameter used to quantify the cost of a material derivative $\partial_t + v_q \cdot \nabla$.\footnote{For technical reasons, $\tau_q^{-1}$ will be chosen to be slightly shorter than $\delta_q^{\frac{1}{2}}\lambda_q$. For the heuristic calculations, one may ignore this modification and simply use $\tau_q^{-1}=\delta_q^{\frac{1}{2}}\lambda_q$.}

\end{enumerate}
\end{definition}

\begin{definition}[\textbf{Parameters Introduced in Section~\ref{ss:concentratedpipes}}]\label{d:parameters:2.3}\
\begin{enumerate}[(1)]
    \item  $n$ - The primary parameter which will be used to divide up the frequencies between $\lambda_q$ and $\lambda_{q+1}$ and which will take non-negative integer values. The divisions will be used both for the frequencies of the higher order stresses in Section~\ref{ss:higherorderstresses} as well as the thickness of the intermittent pipe flows used to correct the higher order stresses.
    \item  $\nmax$ - A large integer which is fixed independently of $q$ and which sets the largest allowable value of $n$.
    \item \label{item:r:q+1:n:def} $\displaystyle{r_{q+1,n} = \left( \lambda_q \lambda_{q+1}^{-1}\right)^{\left(\frac{4}{5}\right)^{n+1}}}$ - The parameter quantifying intermittency, or the thickness of a tube periodized at unit scale for values of $n$ such that $0\leq n \leq \nmax$.\footnote{In particular, this choice gives $r_{q+1,n+1} = r_{q+1,n}^{\frac 45}$. In our proof, the inequality $r_{q+1,n}^3 \ll r_{q+1,n+1}^4$ plays a crucial role. In order to absorb $q$ independent constants, as well as to ensure that there is a sufficient gap between these parameters to ensure decoupling, we have chosen to work with the $\frac 45$ instead of the $\frac 34$ geometric scale.}
    \item  $\displaystyle{\lambda\qn = \lambda_{q+1}\rqnperp = \lambda_q^{\left(\frac{4}{5}\right)^{n+1}}\lambda_{q+1}^{1-\left(\frac{4}{5}\right)^{n+1}}}$ - The minimum frequency present in an intermittent pipe  flow $\WW_{q+1,n}$.  Equivalently, $\left(\lambda_{q+1}\rqnperp\right)^{-1}$ is the scale to which $\WW_{q+1,n}$ is periodized.
\end{enumerate}
\end{definition}

\begin{figure}[htb!]
\centering
\includegraphics[width=\textwidth]{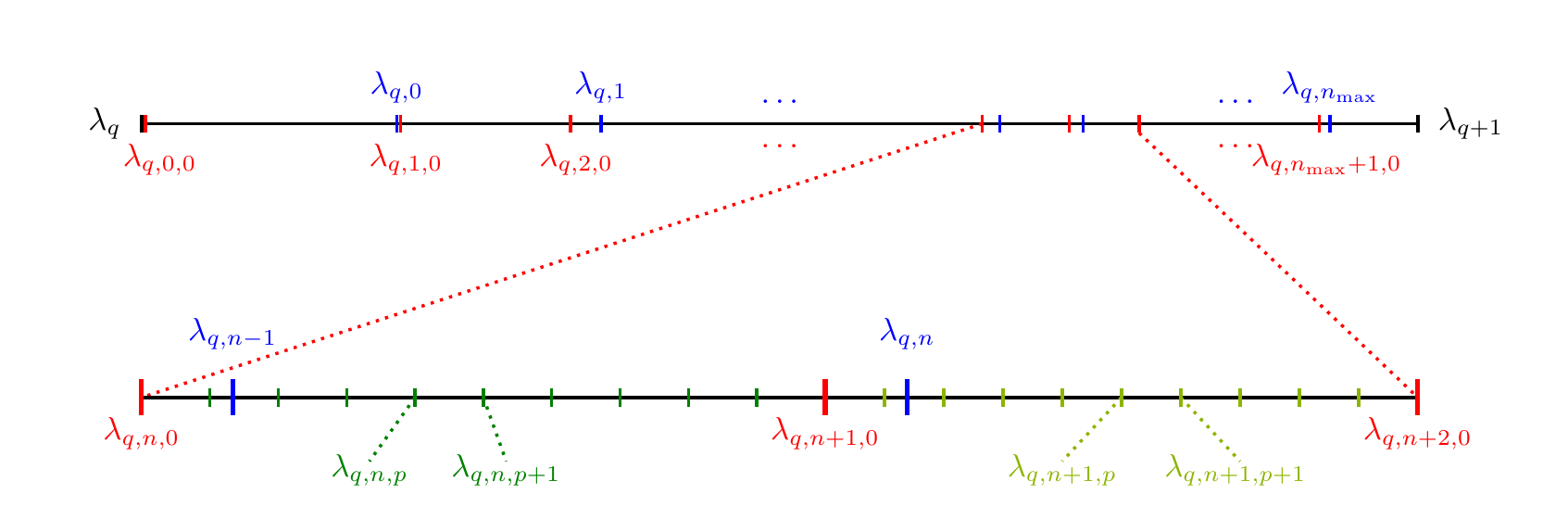}
\caption{\small Schematic of the frequency parameters appearing in Definitions~\ref{d:parameters:2.2} and~\ref{d:parameters:2.4}.}
\end{figure}

\begin{definition}[\textbf{Parameters Introduced in Section~\ref{ss:higherorderstresses}}]
\label{d:parameters:2.4}\
\begin{enumerate}[(1)]
\item  For $2\leq n \leq \nmax$, $\displaystyle \lambda_{q,n,0} = \lambda_q^{\left(\frac{4}{5}\right)^{n-1}\cdot\frac{5}{6}}\lambda_{q+1}^{1-\left(\frac{4}{5}\right)^{n-1}\cdot\frac{5}{6}}$ is the minimum frequency present in the higher order stress $\RR\qn$. Conversely, $\lambda_{q,n+1,0}$ is the maximum frequency present in $\RR\qn$. When $n=0$, we set $\lambda_{q,0,0}=\lambda_q$ to be the maximum frequency present in $\RR_{q,0}=\RR_q$, and when $n=1$, $\lambda_{q,1,0}=\lambda_{q,0}$ is the minimum frequency present in $\RR_{q,1}$, while $\lambda_{q,2,0}$ is the maximum frequency. 
\item   $p$ - A secondary parameter which takes positive integer values and which will be used to divide up the frequencies in between $\lambda_{q,n,0}$ and $\lambda_{q,n+1,0}$, as well as the higher order stresses.  
\item  $\pmax$ - A large integer, fixed independently of $q$, which is the largest allowable value of $p$.
\item $\displaystyle \lambda\qnp = \lambda_{q,n,0}^{1-\frac{p}{\pmax}}\lambda_{q,n+1,0}^{\frac{p}{\pmax}}$ - The maximum frequency present in the higher order stress $\RR\qnp$ for $1\leq n \leq \nmax$ and $1\leq p \leq\pmax$.  Conversely, $\lambda\qnpminus$ is the minimum frequency in $\RR\qnp$. When $n=0$ and $p$ takes any value, we adopt the convention that $\lambda_{q,0,p}=\lambda_q$.
\item   $\displaystyle f\qn= \lambda_{q,n+1,0}^{\frac{1}{\pmax}} \lambda_{q,n,0}^{-\frac{1}{\pmax}}$ - The increment between frequencies $\lambda\qnpminus$ and $\lambda\qnp$ for $n\geq 1$.  We have the equalities $$ \lambda\qnp=\lambda_{q,n,0}f\qn^p , \qquad  \lambda_{q,n+1,0}=\lambda_{q,n,0}f\qn^\pmax\, . $$
For ease of notation, when $n=0$ we set $f\qn=1$.
\item For $n=0$ and $p=1$, $\delta_{q+1,0,1}:=\delta_{q+1}$ is the amplitude of $\RR_q:=\RR_{q,0}$.  For $n=0$ and $p\geq 2$, $\delta_{q+1,0,p}=0$, since there are no higher order stresses at $n=0$. For $n\geq 1$ and any value of $p$, the amplitude of $\RR_{q,n,p}$ is given by
$$\displaystyle\delta\qplusnp:= \frac{\delta_{q+1}\lambda_q}{\lambda\qnpminus}\cdot \prod_{n'<n} f_{q,n'}.$$
One should view the product of $f_{q,n'}$ terms as a negligible error, which is justified by calculating
\begin{align}
\prod\limits_{0\leq n' \leq \nmax} f_{q,n'} &= \left( \frac{\lambda_{q,\nmax+1,0}}{\lambda_{q,1,0}} \right)^\frac{1}{\pmax}  \leq \left( \frac{\lambda_{q+1}}{\lambda_q} \right)^{\frac{1}{\pmax}} \label{e:fqn:inequality}
\end{align}
and assuming that $\pmax$ is large.
\end{enumerate}
\end{definition}

\begin{definition}[\textbf{Parameters Introduced in Section~\ref{ss:cutoffs}}]\label{d:parameters:2.5}\
\begin{enumerate}[(1)]
\item  $\varepsilon_\Gamma$ - A very small positive number.
\item $\displaystyle{\Gamma_{q+1}= \left( \lambda_{q+1} \lambda_q^{-1}\right)^{\varepsilon_\Gamma}}$ -  A parameter which will be used to quantify deviations in amplitude. In particular, $\Gamma_q$ will be used to quantify amplitudes of both velocity fields and (higher-order) stresses.
\end{enumerate}
\end{definition}

\subsection{Inductive assumptions}\label{ss:inductive:assumptions}
For every non-negative integer $q$ we will construct a solution $(v_q,p_q, \RR_q)$ to the Euler-Reynolds system
\begin{subequations}\label{e:euler_reynolds}
\begin{align}
\partial_t v_q + \div (v_q\otimes v_q)+\nabla p_q &= \div\RR_q\\
\div v_q &= 0\, .
\end{align}
\end{subequations}
Here $\RR_q$ is assumed to be a trace-free symmetric matrix. The relative size of the approximate solution $v_q$ and the Reynolds stress error $\RR_q$ will be measured in terms of the frequency parameter $\lambda_q$ and the amplitude parameter $\delta_{q}$, which are defined in Definition~\ref{d:parameters:2.2}. We will propagate the following basic inductive estimates on $(v_q,\RR_q)$:\footnote{By $\left\| v_q \right\|_{H^1}$, we actually mean $\left\| v_q \right\|_{C^0_t H^1_x}$.  Similarly, $\|\RR_q\|_{L^1}$ stands for $\| \RR_q\|_{C^0_t L^1_x}$. Unless stated explicitly otherwise, all the norms used in this paper represent analogous uniform in time estimates and will be abbreviated as such.}
\begin{align}
\label{e:intro_v_est}
\norm{v_q}_{H^1} &\leq \delta_{q}^{\frac12}\lambda_q\\
\| \RR_q\|_{L^1} &\leq \delta_{q+1}\label{e:intro_R_est}.
\end{align}
We shall see later that in order to build solutions belonging to $\dot{H}^{\beta}$ for $\beta$ approaching $\frac{1}{2}$, we must propagate additional estimates on higher order material and spatial derivatives of both $v_q$ and $\RR_q$ in $L^2$ and $L^1$, respectively. Roughly speaking, every spatial derivative on either $v_q$ or $\RR_q$ costs a factor of $\lambda_q$.  Additional material derivatives are more delicate and will be discussed further in Section~\ref{ss:cutoffs}, but for the time being, one may imagine that each material derivative $\Dtq:=\partial_t + v_q\cdot \nabla$ on $v_q$ or $\RR_q$ costs a factor of $\tau_q^{-1}$.

\subsection{Intermittent pipe flows}
\label{ss:concentratedpipes}
Pipe  flows, both homogeneous and intermittent, have proven to be one of the most useful components of many convex integration schemes.  Homogeneous pipe  flows were introduced first by Daneri and Sz\'{e}kelyhidi Jr.~\cite{DaneriSzekelyhidi17}. The prototypical pipe  flow in the $\vec{e}_3$ direction is constructed using a smooth function $\rho:\mathbb{R}^2\rightarrow\mathbb{R}$ which is compactly supported, for example in a ball of radius $1$ centered at the origin, and has zero mean. Letting $\varrho:\mathbb{T}^2\rightarrow\mathbb{R}$ be the $\mathbb{T}^2$-periodized version of $\rho$, the $\mathbb{T}^3$-periodic pipe flow $\WW:\mathbb{T}^3\rightarrow\mathbb{R}^3$ is defined as
\begin{equation}\label{def:intro:pipe:flow}
    \WW(x_1,x_2,x_3) = \varrho(x_1,x_3)  {e}_2 \, .
\end{equation}
It is immediate that $\WW$ is divergence-free and a stationary solution to the Euler equations. Pipe flows such as $\WW$ have been used in convex integration schemes which produce solutions in $L^\infty$-based spaces~\cite{DaneriSzekelyhidi17,Isett2018,BDLSV17}.  At the $q^{\textnormal{th}}$ stage of the iteration, the $\frac{\mathbb{T}^3}{\lambda_{q+1}}$-periodized pipe flow $ \WW\left(\lambda_{q+1}\cdot\right)$ is used to construct the perturbation.

By contrast, intermittent pipe flows are \emph{not} spatially homogeneous.  Intermittency in the context of convex integration schemes was introduced by the first and last authors in \cite{BV19} via \emph{intermittent Beltrami flows}, which are defined via their Fourier support and may be likened to modified and renormalized Dirichlet kernels.  Intermittent pipe  flows were introduced by Modena and Sz\'{e}kelyhidi Jr.~in the context of the transport and transport-diffusion equation~\cite{ModenaSZ17} and have also been utilized for the higher dimensional (at least four dimensional\footnote{In three dimensions, intermittent pipe  flows are not sufficiently sparse to handle the error term arising from the Laplacian.  This issue was addressed by Colombo and the first and last authors in~\cite{BCV18} through the usage of \emph{intermittent jets}, and similar objects have been used in subsequent papers as well (see work of Brue, Colombo, and De Lellis \cite{brue2020positive}, Cheskidov and Luo~\cite{cheskidov2020nonuniqueness,cheskidov2020sharp}).}) Navier-Stokes equations \cite{Luo18,tao19}. The precise objects we use are defined in \eqref{def:pipe:flow:main} in Proposition~\ref{pipeconstruction}, but let us briefly describe some of their important attributes. The intermittency is quantified by the parameter $r_{q+1,n}\ll 1$. Let $\rho_{r_{q+1,n}} \colon \mathbb{R}^2\rightarrow\mathbb{R}$ be defined by $\rho_{r_{q+1,n}}(\cdot)=\rho\left(\frac{\cdot}{r_{q+1,n}}\right)$, and let $\varrho_{r_{q+1,n}}$ be the $\mathbb{T}^2$-periodized version of $\rho_{r_{q+1,n}}$. Thus one can see that $r_{q+1,n}$ describes the \emph{thickness of the pipes at unit scale.} In order to make the intermittent pipe flows of unit size in $L^2(\mathbb{T}^3)$, one must multiply by a factor of $\rqnperp^{-1}$, meaning that the Lebesgue norms of the resulting object $\mathbb{W}_{\rqnperp}$ scale like
\begin{equation}\label{e:intro:pipescaling}
    \norm{\mathbb{W}_{\rqnperp}}_{L^p(\T^3)} \sim \rqnperp^{\frac{2}{p}-1}. 
\end{equation}
Let $\WW_{q+1,n}$ be the $\frac{\mathbb{T}^3}{(\rqnperp\lambda_{q+1})}$-periodic version of $\WW_{\rqnperp}$.  Notice that this implies that the thickness of the pipes comprising $\WW_{q+1,n}$ is of order $\lambda_{q+1}^{-1}$ for all $n$, and that the Lebesgue norms of the periodized object $\WW_{q+1,n}$ depend only on $\rqnperp$. Per Definition~\ref{d:parameters:2.3}, the thickness of the pipes used in the perturbation at stage $q+1$ will be quantified by 
$$r_{q+1,n} = \left(\frac{\lambda_q}{\lambda_{q+1}}\right)^{\left(\frac{4}{5}\right)^{n+1}}.$$ 
This choice will be jusified upon calculation of the heuristic bounds.

\begin{figure}[h]
\centering
\includegraphics[width=0.8\textwidth]{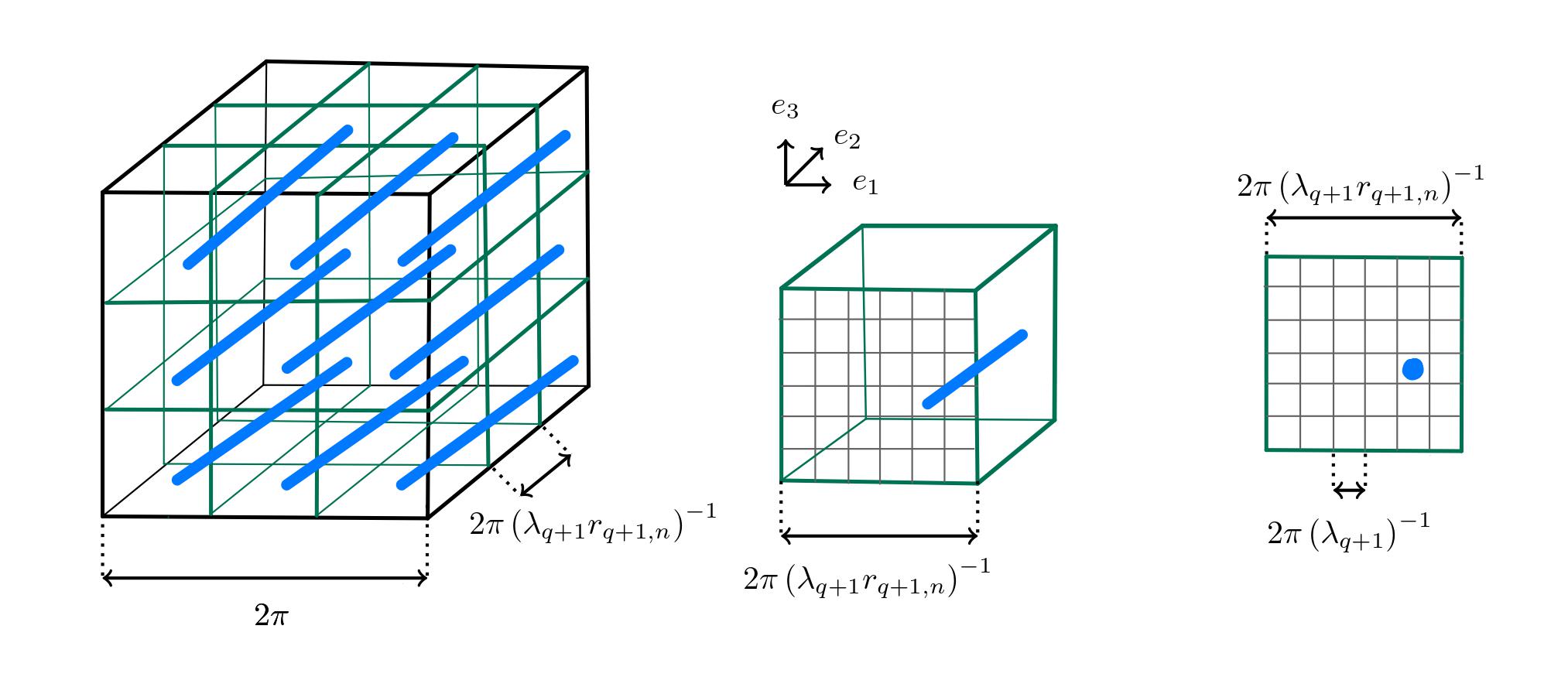}
\caption{\small A pipe flow $\WW_{q+1,n}$ which is periodized to scale $(\lambda_{q+1} r_{q+1,n})^{-1}=\lambda_{q,n}^{-1}$ is placed in a direction parallel to the $e_2$ axis. Upon taking into account periodic shifts, we note that there are $r_{q+1,n}^{-2}$ many options to place this pipe. This degree of freedom will be used later, see e.g.~Figure~\ref{fig:Placing}.}\label{fig:choices}
\end{figure}

\subsubsection{Lagrangian coordinates, intermittency, and placements}
In order to achieve the optimum regularity $\beta$, we will define the pipe flows which comprise the perturbation at stage $q+1$ in Lagrangian coordinates corresponding to the velocity field $v_q$. Due to the inherent instability of Lagrangian coordinates over timescales longer than that dictated by the Lipschitz norm of the velocity field, there will be many sets of coordinates used in different time intervals which are then patched together using a partition of unity. This technique has been used frequently in recent convex integration schemes, beginning with work of Isett~\cite{Isett12}, the first author, De Lellis, and Sz\'{e}kelyhidi Jr.~\cite{buckmaster2013transporting}, and Isett, the first author, De Lellis, and Sz\'{e}kelyhidi Jr.~\cite{BDLISZ15}, but perhaps most notably in the proof of the Onsager conjecture  by Isett~\cite{Isett2018} and the subsequent strengthening to dissipative solutions by the first and last authors, De Lellis, and Sz\'{e}kelyhidi Jr.~\cite{BDLSV17}. 

The proof of Onsager's conjecture employs the gluing technique to prevent pipe  flows defined using different Lagrangian coordinate systems from overlapping. The intermittent quality of our building blocks, and thus the approximate solution $v_q$, appears to obstruct the successful implementation of the gluing technique, since gluing requires a sharp control on the global Lipschitz norm of the velocity field which will be unavailable. Thus, we cannot use the gluing technique and must control in a different fashion the possible interactions between two intermittent pipe flows defined using different Lagrangian coordinate systems.  

To control these interactions, we have introduced a \emph{placement technique} (cf.~Proposition~\ref{prop:disjoint:support:simple:alternate}) which is used to completely prevent all such interactions.  This placement technique is predicated on a simple observation about intermittent pipe flows, which   to our knowledge has not yet been used in any convex integration schemes to date.  When the diameter of the pipe at unit scale is of size $\rqnperp$, there are $(\rqnperp)^{-2}$ disjoint choices for the support of pipe.  These choices simply correspond to shifting the intersection of the axis of the pipe in the plane which is perpendicular to the axis, cf.~Proposition~\ref{prop:pipe:shifted}.  This degree of freedom is unaffected by periodization and is depicted in Figure~\ref{fig:choices} for a $\frac{\mathbb{T}^3}{\lambda_{q+1}r_{q+1,n}}$-periodic intermittent pipe flow $\WW_{q+1,n}$. We will exploit this degree of freedom to choose placements for each set of pipes which \emph{entirely avoid} other sets of pipes on small discretized regions of space-time.  The space-time discretization is made possible through the usage of cutoff functions which will be discussed in more detail later in Section~\ref{ss:cutoffs}.  We remark that De Lellis and Kwon~\cite{delellis2020nonuniqueness} have introduced a placement technique in the context of $C^\alpha$, globally dissipative solutions to the 3D Euler equations which is predicated on restricting the timescale of the Lagrangian coordinate systems to be significantly shorter than the Lipschitz timescale. This restriction significantly limits the regularity of the final solution and is thus not suited for a intermittent scheme aimed at $H^{\frac{1}{2}-}$ regularity.

\subsection{Higher order stresses}\label{ss:higherorderstresses}

\subsubsection{Regularity beyond \texorpdfstring{$\sfrac{1}{3}$}{onethird}}\label{sss:onehalf}
The resolution of the flexible side of the Onsager conjecture in \cite{Isett2018} and \cite{BDLSV17} mentioned previously shows that given some prescribed regularity index $\beta\in(0,\frac{1}{3})$, one can construct dissipative weak solutions $u$ in $C^\beta$. Conversely, following on partial work by Eyink~\cite{Eyink94}, Constantin, E, and Titi~\cite{ConstantinETiti94} have proven that conservation of energy in the Euler equations requires only that $u\in L^3_t\left(B^\alpha_{3,\infty}\right)$ for $\alpha>\sfrac{1}{3}$. This leaves open the possibility of building dissipative weak solutions with more than $\frac{1}{3}$-many derivatives in $L^p\left(\mathbb{T}^3\right)$ (uniformly in time in our case) for $p<3$. 

Let us present a heuristic estimate which indicates a regularity limit of $H^{\frac 12}$ for solutions produced via convex integration schemes. For this purpose, let us focus on one of the principal parts of the stress in an intermittent convex integration schemes (for the familiar reader, this is part of the oscillation error). The perturbations include a coefficient function $a$ which depends on $\RR_q$ and thus for which derivatives cost $\lambda_q$ and which has amplitude $\delta_{q+1}^{\sfrac{1}{2}}$ (the square root of the amplitude of the stress).  These coefficient functions are multiplied by intermittent pipe flows $\WW_{q+1,0}$ for which derivatives cost $\lambda_{q+1}$ and which have unit size in $L^2$, but are only periodized to scale $\left(\lambda_{q+1}r_{q+1,0}\right)^{-1}$. When the divergence lands on the square of the coefficient function $ a^2$ in the nonlinear term, the resulting error term satisfies the estimate
\begin{align}\label{eq:div:inverse:example}
    \left\| \div^{-1} \left(\nabla ( a^2) \mathbb{P}_{\neq 0} (\WW_{q+1,0}\otimes\WW_{q+1,0})\right) \right\|_{L^1} \leq \frac{\delta_{q+1}\lambda_q}{\lambda_{q+1}r_{q+1,0}}.
\end{align}
The numerator is the size of $\nabla (a^2)$ in $L^1$, while the denominator is the gain induced by inverting the divergence at $\lambda_{q+1}r_{q+1,0}$, which is the minimum frequency of $\mathbb{P}_{\neq 0} (\WW_{q+1,0}\otimes\WW_{q+1,0}) = \WW_{q+1,0}\otimes\WW_{q+1,0} - \fint_{\T^3} \WW_{q+1,0}\otimes \WW_{q+1,0}$.  Note that we have used implicitly that $\WW_{q+1,0}$ has unit $L^2$ norm, and that by periodicity $\mathbb{P}_{\neq 0} (\WW_{q+1,0}\otimes\WW_{q+1,0})$ decouples from $\nabla (a^2)$. This error would be minimized when $r_{q+1,0}=1$, in which case
\begin{align}\label{eq:onehalf}
     \frac{\delta_{q+1}\lambda_q}{\lambda_{q+1}} < \delta_{q+2}  
    &\iff \lambda_{q+1}^{-2\beta+\frac{1}{b}} < \lambda_{q+1}^{-2\beta b+1}\notag\\
    &\iff 2\beta b^2 - 2\beta b < b-1\notag\\
    &\iff 2\beta b (b-1) < b-1\notag\\
    &\iff \beta< \frac{1}{2b}.
\end{align}
Any intermittency parameter $r_{q+1,0}\ll1$ would \textit{weaken} this estimate since the gain induced from inverting the divergence will only be $\lambda_{q+1}r_{q+1,0}\ll\lambda_{q+1}$. On the other hand, we will see that a small choice of $r_{q+1,0}$ \emph{strengthens all other error terms}, and because of this, in our construction we will choose $r_{q+1,0}$ as in Definition~\ref{d:parameters:2.3}, item~\eqref{item:r:q+1:n:def}. One may refer to the blog post of Tao~\cite{tao19} for a slightly different argument which reaches the same apparent regularity limit. This apparent regularity limit is independent of dimension,
and we believe that the method in this paper can not be modified to yield weak solutions with regularity $L^\infty_t W^{s,p}_x$ with $s>1/2$, for any $p \in [1,2]$.

The higher order stresses mentioned in Section~\ref{sec:intro:ideas} will compensate for the losses incurred in this nonlinear error term when $r_{q+1,0}\ll1$.  As we shall describe in the next section, we use the phrase ``higher order stresses" to describe errors which are higher in frequency and smaller in amplitude than $\RR_q$, but not sufficiently small enough or at high enough frequency to belong to $\RR_{q+1}$.  Similarly, ``higher order perturbations" are used to correct the higher order stresses and thus increase the extent to which an approximate solution solves the Euler equations. 

\subsubsection{Specifics of the higher order stresses}\label{ss:higher:order:stress:details}
In convex integration schemes which measure regularity in $L^\infty$ (i.e. using H\"{o}lder spaces $C^\alpha$), pipe  flows interact through the nonlinearity to produce low ($\approx\lambda_{q}$) and high ($\approx \lambda_{q+1}$) frequencies. We denote by $w_{q+1,0}$ the perturbation designed to correct $\RR_{q}$. In the absence of intermittency, the low frequencies from the self-interaction of $w_{q+1,0}$ cancel the Reynolds stress error $\RR_q$, and the high frequencies are absorbed by the pressure up to an error small enough to be placed in $\RR_{q+1}$. In an intermittent scheme, the self-interaction of the intermittent pipe flows comprising $w_{q+1,0}$ produces low, intermediate, and high frequencies. The low and high frequencies play a similar role as before. However, the intermediate frequencies cannot be written as a gradient, nor are small enough to be absorbed in $\RR_{q+1}$. This issue has limited the available regularity on the final solution in many previous intermittent convex integration schemes.  In order to reach the threshold $H^{\frac 12}$, we address this issue using higher order Reynolds stress errors $\RR_{q,n}$ for $n=1,2,\dots, {\nmax}$, cf.~Figure~\ref{figure:w:0}.  
\begin{figure}[htb!]
\centering
\includegraphics[scale=.8]{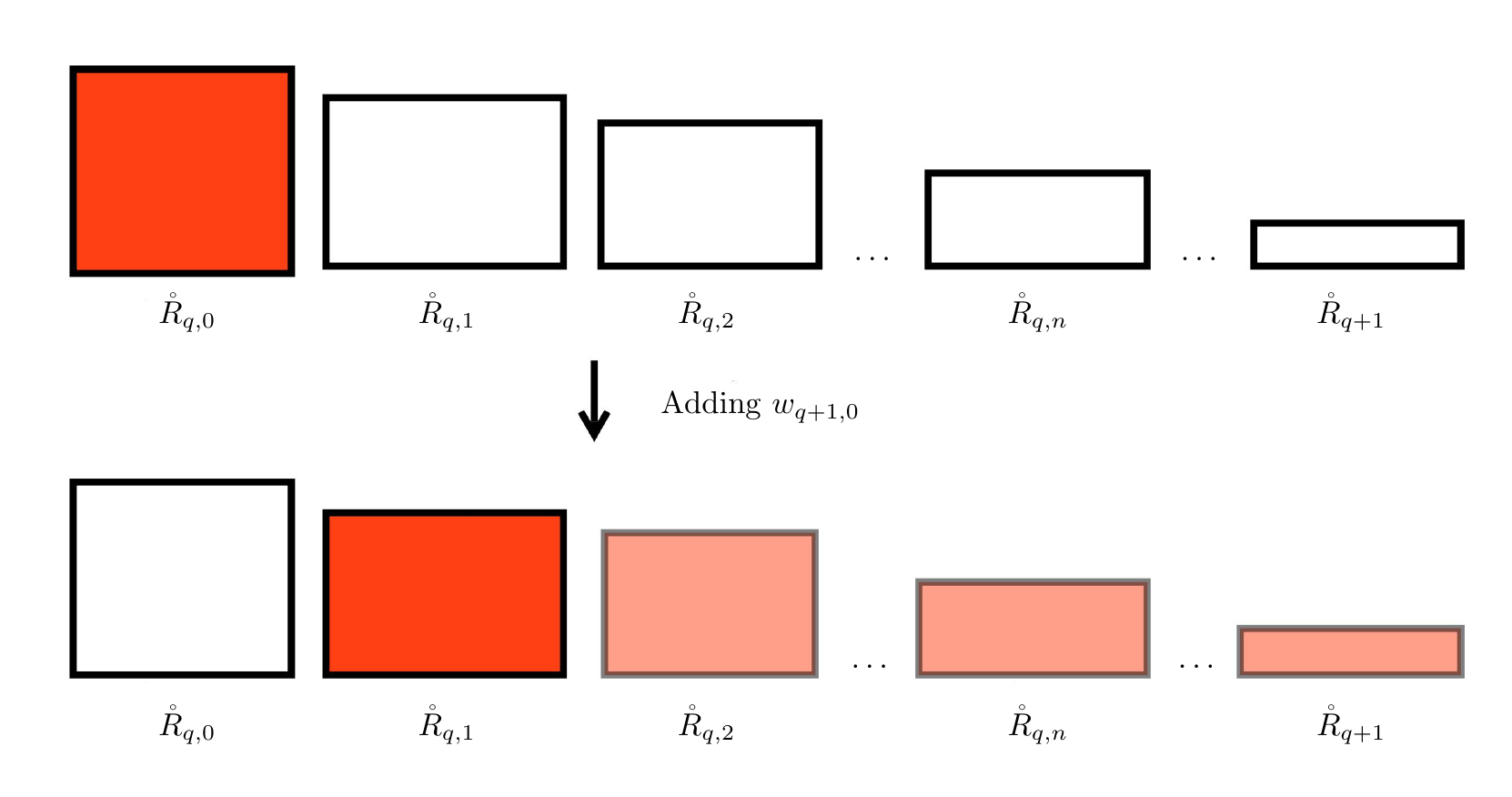}
\caption{\small Adding the increment $w_{q+1,0}$ corrects the stress $\RR_{q,0}=\RR_q$, but produces error terms which live at frequencies that are intermediate between $\lambda_q$ and $\lambda_{q+1}$, due to the intermittency of $w_{q+1,0}$. These new errors are sorted into higher order stresses $\RR_{q,n}$ for $1\leq n \leq \nmax$, as depicted above. The heights of the boxes corresponds to the amplitude of the errors that will fall into them, while the frequency support of each box increases from $\lambda_q$ for $\RR_{q,0}=\RR_q$, to $\lambda_{q+1}$ for $\RR_{q+1}$. }\label{figure:w:0}
\end{figure}

After the addition of $w_{q+1,0}$ to correct $\RR_q$, which is labeled in Figure~\ref{figure:w:1} as $\RR_{q,0}$, low frequency error terms are produced, which we   divide into higher order stresses.  To correct the error term of this type at the \emph{lowest} frequency, which is labeled $\RR_{q,1}$ in Figure~\ref{figure:w:1}, we   add a sub-perturbation $w_{q+1,1}$.  The subsequent bins are lighter in color to emphasize that they are not yet full; that is, there are more error terms which have yet to be constructed but will be sorted into such bins. The emptying of the bins then proceeds inductively on $n$, as we add higher order perturbations $w_{q+1,n}$, which are designed to correct $\RR_{q,n}$. For $1\leq n \leq \nmax$, the frequency support of $\RR_{q,n}$ is\footnote{In reality, the higher order stresses are not compactly supported in frequency.  However, they will satisfy derivative estimates to very high order which are characteristic of functions with compact frequency support.}
\begin{equation}\label{e:intro:rqnsupport}
\left\{k \in {\mathbb Z}^3 \colon \lambda_{q,n,0} \leq |k| < \lambda_{q,n+1,0}\right\}.
\end{equation}
This division will be justified upon calculation of the heuristic bounds in Section \ref{ss:reynoldsheuristic}.

\begin{figure}[h]
\centering
\includegraphics[scale=.8]{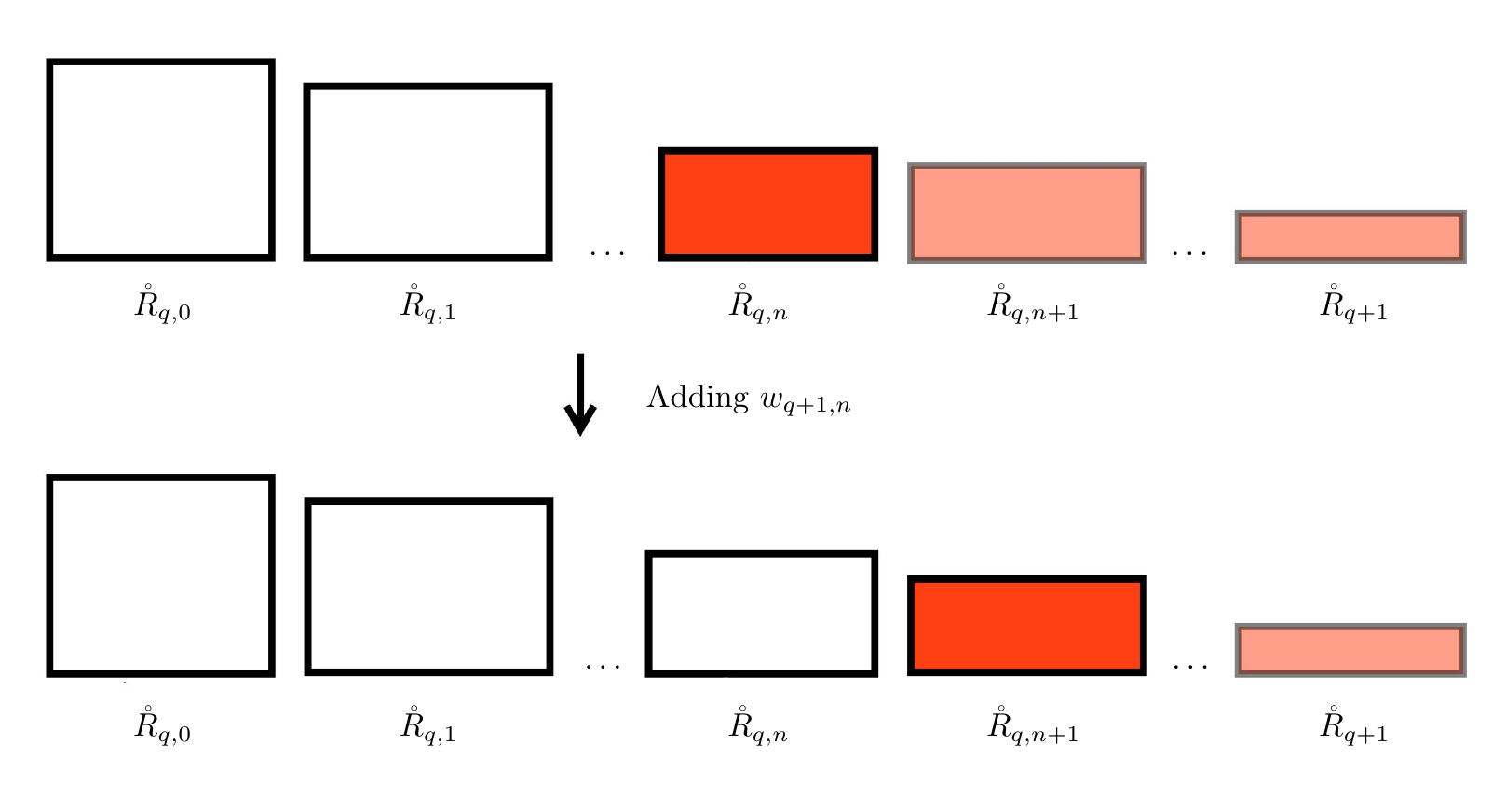}
\caption{\small Adding $w_{q+1,n}$ to correct $\RR_{q,n}$ produces error terms which are distributed among the Reynolds stresses $\RR_{q,n'}$ for $n+1\leq n'\leq\nmax$.}\label{figure:w:1}
\end{figure}

Let us now explain the motivation for the division of $\RR_{q,n}$ into the further subcomponents $\RR_{q,n,p}$. Suppose that we add a perturbation $w_{q+1,n}$ to correct $\RR_{q,n}$ for $n\geq 1$. The amplitude of $w_{q+1,n}$ would depend on the amplitude of $\RR_{q,n}$, which in turn depends on the gain induced by inverting the divergence to produce $\RR_{q,n}$, which depends then on the minimum frequency $\lambda_{q,n,0}$.  However, derivatives on the \textit{low} frequency coefficient function used to define $w_{q+1,n}$ would depend on the \textit{maximum} frequency of $\RR_{q,n}$, which is $\lambda_{q,n+1,0}$. The (sharp-eyed) reader may at this point object that the first derivative on the low-frequency coefficient function $\nabla ( a  ( \RR\qn) )$ should be cheaper, since $\RR\qn$ is obtained from inverting the divergence, and taking the gradient of the cutoff function written above should thus morally involve bounding a zero-order operator.  However, constructing the low-frequency coefficient function presents technical difficulties which prevent us from taking advantage of this intuition.  In fact, the failure of this intuition is the sole reason for the introduction of the parameter $p$, as one may see from the heuristic estimates later. In any case, increasing the regularity $\beta$ of the final solution requires minimizing this gap between the gain in amplitude provided by inverting the divergence and the cost of a derivative, and so we subdivide $\RR_{q,n}$ into further components $\RR_{q,n,p}$ for $1\leq p \leq p_{max}$.\footnote{There are certainly a multitude of ways to manage the bookkeeping for amplitudes and frequencies. Using both $n$ and $p$ is convenient because then $n$ is the only index which quantifies the rate of periodization.} Both ${\nmax}$ and $\pmax$ are fixed independently of $q$.  Each component $\RR_{q,n,p}$ then will have frequency support in the set
\begin{equation}\label{e:intro:rqnpsupport}
\left\{ k \in {\mathbb Z}^3 \colon \lambda\qnpminus \leq |k| < \lambda\qnp \right\} = \left\{k \in {\mathbb Z}^3 \colon \lambda_{q,n,0} f\qn^{p-1} \leq |k| < \lambda_{q,n,0} f\qn^p \right\}.
\end{equation}
Notice that by the definition of $f_{q,n}$ in Definition~\ref{d:parameters:2.4}, \ref{e:intro:rqnpsupport} defines a partition of the frequencies in between $\lambda_{q,n,0}$ and $\lambda_{q,n+1,0}$ for $1\leq p \leq \pmax$.  Figure~\ref{figure:dividing} depicts this division, and we shall describe in the heuristic estimates how each subcomponent $\RR_{q,n,p}$ is corrected by $w_{q+1,n,p}$, with all resulting errors absorbed into either $\RR_{q+1}$ or $\RR_{q,n'}$ for $n'>n$.
\begin{figure}[h]
\centering
\includegraphics[width=0.8\textwidth]{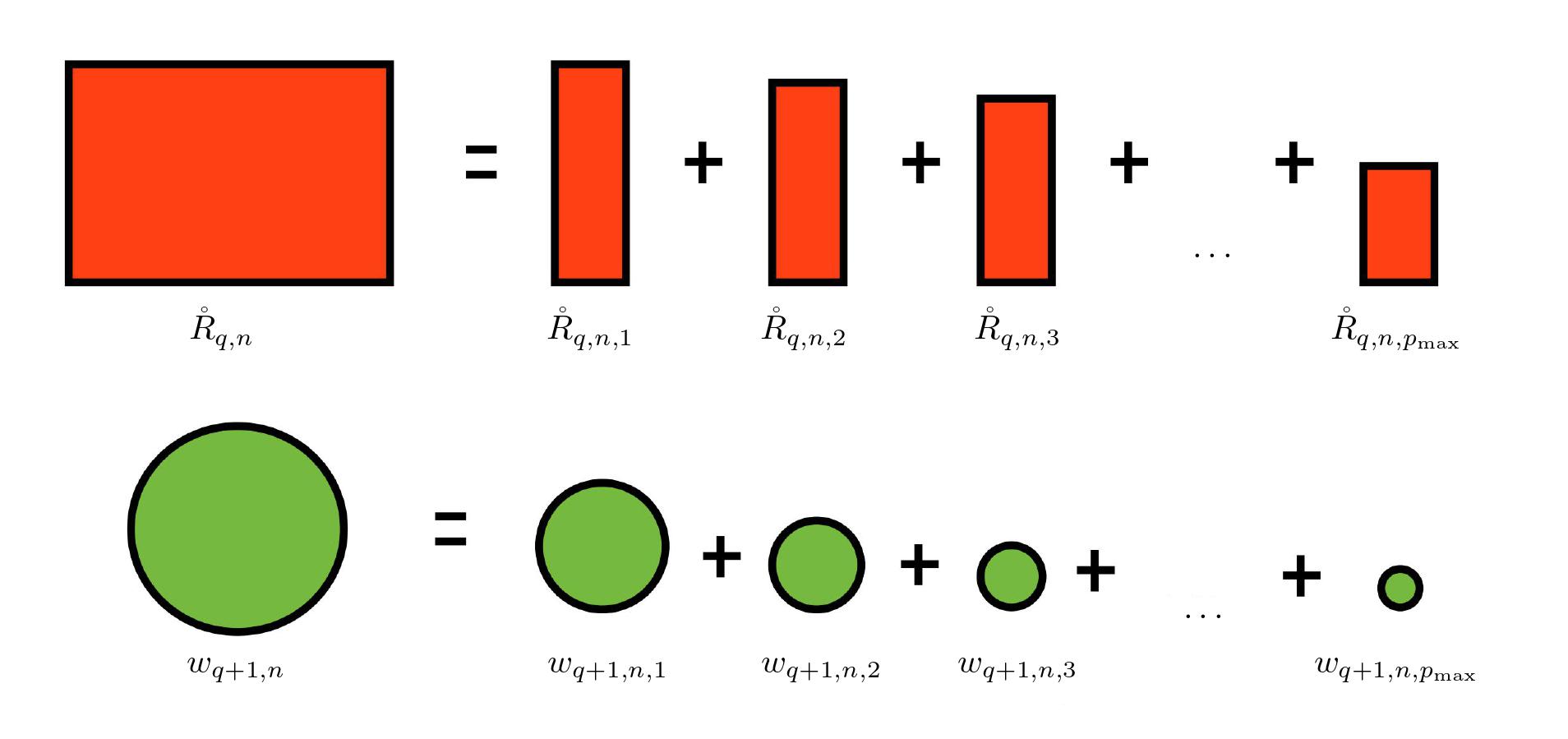}
\caption{\small The higher order stress $\RR_{q,n}$ is decomposed into components $\RR\qnp$, which increase in frequency and decrease in amplitude as $p$ increases. We use the base of the red boxes to indicate support in frequency, where frequency is increasing from left to right, and the height to indicate amplitudes.  Each subcomponent $\RR\qnp$ is corrected by its own corresponding sub-perturbation $w_{q+1,n,p}$, which has a commensurate frequency and amplitude.}\label{figure:dividing}
\end{figure}

Thus, the net effect of the higher order stresses is that one may take errors for which the inverse divergence provides a weak estimate due to the presence of relatively low frequencies and push them to higher frequencies for which the inverse divergence estimate is stronger. We will repeat this process until all errors are moved (almost) all the way to frequency $\lambda_{q+1}$, at which point they are absorbed into $\RR_{q+1}$. Heuristically, this means that in constructing the perturbation $w_{q+1}$ at stage $q$, we have \emph{eliminated} all the higher order error terms which arise from self-interactions of intermittent pipe flows, thus producing a solution $v_{q+1}$ to the Euler-Reynolds system at level $q+1$ which is \emph{as close as possible} to a solution of the Euler equations. We point out that one side effect of the higher order perturbations is that the total perturbation $w_{q+1}$ has spatial support which is \emph{not} particularly sparse, since as $n$ increases the perturbations $w_{q+1,n}$ become successively less intermittent and thus more homogeneous.  At the same time, the frequency support of our solution is also not too sparse, since $b$ is close to $1$ and $r_{q+1,0}=\left( \lambda_q \lambda_{q+1}^{-1}\right)^{\frac 45}$, so that many of the frequencies between $\lambda_q$ and $\lambda_{q+1}$ are active.

\subsection{Cut-off functions}\label{ss:cutoffs}
\subsubsection{Velocity and stress cut-offs}
The concept of a turnover time, which is proportional to the inverse of the gradient of the mean flow $v_q$, is crucial to the previous convex integration schemes mentioned earlier which utilized Lagrangian coordinates. Since the perturbation is expected to be roughly flowed by the mean flow $v_q$, the turnover time determines a timescale on which the perturbation is expected to undergo significant deformations.  An important property of pipe flows, first noted by Daneri and Sz\'{e}kelyhidi Jr.~in~\cite{DaneriSzekelyhidi17} and utilized crucially by Isett~\cite{Isett2018} towards the proof of Onsager's conjecture, is that the length of time for which pipe flows written in Lagrangian coordinates remain approximately stationary solutions to Euler depends \textit{only} on the Lipschitz norm of the transport velocity $v_q$ and not the Lipschitz norms of the original (undeformed) pipe flow. However, the timescale under which pipe flows transported by an intermittent velocity field remain coherent is space-time dependent, in contrast to previous convex integration schemes in which the timescale was uniform across $\mathbb{R}\times\mathbb{T}^3$. As such, we will need to introduce space-time cut-offs $\psi_{i,q}$ in order to determine the local turnover time. In particular, the cut-off $\psi_{i,q}$ will be defined such that
\begin{align}
\norm{\nabla v_q}_{L^{\infty}(\supp \psi_{i,q})}\lesssim \delta_q^{\sfrac 12}\lambda_q \Gamma_{q+1}^{i} := \tau_q^{-1}\Gamma_{q+1}^i\,.\label{e:local_turnover}
\end{align}
With such cut-offs defined, we then define in addition a family of temporal cut-offs $\chi_{i,k,q}$ which will be used to restrict the timespan of the intermittent pipe flows in terms of the local turnover. Each cut-off function $\chi_{i,k,q}$ will have temporal support contained in an interval of length
\begin{equation}
\label{e:tau_choice}
\tau_q \Gamma_{q+1}^{-i}.
\end{equation}

It should be noted that we will design the cut-offs so that we can deduce much more on its support than \eqref{e:local_turnover}. Since the material derivative $D_{t,q}:=\partial_t+v_q\cdot\nabla$ will play an important role, we will require estimates involving material derivatives $D_{t,q}^N$ of very high order.\footnote{The loss of material derivative in the transport error means that to produce solutions with regularity approaching $\dot{H}^{\frac{1}{2}}$, we have to propagate material derivative estimates of arbitrarily high order on the stress.} We expect the cost of a material derivative to be related to the turnover time, which itself is local in nature. As such, high order material derivative estimates will be done on the support of the cut-off functions and will be of the form
\[\norm{\psi_{i,q}D_{t,q}^N \RR\qnp}_{L^r}\,.\]

In addition to the family of cut-offs $\psi_{i,q}$ and $\chi_{i,k,q}$, we will also require stress cut-offs $\omega_{i,j,q,n,p}$ which determine the local size of the Reynolds stress errors $\RR\qnp$; in particular $\omega_{i,j,q,n,p}$ will be defined such that
\begin{align}
\left\| \nabla^M \RR\qnp \right\|_{L^\infty(\supp \omega_{i,j,q,n,p})} \leq \delta_{q+1,n,p} \Gamma_{q+1}^{2j} \lambda\qnp^M \,. \label{eq:product:intro:stress}
\end{align}
Previous intermittent convex integration schemes have managed to successfully cancel intermittent stress terms with much simpler stress cutoff functions than the ones we use.  However, mitigating the loss of spatial derivative in the oscillation error means that we have to propagate sharp spatial derivative estimates of arbitrarily high order on the stress in order to produce solutions with regularity approaching $\dot{H}^\frac{1}{2}$.  Due to this requirement, we then have to estimate the \emph{second} derivative (and higher) of the stress cutoff function
$$ \left\| \nabla^2 \left(  \omega^2\left(\RR\qnp\right) \right) \right\|_{L^1} \,,  $$
which in turn necessitates bounding the local $L^2$ norm of $\nabla \RR\qnp$ due to the term
$$  \left\|  \left(\nabla^2 (\omega^2 )\right) \left(\RR\qnp\right) \;\left|\nabla\RR\qnp\right|^2 \right\|_{L^1}\, .  $$
Given inductive estimates about the derivatives of $\RR_{q}$ \emph{only in} $L^1$ which have not been upgraded to $L^p$ for $p>1$, this term will obey a fatally weak estimate, which is why we must estimate $\RR\qnp$ in $L^\infty$ as in \eqref{eq:product:intro:stress}.

\subsubsection{Checkerboard cut-offs}\label{ss:checkerboard:heuristics}
As mentioned in the discussion of intermittent pipe flows, we must prevent pipes originating from different Lagrangian coordinate systems from intersecting. The first step is to reduce the complexity of this problem by restricting the size of the spatial domain on which intersections must be prevented. Towards this end, consider the maximum frequency of the original stress $\RR_q=\RR_{q,0}$, or any of the higher order stresses $\RR_{q,n}$ for $n\geq 1$.  We may write these frequencies as $\lambda_{q+1}r_1$ for $\lambda_q \lambda_{q+1}^{-1}\leq r_1<1$. We then decompose $\RR\qn$ using a checkerboard partition of unity comprised of bump functions which follow the flow of $v_q$ and have support of diameter $\left(\lambda_{q+1}r_1\right)^{-1}$.  These two properties ensure that we have \emph{preserved} the derivative bounds on $\RR\qn$.  Thus, we fix the set $\Omega$ to be the support of an individual checkerboard cutoff function in this partition of unity at a fixed time, cf.~\eqref{eq:Omega:diameter:alt}. 

Suppose furthermore that $\Omega$ is inhabited by disjoint sets of deformed intermittent pipe flows which are periodized to spatial scales no finer than $\left(\lambda_{q+1}r_2\right)^{-1}$ for $0<r_1<r_2<1$.  In practice, $r_2$ will be $r_{q+1,n}$, where $r_{q+1,n}$ is the amount of intermittency used in the pipes which comprise the perturbation $w_{q+1,n}$ which is used to correct $\RR\qn$. The pipes which already inhabit $\Omega$ may first be from previous generations of perturbations $w_{q+1,n'}$ for $n'<n$, in which case they are periodized to spatial scales much broader than $\left(\lambda_{q+1}r_2\right)^{-1}$, or from an overlapping checkerboard cutoff function used to decompose $\RR\qn$ on which a placement of pipes periodized to spatial scale $\left(\lambda_{q+1}r_2\right)^{-1}$ has already been chosen.  In either case, these pipes will have been deformed by the velocity field $v_q$ on the time-scale given by the inverse of the local Lipschitz norm.  We represent the support of these deformed pipe flows in terms of axes $\{A_i\}_{i\in\mathcal{I}}$ around which the pipes $\{P_i\}_{i\in\mathcal{I}}$ are concentrated to thickness $\lambda_{q+1}^{-1}$ (recall from Section~\ref{ss:concentratedpipes} that all intermittent pipe flows used in our scheme have this thickness). 

We will now explain that one may choose a new set of (straight, i.e. not deformed) intermittent pipe flows $\WW_{r_2,\lambda_{q+1}}$ periodized to scale $\left(\lambda_{q+1}r_2\right)^{-1}$ which are disjoint from each deformed pipe $P_i$ and \emph{on the support of} $\Omega$ and \emph{under appropriate restrictions on $r_1$ and $r_2$}.  Heuristically, this task becomes easier when $r_2$ is smaller, since this means both that we have more choices of placement for the new set, and there are less pipes $P_i$ inhabiting $\Omega$.  Conversely, this task becomes more difficult when $r_1$ is smaller, since then $\Omega$ is larger and will contain more pipes $P_i$.  We assume throughout that the deformations of the $P_i$'s are mild enough to preserve the expected length, curvature, and spacing bounds between neighboring pipes that arise from writing pipes in Lagrangian coordinates and flowing for a length of time which is strictly less than the inverse of the Lipschitz norm of the velocity field.

First, we can estimate the cardinality of the set $\mathcal{I}$ (which indexes the axes $A_i$ and pipes $P_i$) from above by $ r_2^2 r_1^{-2}$.  To understand this bound, first note that if we had \emph{straight} pipes $P_i$ periodized to scale $\left(\lambda_{q+1}r_2\right)^{-1}$ inhabiting a \emph{cube} of side length $\left(\lambda_{q+1}r_1\right)^{-1}$, this bound would hold.  Using the fact that our deformed pipes obey similar length, curvature, and spacing bounds as straight pipes and that our set $\Omega$ can be considered as a subset of a cube with side length proportional to $\left(\lambda_{q+1}r_1\right)^{-1}$, the same bound will hold up to dimensional constants.  Secondly, by the intermittency of the desired set of new pipes, we have $r_2^{-2}$ choices for the placement of the new set, as indicated in Figure~\ref{fig:choices}.  

To finish the argument, we must estimate how many of these $r_2^{-2}$ choices would lead to non-empty intersections between the new pipes and any $P_i$.  To calculate this bound, we will imagine the placement of the new set of straight pipes as occurring on a two-dimensional plane which is perpendicular to the axes of the pipes.  After projecting each $P_i$ onto this two-dimensional plane, our task is to choose the intersection points of the new pipes with the plane so that the new pipes do not intersect the shadows of the $P_i$'s.  


\begin{figure}[h]
\centering
\includegraphics[width=\textwidth]{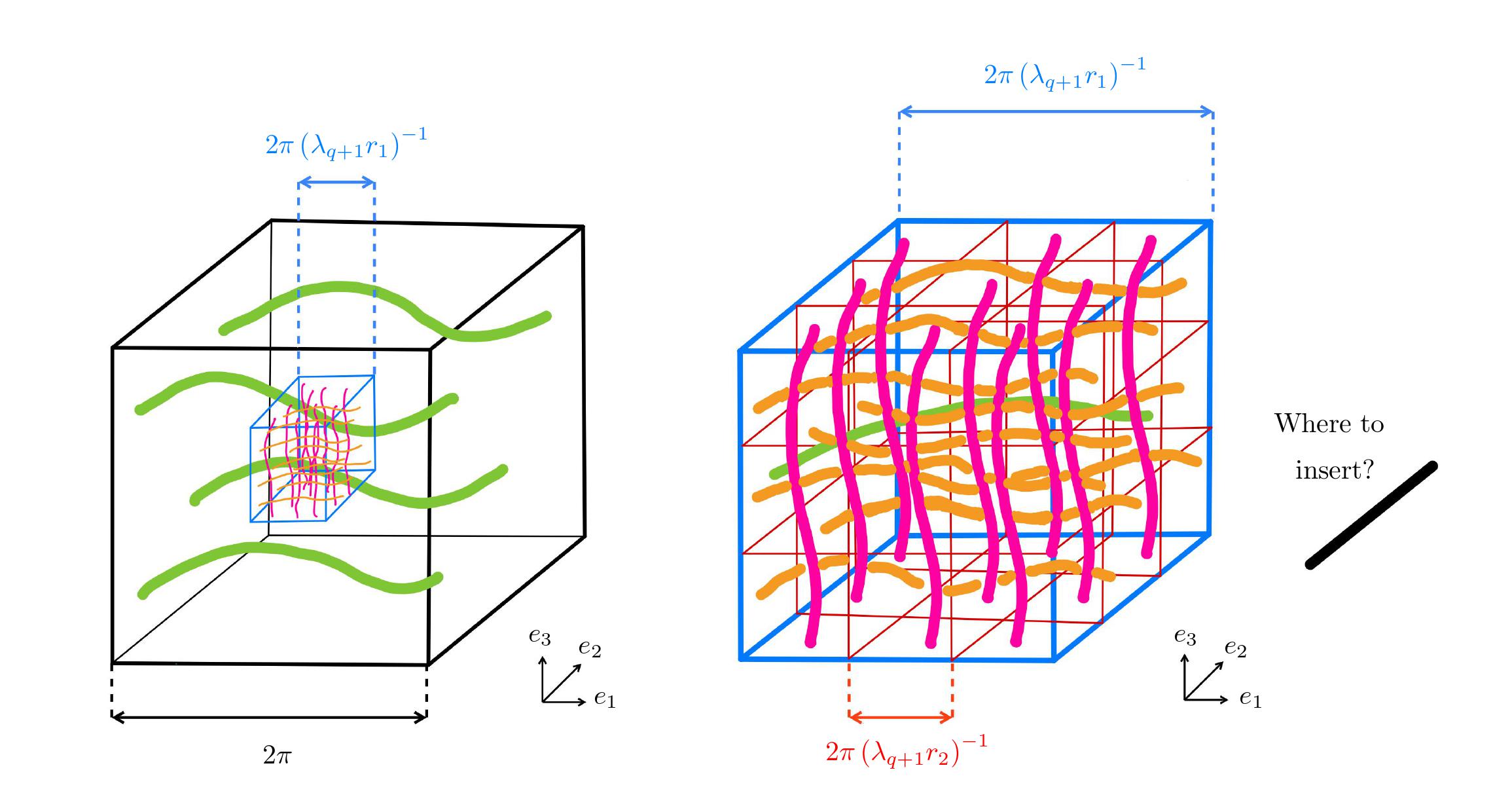}
\caption{\small In the figure on the left we display $\T^3$, in which we have: in green, a set of pipe flows (old generation, very sparse) that were deformed by $v_q$; in blue, the support of a cutoff function $ \zeta_{q,i,k,n,\vec{l}}$, whose diameter is $\approx (\lambda_{q+1} r_1)^{-1}$. Due to the sparseness, very few (if any!) of these green pipes intersect the blue region. The figure on the right further zooms into the blue region, to emphasize its contents. On the support of  $ \zeta_{q,i,k,n,\vec{l}}$ we have displayed two sets of deformed pipe flows, in pink and orange. These pipes flows were deformed also by $v_q$, from a nearby time at which they were straight and periodic at scale $(\lambda_{q+1}r_2)^{-1}$. At the current time, at which the above figure is considered, these pipe flows aren't quite periodic anymore, but they are close. The question now is: can we place a straight pipe flow, periodic at scale $(\lambda_{q+1}r_2)^{-1}$, whose axis is orthogonal to the front face of the blue box (pictured in black), and which does not intersect {\em any of the existing pipes in this region?} To see that this is possible, in Figure~\ref{fig:Placing} we estimate the area of shadows on this face of the cube.}
\label{fig:Existing_Pipes}
\end{figure}

\begin{figure}[h]
\centering
\includegraphics[width=0.65\textwidth]{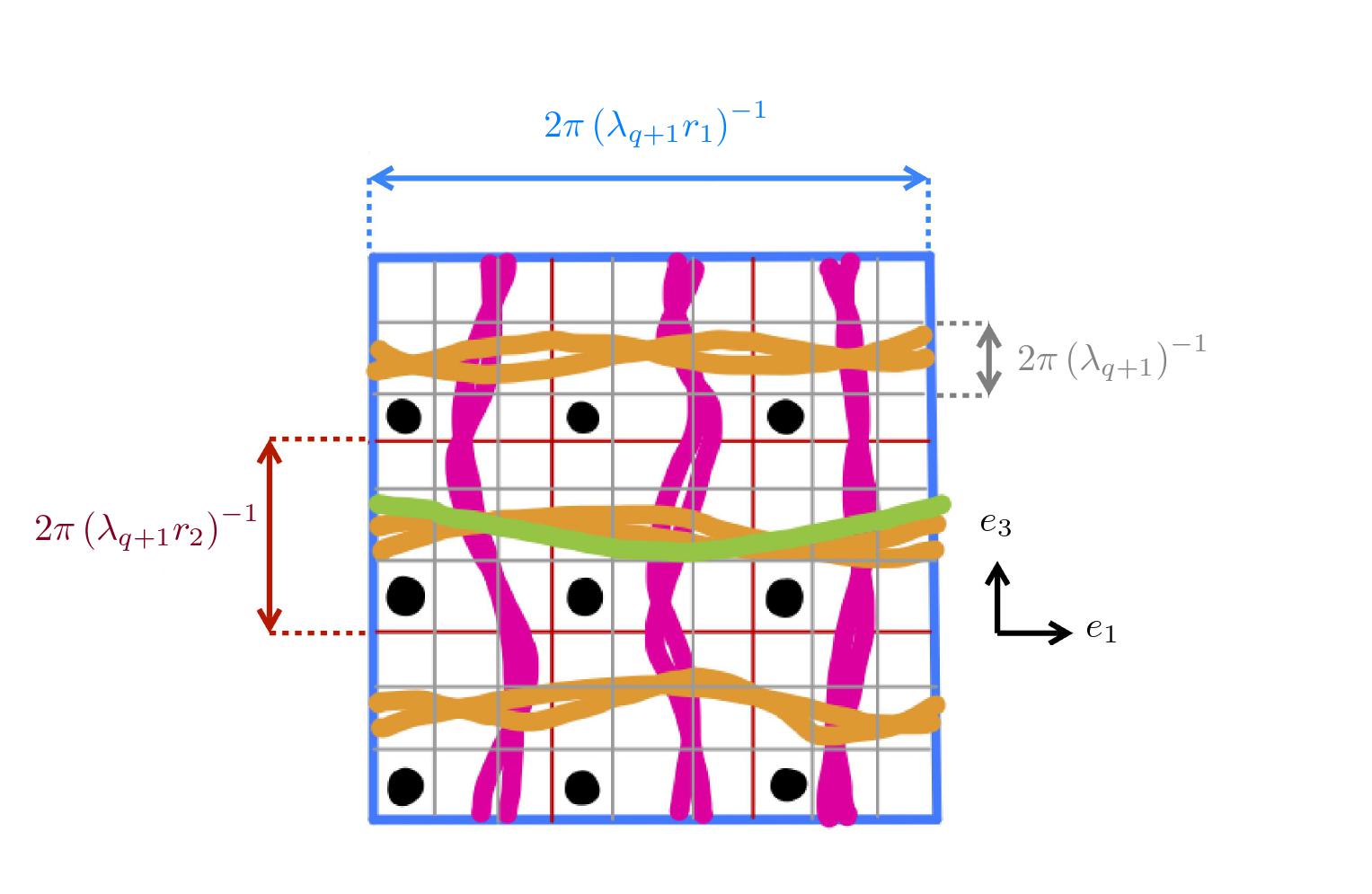}
\caption{\small As mentioned in the caption of Figure~\ref{fig:Existing_Pipes}, we consider the image on the right and take the projection of all pipes present in the blue box (green, pink, orange), onto the front face of the cube (parallel to the $e_3-e_1$ plane). Because these existing pipes were bent by $v_q$, the shadow does not consist of straight lines, and in fact the projections can overlap. By estimating the area of this projection, we see that if $r_2^4 \ll r_1^3$ then there is enough room left to insert a new pipe flow with orientation axis $e_2$ (represented by the black disks in the above figure), which will not intersect any of the projections of the existing pipes, and thus not intersect the existing pipes themselves.}
\label{fig:Placing}
\end{figure}

Given one of the deformed pipes $P_i$, since its thickness is $\lambda_{q+1}^{-1}$ and its length inside $\Omega$ is proportional to the diameter of $\Omega$, specifically $\left(\lambda_{q+1}r_1\right)^{-1}$, we may cover the shadow of $P_i$ on the plane with $\approx r_1^{-1}$ many balls of diameter $\lambda_{q+1}^{-1}$.  Covering all the $P_i$'s thus requires $  \approx  r_2^2 r_1^{-2}\cdot r_1^{-1}$ balls of diameter $\lambda_{q+1}^{-1}$.   Now, imagine the intersection of the new set of pipes with the plane.  Each choice of placement defines this intersection as essentially a set of balls of diameter $\approx\lambda_{q+1}^{-1}$ equally spaced at distance $\left(\lambda_{q+1}r_2\right)^{-1}$.  The intermittency ensures that there are $r_2^{-2}$ disjoint choices of placement, i.e. $r_2^{-2}$ disjoint sets of balls which represent the intersection of a particularly placed new set of pipes with the plane.  As long as 
$$   r_2^2 r_1^{-2} \cdot r_1^{-1} \ll r_2^{-2} \quad \iff r_2^4 \ll r_1^3 $$
there must exist at least one choice of placement which does not produce \emph{any} intersections between $\WW_{r_2,\lambda_{q+1}}$ and the $P_i$'s.  Notice that if $r_1$ is too small or if $r_2$ is too large, this inequality will not be satisfied, thus validating our previous heuristics about $r_1$ and $r_2$.  

To obey the relative intermittency inequality between $r_1$ and $r_2$ derived above for placements of new intermittent pipes on sets of a certain diameter, we will utilize cutoff functions
$$  \zeta_{q,i,k,n,\vec{l}} $$
which are defined using a variety of parameters. The index $q$ describes the stage of the convex integration scheme, while $i$ and $k$ refer to the velocity and temporal cutoffs defined above.  The parameter $n$ corresponds to a higher order stress $\RR_{q,n}$ and refers to its minimum frequency $\lambda_{q,n,0}$, quantifying the value of $(\lambda_{q+1}r_1)^{-1}$ and the diameter of the support as described earlier.  The parameter $\vec{l}=(l,w,h)$ depends on $q$ and $n$ and provides an enumeration of the (three-dimensional) checkerboard covering $\mathbb{T}^3$ at scale $\left(\lambda_{q,n,0}\right)^{-1}$. On the support of one of these checkerboard cutoff functions, we can inductively place pipes periodized to scale $\left(\lambda_{q+1}r_2\right)^{-1}=\lambda\qn^{-1}$ which are disjoint.  The checkerboard cutoff functions and the pipes themselves all follow the same velocity field, and so ensuring the disjointness at a single time slice is sufficient.

\subsubsection{Cumulative cut-off function}
Finally, the variety of cut-offs described above will be combined into the family of cut-offs
\begin{equation*}
\eta_{i,j,k,q,n,p,\vec{l}}:= \eta_{i,j,k,q,n,p}:= \chi_{i,k,q}\psi_{i,q} \omega_{i,j,q,n,p}\zeta_{q,i,k,n,\vec{l}},
\end{equation*}
which have timespans of $\tau_{q}\Gamma_{q+1}^{-i}$ and $L^2$ norms
\begin{equation}\label{e:eta_intro_est}
\norm{\eta_{i,j,k,q,n,p,\vec{l}}}_{L^2}\lesssim \Gamma^{-\frac{i}{2}}_{q+1}\cdot\Gamma^{-\frac{j}{2}}_{q+1}
\end{equation}
We will also require a cut-off $\eta_{i\pm,j\pm,k\pm,q,n,p,\vec{l}}$ which is defined to be $1$ on the support of $\eta_{i,j,k,q,n,\vec{l}}$ and satisfies the estimate
\begin{equation}\label{e:eta_pm_intro_est}
\norm{\eta_{i\pm,j\pm,k\pm,q,n,\vec{l}}}_{L^2}\lesssim  \Gamma^{-\frac{i}{2}}_{q+1}\cdot\Gamma^{-\frac{j}{2}}_{q+1}.
\end{equation}
We remark that \eqref{e:eta_intro_est} and \eqref{e:eta_pm_intro_est} are only heuristics (see Lemma~\ref{lemma:cumulative:cutoff:Lp} for the precise estimate). Designing the cut-offs turned out to be for the authors perhaps the most significant technical challenge of the paper. Their definition will be inductive and estimates involving them will involve several layers of induction.

\subsection{The perturbation}\label{ss:perturbation}
The intermittent pipe flows of Section \ref{ss:concentratedpipes}, the higher order stresses of Section \ref{ss:higherorderstresses}, and the cut-off functions of Section \ref{ss:cutoffs} provide the key ingredients in the construction of the perturbation 
$${w_{q+1} := \sum_{n=0}^\nmax \sum_{p=1}^\pmax w_{q+1,n,p}} := \sum_{n=0}^\nmax w_{q+1,n}.$$
In the above double sum, we will adopt the convention that $w_{q+1,0,p}=0$ unless $p=1$ to streamline notation. Let us emphasize that $w_{q+1}$ is constructed \emph{inductively} on $n$ for the following reason.  Each perturbation $  w_{q+1,n}=\sum_{p=1}^\pmax w_{q+1,n,p}$ will contribute error terms to all higher order stresses $\RR_{q,\nn,p}$ for $\nn>n$ and $1\leq p \leq \pmax$, and so $  \RR_{q,\nn}=\sum_{p=1}^\pmax \RR_{q,\nn,p}$ is not a well-defined object until each $w_{q+1,n'}$ has been constructed for all $n'<n$. For the purposes of the following heuristics, we will abbreviate the cutoff functions by $a_{n,p}$, and ignore summation over many of the indexes which parametrize the cutoff functions, as they are not necessary to understand the heuristic estimates.  We will freely use the heuristic that the cutoff functions allow us to use the $L^\infty_t H^1_x$ norm of $v_q$ to control terms (usually related to the turnover time) which previously required global Lipschitz bounds on $v_q$.

 Let $\Phi_{q,k}:\mathbb{R}\times\mathbb{T}^3\rightarrow\mathbb{T}^3$ be the solution to the transport equation
$$  \partial_t \Phi_{q,k} + v_q \cdot \nabla \Phi_{q,k} = 0  $$
with initial data given to be the identity at time $t_k=k\tau_q$. We mention that this definition is \emph{purely} heuristic, since as mentioned previously, the Lagrangian coordinate systems will have to be indexed by another parameter which encodes the fact that $\nabla v_q$ is spatially inhomogeneous.\footnote{The actual transport maps used in the proof are defined in Definition~\ref{def:transport:maps}.} For the time being let us ignore this issue.  Each map $\Phi_{q,k}$ has an effective timespan $\tau_{q}=(\delta_q^\frac{1}{2}\lambda_q)^{-1}$, at which point one resets the coordinates and defines a new transport map $\Phi_{q,k+1}$ starting from the identity.  Let $\WW_{q+1,n}$ denote the pipe flow with intermittency $\rqnperp$ periodized to scale $\left(\lambda_{q+1}\rqnperp\right)^{-1}$. The perturbation $w_{q+1,n,p}$ is then defined heuristically by
\begin{align*}
    w_{q+1,n,p}(x,t) &= \sum_{k} a_{n,p}\left( \RR_{q,n,p}(x,t) \right) \left( \nabla\Phi_{{q,k}}(x,t) \right)^{-1}(x,t) \WW_{q+1,n}(\Phi_{q,k}(x,t)).
\end{align*}
We have adopted the convention that $\RR_{q}=\RR_{q,0}=\RR_{q,0,1}$ and $\RR_{q,0,p}=0$ if $p\geq 2$. Composing with $\Phi_{q,k}$ adapts the pipe flows to the Lagrangian coordinate system associated to $v_q$ so that $(\nabla\Phi_{q,k})^{-1}\WW_{q+1,n}(\Phi_{q,k})$ is Lie-advected and remains divergence-free to leading order.  The perturbation $w_{q+1,n,p}$ has the following properties:
\begin{enumerate}[(1)]
    \item  The thickness (at unit scale) of the pipes on which $w_{q+1,n,p}$ is supported depends only on $q$ and $n$ and is quantified by 
    \begin{equation}\label{e:intro:rqnchoice}
    r_{q+1,n} = \left(\frac{\lambda_q}{\lambda_{q+1}}\right)^{\left(\frac{4}{5}\right)^{n+1}}.
    \end{equation}
    Thus, the perturbations become \textit{less} intermittent as $n$ increases, since the thickness of the pipes (periodized at unit scale) becomes larger as $n$ increases.  Notice that the maximum frequency of $\RR\qnp$ is $\lambda\qnp$ for $n\geq 1$ per \eqref{e:intro:rqnpsupport}, and $\lambda_q$ for $n=0$, while the minimum frequency of the intermittent pipe flow $\WW_{q+1,n}$ used to construct $w_{q+1,n,p}$ is $\lambda\qn$.  Referring back to Definition~\ref{d:parameters:2.3} and Definition~\ref{d:parameters:2.4}, we have that for $1\leq n \leq \nmax$ and $1\leq p \leq \pmax$,
    $$   \lambda\qnp = \lambda_{q,n,0}^{1-\frac{p}{\pmax}}\lambda_{q,n+1,0}^{\frac{p}{\pmax}} \leq \lambda_{q,n+1,0}=\lambda_q^{\ff^{n}\cdot\frac{5}{6}}\lambda_{q+1}^{1-\ff^{n}\cdot\frac{5}{6}} \ll  \lambda_q^{\left(\frac{4}{5}\right)^{n+1}} \lambda_{q+1}^{1-\left(\frac{4}{5}\right)^{n+1}} = \lambda\qn, $$
    which ensures that the low frequency portion of $w_{q+1,n,p}$ decouples from the high frequency intermittent pipe flow $\WW_{q+1,n}$.  For $n=0$, the maximum frequency of $\RR_{q,0}=\RR_q$ is $\lambda_q$, which is much less than $\lambda_{q,0}$ per Definition~\ref{d:parameters:2.3}.
    \item  The $L^2$ size of $w_{q+1,n,p}$ is equal to the square root of the $L^1$ norm of $\RR_{q,n,p}$, which in turn depends on the minimum frequency of $\RR_{q,n,p}$ and will be $\delta_{q+1,n,p}$, where we define $\delta_{q+1,0,p}=\delta_{q+1}$. For $n\geq 1$ and $1\leq p \leq \pmax$, we have from Definition~\ref{d:parameters:2.5} that
    \begin{equation*}
    \displaystyle\delta_{q+1,n,p}=\frac{\delta_{q+1}\lambda_{q}}{\lambda_{q,n,p-1}} \prod_{n'<n} f_{q,n'}.
    \end{equation*}
    \item  For $n\geq 1$, derivatives on the low frequency coefficient function of $w_{q+1,n,p}$ cost the maximum frequency of $\RR_{q,n,p}$, which is $\lambda\qnp$.  For $n=0$, $\RR_{q,0}=\RR_q$, so that each spatial derivative on the coefficient function of $w_{q+1,0}$ costs $\lambda_q$.
    \item  The transport error and Nash error created by the addition of $w_{q+1,n,p}$ are small enough to be absorbed into $\RR_{q+1}$ for every $n$ .
    \item  Per Definition~\ref{d:parameters:2.3}, the oscillation error which results from $w_{q+1,n,p}$ interacting with itself has minimum frequency 
$$\lambda\qn=\lambda_{q+1}\rqnperp = \lambda_q^{\left(\frac{4}{5}\right)^{n+1}}\lambda_{q+1}^{1-\left(\frac{4}{5}\right)^{n+1}}.$$
\end{enumerate}

\subsection{The Reynolds stress error and heuristic estimates}\label{ss:reynoldsheuristic}

Note that since the relation \eqref{e:euler_reynolds} is linear in the Reynolds stress, replacing $q$ with $q+1$, the right hand side can be split into three components:
\begin{equation}\label{e:Rey_Decomp}
\begin{split}
\div (w_{q+1}\otimes w_{q+1}+\mathring{R}_{q})\\ 
\partial_t w_{q+1} + v_{q} \cdot \nabla w_{q+1} \\
w_{q+1} \cdot \nabla v_{q}~,
\end{split}
\end{equation}
which we call the \emph{oscillation error}, \emph{transport error} and \emph{Nash error} respectively. 

\subsubsection{Type 1 oscillation error}
In this section, we sketch the heuristic estimates which justify the following principle: the low frequency, high amplitude errors arising from the self interaction of an intermittent pipe flow can be transferred to higher frequencies and smaller amplitudes through the higher order stresses and perturbations. We shall show that the following estimates are self-consistent and allow for the constructions of solutions approaching the regularity threshold $\dot{H}^\frac{1}{2}$:
\begin{equation}\label{e:intro:type1:1}
\left\| \nabla^M \RR_q \right\|_{L^1} \leq \delta_{q+1} \lambda_q^M
\end{equation}
\begin{equation}\label{e:intro:type1:2}
 \left\| \nabla^M\RR\qnp \right\|_{L^1} \leq \frac{\delta_{q+1}\lambda_{q}}{\lambda_{q,n,p-1}} \prod_{n'<n} f_{q,n'} \lambda\qnp^M = \delta_{q+1,n,p} \lambda\qnp^M \,.
\end{equation}
The higher order stress $\RR_{q,n,p}$ is defined using the spatial Littlewood-Paley projection operator
$$  \LPqnp := \mathbb{P}_{\left[\lambda_{q,n,p-1},\lambda\qnp\right)} = \Proj_{\geq \lambda_{q,n,p-1}} \Proj_{< \lambda\qnp}, $$
which projects onto the frequencies from \eqref{e:intro:rqnpsupport}.  We define $\RR_{q,n,p}$ as follows:
\begin{equation}\label{e:intro:heuristic:rqnpdefinition}
    \RR_{q,n,p} := \sum_{n'<n}\sum\limits_{p'=1}^\pmax  \div^{-1}\left( \nabla\left(a_{n',p'}^2 (\RR_{q,n',p'} ) \nabla\Phi_{q,k}^{-1} \otimes \nabla\Phi_{q,k}^{-T} \right) : \left(\LPqnp\left( \WW_{q+1,n'} \otimes \WW_{q+1,n'} \right)\right)(\Phi_{q,k}) \right).
\end{equation}

We pause here to point out an important consequence of this definition.  Let $n'$ be fixed, and consider the right side of the above equality. 
Then, due to the periodicity of $\WW_{q+1,n'}$ at  scale $(\lambda_{q+1} r_{q+1,n'})^{-1}$ we have\footnote{We denote by $\Proj_{\neq 0}$ the operator which subtracts from a function its mean in space.}
\begin{align*}
\WW_{q+1,n'}\otimes\WW_{q+1,n'} &= \mathbb{P}_{= 0} \left(\WW_{q+1,n'}\otimes\WW_{q+1,n'}\right) + \mathbb{P}_{\neq 0}\left(\WW_{q+1,n'}\otimes\WW_{q+1,n'}\right) \notag\\
&= \mathbb{P}_{= 0} \left(\WW_{q+1,n'}\otimes\WW_{q+1,n'}\right) + \mathbb{P}_{\geq \lambda_{q+1}r_{q+1,n'}}\left(\WW_{q+1,n'}\otimes\WW_{q+1,n'}\right).\notag
\end{align*}
For $n'\geq 1$, we have that
$$\lambda_{q+1}r_{q+1,n'}=\lambda_q^{\ff^{n'+1}}\lambda_{q+1}^{1-\ff^{n'+1}}\gg \lambda_q^{\ff^{n'}\cdot\frac{5}{6}}\lambda_{q+1}^{1-\ff^{n'}\cdot\frac{5}{6}} = \lambda_{q,n'+1,0}= \lambda_{q,n',\pmax} ,$$
where $\lambda_{q,n'+1,0}$ is the minimum frequency of $  \RR_{q,n'+1} = \sum_{p'=0}^\pmax \RR_{q,n'+1,p'}$, 
while for $n'=0$ we have that
$$  \lambda_{q+1}r_{q+1,0} = \lambda_{q,1} = \lambda_q^{\ff} \lambda_{q+1}^{1-\ff} = \lambda_{q,1,0},  $$
which is the minimum frequency of $\RR_{q,1}$.  Therefore, we have shown that the error terms arising from \emph{all} non-zero modes of $\WW_{q+1,n'}\otimes\WW_{q+1,n'}$ are accounted for in the higher order stresses $\RR_{q,\nn}$ for $\nn>n'$.  Thus, the higher order stresses created by the interaction of $w_{q+1,n'}$ will be absorbed into higher order stresses with \emph{strictly larger} values of $n$.

Now assuming that $\RR_{q,n',p'}$ and $w_{q+1,n',p'}$ are well-defined for all $n'<n$ and $1\leq p'\leq\pmax$ and using the heuristic estimates from the previous section for $w_{q+1,n',p'}$, we can estimate the component of $\RR_{q,n,p}$ coming from $w_{q+1,n',p'}$ by recalling \eqref{e:intro:heuristic:rqnpdefinition} and writing 
\begin{align*}
    \left\| \RR_{q,n,p} \right\|_{L^1} &\leq \sum_{n'<n} \frac{\delta_{q+1,n',p'}\lambda_{q,n',p'}}{\lambda\qnpminus} \notag\\
    &= \sum_{n'<n} \frac{\frac{\delta_{q+1}\lambda_{q}}{\lambda_{q,n',p'-1}} \prod_{n^{''}<n'} f_{q,n^{''}}{\lambda_{q,n',p'}}}{\lambda\qnpminus}  \\
    &\leq \sum_{n'<n} \frac{\delta_{q+1}\lambda_q}{\lambda\qnpminus}\prod_{n^{''} \leq n'} f_{q,n^{''}}\notag\\
    & \lesssim \frac{\delta_{q+1}\lambda_q}{\lambda\qnpminus}\prod_{n^{''} < n} f_{q,n^{''}} = \delta_{q+1,n,p} \,.
\end{align*}
The denominator comes from the gain induced by the combination of the inverse divergence and the Littlewood-Paley projector $\LPqnp$. The numerator is the amplitude of $\nabla |a_{n',p'} (\RR_{q,n',p'} ) |^2$, computed using the chain rule and the assumption \eqref{e:intro:type1:2} on $\nabla\RR_{q,n',p'}$. We have used that the $L^2$ norm of $\WW_{q+1,n'}$ is normalized to unit size.  Any derivatives on $\RR\qnp$ will cost $\lambda\qnp$, which is the maximum frequency in the Littlewood-Paley projector $\LPqnp$. Thus, all terms which will land in $\RR_{q,n,p}$ will satisfy the correct estimates \textit{given that $\RR_{q,n',p'}$ satisfies the correct estimates for $n'<n$ and $1\leq p'\leq\pmax$}. Since $\RR_q=:\RR_{q,0}$ satisfies the inductive assumptions, we can initiate this iteration at level $n=0$ while satisfying \eqref{e:intro:type1:1}. 

Now that $\RR_{q,n,p}$ satisfies the appropriate estimates, we can correct it with a perturbation $w_{q+1,n,p}$ as described in the previous section.  As before, since $\WW_{q+1,n}$ has minimum frequency
\begin{equation*}
\lambda\qn=\lambda_{q+1}\rqnperp = \lambda_q^{\left(\frac{4}{5}\right)^{n+1}}\lambda_{q+1}^{1-\left(\frac{4}{5}\right)^{n+1}} \gg \lambda_q^{\left(\frac{4}{5}\right)^{n}\cdot\frac{5}{6}}\lambda_{q+1}^{1-\left(\frac{4}{5}\right)^{n}\cdot\frac{5}{6}} = \lambda_{q,n+1,0} \,,
\end{equation*}
and the minimum frequency in $\RR_{q,n+1}$ is $\lambda_{q,n+1,0}$, \textit{every error term resulting from the self interaction of $w_{q+1,n,p}$ will be absorbed into higher order stresses $\RR_{q,\nn}$ for $\nn>n$.}  Therefore, we can induct on $n$ to add a sequence of perturbations $  w_{q+1,n}=\sum_{p=1}^\pmax w_{q+1,n,p}$ such that all nonlinear error terms are canceled by subsequent perturbations.  Upon reaching $\nmax$ and recalling \eqref{e:fqn:inequality}, we can estimate the final nonlinear error term by
\begin{align}
     \frac{\delta_{q+1}\lambda_q}{\lambda_{q+1}r_{q+1,\nmax}} \prod_{n^{'}<\nmax} f_{q,n^{'}} \leq \delta_{q+2} 
    &  \impliedby \delta_{q+1}\left(\frac{\lambda_q}{\lambda_{q+1}}\right)^{1-\left(\frac{4}{5}\right)^{\nmax+1}-\frac{1}{\pmax}} \leq \delta_{q+2} \notag\\
    &  \iff \lambda_{q+1}^{-2\beta}\lambda_{q+1}^{\left(\frac{1}{b}-1\right)\left(1-\left(\frac{4}{5}\right)^{\nmax+1}-\frac{1}{\pmax}\right)} \leq \lambda_{q+1}^{-2\beta b} \notag \\
    &  \iff 2\beta b(b-1) \leq (b-1)\left(1-\left(\frac{4}{5}\right)^{\nmax+1}-\frac{1}{\pmax}\right) \notag\\
    &  \iff \beta \leq \frac{1}{2b}\left(1-\left(\frac{4}{5}\right)^{\nmax+1}-\frac{1}{\pmax}\right).\notag
\end{align}
Choosing $b$ to be close to $1$ and $\nmax$ and $\pmax$ sufficiently large shows that these error terms are commensurate with $\dot{H}^{\frac{1}{2}-}$ regularity.
 
\subsubsection{Type 2 oscillation error}
We now consider the second type of oscillation error, which would arise as a result of two \textit{distinct} pipes intersecting and thus serves no purpose in the cancellation of stresses.  Beginning with $\RR_q=\RR_{q,0}$, we have that every derivative on $\RR_{q,0}$ costs $\lambda_q$.  Therefore, we may decompose $\RR_{q,0}$ using a checkerboard partition of unity at scale $\lambda_q^{-1}$. Referring back to the discussion of the checkerboard cutoff functions, this sets the value of $r_1$ to be $ \lambda_q \lambda_{q+1}^{-1}$.  Now, suppose that on a single square of this checkerboard, we have placed a set of intermittent pipe flows $\WW_{q+1,0}$ which are periodized to scale $\left(\lambda_{q+1}r_{q+1,0}\right)^{-1}$.  After flowing the pipes and the checkerboard square by $v_q$ for a short length of time\footnote{The length of time is equal to the local Lipschitz norm of $v_q$ on the support of the cutoff $\psi_{i,q}$, given by the time-cutoff hidden in $a_{n,p}$.}, we must place a new set of pipes $\WW'_{q+1,0}$ which are disjoint from the flowed pipes $\WW_{q+1,0}$.  Given the choice of $r_1$, this will be possible provided that
\begin{align}
r_{q+1,0} = r_2 \ll r_1^\frac{3}{4} \,.
\label{eq:blah:blah:blah:blah}
\end{align}
Thus, the \emph{minimum} amount of intermittency needed to successfully place disjoint sets of intermittent pipes is $ \left( \lambda_q \lambda_{q+1}^{-1}\right)^\frac{3}{4}$.  Per Definition~\ref{d:parameters:2.3}, our choice of $r_{q+1,0}$ is $\left( \lambda_q \lambda_{q+1}^{-1}\right)^\frac{4}{5}$, which is then sufficiently small.

Let us now assume that we have successfully corrected $\RR_{q,n'}$ for $n'<n$, and that we wish to correct $  \RR_{q,n}=\sum_{p=1}^\pmax \RR\qnp$ with a perturbation $  w_{q+1,n}=\sum_{p=1}^\pmax w_{q+1,n,p}$.  First, we recall that
$$ \left\| \nabla^M \RR\qnp \right\|_{L^1} \lesssim \delta_{q+1,n,p} \lambda\qnp^M. $$
Therefore, we can multiply $\RR\qnp$ by a checkerboard partition of unity at scale $\lambda_{q,n,0}^{-1} \gg \lambda\qnp^{-1}$ while preserving these bounds.  We must choose values of $r_1$ and $r_2$, as in Section~\ref{ss:checkerboard:heuristics}. Since for $n\geq 2$
$$ \lambda_{q+1}r_{1} = \lambda_{q,n,0} = \lambda_q^{\left(\frac{4}{5}\right)^{n-1}\cdot\frac{5}{6}} \lambda_{q+1}^{1-\left(\frac{4}{5}\right)^{n-1}\cdot\frac{5}{6}} = \lambda_{q+1} \cdot \left(\frac{\lambda_q}{\lambda_{q+1}}\right)^{\left(\frac{4}{5}\right)^{n-1}\cdot\frac{5}{6}}, $$
and for $n=1$
$$  \lambda_{q,1,0} = \lambda_q^{\frac{4}{5}}\lambda_{q+1}^{\frac{1}{5}} \gg \lambda_{q+1}\cdot \left(\frac{\lambda_q}{\lambda_{q+1}}\right)^{\left(\frac{4}{5}\right)^{1-1}\cdot\frac{5}{6}}, $$
we have that for all $n\geq 1$
$$r_1 \geq \left(\frac{\lambda_q}{\lambda_{q+1}}\right)^{\left(\frac{4}{5}\right)^{n-1}\cdot\frac{5}{6}}.$$
Recall that $\RR\qnp$ will be corrected by $w_{q+1,n,p}$, which is constructed using intermittent pipe flows $\WW_{q+1,n}$ with intermittency
$$  r_{q+1,n} = \left(\frac{\lambda_q}{\lambda_{q+1}}\right)^{\left(\frac{4}{5}\right)^{n+1}} = r_2.  $$
Thus in order to succeed in placing pipes $\WW_{q+1,n}$ which avoid both previous generations of pipes, which are periodized to scales rougher than $\WW_{q+1,n}$, and pipes from the same generation on overlapping cutoff functions, we must ensure that
\begin{align*}
r_2 &\ll r_1^\frac{3}{4} \notag \\
\iff \left(\frac{\lambda_q}{\lambda_{q+1}}\right)^{\left(\frac{4}{5}\right)^{n+1}} &\ll \left(\frac{\lambda_q}{\lambda_{q+1}}\right)^{\left(\frac{4}{5}\right)^{n-1}\cdot\frac{5}{6}\cdot\frac{3}{4}} \notag\\
\iff \left(\frac{4}{5}\right)^{n-1}\cdot\frac{5}{6}\cdot\frac{3}{4} &< \left(\frac{4}{5}\right)^{n+1} \notag\\
\iff \frac{1}{2} &< \left(\frac{4}{5}\right)^3 = \frac{64}{125}.
\end{align*}
So our choice of $r_{q+1,n}$ is sufficient to ensure that we can successfully place intermittent pipe flows when constructing $w_{q+1,n,p}$ which are disjoint from all other pipe flows from either previous generations ($n'<n$) or the same generation (the same value of $n$).

\subsubsection{Nash and transport errors}
The heuristic for the Nash and transport errors is that our choice of $\rqnperp$ provides much more intermittency than is needed to ensure that linear errors arising from $w_{q+1,n,p}$ can be absorbed into $\RR_{q+1}$.\footnote{One may verify that in three dimensions, the minimum amount of intermittency needed to absorb the Nash and transport errors arising from $w_{q+1,0}$ into $\RR_{q+1}$ at regularity approaching $\dot{H}^\frac{1}{2}$ is $r_{q+1,0}= \lambda_q^\frac{1}{2}  \lambda_{q+1}^{-\frac{1}{2}}$.  In general, one can further verify that  given errors supported at frequency $\lambda_q^\alpha\lambda_{q+1}^{1-\alpha}$, one could correct them using intermittent pipe flows with minimum frequency $\lambda_q^\frac{\alpha}{2}\lambda_{q+1}^{1-\frac{\alpha}{2}}$ while absorbing the resulting Nash and transport errors into $\RR_{q+1}$. One should compare this with \eqref{eq:blah:blah:blah:blah}, which shows that the placement technique requires more intermittency, which at level $n=0$ corresponds to $\lambda_q^\frac{3}{4}  \lambda_{q+1}^{-\frac{3}{4}}$.}  In other words, the Type 2 oscillation errors required much more intermittency than the Nash and transport errors will. 

Let us start with the Nash error arising from the addition of $w_{q+1,0,1}$, which is designed to correct $\RR_q$.  Using decoupling, the cost of a derivative on $\WW_{q+1,0}$  being $\lambda_{q+1}$ (so that inverting the divergence gains a factor of $\lambda_{q+1}$), the size of $\nabla v_q$ in $L^2$, and the $L^1$ size of $\WW_{q+1,n}$ being $r_{q+1,0}$, the size of this error is
\begin{align}
\frac{1}{\lambda_{q+1}} \delta_{q+1}^{\sfrac{1}{2}} \delta_q^{\sfrac{1}{2}} \lambda_q r_{q+1,0} & = \frac{1}{\lambda_{q+1}} \delta_{q+1}^{\sfrac{1}{2}} \delta_q^{\sfrac{1}{2}} \lambda_q \left(\frac{\lambda_q}{\lambda_{q+1}}\right)^{\ff} . \notag
\end{align}
This is (much) less than $\delta_{q+2}$ since
\begin{align}
     \frac{\delta_{q+1}^{\sfrac{1}{2}}\delta_q^{\sfrac{1}{2}}\lambda_q^{\sfrac{3}{2}}}{\lambda_{q+1}^{\sfrac{3}{2}}}\leq \delta_{q+2}  
    &\iff \lambda_{q+1}^{-\beta} \lambda_{q+1}^{-\frac{\beta}{b}}\lambda_{q+1}^{\frac{1}{b}\cdot\frac{3}{2}}\lambda_{q+1}^{-\frac{3}{2}} \leq \lambda_{q+1}^{-2\beta b}  \nonumber \\
    &\iff 2\beta b^2 - \beta b - \beta \leq (b-1)\cdot\frac{3}{2} \nonumber \\
   &\iff \beta (2b+1) (b-1) \leq (b-1)\cdot\frac{3}{2}.
   \label{eq:Nash:transport:heuristic}
\end{align}
Choosing $b$ close to $1$ will make this error commensurate with $\dot{H}^{\frac{1}{2}-}$ regularity.

Let us now estimate the Nash error arising from the addition of $w_{q+1,n,p}$ for $n\geq 2$, given by
$$ \norm{ \div^{-1}\left(\left( a_{n,p} \nabla \Phi_{q,k}^{-1} \WW_{q+1,n}(\Phi_{q,k}) \right) \cdot \nabla v_q \right)}_{L^1}.   $$
Using again decoupling, the cost of a derivative on $\WW_{q+1,n}$  being $\lambda_{q+1}$ (so that inverting the divergence gains a factor of $\lambda_{q+1}$), the size of $\nabla v_q$ in $L^2$, the $L^1$ size of $\WW_{q+1,n}$ being $\rqnperp$, and \eqref{e:fqn:inequality}, we have that for $n\geq 2$, the size of this error is
\begin{align*}
    \frac{1}{\lambda_{q+1}} \cdot \delta_{q+1,n,p}^\frac{1}{2} \rqnperp \cdot \delta_q^\frac{1}{2}\lambda_q &\leq \frac{1}{\lambda_{q+1}} \cdot \delta_{q+1,n,1}^\frac{1}{2} \rqnperp \cdot \delta_q^\frac{1}{2}\lambda_q  \notag \\
    & = \frac{1}{\lambda_{q+1}} \left( \frac{\delta_{q+1}\lambda_q}{\lambda_{q,n,0}} \right)^\frac{1}{2}\left(\prod_{n'<n} f_{q,n'}\right)^\frac{1}{2} \left( \frac{\lambda_{q}}{\lambda_{q+1}} \right)^{\ff^{n+1}} \delta_q^\frac{1}{2}\lambda_q  \notag\\
    & \leq \frac{1}{\lambda_{q+1}} \left( \frac{\delta_{q+1}\lambda_q}{\lambda_q^{\ff^{n-1}\cdot\frac{5}{6}}\lambda_{q+1}^{1-\ff^{n-1}\cdot\frac{5}{6}}} \right)^\frac{1}{2} \left( \frac{\lambda_{q}}{\lambda_{q+1}} \right)^{\ff^{n+1}-\frac{1}{2\pmax}} \delta_q^\frac{1}{2}\lambda_q \,.
\end{align*}
Since
$$ \frac{1}{2\pmax} + \frac{1}{2}\cdot\frac{5}{6}\cdot\left(\frac{4}{5}\right)^{n-1} < \left(\frac{4}{5}\right)^{n+1}  $$
independently of $n\geq 2$ if $\pmax$ is sufficiently large, the Nash error will be smaller than $\delta_{q+2}$ based on the preceding estimates. Furthermore, one may check that $\delta_{q+1,1,1}^\frac{1}{2}r_{q+1,1}<\delta_{q+1,2,1}^\frac{1}{2}r_{q+1,2},$ so that the Nash error arising from the addition of $w_{q+1,1,p}$ is also satisfactorily small for all $p$.

Now let us consider the transport error. The size of the transport error arising from the addition of $w_{q+1,n,p}$ is
\begin{align}\label{heuristic:transport}
    \norm{ \div^{-1} \left( (D_{t,q} a_{n,p}) \nabla \Phi_{q,l}^{-1} \WW_{q+1,n} \right) }_{L^1} 
    &\leq \frac{1}{\lambda_{q+1}} \tau_{q}^{-1} \delta_{q+1,n,p}^\frac{1}{2} \rqnperp \nonumber\\
    & = \frac{1}{\lambda_{q+1}} \cdot \delta_{q+1,n,p}^\frac{1}{2} \rqnperp \cdot \delta_q^\frac{1}{2}\lambda_q.
\end{align}
Thus, the transport error is the same size as the Nash error and is sufficiently small to be put into $\RR_{q+1}$.

\section{Inductive assumptions}\label{section:inductive:assumptions}

While in Section~\ref{sec:outline} we have outlined in broad terms the main steps in the proof of Theorem~\ref{thm:main}, along with the heuristics for some of the choices we have made in our proof, starting with the current section, we work with precise estimates. 

In Section~\ref{sec:inductive:general:notation} we introduce some of the notation used in the proof, such as the Euler-Reynolds system, the mollified velocity, velocity increments, material/directional derivatives, our notation for geometric upper bounds with tho different bases, and our notation for $\norm{\cdot}_{L^p}$. 

In Section~\ref{sec:inductive:estimates} we introduce the principal amplitude and frequency parameters used in proof (the precise definitions and the order of choosing these parameters is detailed in Section~\ref{sec:parameters:DEF}). Next, in Sections~\ref{sec:inductive:primary:velocity} and~\ref{sec:inductive:primary:stress} we state the {\em primary inductive assumptions} for the velocity, velocity increments, and Reynolds stress. These primary estimates hold on the support of previous generation velocity cutoff functions, which are inductively assumed to satisfy a number of properties, listed in Section~\ref{sec:cutoff:inductive}. Lastly, in Section~\ref{sec:inductive:secondary:velocity} we list a number of bounds for the velocity increments and mollified velocities, which involve all possible combinations of space and material derivatives, up to a certain order. These bounds are technical in nature, and should be ignored at a first reading; their sole purpose is to allow us to bound commutators between $D^n$ and $D_{t,q}^m$ for very high values of $n$ and $m$.  

In conclusion, in Section~\ref{sec:proof:of:main:theorem} we show that if we are able to propagate the previously stated inductive estimates from step $q$ to step $q+1$, for every $q\geq 0$, then  Theorem~\ref{thm:main} follows. At the end of the section we discuss the modifications to the proof which would be necessary in order to obtain other types of flexibility statements.

\subsection{General notations}
\label{sec:inductive:general:notation}
As is standard in convex integration schemes for the Euler system~\cite{DeLellisSzekelyhidi09}, we introduce the Euler-Reynolds system for the unknowns $(v_q,\RR_q)$:
\begin{subequations}
\label{eq:Euler:Reynolds:again}
\begin{align}
\partial_t v_q + \div(v_q\otimes v_q) +\nabla p_{q} &= \div \RR_{q}  \\
\div v_q &= 0.
\end{align}
\end{subequations}
Here and throughout the paper, the pressure $p_q$ is uniquely defined by solving $\Delta p_q = \div \div (\RR_q - v_q\otimes v_q)$, with $\int_{\T^3} p_q dx = 0$. 

In order to avoid the usual derivative-loss issue in convex integration schemes, for $q\geq 0$ we use the space-time mollification operator defined in Section~\ref{sec:mollifiers:Fourier} -- equation~\eqref{mollifier:operators}, to smoothen out the velocity field $v_q$ as:
\begin{align}\label{vlq}
    \vlq := \Pqxt v_q 
    \,.
\end{align}
In particular, we note that spatial mollification is performed at scale $\tilde \lambda_q^{-1}$ (which is just slightly smaller than $\lambda_q^{-1}$), while temporal mollification is at scale $\tilde \tau_{q-1}$ (which is a lot smaller than $\tau_{q-1}$). 

Next, for all $q \geq 1$, define
\begin{align}\label{eq:cutoffs:wu}
    w_{q}:=v_{q}-\vlqminus, \qquad u_q:= \vlq - \vlqminus.
\end{align}
For consistency of notation, define $w_0 = v_0$ and $u_0 = v_{\ell_0}$.
Note that
\begin{align}
u_q = \Pqxt w_q   + (\Pqxt \vlqminus - \vlqminus)
\label{inductive:velocity:frequency}
\end{align}
so that we may morally think that $u_q = w_q + $ a small error term (the smallness of this error term will be ensured by choosing a mollifier with a large number of vanishing moments, cf.~\eqref{eq:phi}).

We   use the following notation for the material derivative corresponding to the vector field $\vlq$:
\begin{align}\label{eq:cutoffs:dtq}
    D_{t,q} := \partial_t + \vlq \cdot \nabla.
\end{align}
With this notation, we have that
\begin{align}\label{eq:cutoffs:dtqdtq-1}
  D_{t,q} = D_{t,q-1} + u_q \cdot \nabla  .
\end{align}
We also introduce the directional derivatives
\begin{align}
\label{eq:Dq:definition}
D_q := u_{q} \cdot \nabla
\end{align}
which allow us to transfer information between $D_{t,q-1}$ and $D_{t,q}$ via $D_{t,q} = D_{t,q-1} + D_q$.

\begin{remark}[\textbf{Geometric upper bounds with two bases}]
\label{rem:min:max:exponents}
If for a sequence of numbers $\{ a_n\}_{n\geq 0}$, and for two parameters $0< \lambda < \Lambda$ we have the bounds 
\begin{align*}
a_n &\leq \lambda^n, \quad \mbox{for all} \quad n \leq N_* \\
a_n &\leq \lambda^{N_*} \Lambda^{n-N_*} \quad \mbox{for all}\quad n>N_*,
\end{align*} 
for some $N_* \geq 1$, we will abbreviate these bounds as 
\begin{align*}
a_n \leq \MM{n,N_*,\lambda,\Lambda} \,,
\end{align*}
where we define 
\begin{align}
\MM{n,N_*,\lambda,\Lambda} := \lambda^{\min\{n,N_*\}} \Lambda^{\max\{n-N_*,0\}}   
\label{eq:scripty:M:def} 
\end{align}
for all $n\geq 0$. The first entry of $\MM{\cdot,\cdot,\lambda,\Lambda}$ measures the index in the sequence (typically number of derivatives considered) and the second entry determines the index after which the base of the geometric bound changes from $\lambda$ to $\Lambda$.
This notation has the following consequence, which will be used throughout the paper:  if $1 \leq \lambda \leq \Lambda$,  then 
\begin{align}
\MM{a,N_*,\lambda,\Lambda} \MM{b,N_*,\lambda,\Lambda} \leq \MM{a+b,N_*,\lambda,\Lambda}.
\label{eq:min:max:exponents:prod}
\end{align}
When either $a$ or $b$ are larger than $N_*$ the above inequality creates a loss; for $a+b\leq N_*$, it is an equality.
\end{remark}

\begin{remark}[\textbf{Norms are uniform in time}]\label{rem:norms:are:uniform:inductive}
Throughout this section, and the remainder of the paper, in order to abbreviate notation we shall use the notation $\norm{f}_{L^p}$ to denote $\norm{f}_{L^\infty_t (L^p(\T^3))}$. That is, all $L^p$ norms stand for {\em $L^p$ norms in  space, uniformly in time}. Similarly, when we wish to emphasize a set dependence of an $L^p$ norm, we write $\norm{f}_{L^p(\Omega)}$, for some space-time set $\Omega \subset \R \times \T^3$, to stand for $\norm{{\mathbf{1}}_{\Omega}\; f}_{L^\infty_t (L^p(\T^3))}$.
\end{remark}

\subsection{Inductive estimates} 
\label{sec:inductive:estimates}

The proof is based on propagating estimates for solutions $(v_q,\RR_q)$ of the Euler-Reynolds system~\eqref{eq:Euler:Reynolds:again}, inductively for $q\geq 0$. In order to state these bounds, we first need to fix a number of parameters in terms of which these inductive estimates are stated. We start by picking a regularity exponent $\beta \in (\sfrac 13, \sfrac 12)$, and a super-exponential rate parameter $b \in (1,\sfrac 32)$ such that $2\beta b < 1$. In terms of this choice of $\beta$ and $b$, a number of additional parameters ($\nmax, \ldots \Nfin$) are fixed, whose precise definition is summarized for convenience in items~\eqref{item:nmax:pmax:DEF}--\eqref{item:Nfin:DEF} of Section~\ref{sec:parameters:DEF}. Note that at this point the parameter $a_*(\beta,b)$ from item~\eqref{item:astar:DEF} in Section~\ref{sec:parameters:DEF} is not yet fixed. With this choice, we then introduce the fundamental $q$-dependent frequency and amplitude parameters from Section~\ref{ss:q:dependent:parameters}. We state here for convenience the main $q$-dependent parameters defined in  \eqref{def:lambda:q:actual}, \eqref{def:delta:q:actual}, \eqref{def:Gamma:q:actual}, and \eqref{def:tau:q:actual}:
\begin{subequations}
\label{eq:Gamma:q+1:def:*}
\begin{alignat}{2}
&\lambda_q = 2^{ \big{\lceil} {(b^q) \log_2 a} \big{\rceil}} \approx \lambda_{q-1}^b \,, 
\qquad &&\delta_q =  \lambda_1^{\beta(b+1)} \lambda_q^{-2\beta} \,,   \\
&\tau_q^{-1} = \delta_q^{\sfrac 12} \lambda_q \Gamma_{q+1}^{\cstar+11} \,,
\qquad &&\Gamma_{q+1} =   \left(\frac{\lambda_{q+1}}{\lambda_q}\right)^{\eps_\Gamma}  \approx \lambda_q^{(b-1)\eps_\Gamma}\,,
\end{alignat}
\end{subequations}
where the constant $\cstar$ is defined by \eqref{eq:cstar:DEF}.
The $\approx$ symbols in \eqref{eq:Gamma:q+1:def:*} mean that the left side of the $\approx$ symbol lies between two (universal) constant multiples of the right side (see e.g.~\eqref{eq:lambda:q:to:q+1}).

\begin{remark}[\textbf{Usage of the symbol} \texorpdfstring{$\lesssim$}{lesssim} and choice of \texorpdfstring{$a_*$}{a*}]
Throughout the subsequent sections, we will make frequent use of the symbol $\lesssim$.  We emphasize that any implicit constants indicated by $\lesssim$ are only allowed to depend on the parameters defined in Section~\ref{sec:parameters:DEF}, items \eqref{item:beta:DEF}--\eqref{item:Nfin:DEF}. The implicit constants in $\les$ are however always independent of the parameters $a$ and $q$, which appear in \eqref{eq:Gamma:q+1:def:*}. This allows us at the end of the proof, cf.~item~\eqref{item:astar:DEF} in Section~\ref{sec:parameters:DEF} to choose $a_* (\beta,b)$ to be sufficiently large so that for all $a \geq a_*(\beta,b)$ and all $q\geq 0$, the parameter $\Gamma_{q+1}$ appearing in \eqref{eq:Gamma:q+1:def:*} is larger than all the implicit constants in $\les$ symbols encountered throughout the paper. That is, upon choosing $a_*$ sufficiently large, any inequality of the type $A \les B$ which appears in this manuscript, may be rewritten as $A \leq \Gamma_{q+1} B$, for any $q\geq 0$.
\end{remark}

In order to state the inductive assumptions we use four large integers, defined precisely in Section~\ref{sec:parameters:DEF}. For the moment it is just important to note that these fixed parameters are independent of $q$ and that they satisfy the ordering 
\begin{align}
1 \ll \NcutSmall \ll   \Nindvt \ll \Nindv \ll \Nfin 
\,.
\label{eq:Nind:first}
\end{align} 
The precise definitions of these integers, and the meaning of the $\ll$ symbols in \eqref{eq:Nind:first}, are given in \eqref{eq:Ncut:DEF}, \eqref{eq:Nind:t:DEF}, \eqref{eq:Nind:v:DEF}, and \eqref{eq:Nfin:DEF}. Roughly speaking, the role of these parameters is as follows:
\begin{itemize}
\item $\NcutSmall$ is the number of sharp material derivatives which are built into the velocity and stress cutoff functions.
\item $\Nindt$ is the number of sharp material derivatives propagated for velocities and stresses.  
\item $\Nindv$ is used to quantify  the number of (lossy) higher order space and time derivatives for velocities and stresses. 
\item $\Nfin$ is used to quantify the highest order derivatives appearing in the proof.
\end{itemize}
Next, we state the inductive assumptions for the velocity increments and stresses at various levels $q\geq 0$.
Throughout the section we frequently refer to the notation $\MM{n,N_*,\lambda,\Lambda}$ from \eqref{eq:scripty:M:def}.

\subsubsection{Primary inductive assumption for velocity increments}
\label{sec:inductive:primary:velocity}
We make $L^2$ inductive assumptions for $u_{q'}=v_{\ell_{q'}}-v_{\ell_{q'-1}}$ at levels $q'$ strictly below $q$. For all $0 \leq q' \leq q-1$ we assume that
\begin{align}
\norm{\psi_{i,q'-1} D^{n} D^m_{t,q'-1} u_{q'}}_{L^2}\leq \delta_{q'}^{\sfrac 12} \MM{n,2 \Nindv,\lambda_{q'},\tilde{\lambda}_{q'}} \MM{m, \Nindvt ,\Gamma_{q'}^i \tau_{q'-1}^{-1} , \tilde{\tau}_{q'-1}^{-1}}
\label{eq:inductive:assumption:derivative}
\end{align}
holds for all {$n+m \leq \Nfin$}. 

At level $q$, we assume that the velocity increment $w_q$ satisfies
\begin{align}
\norm{\psi_{i,q-1} D^{n} D^m_{t,q-1} w_{q}}_{L^2}
\leq 
\Gamma_q^{-1} \delta_{q}^{\sfrac 12} \lambda_{q}^{n} \MM{m,\Nindvt, \Gamma_{q}^{i-1} \tau_{q-1}^{-1} , \Gamma_q^{-1} \tilde{\tau}_{q-1}^{-1}} 
\label{eq:inductive:assumption:derivative:q}
\end{align}
for  $n,m\leq  7 \Nindv $.
Moreover, recalling from \eqref{eq:time:support} that $\supp_t f$ denotes the temporal support of a function $f$, we inductively assume that
\begin{align}
\supp_t (\RR_{q-1}) \subset [T_1,T_2] \quad \Rightarrow \quad \supp_t (w_{q})\subset 
\left[T_1 - (\lambda_{q-1} \delta_{q-1}^{\sfrac 12})^{-1},T_2 + (\lambda_{q-1} \delta_{q-1}^{\sfrac 12})^{-1} \right] 
\,.
\label{eq:perturbation:time:support}
\end{align}

\subsubsection{Inductive assumption for the stress}
\label{sec:inductive:primary:stress}
For the Reynolds stress $\RR_q$, we make $L^1$ inductive assumptions
\begin{align}
\norm{\psi_{i,q-1} D^n D_{t,q-1}^m \mathring R_{q}}_{L^1} 
\leq  \Gamma_q^\shaq  \delta_{q+1} \lambda_{q}^n \MM{m,\Nindt, \Gamma_{q}^{i+1} \tau_{q-1}^{-1} , \Gamma_q^{-1} \tilde \tau_{q-1}^{-1} }
\label{eq:Rq:inductive:assumption}
\end{align}
for all $0 \leq n, m\leq   3 \Nindv $.

\subsubsection{Inductive assumption for the previous generation velocity cutoff functions}
\label{sec:cutoff:inductive}
More assumptions are needed in relation to the previous velocity perturbations and old cutoffs functions. First, we assume that the velocity cutoff functions form a partition of unity for $q'\leq q-1$:
\begin{align}\label{eq:inductive:partition}
    \sum_{i\geq 0} \psi_{i,q'}^2 \equiv 1, \qquad \mbox{and} \qquad \psi_{i,q'}\psi_{i',q'}=0 \quad \textnormal{for}\quad|i-i'| \geq 2.
\end{align}
Second, we assume that there exists an $\imax = \imax(q) > 0$, which is bounded uniformly in $q$ as 
\begin{align}
\frac{b+1}{b-1} \leq \imax(q) \leq  \frac{4}{\eps_\Gamma (b-1)}
\,,
\label{eq:imax:upper:lower}
\end{align}
such that 
\begin{align}
\psi_{i,q'} &\equiv 0 \quad \mbox{for all} \quad i > \imax(q')\,,
\qquad \mbox{and} \qquad
\Gamma_{q'+1}^{\imax(q')} \leq \lambda_{q'}^{\sfrac 53}
\label{eq:imax:old}\,, 
\end{align}
for all $q'\leq q-1$. For  all $0 \leq q' \leq q-1$ and $0 \leq i \leq \imax$ we assume the following pointwise derivative bounds for the cutoff functions $\psi_{i,q'}$. For mixed space and material derivatives (recall the notation from \eqref{eq:cutoffs:dtq}) we assume that
\begin{align}
&\frac{{\mathbf{1}}_{\supp \psi_{i,q'}}}{\psi_{i,q'}^{1- (K+M)/\Nfin}} \left|\left(\prod_{l=1}^k D^{\alpha_l} D_{t,q'-1}^{\beta_l}\right) \psi_{i,q'}\right| 
\notag\\
&\qquad \les \MM{K,\NindLarge , \Gamma_{q'}  \lambda_{q'}, \Gamma_{q'} \tilde \lambda_{q'} }
\MM{M,\NindSmall - \NcutSmall,  \Gamma_{q'+1}^{i+3}  \tau_{q'-1}^{-1}, \Gamma_{q'+1}^{-1}  \tilde \tau_{q'}^{-1}}
\label{eq:sharp:Dt:psi:i:q:old}
\end{align}
for  $K,M,k \geq 0$ with $0 \leq K + M \leq \Nfin$, where $\alpha,\beta \in {\mathbb N}^k$ are such that $|\alpha|=K$ and $|\beta|=M$. 
Lastly, we consider mixtures of space, material, and directional derivatives (recall the notation from \eqref{eq:Dq:definition}). 
Then with $K,M, \alpha, \beta$ and $k$ as above, and with $N\geq 0$, we assume that
\begin{align}
&\frac{{\mathbf{1}}_{\supp \psi_{i,q'}}}{\psi_{i,q'}^{1- (N+K+M)/\Nfin}} \left| D^N \left( \prod_{l=1}^k D_{q'}^{\alpha_l} D_{t,q'-1}^{\beta_l}\right)  \psi_{i,q'} \right| \notag\\
&\qquad  \les  \MM{N,\NindLarge,  \Gamma_{q'}  \lambda_{q'},  \Gamma_{q'} \tilde \lambda_{q'}  } 
(\Gamma_{q'+1}^{i-\cstar} \tau_{q'}^{-1})^K 
\MM{M,\Nindt-\NcutSmall,  \Gamma_{q'+1}^{i+3}  \tau_{q'-1}^{-1}, \Gamma_{q'+1}^{-1}  \tilde \tau_{q'}^{-1}}
\label{eq:sharp:Dt:psi:i:q:mixed:old}
\end{align}
as long as $0 \leq N+ K+ M \leq \Nfin$.

In addition to the above pointwise estimates for the cutoff functions $\psi_{i,q'}$, we also assume that we have a good $L^1$ control. More precisely, we postulate  that
\begin{align}
\norm{\psi_{i,q'}}_{L^1}  \lesssim  \Gamma_{q'+1}^{-2i+\CLebesgue} \qquad \mbox{where} \qquad \CLebesgue = \frac{4+b}{b-1}  
\label{eq:psi:i:q:support:old}
\end{align}
holds for $0\leq q' \leq q-1$ and all $0\leq i \leq \imax(q')$.

\subsubsection{Secondary inductive assumptions for velocities}
\label{sec:inductive:secondary:velocity}
Next, for $0\leq q'\leq q-1$, $0 \leq i \leq \imax$, $k\geq 1$, $K,M\geq 0$, $\alpha, \beta \in \N^k$ with $|\alpha| = K$ and $|\beta| = M$, we assume that the following mixed space-and-material derivative bounds hold
\begin{align}
&\norm{ \Big( \prod_{i=1}^k D^{\alpha_i} D_{t,q'-1}^{\beta_i} \Big) u_{q'} }_{L^\infty(\supp \psi_{i,q'})} 
\notag\\
&\qquad \les 
(\Gamma_{q'+1}^{i+1} \delta_{q'}^{\sfrac 12}) \MM{K,2\Nindv,\Gamma_{q'} \lambda_{q'},\tilde \lambda_{q'}} \MM{M,\Nindt, \Gamma_{q'+1}^{i+3}  \tau_{q'-1}^{-1},  \Gamma_{q'+1}^{-1} \tilde \tau_{q'}^{-1}}
\label{eq:nasty:D:wq:old}
\end{align}
for $K+M  \leq \sfrac{3 \Nfin}{2} +1 $, 
\begin{align}
&\norm{ \Big( \prod_{i=1}^k D^{\alpha_i} D_{t,q'}^{\beta_i} \Big) D \vlqprime }_{L^\infty(\supp \psi_{i,q'})} 
\notag\\
&\qquad \les
(\Gamma_{q'+1}^{i+1} \delta_{q'}^{\sfrac 12}\tilde \lambda_{q'}  ) \MM{K , 2\Nindv,\Gamma_{q'} \lambda_{q'},\tilde \lambda_{q'}} \MM{M,\Nindt,\Gamma_{q'+1}^{i-\cstar} \tau_{q'}^{-1},   \Gamma_{q'+1}^{-1} \tilde \tau_{q'}^{-1}}
\label{eq:nasty:D:vq:old}
\end{align}
for $K+M  \leq \sfrac{3 \Nfin}{2}$, 
and
\begin{align}
&\norm{ \Big( \prod_{i=1}^k D^{\alpha_i} D_{t,q'}^{\beta_i} \Big)  \vlqprime }_{L^\infty(\supp \psi_{i,q'})} 
\notag\\
&\qquad \les
(\Gamma_{q'+1}^{i+1} \delta_{q'}^{\sfrac 12}\lambda_{q'}^2 ) \MM{K,2 \Nindv,\Gamma_{q'} \lambda_{q'},\tilde \lambda_{q'}} \MM{M,\Nindt,\Gamma_{q'+1}^{i-\cstar} \tau_{q'}^{-1},   \Gamma_{q'+1}^{-1} \tilde \tau_{q'}^{-1}}
\label{eq:bob:Dq':old}
\end{align}
for $K+M  \leq \sfrac{3 \Nfin}{2}  +1$. Additionally, for $N\geq 0$ we postulate that mixed space-material-directional derivatives satisfy 
\begin{subequations}
\label{eq:nasty:Dt:wq:WEAK:all}
\begin{align}
&\norm{D^N \Big( \prod_{i=1}^k D_{q'}^{\alpha_i} D_{t,q'-1}^{\beta_i} \Big) u_{q'} }_{L^\infty(\supp \psi_{i,q'})} \notag\\
&\qquad \les 
(\Gamma_{q'+1}^{i+1} \delta_{q'}^{\sfrac 12})^{K+1} \MM{N+K,2\Nindv,\Gamma_{q'} \lambda_{q'},\tilde \lambda_{q'}}   \MM{M,\Nindt,  \Gamma_{q'+1}^{i+3}  \tau_{q'-1}^{-1},  \Gamma_{q'+1}^{-1} \tilde \tau_{q'}^{-1}}   \label{eq:nasty:Dt:wq:old} \\
&\qquad \les 
(\Gamma_{q'+1}^{i+1} \delta_{q'}^{\sfrac 12}) \MM{N,2\Nindv,\Gamma_{q'} \lambda_{q'},\tilde \lambda_{q'}} (\Gamma_{q'+1}^{i-\cstar}  \tau_{q'}^{-1})^{K}  \MM{M,\Nindt, \Gamma_{q'+1}^{i+3}  \tau_{q'-1}^{-1},  \Gamma_{q'+1}^{-1} \tilde \tau_{q'}^{-1}}
\label{eq:nasty:Dt:wq:WEAK:old}
\end{align}
\end{subequations}
whenever  $N+ K+M \leq  \sfrac{3 \Nfin}{2} +1$.

\begin{remark}
\label{rem:D:t:q':orangutan}
Identity~\eqref{eq:cooper:1} shows that \eqref{eq:nasty:Dt:wq:WEAK:old} automatically implies the bound
\begin{align}
&\norm{D^N  D_{t,q'}^{M}  u_{q'} }_{L^\infty(\supp \psi_{i,q'})} \notag\\
&\qquad \les 
(\Gamma_{q'+1}^{i+1} \delta_{q'}^{\sfrac 12}) \MM{N,2\Nindv,\Gamma_{q'} \lambda_{q'},\tilde \lambda_{q'}}
\MM{M,\Nindt, \Gamma_{q'+1}^{i-\cstar}  \tau_{q'}^{-1},  \Gamma_{q'+1}^{-1} \tilde \tau_{q'}^{-1}}
\label{eq:nasty:Dt:uq:orangutan}
\end{align}
for all $N+M \leq \sfrac{3\Nfin}{2}+1$. To see this, we take  $B = D_{t,q'-1}$ and $A = D_{q'}$, so that $A+B = D_{t,q'}$. The estimate \eqref{eq:nasty:Dt:uq:orangutan} now is a consequence of identity~\eqref{eq:cooper:1} and the parameter inequalities $\Gamma_{q'+1}^{\cstar +3} \tau_{q'-1}^{-1} \leq \tau_{q'}^{-1}$ (which follows from \eqref{eq:Tau:q-1:q}) and $\Gamma_{q'+1}^{i-\cstar+1} \tau_{q'}^{-1} \leq \tilde \tau_{q'}^{-1}$ (which is a consequence of \eqref{eq:imax:old} and \eqref{eq:Lambda:q:t:1}).
In a similar fashion, the bound \eqref{eq:sharp:Dt:psi:i:q:mixed:old} and identity~\eqref{eq:cooper:1} imply that 
\begin{align}
&\frac{{\mathbf{1}}_{\supp \psi_{i,q'}}}{\psi_{i,q'}^{1- (N+M)/\Nfin}} \left| D^N  D_{t,q'}^{M}  \psi_{i,q'} \right| \notag\\
&\qquad  \les  \MM{N,\NindLarge,  \Gamma_{q'}  \lambda_{q'},  \Gamma_{q'} \tilde \lambda_{q'}  } 
\MM{M,\Nindt-\NcutSmall, \Gamma_{q'+1}^{i-\cstar} \tau_{q'}^{-1}, \Gamma_{q'+1}^{-1}  \tilde \tau_{q'}^{-1}}
\label{eq:nasty:Dt:psi:i:q:orangutan}
\end{align}
for all $N+M \leq \Nfin$. Indeed, the above estimates follow from the same parameter inequalities mentioned above, and from identity~\eqref{eq:cooper:1} with $A = D_{q'}$ and $B = D_{t,q'-1}$.
\end{remark}
 
\begin{remark}
The inductive assumptions for the velocities given in Sections~\ref{sec:inductive:primary:velocity} and~\ref{sec:inductive:secondary:velocity}, with the definition of the mollifier operator ${\mathcal P}_{q,x,t}$ in Section~\ref{sec:mollifiers:Fourier}, imply that the new velocity field $v_q = w_q + v_{\ell_{q-1}}$ is very close to its mollification $v_{\ell_q}$, uniformly in space and time. That is, we have 
\begin{align}
\label{eq:vq:minus:mollified:ind}
\norm{D^n D_{t,q-1}^m (v_{\ell_q} - v_q)}_{L^\infty} 
\leq 
\lambda_q^{-2}   \delta_{q}^{\sfrac 12} \MM{n,2\Nindv,\lambda_q,\tilde \lambda_q} \MM{m, \Nindvt ,\tau_{q-1}^{-1}\Gamma_{q}^{i-1}, \Tilde{\tau}_{q-1}^{-1} \Gamma_q^{-1}}
\end{align}
for all  $n, m \leq 3 \Nindv$. The proof of the above bound is given in Lemma~\ref{lem:mollifying:ER}, cf.~estimate \eqref{eq:vq:minus:mollified}.
\end{remark}

\subsection{Main inductive proposition}
The main inductive proposition, which propagates the inductive estimates in  Section~\ref{sec:inductive:estimates} from step $q$ to step $q+1$, is as follows.

\begin{proposition}\label{p:main}
Fix $\beta \in [\sfrac 13,\sfrac 12)$ and choose $b\in (1,\sfrac{1}{2\beta})$. Solely in terms of $\beta$ and $b$, define the parameters $\nmax$, $\CLebesgue$, $\mathsf{C_R}$, $\cstar$, $\eps_\Gamma$, $\alpha_{\mathsf R}$, $\NcutSmall$, $\NcutLarge$, $\Nindt$, $\Nindv$, $\Ndec$, $\dpot$, and $\Nfin$, by the definitions in Section~\ref{sec:parameters:DEF}, items \eqref{item:beta:DEF}--\eqref{item:Nfin:DEF}.
Then, there exists a sufficiently large $a_* = a_*(\beta,b) \geq 1$, such that for any $a\geq a_*$, the following statement holds for any $q\geq 0$.
Given a velocity field $v_q$ which solves the Euler-Reynolds system with stress $\RR_q$, define $v_{\ell_q}, w_q$, and $u_q$ via \eqref{vlq}--\eqref{eq:cutoffs:wu}. Assume that $\{ u_{q'} \}_{q'=0}^{q-1}$ satisfies \eqref{eq:inductive:assumption:derivative}, $w_q$ obeys \eqref{eq:inductive:assumption:derivative:q}--\eqref{eq:perturbation:time:support}, $\RR_q$ satisfies \eqref{eq:Rq:inductive:assumption}, and that for every $q'\leq q-1$ there exists a partition of unity $\{ \psi_{i,q'}\}_{i\geq 0}$ such that properties \eqref{eq:inductive:partition}--\eqref{eq:imax:old} and estimates \eqref{eq:sharp:Dt:psi:i:q:old}--\eqref{eq:nasty:Dt:wq:WEAK:all} hold. Then, there exists a velocity field $v_{q+1}$, a stress $\RR_{q+1}$, and a partition of unity $\{\psi_{i,q}\}_{q\geq 0}$, such that $v_{q+1}$ solves the Euler-Reynolds system with stress $\RR_{q+1}$, $u_q$ satisfies \eqref{eq:inductive:assumption:derivative} for $q'\mapsto q$, $w_{q+1}$ obeys \eqref{eq:inductive:assumption:derivative:q}--\eqref{eq:perturbation:time:support} for $q\mapsto q+1$, $\RR_{q+1}$ satisfies \eqref{eq:Rq:inductive:assumption} for $q\mapsto q+1$, and the $\psi_{i,q}$ are such that \eqref{eq:inductive:partition}--\eqref{eq:nasty:Dt:wq:WEAK:all} hold when $q' \mapsto q$.
\end{proposition}

\subsection{Proof of Theorem~\ref{thm:main}}
\label{sec:proof:of:main:theorem}
Choose the parameters $\beta, b, \ldots, a_*$, as described in Section~\ref{sec:parameters:DEF}, and assume that with these parameter choices, and for {\em any} $a\geq a_*$, we are able to propagate the inductive bounds  claimed in Sections~\ref{sec:inductive:primary:velocity}--\ref{sec:inductive:secondary:velocity}
from step $q$ to step $q+1$, for all $q\geq 0$; this is achieved in Sections~\ref{sec:cutoff}--\ref{s:stress:estimates}. We next show that if $a\geq a_*$ is chosen sufficiently large, depending additionally on the $v_{\mathrm{start}}$, $v_{\mathrm{end}}$, $T>0$, and $\epsilon>0$ from the statement of Theorem~\ref{thm:main}, then the inductive assumptions imply Theorem~\ref{thm:main}. 

Without loss of generality, assume that $\int_{\T^3} v_{\mathrm{start}}(x) dx = \int_{\T^3} v_{\mathrm{end}}(x) dx = 0$. Since these functions lie in $L^2(\T^3)$,  there exists $R >0$ such that 
upon defining 
\begin{align*}
    v_{0}^{(1)}:= \Proj_{\leq R} v_{\mathrm{start}}
    \,,
    \qquad \mbox{and} \qquad 
    v_{0}^{(2)}:= \Proj_{\leq R} v_{\mathrm{end}}
    \,,
\end{align*}
where $\Proj_{\leq R}$ denotes the Fourier truncation operator to frequencies $|\xi| \leq R$, we have that 
\begin{align}
    \|v_0^{(1)} - v_{\mathrm{start}}\|_{L^2(\T^3)} 
    + 
    \|v_0^{(2)} - v_{\mathrm{end}}\|_{L^2(\T^3)} 
    \leq 
    \frac{\epsilon}{2}
    \,.
    \label{eq:initial:data:is:close}
\end{align}
Note that $v_0^{(1)}, v_0^{(2)} \in C^\infty(\T^3)$, and thus
by the classical local well-posedness theory plus propagation of regularity (see Foias, Frisch, and Temam~\cite{FoiasFrischTemam75}), there exists $T_0>0$ and unique strong solutions $v^{(1)} \in C^\infty( (-T_0,T_0) \times \T^3)$ and $v^{(2)}\in C^\infty( (T-T_0,T+T_0)\times \T^3)$ of the 3D Euler system \eqref{e:eulereq}, such that $v^{(1)}(x,0) = v_0^{(1)}(x)$ and $v^{(2)}(x,T) = v_0^{(2)}(x)$. Without loss of generality, we may take $T_0 \leq \sfrac{T}{2}$.

Next, let $\varphi\colon[0,T] \to [0,1]$ be a non-increasing $C^\infty$ smooth function such that $\varphi \equiv 1$ on $[0,\sfrac{T_0}{2}]$ and $\varphi \equiv 0$ on $[T_0,T]$. 
Define the $C^\infty$-smooth function 
\begin{align}
v_0(x,t) := \varphi(t) v^{(1)}(x,t)  +  \varphi(T-t)  v^{(2)}(x,t) 
\,.
\label{eq:the:v0:def}
\end{align}
On $[0,T]$, $v_0$ solves the Euler-Reynolds system \eqref{eq:Euler:Reynolds:again} for a suitable zero mean scalar pressure $p_0$, with the $C^\infty$-smooth stress $\RR_0$ defined by 
\begin{align}
    \RR_0(x,t) 
    &:= (\partial_t \varphi)(t) \RSZ v^{(1)}(x,t) - (\partial_t \varphi)(T-t)\RSZ v^{(2)}(x,t) \notag\\
    &\quad + \varphi(t) (\varphi(t) - 1)  (v^{(1)}  \mathring \otimes  v^{(1)})(x,t)
    + \varphi(T-t) (\varphi(T-t) - 1)  (v^{(2)}  \mathring \otimes  v^{(2)})(x,t)
    \,,
\label{eq:the:R0:def}
\end{align}
where $\RSZ$ is the classical nonlocal inverse-divergence operator (see~\eqref{eq:RSZ} for the definition). 
From the above definition and the fact that $\varphi \equiv 1$ on $[0,\sfrac{T_0}{2}]$, we deduce that 
\begin{align}
    \supp_t (\RR_0) \subset [\sfrac{T_0}{2}, T- \sfrac{T_0}{2} ]
    \,.
    \label{eq:the:R0:time:support}
\end{align}
This fact will be needed towards the end of the proof.

For consistency of notation, we also define $v_{-1} = v_{\ell_{-1}} = u_{-1} = 0$, so that $v_0 = w_0$ holds by~\eqref{eq:cutoffs:wu}. For the velocity cutoffs, we let $\psi_{0,-1} = 1$ and $\psi_{i,-1} = 0$ for all $i\geq 1$. It is then immediate to check that the $\{\psi_{i,-1}\}_{i\geq 0}$ satisfy the inductive assumptions \eqref{eq:inductive:partition}--\eqref{eq:psi:i:q:support:old}, for $q'=-1$, with the derivative bounds \eqref{eq:sharp:Dt:psi:i:q:old} and \eqref{eq:sharp:Dt:psi:i:q:mixed:old} being empty statements for $K+M \geq 1$, respectively when $N+K+M\geq 1$. Moreover, the bounds \eqref{eq:inductive:assumption:derivative} and  \eqref{eq:nasty:D:wq:old}--\eqref{eq:nasty:Dt:wq:WEAK:old} hold for $q'=-1$ since the left side of these inequalities vanishes identically. Lastly, the assumption \eqref{eq:perturbation:time:support} is empty since there is no $\RR_{-1}$ stress to speak of. 

It thus remains to verify that the pair $(v_0,\RR_0)$ defined in \eqref{eq:the:v0:def}--\eqref{eq:the:R0:def} satisfies the estimates \eqref{eq:inductive:assumption:derivative:q} and \eqref{eq:Rq:inductive:assumption}, where by the above choices we have $D_{t,-1} = \partial_t$. Note that the parameter $\Nindv$ was already chosen; thus,  we have that
\begin{align}
C_{\mathrm{datum}}
&:= 
\max_{0\leq n,m \leq 7 \Nindv}  \norm{D^n \partial_t^m v_0}_{L^\infty(0,T; L^2(\T^3)) } 
 +
\max_{0\leq n,m \leq 3 \Nindv} \norm{D^n \partial_t^m \RR_0}_{L^\infty(0,T; L^1(\T^3))} 
< \infty
\, .
\label{eq:C:datum:def}
\end{align}
Note that $C_{\mathrm{datum}}$ only depends on $v_{\mathrm{start}}$, $v_{\mathrm{end}}$, the cutoff frequency $R>0$, the choice of the cutoff function $\varphi$, on $T>0$, and on the parameter $\Nindv$. In particular, $C_{\mathrm{datum}}$ does not depend on the parameter $a$, which is the base of the exponential defining $\lambda_q$ in \eqref{eq:Gamma:q+1:def:*}.
Defining $\tau_{-1} = \Gamma_0^{-1} = \lambda_0^{-\eps_\Gamma}$ and $\tilde \tau_{-1} = \Gamma_0^{-3} = \lambda_0^{-3\eps_\Gamma}$ (these parameters are never used again), and that $\lambda_0 \geq a \geq a_* \geq 1$, we thus have that \eqref{eq:inductive:assumption:derivative:q} and \eqref{eq:Rq:inductive:assumption} hold if we ensure that 
\begin{align}
    C_{\mathrm{datum}}
    \leq \Gamma_0^{-1} \delta_0^{\sfrac 12} 
    \qquad \mbox{and} \qquad
    C_{\mathrm{datum}}
    \leq \Gamma_0^{\shaq} \delta_1
    \,.
    \label{eq:zero:level:to:check}
\end{align}
Using that $\eps_\Gamma$ is sufficiently small with respect to $\beta$ and $b$, we have that 
$
\Gamma_0^{-1} \delta_0^{\sfrac 12} = \lambda_0^{-\eps_\Gamma} \lambda_1^{(b+1)\beta/2} \lambda_0^{- \beta} \geq (\lambda_1 \lambda_0^{-1})^{(b+1)\beta/2} \geq (a^{b-1}/2)^{\beta}
$.
Also, by using that $\eps_\Gamma$ is chosen to be sufficiently small with respect to $\beta$ and $b$, we have that $\Gamma_0^{\shaq} \delta_1 = \lambda_0^{(4b+1) \eps_\Gamma} \lambda_1^{(b-1)\beta} \geq (\lambda_1 \lambda_0^{-1})^{(b-1)\beta} \geq  (a^{b-1}/2)^{(b-1)\beta}$. Thus, if in addition to $a\geq a_*$, as specified by item~\eqref{item:astar:DEF} in Section~\ref{sec:parameters:DEF}, if we choose $a \geq a_*$ to be sufficiently large in terms of $\beta, b$ and the constant $C_{\mathrm{datum}}$ from \eqref{eq:C:datum:def}, in order to ensure that 
\[ 
 a^{(b-1)^2 \beta} \geq 4 C_{\mathrm{datum}}\,,
\]
then the condition \eqref{eq:zero:level:to:check} is satisfied. 
We make this choice of $a$, and thus all the estimates claimed in Sections~\ref{sec:inductive:primary:velocity}--\ref{sec:inductive:secondary:velocity} hold true for the base step in the induction, the case $q=0$. 

Proceeding inductively, these estimates thus hold true for all $q\geq 0$. This allows us to define a function $v \in C^0(0,T; H^{\beta'}(\T^3))$ for any $\beta'<\beta$ via the absolutely convergent series\footnote{We may equivalently define $v = \lim_{q\to \infty} v_q = \lim_{q\to \infty} w_q + \sum_{q'=0}^{q-1} u_{q'} = \sum_{q'\geq 0} u_{q'}$. We choose to work with \eqref{eq:v:final:def} because it highlights the dependence on $v_0$.}
\begin{align}
v = \lim_{q\to \infty} v_q =  v_0 + \sum_{q\geq 0} (v_{q+1} - v_q) = v_0 + \sum_{q\geq 0} \left( w_{q+1} + (v_{\ell_q} - v_q) \right)
\,,
\label{eq:v:final:def}
\end{align}
where we recall the notation \eqref{vlq} and \eqref{eq:cutoffs:wu}. Indeed, by \eqref{eq:inductive:assumption:derivative:q}, \eqref{eq:inductive:partition}, and interpolation, we have that $\norm{w_q}_{H^{\beta'}} \leq 2 \Gamma_q^{-1} \delta_q^{\sfrac 12} \lambda_q^{\beta'} = 2 \Gamma_{q}^{-1} \lambda_1^{\sfrac{(b+1)\beta}{2}} \lambda_q^{-(\beta - \beta')}$ which is summable for $q\geq 0$ whenever $\beta'<\beta$. By appealing to the bound \eqref{eq:vq:minus:mollified:ind}, we furthermore obtain that $\norm{v_{\ell_q} - v_q}_{H^{\beta'}} \les \lambda_q^{-2} \delta_q^{\sfrac 12} \lambda_q^{\beta'} \les \lambda_1^{\sfrac{(b+1)\beta}{2}} \lambda_q^{-2 -(\beta - \beta')}$, which is again summable over $q\geq 0$. This justifies the definition of $v$ in \eqref{eq:v:final:def}, and the fact that $v \in C^0(0,T; H^{\beta'}(\T^3))$ for any $\beta'<\beta$. Finally, we note that by additionally appealing to \eqref{eq:Rq:inductive:assumption}, which yields $\|\RR_q\|_{L^1} \les \Gamma_q^{\shaq} \delta_{q+1} \to 0$ as $q\to \infty$, in view of \eqref{eq:Euler:Reynolds:again} the function $v$ defined in \eqref{eq:v:final:def} is a weak solution of the Euler equations on $[0,T]$.

In order to complete the proof, we return to \eqref{eq:v:final:def} and note that due to \eqref{eq:perturbation:time:support} (with $q=1$), the property \eqref{eq:the:R0:time:support} of $\RR_0$, and the fact that $\lambda_0 \delta_0^{\sfrac 12} = \lambda_0^{1-\beta} \lambda_1^{\sfrac{(b+1)\beta}{2}} \geq \sfrac{4}{T_0}$ (which holds upon choosing $a$ sufficiently large with respect to $T_0, \beta, b$), we have that $w_1 \equiv 0$ on the set $[0,\sfrac{T_0}{4}]\times \T^3 \cup [T-\sfrac{T_0}{4},T] \times \T^3$. Thus, from \eqref{eq:v:final:def} and the previously established bounds for $w_q$ (via \eqref{eq:inductive:assumption:derivative:q}, \eqref{eq:inductive:partition}) and $v_{\ell_q} -v_q$ (via \eqref{eq:vq:minus:mollified:ind}), we have that
\begin{align}
\norm{v - v_0}_{L^\infty([0,\sfrac{T_0}{4}] \cup [T-\sfrac{T_0}{4},T];L^2(\T^3))} 
&\leq \sum_{q\geq 2} \norm{w_q}_{L^\infty([0,T];L^2(\T^3))}
+ \sum_{q\geq 0} \norm{v_{\ell_q} - v_q}_{L^\infty([0,T];L^2(\T^3))}
\notag\\
&\leq 2 \lambda_1^{\sfrac{(b+1)\beta}{2}} \sum_{q\geq 2} \Gamma_q^{-1} \lambda_q^{-\beta} 
+ \lambda_1^{(b+1)\beta} \sum_{q\geq 0} \lambda_q^{-2-\beta}
\notag\\
&\leq 4 \lambda_1^{\sfrac{(b+1)\beta}{2}} \Gamma_2^{-1} \lambda_2^{-\beta} + 2 \lambda_1^{(b+1)\beta} \lambda_0^{-2-\beta}
\notag\\
&\leq 8\Gamma_2^{-1} \lambda_1^{\sfrac{(b+1)\beta}{2}}  \lambda_1^{-\beta b} + 4 \lambda_0^{(b+1)b\beta} \lambda_0^{-2-\beta}
\notag\\
&\leq  \lambda_1^{-\sfrac{(b-1)\beta}{2}}   + 4 \lambda_0^{- \sfrac 12}
\notag\\
&\leq \frac{\epsilon}{2}
\label{eq:v:minus:v0:bound}
\end{align}
once $a$ (and thus $\lambda_0$ and $\lambda_1$) is taken to be sufficiently large with respect to $b, \beta$, and $\epsilon$. Here, in the second-to-last inequality we have used that $\beta (b^2+b-1) \leq \sfrac 32$, which holds since $\beta<\sfrac 12$ and $b< \sfrac 32$. Combining \eqref{eq:v:minus:v0:bound} with the definition of the functions $v^{(1)}$, $v^{(2)}$, and $v_0$, and the bound \eqref{eq:initial:data:is:close},  we deduce that 
$\norm{v(\cdot,0) - v_{\mathrm{start}}}_{L^2(\T^3)} \leq \epsilon$ and $\norm{v(\cdot,T) - v_{\mathrm{end}}}_{L^2(\T^3)} \leq \epsilon$. This concludes the proof of Theorem~\ref{thm:main}, with $\beta$ being replaced by an arbitrary $\beta' \in (0,\beta)$.

\begin{remark}[\textbf{Modifications for achieving compact support in time}]
\label{rem:FU:P}
The proof outlined above may be easily modified to show the existence of infinitely many weak solutions in $C^0_t H^{\sfrac 12-}_x$ which are nontrivial and have compact support in time, as mentioned in Remark~\ref{rem:flexibility}. The argument is as follows.  Let $\varphi(t)$ be a $C^\infty$ smooth cutoff function, with $\varphi \equiv 1$ on $-[\sfrac{T}{4},\sfrac{T}{4}]$ and $\varphi \equiv 0$ on $\R \setminus [-\sfrac{T}{2},\sfrac{T}{2}]$. Then, instead of \eqref{eq:the:v0:def}, we define define $v_0(x,t) = E \varphi(t) (\sin(x_3),0,0)$. Note that the kinetic energy of $v_0$ at time $t=0$ is larger $E (2\pi)^{\sfrac 32}/2 \geq 2 E$, and that $v_0$ has time-support in $[-\sfrac{T}{2},\sfrac{T}{2}]$. Since $(\sin(x_3),0,0)$ is a shear flow, the zero order stress $\RR_0$ is given by $E \varphi'(t)$ multiplied by a matrix whose entries are zero, except for the $(1,3)$ and $(3,1)$ entries which equal $-\cos(x_3)$ (see~\cite[Section 5.2]{BV_EMS19} for details). The point is that $R_0$ is smooth, and its time-support lies in the interval $\sfrac{T}{4} \leq |t| \leq \sfrac{T}{2}$, which plays the role of \eqref{eq:the:R0:time:support}.
Using the same argument used in the proof of Theorem~\ref{thm:main}, we may show that for $a$ sufficiently large, the above defined pair $(v_0,\RR_0)$ satisfies the inductive assumptions at level $q=0$, and that these inductive assumptions may be propagated to all $q\geq 0$. As in~\eqref{eq:v:minus:v0:bound}, we deduce that the limiting weak solution solution $v$ has kinetic energy at time $t=0$ which is strictly larger than $E$. The fact that $\supp_t v_0, \supp_t \RR_0 \subset [-\sfrac{T}{2},\sfrac{T}{2}]$, combined with the inductive assumption \eqref{eq:perturbation:time:support} and the fact that the mollification procedure in Lemma~\ref{lem:mollifying:ER} expands time-supports by at most a factor of $\tilde \tau_{q-1} \ll (\lambda_{q-1} \delta_{q-1}^{\sfrac 12})^{-1}$, implies that the the weak solution $v$ has time-support in the set $|t| \leq \sfrac{T}{2} + 4 \sum_{q\geq 0} (\lambda_q \delta_{q}^{\sfrac 12})^{-1} \leq \sfrac{T}{2} + 8 \lambda_0^{\beta -1}$. Choosing $a$  sufficiently large shows that $\supp_t v \subset [-T,T]$. 
\end{remark}

\begin{remark}[\textbf{Modifications for attaining a given energy profile}]
\label{rem:FU:L}
The intermittent convex integration scheme described in this paper may be modified to show that within the regularity class $C^0_t H^{\sfrac 12-}_x$, weak solutions of 3D Euler may be constructed to attain any given smooth energy profile, as mentioned in Remark~\ref{rem:flexibility}. The main modifications required to prove this fact are as follows. As in previous schemes (see e.g.~De Lellis and Sz\'ekelyhidi Jr.~\cite{DeLellisSzekelyhidi13}, equations (7) and (9), or~\cite{BV19}, equations (2.5) and (2.6), etc.) we need to measure the distance between the energy resolved at step $q$ in the iteration, and the desired energy profile $e(t)$. The energy pumping produced in steps $q\mapsto q+1$ by the additions of pipe flows which comprise the velocity increments $w_{q+1}$, and the error due to mollification, was already understood in detail in Daneri and Sz\'ekelyhidi Jr.~\cite{DaneriSzekelyhidi17} and in~\cite{BDLSV17}. An additional difficulty in this paper is due to the presence of the higher order stresses: the energy profile would have to be inductively adjusted also throughout the steps $n\mapsto n+1$ and $p\mapsto p+1$. The other difficulty is the presence of the cutoff functions. This issue was however already addressed in~\cite{BV19}, cf.~Sections 4.5, 4.5, 6; albeit for a simpler version of the cutoff functions, which only included the stress cutoffs. With some effort, the argument in~\cite{BV19} may be indeed modified to deal with the cutoff functions present in this work. 
\end{remark}

\section{Building blocks}
\label{sec:building:blocks}

\newcommand{\bard}{{\Bar{d}}}

In  Section~\ref{ss:careful:construction}, we specify in Propositions~\ref{p:split} and Proposition~\ref{prop:pipe:shifted} the axes and shifts, respectively, that will characterize our intermittent pipe flows. A sufficiently diverse set of vector directions for the axes ensures that we can span a neighborhood of the identity in the space of symmetric $3\times 3$ matrices using positive linear combinations of simple tensors. Proposition~\ref{prop:pipe:shifted} crucially describes the $r^{-2}$ choices of placement afforded by the parameter $r$, which quantifies the diameter of the pipe. Then in Proposition~\ref{pipeconstruction}, we construct the intermittent pipe flows used in the rest of the paper and specify the essential properties. Section~\ref{ss:axis:control} contains Lemma~\ref{lem:axis:control}, which studies the evolution of the axes of the pipes under flow by an incompressible velocity field and related properties.  Section~\ref{ss:placements} contains Proposition~\ref{prop:disjoint:support:simple:alternate}, which is the placement lemma used to eliminate the Type 2 oscillation errors.  We remark that the results of this section are \emph{only} used in  Section~\ref{s:stress:estimates} - first to ensure the cancellation of errors in  Section~\ref{ss:stress:error:identification}, and second to show that the Type 2 errors vanish in  Section~\ref{ss:stress:oscillation:2}.

\subsection{A careful construction of intermittent pipe flows}\label{ss:careful:construction}
We recall from~\cite[Lemma~1]{Nash54} or~\cite[Lemma~2.4]{DaneriSzekelyhidi17} a version of the following geometric decomposition:
\begin{proposition}[\textbf{Choosing Vectors for the Axes}]\label{p:split}
Let $B_{\sfrac 12}(\Id)$ denote the ball of symmetric $3\times 3$ matrices, centered at $\Id$, of radius $\sfrac 12$. Then, there exists a finite subset $\Xi \subset\mathbb{S}^{2} \cap {\mathbb{Q}}^{3}$, for every $\xi \in \Xi$ there exists a smooth positive function 
$\gamma_{\xi}\colon C^{\infty}\left(B_{\sfrac 12} (\Id)\right) \to \R$, 
such that for each $R\in B_{\sfrac 12} (\Id)$ we have the identity
\begin{equation}\label{e:split}
R = \sum_{\xi\in\Xi} \left(\gamma_{\xi}(R)\right)^2  \xi\otimes \xi.
\end{equation}
Additionally, for every $\xi$ in $\Xi$, there exist vectors $\xi^{(2)},\xi^{(3)} \in \mathbb{S}^2 \cap \mathbb{Q}^3$ such that $\{\xi,\xi^{(2)},\xi^{(3)}\}$ is an orthonormal basis of $\R^3$, and there exists a least positive integer $n_\ast$ such that $n_\ast \xi, n_\ast \xi^{(2)}, n_\ast \xi^{(3)} \in \mathbb{Z}^3$, for every $\xi\in \Xi$.
\end{proposition}

In order to adapt the proof of Proposition~\ref{prop:disjoint:support:simple:alternate} to pipe flows oriented around axes which are not parallel to the standard basis vectors $e_1$, $e_2$, or $e_3$, it is helpful to consider functions which are periodic not only with respect to $\mathbb{T}^3$, but also with respect to a torus for which one face is perpendicular to the axis of the pipe (i.e., that one edge of the torus is parallel to the axis).

\begin{figure}[htb!]
\centering
\includegraphics[width=0.8\textwidth]{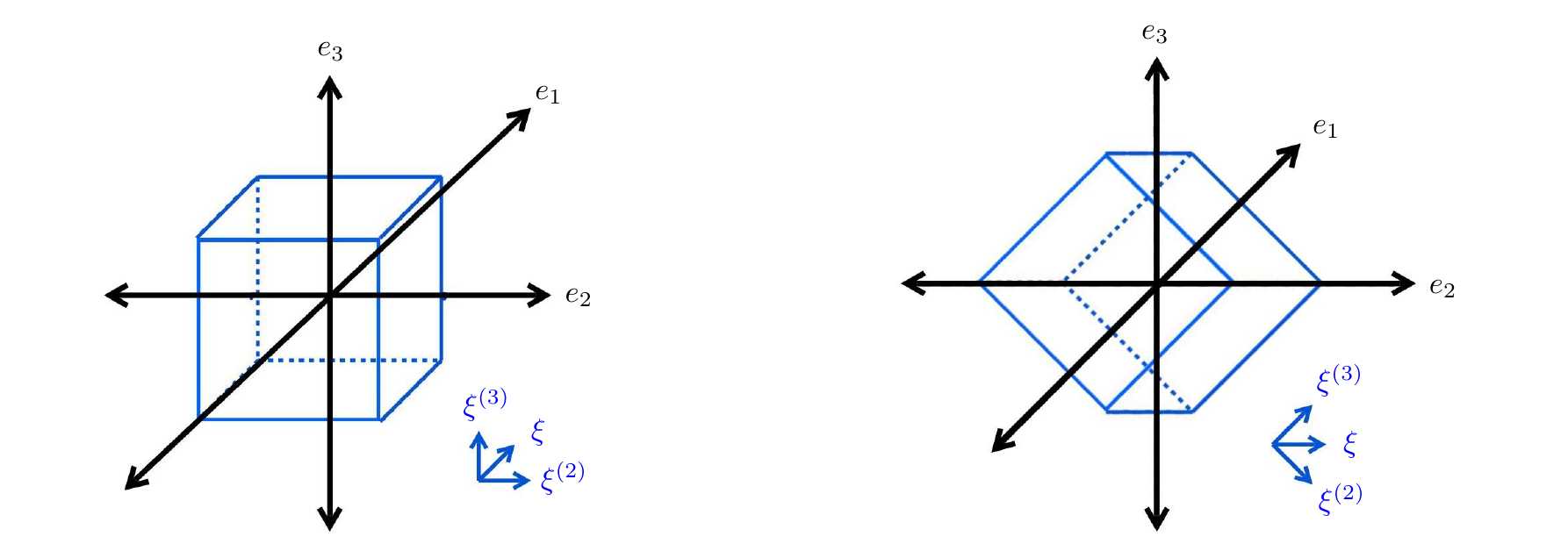}
\caption{\small The torus on the left, $\mathbb{T}^3$,  has axes parallel to the usual coordinate axes, while the torus on the right, denoted $\mathbb{T}^3_\xi$, has been rotated and has axes parallel to a new set of vectors $\xi$, $\xi^{(2)}$, and $\xi^{(3)}$.}\label{fig:T3:xi}
\end{figure}

\begin{definition}[\textbf{$\Tthreexi$-periodicity}]\label{def:periodicity}
Let $\{\xi,\xi^{(2)},\xi^{(3)}\}\subset\mathbb{S}^2\cap\mathbb{Q}^3$ be an orthonormal basis for $\mathbb{R}^3$, and let $f:\mathbb{R}^3\rightarrow\mathbb{R}^n$. We say that $f$ is $\Tthreexi$-periodic if for all $(k_1,k_2,k_3)\in \mathbb{Z}^3$ and $(x_1,x_2,x_3)\in\mathbb{R}^3$,
\begin{equation}\label{e:periodicity:1}
    f\left( (x_1,x_2,x_3) + \twopi\left( k_1 \xi +  k_2 \xi^{(2)} + k_3 \xi^{(3)} \right) \right) = f(x_1,x_2,x_3)
\end{equation}
and write $f:\Tthreexi\rightarrow\mathbb{R}^n$. If $\{\xi,\xi^{(2)},\xi^{(3)}\}=\{e_1,e_2,e_3\}$, i.e. the standard basis for $\mathbb{R}^3$, we drop the subscript $\xi$ and   write $\mathbb{T}^3$. For sets $\mathcal{S}\subset\mathbb{R}^3$, we say that $\mathcal{S}$ is $\Tthreexi$-periodic if the indicator function of $\mathcal{S}$ is $\Tthreexi$-periodic. Additionally, if $L$ is a positive number, we say that $f$ is $\left(\frac{\T^3_\xi}{L} \right)$-periodic if $$
f\left( (x_1,x_2,x_3) + \frac{\twopi}{L} \left( k_1 \xi +  k_2 \xi^{(2)} + k_3 \xi^{(3)} \right) \right) = f(x_1,x_2,x_3)
$$
for all $(k_1,k_2,k_3)\in \mathbb{Z}^3$ and $(x_1,x_2,x_3)\in\mathbb{R}^3$.  Note that if $L$ is a positive integer, $\frac{\mathbb{T}_\xi^3}{L}$-periodicity implies $\mathbb{T}_\xi^3$-periodicity.
\end{definition}

We can now construct shifted intermittent pipe flows concentrated around axes with a prescribed vector direction $\xi$ while imposing that each flow is supported in a single member of a large collection of disjoint sets. For the sake of clarity, we split the construction into two steps.  First, in Proposition~\ref{prop:pipe:shifted} we construct the shifts and then periodize and rotate the scalar-valued flow profiles and potentials associated to the pipe flows $\WW_{\xi,\lambda,r}$.  The support and placement properties are ensured at the level of the flow profile and potential.  Next, we use the flow profiles to construct the pipe flows themselves in Proposition~\ref{pipeconstruction}.

\begin{proposition}[\textbf{Rotating, Shifting, and Periodizing}]\label{prop:pipe:shifted}
Fix $\xi\in\Xi$, where $\Xi$ is as in Proposition~\ref{p:split}. Let ${r^{-1},\lambda\in\mathbb{N}}$ be given such that $\lambda r\in\mathbb{N}$. Let $\varkappa:\mathbb{R}^2\rightarrow\mathbb{R}$ be a smooth function with support contained inside a ball of radius $\frac{1}{4}$. Then for $k\in\{0,...,r^{-1}-1\}^2$, there exist functions $\varkappa^k_{\lambda,r,\xi}:\mathbb{R}^3\rightarrow\mathbb{R}$ defined in terms of $\varkappa$, satisfying the following additional properties:
\begin{enumerate}[(1)]
    \item \label{item:point:1} 
    We have that  $\varkappa^k_{\lambda,r,\xi}$ is simultaneously $\left(\frac{\mathbb{T}^3}{\lambda r}  \right)$-periodic and $\left(\frac{\Tthreexi}{\lambda r n_\ast}  \right)$-periodic.
    
    \item \label{item:point:2}  Let $F_\xi$ be one of the two faces of the cube $\frac{\Tthreexi}{\lambda r n_\ast}$ which is perpendicular to $\xi$. Let $\mathbb{G}_{\lambda,r}\subset F_\xi\cap \twopi \mathbb{Q}^3$ be the grid consisting of $r^{-2}$-many points spaced evenly at distance $\twopi  (\lambda n_\ast  )^{-1}$ on $F_\xi$ and containing the origin.  Then each grid point $g_{k}$ for $k\in\{0,...,r^{-1}-1\}^{2}$ satisfies
    \begin{equation}\label{e:shifty:support}
    \left(\supp\varkappa_{\lambda,r,\xi}^k\cap F_\xi \right) \subset \left\{x: |x-g_{k}| \leq \twopi\left(4\lambda n_\ast\right)^{-1} \right\}.
    \end{equation}
    
    \item \label{item:point:2a} The support of $\varkappa_{\lambda,r,\xi}^k$ consists of a pipe (cylinder) centered around a $\left(\frac{\mathbb{T}^3}{\lambda r}  \right)$-periodic and $\left(\frac{\Tthreexi}{\lambda r n_\ast}  \right)$-periodic line parallel to $\xi$, which passes through the point $g_k$. The radius of the cylinder's cross-section is as given in \eqref{e:shifty:support}.
    
    \item \label{item:point:3} For $k\neq k'$, $\supp \varkappa_{\lambda,r,\xi}^k \cap \supp \varkappa_{\lambda,r,\xi}^{k'}=\emptyset$.
\end{enumerate}
\end{proposition}

\begin{figure}[htb!]
\centering
\includegraphics[width=0.45\textwidth]{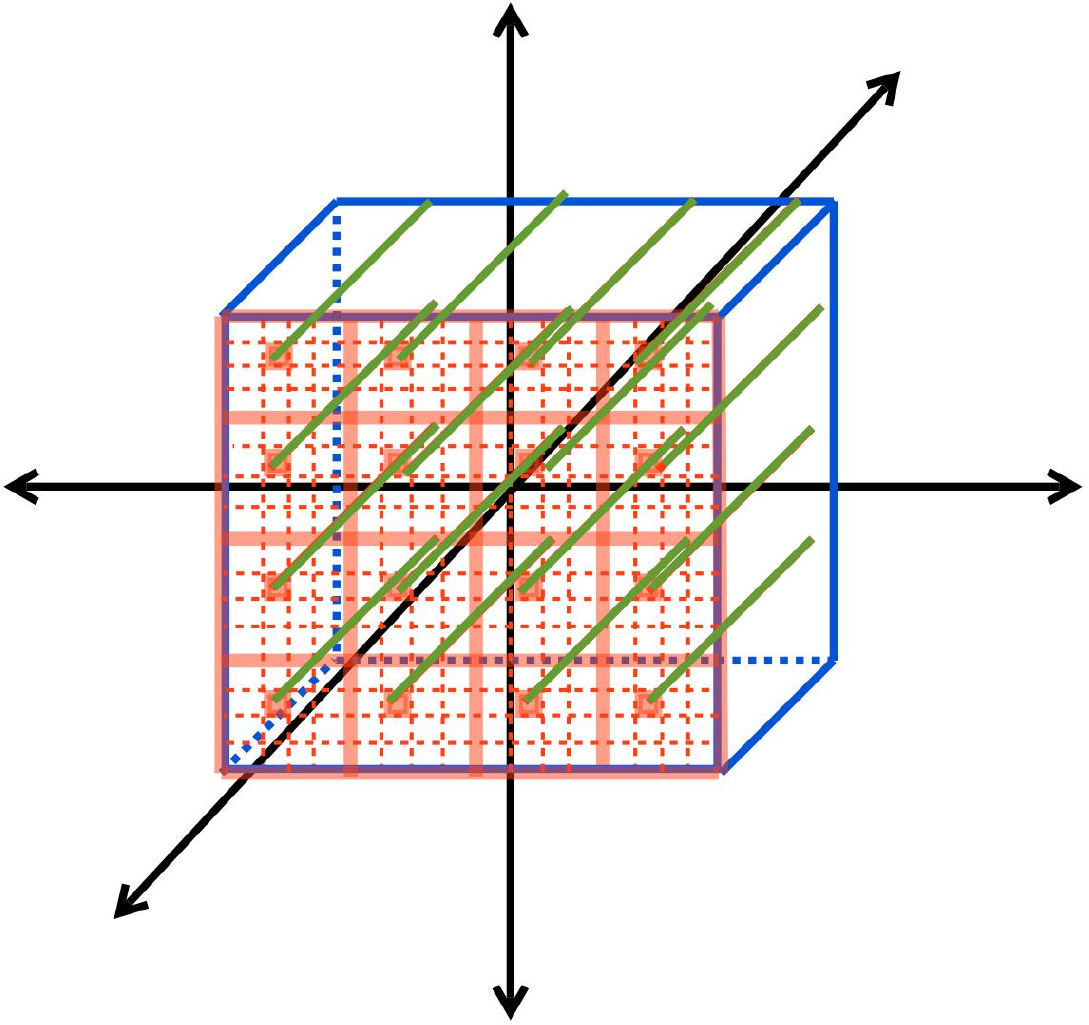}
\caption{\small We have pictured above a grid on the front face of $\mathbb{T}^3$ in which there are $4^2=(\lambda r)^2$ many periodic cells, each with $4^2=r^{-2}$ many subcells of diameter $16^{-1}=\lambda^{-1}$. The periodized axes of the pipes are the green lines, and they have been placed in the highlighted red squares on the front face of the torus.}\label{fig:choices2}
\end{figure}

\begin{proof}[Proof of Proposition~\ref{prop:pipe:shifted}]
For ${r^{-1}\in\mathbb{N}}$, which quantifies the rescaling, and for  $k= (k_1,k_2) \in\{0,...,r^{-1}-1\}^2$ which quantifies the shifts, define $\varkappa_{r}^k$ to be the rescaled and shifted function
\begin{equation}\label{e:shift:rescaled}
    \varkappa_{r}^k\left(x_1,x_2\right) := \frac{1}{\twopi r} \varkappa\left(\frac{x_1}{2\pi r}-k_1,\frac{x_2}{2\pi r}-k_2\right).
\end{equation}
Then $(x_1,x_2)\in\supp\varkappa_{r}^k$ if and only if
\begin{equation}\label{e:shift:base:support:0}
\left| \frac{x_1}{2\pi r}- k_1 \right|^2 +  \left| \frac{x_2}{2\pi r}- k_2 \right|^2 \leq \frac{1}{16}.
\end{equation}
This implies that
\begin{equation}\label{e:shift:base:support:1}
k_1-\frac{1}{4} \leq \frac{x_1}{2\pi r} \leq k_1+\frac{1}{4}, \qquad k_2-\frac{1}{4} \leq \frac{x_2}{2\pi r} \leq k_2+\frac{1}{4}.
\end{equation}
Since these inequalities cannot be satisfied by a single pair $(x,y)$ for both $k=(k_1,k_2)$ and $k'=(k_1',k_2')$ simultaneously when $k  \neq k'$, it follows that
\begin{equation}\label{e:shift:r:support:2}
    \supp \varkappa_{r}^k \cap \supp \varkappa_{r}^{k'} = \emptyset 
\end{equation}
for all $k\neq k'$.
Also, notice that plugging $k_1=0$ and $k_1=r^{-1}-1$ into \eqref{e:shift:base:support:1} shows that the set of $x_1$ for which there exists $(k_1,k_2)$ such that $\varkappa_{r}^k(x)\neq 0$ is contained in
$$  \left\{ -\frac{r\pi}{2} \leq x_1 \leq  2 \pi -\frac{3r\pi}{2} \right\}, $$
which is a set with diameter strictly less than $2\pi$. Therefore, periodizing in $x_1$ will not cause overlap in the supports of the periodized objects. Arguing similarly for $x_2$ and enumerating the pairs $(k_1,k_2)$ with $k\in\{0,...,r^{-1}-1\}^2$, we overload notation and 
denote by $\varkappa_{r}^k$, the $\mathbb{T}^2$-periodized version of $\varkappa_{r}^k$.  Thus we have produced $r^{-2}$-many functions which are $\mathbb{T}^2$-periodic and which have disjoint supports. 

Now define $\mathbb{G}_{r}\subset\mathbb{T}^2$ to be the grid containing $r^{-2}$-many points evenly spaced at distance $\twopi r$ and containing the origin.  Then
$$
\mathbb{G}_r = \left\{ g_k^0 := \twopi r k : k\in\{0,...,r^{-1}-1\}^2 \right\}\subset \twopi \mathbb{Q}^2
\,.
$$
Thus the support of each function $\varkappa_r^k$ contains $g_k^0$ as the center of its support but no other grid points.

Let $\xi\in\Xi$ be fixed, with the associated orthonormal basis $\{ \xi, \xi^{(2)},\xi^{(3)}\}$. For $x = (x_1,x_2,x_3) \in\mathbb{R}^3$ and  $\lambda r\in\mathbb{N}$, define
\begin{equation}\label{e:shift:vartheta:xi:r}
    \varkappa_{\lambda,r,\xi}^k(x):= \varkappa_r^k\left( n_\ast \lambda r x \cdot \xi^{(2)}, n_\ast \lambda r x \cdot \xi^{(3)} \right).
\end{equation}
Then for $(k_1,k_2,k_3)\in\mathbb{Z}^3$, 
\begin{align}
    \varkappa_{\lambda,r,\xi}^k\left(x+\frac{\twopi}{\lambda r}(k_1,k_2,k_3)\right) 
    &= \varkappa_r^k\left( n_\ast \lambda r \Bigl(x+\frac{\twopi}{\lambda r}(k_1,k_2,k_3)\Bigr) \cdot \xi^{(2)}, n_\ast \lambda r \Bigl(x+\frac{\twopi}{\lambda r}(k_1,k_2,k_3)\Bigr) \cdot \xi^{(3)} \right) \nonumber \\
    &= \varkappa_r^k\left( n_\ast \lambda r x \cdot \xi^{(2)}, n_\ast \lambda r x \cdot \xi^{(3)} \right)\nonumber\\
    &= \varkappa_{\lambda,r,\xi}^k(x) \nonumber
\end{align}
since $n_\ast \xi^{(2)}, n_\ast \xi^{(3)} \in \mathbb{Z}^3$ and $\varkappa_r^k$ is $\mathbb{T}^2$-periodic, and thus $\varkappa_{\lambda,r,\xi}^k$ is $\frac{\mathbb{T}^3}{\lambda r}$-periodic. Similarly, 
\begin{align}
    &\varkappa_{\lambda,r,\xi}^k\left(x+\frac{\twopi}{\lambda r n_\ast}(k_1 \xi + k_2 \xi^{(2)} + k_3\xi^{(3)})\right)\nonumber\\
    &= \varkappa_r^k\left( n_\ast \lambda r \Bigl(x+\frac{\twopi}{\lambda r n_\ast}(k_1 \xi + k_2 \xi^{(2)} + k_3\xi^{(3)})\Bigr) \cdot \xi^{(2)}, n_\ast \lambda r \Bigl(x+\frac{\twopi}{\lambda r n_\ast}(k_1 \xi + k_2 \xi^{(2)} + k_3\xi^{(3)})\Bigr) \cdot \xi^{(3)} \right) \nonumber \\
    &= \varkappa_r^k\left( n_\ast \lambda r x \cdot \xi^{(2)}, n_\ast \lambda r x \cdot \xi^{(3)} \right)\nonumber\\
    &= \varkappa^k_{\lambda,r,\xi}(x) \nonumber
\end{align}
since 
$$\twopi(k_1 \xi + k_2 \xi^{(2)} + k_3\xi^{(3)})\cdot\xi^{(2)}=\twopi k_2,\qquad \twopi(k_1 \xi + k_2 \xi^{(2)} + k_3\xi^{(3)})\cdot\xi^{(3)}=\twopi k_3$$
and $\varkappa_r^k$ is $\mathbb{T}^2$-periodic.  Thus $\varkappa_{\lambda,r,\xi}^k$ is $\frac{\Tthreexi}{\lambda r n_\ast}$-periodic, and as a consequence $\frac{\Tthreexi}{\lambda r}$-periodic as well. Therefore, we have proved point 
\ref{item:point:1}.  

To prove point \ref{item:point:2}, define  
\begin{equation}\label{e:shift:G:lambda}
    \mathbb{G}_{\lambda,r} = \left\{ g_k := \twopi k_1  \left(\lambda n_\ast\right)^{-1}\xi^{(2)} +  \twopi k_2 \left(\lambda n_\ast\right)^{-1}\xi^{(3)}: k_1,k_2 \in\{0,...,r^{-1}-1\} \right\}
    \,.
\end{equation}
We claim that $\varkappa_{\lambda,r,\xi}^k|_{F_\xi}$ is supported in a $2\pi(4\lambda n_\ast)^{-1}$-neighborhood of $g_k$. To prove the claim, let $x\in F_\xi$ be such that $\varkappa_{\lambda,r,\xi}^k(x)\neq 0$.  Then since
\begin{align*}
    \varkappa_{\lambda,r,\xi}^k(x) &= \varkappa_r^k\left(n_\ast \lambda r x \cdot \xi^{(2)}, n_\ast \lambda r x \cdot \xi^{(3)}\right),
\end{align*}
we can use \eqref{e:shift:base:support:0} to assert that $x\in\supp\varkappa_{\lambda,r,\xi}^k$ if and only if $x=(x_1,x_2,x_3)$ satisfies
\begin{align*}
    & \left| \frac{n_\ast \lambda r x\cdot \xi^{(2)}}{2\pi r}- k_1 \right|^2 +  \left| \frac{n_\ast \lambda r x\cdot \xi^{(3)}}{2\pi r}- k_2 \right|^2 \leq \frac{1}{16} \\
    &\quad\iff \left| (x_1,x_2,x_3) - \left( \frac{2\pi}{n_\ast \lambda} k_1 \xi^{(2)} + \frac{2\pi}{n_\ast \lambda} k_2 \xi^{(3)} \right) \right|^2 \leq \left(\frac{2\pi}{4n_\ast\lambda}\right)^2,
\end{align*}
which proves the claim.  

Items~\ref{item:point:2a} and~\ref{item:point:3} follow immediately after noting that $\varkappa_{\lambda,r,\xi}^k$ is constant on every plane parallel to $F_\xi$, and that the grid points $g_k\in\mathbb{G}_{\lambda,r}$ around which the supports of $\varkappa_{\lambda,r,\xi}^k$ are centered, are spaced at a distance which is twice the diameters of the supports. 
\end{proof}

\begin{proposition}[\textbf{Construction and properties of shifted Intermittent Pipe Flows}]
\label{pipeconstruction}
Fix a vector $\xi$ belonging to the set of rational vectors $\Xi\subset\mathbb{Q}^{3}$ from Proposition~\ref{prop:pipe:shifted}, $r^{-1},\lambda \in \mathbb{N}$ with $\lambda r\in \mathbb{N}$, and large integers $2\Nfin$ and $\dpot$. There exist vector fields $\WW^k_{\xi,\lambda,r}:\mathbb{T}^3\rightarrow\mathbb{R}^3$ for $k\in\{0,...,r^{-1}-1\}^2$ and implicit constants depending on $\Nfin$ and $\dpot$ but not on $\lambda$ or $r$ such that:
\begin{enumerate}[(1)]
    \item\label{item:pipe:1} There exists $\varrho:\mathbb{R}^2\rightarrow\mathbb{R}$ given by the iterated Laplacian $\Delta^\dpot  \vartheta =: \varrho$ of a potential $\vartheta:\mathbb{R}^2\rightarrow\mathbb{R}$ with compact support in a ball of radius $\frac{1}{4}$ such that the following holds.  Let $\varrho_{\xi,\lambda,r}^k$ and $\vartheta_{\xi,\lambda,r}^k$ be defined as in Proposition~\ref{prop:pipe:shifted}.  Then there exist $\UU^k_{\xi,\lambda,r}:\mathbb{T}^3\rightarrow\mathbb{R}^3$ such that
\begin{equation}\label{def:pipe:flow:main}
\displaystyle{\curl \UU^k_{\xi,\lambda,r} = \xi \lambda^{-2\dpot }\Delta^\dpot  \left(\vartheta^k_{\xi,\lambda,r}\right) = \xi \varrho^k_{\xi,\lambda,r} =: \WW^k_{\xi,\lambda,r}}
\, .
\end{equation}
    \item\label{item:pipe:2} Each of the sets of functions $\{\UU_{\xi,\lambda,r}^k\}_{k}$, $\{\varrho_{\xi,\lambda,r}^k\}_{k}$, $\{\vartheta_{\xi,\lambda,r}^k\}_{k}$, and $\{\WW_{\xi,\lambda,r}^k\}_{k}$ satisfy items~\ref{item:point:1}--\ref{item:point:3}. In particular, when $k\neq k'$, we have that the intersection of the supports of $\WW_k^{\xi,\lambda,r}$ and $\WW_{\xi,\lambda,r}^{k'}$ is empty, and similarly for the other sets of functions.
    \item\label{item:pipe:3} $\WW^k_{\xi,\lambda,r}$ is a stationary, pressureless solution to the Euler equations, i.e.
    $$  \div \WW^k_{\xi,\lambda,r} = 0, \qquad \div\left( \WW^k_{\xi,\lambda,r} \otimes \WW^k_{\xi,\lambda,r} \right) = 0 . $$
    \item\label{item:pipe:4} $\displaystyle{\frac{1}{|\mathbb{T}^3|}\int_{\mathbb{T}^3} \WW^k_{\xi,\lambda,r} \otimes \WW^k_{\xi,\lambda,r} = \xi \otimes \xi }$
    \item\label{item:pipe:5} For all $n\leq 2 \Nfin$, 
    \begin{equation}\label{e:pipe:estimates:1}
    {\left\| \nabla^n\vartheta^k_{\xi,\lambda,r} \right\|_{L^p(\mathbb{T}^3)} \lesssim \lambda^{n}r^{\left(\frac{2}{p}-1\right)} }, \qquad {\left\| \nabla^n\varrho^k_{\xi,\lambda,r} \right\|_{L^p(\mathbb{T}^3)} \lesssim \lambda^{n}r^{\left(\frac{2}{p}-1\right)} }
    \end{equation}
    and
    \begin{equation}\label{e:pipe:estimates:2}
    {\left\| \nabla^n\UU^k_{\xi,\lambda,r} \right\|_{L^p(\mathbb{T}^3)} \lesssim \lambda^{n-1}r^{\left(\frac{2}{p}-1\right)} }, \qquad {\left\| \nabla^n\WW^k_{\xi,\lambda,r} \right\|_{L^p(\mathbb{T}^3)} \lesssim \lambda^{n}r^{\left(\frac{2}{p}-1\right)} }.
    \end{equation}
    \item\label{item:pipe:6} Let $\Phi:\mathbb{T}^3\times[0,T]\rightarrow \mathbb{T}^3$ be the periodic solution to the transport equation
\begin{subequations}\label{e:phi:transport}
\begin{align}
\partial_t \Phi + v\cdot\nabla \Phi &=0, \\
\Phi_{t=t_0} &= x\, ,
\end{align}
\end{subequations}
with a smooth, divergence-free, periodic velocity field $v$. Then
\begin{equation}\label{eq:pipes:flowed:1}
\nabla \Phi^{-1} \cdot \left( \WW^k_{\xi,\lambda,r} \circ \Phi \right) = \curl \left( \nabla\Phi^T \cdot \left( \mathbb{U}^k_{\xi,\lambda,r} \circ \Phi \right) \right).
\end{equation}
\item\label{item:pipe:7} For $\mathbb{P}_{[\lambda_1,\lambda_2]}$ a Littlewood-Paley projector, $\Phi$ as in \eqref{e:phi:transport}, and $A=(\nabla\Phi)^{-1}$,
\begin{align}
&\bigg{[} \nabla \cdot \bigg{(} A \, \mathbb{P}_{[\lambda_1,\lambda_2]} \left(  \WW_{\xi,\lambda,r} \otimes  \WW_{\xi,\lambda,r} \right)(\Phi) A^T \bigg{)} \bigg{]}_i\nonumber\\
&\qquad = A_k^j \, \mathbb{P}_{[\lambda_1,\lambda_2]} \left( \WW^k_{\xi,\lambda,r}  \WW_{\xi,\lambda,r}^l\right)(\Phi) \partial_j A_l^i \nonumber\\
&\qquad = A_k^j \xi^k \xi^l \partial_j A_l^i \, \mathbb{P}_{[\lambda_1,\lambda_2]}\left( \left( \varrho^k_{\xi,\lambda,r} \right)^2 \right)
\label{eq:pipes:flowed:2}
\end{align}
for $i=1,2,3$. 
\end{enumerate}
\end{proposition}
\begin{remark}
The identity \eqref{eq:pipes:flowed:2} is one of the main advantages of pipe flows over Beltrami flows. The utility of this identity is that when checking whether a pipe flow $\WW_{\xi,\lambda,r}$ which has been deformed by $\Phi$ is still an approximately stationary solution of the pressureless Euler equations, one does not need to estimate any derivatives of $\WW_{\xi,\lambda,r}$ - only derivatives on the flow map $\Phi$, which will cost much less than $\lambda$.
\end{remark}
\begin{remark}
The formulation of \eqref{eq:pipes:flowed:2} is useful for our inversion of the divergence operator, which is presented in Proposition~\ref{prop:Celtics:suck} and the subsequent remark. We refer to the statement of that proposition and the subsequent remark for further properties related to \eqref{eq:pipes:flowed:2}.
\end{remark}

\begin{proof}[Proof of Proposition~\ref{pipeconstruction}]
With the definition $\WW_{\xi,\lambda,r}^k:=\xi \varrho_{\xi,\lambda,r}^k$, the equality $\lambda^{-2\dpot }\Delta^\dpot  (\vartheta_{\xi,\lambda,r}^k)=\varrho_{\xi,\lambda,r}^k$ follows from the proof of Proposition~\ref{prop:pipe:shifted}, specifically equations \eqref{e:shift:rescaled}, \eqref{e:shift:rescaled}, and \eqref{e:shift:vartheta:xi:r}. The equality $\curl\UU_{\xi,\lambda,r}^k=\WW_{\xi,\lambda,r}$ follows as well using the standard vector calculus identity $\curl \curl=\nabla \div-\Delta$. Secondly, properties \eqref{item:point:1}, \eqref{item:point:2}, and \eqref{item:point:3} from Proposition~\ref{prop:pipe:shifted} for $\vartheta_{\xi,\lambda,r}^k$ follow from Proposition~\ref{prop:pipe:shifted} applied to $\varkappa=\vartheta$. The same properties for $\varrho_{\xi,\lambda,r}^k$, $\UU_{\xi,\lambda,r}^k$, and $\WW_{\xi,\lambda,r}^k$ follow from differentiating.  Next, it is clear that $\WW_{\xi,\lambda,r}^k$ solves the pressureless Euler equations since $\xi\cdot\nabla\varrho_{\xi,\lambda,r}^k=0$. The normalization in \eqref{item:pipe:4} follows from imposing that 
\begin{equation*}
\frac{1}{(\twopi)^2}\int_{\mathbb{R}^2} (\Delta^\dpot  \vartheta(x_1,x_2))^2 \,dx_1\,dx_2 = 1,
\end{equation*}
recalling that orthogonal transformations, shifts, and scaling do not alter the $L^p$ norms of $\mathbb{T}^3$-periodic functions, and using \eqref{e:shift:rescaled}. The estimates in \eqref{item:pipe:5} follow similarly using \eqref{e:shift:rescaled}.  The proof of \eqref{eq:pipes:flowed:1} in \eqref{item:pipe:6} can be found in the paper of Daneri and Sz\'{e}kelyhidi Jr. \cite{DaneriSzekelyhidi17}.

The proof of \eqref{eq:pipes:flowed:2} from \eqref{item:pipe:7} is simple and similar in spirit to \eqref{item:pipe:6} but perhaps not standard, and so we will check it explicitly here. 
We first set $\mathcal{P}$ to be the $\mathbb{T}^3$-periodic convolution kernel associated with the projector $\mathbb{P}_{[\lambda_1,\lambda_2]}$ and write
\begin{align}
    \nabla \cdot &\bigg{(} (\nabla\Phi)^{-1} \mathbb{P}_{[\lambda_1,\lambda_2]} \left(  \WW_{\xi,\lambda,r} \otimes  \WW_{\xi,\lambda,r} \right)(\Phi)(\nabla\Phi)^{-T} \bigg{)} (x) \notag \\
    &= \nabla_x \cdot \bigg{(} (\nabla\Phi)^{-1}(x) \left(\int_{\mathbb{T}^3} \mathcal{P}(y) (\WW_{\xi,\lambda,r} \otimes \WW_{\xi,\lambda,r}) (\Phi(x-y)) \,dy \right) (\nabla\Phi)^{-T}(x) \bigg{)} \notag\\
    &= \nabla_x \cdot \bigg{(}\int_{\mathbb{T}^3} (\nabla\Phi)^{-1}(x) \mathcal{P}(y) (\WW_{\xi,\lambda,r} \otimes \WW_{\xi,\lambda,r}) (\Phi(x-y))  (\nabla\Phi)^{-T}(x) \,dy \bigg{)} \notag \\
    &= \nabla_x \cdot \left( \int_{\mathbb{T}^3}  \mathcal{P}(y) \left( (\nabla\Phi)^{-1}(x) \WW_{\xi,\lambda,r}(\Phi(x-y)) \right) \otimes \left( (\nabla\Phi)^{-1}(x) \WW_{\xi,\lambda,r}(\Phi(x-y)) \right)  \,dy  \right). \label{eq:pipe:flow:proof:1}
\end{align}
Then applying \eqref{eq:pipes:flowed:1}, we obtain that \eqref{eq:pipe:flow:proof:1} is equal to
\begin{align}
 \int_{\mathbb{T}^3}  \mathcal{P}(y) \left( (\nabla\Phi)^{-1}(x) \WW_{\xi,\lambda,r}(\Phi(x-y)) \right) \cdot \nabla_x \left( (\nabla\Phi)^{-1}(x) \WW_{\xi,\lambda,r}(\Phi(x-y)) \right)  \,dy\notag.
\end{align}
Writing out the $i^{th}$ component of this vector and using the notation $A=(\nabla\Phi)^{-1}$, we obtain
\begin{align}
    \bigg{[}\int_{\mathbb{T}^3}  \mathcal{P}(y) &\left( A(x) \WW_{\xi,\lambda,r}(\Phi(x-y)) \right) \cdot \nabla_x \left( A(x) \WW_{\xi,\lambda,r}(\Phi(x-y)) \right)  \,dy \bigg{]}_i \notag\\
    &= \int_{\mathbb{T}^3}  \mathcal{P}(y)  A_k^j (x) \WW_{\xi,\lambda,r}^k(\Phi(x-y)) A_l^i (x) \partial_n\WW^l_{\xi,\lambda,r}(\Phi(x-y)) \partial_j\Phi_n(x) \,dy \notag\\
    &\qquad + \int_{\mathbb{T}^3}  \mathcal{P}(y)  A_k^j(x) \WW_{\xi,\lambda,r}^k(\Phi(x-y))   \partial_j A_l^i(x) \WW^l_{\xi,\lambda,r}(\Phi(x-y)) \,dy \,. \label{eq:pipes:flowed:3}
\end{align}
Since the second term in \eqref{eq:pipes:flowed:3} can be rewritten as
\begin{align*}
    \int_{\mathbb{T}^3}  \mathcal{P}(y) A_k^j(x) & \WW_{\xi,\lambda,r}^k(\Phi(x-y))   \partial_j A_l^i (x) \WW^l_{\xi,\lambda,r}(\Phi(x-y)) \,dy \\
    &= A_k^j(x) \mathbb{P}_{[\lambda_1,\lambda_2]} \left( \WW^k_{\xi,\lambda,r}  \WW_{\xi,\lambda,r}^l\right)(\Phi(x)) \partial_j A_l^i(x),
\end{align*}
to conclude the proof, we must show that the first term in \eqref{eq:pipes:flowed:3} is equal to $0$. Using that
$$A_k^j\partial_j \Phi^n =\delta_{nk}$$
and 
$$\WW_{\xi,\lambda,r}^k\partial_k \WW_{\xi,\lambda,r}^l=0$$
for all $l$, we can simplify the first term as
\begin{align*}
    \int_{\mathbb{T}^3}  \mathcal{P}(y) A_k^j(x) & \WW_{\xi,\lambda,r}^k(\Phi(x-y)) A_l^i(x) \partial_n\WW^l_{\xi,\lambda,r}(\Phi(x-y)) \partial_j\Phi_n(x) \,dy \notag \\
    &= \int_{\mathbb{T}^3}  \mathcal{P}(y)  \delta_{nk} \WW_{\xi,\lambda,r}^k(\Phi(x-y)) A_l^i(x) \partial_n\WW^l_{\xi,\lambda,r}(\Phi(x-y)) \,dy \notag \\
    & = \int_{\mathbb{T}^3}  \mathcal{P}(y) \WW_{\xi,\lambda,r}^k(\Phi(x-y)) A_l^i(x) \partial_k\WW^l_{\xi,\lambda,r}(\Phi(x-y)) \,dy \notag \\
    &=0,
\end{align*}
proving \eqref{eq:pipes:flowed:2}.
\end{proof}

\subsection{Deformed pipe flows and curved axes}\label{ss:axis:control}
\begin{lemma}[\textbf{Control on Axes, Support, and Spacing}]
\label{lem:axis:control}
Consider a convex neighborhood  of space $\Omega\subset \mathbb{T}^3$. Let $v$ be an incompressible velocity field, and define the flow $X(x,t)$
\begin{subequations}\label{e:flow}
\begin{align}
\partial_t X(x,t) &= v\left(X(x,t),t\right) \\
X_{t=t_0} &= x\, ,
\end{align}
\end{subequations}
and inverse $\Phi(x,t)=X^{-1}(x,t)$
\begin{subequations}\label{e:transport}
\begin{align}
\partial_t \Phi + v\cdot\nabla \Phi &=0 \\
\Phi_{t=t_0} &= x\, .
\end{align}
\end{subequations}
Define $\Omega(t):=\{ x\in\mathbb{T}^3 : \Phi(x,t) \in \Omega \} = X(\Omega,t)$. For an arbitrary $C>0$, let $\tau>0$ be a parameter such that
\begin{equation}\label{eq:tau:axis:support}
 \tau\leq\left(\delta_q^{\sfrac{1}{2}}\lambda_q\Gamma_{q+1}^{C+2}\right)^{-1}  \, .
\end{equation}
Furthermore, suppose that  the vector field $v$ satisfies the Lipschitz bound\footnote{The implicit constant in this inequality is assumed to be independent of $q$, cf. (\ref{eq:nasty:D:vq}).}
\begin{equation}\label{e:axis:derivative:bounds}
\sup_{t\in [t_0 - \tau,t_0+\tau]} \norm{\nabla v(\cdot,t) }_{L^\infty(\Omega(t))} \lesssim \delta_q^{\sfrac{1}{2}}\lambda_q\Gamma_{q+1}^C \,.
\end{equation}
Let $\WW^k_{\lambda_{q+1},r,\xi}:\mathbb{T}^3\rightarrow\mathbb{R}^3$ be a set of straight pipe flows constructed as in Proposition~\ref{prop:pipe:shifted} and Proposition~\ref{pipeconstruction} which are $\frac{\mathbb{T}^3}{\lambda_{q+1}r}$-periodic for $\frac{\lambda_q}{\lambda_{q+1}}\leq r\leq 1$ and concentrated around axes $\{A_i\}_{i\in\mathcal{I}}$ oriented in the vector direction $\xi$ for $\xi\in\Xi$.  Then $\WW:=\WW^k_{\lambda_{q+1},r,\xi}(\Phi(x,t)):\Omega(t)\times[t_0-\tau,t_0+\tau]$ satisfies the following conditions:
\begin{enumerate}[(1)]
	\item  We have the inequality
	\begin{equation}\label{eq:diameter:inequality}
	\textnormal{diam}(\Omega(t)) \leq \left(1+\Gamma_{q+1}^{-1}\right)\textnormal{diam}(\Omega).
	\end{equation}
    \item If $x$ and $y$ with $x\neq y$ belong to a particular axis $A_i\subset\Omega$, then 
    \begin{equation}\label{e:axis:variation}
    \frac{X(x,t)-X(y,t)}{|X(x,t)-X(y,t)|} = \frac{x-y}{|x-y|} + \delta_i(x,y,t)    
    \end{equation}
    where $|\delta_i(x,y,t)|<\Gamma_{q+1}^{{-1}}$.
    \item Let $x$ and $y$ belong to a particular axis $A_i\subset\Omega$.  Denote the length of the axis $A_i(t):=X(A_i\cap\Omega,t)$ in between $X(x,t)$ and $X(y,t)$ by $L(x,y,t)$.  Then
    \begin{equation}\label{e:axis:length}
    L(x,y,t) \leq \left(1+\Gamma_{q+1}^{-1}\right)\left| x-y \right|.
    \end{equation}
    \item The support of $\WW$ is contained in a $\displaystyle\left(1+\Gamma_{q+1}^{-1}\right)\frac{\twopi}{4n_\ast\lambda_{q+1}}$-neighborhood of 
    \begin{equation}\label{e:axis:union}
       \bigcup_{i} A_i(t).
    \end{equation}
\item $\WW$ is ``approximately periodic" in the sense that for distinct axes $A_i,A_j$ with $i\neq j$ and $\dist(A_i\cap\Omega,A_j\cap\Omega)=d$,
\begin{equation}\label{e:axis:periodicity:1}
    \left(1-\Gamma_{q+1}^{-1}\right)d \leq \dist\left(A_i(t),A_j(t)\right)\leq \left(1+\Gamma_{q+1}^{-1}\right)d.
\end{equation}
\end{enumerate}
\end{lemma}
\begin{proof}[Proof of Lemma~\ref{lem:axis:control}]
First, we have that for $x,y\in\Omega$,
\begin{align}
    \left| X(x,t)-X(y,t) \right| &= \left| x -y  + \int_{t_0}^t \partial_s X(x,s)-\partial_s X(y,s) \,ds \right| \nonumber \\
    &\leq |x-y| + \int_{t_0}^t \left| v\left( X(x,s),s\right)-v\left( X(y,s),s\right) \right| \,ds. \nonumber
\end{align}
Furthermore,
\begin{align}
    \left| v^\ell\left(X(x,s),s\right)-    v^\ell\left(X(y,s),s\right) \right| &= \left| \int_0^1 \partial_j v^\ell\left(X(x+t(y-x),s),s\right) \partial_k X^j(x+t(y-x),s)(y-x)^k \,dt \right|\nonumber\\
    &\leq \left\| \nabla v \right\|_{L^\infty(\Omega(s))} \left\| \nabla X \right\|_{L^\infty(\Omega(s))} |x-y|\nonumber\\
    &\leq \frac{3}{2} \delta_q^\frac{1}{2}\lambda_q \Gamma_{q+1}^C |x-y| \, .\nonumber
\end{align}
Integrating this bound from $t_0$ to $t$ and using a factor of $\Gamma_{q+1}$ to absorb the constant, we deduce that
\begin{equation}\label{e:axis:xminusy}
    \left(1-\Gamma_{q+1}^{-1}\right) |x-y| \leq \left| X(x,t)-X(y,t) \right| \leq \left(1+\Gamma_{q+1}^{-1}\right) |x-y|.
\end{equation}
The inequality in \eqref{eq:diameter:inequality} follows immediately.

To prove \eqref{e:axis:variation}, we will show that for $x,y\in\Omega\cap A_i$ for a chosen axis $A_i$,
$$  \left| \frac{x-y}{|x-y|} - \frac{X(x,t)-X(y,t)}{|X(x,t)-X(y,t)|} \right| < \Gamma_{q+1}^{-1}. $$
At time $t_0$, the above quantity vanishes.  Differentiating inside the absolute value in time, we have that
\begin{align*}
    &\frac{d}{dt} \left[ \frac{X(x,t)-X(y,t)}{|X(x,t)-X(y,t)|} \right] \nonumber\\
    &= \frac{\partial_t X(x,t)-\partial_t X(y,t)}{|X(x,t)-X(y,t)|} -  \frac{ X(x,t)- X(y,t)}{|X(x,t)-X(y,t)|^2}\frac{\left(\partial_t X(x,t)-\partial_t X(y,t)\right)\cdot\left(X(x,t)-X(y,t)\right)}{|X(x,t)-X(y,t)|}\nonumber\\
    &= \frac{v(X(x,t),t)-v(X(y,t),t)}{|X(x,t)-X(y,t)|} -  \frac{ X(x,t)- X(y,t)}{|X(x,t)-X(y,t)|}\frac{\left(v(X(x,t),t)-v(X(y,t),t)\right)\cdot\left(X(x,t)-X(y,t)\right)}{|X(x,t)-X(y,t)|^2}
    \,.
\end{align*}
Utilizing the mean value theorem and the Lipschitz bound on $v$ and \eqref{e:axis:xminusy}, we deduce
\begin{align*}
    &\left|\frac{v(X(x,t),t)-v(X(y,t),t)}{|X(x,t)-X(y,t)|} -  \frac{ X(x,t)- X(y,t)}{|X(x,t)-X(y,t)|}\frac{\left(v(X(x,t),t)-v(X(y,t),t)\right)\cdot\left(X(x,t)-X(y,t)\right)}{|X(x,t)-X(y,t)|^2}\right|\\
    &\leq 2\norm{\nabla v}_{L^\infty}  \\
    &\leq 2 \delta_q^\frac{1}{2}\lambda_q\Gamma_{q+1}^C .
\end{align*}
Integrating in time from $t_0$ to $t$ for $|t-t_0|\leq \left(\delta_q^\frac{1}{2}\lambda_q\Gamma_{q+1}^{C+2}\right)^{-1}$ and using the extra factors of $\Gamma_{q+1}$ to again kill the constants, we obtain \eqref{e:axis:variation}.

To prove \eqref{e:axis:length}, we parametrize the curve using $X$ to obtain
\begin{align*}
    L(x,y,t)
    =  \int_0^{1} \left| \nabla X(x+r(y-x),t)\cdot(x-y) \right| \,dr  
    \leq \left(1+\Gamma_{q+1}^{-1}\right)|x-y|.
\end{align*}

The claims in \eqref{e:axis:union} and \eqref{e:axis:periodicity:1} follow immediately from \eqref{e:axis:xminusy} and \eqref{e:shifty:support}.
\end{proof}

\subsection{Placements via relative intermittency}\label{ss:placements}
We now state and prove the main proposition regarding the placement of a new set of intermittent pipe flows which do not intersect with previously placed and possibly deformed pipes \emph{within} a subset $\Omega$ of the full torus $\mathbb{T}^3$. We do not claim that intersections do not occur outside of $\Omega$.  In applications, $\Omega$ will be the support of a cutoff function.\footnote{Technically, $\Omega$ will be a set slightly larger than the support of a cutoff function. See \eqref{eq:overlap:2:1}, \eqref{eq:overlap:3:1}, and \eqref{def:the:real:omega}.} We state the proposition for new pipes periodized to spatial scale $\left(\lambda_{q+1}r_2\right)^{-1}$ with axes parallel to a direction vector $\xi\in\Xi$. By ``relative intermittency," we mean the inequality \eqref{eq:r1:r2:condition:alt} satisfied by $r_1$ and $r_2$. The proof proceeds, first in the case $\xi=e_3$, by an elementary but rather tedious counting argument for the number of cells in a two-dimensional grid which may intersect a set concentrated around a smooth curve.  In applications, these correspond to a piece of a periodic pipe flow concentrated around its deformed axis and then projected onto a plane. Then using (1) and (2) from Proposition~\ref{prop:pipe:shifted}, we describe the minor adjustments needed to obtain the same result for new pipes with axes parallel to arbitrary direction vectors $\xi\in\Xi$.

\begin{proposition}[\textbf{Placing straight pipes which avoid bent pipes}]
\label{prop:disjoint:support:simple:alternate} 
Consider a neighborhood of space $\Omega\subset \mathbb{T}^3$ such that 
 \begin{equation}\label{eq:Omega:diameter:alt}
\textnormal{diam}(\Omega) \leq 16 (\lambda_{q+1} r_1)^{-1},
\end{equation}
where $\sfrac{\lambda_q}{\lambda_{q+1}}\leq r_1 \leq 1$.  
Assume that  there exist smooth $\mathbb{T}^3$-periodic curves $\{ A_n\}_{n=1}^{N_\Omega} \subset\Omega$\footnote{That is, the range of each curve is contained in $\Omega$; otherwise replace the curves with $A_n \cap \Omega$.} and $\mathbb{T}^3$-periodic sets $\{ S_n \}_{n=1}^{N_\Omega} \subset\Omega$ satisfying the following properties:
\begin{enumerate}[(1)]
\item There exists a   positive constant $\const_A$ and a parameter $r_2$, with $r_1<r_2<1$ such that
\begin{equation}
N_\Omega \leq \const_A   r_2^2 r_1^{-2} \,. \label{eq:Npipe:bound}
\end{equation}
\item For any $x,x' \in A_n$, let the length of the curve $A_n$ which lies between $x$ and $x'$ be denoted by $L_{n,x,x'}$.  Then, for every $1\leq n \leq N_\Omega$ we have
\begin{align}
\label{eq:curve:length:straight:alt}
L_{n,x,x'}  \leq 2 \abs{x-x'}\,.
\end{align}
   
\item For every $1\leq n \leq N_\Omega$, we have that $S_n$ is contained in a $2\pi(1+\Gamma_{q+1}^{-1})\left(4 n_* \lambda_{q+1}\right)^{-1}$-neighborhood of $A_n$.
\end{enumerate}
Then, there exists a {\em geometric constant} $C_*\geq 1$ such that if 
\begin{align}\label{eq:r1:r2:condition:alt}
C_* \const_A r_2^4 \leq r_1^3,
\end{align}
then, for any $\xi \in \Xi$ (recall the set $\Xi$ from Proposition~\ref{p:split}), we can find a set of pipe flows $\WW^{k_0}_{\lambda_{q+1},r_2,\xi} \colon \T^3 \to \R^3$ which are $\frac{\T^3}{\lambda_{q+1}r_2}$-periodic, concentrated to width $\frac{2\pi}{4\lambda_{q+1} n_*}$ around axes with vector direction $\xi$, satisfy the properties listed in Proposition~\ref{pipeconstruction}, and for all $n\in\left\{1,...,N_\Omega \right\}$,
\begin{align}\label{e:disjoint:conclusion}
\supp  \WW^{k_0}_{\lambda_{q+1},r_2,\xi}  \cap S_n = \emptyset.
\end{align}
\end{proposition}
\begin{remark}
As mentioned previously, the sets $S_n$ will be supports of previously placed pipes oriented around {\em deformed} axes $A_n$. The properties of $S_n$ and $A_n$ will follow from Lemma~\ref{lem:axis:control}.
\end{remark}

\begin{proof}[Proof of Proposition~\ref{prop:disjoint:support:simple:alternate}]
For simplicity, we first give the proof for $\xi = e_3$, and explain how to treat the case of general $\xi \in \Xi$ at the end of the proof.

The proof will proceed by measuring the size of the shadows of the $\{S_n\}_{n=1}^{N_\Omega}$ when projected onto the face of the cube $\mathbb{T}^3$ which is perpendicular to $e_3$, so it will be helpful to set some notation related to this projection. Let $F_{e_3}$ be the face of the torus $\mathbb{T}^3$ which is perpendicular to $e_3$.  For the sake of concreteness, we will occasionally identify $F_{e_3}$ with the set of points $x = (x_1,x_2,x_3)\in\mathbb{T}^3$ such that $x_3=0$, or use that $F_{e_3}$ is isomorphic to $\mathbb{T}^2$. Let $A_n^p$ be the projection of $A_n$ onto $F_{e_3}$ defined by
\begin{equation}\label{eq:Fe3:definition}
    A_n^p := \left\{ (x_1,x_2)\in F_{e_3}: (x_1,x_2,x_3)\in A_n \right\},
\end{equation}
and let $S_n^p$ be defined similarly as the projection of $S_n$ onto $F_{e_3}$.  For $x = (x_1,x_2,x_3)\in\mathbb{T}^3$ and $x' = (x_1',x_2',x_3')\in\mathbb{T}^3$ we let $P(x)=(x_1,x_2)\in F_{e_3}$ and $P(x')=(x_1',x_2')\in F_{e_3}$ be the projection of these points onto $F_{e_3}$. Since projections do not increase distances, we have that
\begin{equation}\label{e:distance:decreases}
    \left| P(x)-P(x') \right| \leq \left| x  - x' \right|.
\end{equation}
Since both $A_n$ and $A_n^p$ are smooth curves\footnote{Technically, the proof still applies if $A_n^p$ is self-intersecting, but the conclusions of Lemma~\ref{lem:axis:control} eliminate this possibility, so we shall ignore this issue and use the word ``smooth''.} and can be approximated by piecewise linear polygonal paths, \eqref{e:distance:decreases}, \eqref{eq:Omega:diameter:alt}, and \eqref{eq:curve:length:straight:alt} imply that if $L_{n,x,x'}^p$ is the length of the projected curve $A_n^p$ in between the points $P(x)$ and $P(x')$, then
\begin{equation}\label{e:projected:distance}
    L_{n,x,x'}^p \leq 2|x-x'| \leq 32 \left(\lambda_{q+1}r_1\right)^{-1}.
\end{equation}
In particular, taking $x$ and $x'$ to be the endpoints of the curve $A_n$, we obtain a bound for the total length of $A_{n}^p$.  Additionally, \eqref{e:distance:decreases} and the third assumption of the lemma imply that $S_n^p$ is contained inside a $2\pi (1+\Gamma_{q+1}^{-1} )(4n_* \lambda_{q+1})^{-1}$-neighborhood of $A_n^p$.  Finally, since $\WW_{\lambda_{q+1},r_2,e_3}^k$ is independent of $x_3$ for all $k\in\{0,...,r_2^{-1}-1\}^2$, it is clear that the conclusion \eqref{e:disjoint:conclusion} will be achieved if we can show that there exists a shift $k_0$ such that 
\begin{equation}\label{e:disjoint:F}
    S_n^p \cap \left( \supp \WW_{\lambda_{q+1},r_2,e_3}^{k_0} \cap \{x_3=0\} \right) = \emptyset
    \,,
\end{equation}
for all $1 \leq n \leq N_\Omega$.
To prove \eqref{e:disjoint:F}, we will apply a covering argument to each $S_n^p$.  

Let $\mathbb{S}_{\lambda_{q+1}}$ be the grid of $(\lambda_{q+1} n_*)^2$-many open squares contained in $F_{e_3}$, evenly centered around a grid of $(\lambda_{q+1} n_*)^2$-many points $\mathbb{G}_{\lambda_{q+1}}$ which contains the origin.  By Proposition~\ref{prop:pipe:shifted}, for each choice of $k = (k_1,k_2) \in \{0,\ldots,r_2^{-1}-1\}^2$,  the support of the shifted pipe $\WW_{\lambda_{q+1},r_2,e_3}^{k}$ intersects $F_{e_3}$ in a $\frac{2\pi}{4\lambda_{q+1} n_*}$-neighborhood of a finite subcollection of grid points from $\mathbb{G}_{\lambda_{q+1}}$, which we call $\mathbb{G}_{\lambda_{q+1}}^k$, and which by construction is $\frac{\T^3}{\lambda_{q+1} r_2 n_*}$-periodic.  Furthermore, two subcollections for $k\neq k'$ contain no grid points in common. Let $\mathbb{S}^k_{\lambda_{q+1}}$ be the set of  open  squares centered around grid points in $\mathbb{G}_{\lambda_{q+1}}^k$, so that $\mathbb{S}^k_{\lambda_{q+1}}$ and $\mathbb{S}^{k'}_{\lambda_{q+1}}$ are disjoint if $k\neq k'$. To prove \eqref{e:disjoint:F}, we will  identify a shift $k_0$ such that the set of squares $\mathbb{S}_{\lambda_{q+1}}^{k_0}$ has empty intersection with $S_n^p$ for all $n$.  Then by Proposition~\ref{prop:pipe:shifted}, we have that the pipe flow $\WW_{\lambda_{q+1},r_2,e_3}^{k_0}$ intersects $F_{e_3}$ inside of $\mathbb{S}_{\lambda_{q+1}}^{k_0}$, and so we will have verified \eqref{e:disjoint:F}.

\begin{figure}[htb!]
\centering
\includegraphics[width=0.9\textwidth]{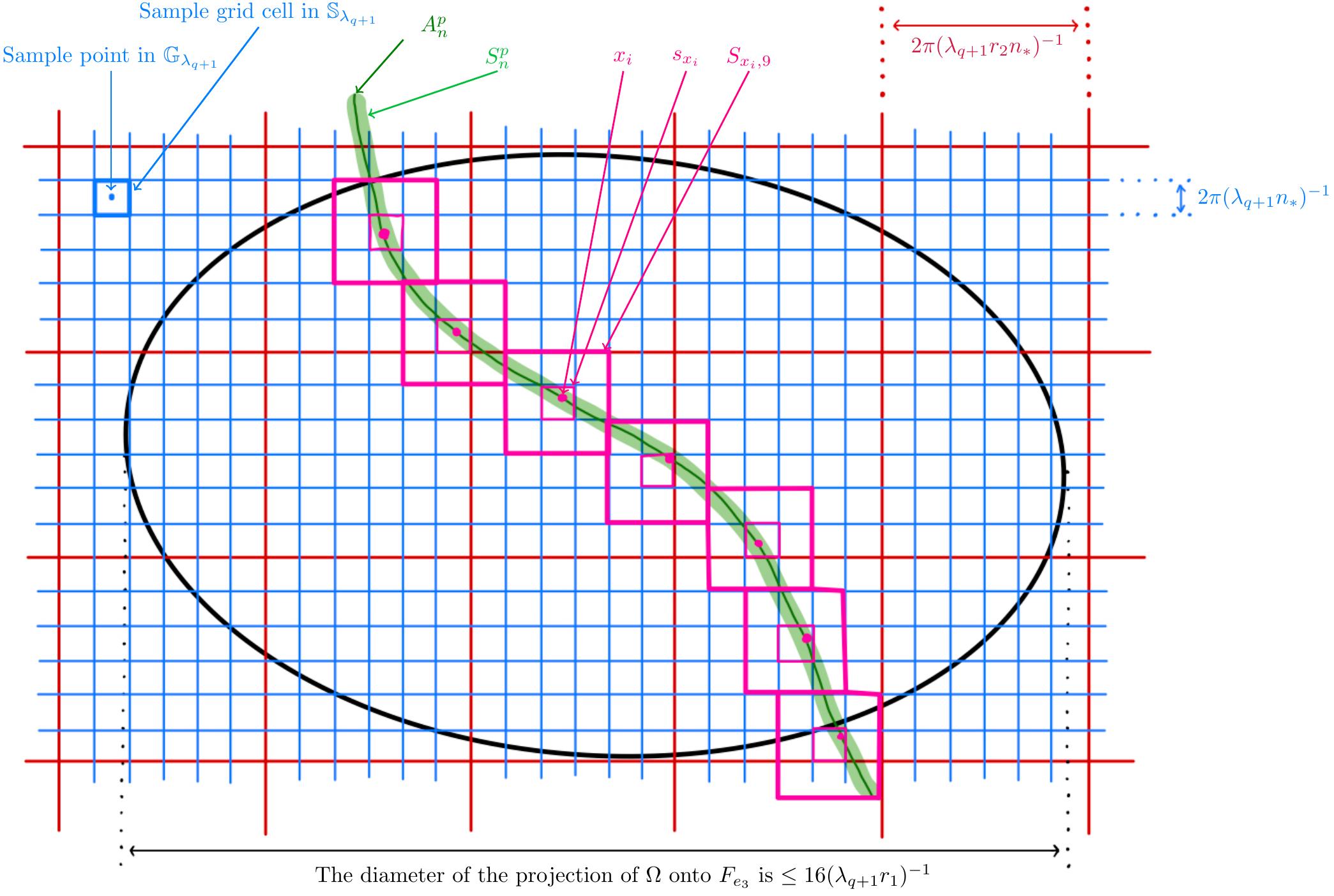}
\caption{\small The boundary of the projection of $\Omega$ onto the face $F_{e_3}$ is represented by the black oval. The blue grid cells represent the elements of $\mathbb{S}_{\lambda_{q+1}}$, while the center points are the elements of $\mathbb{G}_{\lambda_{q+1}}$. A projected pipe $S_{n}^{p}$ with axis $A_{n}^{p}$ is represented in shades of green. A point $x_i \in A_{n}^p$, its associated grid cell $s_{x_i}$, and its $3\times3$ cluster $S_{x_i,9}$ are represented in pink. The union of the pink clusters, $\cup_i S_{x_i,9}$, generously covers $S_{n}^p$.}
\label{fig:covering:1}
\end{figure}

In order to identify a suitable shift $k_0$ such that $\mathbb{S}_{\lambda_{q+1}}^{k_0}$ has empty intersection with $S_n^p$, we first present a generous cover for $S_n^p$; see Figure~\ref{fig:covering:1}. Let $x_1\in A_n^p$ be arbitrary.  Set $s_{x_1}\in\mathbb{S}_{\lambda_{q+1}}$ to be the grid square of sidelength $\frac{2\pi}{\lambda_{q+1}n_*}$ containing $x_1$,\footnote{If $x_1$ is on the boundary of more than one square, any choice of $s_{x_1}$ will work.} and let $S_{x_1,9}$ be the $3\times 3$ cluster of squares surrounding $s_{x_1}$. Then either $x_1$ is within distance $\frac{\twopi}{\lambda_{q+1}n_*}$ of an endpoint of $A_n^p$, or the length of $A_n^p \cap S_{x_1,9}$ is at least $\frac{\twopi}{n_*\lambda_{q+1}}$. If possible, choose $x_2\in A_n^p$ so that $S_{x_2,9}$ is disjoint from $S_{x_1,9}$, and iteratively continue choosing $x_i\in A_n^p$ with $S_{x_i,9}$ disjoint from $S_{x_j,9}$ with $1\leq j \leq i-1$. Due to aforementioned observation about the lower bound on the length of $A_n^p$ in each $S_{x_i,9}$, after a finite number of steps, which we denote by $i_n$, one cannot choose $x_{i_{n+1}}\in A_n^p$ so that $S_{x_{i_{n+1}},9}$ is disjoint from previous clusters; see Figure~\ref{fig:covering:1}. By the length constraint on $A_n^p$ and the observations on the length of $A_n^p\cap S_{x_i,9}$ for each $i$, we obtain the bound
\begin{align*}
    32 (\lambda_{q+1}r_1)^{-1} &\geq |A_n^p|  \geq (i_n-2)\twopi \left(n_*\lambda_{q+1}\right)^{-1}
\end{align*}
which implies that $i_n$ may be bounded from above as
\begin{equation}\label{e:length:constraint}
    i_n \leq \frac{32r_1^{-1}n_*}{\twopi}+2 \leq 6n_*r_1^{-1}+2 \leq 8n_*r_1^{-1}
\end{equation}
since $r_1^{-1} \geq 1$. By the definition of $i_n$, any point $x\in A_n^p$ which does not belong to any of the clusters $\{ S_{x_i,9} \}_{i=1}^{i_n}$, must be such that $S_{x,9}$ has non-empty intersection with $S_{x_j,9}$ for some $j\leq i_n$. 
Thus, if we denote by $S_{x_j,81}$ be the cluster of $9\times 9$ grid squares centered at $x_j$, it follows that $x$ belongs to $S_{x_j,81}$, 
and thus $A_n^p \subset \cup_{i\leq  i_n} S_{x_i,81}$.  Furthermore, since it was observed earlier that $S_n^p$ is contained inside a $2\pi(1+\Gamma_{q+1}^{-1})\left(4n_*\lambda_{q+1}\right)^{-1}$-neighborhood of $A_n^p$, we have in addition that
\begin{equation*}
S_n^p \subset \bigcup_{i=1}^{i_n} S_{x_i,81}.    
\end{equation*}
Thus, we have covered $S_n^p$ using no more than
\begin{equation*}
 81 i_n \leq  81 \cdot 8 n_* r_{1}^{-1} = 648 n_* r_1^{-1} 
\end{equation*}
grid squares. Set $C_\ast=1300 n_*$. Repeating this argument for every $1\leq n \leq N_\Omega$ and taking the union over $n$, we have thus covered $\cup_{n\leq N_\Omega} S_n^p$  using no more than 
\begin{align}
\frac 12 C_\ast \const_A \cdot  r_2^2 r_1^{-2} \cdot r_{1}^{-1} < r_2^{-2}
\label{eq:pigeonhole}
\end{align}
grid squares of sidelength $\frac{2\pi}{\lambda_{q+1} n_*}$; the strict inequality in \eqref{eq:pigeonhole} follows from the assumption \eqref{eq:r1:r2:condition:alt}.
 
\begin{figure}[htb!]
\centering
\includegraphics[width=0.9\textwidth]{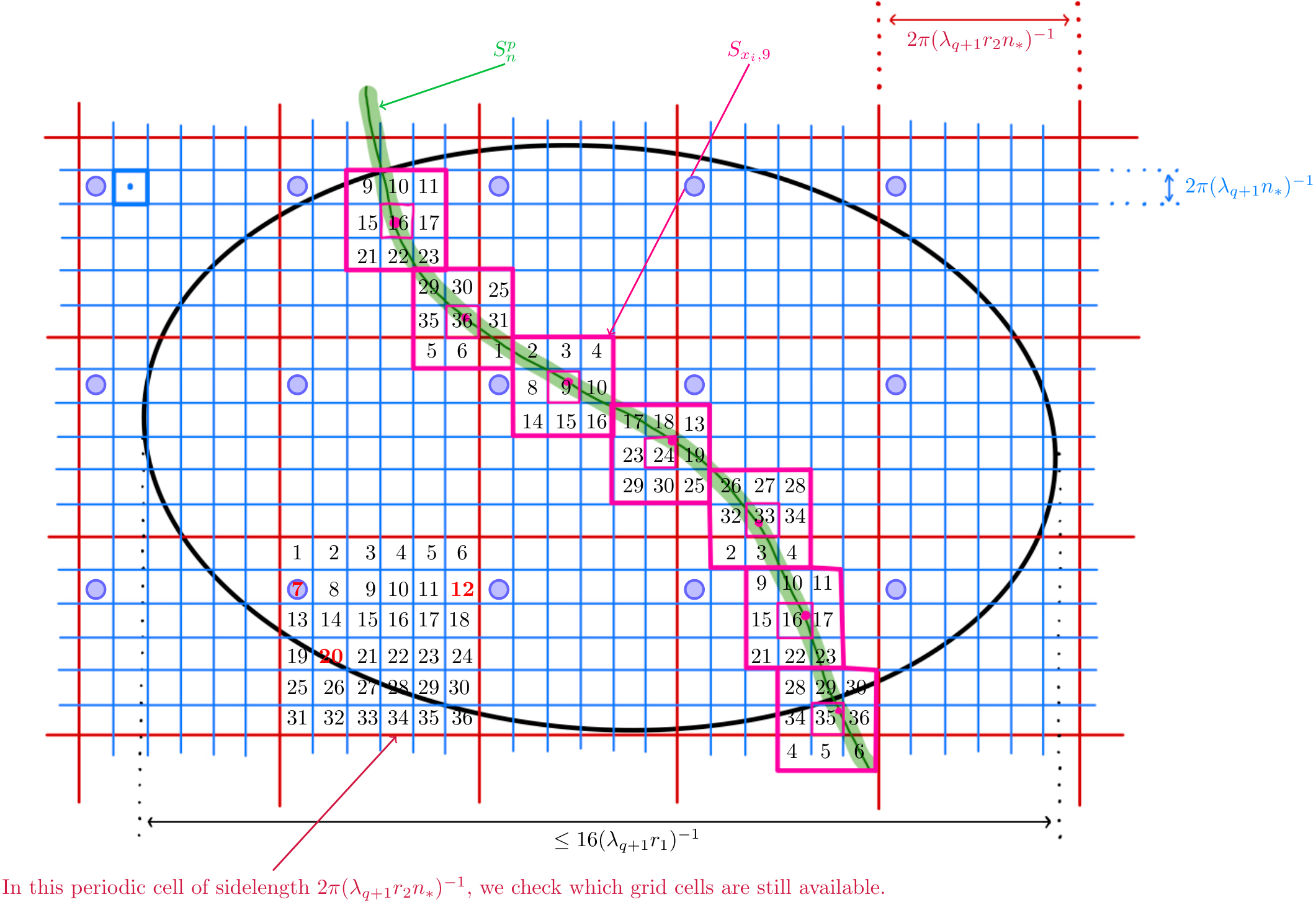}
\caption{\small We revisit Figure~\ref{fig:covering:1}. The union of the pink clusters $S_{x_i,9}$ covers $S_n^p$. We would like to determine which set $\mathbb{S}_{\lambda_{q+1}}^{k_0}$ of $\frac{2\pi}{\lambda_{q+1} r_2 n_*}$-periodic grid cells is free (we index these cells by the shift parameter $k_0$), so that we can place a $\frac{2\pi}{\lambda_{q+1} r_2 n_*}$-periodic pipe flow $\WW_{\lambda_{q+1},r_2,e_3}^{k_0}$ at the centers of the cells.  This pipe flow then will not intersect the cells taken up by the union of the pink clusters $\cup_i S_{x_i,9}$. Towards this purpose, consider one of the red periodic cells of sidelength $\frac{2\pi}{\lambda_{q+1} r_2 n_*}$; e.g.~bottom row, second from left. This cell contains $r_2^{2}$-many blue cells of sidelength $\frac{2\pi}{\lambda_{q+1} n_*}$, which in the figure we index by an integer $k \in \{1,\ldots,36\}$ (that is, $r_2=6$). In order to determine which of these blue cells are ``free,'' we verify for every $k$ whether a periodic copy of the $k$-cell lies in union of the pink clusters $\cup_i S_{x_i,9}$; if yes, we label this index $k$ in black, and we also label with the same number the cell in  $\cup_i S_{x_i,9}$ where this cell appears.  For instance, the cell with label $9$ appears three times within the union of the pink cluster; the cell with label $3$ appears twice; while the cell with label $36$ appears just once. In the above figure we discover that there are only three ``free'' blue cells, corresponding to the red indices $7$, $12$, and $20$. Any of these indices indicates a location where we may place a new pipe flow $\WW_{\lambda_{q+1},r_2,e_3}^{k_0}$; in the figure we have chosen $k_0$ to correspond to the label $7$, and have represented by a $\frac{2\pi}{\lambda_{q+1} r_2 n_*}$-periodic array of purple circles the intersections of the pipes in $\WW_{\lambda_{q+1},r_2,e_3}^{k_0}$ with  $F_{e_3}$.}
\label{fig:covering:2}
\end{figure}

In order to conclude the proof, we appeal to a pigeonhole argument, made possible by the bound
\eqref{eq:pigeonhole}. Indeed, the left side of \eqref{eq:pigeonhole} represents as an upper bound on the number of grid cells in $\mathbb{S}_{\lambda_{q+1}}$  which are deemed ``occupied'' by $\cup_{n\leq N_\Omega} S_n^p$, while the right side of \eqref{eq:pigeonhole} represents the number of possible choices for the shifts $k_0 \in\{0,...,r_2^{-1}-1\}^2$ belonging to the  $\frac{2\pi}{\lambda_{q+1} r_2 n_*}$-periodic  subcollection $\mathbb{S}_{\lambda_{q+1}}^{k_0}$. See Figure~\ref{fig:covering:2} for details. We conclude by \eqref{eq:pigeonhole} and the pigeonhole principle that there exists a ``free'' shift  $k_0 \in\{0,...,r_2^{-1}-1\}^2$ 
such that \emph{none} of the squares in $\mathbb{S}_{\lambda_{q+1}}^{k_0}$ intersect the covering $\cup_{i\leq i_n} S_{x_i,81}$ of $\cup_{n\leq N_\Omega} S_n^p$. Choosing the pipe flow $\WW_{\lambda_{q+1},r_2,e_3}^{k_0}$, we have proven \eqref{e:disjoint:F}, concluding the proof of the lemma when $\xi=e_3$.

To prove the Proposition when $\xi \neq e_3$, first consider the portion\footnote{Recall that $\Omega$ is a $\mathbb{T}^3$-periodic set but can be considered as a subset of $\mathbb{R}^3$, cf.~Definition~\ref{def:periodicity}.} of $\Omega\subset\mathbb{R}^3$ restricted to the cube $[-\pi,\pi]^3$, denoted $\Omega|_{[-\pi,\pi]^3}$, and consider similarly $S_n|_{[-\pi,\pi]^3}$ and $A_n|_{[-\pi,\pi]^3}$.  Let $3\Tthreexi$ be the $3\times 3\times 3$ cluster of periodic cells for $\Tthreexi$ centered at the origin. Then $[-\pi,\pi]^3$ is contained in this cluster, and in particular $[-\pi,\pi]^3$ has \emph{empty} intersection with the boundary of $3\mathbb{T}_\xi^3$ (understood as the boundary of the $3\mathbb{T}_\xi^3$-periodic cell centered at the origin when simply viewed as a subset of $\mathbb{R}^3$).  Thus $\Omega|_{[0,2\pi]^3}$, $S_n|_{[-\pi,\pi]^3}$, and $A_n|_{[-\pi,\pi]^3}$ also have empty intersection  with the boundary of $3\Tthreexi$ and may be viewed as $3\mathbb{T}_\xi^3$-periodic sets.  Up to a dilation which replaces $3\Tthreexi$ with $\Tthreexi$, we have exactly satisfied the assumptions of the proposition, but with $\mathbb{T}^3$-periodicity replaced by $\mathbb{T}_\xi^3$-periodicity.  This dilation will shrink everything by a factor of $3$, which we may compensate for by choosing a pipe flow $\WW_{3\lambda_{q+1},r_2,\xi}$, and then undoing the dilation at the end. Any constants related to this dilation are $q$-independent and may be absorbed into the geometric constant $C_*$ at the end of the proof.  
At this point we may then redo the proof of the proposition with minimal adjustments.  In particular, we replace the  projection of $S_n$ and $A_n$ onto the face $F_{e_3}$ of the box $\mathbb{T}^3$ with the projection of the restricted and dilated versions of $S_n$ and $A_n$ onto the face $F_\xi$ of the box $\mathbb{T}^3_\xi$.  We similarly replace the grids and squares on $F_{e_3}$ with grids and squares on $F_\xi$, exactly analogous to \eqref{e:shifty:support}.  The covering argument then proceeds exactly as before.  The proof produces pipes belonging to the intermittent pipe flow $\WW_{3\lambda_{q+1},r_2,\xi}^{k_0}$ which are $\frac{\mathbb{T}^3}{3\lambda_{q+1} n_* r_2}$-periodic and disjoint from the dilated and restricted versions of the $S_n$'s.  Undoing the dilation, we find that $\WW_{\lambda_{q+1},r_2,\xi}^{k_0}$ is $\frac{\mathbb{T}^3}{\lambda_{q+1}r_2}$-periodic and disjoint from each $S_n$.  Then all the conclusions of Proposition~\ref{prop:disjoint:support:simple:alternate} have been achieved, finishing the proof.
\end{proof}

\section{Mollification}
\label{sec:mollification:stuff}

Because the principal inductive assumptions for the velocity increments \eqref{eq:inductive:assumption:derivative:q} and the Reynolds stress~\eqref{eq:Rq:inductive:assumption} are only assumed to hold for a limited number of space and material derivatives ($\leq 7 \Nindv$ and $\leq 3 \Nindv$ respectively), and because in our proof we need to appeal to derivative bounds of much higher orders, it is customary to employ a {\em mollification step} prior to adding the convex integration perturbation. This mollification step is discussed in Lemma~\ref{lem:mollifying:ER}. Note that the mollification step is only employed once (for every inductive step $q\mapsto q+1$), and is not repeated for the higher order stresses $R_{q,n,p}$. In particular, Lemma~\ref{lem:mollifying:ER} already shows that the inductive assumption~\eqref{eq:inductive:assumption:derivative} holds for $q'=q$. 

\begin{lemma}[\textbf{Mollifying the Euler-Reynolds system}]\label{lem:mollifying:ER}
Let $(v_q,\RR_q)$ solve the Euler-Reynolds system \eqref{eq:Euler:Reynolds:again}, and assume that $\psi_{i,q'}, u_{q'}$ for $q'<q$, $w_q$, and $\RR_q$ satisfy \eqref{eq:inductive:assumption:derivative}--\eqref{eq:nasty:Dt:wq:WEAK:old}. Then, we mollify $(v_q,\RR_q)$ at spatial scale $\tilde \lambda_q^{-1}$ and temporal scale $\tilde \tau_{q-1}$ (cf.~the notation in \eqref{mollifier:operators}), and accordingly define 
\begin{align}
\vlq:=\Pqxt v_q
\qquad \mbox{and} \qquad 
\RR_{\ell_q}:=\Pqxt \RR_q \,.
\label{eq:vlq:Rlq:def}
\end{align}
The mollified pair $(v_{\ell_q},\RR_{\ell_q})$ satisfy
\begin{subequations}
\label{eq:vlq:equation}
\begin{align}
\partial_t \vlq + \div(\vlq\otimes\vlq) +\nabla p_{\ell_q} 
&= \div \RR_{\ell_q} + \div \RR_{q}^{\textnormal{comm}} \,,\\
\div \vlq &= 0 
\,.
\end{align}
\end{subequations}
The commutator stress $\RR_q^{\textnormal{comm}}$ satisfies the estimate (consistent with \eqref{eq:Rq:inductive:assumption} at level $q+1$)
\begin{align}
\label{eq:Rqcomm:bound}
\left\| D^n \Dtq^m \RR_q^\textnormal{comm} \right\|_{L^\infty} 
\leq  \Gamma_{q+1}^{-1}  \shaqqplusone \delta_{q+2} \lambda_{q+1}^n 
\MM{m, \Nindt, \tau_{q}^{-1}, \Gamma_q^{-1}\tilde \tau_q^{-1} }
\end{align}
for all  $n, m \leq 3 \Nindv$, and the we have that
\begin{align}
\label{eq:vq:minus:mollified}
\norm{D^n D_{t,q-1}^m (v_{\ell_q} - v_q)}_{L^\infty} 
\leq 
\lambda_q^{-2}   \delta_{q}^{\sfrac 12} \MM{n,2\Nindv,\lambda_q,\tilde \lambda_q} \MM{m, \Nindvt ,\tau_{q-1}^{-1}\Gamma_{q}^{i-1}, \Tilde{\tau}_{q-1}^{-1} \Gamma_q^{-1}}
\end{align}
for all  $n, m \leq 3 \Nindv$. 
Furthermore, 
\[
u_q = v_{\ell_q} - v_{\ell_{q-1}}
\] 
satisfies the bound \eqref{eq:inductive:assumption:derivative} with $q'$ replaced by $q$, namely 
\begin{align}\label{eq:mollified:velocity}
    \left\| \psi_{i,q-1} D^n D_{t,q-1}^m u_q \right\|_{L^2} \leq \delta_{q}^{\sfrac{1}{2}} \MM{n, 2 \Nindv,\lambda_{q},\tilde{\lambda}_q} \MM{m, \Nindvt ,\Gamma_{q}^i \tau_{q-1}^{-1}, \Tilde{\tau}_{q-1}^{-1}}.
\end{align}
for all $n+m \leq 2\Nfin$. In fact, when either $n\geq 3 \Nindv$ or $m\geq 3 \Nindv$ are such that $n+m\leq 2\Nfin$, then the above estimate holds uniformly
\begin{align}\label{eq:mollified:velocity:sup}
    \left\|  D^n D_{t,q-1}^m u_q \right\|_{L^\infty} \leq \Gamma_q^{-1} \delta_{q}^{\sfrac{1}{2}} \MM{n, 2 \Nindv,\lambda_{q},\tilde{\lambda}_q} \MM{m, \Nindvt , \tau_{q-1}^{-1} , \Tilde{\tau}_{q-1}^{-1}}.
\end{align}
Finally, $\RR_{\ell_q}$ satisfies bounds which extend \eqref{eq:Rq:inductive:assumption} to  
\begin{align}\label{eq:mollified:stress:bounds}
    \left\| \psi_{i,q-1} D^n D_{t,q-1}^m \RR_{\ell_q} \right\|_{L^1} \lessg \Gamma_q^\shaq \delta_{q+1} \MM{n,  2 \Nindv ,\lambda_q,\Tilde{\lambda}_q}\MM{m, \NindRt, \Gamma_q^{i+2} \tau_{q-1}^{-1}, \Tilde{\tau}_{q-1}^{-1}} 
\end{align}
for all $n+m\leq 2\Nfin$. In fact, the above estimate holds uniformly
\begin{align}\label{eq:mollified:stress:bounds:sup}
    \left\|  D^n D_{t,q-1}^m \RR_{\ell_q} \right\|_{L^\infty} \lessg \Gamma_q^{-1} \Gamma_q^\shaq \delta_{q+1} \MM{n,  2 \Nindv ,\lambda_q,\Tilde{\lambda}_q}\MM{m, \NindRt,  \tau_{q-1}^{-1}, \Tilde{\tau}_{q-1}^{-1}} 
\end{align}
whenever either  $n\geq 3 \Nindv$ or $m\geq 3 \Nindv$ are such that $n+m\leq 2\Nfin$. 
\end{lemma}

\begin{remark}[\textbf{$L^\infty$ estimates on the support of $\psi_{i,q-1}$}]
The bounds \eqref{eq:mollified:velocity:sup} and \eqref{eq:mollified:stress:bounds:sup} provide $L^\infty$ estimates for $D^n D_{t,q-1}^{m}$ applied to $u_q$ and respectively $\RR_{\ell_q}$, but only when either $n$ or $m$ are sufficiently large. 
In the remaining cases, we  note that~\eqref{eq:mollified:velocity}, combined with the partition of unity property \eqref{eq:inductive:partition}, and the  inductive assumption~\eqref{eq:sharp:Dt:psi:i:q:old} (with $M=0$, and $K=4$), implies the bound 
\begin{align}
 \norm{D^{n} D^m_{t,q-1} u_q}_{L^\infty(\supp \psi_{i,q-1})} \lesssim  \delta_{q}^{\sfrac 12}
 \tilde \lambda_{q}^{\sfrac 32}\MM{n ,2\Nindv,\lambda_{q}, \tilde \lambda_q} \MM{m,\Nindt, \tau_{q-1}^{-1} \Gamma_{q}^{i+1}, \tilde \tau_{q-1}^{-1}}  
\label{eq:D:K:psi:i:q}
\end{align}
for all $n, m \leq 3\Nindv$. Indeed, we may apply Lemma~\ref{lem:Sobolev:cutoffs} (estimate \eqref{eq:L2:to:Linfty}) with $\psi_i = \psi_{i,q-1}$, $f = u_{q}$, $C_f = \delta_q^{\sfrac 12}$, $\rho = \lambda_{q-1} \Gamma_{q-1} \leq \lambda_q$ (cf.~\eqref{eq:Lambda:q:x:1}), $\lambda = \lambda_q$, $\tilde \lambda = \tilde \lambda_q$, $\mu_i = \tau_{q}^{-1} \Gamma_q^i$, $\tilde \mu_i = \tilde \tau_{q-1}^{-1}$, $N_x = 2 \Nindv$, $N_t = \Nindt$, and $N_\circ = 2\Nfin$, to conclude that  \eqref{eq:D:K:psi:i:q} holds for all $n+m\leq 2\Nfin - 2$, and in particular for $n,m \leq 3\Nindv$.

A similar argument, shows that estimate \eqref{eq:mollified:stress:bounds} and Lemma~\ref{lem:Sobolev:cutoffs} imply
\begin{align}
 \norm{D^{n} D^m_{t,q-1} \RR_{\ell_q}}_{L^\infty(\supp \psi_{i,q-1})} 
 \lesssim \Gamma_q^\shaq \delta_{q+1} \tilde \lambda_{q}^3  \MM{n ,  2 \Nindv ,\lambda_q,\Tilde{\lambda}_q}\MM{m, \NindRt, \Gamma_q^{i+3} \tau_{q-1}^{-1}, \Tilde{\tau}_{q-1}^{-1}}  
 \label{eq:Reynolds:L:infty:lossy}
\end{align}
for $n+m \leq 2\Nfin - 4$, and in particular for $n , m \leq 3\Nindv$.
\end{remark}

\begin{proof}[Proof of Lemma~\ref{lem:mollifying:ER}]
The bound \eqref{eq:Rqcomm:bound} requires a different proof than \eqref{eq:mollified:velocity} and \eqref{eq:mollified:stress:bounds}, so that we start with the former.

\textbf{Proof of \eqref{eq:Rqcomm:bound}.\,} Recall that 
\begin{align}
    \label{eq:RQ:comm:DEF}
\RR_q^{\textnormal{comm}} = \Pqxt v_q \mathring \otimes \Pqxt v_q - \Pqxt( v_q\mathring\otimes v_q) \,.
\end{align}
We note cf.~\eqref{mollifier:operators} that $\Pqxt$ mollifies in space at length scale $\tilde \lambda_q$, and in time at time scale $\tilde \tau_{q-1}^{-1}$. Let us denote by $K_{q}$ the space-time mollification kernel for $\Pqxt$, which thus equals the product of the bump functions $\phi_{\tilde \lambda_q}^{(x)} \phi_{\tilde \tau_{q-1}^{-1}}^{(t)}$. 
For brevity of notation, (locally in this proof) it is convenient to denote space-time points as $(x,t),(y,s),(z,r)\in\mathbb{T}^3\times\mathbb{R}$  
\begin{align}
\label{eq:strange:notation} 
(x,t) = \theta, \qquad (y,s) = \kappa, \qquad (z,r) = \zeta. 
\end{align}
Using this notation we may write out the commutator stress $\RR_q^{\textnormal{comm}}$ explicitly,  and symmetrizing the resulting expression leads to the formula
\begin{align}\label{eq:commutator:stress:1}
\RR_q^{\textnormal{comm}}(\theta) = \frac{-1}{2} \int\!\!\!\int_{(\mathbb{T}^3\times\mathbb{R})^2} \left( v_q(\theta-\kappa)-v_q(\theta-\zeta) \right) \mathring\otimes \left( v_q(\theta-\kappa)-v_q(\theta-\zeta) \right) K_q(\kappa)K_q(\zeta) \,d\kappa\,d\zeta \,.
\end{align}
Expanding $v_q$ in a Taylor series in space and time around $\theta$ yields the formula
\begin{align}
    v_q(\theta-\kappa) = v_q(\theta) + \sum_{|\alpha|+m=1}^{{N_{\textnormal{c}}-1}} \frac{1}{\alpha! m!}D^\alpha \partial_t^m v_q(\theta) (-\kappa)^{(\alpha,m)} + R_{N_\textnormal{c}}(\theta,\kappa)
\end{align}
where the remainder term with $N_\textnormal{c}$ derivatives is given by
\begin{align}
    R_{N_{\textnormal{c}}}(\theta,\kappa) = \sum_{|\alpha| + m =N_\textnormal{c}} \frac{N_{\textnormal{c}}}{\alpha! m!} (-\kappa)^{(\alpha,m)} \int_0^1 (1-\eta)^{N_\textnormal{c}-1} D^\alpha \partial_t^m v_q (\theta-\eta\kappa) \,d\eta.
    \label{eq:R:q:comm:remainder}
\end{align}
The value of $N_\textnormal{c}$ will be chosen later so that $ \Nindt \ll N_\textnormal{c} = \Nindv -2$, more precisely, such that conditions \eqref{eq:N:c:condition:1} and \eqref{eq:N:c:condition:2} hold. 

Using that by \eqref{eq:phi} all moments of $K_q$ vanish up to order $N_\textnormal{c}$, we rewrite \eqref{eq:commutator:stress:1} as 
\begin{align}
\RR_q^\textnormal{comm}(\theta) 
&= \int_{\mathbb{T}^3\times\mathbb{R}}  \sum_{|\alpha|+m=1}^{N_\textnormal{c}-1}\frac{(-\kappa)^{(\alpha,m)}}{\alpha! m!}   D^\alpha \partial_t^m v_q(\theta) \symring R_{N_\textnormal{c}}(\theta,\kappa)   K_q(\kappa) \,d\kappa \notag \\
&\quad -  \int_{\mathbb{T}^3\times\mathbb{R}}   R_{N_\textnormal{c}}(\theta,\kappa) \mathring\otimes  R_{N_\textnormal{c}}(\theta,\kappa) K_q(\kappa) \,d\kappa \notag\\
&\quad - \int\!\!\!\int_{(\mathbb{T}^3\times\mathbb{R})^2}   R_{N_\textnormal{c}}(\theta,\kappa) \mathring\otimes_{\mathrm{sym}} R_{N_\textnormal{c}}(\theta,\zeta)K_q(\kappa)K_q(\zeta)   \,d\kappa\,d\zeta 
\notag\\
&=: \RR_{q,1}^\textnormal{comm}(\theta) + \RR_{q,2}^\textnormal{comm}(\theta) 
 + \RR_{q,3}^\textnormal{comm}(\theta) \,,
\label{eq:commutator:stress:2}
\end{align}
where we have used the notation \eqref{eq:otimes:symm}.

In order to prove \eqref{eq:Rqcomm:bound}, we first show that every term in $D^n\Dtq^m\RR_q^\textnormal{comm}$ can be decomposed into products of pure space and time differential operators applied to products of $\vlq$ and $v_q$. 
More generally, for any sufficiently smooth function $F = F(x,t)$ and for any $n,m\geq 0$, the Leibniz rule implies that
\begin{subequations}
\begin{align}
D^n\Dtq^m F
&=D^n (\partial_t + v_{\ell_q}\cdot \nabla_x)^m F
= \sum_{\substack{m'\leq m \\ n'+m' \leq n+m}} d_{n,m,n',m'}(x,t) D^{n'} \partial_t^{m'} F 
\label{eq:RNC} \\
d_{n,m,n',m'}(x,t) 
&= \sum_{k = 0}^{m-m'} \sum_{\substack{\{ \gamma \in {\mathbb N}^k \colon |\gamma| = n-n'+k,\\ \beta \in {\mathbb N}^k \colon |\beta| = m-m'-k\}}} c(m,n, k, \gamma, \beta)  \prod_{\ell=1}^{k} \left( D^{\gamma_\ell} \partial_t^{\beta_\ell}  \vlq(x,t) \right)
\label{eq:DNC}
\end{align}
\end{subequations}
where $c(m,n, k, \gamma, \beta)$ denotes an explicitly computable combinatorial coefficient  which depends only on the factors inside the parenthesis, and are in particular independent of $q$ (which is why we do not carefully track these coefficients). Identity \eqref{eq:RNC}--\eqref{eq:DNC} holds because $D$ and $\partial_t$ commute; the proof is based on induction on $n$ and $m$. Clearly, if $D_{t,q}$ in \eqref{eq:RNC} is replaced by $D_{t,q-1}$, then the same formula holds, with the $v_{\ell_q}$ factors in \eqref{eq:DNC} being replaced by $v_{\ell_{q-1}}$.

In order to prove \eqref{eq:Rqcomm:bound} we consider \eqref{eq:RNC}--\eqref{eq:DNC} for $n , m \leq 3 \Nindv$ and with $F = \RR_q^\textnormal{comm}$. In order to estimate the factors $d_{n,m,n',m'}$ in \eqref{eq:DNC}, we need to bound  $D^n \partial_t^m v_q$ for $n\leq  6 \Nindv + N_{\textnormal{c}}  $ and $m \leq  3 \Nindv + N_{\textnormal{c}} $, with $n+m \leq  6 \Nindv + N_{\textnormal{c}}$. Recall that $v_{q} = w_q + v_{\ell_{q-1}}$ and thus we will obtain the needed estimate from bounds on $D^n \partial_t^m w_q$ and $D^n \partial_t^m v_{\ell_{q-1}}$. We start with the latter. 

We recall that $v_{\ell_{q-1}} = w_{q-1} + v_{\ell_{q-2}}$.  Using \eqref{eq:inductive:partition} with $q'=q-2$ and the inductive assumption \eqref{eq:inductive:assumption:derivative:q} with $q$ replaced with $q-1$, we obtain from Sobolev interpolation that $\norm{w_{q-1}}_{L^\infty} \les \norm{w_{q-1}}_{L^2}^{\sfrac 14} \norm{D^2 w_{q-1}}_{L^2}^{\sfrac 34} \les \delta_{q-1}^{\sfrac 12} \lambda_{q-1}^{\sfrac 32}$.  Additionally, combining \eqref{eq:bob:Dq':old} with $q'=q-2$ and \eqref{eq:imax:old} with $q'=q-2$, we obtain $\norm{v_{\ell_{q-2}}}_{L^\infty} \les \lambda_{q-2}^2 \Gamma_{q-1}^{\imax + 1} \delta_{q-1}^{\sfrac 12} \les \lambda_{q-2}^4 \delta_{q-1}^{\sfrac 12}$. Jointly, these two estimate imply
\begin{align*}
\norm{v_{q-1}}_{L^\infty} \les \norm{w_{q-1}}_{L^\infty} + \norm{v_{\ell_{q-2}}}_{L^\infty} \les \delta_{q-1}^{\sfrac 12} \lambda_{q-1}^4\,.
\end{align*}
Now, using that $v_{\ell_{q-1}} =  \mathcal{P}_{q-1,x,t} v_{q-1}$, and that the mollifier operator $\mathcal{P}_{q-1,x,t}$ localizes at scale $\tilde \lambda_{q-1}$ in space and $\tilde \tau_{q-2}^{-1}$ in time, we deduce the global estimate 
\begin{align}
\norm{D^n \partial_t^m v_{\ell_{q-1}}}_{L^\infty}
\les (\lambda_{q-1}^4 \delta_{q-1}^{\sfrac 12}) 
\tilde \lambda_{q-1}^n 
\tilde \tau_{q-2}^{-m}
 \label{eq:space:time:v:ell:q-1:rough}
\end{align}
for  $n+m \leq 2\Nfin $. Note that from the definitions \eqref{eq:tilde:lambda:q:def} and \eqref{eq:tilde:tau:q:def}, it is immediate that $\tilde \tau_{q-2}^{-1} \ll \Gamma_q^{-1} \tilde \tau_{q-1}^{-1}$. 

As mentioned earlier, the bound for the space-time derivatives of $v_{\ell_{q-1}}$ needs to be combined with similar estimates for $w_q$ in order to yield a control of $v_q$. For this purpose, we appeal to the Sobolev embedding $H^2 \subset L^\infty$ and the bound \eqref{eq:inductive:assumption:derivative:q} (in which we take a supremum over $0 \leq i \leq \imax$ and use \eqref{eq:Lambda:q:t:1}) to deduce
\begin{align}
\norm{D^{n} D^m_{t,q-1} w_{q}}_{L^\infty}
\les 
\norm{D^{n} D^m_{t,q-1} w_{q}}_{H^2}
\les (\delta_{q}^{\sfrac 12} \lambda_{q}^{2}) \lambda_q^n  (\tilde{\tau}_{q-1}^{-1}\Gamma_q^{-1})^m
\label{eq:space:material:w:q:rough}
\end{align}
for all $n \leq 7 \Nindv - 2$ and $m \leq 7 \Nindv$. Using the above estimate we may  apply Lemma~\ref{lem:cooper:1} with the decomposition $\partial_t = - v_{\ell_{q-1}} \cdot \nabla + D_{t,q-1} = A + B$, $v = - v_{\ell_{q-1}}$ and $f = w_q$. The conditions \eqref{eq:cooper:v}  in Lemma~\ref{lem:cooper:1} holds in view of the inductive estimate \eqref{eq:bob:Dq':old} at level $q-1$, with the following choice of parameters: $p = \infty$, $\Omega = \T^3$, $\const_v  = \lambda_{q-1}^4 \delta_{q-1}^{\sfrac 12}$, $N_x = \Nindv-2$, $\lambda_v  = \Gamma_{q-1} \lambda_{q-1}$, $\tilde \lambda_v  = \tilde \lambda_{q-1}$, $N_t = \Nindt$, $ \mu_v  = \lambda_{q-1}^2 \tau_{q-1}^{-1}$, $\tilde \mu_v =  \Gamma_q^{-1} \tilde \tau_{q-1}^{-1}$, and $N_* = \sfrac{3\Nfin}{2}$.
On the other hand, using \eqref{eq:space:material:w:q:rough} we have that  condition \eqref{eq:cooper:f}  holds  with the  parameters: $p = \infty$, $\Omega = \T^3$, $\const_f  = \delta_{q}^{\sfrac 12} \lambda_{q}^{2}$,  $\lambda_f = \tilde \lambda_f = \lambda_q$, $ \mu_f  = \tilde \mu_f = \Gamma_q^{-1} \tilde \tau_{q-1}^{-1}$, and $N_* = 7 \Nindv-2$. We deduce from \eqref{eq:cooper:f:*} and the inequalities $\tilde \lambda_{q-1} \leq \lambda_q$ and $\lambda_{q-1}^4 \delta_{q-1}^{\sfrac 12} \lambda_q \leq \Gamma_q^{-1} \tilde \tau_{q-1}^{-1} $ (cf.~\eqref{eq:Lambda:q:x:1:NEW},~\eqref{eq:Lambda:q:t:1}, and~\eqref{eq:tilde:tau:q:def}), that 
\begin{align}
\norm{D^{n} \partial_t^m w_{q}}_{L^\infty}
 \les (\delta_{q}^{\sfrac 12} \lambda_{q}^{2}) \lambda_q^n  (\tilde{\tau}_{q-1}^{-1}\Gamma_q^{-1})^m
\label{eq:space:time:w:q:rough}
\end{align}
holds for $n + m \leq 7 \Nindv-2$. 

By combining \eqref{eq:space:time:v:ell:q-1:rough} and \eqref{eq:space:time:w:q:rough} with the definition \eqref{eq:cutoffs:wu} we thus deduce 
\begin{align}
\norm{D^{n} \partial_t^m v_{q}}_{L^\infty}
 \les (\lambda_{q-1}^4  \delta_{q-1}^{\sfrac 12}) \lambda_q^n  (\tilde{\tau}_{q-1}^{-1}\Gamma_q^{-1})^m
\label{eq:space:time:v:q:rough}
\end{align}
for all $n + m \leq 7 \Nindv-2$, where we have used that  $\lambda_{q-1}^4  \delta_{q-1}^{\sfrac 12} \geq \delta_{q}^{\sfrac 12} \lambda_{q}^{2}$ and that $\tilde\tau_{q-2}^{-1} \leq \Gamma_q^{-1} \tilde \tau_{q-1}^{-1}$. By the definition of $v_{\ell_q}$ in \eqref{eq:vlq:Rlq:def} we thus also deduce that 
\begin{align}
\norm{D^{n} \partial_t^m v_{\ell_q}}_{L^\infty}
 \les (\lambda_{q-1}^4   \delta_{q-1}^{\sfrac 12}) \lambda_q^n  (\tilde{\tau}_{q-1}^{-1}\Gamma_q^{-1})^m
\label{eq:space:time:v:ell:q:rough}
\end{align}
for all $n + m \leq 7 \Nindv-2$. Note that by the definition of the mollifier operator $\Pqxt$, any further space derivative on $v_{\ell_{q}}$ costs a factor of $\tilde \lambda_q$, while additional temporal derivatives cost $\tilde \tau_{q-1}$, up to a $2\Nfin$ total number of derivatives.

With \eqref{eq:space:time:v:ell:q:rough} in hand, we may return to \eqref{eq:DNC} and deduce that for $n , m\leq 3 \Nindv$,  we have
\begin{align}
 \norm{d_{n,m,n',m'}}_{L^\infty} 
 &\les \sum_{k=0}^{m-m'}  \lambda_q^{n-n'+k} (\tilde \tau_{q-1}^{-1} \Gamma_q^{-1} )^{m-m'-k} (\lambda_{q-1}^4 \Gamma_q \delta_{q-1}^{\sfrac 12})^k\notag\\
 &\les    \lambda_q^{ n-n'} (\tilde \tau_{q-1}^{-1}  \Gamma_q^{-1})^{m-m'} \,.
 \label{eq:spindrift:1}
\end{align}
In the last inequality above we have used that $\lambda_q  \lambda_{q-1}^4 \Gamma_q \delta_{q-1}^{\sfrac 12} \leq \tilde \tau_{q-1}^{-1}  \Gamma_q^{-1}$, which is a consequence of \eqref{eq:Lambda:q:x:1:NEW},~\eqref{eq:Lambda:q:t:1}, and~\eqref{eq:tilde:tau:q:def}.

Returning to \eqref{eq:RNC} with $F = \RR_q^\textnormal{comm}$, we   use the expansion in \eqref{eq:commutator:stress:2}, the definition \eqref{eq:R:q:comm:remainder}, and the bound \eqref{eq:space:time:v:q:rough} to estimate $D^{n'} \partial_t^{m'} \RR_q^\textnormal{comm}$ when $n', m' \leq 3 \Nindv$. Using \eqref{eq:space:time:v:q:rough} and the choice
\begin{align}
N_{\textnormal{c}} =  \Nindv - 2  \,,
\label{eq:N:c:condition:1}
\end{align}
which is required in order to ensure that $n' + m' + N_c \leq 7 \Nindv-2$, we first obtain the pointwise estimate 
\begin{align}
\abs{D^{n''} \partial_{t}^{m''} R_{N_{\textnormal{c}}}(\theta,\kappa)}
\les (\lambda_{q-1}^4  \delta_{q-1}^{\sfrac 12}) \sum_{|\alpha|+m_1 = N_{\textnormal{c}}} \abs{\kappa^{(\alpha,m_1)}} \lambda_q^{n''+|\alpha|} (\tilde \tau_{q-1}^{-1} \Gamma_q^{-1})^{m'' + m_1}
\,,
\label{eq:derivatives:of:RNC}
\end{align}
where we recall the notation in \eqref{eq:strange:notation}. Using \eqref{eq:derivatives:of:RNC}, the Leibniz rule, and the fact that $\lambda_q \Gamma_q \leq \tilde \lambda_q$,  we may estimate
\begin{align*}
&\norm{D^{n'} \partial_t^{m'}  \RR_{q,2}^\textnormal{comm}  }_{L^\infty} \notag\\
&\les (\lambda_{q-1}^4  \delta_{q-1}^{\sfrac 12})^2 
  \sum_{|\alpha|+m_1 = N_{\textnormal{c}}} \sum_{|\alpha'|+m_2 = N_{\textnormal{c}}}   
 \lambda_q^{n'+|\alpha|+|\alpha'|} (\tilde \tau_{q-1}^{-1} \Gamma_q^{-1})^{m'+m_1+m_2} 
 \int_{\T^3\times \R} |\kappa^{(\alpha+\alpha',m_1+m_2)}| |K_q(\kappa)| d\kappa 
 \notag\\
 &\les (\lambda_{q-1}^4   \delta_{q-1}^{\sfrac 12})^2 
  \sum_{|\alpha|+m_1 = N_{\textnormal{c}}} \sum_{|\alpha'|+m_2 = N_{\textnormal{c}}}   
 \lambda_q^{n'+|\alpha|+|\alpha'|} (\tilde \tau_{q-1}^{-1} \Gamma_q^{-1})^{m'+m_1+m_2}  
 \tilde \lambda_q^{-|\alpha|-|\alpha'|} \tilde \tau_{q-1}^{m_1+m_2} 
 \notag\\
  &\les (\lambda_{q-1}^4 \delta_{q-1}^{\sfrac 12})^2 \lambda_q^{n'} (\tilde \tau_{q-1}^{-1} \Gamma_q^{-1})^{m'} \Gamma_q^{-2N_{\textnormal{c}}}
\end{align*}
whenever $n', m' \leq 3 \Nindv$.  It is clear that  a very similar argument  also gives the bound 
\begin{align*}
\norm{D^{n'} \partial_t^{m'}  \RR_{q,3}^\textnormal{comm}  }_{L^\infty} 
\les  (\lambda_{q-1}^4  \delta_{q-1}^{\sfrac 12})^2 \lambda_q^{n'} (\tilde \tau_{q-1}^{-1} \Gamma_q^{-1})^{m'} \Gamma_q^{-2N_{\textnormal{c}}}
\end{align*}
for the same range of $n'$ and $m'$. Lastly, by combining \eqref{eq:derivatives:of:RNC}, \eqref{eq:space:time:v:q:rough}, and the Leibniz rule, we similarly deduce
\begin{align*}
&\norm{D^{n'} \partial_t^{m'}  \RR_{q,1}^\textnormal{comm}  }_{L^\infty}
\notag\\
&\les (\lambda_{q-1}^4   \delta_{q-1}^{\sfrac 12})^2 
  \sum_{|\alpha|+m_1 = 1}^{N_{\textnormal{c}}-1} \sum_{|\alpha'|+m_2 = N_{\textnormal{c}}}   
 \lambda_q^{n'+|\alpha|+|\alpha'|} (\tilde \tau_{q-1}^{-1} \Gamma_q^{-1})^{m'+m_1+m_2}  
 \int_{\T^3\times \R} \abs{\kappa^{(\alpha+\alpha',m_1+m_2)}} |K_q(\kappa)| d\kappa 
 \notag\\
 &\les (\lambda_{q-1}^4  \delta_{q-1}^{\sfrac 12})^2 
 \sum_{|\alpha|+m_1 = 1}^{N_{\textnormal{c}}-1} \sum_{|\alpha'|+m_2 = N_{\textnormal{c}}}   
 \lambda_q^{n'+|\alpha|+|\alpha'|} (\tilde \tau_{q-1}^{-1} \Gamma_q^{-1})^{m'+m_1+m_2}  
 \tilde \lambda_q^{-|\alpha|-|\alpha'|} \tilde \tau_{q-1}^{m_1+m_2} 
 \notag\\
  &\les (\lambda_{q-1}^4 \delta_{q-1}^{\sfrac 12})^2 \lambda_q^{n'} (\tilde \tau_{q-1}^{-1} \Gamma_q^{-1})^{m'} \Gamma_q^{-N_{\textnormal{c}}-1} \,.
\end{align*}
Combining the above three bounds, identity \eqref{eq:commutator:stress:2} yields 
\begin{align}
\norm{D^{n'} \partial_t^{m'}  \RR_{q}^\textnormal{comm}  }_{L^\infty}
\les   (\lambda_{q-1}^4 \delta_{q-1}^{\sfrac 12})^2 \lambda_q^{n'} (\tilde \tau_{q-1}^{-1} \Gamma_q^{-1})^{m'} \Gamma_q^{-N_{\textnormal{c}}-1} 
\label{eq:derivatives:of:RNC:2}
\end{align}
whenever $n', m' \leq 3 \Nindv$. 

Lastly, by combining \eqref{eq:RNC} with \eqref{eq:spindrift:1} and \eqref{eq:derivatives:of:RNC:2} we obtain 
\begin{align*}
\norm{D^{n} D_{t,q}^{m}  \RR_{q}^\textnormal{comm}}_{L^\infty}
\les (\lambda_{q-1}^4  \delta_{q-1}^{\sfrac 12})^2 \lambda_q^{n} (\tilde \tau_{q-1}^{-1} \Gamma_q^{-1})^{m} \Gamma_q^{-N_{\textnormal{c}}-1} 
\end{align*}
for all $n , m \leq 3 \Nindv$. Therefore, in order to verify \eqref{eq:Rqcomm:bound}, we need to verify that 
\begin{align*}
(\lambda_{q-1}^4  \delta_{q-1}^{\sfrac 12})^2 \lambda_q^{n} (\tilde \tau_{q-1}^{-1} \Gamma_q^{-1})^{m} \Gamma_q^{-N_{\textnormal{c}}}  \leq \Gamma_{q+1}^{-1}  \shaqqplusone \delta_{q+2} \lambda_{q+1}^n \MM{m, \Nindt, \tau_{q}^{-1}, \tilde \tau_q^{-1} \Gamma_q^{-1}}
\end{align*}
for all $0 \leq n , m \leq 3 \Nindv$. Since $\lambda_q \leq \lambda_{q+1}$, $\tilde \tau_{q-1}^{-1} \leq \tilde \tau_q^{-1}$,  and $\tilde \tau_{q-1}^{-1} \Gamma_q^{-1} \geq \tau_q^{-1} \geq \tau_{q-1}^{-1}$, the above condition is ensured by the more restrictive condition 
\begin{align}
 \lambda_{q-1}^8 \Gamma_{q+1}^{1+{\mathsf{C_R}}} \frac{\delta_{q-1}}{\delta_{q+2}}  \left(\frac{\tilde \tau_{q-1}^{-1} \Gamma_q^{-1} }{\tau_q^{-1}}\right)^{\Nindt} 
\leq  \lambda_{q-1}^8 \Gamma_{q+1}^{1+{\mathsf{C_R}}} \frac{\delta_{q-1}}{\delta_{q+2}}  \left(\frac{\tilde \tau_{q-1}^{-1}  }{\tau_{q-1}^{-1}}\right)^{\Nindt} 
\leq  \Gamma_q^{N_{\textnormal{c}}}    = \Gamma_{q}^{\Nindv-2}
\label{eq:N:c:condition:2:proto} 
\end{align}
which holds as soon as $ \Nindv $ is chosen sufficiently large with respect to $\Nindt$; see \eqref{eq:N:c:condition:2} below. This completes the proof of \eqref{eq:Rqcomm:bound}.

\textbf{Proof of~\eqref{eq:mollified:velocity} and~\eqref{eq:mollified:velocity:sup}.\,}
Using H\"older's inequality and the extra factor of $\Gamma_{q}^{-1}$ present in \eqref{eq:mollified:velocity:sup}, it is clear than for all $n,m$ such that \eqref{eq:mollified:velocity:sup} holds, the estimate \eqref{eq:mollified:velocity}  is also true. The proof is thus split in three parts: first we consider $n,m\leq 3\Nindv$, then we consider $m>3\Nindv$, and lastly $n>3\Nindv$.

We start with the proof of \eqref{eq:mollified:velocity}. In view of \eqref{inductive:velocity:frequency}, we first bound the main term, $\Pqxt w_q$, which we claim may be estimated as
\begin{align}
\label{eq:mollified:velocity:main}
    \left\| \psi_{i,q-1} D^n D_{t,q-1}^m \Pqxt w_q \right\|_{L^2} \leq \frac 12 \delta_{q}^{\sfrac 12} \MM{n, 2 \Nindv,\lambda_{q},\tilde{\lambda}_q} \MM{m, \Nindvt ,\tau_{q-1}^{-1}\Gamma_{q}^i, \Tilde{\tau}_{q-1}^{-1}}.
\end{align}
for all $n,m \leq 3\Nindv$, and as 
\begin{align}
\label{eq:mollified:velocity:main:sup}
    \left\|  D^n D_{t,q-1}^m \Pqxt w_q \right\|_{L^\infty} 
    \leq \Gamma_q^{-2} \delta_{q}^{\sfrac 12} \MM{n, 2 \Nindv,\lambda_{q},\tilde{\lambda}_q} \MM{m, \Nindvt ,\tau_{q-1}^{-1} , \Tilde{\tau}_{q-1}^{-1}}.
\end{align}
when $n+m \leq 2\Nfin$, and either $n>3\Nindv$ or $m>3\Nindv$. By the definition of $\Pqxt$ in \eqref{mollifier:operators}, in view of the  moment condition  \eqref{eq:phi} for the associated mollifier kernel, we have that 
\begin{align}
& \PP_{q,x,t} w_q(\theta) - w_q(\theta) \notag\\
&= \sum_{|\alpha| + m'' = N_{\textnormal{c}}} \frac{N_{\textnormal{c}}}{\alpha! m''!}  \int\!\!\!\int_{\mathbb{T}^3\times\mathbb{R}} K_q(\kappa) (-\kappa)^{(\alpha,m'')} \int_0^1 (1-\eta)^{N_{\textnormal{c}}-1} D^\alpha \partial_t^{m''} w_q (\theta-\eta\kappa) \,d\eta   d\kappa 
\label{eq:mollify:minus:identity}
\end{align}
where we have appealed to the notation in \eqref{eq:strange:notation}, and $N_{\textnormal{c}} = \Nindv-2$. For $n, m \leq 3 \Nindv$,
we appeal to the identity \eqref{eq:RNC} with $F = \PP_{q,x,t} w_q - w_q$,  and with $D_{t,q}$ replaced by $D_{t,q-1}$, to obtain
\begin{align}
\norm{D^n D_{t,q-1}^m (\PP_{q,x,t} w_q - w_q)}_{L^\infty}  
\les \sum_{\substack{m'\leq m \\ n'+m' \leq n+m}} \norm{d_{n,m,n',m'}}_{L^\infty} \norm{ D^{n'} \partial_t^{m'} (\PP_{q,x,t} w_q - w_q)}_{L^\infty}  
\label{eq:MNF:1}
\end{align}
where 
\begin{align*}
d_{  n,  m,   n',   m'}  
&= \sum_{k = 0}^{  m-   m'} \sum_{\substack{\{ \gamma \in {\mathbb N}^k \colon |\gamma| =   n-   n'+k,\\ \beta \in {\mathbb N}^k \colon |\beta| =   m-   m'-k\}}} c(  m,   n, k, \gamma, \beta)  \prod_{\ell=1}^{k} \left( D^{\gamma_\ell} \partial_t^{\beta_\ell}  \vlqminus(x,t) \right) \,.
\end{align*}
From \eqref{eq:space:time:v:ell:q-1:rough}, and the parameter inequality $\lambda_{q-1}^4 \delta_{q-1}^{\sfrac 12} \tilde \lambda_{q-1} \leq \Gamma_q^{-1} \tilde \tau_{q-1}^{-1}$ we deduce the bound
\[
\norm{D^{n''} \partial_t^{m''} v_{\ell_{q-1}}}_{L^\infty}
\les  \tilde \lambda_{q-1}^{n''-1}  (\Gamma_q^{-1} \tilde \tau_{q-1}^{-1}  )^{m''+1}
\]
for  $n''+m'' \leq 2\Nfin$, and therefore
\begin{align}
\norm{d_{n,m,n',m'}}_{L^\infty} \les \lambda_q^{ n-n'} (\tilde \tau_{q-1}^{-1}  \Gamma_q^{-1})^{m-m'}
\,.
\label{eq:MNF:2}
\end{align} 
Combining this estimate with  the bound  \eqref{eq:space:time:w:q:rough}, we deduce that 
\begin{align}
\norm{D^nD_{t,q-1}^m (\PP_{q,x,t} w_q - w_q)}_{L^\infty}
&\les \sum_{\substack{m'\leq m \\ n'+m' \leq n+m}} \lambda_q^{ n-n'} (\tilde \tau_{q-1}^{-1}  \Gamma_q^{-1})^{m-m'} \norm{ D^{n'} \partial_t^{m'} (\PP_{q,x,t} w_q - w_q)}_{L^\infty} 
\notag\\
&\les \sum_{\substack{m'\leq m \\ n'+m' \leq n+m}} \sum_{|\alpha| + m'' = N_{\textnormal{c}}} \lambda_q^{ n-n'} (\tilde \tau_{q-1}^{-1}  \Gamma_q^{-1})^{m-m'} \notag\\
&\qquad \qquad \times (\delta_q^{\sfrac 12} \lambda_q^2) \lambda_q^{n' + |\alpha| } (\tilde \tau_{q-1}^{-1} \Gamma_q^{-1})^{m'+ m''} 
\int_{\T^3\times \R} \abs{\kappa^{(\alpha,m'')}} |K_q(\kappa)| d\kappa
\notag\\
&\les (\delta_q^{\sfrac 12} \lambda_q^2)\sum_{|\alpha| + m'' = N_{\textnormal{c}}} \lambda_q^{ n + |\alpha|} (\tilde \tau_{q-1}^{-1}  \Gamma_q^{-1})^{m+m''}  
\tilde \lambda_q^{-|\alpha|} \tilde \tau_{q-1}^{m''}
\notag\\
&\les (\delta_q^{\sfrac 12} \lambda_q^2) \lambda_q^{ n } (\tilde \tau_{q-1}^{-1}  \Gamma_q^{-1})^{m}  
\Gamma_q^{-N_{\textnormal{c}}}\,.
\label{eq:dex:1}
\end{align}
Next, we claim that the above estimate is consistent with \eqref{eq:mollified:velocity:main}:  for $n,m\leq 3 \Nindv$ we have
\begin{align}
(\delta_q^{\sfrac 12} \lambda_q^2) \lambda_q^{ n } (\tilde \tau_{q-1}^{-1}  \Gamma_q^{-1})^{m}  
\Gamma_q^{-N_{\textnormal{c}}}
\les \Gamma_{q}^{-1}   \delta_{q}^{\sfrac 12} \lambda_q^n \MM{m, \Nindvt ,\tau_{q-1}^{-1}\Gamma_{q}^{i-1}, \Tilde{\tau}_{q-1}^{-1} \Gamma_q^{-1}}
\,.
\label{eq:dex:2}
\end{align}
Recalling the definition of $N_{\textnormal{c}}$ in \eqref{eq:N:c:condition:1}, the above bound is in turn implied by the estimate
\begin{align*}
\Gamma_{q}^3  \lambda_q^2  \left(\frac{\tilde \tau_{q-1}^{-1}}{\tau_{q-1}^{-1}}\right)^{\Nindt}  
\leq
\Gamma_q^{\Nindv}
\end{align*}
which holds since $\Nindv \gg \Nindt$; in fact, it is easy to see that the above condition  is less stringent than \eqref{eq:N:c:condition:2:proto}. Summarizing  \eqref{eq:dex:1}--\eqref{eq:dex:2}, and appealing to the inductive assumption \eqref{eq:inductive:assumption:derivative:q}, we deduce that
\begin{align}
\norm{\psi_{i,q-1} D^{n} D^m_{t,q-1} \PP_{q,x,t} w_{q}}_{L^2}
&\les \norm{\psi_{i,q-1} D^{n} D^m_{t,q-1}  w_{q}}_{L^2} + \norm{ D^{n} D^m_{t,q-1} (\PP_{q,x,t} w_{q}-w_q)}_{L^\infty} \notag\\
&\les \Gamma_{q}^{-1}   \delta_{q}^{\sfrac 12} \lambda_q^n \MM{m, \Nindvt ,\tau_{q-1}^{-1}\Gamma_{q}^{i-1}, \Tilde{\tau}_{q-1}^{-1} \Gamma_q^{-1}}
\label{eq:dex:3}
\end{align}
for all $0 \leq n,m \leq 3 \Nindv$. The above estimate verifies \eqref{eq:mollified:velocity:main}. 

We next turn to the proof of \eqref{eq:mollified:velocity:main:sup}. The key observation is that when establishing \eqref{eq:dex:3}, the two main properties of the mollification kernel $K_q(\kappa)$ which we have used are: the vanishing of the moments $\int\!\!\!\int_{\mathbb{T}^3\times\mathbb{R}} K_q(\kappa) (-\kappa)^{(\alpha,m'')} d\kappa = 0$ for $1 \leq |\alpha|+m'' \leq \Nindv$ and the fact that $\|K_q(\kappa) (-\kappa)^{(\alpha,m'')}\|_{L^1(d\kappa)} \les \tilde \lambda_q^{-|\alpha|} \tilde \tau_{q-1}^{m''}$  for all $|\alpha| + m'' \leq \Nindv$. We claim that, for any $\tilde n + \tilde m \leq 2 \Nfin$, the kernel 
\begin{align*}
K_q^{(\tilde n,\tilde m)} (y,s) :=  D_y^{\tilde n} \partial_s^{\tilde m} K_q(y,s) \tilde\lambda_q^{-\tilde n} \tilde \tau_{q-1}^{\tilde m}
\end{align*}
satisfies exactly the same two properties. The second property, about the $L^1$ norm, is immediate by scaling and the above definition, from the properties of the Friedrichs mollifier densities $\phi$ and $\tilde \phi$ from \eqref{eq:phi}. Concerning the vanishing moment condition, we note that $K_q^{(n,m)}$ has in fact more vanishing moments than $K_q$, as is easily seen from integration by parts in $\kappa$. The upshot of this observation is that in precisely the same way that \eqref{eq:dex:3}
was proven, we may show that 
\begin{align}
\norm{D^{n} D^m_{t,q-1} D^{\tilde n} \partial_t^{\tilde m} \PP_{q,x,t} w_{q}}_{L^2}
&\les \sum_{i = 0}^{\imax} \norm{\psi_{i,q-1} D^{n} D^m_{t,q-1}  w_{q}}_{L^2} + \norm{ D^{n} D^m_{t,q-1} (D^{\tilde n} \partial_t^{\tilde m} \PP_{q,x,t} w_{q}-w_q)}_{L^\infty} \notag\\
&\les \Gamma_{q}^{-1}   \delta_{q}^{\sfrac 12} \lambda_q^n \tilde \lambda_q^{\tilde n}(\Tilde{\tau}_{q-1}^{-1} \Gamma_q^{-1})^m (\Tilde{\tau}_{q-1}^{-1})^{\tilde m}
\label{eq:dex:4}
\end{align}
for all $0 \leq n,m \leq 3 \Nindv$, and for all $0 \leq \tilde n + \tilde m \leq 2 \Nfin$. Here we have used \eqref{eq:inductive:partition} and \eqref{eq:imax:old} with $q'=q-1$, and the parameter inequality $\tau_{q-1}^{-1} \Gamma_q^{\imax-1} \leq \tau_{q-1}^{-1} \lambda_{q-1}^2 \leq \Tilde{\tau}_{q-1}^{-1} \Gamma_q^{-1}$. 

Next, consider $n + m \leq 2 \Nfin$ such that   $n \leq 3 \Nindv$ and $m > 3 \Nindv$. Define  $\bar m =  m - 3\Nindv > 0$, which are the number of excess material derivatives not covered by the bound \eqref{eq:dex:3}.  We rewrite the term which we need to estimate in \eqref{eq:mollified:velocity:main:sup} as 
\begin{align}
\norm{ D^{n} D^m_{t,q-1} \PP_{q,x,t} w_{q}}_{L^\infty} = \norm{D^{n} D^{3 \Nindv}_{t,q-1}   D_{t,q-1}^{\bar m} \PP_{q,x,t} w_{q}}_{L^\infty} 
\,.
\label{eq:dex:5}
\end{align}
Using \eqref{eq:RNC}--\eqref{eq:DNC} we expand $ D_{t,q-1}^{\bar m}$ into space and time derivatives and apply the Leibniz rule to deduce
\begin{subequations}
\begin{align}
 D_{t,q-1}^{\bar m} \PP_{q,x,t} w_{q}
&= \sum_{\substack{\bar m'\leq \bar m \\ \bar n'+\bar m' \leq   \bar m}} d_{\bar m, \bar n', \bar m'} D^{\bar n'} \partial_t^{\bar m'} \PP_{q,x,t} w_{q} 
\label{eq:dex:6a} \\
d_{ \bar m, \bar n', \bar m'}(x,t) 
&= \sum_{k = 0}^{\bar m- \bar m'} \sum_{\substack{\{ \gamma \in {\mathbb N}^k \colon |\gamma| =  - \bar n'+k,\\ \beta \in {\mathbb N}^k \colon |\beta| = \bar m- \bar m'-k\}}} c(\bar m,  k, \gamma, \beta)  \prod_{\ell=1}^{k} \left( D^{\gamma_\ell} \partial_t^{\beta_\ell}  \vlqminus (x,t) \right)
\,.
\label{eq:dex:6}
\end{align}
\end{subequations}
Using the Leibniz rule,   the previously established bound \eqref{eq:dex:4}, and the Sobolev embedding $H^2\subset L^\infty$, we deduce that 
\begin{align}
&\norm{D^{n} D^{3 \Nindv}_{t,q-1}  D_{t,q-1}^{\bar m} \PP_{q,x,t} w_{q}}_{L^\infty}  \notag\\
&\les \sum_{a=0}^n \sum_{b = 0}^{3\Nindv} \sum_{\substack{\bar m'\leq \bar m \\ \bar n'+\bar m' \leq \bar m}} \norm{D^a D_{t,q-1}^{b} d_{\bar m, \bar n', \bar m'}}_{L^\infty}  
\norm{ D^{n - a} D_{t,q-1}^{3\Nindv-b} D^{\bar n'} \partial_t^{\bar m'} \PP_{q,x,t} w_{q}}_{L^\infty} \notag\\
&\les \sum_{a=0}^n \sum_{b = 0}^{3\Nindv} \sum_{\substack{\bar m'\leq \bar m \\ \bar n'+\bar m' \leq \bar m}} \norm{D^a D_{t,q-1}^{b} d_{ \bar m, \bar n', \bar m'}}_{L^\infty} 
\Gamma_{q}^{-1}   \delta_{q}^{\sfrac 12} \lambda_q^{n-a} \tilde \lambda_q^{\bar n'+2} (\Tilde{\tau}_{q-1}^{-1} \Gamma_q^{-1})^{3\Nindv-b} (\Tilde{\tau}_{q-1}^{-1})^{\bar m'} \,.
\label{eq:dex:7}
\end{align}

Thus, in order to obtain the desired bound on \eqref{eq:dex:5}, we need to estimate space and material derivatives $D^a D_{t,q-1}^{b}$ of the term defined in \eqref{eq:dex:6}, and in particular for $D^{\gamma_\ell} \partial_t^{\beta_\ell} \vlqminus$. We may however appeal to \eqref{eq:MNF:1}--\eqref{eq:MNF:2} with $(\Pqxt w_q-w_q)$ replaced by $D^{\gamma_\ell} \partial_t^{\beta_\ell} \vlqminus$, and to the bound \eqref{eq:space:time:v:ell:q-1:rough} to deduce that 
\begin{align*}
 \norm{D^{a'} D_{t,q-1}^{b'}  D^{\gamma_\ell} \partial_t^{\beta_\ell} \vlqminus}_{L^\infty}
 &\les \sum_{\substack{b'' \leq b'\\a''+b''\leq a'+b'}} \lambda_q^{a'-a''} (\Gamma_q^{-1} \tilde \tau_{q-1}^{-1})^{b'-b''} \norm{D^{a''} \partial_t^{b''} D^{\gamma_\ell} \partial_t^{\beta_\ell} \vlqminus}_{L^\infty} \notag\\
 &\les (\lambda_{q-1}^4 \delta_{q-1}^{\sfrac 12}) \lambda_q^{a'} \tilde \lambda_{q-1}^{\gamma_\ell} (\Gamma_q^{-1} \tilde \tau_{q-1}^{-1})^{b'+\beta_\ell} \notag\\
 &\les \lambda_q^{a'} \tilde \lambda_{q-1}^{\gamma_\ell-1} (\Gamma_q^{-1} \tilde \tau_{q-1}^{-1})^{b'+\beta_\ell+1}  \,.
\end{align*}
where in the last estimate we have used   the parameter inequality $\lambda_{q-1}^4 \delta_{q-1}^{\sfrac 12} \tilde \lambda_{q-1} \leq \Gamma_q^{-1} \tilde \tau_{q-1}^{-1}$.
Using the above bound and the definition \eqref{eq:dex:6} we deduce that 
\begin{align}
&\norm{D^a D_{t,q-1}^{b} d_{ \bar m, \bar n', \bar m'}}_{L^\infty}   
\les \lambda_q^{a} \tilde \lambda_{q-1}^{-\bar n'}  (\Gamma_q^{-1} \tilde \tau_{q-1}^{-1})^{b + \bar m - \bar m'}\,.
\label{eq:Lakers:4-2}
\end{align}
The above display may be combined with \eqref{eq:dex:7} and yields
\begin{align}
&\norm{ D^{n} D^{3 \Nindv}_{t,q-1}   D_{t,q-1}^{\bar m} \PP_{q,x,t} w_{q}}_{L^\infty}  \notag\\
&\les \Gamma_{q}^{-1}   \delta_{q}^{\sfrac 12}\lambda_q^{n} \tilde \lambda_q^2  \sum_{ b = 0}^{3\Nindv} \sum_{\substack{\bar m'\leq \bar m \\ \bar n'+\bar m' \leq \bar n+\bar m}}
\tilde \lambda_{q-1}^{-\bar n'}  (\Gamma_q^{-1} \tilde \tau_{q-1}^{-1})^{b + \bar m - \bar m'} (\Tilde{\tau}_{q-1}^{-1} \Gamma_q^{-1})^{3\Nindv-b} (\Tilde{\tau}_{q-1}^{-1})^{\bar m'}
\notag\\
&\les \Gamma_{q}^{-1}   \delta_{q}^{\sfrac 12}\lambda_q^{n}   \tilde \lambda_q^2 \sum_{ \bar m'\leq \bar m}
 (\Gamma_q^{-1} \tilde \tau_{q-1}^{-1})^{   m - \bar m'} (\Tilde{\tau}_{q-1}^{-1})^{\bar m'}
\label{eq:dex:8}
\end{align}
where we have recalled that $3\Nindv + \bar m = m$.
The above estimate has to be compared with the right side of \eqref{eq:mollified:velocity:main:sup}, and for this purpose we note that for $\bar m' \leq \bar m = m - 3 \Nindv$  we have  
\begin{align*}
\lambda_q^{n}   (\Gamma_q^{-1} \tilde \tau_{q-1}^{-1})^{   m - \bar m'}  (\Tilde{\tau}_{q-1}^{-1})^{\bar m'}
&\les \MM{n, 2 \Nindv, \lambda_q,\tilde \lambda_q}  \Gamma_q^{-(m - \bar m')} (\tilde \tau_{q-1}^{-1})^{-m}
\notag\\
&\les  \Gamma_q^{-3 \Nindv} (\tilde \tau_{q-1}^{-1} \tau_{q-1})^{\Nindt}  \MM{n, 2 \Nindv, \lambda_q,\tilde \lambda_q} \MM{  m, \Nindt,  \tau_{q-1}^{-1} , \tilde \tau_{q-1}^{-1}} 
\end{align*}
where we have used the fact that $m - \bar m' \geq m - \bar m = 3\Nindv$. Taking $\Nindv\gg \Nindt$ such that 
\begin{align}
\tilde \lambda_q^2 (\tilde \tau_{q-1}^{-1} \tau_{q-1})^{\Nindt}  \leq \Gamma_q^{3\Nindv -2 }\,,
\label{eq:Nind:cond:11}
\end{align}
a condition which is satisfied due to \eqref{eq:N:c:condition:2:also:new}, 
it follows from \eqref{eq:dex:8} that \eqref{eq:mollified:velocity:main:sup} holds
whenever $m>3\Nindv$,  $n\leq 3\Nindv$, and $m+n \leq 2\Nfin$.

It remains to consider the case $n> 3\Nindv$, and $n+m \leq 2 \Nfin$.  In this case we still use \eqref{eq:dex:6a}--\eqref{eq:dex:6}, but with $\bar m$ replaced by $m$, and similarly to \eqref{eq:dex:7}, but by appealing to the bounds \eqref{eq:space:time:v:ell:q-1:rough} and \eqref{eq:MNF:2} instead of \eqref{eq:Lakers:4-2}, we obtain
\begin{align*}
&\norm{D^{n}  D_{t,q-1}^{m} \PP_{q,x,t} w_{q}}_{L^\infty}  \notag\\
&\les \sum_{a=0}^n   \sum_{\substack{\bar m'\leq   m \\ \bar n'+\bar m' \leq   m}}
 \norm{D^a  d_{  m, \bar n', \bar m'}}_{L^\infty}  \norm{D^{n-a+\bar n'} \partial_t^{\bar m'} \Pqxt w_q}_{L^\infty}
\notag\\
&\les  \sum_{a=0}^n  \sum_{\substack{\bar m'\leq   m \\ \bar n'+\bar m' \leq   m}} 
\lambda_q^{a-\bar n'}    (\Gamma_q^{-1} \tilde \tau_{q-1}^{-1})^{  m - \bar m'}
\Gamma_{q}^{-1}   \delta_{q}^{\sfrac 12}  \tilde \lambda_q^2 \MM{n-a+\bar n',3\Nindv,\lambda_q, \tilde \lambda_q}   (\Tilde{\tau}_{q-1}^{-1})^{\bar m'}
\notag\\
&\les  
\Gamma_{q}^{-1}   \delta_{q}^{\sfrac 12} \tilde \lambda_q^2
\MM{n , 3 \Nindv, \lambda_q,\tilde \lambda_q} (\Tilde{\tau}_{q-1}^{-1})^m \,.
\end{align*}
To conclude the proof of \eqref{eq:mollified:velocity:main:sup} in this case, we note that for $n\geq 3\Nindv$ the definition \eqref{eq:tilde:lambda:q:def} implies
\begin{align*}
\MM{n , 3 \Nindv, \lambda_q,\tilde \lambda_q}  \leq \Gamma_{q+1}^{-5 \Nindv} \MM{n , 2 \Nindv, \lambda_q,\tilde \lambda_q}  
\end{align*}
and this factor is sufficiently small to absorb losses due to bad material derivative estimates. Indeed, we have that 
\begin{align*}
&\Gamma_{q}^{-1}   \delta_{q}^{\sfrac 12} \tilde \lambda_q^2
\MM{n , 3 \Nindv, \lambda_q,\tilde \lambda_q} (\Tilde{\tau}_{q-1}^{-1})^m
\notag\\
&\les  \Gamma_{q}^{-3}   \delta_{q}^{\sfrac 12} 
 \MM{n , 2 \Nindv, \lambda_q,\tilde \lambda_q}   \MM{m, \Nindvt ,\tau_{q-1}^{-1} , \Tilde{\tau}_{q-1}^{-1}} \Gamma_q^2 \tilde \lambda_q^2 \left(\frac{\Tilde{\tau}_{q-1}^{-1}}{\tau_{q-1}^{-1}} \right)^{\Nindt} \Gamma_{q+1}^{-5 \Nindv}
 \notag\\
 &\les  \Gamma_{q}^{-1}   \delta_{q}^{\sfrac 12} 
 \MM{n ,2  \Nindv, \lambda_q,\tilde \lambda_q}   \MM{m, \Nindvt ,\tau_{q-1}^{-1}, \Tilde{\tau}_{q-1}^{-1}} 
\end{align*}
by appealing to the  condition $\Nindv \gg \Nindt$ given in \eqref{eq:N:c:condition:2:new}.
This concludes the proof of  \eqref{eq:mollified:velocity:main:sup} for all $n+m \leq 2\Nfin$, if either $n$ or $m$ are larger than $3\Nindv$.

The bounds \eqref{eq:mollified:velocity:main}--\eqref{eq:mollified:velocity:main:sup} estimate the leading order contribution to $u_q$. According to the decomposition~\eqref{inductive:velocity:frequency}, the proofs of \eqref{eq:mollified:velocity} and \eqref{eq:mollified:velocity:sup} are completed if we are able to verify that 
\begin{align}
\label{eq:mollified:velocity:secondary}
    \left\| D^n D_{t,q-1}^m ( \Pqxt - \Id) \vlqminus   \right\|_{L^\infty} \leq \Gamma_q^{-2} \delta_{q}^{\sfrac 12} \MM{n,2  \Nindv,\lambda_{q},\tilde{\lambda}_q} \MM{m, \Nindvt ,\tau_{q-1}^{-1}, \Tilde{\tau}_{q-1}^{-1}}
\end{align}
holds for all $n+m\leq 2\Nfin$. 

In order to establish this bound, we appeal to \eqref{eq:MNF:1}--\eqref{eq:MNF:2} and obtain
\begin{align}
\norm{D^n D_{t,q-1}^m (\PP_{q,x,t} -\Id )\vlqminus}_{L^\infty}  
 \les \sum_{\substack{m'\leq m \\ n'+m' \leq n+m}} 
\lambda_q^{ n-n'} (\tilde \tau_{q-1}^{-1}  \Gamma_q^{-1})^{m-m'} \norm{ D^{n'} \partial_t^{m'} (\PP_{q,x,t} -\Id )\vlqminus}_{L^\infty}   
\label{eq:Lakers:Heat:4-2}
\end{align}
for $n,m\geq 0$ such that $n+m \leq 2\Nfin$.
Here we distinguish two cases. If either $n> 3\Nindv$ or $m>3\Nindv$, then we simply appeal to \eqref{eq:space:time:v:ell:q-1:rough}, use that $\Pqxt$ commutes with $D$ and $\partial_t$, and obtain from the above display that 
\begin{align*}
&\norm{D^n D_{t,q-1}^m (\PP_{q,x,t} -\Id )\vlqminus}_{L^\infty}  \notag\\
&\les\sum_{\substack{m'\leq m \\ n'+m' \leq n+m}} 
\lambda_q^{ n-n'} (\tilde \tau_{q-1}^{-1}  \Gamma_q^{-1})^{m-m'}  (\lambda_{q-1}^4 \delta_{q-1}^{\sfrac 12}) \tilde \lambda_{q-1}^{n'}  \tilde \tau_{q-2}^{-m'}\notag\\
&\les (\lambda_{q-1}^4 \delta_{q-1}^{\sfrac 12}) \lambda_q^n (\tilde \tau_{q-1}^{-1}  \Gamma_q^{-1})^{m}
\notag\\
&\les (\lambda_{q-1}^4 \delta_{q-1}^{\sfrac 12}) (\tau_{q-1} \tilde \tau_{q-1}^{-1})^{\Nindt} \Gamma_q^{-3\Nindv} \MM{n,2  \Nindv,\lambda_{q},\tilde{\lambda}_q}  \MM{m, \Nindvt ,\tau_{q-1}^{-1}, \Tilde{\tau}_{q-1}^{-1}} \notag\\
&\les \Big(  \lambda_{q-1}^4 \delta_{q-1}^{\sfrac 12} \Gamma_q^2 \delta_q^{-\sfrac 12}  (\tau_{q-1} \tilde \tau_{q-1}^{-1})^{\Nindt} \Gamma_q^{-3\Nindv}\Big) \Gamma_q^{-2} \delta_q^{\sfrac 12} \MM{n,2  \Nindv,\lambda_{q},\tilde{\lambda}_q}  \MM{m, \Nindvt ,\tau_{q-1}^{-1}, \Tilde{\tau}_{q-1}^{-1}}
\end{align*}
Using that $\Nindv\gg \Nindt$, as described in \eqref{eq:N:c:condition:2:also:new}, the above estimate then readily implies \eqref{eq:mollified:velocity:secondary}.

We are thus left to consider \eqref{eq:Lakers:Heat:4-2} for $n,m\leq 3\Nindv$. In this case, the bound for the term $\Vert D^{n'} \partial_t^{m'} (\PP_{q,x,t} -\Id )\vlqminus\Vert_{L^\infty}$ present in \eqref{eq:Lakers:Heat:4-2} is different. Similarly to \eqref{eq:mollify:minus:identity} we use that the kernel $K_q$ has vanishing moments of orders between $1$ and $\Nindv$, and thus we have 
\begin{align}
& \PP_{q,x,t} \vlqminus(\theta) - \vlqminus(\theta) \notag\\
&= \sum_{|\alpha| + m'' = \Nindv} \frac{\Nindv}{\alpha! m''!}  \int\!\!\!\int_{\mathbb{T}^3\times\mathbb{R}} K_q(\kappa) (-\kappa)^{(\alpha,m'')} \int_0^1 (1-\eta)^{\Nindv-1} D^\alpha \partial_t^{m''} \vlqminus(\theta-\eta\kappa) \,d\eta   d\kappa \,.
\label{eq:mollify:minus:identity:2}
\end{align}
Using \eqref{eq:space:time:v:ell:q-1:rough} and \eqref{eq:mollify:minus:identity:2}, we may then estimate 
\begin{align*}
\norm{ D^{n'} \partial_t^{m'} (\PP_{q,x,t} -\Id )\vlqminus}_{L^\infty} 
&\les (\lambda_{q-1}^4 \delta_{q-1}^{\sfrac 12}) \sum_{|\alpha|+m'' = \Nindv} \tilde\lambda_q^{-|\alpha|} \tilde \tau_{q-1}^{m''} \tilde \lambda_{q-1}^{n'+|\alpha|} (\Gamma_q^{-1} \tilde \tau_q^{-1})^{m'+m''}
\notag\\
&\les (\lambda_{q-1}^4 \delta_{q-1}^{\sfrac 12}) \Gamma_q^{-\Nindv} \lambda_{q}^{n'} (\Gamma_q^{-1} \tilde \tau_q^{-1})^{m'}
\,.
\end{align*}
Combining the above display with \eqref{eq:Lakers:Heat:4-2} we arrive at 
\begin{align}
&\norm{D^n D_{t,q-1}^m (\PP_{q,x,t} -\Id )\vlqminus}_{L^\infty}  \notag\\
&\les  (\lambda_{q-1}^4 \delta_{q-1}^{\sfrac 12}) \Gamma_q^{-\Nindv} \lambda_{q}^{n} (\Gamma_q^{-1} \tilde \tau_{q-1}^{-1})^{m}
\notag\\
&\les  (\lambda_{q-1}^4 \delta_{q-1}^{\sfrac 12}) (\tilde\tau_{q-1}^{-1} \tau_{q-1})^{\Nindt} \Gamma_q^{-\Nindv} \MM{n,2\Nindv,\lambda_{q},\tilde \lambda_q} \MM{m,\Nindt, \tau_{q-1}^{-1}, \tilde \tau_{q-1}^{-1}}\,.
\label{eq:junk:inequality:1}
\end{align}
Using that $\Nindv\gg \Nindt$, see condition \eqref{eq:N:c:condition:2:also:new}, the above estimate concludes the proof of \eqref{eq:mollified:velocity:secondary}. 

Combining the bounds \eqref{eq:mollified:velocity:main}, \eqref{eq:mollified:velocity:main:sup}, and \eqref{eq:mollified:velocity:secondary} concludes the proofs of \eqref{eq:mollified:velocity}  and \eqref{eq:mollified:velocity:sup}.

\textbf{Proof of~\eqref{eq:vq:minus:mollified}.\,}
By \eqref{eq:cutoffs:wu} we have that 
\begin{align*}
v_{\ell_q} - v_q  = (\Pqxt- \Id) v_q = (\Pqxt- \Id) w_q + (\Pqxt- \Id) v_{\ell_{q-1}} 
\,.
\end{align*}
From \eqref{eq:dex:1} and \eqref{eq:dex:2} we deduce that the first term on the right side of the above display is bounded as
\begin{align*}
&\norm{D^n D_{t,q-1}^m (\Pqxt- \Id) w_q}_{L^\infty} \notag\\
&\qquad \les \left(\delta_q^{\sfrac 12} \Gamma_q^2 \lambda_q^2 (\tilde \tau_{q-1}^{-1} \tau_{q-1})^{\Nindt} \Gamma_q^{-\Nindv}\right)  \lambda_q^n \MM{m,\Nindt,\tau_{q-1}^{-1} \Gamma_q^{i-1},\tilde \tau_{q-1}^{-1} \Gamma_q^{-1}}\,,
\end{align*}
while the second term is estimated from \eqref{eq:junk:inequality:1} as
\begin{align*}
&\norm{D^n D_{t,q-1}^m (\PP_{q,x,t} -\Id )\vlqminus}_{L^\infty} \notag\\
&\qquad \les\left(\delta_{q-1}^{\sfrac 12} \lambda_{q-1}^4 (\tilde\tau_{q-1}^{-1} \tau_{q-1})^{\Nindt} \Gamma_q^{-\Nindv} \right)   \MM{n,2\Nindv,\lambda_{q},\tilde \lambda_q} \MM{m,\Nindt, \tau_{q-1}^{-1}, \tilde \tau_{q-1}^{-1}}\,,
\end{align*}
for $n,m\leq 3\Nindv$. Since $\Nindv \gg \Nindt$, see e.g.~the parameter inequality \eqref{eq:N:c:condition:2}, the above two displays directly imply \eqref{eq:vq:minus:mollified}. 

\textbf{Proof of~\eqref{eq:mollified:stress:bounds} and~\eqref{eq:mollified:stress:bounds:sup}.\,} 
The argument is nearly identical to how the inductive bounds on $w_q$ in \eqref{eq:inductive:assumption:derivative:q} were shown earlier to imply bounds for $\PP_{q,x,t} w_q$ as in \eqref{eq:mollified:velocity:main}. The crucial ingredients in this proof were: that for each material derivative the bound on the mollified function $\PP_{q,x,t} w_q$ is relaxed by a factor of $\Gamma_q$, that the cost of space derivatives is relaxed from $\lambda_q$ to $\tilde \lambda_q$ when $n \geq \Nindv$, and that the available number of estimates on the un-mollified function $w_q$ was much larger than $\Nindv$ (more precisely $7 \Nindv$). But the same ingredients are available for the transfer of estimates from $\RR_q$ to $\RR_{\ell_q} = \PP_{q,x,t} \RR_q$. Indeed, the derivatives available in~\eqref{eq:Rq:inductive:assumption} extend significantly past $\Nindv$ (this time up to $3\Nindv$), when comparing the desired bound on $\RR_{\ell_q}$ in \eqref{eq:mollified:stress:bounds} with the available inductive bound in~\eqref{eq:Rq:inductive:assumption} we note that the cost of each material derivative is relaxed by a factor of $\Gamma_q$, and that the cost of each additional space derivative is relaxed from $\lambda_q$ to $\tilde \lambda_q$ when $n$ is sufficiently large. To avoid redundancy, we omit these details.  
\end{proof} 

\section{Cutoffs}
\label{sec:cutoff}

This section is dedicated to the construction of the cutoff functions described in Section~\ref{ss:cutoffs}, which play the role of a joint Eulerian-and-Lagrangian Littlewood-Paley frequency decompositon, which in addition keeps track of the size of objects in physical space. During a first pass at the paper, the reader may skip this technical section --- if the  Lemmas~\ref{lem:partition:of:unity:psi}, \ref{lem:maximal:i}, \ref{lem:Dt:Dt:wq:psi:i:q:multi}, \ref{lem:sharp:Dt:psi:i:q},
\ref{lem:maximal:j},
\ref{lem:D:Dt:omega:sharp},
\ref{lem:omega:support},
\ref{lem:checkerboard:estimates},
\ref{lemma:cumulative:cutoff:Lp}, and Corollaries~\ref{cor:deformation} and~\ref{cor:D:Dt:Rn:sharp} are taken for granted. 

This section is organized as follows. In Section~\ref{sec:cutoff:velocity:definitions} we define the velocity cutoff functions $\psi_{i,q}$, recursively in terms of the previous level (meaning $q-1$) velocity cutoff functions $\psi_{i',q-1}$ which are assumed to satisfy the inductive bounds and properties mentioned in Section~\ref{sec:cutoff:inductive}. In Section~\ref{sec:cutoff:velocity:properties} we then verify that the velocity cutoff functions at level $q$, and the velocity fields $u_q$ and $v_{\ell_q}$ satisfy all the inductive estimates claimed in Sections~\ref{sec:cutoff:inductive} and \ref{sec:inductive:secondary:velocity}, for $q'=q$. This section is the bulk of Section~\ref{sec:cutoff}; and it is here that the various  commutators between Eulerian (space and time) derivatives and Lagrangian derivatives cause a plethora of difficulties. 

\begin{remark}[\textbf{Inductive assumptions which involve cutoffs and commutators}]\label{remark:cutoffs:inductive}
We note that by the conclusion of Section~\ref{sec:cutoff:velocity:properties} we have verified all the inductive assumptions from Section~\ref{sec:inductive:estimates}, except for \eqref{eq:inductive:assumption:derivative:q}--\eqref{eq:perturbation:time:support} for the new velocity increment $w_{q+1}$, and \eqref{eq:Rq:inductive:assumption} for the new stress $\RR_{q+1}$.  These three inductive assumptions will be revisited, broken down, and restated in Section~\ref{section:statements} and proven in Section~\ref{s:stress:estimates}.
\end{remark}

Next, in Section~\ref{sec:cutoff:temporal:definitions} we introduce the temporal cutoffs $\chi_{i,k,q}$, indexed by $k$ which are meant to subdivide the support of the velocity cutoff $\psi_{i,q}$ into time slices of width inversely to the {\em local Lipschitz norm} of $v_{\ell_q}$. This allows us in Section~\ref{sec:cutoff:flow:maps} to properly define and estimate the Lagrangian flow maps induced by the incompressible vector field $v_{\ell_q}$, on the support of $\psi_{i,q} \chi_{i,k,q}$. We next turn to defining the stress cutoff functions $\omega_{i,j,q,n,p}$, indexed by $j$, for the stress $\RR_{q,n,p}$, on the support of $\psi_{i,q}$. Coupling the stress and velocity cutoffs in this way allows us in Section~\ref{sec:cutoff:stress:properties} to sharply estimate spatial and  material derivatives of these higher order stresses, but also to estimate the derivatives of the stress cutoffs themselves. At last, we define in Section~\ref{sec:cutoff:checkerboard:definitions} the checkerboard cutoffs $\zeta_{q,i,k,n,\vec{l}}$, indexed by an address $\vec{l} = (l,w,h)$ which identifies a specific cube of side-length $2\pi/\lambda_{q,n,0}$ within $\T^3$. This specific size of the support of $\zeta_{q,i,k,n,\vec{l}}$ is important for ensuring that Oscillation Type $2$ errors vanish (see Lemmas~\ref{lem:overlap:3} and~\ref{lem:osc:2}). These cutoff functions are flowed by the backwards Lagrangian flows $\Phi_{i,k,q}$ defined earlier, explaining their dependence on the indices $q,i,k$. Lastly, the cumulative cutoff function $\eta_{i,j,k,q,n,p,\vec{l}}$ is defined in Section~\ref{sec:cutoff:total:definitions}, along with some of its principal properties. We emphasize that this cumulative cutoff has embedded into it information about the local size and cost of space/Lagrangian derivatives of both the velocity, the stress, and the Lagrangian maps.

\subsection{Definition of the velocity cutoff functions}
\label{sec:cutoff:velocity:definitions}
For all $q\geq 1$ and $0\leq m\leq\NcutSmall$, we construct the following cutoff functions.  The proof is contained in Appendix~\ref{app:lemma:5:1}.  

\begin{lemma}\label{lem:cutoff:construction:first:statement}
For all $q\geq 1$ and $0\leq m \leq \NcutSmall$, there exist smooth cutoff functions $\tilde\psi_{m,q},\psi_{m,q}:[0,\infty)\rightarrow[0,1]$ which satisfy the following.
\begin{enumerate}[(1)]
    \item\label{item:cutoff:1} The support of $\tilde\psi_{m,q}$ is precisely the set $\left[0,\Gamma_q^{2(m+1)}\right]$, and furthermore
    \begin{enumerate}[(a)]
      \item On the interval $\left[0,\frac{1}{4}\Gamma_q^{2(m+1)}\right]$, $\tilde\psi_{m,q}\equiv 1$.  
      \item  On the interval $\left[\frac{1}{4}\Gamma_q^{2(m+1)},\Gamma_q^{2(m+1)}\right]$, $\tilde\psi_{m,q}$ decreases from $1$ to $0$.  
      \end{enumerate}
    	\item\label{item:cutoff:2} The support of $\psi_{m,q}$ is precisely the set $\left[\frac{1}{4},\Gamma_q^{2(m+1)}\right]$, and furthermore
    	\begin{enumerate}[(a)]
    	\item  On the interval $\left[\frac{1}{4},1\right]$, $\psi_{m,q}$ increases from $0$ to $1$. 
    	\item  On the interval $\left[1,\frac{1}{4}\Gamma_q^{2(m+1)}\right]$, $\psi_{m,q}\equiv 1$.
    	\item  On the interval $\left[\frac{1}{4}\Gamma_q^{2(m+1)},\Gamma_q^{2(m+1)}\right]$, $\psi_{m,q}$ decreases from $1$ to $0$. 
    	\end{enumerate}
    \item For all $y\geq 0$, a partition of unity is formed as
    \begin{align}
    \tilde \psi_{m,q}^2(y) + \sum_{{i\geq 1}} \psi_{m,q}^2\left(\Gamma_{q}^{-2i(m+1)} y\right) = 1 
    \label{eq:tilde:partition}
    \end{align}
    \item $\tilde\psi_{m,q}$ and $\psi_{m,q}\left(\Gamma_q^{-2i(m+1)}\cdot\right)$ satisfy
    \begin{align}
   \supp \tilde\psi_{m,q}(\cdot) \cap \supp \psi_{m,q}\left(\Gamma_q^{-2i(m+1)}\cdot\right) &= \emptyset \quad \textnormal{if} \quad i \geq 2,\notag\\
   \supp \psi_{m,q}\left(\Gamma_q^{-2i(m+1)}\cdot\right) \cap \supp \psi_{m,q}\left(\Gamma_q^{-2i'(m+1)}\cdot\right) &= \emptyset \quad \textnormal{if} \quad |i-i'|\geq 2. \label{eq:psi:support:base:case}
    \end{align}
    \item For $0\leq N \leq \Nfin$, when $0\leq y<\Gamma_q^{2(m+1)}$ we have
    \begin{align}
    \frac{|D^N \tilde \psi_{m,q}(y)|}{(\tilde \psi_{m,q}(y))^{1-N/\Nfin}}
    &\lesssim 
     \Gamma_q^{-2N(m+1)}. \label{eq:DN:psi:q:0}
    \end{align}
For $\frac{1}{4}<y<1$ we have
    \begin{align}
     \frac{|D^N  \psi_{m,q}(y)|}{( \psi_{m,q}(y))^{1- N / \Nfin}} &\lesssim 1 \label{eq:DN:psi:q},
    \end{align}
    while for $\frac{1}{4}\Gamma_q^{2(m+1)}<y<\Gamma_q^{2(m+1)}$ we have
    \begin{align}
     \frac{|D^N  \psi_{m,q}(y)|}{( \psi_{m,q}(y))^{1- N / \Nfin}} &\lesssim \Gamma_q^{-2N(m+1)} \label{eq:DN:psi:q:gain}.
    \end{align}
In each of the above inequalities, the implicit constants depend on $N$ but not $m$ or $q$.
\end{enumerate}
\end{lemma}

\begin{definition}\label{def:istar:j}
Given $i,j,q \geq 0$, we define
\begin{align*}
i_* = i_*(j,q) = i_*(j) = \min\{ i \geq 0 \colon \Gamma_{q+1}^{i} \geq \Gamma_q^{j} \}.
\end{align*}
\end{definition}
In view of the definition \eqref{eq:Gamma:q+1:def:*}, we see that
\[
i_*(j) = \left\lceil j \frac{\log (\lambda_q)-\log(\lambda_{q-1})}{\log(\lambda_{q+1})-\log(\lambda_q)} \right\rceil = \left\lceil j \frac{\log\left(\left\lceil a^{b^q}\right\rceil\right)-\log\left(\left\lceil a^{b^{q-1}}\right\rceil\right)}{\log\left(\left\lceil a^{b^{q+1}}\right\rceil\right)-\log\left(\left\lceil a^{b^q}\right\rceil\right)} \right\rceil.
\]
One may check that as $q\rightarrow\infty$ or $a\rightarrow\infty$, $i_*(j)$ converges to $\left\lceil\frac{j}{b}\right\rceil$ for any $j$, and so if $a$ is sufficiently large, $i_*(j)$ is bounded from above and below independently of $q$ for each $j$. Note that in particular, for $j=0$ we have that $i_*(j)=0$.

At stage $q\geq 1$ of the iteration (by convention $w_0=u_0=0$) and for $m\leq\NcutSmall$ and $j_m\geq 0$, we can now define
\begin{align}
h_{m,j_m,q}^2(x,t):= 
\sum_{n=0}^{\NcutLarge} \Gamma_{q+1}^{-2i_*\left(j_m\right)} \delta_{q}^{-1} \left(\lambda_{q}\Gamma_q\right)^{-2n} \left(\tau_{q-1}^{-1}\Gamma_{q+1}^{i_*(j_m)+2}\right)^{-2m}    |D^{n} D_{t,q-1}^m u_{q}(x,t)|^2.
\label{eq:h:j:q:def}
\end{align}

\begin{definition}[\textbf{Intermediate Cutoff Functions}]\label{def:intermediate:cutoffs}
Given $q\geq 1$, $m\leq\NcutSmall$, and $j_m\geq 0$ we define $\psi_{m,i_m,j_m,q}$ by
\begin{align}
 \psi_{m,i_m,j_m,q}(x,t) 
 &= \psi_{m,q+1}  \left( \Gamma_{q+1}^{-2(i_m-i_*(j_m))(m+1)} h_{m,j_m,q}^2  (x,t) \right) 
\label{eq:psi:i:j:def}
\end{align}
for $i_m> i_*(j_m)$, 
while for $i_m=i_*(j_m)$,
\begin{align}
 \psi_{m,i_*(j_m),j_m,q}(x,t) 
 &= \tilde \psi_{m,q+1} \left( h_{m,j_m,q}^2(x,t) \right).
\label{eq:psi:i:i:def}
\end{align}
The intermediate cutoff functions $\psi_{m,i_m,j_m,q}$ are equal to zero for $ i_m < i_*(j_m)$.  
\end{definition}
The indices $i_m$ and $j_m$ will be shown to run up to some maximal values $i_{\mathrm{max}}$ and $\tilde{i}_{\textnormal{max}}$ to be determined in the proof (see Lemma~\ref{lem:maximal:i} and \eqref{eq:max:j:i:q}). With this notation and in view of \eqref{eq:tilde:partition} and \eqref{eq:psi:support:base:case}, it immediately follows that
\begin{align}
\sum_{i_m\geq0} \psi_{m,i_m,j_m,q}^2 = \sum_{i_m\geq i_*(j_m)} \psi_{m,i_m,j_m,q}^2 = \sum_{\{ i_m \colon \Gamma_{q+1}^{i_m} \geq \Gamma_q^{j_m} \}} \psi_{m,i_m,j_m,q}^2 \equiv 1
\label{eq:psi:i:j:partition:0}
\end{align}
for any $m$ and for $|i_m-i'_m|\geq 2$, 
\begin{equation}\label{eq:intermediate:overlapping}
  \psi_{m,i_m,j_m,q}\psi_{m,i_m',j_m,q}=0.
\end{equation}

\begin{definition}[\textbf{$m^{\textnormal{th}}$ Velocity Cutoff Function}]\label{def:psi:m:im:q:def}
For $q\geq 1$ and $i_m\geq 0$\footnote{Later we will show that $\psi_{m,i_m,q}\equiv 0$ if $i\geq \imax$}, we inductively define the $m^{\textnormal{th}}$ velocity cutoff function
\begin{equation}\label{eq:psi:m:im:q:def}
\psi_{m,i_m,q}^2 = \sum\limits_{\{j_m\colon i_m\geq i_*(j_m)\}} \psi_{j_m,q-1}^2 \psi_{m,i_m,j_m,q}^2.
\end{equation}
\end{definition}

In order to define the full velocity cutoff function, we use the notation
\begin{equation}\label{eq:i:tuple:def}
  \Vec{i} =  \{i_m\}_{m=0}^{\NcutSmall} = \left( i_0,...,i_{\NcutSmall} \right) \in \mathbb{N}_0^{\NcutSmall+1}
\end{equation}
to denote a tuple of non-negative integers of length $\NcutSmall+1$.

\begin{definition}[\textbf{Velocity cutoff function}]
\label{def:psi:i:q:def}
For $0 \leq i \leq i_{\textrm{max}}(q)$ and $q \geq 0$, we inductively define the velocity cutoff function $\psi_{i,q}$ as follows. When $q= 0$, we let 
\begin{align}
\psi_{i,0}=\begin{cases}1&\mbox{if }i=0\\ 0& \mbox{otherwise}.
\label{eq:psi:i:j:0:def}
\end{cases}
\end{align}
Then, we inductively on $q$ define
\begin{align}
    \psi_{i,q}^2 = \sum\limits_{\left\{\Vec{i}\colon\max\limits_{0\leq m\leq\NcutSmall} i_m =i\right\}} \prod\limits_{m=0}^{\NcutSmall} \psi_{m,i_m,q}^2.
  \label{eq:psi:i:q:recursive}
\end{align}
for all $q \geq 1$. 
\end{definition}

The sum used to define $\psi_{i,q}$ for $q\geq 1$ is over all tuples with a maximum entry of $i$.  The number of such tuples is clearly $q$-independent once it is demonstrated in Lemma~\ref{lem:maximal:i} that $i_m\leq\imax(q)$ (which implies $i\leq \imax(q)$), and $\imax(q)$ is bounded above independently of $q$.

For notational convenience, given an $\Vec{i}$ as in the sum of \eqref{eq:psi:i:q:recursive}, we shall denote 
\begin{align}
\supp\left( \prod\limits_{m=0}^{\NcutSmall} \psi_{m,i_m,q}  \right) = \bigcap_{m=0}^{\NcutSmall} \supp(\psi_{m,i_m,q}) =: \supp (\psi_{\Vec{i},q} )\,.
\label{eq:new:supp:notation}
\end{align}
In particular, we will frequently use that $(x,t) \in \supp(\psi_{i,q})$ if and only if there exists $\Vec{i}\in \N_0^{\NcutSmall+1}$ such that $\max_{0\leq m\leq\NcutSmall} i_m =i$, and $(x,t) \in \supp(\psi_{\Vec{i},q})$.

\subsection{Properties of the velocity cutoff functions}
\label{sec:cutoff:velocity:properties}
\subsubsection{Partitions of unity}
 \begin{lemma}[$\psi_{m,i_m,q}$ - Partition of unity]
\label{lem:partition:of:unity:psi:m}
For all $m$, we have that 
\begin{align}
\sum_{i_m\geq 0} \psi_{m,i_m,q}^2\equiv 1\,, \qquad \psi_{m,i_m,q}\psi_{m,i'_m,q}=0\quad\textnormal{for}\quad|i_m-i'_m|\geq 2. \label{eq:lemma:partition:1}
\end{align}
\end{lemma}
\begin{proof}[Proof of Lemma~\ref{lem:partition:of:unity:psi:m}]
The proof proceeds inductively. When $q=0$ there is nothing to prove as $\psi_{m,i_m,q}$ is not defined. Thus we assume $q\geq 1$.  From \eqref{eq:psi:i:j:0:def} for $q=0$ and \eqref{eq:inductive:partition} for $q\geq 1$, we assume that the functions $\{ \psi_{j,q-1}^2 \}_{j\geq 0}$ form a partition of unity. To show the first part of \eqref{eq:lemma:partition:1}, we may use \eqref{eq:psi:i:j:partition:0} and \eqref{eq:psi:m:im:q:def} and reorder the summation to obtain
\begin{align*}
\sum_{i_m\geq 0} \psi_{m,i_m,q}^2 =&   \sum_{i_m\geq 0} \sum_{\{ j_m\colon i_*(j_m) \leq  i_m\}}~
 \psi_{j_m,q-1}^2 \psi_{m,i_m,j_m, q}^2(x,t) \notag \\
 =&   \sum_{j_m\geq 0} ~
 \psi_{j_m, q-1}^2  \underbrace{\sum_{\{i_m\colon i_m\geq  i_*(j_m)\}} \psi_{m,i_m,j_m, q}^2}_{\equiv 1\mbox{ by } \eqref{eq:psi:i:j:partition:0}} 
 =   \sum_{j_m\geq 0} ~
 \psi_{j_m,q-1}^2  \equiv 1.
\end{align*}
The last equality follows from the inductive assumption \eqref{eq:inductive:partition}.

The proof of the second claim is more involved and will be split into cases.  Using the definition in \eqref{eq:psi:m:im:q:def}, we have that 
\begin{align*}
    \psi_{m,i_m,q}\psi_{m,i'_m,q} = \sum\limits_{\{j_m:i_m\geq i_*(j_m)\}} \sum\limits_{\{j_m':i_m'\geq i_*(j_m')\}} \psi_{j_m,q-1}^2\psi_{j'_m,q-1}^2 \psi_{m,i_m,j_m,q}^2\psi_{m,i'_m,j'_m,q}^2.
\end{align*}
Recalling the inductive assumption \eqref{eq:inductive:partition}, we have that the above sum  only includes pairs of indices $j_m$ and $j'_m$ such that $|j_m-j'_m|\leq 1$. So we may assume that 
\begin{equation}\label{eq:overlap:assumption}
(x,t) \in \supp \psi_{m,i_m,j_m,q} \cap \supp \psi_{m,i'_m,j'_m,q},
\end{equation}
where $|j_m-j'_m|\leq 1$. The first and simplest case is the case $j_m=j'_m$. We then appeal to \eqref{eq:intermediate:overlapping} to deduce that it must be the case that $|i_m-i'_m|\leq 1$ in order for
\eqref{eq:overlap:assumption} to be true.

Before moving to the second and third cases, we first show that by symmetry it will suffice to prove that $\psi_{m,i_m,q}\psi_{m,i'_m,q}=0$ when $i_m'\leq i_m-2$.  Assuming this has been proven, let $i_{m_1},i_{m_2}$ be given with $|i_{m_1}-i_{m_2}|\geq 2$.  Without loss of generality we may assume that $i_{m_1}\geq i_{m_2}$, which implies that $i_{m_1}\geq i_{m_2}+2$.  Using the assumption and setting $i_{m_2}=i'_m$ and $i_{m_1}=i_m$, we deduce that $\psi_{m,i_{m_1},q}\psi_{m,i_{m_2},q}=0$. Thus, we have reduced the proof to showing that $\psi_{m,i_m,q}\psi_{m,i'_m,q}=0$ when $i_m'\leq i_m-2$, which we will show next by contradiction.

Let us consider the second case, $j'_m=j_m+1$. When $i_m = i_*(j_m)$, using that $i_*(j_m)\leq i_*(j_m+1)$, we obtain
\begin{equation*}
i'_m\leq i_m-2=i_*(j_m)-2 < i_*(j_m+1) = i_*(j'_m),
\end{equation*}
and so by Definition~\ref{def:intermediate:cutoffs}, we have that $\psi_{m,i'_m,j'_m,q}=0$. Thus, in this case there is nothing to prove, and we need to only consider the case $i_m > i_*(j_m)$. From \eqref{eq:overlap:assumption}, points \ref{item:cutoff:1} and \ref{item:cutoff:2} from Lemma~\ref{lem:cutoff:construction:first:statement}, and Definition~\ref{def:intermediate:cutoffs}, we have that
\begin{subequations}
\begin{align}
    h_{m,j_m,q}(x,t) &\in \left[ \frac{1}{2} \Gamma_{q+1}^{(m+1)(i_m-i_*(j_m))}, \Gamma_{q+1}^{(m+1)(i_m+1-i_*(j_m))} \right], \label{eq:hm:overlap:0} \\
    h_{m,j_m+1,q}(x,t) &\leq  \Gamma_{q+1}^{(m+1)(i'_m+1-i_*(j_m+1))} . \label{eq:hm:overlap:1}
\end{align}
\end{subequations}
Note that from the definition of $h_{m,j_m,q}$ in \eqref{eq:h:j:q:def}, we have that 
\begin{equation*}
 \Gamma_{q+1}^{(m+1)(i_*\left(j_m+1\right)-i_*\left(j_m\right))}   h_{m,j_m+1,q} =  h_{m,j_m,q}.
\end{equation*}
Then, since $ i'_m\leq i_m-2$, from \eqref{eq:hm:overlap:1} we have that
\begin{align*}
 \Gamma_{q+1}^{-(m+1)(i_m-i_*(j_m))}   h_{m,j_m,q} 
 &=  \Gamma_{q+1}^{-(m+1)(i_m-i_*(j_m))}   h_{m,j_m+1,q}  \Gamma_{q+1}^{(m+1)(i_*\left(j_m+1\right)-i_*\left(j_m\right))}   \notag \\
    &\leq  \Gamma_{q+1}^{-(m+1)(i_m-i_*(j_m))}    \Gamma_{q+1}^{(m+1)(i'_m+1-i_*(j_m+1))}  \Gamma_{q+1}^{(m+1)(i_*\left(j_m+1\right)-i_*\left(j_m\right))}   \notag\\
    &=   \Gamma_{q+1}^{(m+1)(i'_m+1 -i_m )}  \notag\\
    &\leq \Gamma_{q+1}^{-(m+1)} \,.
\end{align*}
Since $m\geq 0$, the above estimate contradicts the lower bound on  $h_{m,j_m,q}$ in \eqref{eq:hm:overlap:0} because $\Gamma_{q+1}^{-1} \ll \sfrac 12$ 
for $a$ sufficiently large. 

We move to the third and final case, $j'_m=j_m-1$. As before, if $i_m=i_*(j_m)$, then since $i_*(j_m)\leq i_*(j_m-1)+1$, we have that 
\begin{equation*}
    i'_m \leq i_m - 2 = i_*(j_m)-2 \leq i_*(j_m-1) -1 < i_*(j_m-1) = i_*(j'_m)\,,
\end{equation*}
which by Definition~\ref{def:intermediate:cutoffs} implies that $\psi_{m,i'_m,j'_m,q}=0$, and there is nothing to prove. Thus, we only must consider the case $i_m > i_*(j_m)$. Using the definition \eqref{eq:h:j:q:def} we have that
\begin{equation*}
   h_{m,j_m,q} = \Gamma_{q+1}^{(m+1)(i_*(j_m-1) -i_*(j_m))}    h_{m,j_m-1,q}
    \,.  
\end{equation*}
On the other hand, for $i'_m \leq i_m-2$ we have from \eqref{eq:hm:overlap:1} that 
\begin{align*}
    h_{m,j_m-1,q} 
    \leq \Gamma_{q+1}^{(m+1)(i'_m+1-i_*(j_m-1))}
    \leq \Gamma_{q+1}^{(m+1)(i_m-1-i_*(j_m-1))} \,.
\end{align*}
Therefore, combining the above two displays and the inequality $-i_*(j_m)\geq -i_*(j_m-1)-1$, we obtain the bound
\begin{align*}
   \Gamma_{q+1}^{-(m+1)(i_m-i_*(j_m))}  h_{m,j_m,q} 
   &\leq  \Gamma_{q+1}^{-(m+1)(i_m-i_*(j_m))}   \Gamma_{q+1}^{(m+1)(i_*(j_m-1) -i_*(j_m))} \Gamma_{q+1}^{(m+1)(i_m-1-i_*(j_m-1))}
    \notag\\
      &=     \Gamma_{q+1}^{-(m+1)} \,,
\end{align*}
As before, since $m\geq 0$ this produces a contradiction with the lower bound on $h_{m,j_m,q}$ given in \eqref{eq:hm:overlap:0}, since $\Gamma_{q+1}^{-1} \ll \sfrac 12$.
\end{proof}

With Lemma~\ref{lem:partition:of:unity:psi:m} in hand, we can now verify the inductive assumption \eqref{eq:inductive:partition} at level $q$.  

\begin{lemma}[\textbf{$\psi_{i,q}$ is a partition of unity}]
\label{lem:partition:of:unity:psi}
We have that  for $q\geq 0$,
\begin{align}
\sum_{i\geq 0} \psi_{i,q}^2\equiv 1\,, \qquad \psi_{i,q}\psi_{i',q}=0\quad\textnormal{for}\quad|i-i'|\geq 2. \label{eq:lemma:partition:2}
\end{align}
\end{lemma}
\begin{proof}[Proof of Lemma~\ref{lem:partition:of:unity:psi}]
When $q=0$, both statements are immediate from \eqref{eq:psi:i:j:0:def}. To prove the first claim for $q\geq 1$, let us introduce the notation 
\begin{equation}\label{eq:Lambda:i:def}
    \Lambda_i = \left\{ \Vec{i} = (i_0,...,i_{\NcutSmall}) \colon \max_{0\leq m\leq\NcutSmall}i_m=i. \right\}
\end{equation}
Then
\begin{equation*}
    \psi_{i,q}^2 = \sum_{\Vec{i}\in\Lambda_i} \prod\limits_{m=0}^{\NcutSmall} \psi_{m,i_m,q}^2,
\end{equation*}
and thus
\begin{align}
    \sum_{i\geq 0} \psi_{i,q}^2 
    = \sum_{i\geq 0}\sum_{\Vec{i}\in\Lambda_i} \prod\limits_{m=0}^{\NcutSmall} \psi_{m,i_m,q}^2 
    &=\sum_{\Vec{i}\in\mathbb{N}_0^{\NcutSmall+1}} \left( \prod\limits_{m=0}^{\NcutSmall} \psi_{m,i_m,q}^2 \right) \notag\\
    &= \prod\limits_{m=0}^{\NcutSmall} \left( \sum_{i_m\geq 0} \psi_{m,i_m,q}^2 \right)
    = \prod\limits_{m=0}^{\NcutSmall} 1 
    =1 \notag
\end{align}
after using \eqref{eq:lemma:partition:1}.

To prove the second claim, assume towards a contradiction that there exists $|i-i'|\geq 2$ such that $\psi_{i,q}\psi_{i',q}\geq 0$.  Then 
\begin{align}
    0 &\neq \psi_{i,q}^2 \psi_{i',q}^2 = \sum_{\Vec{i}\in\Lambda_i} \sum_{\Vec{i}'\in\Lambda_{i'}} \prod\limits_{m=0}^{\NcutSmall} \psi_{m,i_m,q}^2\psi_{m,i'_m,q}^2. \label{eq:lemma:partition:nonzero}
\end{align}
In order for \eqref{eq:lemma:partition:nonzero} to be non-vanishing, by \eqref{eq:lemma:partition:1}, there must exist $\Vec{i}=(i_0,...,i_{\NcutSmall})\in\Lambda_i$ and $\Vec{i}'=(i'_0,...,i'_{\NcutSmall})\in\Lambda_{i'}$ such that $|i_m-i'_m|\leq 1$ for all $0\leq m \leq \NcutSmall$. By the definition of $i$ and $i'$, there exist $m_*$ and $m'_*$ such that
\begin{equation*}
    i_{m_*} = \max_m i_m = i, \qquad     i'_{m'_*} = \max_m i'_m = i'.
\end{equation*}
But then
\begin{align}
    i&=i_{m_*} \leq i'_{m_*} + 1 \leq i'_{m'_*} + 1 = i'+1 \notag \\
    i'&=i'_{m'_*} \leq i_{m'_*} + 1 \leq i_{m_*} + 1 = i+1, \notag
\end{align}
implying that $|i-i'|\leq 1$, a contradiction.
\end{proof}

In view of the preceding two lemmas and \eqref{eq:intermediate:overlapping}, and for convenience of notation, we define 
\begin{align}
\psi_{i\pm, q}(x,t)  & =\left(\psi^2_{i-1, q}(x,t)+\psi^2_{i, q}(x,t)+\psi^2_{i+1, q}(x,t)\right)^{\sfrac12}, \label{def:psi:iq:pm}
\end{align}
which are cutoffs with the property that
\begin{align}
\psi_{i\pm, q} &\equiv 1 \quad \mbox{on} \quad {\supp(\psi_{i,q})} \label{e:psi:i:q:overlap}.
\end{align}

\begin{remark}[\textbf{Rewriting $\psi_{i,q}$}]
\label{rem:rewrite:cutoffs}
The definition \eqref{eq:psi:i:q:recursive} is not convenient to use directly for estimating material derivatives of the $\psi_{i,q}$ cutoffs, because differentiating the terms $\psi_{m,i_m,q}$ {\em individually} ignores certain cancellations which arise due to the fact that $\{ \psi_{m,i_m,q}\}_{i_m\geq 0}$ is a partition of unity (as was shown above in Lemma~\ref{lem:partition:of:unity:psi:m}). For this purpose, we {\em re-sum } the terms in the definition \eqref{eq:psi:i:q:recursive} as follows.
For any given $0\leq m \leq \NcutSmall$, we introduce the summed cutoff function 
\begin{align}
\Psi_{m,i,q}^2 = \sum_{i_m=0}^i \psi_{m,i_m,q}^2 
\label{eq:fancy:cutoff}
\end{align}
and note  via Lemma~\ref{lem:partition:of:unity:psi:m} its chief property:
\begin{align}
D(\Psi_{m,i,q}^2) = D(\psi_{m,i,q}^2)  {\mathbf{1}}_{\supp(\psi_{m,i+1,q})} = D(\psi_{m,i,q}^2)  {\mathbf{1}}_{\supp(\psi_{m,i+1,q})}
\,.
\label{eq:fancy:cutoff:supp}
\end{align}
The above inclusion holds because on the support of $\psi_{m,i_m,q}$ with $i_m<i$, we have that $\Psi_{m,i,q} \equiv 1$. With the notation \eqref{eq:fancy:cutoff} we return to the definition \eqref{eq:psi:i:q:recursive} and note that 
\begin{align}
\psi_{i,q}^2 
&= \sum_{m=0}^{\NcutSmall} \psi_{m,i,q}^2 \prod_{m'=0}^{m-1}  \Psi_{m',i,q}^2  \prod_{m''=m+1}^{\NcutSmall} ( \Psi_{m'',i,q}^2 - \psi_{m'',i,q}^2) 
\notag\\
&= \sum_{m=0}^{\NcutSmall} \psi_{m,i,q}^2 \prod_{m'=0}^{m-1}  \Psi_{m',i,q}^2  \prod_{m''=m+1}^{\NcutSmall}  \Psi_{m'',i-1,q}^2
\,.
\label{eq:cutoff:resummation} 
\end{align}
\end{remark}

\begin{remark}[\textbf{Size of maximal $j_m$ in \eqref{eq:psi:m:im:q:def}}]
Define $j_*(i,q) = \max \{ j \colon i_*(j) \leq i \}$ to be the largest index of $j_m$ appearing in the sum in \eqref{eq:psi:m:im:q:def}. We note here that
\begin{align}
\Gamma_{q+1}^{i-1} < \Gamma_q^{j_*(i,q)} \leq \Gamma_{q+1}^i
\label{eq:max:j:i:q}
\end{align}
holds. This fact will be used later on in the proof in conjunction with Lemma~\ref{lem:maximal:i} to bound the maximal values of $j_m$.
\end{remark}

The following lemma is a direct consequence of the definitions of the cutoffs.
\begin{lemma}
\label{lem:h:j:q:size}
If $(x,t)\in \supp( \psi_{m,i_m,j_m, q})$ then
\begin{align}\label{eq:h:psi:supp}
h_{m,j_m,q}\leq  \Gamma_{q+1}^{(m+1)\left(i_m+1-i_*(j_m)\right)}.
\end{align}
Moreover, if $i_m>i_*(j_m)$ we have
\begin{align}\label{eq:psi:supp:upper}
h_{m,j_m,q} \geq  (\sfrac 12) \Gamma_{q+1}^{(m+1)(i_m-i_*(j_m))} 
\end{align}
on the support of $\psi_{m,i_m,j_m,q}$. 
As a consequence, we have
\begin{align}
\norm{D^N D_{t,q-1}^{m} u_q}_{L^\infty( \supp \psi_{m,i_m,q})}
&\leq \delta_q^{\sfrac 12} \Gamma_{q+1}^{i_m+1} (\lambda_q \Gamma_q)^N (\tau_{q-1}^{-1} \Gamma_{q+1}^{i_m+3}  )^{m} \label{eq:derivatives:psi:i:m:q} \\
\norm{D^N D_{t,q-1}^M u_q}_{L^\infty( \supp \psi_{i,q})}
&\leq \delta_q^{\sfrac 12} \Gamma_{q+1}^{i+1} (\lambda_q \Gamma_q)^N (\tau_{q-1}^{-1} \Gamma_{q+1}^{i+3}  )^M  
\label{eq:derivatives:psi:i:q}
\end{align}
for all $0 \leq m,M \leq \NcutSmall$ and $0 \leq N \leq \NcutLarge$.
\end{lemma}
\begin{proof}[Proof of Lemma~\ref{lem:h:j:q:size}]
Estimates \eqref{eq:h:psi:supp} and \eqref{eq:psi:supp:upper} follow directly from the definitions of $\tilde \psi_{m,q+1}$ and $\psi_{m,q+1}$. In order to prove \eqref{eq:derivatives:psi:i:m:q}, we note that for $(x,t) \in \supp (\psi_{m,i_m,q})$, by \eqref{eq:psi:m:im:q:def} there must exist a $j_m$ with $i_*(j_m) \leq i_m$ such that $(x,t) \in \supp (\psi_{m,i_m,j_m,q})$. Using \eqref{eq:h:psi:supp}, we conclude that 
\begin{align}
\norm{ D^N D_{t,q-1}^m u_q  }_{L^\infty(\supp \psi_{m,i_m,j_m,q})} 
&\leq \Gamma_{q+1}^{(m+1)(i_m+1-i_*(j_m))} \Gamma_{q+1}^{i_*(j_m)} (\Gamma_q \lambda_q)^N (\Gamma_{q+1}^{i_*(j_m)+2} \tau_{q-1}^{-1})^m \delta_q^{\sfrac 12}
\notag\\
&=  \delta_q^{\sfrac 12} \Gamma_{q+1}^{i_m+1} \left(\lambda_q \Gamma_q\right)^N \left( \tau_{q-1}^{-1}  \Gamma_{q+1}^{i_m+3} \right)^m  
\label{eq:derivatives:psi:i:j:q}
\end{align}
which completes the proof of \eqref{eq:derivatives:psi:i:m:q}. The proof of \eqref{eq:derivatives:psi:i:q} follows from the fact that we have employed the \textit{maximum} over $m$ of $i_m$ to define $\psi_{i,q}$ in \eqref{def:psi:i:q:def}.
\end{proof}

 An immediate corollary of the bound \eqref{eq:D:K:psi:i:q} and of the previous Lemma is that estimates for the derivatives of $u_q$ are also available on the support of $\psi_{i,q}$, instead of $\psi_{i,q-1}$.
\begin{corollary}
\label{cor:D:Dt:wq:psi:i:q}
For $N,M \leq 3\Nindv$, and $i \geq 0$, we have the bound 
\begin{align}
\norm{D^N D_{t,q-1}^M u_q}_{L^\infty(\supp \psi_{i,q})} 
 \lesssim \Gamma_{q+1}^{i+1} \delta_{q}^{\sfrac 12}   \MM{N,2\NindLarge,\Gamma_q\lambda_q,\tilde \lambda_q}  \MM{M,\Nindt,\Gamma_{q+1}^{i+3} \tau_{q-1}^{-1},\tilde \tau_{q-1}^{-1}} 
 \,.
\label{eq:D:Dt:wq:psi:i:q}
\end{align}
Recall that if either $N>3\Nindv$ or $M>3\Nindv$ are such that $N+M\leq 2\Nfin$, suitable estimates for $D^N D^M_{t,q-1} u_q$ are already provided by \eqref{eq:mollified:velocity:sup}.
\end{corollary}
\begin{proof}[Proof of Corollary~\ref{cor:D:Dt:wq:psi:i:q}]
When $0 \leq N \leq \NcutLarge$ and $0\leq M \leq \NcutSmall \leq \Nindt$, the desired bound was already established in \eqref{eq:derivatives:psi:i:q}.  

For the remaining cases, note that if $0\leq m \leq \NcutSmall$ and $(x,t) \in \supp \psi_{m,i_m,q}$, there exists $j_m\geq 0$ with $i_*(j_m) \leq i_m$, such that $(x,t) \in \supp \psi_{j_m,q-1}$. Thus, we may appeal to \eqref{eq:D:K:psi:i:q} and deduce that 
\begin{align*}
\abs{D^N D^M_{t,q-1} u_q} \les \delta_q^{\sfrac 12} \tilde \lambda_q^{\sfrac 32} \MM{N,2\Nindv,\lambda_q,\tilde \lambda_q} \MM{M,\Nindt,\Gamma_q^{j_m+1} \tau_{q-1}^{-1},\tilde \tau_{q-1}^{-1}} \,.
\end{align*}
Since $i_*(j_m) \leq i_m$ implies $\Gamma_{q}^{j_m} \leq \Gamma_{q+1}^{i_m}$, we deduce that 
\begin{align*}
\norm{D^N D^M_{t,q-1} u_q}_{L^\infty(\supp \psi_{m,i_m,q})} 
\les \delta_q^{\sfrac 12} \tilde\lambda_q^{\sfrac 32} \MM{N,2\Nindv,\lambda_q,\tilde \lambda_q} \MM{M,\Nindt,\Gamma_{q+1}^{i_m+1} \tau_{q-1}^{-1},\tilde \tau_{q-1}^{-1}} \,.
\end{align*}
Note that the above estimate does not have a factor of $\Gamma_{q+1}^{i_m+1}$ next to the $\delta_q^{\sfrac 12}$ at the amplitude. 

We now consider two cases. If $\NcutLarge < N \leq 3\Nindv$, then 
\begin{align*}
 \MM{N,2\Nindv,\lambda_q,\tilde \lambda_q} \les   \Gamma_q^{-\NcutLarge}   \MM{N,2\Nindv,\Gamma_q \lambda_q,\tilde \lambda_q} \,.
\end{align*}
On the other hand, if  $\NcutSmall < M \leq 3\Nindv$, then 
\begin{align*}
\MM{M,\Nindt,\Gamma_{q+1}^{i_m+1} \tau_{q-1}^{-1},\tilde \tau_{q-1}^{-1}}
\les \Gamma_{q+1}^{-2\NcutSmall}
\MM{M,\Nindt,\Gamma_{q+1}^{i_m+3}  \tau_{q-1}^{-1},\tilde \tau_{q-1}^{-1}} \,.
\end{align*}
Combining the above three displays, and recalling the definition of $\psi_{i,q}$ in \eqref{eq:psi:i:q:recursive}, we deduce that if either $N>\NcutLarge$ or $M>\NcutSmall$, we have
\begin{align*}
&\norm{D^N D^M_{t,q-1} u_q}_{L^\infty(\supp \psi_{i,q})} \notag\\
&\les \delta_q^{\sfrac 12} \tilde \lambda_q^{\sfrac 32} \max\{\Gamma_q^{-\NcutLarge},\Gamma_{q+1}^{-2\NcutSmall}\} \MM{N,2\Nindv,\Gamma_q \lambda_q,\tilde \lambda_q} \MM{M,\Nindt,\Gamma_{q+1}^{i+1} \Gamma_q^2 \tau_{q-1}^{-1},\tilde \tau_{q-1}^{-1}} \,,  
\end{align*}
and the proof of \eqref{eq:D:Dt:wq:psi:i:q} is completed by taking $\NcutLarge$ and $\NcutSmall$ sufficiently large to ensure that 
\begin{align}
\tilde \lambda_q^{\sfrac 32} \max\{\Gamma_q^{-\NcutLarge},\Gamma_{q+1}^{-2\NcutSmall}\} \leq 1\,. 
\label{eq:Nind:cond:1}
\end{align}
This condition holds by \eqref{eq:Nind:cond:3}.
\end{proof}

\subsubsection{Pure spatial derivatives}
In this section we prove that the cutoff functions $\psi_{i,q}$ satisfy sharp spatial derivative estimates, which are consistent with \eqref{eq:sharp:Dt:psi:i:q:old} for $q'=q$.
\begin{lemma}[\textbf{Spatial derivatives for the cutoffs}]
\label{lem:sharp:D:psi:i:q}
Fix $q\geq 1$,  $0 \leq m \leq \NcutSmall$, and $i_m\geq0$. For all $j_m \geq 0$ such that $i_m \geq i_*(j_m)$ and all  $N \leq \Nfin$, we have 
\begin{align}
{\mathbf{1}}_{\supp( \psi_{j_m,q-1})} \frac{|D^N \psi_{m,i_m,j_m,q}|}{\psi_{m,i_m,j_m,q}^{1 -  N/\Nfin}} 
\les \MM{N,\Nindv,\lambda_q \Gamma_q, \tilde \lambda_q\Gamma_q }
\,,
\label{eq:sharp:D:psi:i:j:q}
\end{align}
which in turn implies
 \begin{align}
\frac{|D^N \psi_{i,q}|}{\psi_{i,q}^{1-N/\Nfin}} \lesssim \MM{N,\Nindv,\lambda_q \Gamma_q, \tilde \lambda_q\Gamma_q }
\label{eq:sharp:D:psi:i:q}
\end{align}
for all $i\geq 0$, all $N \leq \Nfin$.
\end{lemma}
\begin{proof}[Proof of Lemma~\ref{lem:sharp:D:psi:i:q}]
We first show that \eqref{eq:D:K:psi:i:q}  implies \eqref{eq:sharp:D:psi:i:j:q}.  We distinguish two cases.  The first case is when $\psi=\tilde\psi_{m,q+1}$, or $\psi=\psi_{m,q+1}$ \emph{and} we have the lower bound
\begin{equation}\label{eq:clusterf:1}
h_{m,j_m,q}^2 \Gamma_{q+1}^{-2\left(i_m-i_*(j_m)\right)(m+1)} \geq \frac{1}{4} \Gamma_{q+1}^{2(m+1)}
\end{equation}
so that \eqref{eq:DN:psi:q:gain} applies. The goal is then to apply Lemma~\ref{lem:Faa:di:Bruno} to the function $\psi = \tilde \psi_{m,q+1}$ or  $\psi = \psi_{m,q+1}$ as described above in conjunction with $\Gamma_\psi = \Gamma_{q+1}^{m+1}$, $\Gamma = \Gamma_{q+1}^{(m+1)(i_m-i_*(j_m))}$, and $h(x,t) = (h_{m,j_m,q}(x,t))^2$. The assumption \eqref{eq:Faa:di:Bruno:lem:1} holds by \eqref{eq:DN:psi:q:0} or \eqref{eq:DN:psi:q:gain} for all $N \leq \Nfin$, and so we need to obtain bounds on the derivatives of $h_{m,j_m,q}^2$, which are consistent with assumption \eqref{eq:Faa:di:Bruno:lem:2} of Lemma~\ref{lem:Faa:di:Bruno}. For $B\leq \Nfin$, the Leibniz rule gives
\begin{align}
\left| D^{B} h_{m,j_m,q}^2\right|
&\lesssim (\lambda_q \Gamma_q)^B \sum_{B' =0}^{B} \sum_{n=0}^{\NcutLarge}  \Gamma_{q+1}^{-i_*(j_m)} (\tau_{q-1}^{-1} \Gamma_{q+1}^{i_*(j_m)+2})^{-m} (\lambda_q \Gamma_q)^{-n-B'}  \delta_{q}^{-\sfrac 12} | D^{n + B'} D^m_{t,q-1} u_{q}| 
\notag\\
&\qquad \qquad    \times \Gamma_{q+1}^{-i_*(j_m)} (\tau_{q-1}^{-1} \Gamma_{q+1}^{i_*(j_m)+2})^{-m} (\lambda_q \Gamma_q)^{-n - B + B'} \delta_{q}^{-\sfrac 12} | D^{n + B-B'} D^m_{t,q-1} u_{q}| 
\,.
\label{eq:cutoff:spatial:derivatives:000}
\end{align}
For the terms with $L \in \{n+B',n+B-B'\} \leq \NcutLarge$ we may appeal to appeal to estimate \eqref{eq:h:psi:supp}, which gives
\begin{align}
\Gamma_{q+1}^{-i_*(j_m)} (\tau_{q-1}^{-1} \Gamma_{q+1}^{i_*(j_m)+2})^{-m} (\lambda_q \Gamma_q)^{-L} \delta_{q}^{-\sfrac 12} \norm{D^L D_{t,q-1}^m u_q}_{L^\infty(\supp \psi_{m,i_m,j_m,q})} 
\leq  \Gamma_{q+1}^{(m+1)(i_m+1 - i_*(j_m))} \, .
\label{eq:cutoff:spatial:derivatives:00}
\end{align}
On the other hand, for $\NcutLarge < L \in \{n+B',n+B-B'\} \leq \NcutLarge + B \leq 2\Nfin - \Nindt$,  we may appeal to appeal to estimates \eqref{eq:mollified:velocity:sup} and \eqref{eq:D:K:psi:i:q},  and since $m\leq \NcutSmall < \Nindt$, we deduce that
\begin{align}
&\Gamma_{q+1}^{-i_*(j_m)} (\tau_{q-1}^{-1} \Gamma_{q+1}^{i_*(j_m)+2})^{-m} (\lambda_q \Gamma_q)^{-L} \delta_{q}^{-\sfrac 12} \norm{D^L D_{t,q-1}^m u_q}_{L^\infty(\supp \psi_{j_m,q-1})} \notag \\
&\lesssim  (\Gamma_q^{j_m+1} \Gamma_{q+1}^{-i_*(j_m)-2})^{m}  (\Gamma_q^{-L} \tilde \lambda_q^{\sfrac 32}) \lambda_q^{-L} \MM{L,2\Nindv,\lambda_q,\tilde \lambda_q} \notag\\
&\les \MM{L,2\Nindv,1,\lambda_q^{-1} \tilde \lambda_q} \notag\\
&\leq  \Gamma_{q+1}^{(m+1)(i_m+1 - i_*(j_m))}  \MM{L,2\Nindv,1,\lambda_q^{-1} \tilde \lambda_q} .
\label{eq:cutoff:spatial:derivatives:0}
\end{align}
In the last inequality we have used that $i_m \geq i_*(j_m)$, while in the second to last inequality we have used that if $L\geq \NcutLarge$ then $\Gamma_q^L \geq  \tilde  \lambda_q^{\sfrac 32}$, which follows once $\NcutLarge$ is chosen to be sufficiently large, as in \eqref{eq:Nind:cond:3}. 
Summarizing the bounds \eqref{eq:cutoff:spatial:derivatives:000}--\eqref{eq:cutoff:spatial:derivatives:0}, since $n\leq \NcutLarge$, we arrive at
\begin{align*}
& {\mathbf{1}}_{ \supp(\psi_{j_m,q-1} \psi_{m,i_m,j_m,q}  )} \left| D^{B} h_{m,j_m,q}^2\right|  \notag\\
& \lesssim (\lambda_q \Gamma_q)^B  \MM{2\NcutLarge+B,2\Nindv,1,\lambda_q^{-1} \tilde \lambda_q} \Gamma_{q+1}^{2(m+1)(i_m+1 - i_*(j_m))} \notag\\
& \lesssim   \MM{B,\Nindv,\lambda_q \Gamma_q, \tilde \lambda_q\Gamma_q } \Gamma_{q+1}^{2(m+1)(i_m+1 - i_*(j_m))}
\,.
\end{align*}
whenever $B\leq \Nfin$. Here we have used that $2\NcutLarge \leq \Nindv$. Thus, assumption \eqref{eq:Faa:di:Bruno:lem:2} holds with $C_h = \Gamma_{q+1}^{2(m+1)(i_m+1 - i_*(j_m))}$, $\lambda =  \Gamma_q \lambda_{q}$, $\Lambda = \tilde \lambda_q \Gamma_q$, $N_* = \Nindv$. Note that with these choices of parameters, we have $ C_h \Gamma_\psi^{-2} \Gamma^{-2} =1 $. We may thus apply Lemma~\ref{lem:Faa:di:Bruno} and conclude that 
\begin{align*}
{\mathbf{1}}_{\supp(\psi_{j_m,q-1})} \frac{\left|D^N \psi_{m,i_m,j_m,q} \right|}{\psi_{m,i_m,j_m,q}^{1-N/\Nfin}} \les \MM{N,\Nindv,\lambda_q \Gamma_q, \tilde \lambda_q\Gamma_q }
\end{align*}
for all $N \leq \Nfin$, proving \eqref{eq:sharp:D:psi:i:j:q} in the first case.

Recalling the inequality \eqref{eq:clusterf:1}, the second case is when $\psi=\psi_{m,q+1}$ and
\begin{equation}\label{eq:clusterf:2}
h_{m,j_m,q}^2 \Gamma_{q+1}^{-2\left(i_m-i_*(j_m)\right)(m+1)} \leq \frac{1}{4} \Gamma_{q+1}^{2(m+1)}.
\end{equation}
However, since $\psi_{m,q+1}$ is uniformly equal to $1$ when the left hand side of the above display takes values in $\left[1,\frac{1}{4}\Gamma_{q+1}^{2(m+1)}\right]$, \eqref{eq:sharp:D:psi:i:j:q} is trivially satisfied.  Thus we may reduce to the case that 
\begin{equation}\label{eq:clusterf:3}
h_{m,j_m,q}^2 \Gamma_{q+1}^{-2\left(i_m-i_*(j_m)\right)(m+1)} \leq 1.
\end{equation}
As in the first case, we aim to apply Lemma~\ref{lem:Faa:di:Bruno} with $h=h_{m,j_m,q}^2$, but now with $\Gamma_\psi=1$ and $\Gamma=\Gamma_{q+1}^{(m+1)(i_m-i_*(j_m))}$.  From \eqref{eq:DN:psi:q}, the assumption \eqref{eq:Faa:di:Bruno:lem:1} holds.  Towards estimating derivatives of $h$, for the terms with $L\in\{n+B',n+B-B'\}\leq\NcutLarge$, \eqref{eq:clusterf:3} gives immediately that
\begin{align}
\Gamma_{q+1}^{-i_*(j_m)} (\tau_{q-1}^{-1} \Gamma_{q+1}^{i_*(j_m)+2})^{-m} (\lambda_q \Gamma_q)^{-L} \delta_{q}^{-\sfrac 12} \norm{D^L D_{t,q-1}^m u_q}_{L^\infty(\supp \psi_{m,i_m,j_m,q})} 
\leq  \Gamma_{q+1}^{(m+1)(i_m-i_*(j_m))} \, .
\label{eq:cutoff:spatial:derivatives:00:cf}
\end{align}
Conversely, when $\NcutLarge>L$, we may argue as in the estimates which gave \eqref{eq:cutoff:spatial:derivatives:0}, only this time using that since $i_m\geq i_*(j_m)$, we can achieve the slightly improved bound\footnote{This bound was also available in \eqref{eq:cutoff:spatial:derivatives:0}, but we wrote the worse bound there to match the chosen value of $\const_h$.}
\begin{equation}\label{eq:cutoff:spatial:derivatives:0:cf}
\Gamma_{q+1}^{(m+1)(i_m-i_*(j_m))} \MM{L,2\Nindv,1,\lambda_q^{-1}\tilde\lambda_q}.
\end{equation}
We then arrive at
\begin{align*}
{\mathbf{1}}_{ \supp(\psi_{j_m,q-1} \psi_{m,i_m,j_m,q}  )} &\left| D^{B} h_{m,j_m,q}^2\right|  \notag\\
& \lesssim (\lambda_q \Gamma_q)^B  \MM{2\NcutLarge+B,2\Nindv,1,\lambda_q^{-1} \tilde \lambda_q} \Gamma_{q+1}^{2(m+1)(i_m-i_*(j_m))} \notag\\
& \lesssim   \MM{B,\Nindv,\lambda_q \Gamma_q, \tilde \lambda_q\Gamma_q } \Gamma_{q+1}^{2(m+1)(i_m-i_*(j_m))}
\,.
\end{align*}
whenever $B\leq \Nfin$, again using that $2\NcutLarge \leq \Nindv$. Thus, assumption \eqref{eq:Faa:di:Bruno:lem:2} now holds with $\const_h = \Gamma_{q+1}^{2(m+1)(i_m- i_*(j_m))}$, $\lambda =  \Gamma_q \lambda_{q}$, $\Lambda = \tilde \lambda_q \Gamma_q$, $N_* = \Nindv$. Note that with these new choices of parameters, we still have $\const_h \Gamma_\psi^{-2} \Gamma^{-2} =1 $. We may thus apply Lemma~\ref{lem:Faa:di:Bruno} and conclude that 
\begin{align*}
{\mathbf{1}}_{\supp(\psi_{j_m,q-1})} \frac{\left|D^N \psi_{m,i_m,j_m,q} \right|}{\psi_{m,i_m,j_m,q}^{1-N/\Nfin}} \les \MM{N,\Nindv,\lambda_q \Gamma_q, \tilde \lambda_q\Gamma_q }
\end{align*}
for all $N \leq \Nfin$, proving \eqref{eq:sharp:D:psi:i:j:q} in the second case.

From the definition \eqref{eq:psi:m:im:q:def}, and the bound \eqref{eq:sharp:D:psi:i:j:q} we next estimate derivatives of the $m^{th}$ velocity cutoff function $\psi_{m,i_m,q}$, and claim that 
 \begin{align}
\frac{|D^N \psi_{m,i_m,q}|}{\psi_{m,i_m,q}^{1-N/\Nfin}} \lesssim 
\MM{N,\Nindv,\lambda_{q} \Gamma_{q}, \tilde \lambda_q \Gamma_q}
\label{eq:sharp:D:psi:im:q}
\end{align}
for all $i_m\geq 0$, all $N \leq \Nfin$.  We prove \eqref{eq:sharp:D:psi:im:q} by induction on $N$. When $N=0$ the bound trivially holds, which gives the induction base. For the induction step, assume that \eqref{eq:sharp:D:psi:im:q}  holds  for all $N' \leq N-1$.  By the Leibniz rule we obtain
\begin{align}
D^N (\psi_{m,i_m,q}^2) = 2 \psi_{m,i_m,q} D^N \psi_{m,i_m,q} + \sum_{N'=1}^{N-1} {N \choose N'} D^{N'} \psi_{m,i_m,q}\, D^{N-N'} \psi_{m,i_m,q}
\label{eq:psi:m:i:q:Leibniz}
\end{align}
and thus
\begin{align*}
\frac{D^N \psi_{m,i_m,q}}{\psi_{m,i_m,q}^{1-N/\Nfin}} 
&= \frac{D^N(\psi_{m,i_m,q}^2)}{2 \psi_{m,i_m,q}^{2-N/\Nfin}} - \frac 12 \sum_{N'=1}^{N-1} {N \choose N'} \frac{D^{N'} \psi_{m,i_m,q}}{\psi_{m,i_m,q}^{1-N'/\Nfin}}  \frac{D^{N-N'} \psi_{m,i_m,q}}{\psi_{m,i_m,q}^{1-(N-N')/\Nfin}}. 
\end{align*}
Since $N', N-N' \leq N-1$ by the induction assumption \eqref{eq:sharp:D:psi:im:q} we obtain
\begin{align}
\frac{\left| D^N \psi_{m,i_m,q} \right|}{\psi_{m,i_m,q}^{1-N/\Nfin}} 
&\lesssim \frac{ |D^N(\psi_{m,i_m,q}^2) |}{\psi_{m,i_m,q}^{2-N/\Nfin}} +\MM{N,\Nindv,\lambda_{q} \Gamma_{q}, \tilde \lambda_q \Gamma_q}.
\label{eq:cutoff:spatial:derivatives:1}
\end{align}
Thus, establishing \eqref{eq:sharp:D:psi:im:q}  for the $N$th derivative, reduces to bounding the first term on the right side of the above. For this purpose we recall \eqref{eq:psi:m:im:q:def} and compute 
\begin{align*}
\frac{\left|D^N(\psi_{m,i_m,q}^2)\right|}{\psi_{m,i_m,q}^{2-N/\Nfin}}  
&= \frac{1}{\psi_{m,i_m,q}^{2-N/\Nfin}} \sum_{\{ j_m \colon i_*(j_m) \leq i_m\}} \sum_{K=0}^N {N \choose K} D^K(\psi_{j_m,q-1}^2) D^{N-K}(\psi_{m,i_m,j_m,q}^2) 
\\
&=   \sum_{\{ j_m \colon i_*(j_m) \leq i_m\}} \sum_{K=0}^N \sum_{L_1=0}^K \sum_{L_2=0}^{N-K} {N \choose K} {K \choose L_1} {N-K \choose L_2} 
\frac{\psi_{j_m,q-1}^{2-  K/\Nfin} \psi_{m,i_m,j_m,q}^{2- (N-K) /\Nfin} }{\psi_{m,i_m,q}^{2-N/\Nfin}}  
\\
&\qquad \qquad \times
\frac{D^{L_1} \psi_{j_m,q-1}}{\psi_{j_m,q-1}^{1- L_1/\Nfin}}
\frac{D^{K-L_1} \psi_{j_m,q-1}}{\psi_{j_m,q-1}^{1- (K-L_1)/\Nfin}}
\frac{D^{L_2} \psi_{m,i_m,j_m,q}}{\psi_{m,i_m,j_m,q}^{1- L_2/\Nfin}} 
\frac{D^{N-K-L_2} \psi_{m,i_m,j_m,q}}{\psi_{m,i_m,j_m,q}^{1- (N-K-L_2)/\Nfin}}
\, . 
\end{align*}
Since $K, N-K \leq N$, and $\psi_{j_m,q-1}, \psi_{m,i_m,j_,q} \leq 1$ we have by \eqref{eq:psi:i:q:recursive} that 
\begin{align*}
\frac{\psi_{j_m,q-1}^{2-  K/\Nfin} \psi_{m,i_m,j_m,q}^{2- (N-K)/\Nfin} }{\psi_{m,i_m,q}^{2-N/\Nfin}}   \leq \frac{\psi_{j_m,q-1}^{2-  N/\Nfin} \psi_{m,i_m,j_m,q}^{2-N/\Nfin} }{\psi_{m,i_m,q}^{2-N/\Nfin}}  \leq 1.
\end{align*}
Furthermore, the estimate \eqref{eq:sharp:D:psi:i:j:q}, the inductive assumption \eqref{eq:sharp:Dt:psi:i:q:old}, combined with the parameter estimate $ \Gamma_{q-1} \tilde \lambda_{q-1} \leq \Gamma_{q} \lambda_q$ (see \eqref{eq:Lambda:q:x:1}) and the previous three displays,  conclude the proof of \eqref{eq:sharp:D:psi:im:q}. In particular, note that this upper bound is independent of the value of $i_m$. 

In order to conclude the proof of the Lemma, we argue that \eqref{eq:sharp:D:psi:im:q} implies \eqref{eq:sharp:D:psi:i:q}. Recalling \eqref{eq:psi:i:q:recursive}, we have that $\psi_{i,q}^2$ is given as a sum of products of  $\psi_{m,i_m,q}^2$, for which suitable derivative bounds are available (due to \eqref{eq:sharp:D:psi:im:q}). Thus, the proof of \eqref{eq:sharp:D:psi:i:q} is again done by induction on $N$, mutatis mutandi to the proof of \eqref{eq:sharp:D:psi:im:q}: indeed, we note that $\psi_{m,i_m,q}^2$ was also given as a sum of squares of cutoff functions, for which derivative bounds were available. The proof of the induction step is thus again based on the application of the Leibniz rule for $\psi_{i,q}^2$; in order to avoid redundancy we omit these details.
\end{proof}

\subsubsection{Maximal indices appearing in the cutoff}

A consequence of the inductive assumptions, Lemma~\ref{lem:h:j:q:size}, and of Lemma~\ref{lem:sharp:D:psi:i:q} above, is that we may a priori estimate the maximal  $i$ appearing in $\psi_{i,q}$, labeled as $i_{\mathrm{max}}(q)$. 

\begin{lemma}[\textbf{Maximal $i$ index in the definition of the cutoff}]
\label{lem:maximal:i}
There exists $\imax = \imax(q) \geq 0$, determined by the formula \eqref{eq:imax:def} below, such that 
\begin{align}
\psi_{i,q} \equiv 0 \quad \mbox{for all} \quad i > i_{\mathrm{max}}
\label{eq:imax}
\end{align}
and 
\begin{align}
\Gamma_{q+1}^{i_{\mathrm{max}}} \leq  \lambda_q^{\sfrac 53}
\label{eq:imax:bound}
\end{align}
holds for all $q\geq 0$, where the implicit constant is independent of $q$. Moreover $\imax(q)$ is bounded uniformly in $q$ as
\begin{align}
\imax(q)  \leq  {\frac{4}{\eps_\Gamma (b-1)}}
\,,
\label{eq:imax:upper:bound:uniform}
\end{align}
assuming $\lambda_0$ is sufficiently large.
\end{lemma}
\begin{proof}[Proof of Lemma~\ref{lem:maximal:i}]
Assume $i\geq 0$ is such that $\supp(\psi_{i,q}) \neq \emptyset$. Our goal is to prove that $\Gamma_{q+1}^i \leq \lambda_{q}^{\sfrac 53}$. 

From \eqref{eq:psi:i:q:recursive} it follows that for any $(x,t) \in \supp (\psi_{i,q})$, there must exist at least one $\Vec{i} = (i_0,\ldots,i_{\NcutSmall})$ such that $\max\limits_{0\leq m\leq\NcutSmall} i_m =i$, and with $\psi_{m,i_m,q}(x,t) \neq 0$ for all $0\leq m \leq \NcutSmall$. Therefore, in light of \eqref{eq:psi:m:im:q:def}, for each such $m$ there exists a maximal $j_m$ such that $i_*(j_m) \leq i_m$, with $(x,t) \in \supp(\psi_{j_m,q-1}) \cap \supp(\psi_{m,i_m,j_m,q})$. In particular, this holds for any of the indices $m$ such that $i_m = i$. For the remainder of the proof, we fix such an index $0\leq m \leq \NcutSmall$.

If we have $i = i_m = i_{*}(j_m) = i_*(j_m,q)$, since $(x,t) \in \supp(\psi_{j_m,q-1}) $, then by the inductive assumption \eqref{eq:imax:old}, we have that $j_m \leq \imax(q-1)$. Then, due to \eqref{eq:max:j:i:q}, we have  $\Gamma_{q+1}^{i-1} < \Gamma_q^{j_m} \leq \Gamma_{q}^{\imax(q-1)}$, and thus 
\begin{align}
\label{eq:imax:contradiction:1aa}
\Gamma_{q+1}^i  
\leq 
\Gamma_{q+1} \Gamma_q^{\imax(q-1)}
\leq 
\Gamma_{q+1} \lambda_{q-1}^{\sfrac 53}
<
\lambda_q^{\sfrac 53}.
\end{align}
The last inequality above uses the fact that $\lambda_q^{\sfrac{(b+1)}{2}} \leq \lambda_{q+1}$ since $b>1$ and $a$ is taken sufficiently large. 

On the other hand, if $i = i_m \geq i_{*}(j_m) +1$, from \eqref{eq:psi:supp:upper} we have $|h_{m,j_m,q}(x,t)| \geq (\sfrac{1}{2}) \Gamma_{q+1}^{(m+1)(i_m - i_*(j_m))}$, and by the pigeonhole principle, there exists $0 \leq n \leq \NcutLarge$  with 
\begin{align*}
|D^n D_{t,q-1}^m u_q(x,t)| &\geq \frac{1}{2 \NcutLarge} \Gamma_{q+1}^{(m+1)(i_m - i_*(j_m))} {\Gamma_{q+1}^{i_*(j_m)}} \delta_q^{\sfrac 12} (\lambda_q \Gamma_q)^{n} (\tau_{q-1}^{-1} \Gamma_{q+1}^{i_*(j_m)+2})^m \notag\\
&\geq   \frac{1}{2 \NcutLarge} \Gamma_{q+1}^{i_m} \delta_q^{\sfrac 12} \lambda_q^{n} (\tau_{q-1}^{-1} \Gamma_{q+1}^{i_m+2})^m ,
\end{align*}
and we also know that $(x,t) \in \supp(\psi_{j_m,q-1})$. By \eqref{eq:D:K:psi:i:q},  the fact that $\NcutLarge \leq 2 \NindLarge -2 $, and $\NcutSmall\leq \NindSmall$, we know that 
\begin{align*}
|D^n D_{t,q-1}^m u_q(x,t)| 
&\leq M_b \delta_q^{\sfrac{1}{2}} \lambda_q^{n} \tilde \lambda_q^{\sfrac 32} (\tau_{q-1}^{-1} \Gamma_q^{j_m+1})^m \notag\\
&\leq M_b \delta_q^{\sfrac{1}{2}} \lambda_q^{n} \tilde \lambda_q^{\sfrac 32} (\tau_{q-1}^{-1} \Gamma_{q+1}^{i_*(j_m)+1})^m \notag\\
&\leq M_b \delta_q^{\sfrac{1}{2}} \lambda_q^{n} \tilde \lambda_q^{\sfrac 32} (\tau_{q-1}^{-1} \Gamma_{q+1}^{i_m})^m
\end{align*}
for some constant $M_b$ which is the maximal constant appearing in the $\les$ symbol of  \eqref{eq:D:K:psi:i:q} with $n+m \leq \Nfin$. In particular, $M_b$ is independent of $q$.
The proof is now completed, since the previous two inequalities and the assumption that $i_m = i \geq i_{\mathrm{max}}(q)+1$ imply that 
\begin{equation}\label{eq:imax:contradiction:1}
\Gamma_{q+1}^i \leq 2\NcutLarge M_b \tilde \lambda_q^{\sfrac 32}
\leq \lambda_q^{\sfrac 53}.
\end{equation}

In view of \eqref{eq:imax:contradiction:1aa} and \eqref{eq:imax:contradiction:1}, the value of $i_{\mathrm{max}}$ is chosen as 
\begin{align}
i_{\mathrm{max}}(q)  = \sup \left\{ i' \colon \Gamma_{q+1}^{i'} \leq \lambda_q^{\sfrac 53} \right\}  
\,.
\label{eq:imax:def}
\end{align}
To show that $\imax(q)<\infty$, and in particular that it is bounded independently of $q$, note that 
\begin{align*}
\frac{\log(\lambda_q^{\sfrac 53})}{\log (\Gamma_{q+1})} \to  \frac{\sfrac 53}{\eps_\Gamma (b-1) } \, ,
\end{align*}
as $q\to \infty$. This, assuming $\lambda_0$ is sufficiently large, since $(b-1) \eps_\Gamma \leq \sfrac 15$, the bound \eqref{eq:imax:upper:bound:uniform} holds.  
\end{proof}

\subsubsection{Mixed derivative estimates}

Recall from \eqref{eq:Dq:definition} the notation $D_q = u_q \cdot \nabla$ for the directional derivative in the direction of $u_q$. With this notation, cf.~\eqref{eq:cutoffs:dtqdtq-1} we have $D_{t,q} = D_{t,q-1} + D_q$. Thus, $D_q$ derivatives are useful for transferring bounds on $D_{t,q-1}$ derivatives to bounds for $D_{t,q}$ derivatives.

From the Leibniz rule we have that
\begin{align}
D_q^K = \sum_{j=1}^{K} f_{j,K} D^j
\label{eq:D:q:K:i}
\end{align}
where 
\begin{align}
f_{j,K} = \sum_{\{ \gamma \in \N^K \colon |\gamma|=K-j \} } c_{j,K,\gamma} \prod_{\ell=1}^K D^{\gamma_{\ell}} u_q
\label{eq:D:q:K:ii}
\end{align}
where  $c_{j,K,\gamma}$ are explicitly computable coefficients that depend only on $K,j$, and $\gamma$. Similarly to the coefficients in \eqref{eq:ad:Dt:a:D}, the precise value of these constants is not important, since all the indices appearing throughout the proof are taken to be less than $2 \Nfin$. The decomposition \eqref{eq:D:q:K:i}--\eqref{eq:D:q:K:ii} will be used frequently in this section.

\begin{remark}
Since throughout the paper the maximal number of spatial or material derivatives is bounded from above by $2\Nfin$, which is a number that is independent of $q$, we have not explicitly stated the formula for  the coefficients $c_{a,k,\beta}$ in \eqref{eq:ad:Dt:a:D}, as all these constants will be  absorbed in a $\lesssim$ symbol. We note however that the proof of Lemma~\ref{lem:ad:Dt:a:D} does yield a recursion relation for the $c_{a,k,\beta}$, which may be used if desired to compute the $c_{a,k,\beta}$ explicitly. 
\end{remark}

With the notation in \eqref{eq:D:q:K:ii} we have the following bounds.
\begin{lemma}
\label{lem:D:a:fj}
 For $q\geq 1$ and $1 \leq K \leq  2\Nfin$, the functions $\{ f_{j,K}\}_{j=1}^{K}$ defined in \eqref{eq:D:q:K:ii} obey the estimate
\begin{align}
\norm{D^a f_{j,K}}_{L^\infty (\supp \psi_{i,q})} &\les  (\Gamma_{q+1}^{i+1} \delta_q^{\sfrac 12})^K  \MM{a + K -j,2\Nindv,\Gamma_q\lambda_q,\tilde \lambda_q}.
\label{eq:D:a:fj} 
\end{align}
for any $a \leq 2 \Nfin-K+j$, and any $0\leq i\leq \imax(q)$.
\end{lemma}
\begin{proof}[Proof of Lemma~\ref{lem:D:a:fj}]
Note that no material derivative appears in  \eqref{eq:D:q:K:ii}, and thus to establish \eqref{eq:D:a:fj} we appeal to  Corollary~\ref{cor:D:Dt:wq:psi:i:q} with $M=0$,  and to the bound \eqref{eq:mollified:velocity:sup} with $m=0$. 
From the product rule we obtain that  
\begin{align*}
\norm{D^a f_j}_{L^\infty(\supp \psi_{i,q})} 
&\lesssim \sum_{\{ \gamma \in \N^K \colon |\gamma|=K-j \}} \sum_{\{ \alpha \in \N^k \colon |\alpha| = a\} } \prod_{\ell=1}^K \norm{D^{\alpha_\ell + \gamma_{\ell}} u_q}_{L^\infty(\supp \psi_{i,q})}  \notag\\
&\lesssim \sum_{\{ \gamma \in \N^K \colon |\gamma|=K-j \}} \sum_{\{ \alpha \in \N^k \colon |\alpha| = a\} } \prod_{\ell=1}^K 
\Gamma_{q+1}^{i+1} \delta_{q}^{\sfrac 12}  \MM{\alpha_\ell+\gamma_\ell,2\Nindv,\Gamma_q\lambda_q,\tilde \lambda_q}
 \notag\\
&\lesssim (\Gamma_{q+1}^{i+1} \delta_q^{\sfrac 12})^K  \MM{a + K -j,2 \Nindv,\Gamma_q\lambda_q,\tilde \lambda_q}
\end{align*} 
since $|\gamma|= K-j$.
\end{proof}

Next, we supplement the space-and-material derivative estimates for $u_q$ obtained in \eqref{eq:mollified:velocity:sup} and \eqref{eq:D:Dt:wq:psi:i:q}, with derivatives bounds that combine 
space, directional, and material derivatives.

\begin{lemma}
\label{lem:Dt:Dt:wq:psi:i:q}
For $q\geq 1$ and $0 \leq i \leq \imax$, we have that 
 \begin{align*}
&\norm{D^N D_q^K D_{t,q-1}^M u_q}_{L^\infty(\supp \psi_{i,q})} \notag\\
&\les  
(\Gamma_{q+1}^{i+1} \delta_q^{\sfrac 12})^{K+1} 
\MM{N+K,2 \Nindv,\Gamma_q\lambda_q,\tilde \lambda_q}  
\MM{M,\NindSmall,\Gamma_{q+1}^{i+3}  \tau_{q-1}^{-1},\tilde \tau_{q-1}}\notag\\
&\les  
(\Gamma_{q+1}^{i+1} \delta_q^{\sfrac 12}) 
\MM{N,2\Nindv,\Gamma_{q}\lambda_q,\tilde \lambda_q} (\Gamma_{q+1}^{i-\cstar} \tau_{q}^{-1} )^{K}  \MM{M,\NindSmall,\Gamma_{q+1}^{i+3}   \tau_{q-1}^{-1},\tilde \tau_{q-1}^{-1}}
\end{align*}
holds for $0 \leq K + N  + M \leq 2\Nfin$.  
\end{lemma}
\begin{proof}[Proof of Lemma~\ref{lem:Dt:Dt:wq:psi:i:q}]
The second estimate in the Lemma follows from the parameter inequality 
$\Gamma_{q+1}^{1+\cstar} \tilde \lambda_q \delta_q^{1/2} \leq   \tau_q^{-1}$, 
which is a consequence of \eqref{eq:Lambda:q:x:1:NEW}. In order to prove the first statement, we let $0 \leq a \leq N$ and $1 \leq j \leq K$.
From estimate \eqref{eq:D:Dt:wq:psi:i:q} and \eqref{eq:mollified:velocity:sup} we obtain
\begin{align*}
\norm{D^{N-a+j} D_{t,q-1}^M u_q}_{L^\infty (\supp \psi_{i,q})} 
&\les  (\Gamma_{q+1}^{i+1} \delta_{q}^{\sfrac 12}) \MM{N-a+j,2\Nindv,\Gamma_q\lambda_q,\tilde \lambda_q} \notag\\
&\quad \times \MM{M,\NindSmall,\Gamma_{q+1}^{i+3} \tau_{q-1}^{-1},\tilde \tau_{q-1}^{-1}},
\end{align*}
which may be combined with \eqref{eq:D:q:K:i}--\eqref{eq:D:q:K:ii}, and the bound \eqref{eq:D:a:fj}, to obtain that
\begin{align*}
&\norm{D^N D_q^K D_{t,q-1}^M u_q}_{L^\infty(\supp \psi_{i,q})}  \notag\\
&\quad \lesssim \sum_{a = 0}^N \sum_{j=1}^K   \norm{D^a f_{j,K}}_{L^\infty (\supp \psi_{i,q})} \norm{D^{N-a+j} D_{t,q-1}^M w_q}_{L^\infty (\supp \psi_{i,q})} \notag\\
&\quad \lesssim (\Gamma_{q+1}^{i+1} \delta_q^{\sfrac 12})^{K+1}  \MM{N+K,2\Nindv,\Gamma_q\lambda_q,\tilde \lambda_q} \MM{M,\NindSmall,\Gamma_{q+1}^{i+3}   \tau_{q-1}^{-1},\tilde \tau_{q-1}^{-1}}
\end{align*}
holds, concluding the proof of the lemma. 
\end{proof}

The next Lemma  shows that the inductive assumptions~\eqref{eq:nasty:D:wq:old}--\eqref{eq:nasty:Dt:wq:WEAK:old} hold also for $q' = q$.

\begin{lemma}
\label{lem:Dt:Dt:wq:psi:i:q:multi}
For $q\geq 1$, $k\geq 1$, $\alpha,\beta \in {\mathbb N}^k$ with $|\alpha| = K$ and $|\beta| = M$, we have
\begin{align}
&\norm{ \Big( \prod_{i=1}^k D^{\alpha_i} D_{t,q-1}^{\beta_i} \Big) u_q }_{L^\infty(\supp \psi_{i,q})} 
\notag\\
&\quad \lesssim (\Gamma_{q+1}^{i+1} \delta_{q}^{\sfrac 12}) \MM{K,2\Nindv,\Gamma_{q} \lambda_q,
\tilde \lambda_q} \MM{M,\NindSmall,\Gamma_{q+1}^{i+3}  \tau_{q-1}^{-1}, \Gamma_{q+1}^{-1} \tilde \tau_{q}^{-1}}    \label{eq:nasty:D:wq}
\end{align}
for all $K +M \leq \sfrac{3\Nfin}{2}+1$.  Additionally, for $N\geq 0$, the bound
\begin{align}
&\norm{D^N \Big( \prod_{i=1}^k D_q^{\alpha_i} D_{t,q-1}^{\beta_i} \Big) u_q }_{L^\infty(\supp \psi_{i,q})}\notag\\
&\quad\lesssim   (\Gamma_{q+1}^{i+1} \delta_q^{\sfrac 12})^{K+1} \MM{N+K,2\Nindv,\Gamma_{q} \lambda_q,\tilde \lambda_q}   \MM{M,\NindSmall,\Gamma_{q+1}^{i+3}   \tau_{q-1}^{-1},\Gamma_{q+1}^{-1} \tilde \tau_{q}^{-1}}  \label{eq:nasty:Dt:wq} \\
&\quad  
\lesssim  (\Gamma_{q+1}^{i+1} \delta_{q}^{\sfrac 12}) \MM{N,2 \Nindv ,\Gamma_{q} \lambda_q,\tilde \lambda_q} (\Gamma_{q+1}^{i-\cstar}  \tau_q^{-1})^{K} \MM{M,\NindSmall,\Gamma_{q+1}^{i+3}  \tau_{q-1}^{-1},\Gamma_{q+1}^{-1} \tilde \tau_{q}^{-1}}
\label{eq:nasty:Dt:wq:WEAK}
\end{align}
holds for all  $0 \leq K + M + N \leq \sfrac{3\Nfin}{2}+1$. Lastly, we have the estimate 
\begin{align}
& \norm{ \Big( \prod_{i=1}^k D^{\alpha_i} D_{t,q}^{\beta_i} \Big) D \vlq }_{L^\infty(\supp \psi_{i,q})} \notag\\
&\qquad \lesssim  (\Gamma_{q+1}^{i+1} \delta_{q}^{\sfrac 12}\tilde \lambda_q) \MM{K ,2\Nindv,\Gamma_{q} \lambda_q,\tilde \lambda_q} 
\MM{M,\NindSmall,\Gamma_{q+1}^{i-\cstar} \tau_q^{-1},\Gamma_{q+1}^{-1} \tilde \tau_{q}^{-1}}
\label{eq:nasty:D:vq}
\end{align}
for all  $ K + M \leq \sfrac{3\Nfin}{2}$, and 
\begin{align}
& \norm{ \Big( \prod_{i=1}^k D^{\alpha_i} D_{t,q}^{\beta_i} \Big)  \vlq }_{L^\infty(\supp \psi_{i,q})} \notag\\
&\qquad \lesssim  (\Gamma_{q+1}^{i+1} \delta_{q}^{\sfrac 12}\lambda_q^{2} ) \MM{K,2\Nindv,\Gamma_{q} \lambda_q,\tilde \lambda_q} \MM{M,\NindSmall,\Gamma_{q+1}^{i-\cstar} \tau_q^{-1},\Gamma_{q+1}^{-1} \tilde \tau_{q}^{-1}}
\label{eq:nasty:no:D:vq}
\end{align}
for all  $ K + M \leq \sfrac{3\Nfin}{2}+1$.
\end{lemma}

\begin{remark}
As shown in Remark~\ref{rem:D:t:q':orangutan}, the bound \eqref{eq:nasty:Dt:wq:WEAK} and identity~\eqref{eq:cooper:1} imply that estimate \eqref{eq:nasty:Dt:uq:orangutan} also holds with $q'=q$. 
\end{remark}

\begin{proof}[Proof of Lemma~\ref{lem:Dt:Dt:wq:psi:i:q:multi}]

We note that \eqref{eq:nasty:Dt:wq:WEAK} follows directly from \eqref{eq:nasty:Dt:wq}, by appealing to the parameter inequality $\Gamma_{q+1}^{1+\cstar} \delta_q^{1/2} \tilde \lambda_q \leq  \tau_q^{-1}$, which is a consequence of \eqref{eq:Lambda:q:x:1:NEW}. We first  show that \eqref{eq:nasty:D:wq} holds, then establish \eqref{eq:nasty:Dt:wq}, and lastly, prove the bounds \eqref{eq:nasty:D:vq}--\eqref{eq:nasty:no:D:vq}. 

\textbf{Proof of \eqref{eq:nasty:D:wq}.\,}
The statement is proven by induction on $k$. For $k=1$ the estimate is given by Corollary~\ref{cor:D:Dt:wq:psi:i:q} and the bound \eqref{eq:mollified:velocity:sup}; in fact, for $k=1$ we have derivatives estimates up to level $2\Nfin$, and not just $\sfrac{3\Nfin}{2}+1$. For the induction step, assume that \eqref{eq:nasty:D:wq} holds  for any $k' \leq k-1$. We denote
\begin{align}
P_{k'} = \Big( \prod_{i=1}^{k'} D^{\alpha_i} D_{t,q-1}^{\beta_i}\Big)  u_q
\label{eq:P:k':def}
\end{align}
and write
\begin{align}
\Big( \prod_{i=1}^k D^{\alpha_i} D_{t,q-1}^{\beta_i} \Big) u_q 
&= (D^{\alpha_k} D_{t,q-1}^{\beta_k}) (D^{\alpha_{k-1}} D_{t,q-1}^{\beta_{k-1}}) P_{k-2} \notag\\
&= (D^{\alpha_k+\alpha_{k-1}} D_{t,q-1}^{\beta_k+\beta_{k-1}}) P_{k-2} + D^{\alpha_k} \left[D_{t,q-1}^{\beta_k}, D^{\alpha_{k-1}}\right] D_{t,q-1}^{\beta_{k-1}}P_{k-2}.
\label{eq:Merlot:*}
\end{align}
The first term in \eqref{eq:Merlot:*} already obeys the correct bound, since we know that \eqref{eq:nasty:D:wq}  holds for $k' = k-1$. In order to treat the second term on the right side of \eqref{eq:Merlot:*}, we use Lemma~\ref{lem:Komatsu} to write the commutator as
\begin{align}
&D^{\alpha_k} \left[D_{t,q-1}^{\beta_k}, D^{\alpha_{k-1}}\right] D_{t,q-1}^{\beta_{k-1}}P_{k-2} \notag\\
&= D^{\alpha_k} \sum_{1 \leq |\gamma| \leq \beta_k} \frac{\beta_k!}{\gamma! (\beta_k - |\gamma|)!} \left(\prod_{\ell=1}^{\alpha_{k-1}} (\ad D_{t,q-1})^{\gamma_\ell}(D) \right) D_{t,q-1}^{\beta_k+\beta_{k-1}-|\gamma|}P_{k-2}.
 \label{eq:product:of:ad:Dt:q-1:comm}
\end{align}
From Lemma~\ref{lem:ad:Dt:a:D} and the Leibniz rule we claim that one may expand
\begin{align}
 \prod_{\ell=1}^{\alpha_{k-1}} (\ad D_{t,q-1})^{\gamma_\ell}(D) = \sum_{j=1}^{\alpha_{k-1}} g_j D^{j}
 \label{eq:product:of:ad:Dt:q-1}
\end{align}
for some explicit functions $g_j$ which obey the estimate
\begin{align}
\norm{ D^a g_j }_{L^\infty(\supp \psi_{i,q})} \lesssim    \tilde \lambda_{q-1}^{ a+\alpha_{k-1}-j} 
\MM{|\gamma|,\Nindt,\Gamma_{q+1}^{i} \Gamma_q^{-\cstar}\tau_{q-1}^{-1}, \Gamma_q^{-1}\tilde \tau_{q-1}^{-1}}
\label{eq:product:of:ad:Dt:q-1:bnd}
\end{align}
 for all $a$ such that $a + \alpha_{k-1}- j +|\gamma|\leq \sfrac{3\Nfin}{2}$.
The claim  \eqref{eq:product:of:ad:Dt:q-1:bnd} requires a proof, which we sketch next.  Using the definition~\eqref{eq:psi:m:im:q:def}, the inductive estimate \eqref{eq:nasty:D:vq:old} at level $q'=q-1$ and with $k=1$, the parameter inequality \eqref{eq:Lambda:q:x:1:NEW} at level $q-1$, for any $0\leq m \leq \NcutSmall$ we have that 
\begin{align*}
&\norm{D^a D_{t,q-1}^b D v_{\ell_{q-1}}}_{L^\infty(\supp \psi_{m,i_m,q})}  
\notag\\
&\les \sum_{\{ j_m \colon \Gamma_{q}^{j_m} \leq \Gamma_{q+1}^{i_m}\}} \norm{D^a D_{t,q-1}^b D v_{\ell_{q-1}}}_{L^\infty(\supp \psi_{j_m,q-1})} \notag\\
&\les \sum_{\{ j_m \colon \Gamma_{q}^{j_m} \leq \Gamma_{q+1}^{i_m}\}} 
(\Gamma_{q}^{j_m+1} \delta_{q-1}^{\sfrac 12}) \tilde \lambda_{q-1}^{a+1} 
\MM{b,\Nindt,\Gamma_{q}^{j_m-\cstar}\tau_{q-1}^{-1}, \Gamma_q^{-1}\tilde \tau_{q-1}^{-1}} \notag\\
&\lesssim (\Gamma_{q+1}^{i_m}\Gamma_q \delta_{q-1}^{\sfrac 12}) \tilde \lambda_{q-1}^{a+1}  
\MM{b,\Nindt,\Gamma_{q+1}^{i_m} \Gamma_q^{-\cstar}\tau_{q-1}^{-1}, \Gamma_q^{-1}\tilde \tau_{q-1}^{-1}} \notag\\
&\lesssim  \tilde \lambda_{q-1}^{a} 
\MM{b+1,\Nindt,\Gamma_{q+1}^{i_m} \Gamma_q^{-\cstar}\tau_{q-1}^{-1}, \Gamma_q^{-1}\tilde \tau_{q-1}^{-1}}
\end{align*}
for all $a + b \leq \sfrac{3\Nfin}{2}$. Thus, from the definition \eqref{eq:psi:i:q:recursive} we deduce that 
\begin{align}
\norm{D^a D_{t,q-1}^b D v_{\ell_{q-1}}}_{L^\infty(\supp \psi_{i,q})}   
\lesssim  \tilde \lambda_{q-1}^{a} \MM{b+1,\Nindt,\Gamma_{q+1}^{i} \Gamma_q^{-\cstar}\tau_{q-1}^{-1}, \Gamma_q^{-1}\tilde \tau_{q-1}^{-1}} 
\label{eq:ad:Dt:a:D:Merlot}
\end{align}
for all  $a + b \leq \sfrac{3 \Nfin}{2}$.  When combined with the formula \eqref{eq:ad:Dt:a:D}, which allows us to write
\begin{align}
(\ad D_{t,q-1})^\gamma(D) = f_{\gamma,q-1} \cdot \nabla
\label{eq:ad:Dt:a:D:Merlot:1}
\end{align}
for an explicit function $f_{\gamma,q-1}$ which is defined in terms of $v_{\ell_{q-1}}$, estimate \eqref{eq:ad:Dt:a:D:Merlot} and the Leibniz rule gives the estimate
\begin{align}
\norm{ D^a f_{\gamma,q-1} }_{L^\infty(\supp \psi_{i,q})} \lesssim  \tilde \lambda_{q-1}^{a} \MM{\gamma,\Nindt,\Gamma_{q+1}^{i} \Gamma_q^{-\cstar}\tau_{q-1}^{-1}, \Gamma_q^{-1}\tilde \tau_{q-1}^{-1}}
\label{eq:ad:Dt:a:D:Merlot:2}
\end{align}
for all   $a + \gamma\leq \sfrac{3\Nfin}{2}$.
In order to conclude the proof of \eqref{eq:product:of:ad:Dt:q-1}--\eqref{eq:product:of:ad:Dt:q-1:bnd}, we use \eqref{eq:ad:Dt:a:D:Merlot:1} to write
\begin{align*}
\prod_{\ell=1}^{\alpha_{k-1}} (\ad D_{t,q-1})^{\gamma_\ell}(D) = \prod_{\ell=1}^{\alpha_{k-1}} \left( f_{\gamma_\ell,q-1} \cdot \nabla \right)= \sum_{j=1}^{\alpha_{k-1}} g_j D^j
\end{align*} 
and now the claimed estimate for $g_j$ follows from the previously established bound \eqref{eq:ad:Dt:a:D:Merlot:2} for the $f_{\gamma_\ell,q-1}$'s and their derivatives, and the Leibniz rule.

With \eqref{eq:product:of:ad:Dt:q-1}--\eqref{eq:product:of:ad:Dt:q-1:bnd} in hand, and using estimate \eqref{eq:nasty:D:wq} with $k' = k-1$, we return to \eqref{eq:product:of:ad:Dt:q-1:comm} and obtain
\begin{align}
&\norm{D^{\alpha_k} \left[D_{t,q-1}^{\beta_k}, D^{\alpha_{k-1}}\right] D_{t,q-1}^{\beta_{k-1}}P_{k-2}}_{L^\infty( \supp \psi_{i,q})} \notag\\
&\lesssim \sum_{j=1}^{\alpha_{k-1}} \sum_{1 \leq |\gamma| \leq \beta_k} \norm{ D^{\alpha_k} \left( g_j \; D^j D_{t,q-1}^{\beta_k+\beta_{k-1}-|\gamma|}P_{k-2} \right) }_{L^\infty( \supp \psi_{i,q})} \notag\\
&\lesssim  \sum_{j=1}^{\alpha_{k-1}} \sum_{1 \leq |\gamma| \leq \beta_k}  \sum_{a'=0}^{\alpha_k} \norm{ D^{\alpha_k-a'} g_j}_{L^\infty(\supp \psi_{i,q})} \norm{D^{a'+j} D_{t,q-1}^{\beta_k+\beta_{k-1}-|\gamma|}P_{k-2} }_{L^\infty( \supp \psi_{i,q})}   \notag\\
&\lesssim \sum_{j=1}^{\alpha_{k-1}} \sum_{|\gamma|=1}^{\beta_k} \sum_{a'=0}^{\alpha_k}  
\lambda_{q}^{\alpha_k + \alpha_{k-1} - j - a' }  
\MM{|\gamma|,\Nindt,\Gamma_{q+1}^{i} \Gamma_q^{-\cstar}\tau_{q-1}^{-1}, \Gamma_q^{-1}\tilde \tau_{q-1}^{-1}}
(\Gamma_{q+1}^{i+1} \delta_{q}^{1/2}) 
\notag\\
&\qquad  \times 
\MM{K-\alpha_k-\alpha_{k-1}+j+a',2 \Nindv,\Gamma_q \lambda_q,\tilde \lambda_q} 
\MM{M-|\gamma|,\Nindt,\Gamma_{q+3}^{i+1} \tau_{q-1}^{-1},\Gamma_{q+1}^{-1} \tilde \tau_{q}^{-1}} \notag\\
&\lesssim   
(\Gamma_{q+1}^{i+1} \delta_{q}^{1/2}) \MM{K ,2 \Nindv,\Gamma_q \lambda_q,\tilde \lambda_q} \MM{M,\Nindt,\Gamma_{q+1}^{i+3}   \tau_{q-1}^{-1},\Gamma_{q+1}^{-1} \tilde \tau_{q}^{-1}} 
\label{eq:nasty:D:wq:***}
\end{align}
for $M\leq \Nindt$ and $K+M\leq \sfrac{3\Nfin}{2}+1$. The $+1$ in the range of derivatives is simply a consequence that the summand in the third line of the above display starts with $j\geq 1$ and with $|\gamma|\geq 1$. This concludes the proof of the inductive step for \eqref{eq:nasty:D:wq}.

\textbf{Proof of \eqref{eq:nasty:Dt:wq}.\,} 
This estimate follows from Lemma~\ref{lem:cooper:1}.  Indeed,  letting $ v = f = u_q$, $B = D_{t,q-1}$, $\Omega = \supp \psi_{i,q}$, $p=\infty$, the previously established bound \eqref{eq:nasty:D:wq} allows us to verify conditions \eqref{eq:cooper:v}--\eqref{eq:cooper:f}  of Lemma~\ref{lem:cooper:1} with  $N_* = \sfrac{3\Nfin}{2}+1$, $\const_v = \const_f = \Gamma_{q+1}^{i+1} \delta_{q}^{\sfrac 12}$, $\lambda_v = \lambda_f =  \Gamma_q \lambda_q, \tilde \lambda_v = \tilde \lambda_f = \tilde \lambda_q, N_x = 2\Nindv, \mu_v =  \mu_f = \Gamma_{q+1}^{i+3}  \tau_{q-1}^{-1}, \tilde \mu_v = \tilde \mu_f = \Gamma_{q+1}^{-1} \tilde \tau_{q}^{-1} , N_t = \Nindt$. As $|\alpha| = K$ and $|\beta|=M$, the bound \eqref{eq:nasty:Dt:wq} now is a direct consequence of \eqref{eq:cooper:f:**}.

\textbf{Proof of \eqref{eq:nasty:D:vq} and \eqref{eq:nasty:no:D:vq}.\,}
First we consider the bound \eqref{eq:nasty:D:vq}, inductively on $k$. For the case $k=1$ the main idea is to appeal to estimate \eqref{eq:cooper:f:*} in Lemma~\ref{lem:cooper:1} with the operators $A = D_q, B = D_{t,q-1}$ and the functions $v = u_q$ and $f = D v_{\ell_q}$, so that $D^n (A+B)^m f  = D^n D_{t,q}^m D v_{\ell_q}$. As before, the assumption \eqref{eq:cooper:v} holds due to \eqref{eq:nasty:D:wq}  with $\Omega= \supp \psi_{i,q}$, $N_* = \sfrac{3\Nfin}{2}+1$, $\const_v = \Gamma_{q+1}^{i+1} \delta_{q}^{\sfrac 12}$, $\lambda_v = \Gamma_q \lambda_q, \tilde \lambda_v = \tilde \lambda_q, N_x = 2\Nindv, \mu_v = \Gamma_{q+1}^{i+3}  \tau_{q-1}^{-1}$, $\tilde \mu_v = \Gamma_{q+1}^{-1} \tilde \tau_{q}^{-1}$, and $N_t = \Nindt$. Verifying condition \eqref{eq:cooper:f} is this time more involved, and follows by rewriting $f = D v_{\ell_q} = D u_q + D v_{\ell_{q-1}}$. By using \eqref{eq:nasty:D:wq}, and the parameter inequality $\Gamma_{q+1}^3 \tau_{q-1}^{-1} \leq \Gamma_{q+1}^{-\cstar} \tau_q^{-1}$ (cf.~\eqref{eq:Tau:q-1:q}), we conveniently obtain
\begin{align}
&\norm{ \Big( \prod_{i=1}^k D^{\alpha_i} D_{t,q-1}^{\beta_i} \Big) D u_q }_{L^\infty(\supp \psi_{i,q})} \notag\\
&\qquad \lesssim  (\Gamma_{q+1}^{i+1} \delta_q^{\sfrac 12} \tilde \lambda_q) \MM{K  ,2\Nindv,\Gamma_q \lambda_q,\tilde \lambda_q} 
\MM{M,\Nindt,\Gamma_{q+1}^{i-\cstar}  \tau_q^{-1},\Gamma_{q+1}^{-1} \tilde \tau_{q}^{-1}}
\label{eq:cutoff:nam:2}
\end{align}
for all $|\alpha| + |\beta| = K+M  \leq \sfrac{3\Nfin}{2}$ (note that the maximal number of derivatives is not $\sfrac{3\Nfin}{2} +1$ anymore, but instead it is just  $\sfrac{3\Nfin}{2}$; the reason is that we are estimating $D u_q$ and not $u_q$). On the other hand, from the inductive assumption \eqref{eq:nasty:D:vq:old} with $q' = q-1$ we obtain that 
\begin{align*}
\norm{ \Big( \prod_{i=1}^k D^{\alpha_i} D_{t,q-1}^{\beta_i} \Big) D v_{\ell_{q-1}} }_{L^\infty(\supp \psi_{j,q-1})} 
\les
(\Gamma_{q}^{j+1} \delta_{q-1}^{\sfrac 12}) (\tilde \lambda_{q-1})^{K+1}  
\MM{M,\Nindt, \Gamma_{q}^{j-\cstar} \tau_{q-1}^{-1}, \tilde \tau_{q-1}^{-1}}
\end{align*}
for $K+M  \leq \sfrac{3 \Nfin}{2}$. Recalling the definitions \eqref{eq:psi:m:im:q:def}--\eqref{eq:psi:i:q:recursive} and the notation \eqref{eq:new:supp:notation}, we have that $(x,t) \in \supp (\psi_{i,q})$ if and only if $(x,t)\in \supp(\psi_{\Vec{i},q})$, and thus for every $m\in \{0,\ldots,\NcutSmall\}$, there exists $j_m$ with $\Gamma_q^{j_m} \leq \Gamma_{q+1}^{i_m} \leq \Gamma_{q+1}^i$ and $(x,t) \in \supp(\psi_{j_m,q-1})$. Thus, the above stated estimate and our usual parameter inequalities imply that 
\begin{align}
\norm{ \Big( \prod_{i=1}^k D^{\alpha_i} D_{t,q-1}^{\beta_i} \Big) D v_{\ell_{q-1}} }_{L^\infty(\supp \psi_{i,q})} 
&\les
(\Gamma_{q+1}^{i+1} \delta_{q-1}^{\sfrac 12}\tilde \lambda_{q-1}) (\tilde \lambda_{q-1})^K \MM{M,\Nindt,\Gamma_{q+1}^{i} \Gamma_q^{-\cstar} \tau_{q-1}^{-1},\tilde \tau_{q-1}^{-1}} \notag\\
&\les
(\Gamma_{q+1}^{i+1} \delta_{q}^{\sfrac 12}\tilde \lambda_{q})  (\Gamma_q \lambda_{q})^K \MM{M,\Nindt, \Gamma_{q+1}^{i-\cstar}  \tau_{q}^{-1},\Gamma_{q+1}^{-1} \tilde \tau_{q}^{-1}}
\label{eq:cutoff:nam:4}
\end{align}
whenever $K+M\leq \sfrac{3\Nfin}{2}$. Here we have used that $\delta_{q-1}^{\sfrac 12} \tilde \lambda_{q-1} \leq \delta_q^{\sfrac 12} \tilde \lambda_q$ and that $\Gamma_{q+1}^{i} \Gamma_q^{-\cstar} \tau_{q-1}^{-1} \leq \Gamma_{q+1}^{i-\cstar} \tau_q^{-1} \leq \Gamma_{q+1}^{-1} \tilde \tau_{q}^{-1}$, for all $i \leq \imax$.  In the last inequality, we have used \eqref{eq:tilde:tau:q:def} and \eqref{eq:imax:bound}.
Combining \eqref{eq:cutoff:nam:2} and \eqref{eq:cutoff:nam:4} we may now verify condition \eqref{eq:cooper:f} for $f = D v_{\ell_q}$, with $p = \infty$, $\Omega = \supp (\psi_{i,q})$, $\const_f = \Gamma_{q+1}^{i+1} \delta_q^{\sfrac 12}\tilde \lambda_q$, $\lambda_f = \Gamma_q \lambda_q, \tilde \lambda_f = \tilde \lambda_q, N_x = 2 \Nindv, \mu_f =   \Gamma_{q+1}^{i-\cstar}  \tau_{q}^{-1}, \tilde \mu_f = \Gamma_{q+1}^{-1} \tilde \tau_{q}^{-1}, N_t =  \Nindt$, and $N_*= \sfrac{3\Nfin}{2}$. We may thus appeal to \eqref{eq:cooper:f:*} and obtain that 
\begin{align*}
&\norm{ D^{K} D_{t,q}^{M}  D v_{\ell_{q}} }_{L^\infty(\supp \psi_{i,q})} \notag\\
&\qquad \les 
(\Gamma_{q+1}^{i+1} \delta_{q}^{\sfrac 12}\tilde \lambda_{q})  \MM{K  ,2\Nindv,\Gamma_q \lambda_q,\tilde \lambda_q} \notag\\
&\qquad \times
\MM{M,\Nindt,\max\{ \Gamma_{q+1}^{i-\cstar}  \tau_{q}^{-1}, \Gamma_{q+1}^{i+1} \delta_q^{\sfrac 12} \tilde \lambda_q \}, \max\{ \Gamma_{q+1}^{-1} \tilde \tau_{q}^{-1}, \Gamma_{q+1}^{i+1} \delta_q^{\sfrac 12} \tilde \lambda_q \}}
\end{align*}
 whenever $K+M\leq \sfrac{3\Nfin}{2}$. The parameter inequalities $\Gamma_{q+1}^{\cstar+1} \delta_q^{\sfrac 12} \tilde \lambda_q \leq   \tau_{q}^{-1}$ from \eqref{eq:Lambda:q:x:1:NEW} and $\Gamma_{q+1}^{i+2} \delta_q^{\sfrac 12} \tilde \lambda_q \leq \tilde \tau_{q}^{-1}$, which follows from \eqref{eq:Lambda:q:t:1} and \eqref{eq:imax:bound}, conclude the proof of  \eqref{eq:nasty:D:vq} for  $k=1$.  

In order to prove \eqref{eq:nasty:D:vq} for a general $k$, we proceed by induction. Assume the estimate holds for every $k' \leq k-1$. Proving \eqref{eq:nasty:D:vq} at level $k$ is done in the same way as we have established the induction step (in $k$) for \eqref{eq:nasty:D:wq}. We let 
\begin{align*}
\tilde P_{k'} = \left( \prod_{i=1}^{k'} D^{\alpha_i} D_{t,q}^{\beta_i} \right) D v_{\ell_q}
\end{align*}
and decompose
\begin{align*}
\left( \prod_{i=1}^k D^{\alpha_i} D_{t,q}^{\beta_i} \right) Dv_{\ell_q} 
&= (D^{\alpha_k+\alpha_{k-1}} D_{t,q}^{\beta_k+\beta_{k-1}}) \tilde P_{k-2}  + D^{\alpha_k} \left[D_{t,q}^{\beta_k}, D^{\alpha_{k-1}}\right] D_{t,q}^{\beta_{k-1}}  \tilde P_{k-2}
\end{align*}
and note that the first term is directly bounded using the induction assumption (at level $k-1$). To bound the commutator term,  similarly to \eqref{eq:product:of:ad:Dt:q-1:comm}--\eqref{eq:product:of:ad:Dt:q-1:bnd}, we obtain from Lemmas~\ref{lem:Komatsu} and~\ref{lem:ad:Dt:a:D} that 
\begin{align*}
D^{\alpha_k} \left[D_{t,q}^{\beta_k}, D^{\alpha_{k-1}}\right] D_{t,q}^{\beta_{k-1}} \tilde P_{k-2}  
= D^{\alpha_k} \sum_{1 \leq |\gamma| \leq \beta_k} \frac{\beta_k!}{\gamma! (\beta_k - |\gamma|)!} \left( \sum_{j=1}^{\alpha_{k-1}} \tilde g_j D^j \right) D_{t,q}^{\beta_k+\beta_{k-1}-|\gamma|} \tilde P_{k-2} \, ,
\end{align*}
where one may use the previously established bound \eqref{eq:nasty:D:vq} with $k=1$ (instead of \eqref{eq:ad:Dt:a:D:Merlot}) to estimate
\begin{align}
\norm{D^a \tilde g_j}_{L^\infty(\supp \psi_{i,q})} 
\les  \MM{a + \alpha_{k-1}-j,2\Nindv,\Gamma_q\lambda_q,\tilde \lambda_q} \MM{|\gamma|,\Nindt,\Gamma_{q+1}^{i-\cstar}  \tau_q^{-1},\Gamma_{q+1}^{-1} \tilde \tau_{q}^{-1}}.
\label{eq:cutoff:nam:3}
\end{align}
Note that the above estimate is not merely \eqref{eq:product:of:ad:Dt:q-1:bnd} with $q$ increased by $1$. Rather, the above estimate is proven in the same way that \eqref{eq:product:of:ad:Dt:q-1:bnd} was proven, by first showing that the analogous version of \eqref{eq:ad:Dt:a:D:Merlot:2} is 
\begin{align*}
\norm{D^a f_{\gamma,q}}_{L^\infty( \supp \psi_{i,q})} 
\les   \MM{a,2\Nindv,\Gamma_q\lambda_q,\tilde \lambda_q} \MM{ \gamma ,\Nindt,\Gamma_{q+1}^{i-\cstar}  \tau_q^{-1},\Gamma_{q+1}^{-1} \tilde \tau_{q}^{-1}}\,,
\end{align*}
from which the claimed estimate \eqref{eq:cutoff:nam:3} on $D^a \tilde g_j$ follows.
 The estimate 
\begin{align}
&\norm{D^{\alpha_k} \left[D_{t,q}^{\beta_k}, D^{\alpha_{k-1}}\right] D_{t,q}^{\beta_{k-1}} \tilde P_{k-2}  }_{L^\infty(\supp \psi_{i,q})}  \notag\\
&\quad \lesssim  (\Gamma_{q+1}^{i+1} \delta_{q}^{\sfrac 12}) \MM{K+1,2\Nindv, \Gamma_{q} \lambda_q,\tilde \lambda_q} \MM{M,\Nindt,\Gamma_{q+1}^{i-\cstar}  \tau_q^{-1},\Gamma_{q+1}^{-1} \tilde \tau_{q}^{-1}}
\label{eq:cutoff:nam:5}
\end{align}
follows similarly to \eqref{eq:nasty:D:wq:***}, from the estimate \eqref{eq:cutoff:nam:3}  for $\tilde g_j$, and the bound \eqref{eq:nasty:D:vq} with $k-1$ terms in the product. 
This concludes the proof of  estimate \eqref{eq:nasty:D:vq}.

To conclude the proof of the Lemma, we also need to establish the estimates for $v_{\ell_q}$  claimed in \eqref{eq:nasty:no:D:vq}. The proof of this bound is nearly identical to that of \eqref{eq:nasty:D:vq}, as is readily seen for $k=1$:  we just need to replace $Du_q$ estimates with $u_q$ estimates, and $D v_{\ell_{q-1}}$ bounds with $v_{\ell_{q-1}}$ bounds. For instance, instead of \eqref{eq:cutoff:nam:2}, we appeal to \eqref{eq:nasty:Dt:wq:WEAK} and obtain a bound for $D^K D_{t,q}^M u_q$ which is better than \eqref{eq:cutoff:nam:2} by a factor of $\tilde \lambda_q$, and which holds for $K+M \leq \sfrac{3\Nfin}{2}+1$. This estimate is sharper than required by \eqref{eq:nasty:no:D:vq}. The estimate for $D^K D_{t,q}^M v_{\ell_{q-1}}$ is obtained similarly to \eqref{eq:cutoff:nam:4}, except that instead of appealing to the induction assumption \eqref{eq:nasty:D:vq:old} at level $q'=q-1$, we use \eqref{eq:bob:Dq':old} with $q'=q-1$. The Sobolev loss $\lambda_{q-1}^2$ is then apparent from~\eqref{eq:bob:Dq':old}, and the estimates hold for $K+M \leq \sfrac{3\Nfin}{2}+1$. These arguments establish \eqref{eq:nasty:no:D:vq} with $k=1$. The case of general $k\geq 2$ is treated inductively exactly as before, because the commutator term is bounded in the same way as \eqref{eq:cutoff:nam:5}, except that $K+1$ is replaced by $K$. To avoid redundancy, we omit these details.
\end{proof}

\subsubsection{Material derivatives}
The estimates in the previous sections, which have led up to Lemma~\ref{lem:Dt:Dt:wq:psi:i:q:multi}, allow us to estimate mixed space, directional, and  material derivatives of the velocity cutoff functions $\psi_{i,q}$, which in turn allow us to establish the inductive bounds \eqref{eq:sharp:Dt:psi:i:q:old} and \eqref{eq:sharp:Dt:psi:i:q:mixed:old} with $q' = q$.

In order to achieve this we crucially recall Remark~\ref{rem:rewrite:cutoffs}. 
Note that if we were to directly differentiate \eqref{eq:psi:i:q:recursive}, then we would need to consider all vectors $\Vec{i}\in \N_0^{\NcutSmall+1}$ such that $\max_{0\leq m\leq\NcutSmall} i_m =i$, and then for each one of these $\Vec{i}$ consider the term ${\mathbf{1}}_{\supp(\psi_{\Vec{i},q}) }  D_{t,q-1} (\psi_{m,i_{m},q}^2)$ for each $0\leq m\leq \NcutSmall$; however in this situation we encounter for instance a term with $i_0 = 0$  and $i_{m'}=i$ for all $1\leq m' \leq \NcutSmall$; the bounds available on this term would be catastrophic due to the mismatch  $i_0 < i_{m'}$ for all $m'>0$. Identity \eqref{eq:cutoff:resummation} precisely permits us to avoid this situation, because it has essentially ordered the indices $\{i_m\}_{m=0}^{\NcutSmall}$ to be non-increasing in $m$. Indeed inspecting   \eqref{eq:cutoff:resummation} and using identity \eqref{eq:fancy:cutoff:supp} and the definitions \eqref{eq:new:supp:notation}, \eqref{eq:fancy:cutoff}, we see that 
\begin{align}
(x,t) \in \supp (D_{t,q-1}  \psi_{i,q}^2) \quad \Leftrightarrow \quad 
&\exists \Vec{i}\in \N_0^{\NcutSmall+1} \mbox{ and } \exists 0\leq m \leq\NcutSmall \notag\\
&\mbox{with }i_m \in \{i-1,i\} \mbox{ and }  \max_{0\leq m'\leq \NcutSmall} i_{m'} = i \notag\\
&\mbox{such that } (x,t) \in \supp(\psi_{\Vec{i},q}) \cap \supp(D_{t,q-1} \psi_{m,i_m,q}) \notag\\
&  \mbox{and }  i_{m'} \leq i_{m} \mbox{ whenever }  m < m' \leq \NcutSmall\,.
\label{eq:orangutan}
\end{align}
The generalization of characterization \eqref{eq:orangutan} to higher order material derivatives $D_{t,q-1}^M$ is direct:  $(x,t) \in \supp (D_{t,q-1}^M  \psi_{i,q}^2)$ if and only if there exists $ \Vec{i}\in \N_0^{\NcutSmall+1}$ with maximal index equal to $i$,  such that for every $0\leq m \leq \NcutSmall$ for which $(x,t) \in \supp(\psi_{\Vec{i},q}) \cap \supp(D_{t,q-1} \psi_{m,i_m,q}) $ (there are potentially more than one such $m$ if $M\geq 2$ due to the Leibniz rule), we have $i_{m'} \leq i_m \in \{i-1,i\}$ whenever $m< m'$. In light of this characterization, we have the following bounds: 
\begin{lemma}
\label{lem:sharp:Dt:psi:i:j:q}
Let $q\geq 1$, $0\leq i \leq \imax(q)$, and fix $\Vec{i}\in \N_0^{\NcutSmall+1}$ such that $\max_{0\leq m\leq\NcutSmall} i_m =i$, as in the right side of \eqref{eq:orangutan}.
Fix $0 \leq m \leq \NcutSmall$ such that $ i_m \in \{i-1,i\}$ and such that $i_{m'}\leq i_m$ for all $m \leq m'\leq \NcutSmall$. Lastly, fix $j_m$   such that $i_*(j_m) \leq i_m$. For  $N,K,M,k \geq 0$, $\alpha,\beta \in {\mathbb N}^k$ such that $|\alpha| = K$ and $|\beta| = M $, we have
\begin{align}
& \frac{ {\mathbf{1}}_{\supp(\psi_{\Vec{i},q}) } {\mathbf{1}}_{\supp(\psi_{j_m,q-1})}}{\psi_{m,i_m,j_m,q}^{1- (K+M)/\Nfin}} \left| \left( \prod_{l=1}^k D^{\alpha_l} D_{t,q-1}^{\beta_l}\right)  \psi_{m,i_m,j_m,q} \right|  \notag\\
&\quad  \les
\MM{K,\Nindv,  \Gamma_{q}  \lambda_q, \tilde \lambda_q  \Gamma_{q} }
\MM{M,\Nindt-\NcutLarge, \Gamma_{q+1}^{i+3}  \tau_{q-1}^{-1}, \Gamma_{q+1}^{-1} \tilde \tau_{q}^{-1} }
\label{eq:sharp:Dt:psi:i:j:q}
\end{align}
for all $K$ such that $0 \leq K+M \leq \Nfin$.
Moreover,
\begin{align}
&  \frac{ {\mathbf{1}}_{\supp(\psi_{\Vec{i},q}) } {\mathbf{1}}_{\supp(\psi_{j_m,q-1})}}{\psi_{m,i_m,j_m,q}^{1- (N+K+M)/\Nfin}} \left| D^N \left( \prod_{l=1}^k D_q^{\alpha_l} D_{t,q-1}^{\beta_l}\right)  \psi_{m,i_m,j_m,q} \right|  \notag\\
&\quad  \lesssim 
\MM{N,\Nindv, \Gamma_{q}  \lambda_q, \tilde \lambda_q  \Gamma_{q} } 
(\Gamma_{q+1}^{i-\cstar}  \tau_{q}^{-1})^{K}
\MM{M,\Nindt-\NcutLarge, \Gamma_{q+1}^3  \tau_{q-1}^{-1},  \Gamma_{q+1}^{-1} \tilde \tau_{q}^{-1}} 
\label{eq:sharp:Dt:psi:i:j:q:mixed}
\end{align}
holds whenever $0 \leq N+ K+M \leq \Nfin$.
\end{lemma}
\begin{proof}[Proof of Lemma~\ref{lem:sharp:Dt:psi:i:j:q}]
Note that for $M=0$ estimate \eqref{eq:sharp:Dt:psi:i:j:q} was already established in \eqref{eq:sharp:D:psi:i:j:q}. The bound \eqref{eq:sharp:Dt:psi:i:j:q:mixed} with $M=0$, i.e., an estimate for the $D^N D_q^K \psi_{m,i_m,j_m,q}$, holds by appealing to the expansion \eqref{eq:D:q:K:i}--\eqref{eq:D:q:K:ii},  the bound \eqref{eq:D:a:fj} (which is applicable since in the context of estimate  \eqref{eq:sharp:Dt:psi:i:j:q:mixed} we work on the support of $\psi_{i,q}$),   to the bound \eqref{eq:sharp:Dt:psi:i:j:q} with $M=0$, and to the parameter inequality $\Gamma_{q+1}^{2+\cstar} \delta_q^{\sfrac 12} \tilde \lambda_q \leq \tau_{q}^{-1}$ (which follows from \eqref{eq:Lambda:q:x:1:NEW}).
The rest of the proof is dedicated to the case $M \geq 1$. The proofs are very similar to the proof of Lemma~\ref{lem:sharp:D:psi:i:q}, but we additionally need to appeal to bounds and arguments from the proof of Lemma~\ref{lem:Dt:Dt:wq:psi:i:q:multi}.  

\textbf{Proof of \eqref{eq:sharp:Dt:psi:i:j:q}.\,}
As in the proof of Lemma~\ref{lem:sharp:D:psi:i:q}, we start with the case $k=1$, and estimate $D^K D_{t,q-1}^M \psi_{m,i_m,j_m,q}$ for $K+M\leq \Nfin$, with $M\geq 1$. We note that just as $D$, the operator $D_{t,q-1}$ is a scalar differential operator, and thus the Fa\'a di Bruno argument which was used to  bound \eqref{eq:sharp:D:psi:i:j:q} may be repeated. As was done there, we recall the definitions \eqref{eq:psi:i:j:def}--\eqref{eq:psi:i:i:def} and split the analysis in two cases, according to whether \eqref{eq:clusterf:1} or \eqref{eq:clusterf:3} holds. 

Let us first consider the case \eqref{eq:clusterf:1}. Our goal is to apply  Lemma~\ref{lem:Faa:di:Bruno:*} to the function $\psi = \psi_{m,q+1}$ or $\psi=\tilde \psi_{m,q+1}$, with $\Gamma_\psi = \Gamma_{q+1}^{m+1}$, $\Gamma = \Gamma_{q+1}^{(m+1)(i_m-i_*(j_m))}$,   $h(x,t) =  h_{m,j_m,q}^2(x,t)$, and $D_t = D_{t,q-1}$. Estimate \eqref{eq:Faa:di:Bruno:lem:1:*} holds by \eqref{eq:DN:psi:q:0} and \eqref{eq:DN:psi:q:gain}, so that it remains to obtain a bound on the material derivatives of $(h_{m,j_m,q}(x,t))^2$ and establish a bound which corresponds to \eqref{eq:Faa:di:Bruno:lem:2:*} on the set $\supp(\psi_{\Vec i,q}) \cap \supp(\psi_{j_m,q-1} \psi_{m,i_m,j_m,q})$. Similarly to \eqref{eq:cutoff:spatial:derivatives:000}, for $K'+M' \leq \Nfin$ the Leibniz rule and definition \eqref{eq:h:j:q:def} gives
\begin{align}
\left|D^{K'} D_{t,q-1}^{M'} h_{m,j_m,q}^2\right|
&\lesssim (\lambda_q \Gamma_q)^{K'} 
(\tau_{q-1}^{-1} \Gamma_{q+1}^2 )^{M'} \Gamma_{q+1}^{-2(m+1) i_*(j_m) }   \notag\\
&\times
\sum_{K''=0}^{K'} \sum_{M'' =0}^{M'} \sum_{n=0}^{\NcutLarge}  
( \tau_{q-1}^{-1} \Gamma_{q+1}^2)^{-m-M''} 
(\lambda_q \Gamma_q)^{-n-K''}  \delta_{q}^{-\sfrac 12} 
| D^{n + K''} D^{m+M''}_{t,q-1} u_{q}| 
\notag\\
&\qquad      \times 
( \tau_{q-1}^{-1} \Gamma_{q+1}^2)^{-m-M'+M''} 
(\lambda_q \Gamma_q)^{-n - K' + K''} \delta_{q}^{-\sfrac 12} 
| D^{n + K'-K''} D^{m+M'-M''}_{t,q-1} u_{q}| 
\,.
\label{eq:5:123}
\end{align}
By the characterization \eqref{eq:orangutan}, for every $(x,t)$ in the support described on the left side of \eqref{eq:sharp:Dt:psi:i:j:q}  we have that for every $m \leq R \leq \NcutSmall$, there exists $i_R \leq i_m$ and $j_R$ with $i_*(j_R) \leq i_R$, such that $(x,t) \in \supp  \psi_{j_R,q-1} \psi_{R,i_R,j_R,q}$. As a consequence, for the terms in the sum \eqref{eq:5:123} with $L \in \{n+K'',n+K'-K''\} \leq \NcutLarge$ 
and $R \in \{m+M'',m+M'-M''\} \leq \NcutSmall$, 
we may appeal to estimate \eqref{eq:h:psi:supp} which gives a bound on $h_{R,j_R,q}$, and thus obtain
\begin{align*}
(\tau_{q-1}^{-1} \Gamma_{q+1}^2)^{-R}  (\lambda_q \Gamma_q)^{-L} \delta_{q}^{-\sfrac 12} \norm{D^L D_{t,q-1}^R u_q}_{L^\infty( \supp \psi_{R,i_R,j_R,q} )} 
&\leq \Gamma_{q+1}^{(R+1)i_*(j_R)}  \Gamma_{q+1}^{(R+1) (i_R + 1 - i_*(j_R))} 
\notag\\
&\leq  \Gamma_{q+1}^{(R+1) (i_m + 1)}  \, .
\end{align*}
On the other hand, if $L > \NcutLarge$, or if $R > \NcutSmall$, then by \eqref{eq:mollified:velocity:sup} and \eqref{eq:D:K:psi:i:q}
we have  that 
\begin{align}
&(\tau_{q-1}^{-1} \Gamma_{q+1}^2)^{-R}  (\lambda_q \Gamma_q)^{-L} \delta_{q}^{-\sfrac 12} \norm{D^L D_{t,q-1}^R u_q}_{L^\infty(\supp \psi_{j_m,q-1})} 
\notag\\
&  \leq  \tilde \lambda_q^{\sfrac 32}  \Gamma_q^{-L} \Gamma_{q+1}^{-2R} \MM{L,2\Nindv,1,\lambda_q^{-1} \tilde \lambda_q}
\MM{R,\NindSmall,\Gamma_{q}^{j_m+1},\tau_{q-1} \tilde \tau_{q-1}^{-1}}
\notag\\
&  \leq   \MM{L,2\Nindv,1,\lambda_q^{-1} \tilde \lambda_q}
\MM{R,\NindSmall,\Gamma_{q+1}^{i_m+1},\tau_{q-1}  \tilde \tau_{q-1}^{-1}}
\, .
\label{eq:cutoff:spatial:derivatives:0000}
\end{align}
since $\NcutLarge$ and $\NcutSmall$  were taken sufficiently large to obey \eqref{eq:Nind:cond:3}. Combining \eqref{eq:5:123}--\eqref{eq:cutoff:spatial:derivatives:0000}, we may derive that  
\begin{align}
&{\mathbf{1}}_{\supp(\psi_{\Vec{i},q}) } {\mathbf{1}}_{\supp(\psi_{j_m,q-1})} \left|D^{K'} D_{t,q-1}^{M'} h_{m,j_m,q}^2\right|
\notag\\
&\lesssim \Gamma_{q+1}^{2(m+1)(i_m - i_*(j_m)+1)}   (\lambda_q \Gamma_q)^{K'} ( \tau_{q-1}^{-1} \Gamma_{q+1}^2)^{M'}  \MM{2\NcutLarge+K', 2\Nindv,1,\lambda_q^{-1} \tilde \lambda_q} \Gamma_{q+1}^{-2 m (i_m+1)}  \notag\\
& \quad \times
 \sum_{M''=0}^{M'} 
 \MM{m+M'',\NindSmall,\Gamma_{q+1}^{i_{m}+1},\tau_{q-1}  \tilde \tau_{q-1}^{-1}} 
\MM{m+M'-M'',\NindSmall,\Gamma_{q+1}^{i_{m}+1},\tau_{q-1}  \tilde \tau_{q-1}^{-1}} 
\notag\\
&\lesssim \Gamma_{q+1}^{2(m+1)(i_m - i_*(j_m)+1)}   (\lambda_q \Gamma_q)^{K'} ( \tau_{q-1}^{-1} \Gamma_{q+1}^{i_m+3})^{M'}  \MM{K',  \Nindv,1,\lambda_q^{-1} \tilde \lambda_q}  \notag\\
&\qquad \times
\MM{ M',\NindSmall-\NcutSmall,1,\tau_{q-1} \Gamma_{q+1}^{-(i_{m}+1)} \tilde \tau_{q-1}^{-1}} \notag\\
&\lesssim \Gamma_{q+1}^{2(m+1)(i_m - i_*(j_m)+1)}    \MM{K',  \Nindv, \Gamma_q \lambda_q ,\Gamma_q \tilde \lambda_q} 
\MM{ M',\NindSmall-\NcutSmall,\tau_{q-1}^{-1} \Gamma_{q+1}^{i+3},  \Gamma_{q+1}^{2} \tilde \tau_{q-1}^{-1}} \notag\\
&\lesssim \Gamma_{q+1}^{2(m+1)(i_m - i_*(j_m)+1)}    \MM{K',  \Nindv, \Gamma_q \lambda_q ,\Gamma_q \tilde \lambda_q} \MM{ M',\NindSmall-\NcutSmall,\tau_{q-1}^{-1} \Gamma_{q+1}^{i+3}, \Gamma_{q+1}^{-1} \tilde \tau_{q}^{-1}}
\label{eq:5:126}
\end{align}
for all $K'+M'\leq \Nfin$. Here we have used that $\Nindv \geq 2 \Nindt$, that $m\leq \NcutSmall$, and that $i_m\leq i$. The upshot of \eqref{eq:5:126} is that condition \eqref{eq:Faa:di:Bruno:lem:2:*} in Lemma~\ref{lem:Faa:di:Bruno:*} is now verified, with  $\const_h = \Gamma_{q+1}^{2(m+1)(i_m - i_*(j_m)+1)}$, and $\lambda = \Gamma_q \lambda_q$, $\tilde \lambda = \Gamma_q \tilde \lambda_q$,  $\mu =\tau_{q-1}^{-1} \Gamma_{q+1}^{i_m+3}$, $\tilde  \mu = \Gamma_{q+1}^2 \tilde \tau_{q-1}^{-1}$, $N_x = \Nindv$, and $N_t = \Nindt-\NcutSmall$. We obtain from \eqref{eq:Faa:di:Bruno:lem:3:*} and the fact that $(\Gamma_\psi \Gamma)^{-2} \const_h = 1$ that \eqref{eq:sharp:Dt:psi:i:j:q} holds when $k=1$  for those $(x,t)$ such that $h_{m,j_m,q}(x,t)$ satisfies \eqref{eq:clusterf:1}. The case when $h_{m,j_m,q}(x,t)$ satisfies the bound  \eqref{eq:clusterf:3} is nearly identical, as  was the case in the proof of Lemma~\ref{lem:sharp:D:psi:i:q}. The only changes are that now $\Gamma_\psi = 1$ (according to \eqref{eq:DN:psi:q}), and that the constant $\const_h$ which we read from the right side of \eqref{eq:5:126} is now improved to $\Gamma_{q+1}^{2(m+1)(i_m - i_*(j_m))}$. These two changes offset each other, resulting in the same exact bound. Thus, we have shown that \eqref{eq:sharp:Dt:psi:i:j:q} holds when $k=1$. 

The general case $k\geq 1$ in \eqref{eq:sharp:Dt:psi:i:j:q} is obtained via induction on $k$, in precisely the same fashion as the proof of estimate \eqref{eq:nasty:D:wq} in Lemma~\ref{lem:Dt:Dt:wq:psi:i:q:multi}. At the heart of the matter lies a commutator bound similar to \eqref{eq:nasty:D:wq:***}, which  is proven  in precisely the same way by appealing to the fact that we work on $\supp (\psi_{\Vec{i},q}) \subset \supp (\psi_{i,q})$, and thus bound \eqref{eq:product:of:ad:Dt:q-1:bnd} is available; in turn, this bound provides sharper space and material estimates than required in \eqref{eq:sharp:Dt:psi:i:j:q}, completing the proof. In order to avoid redundancy we omit further details.

\textbf{Proof of \eqref{eq:sharp:Dt:psi:i:j:q:mixed}.\,}
This estimate follows from Lemma~\ref{lem:cooper:1} with $ v = u_q$, $B = D_{t,q-1}$, $f= \psi_{m,i_m,j_m,q}$, $\Omega = \supp(\psi_{\Vec{i},q})  \cap \supp(\psi_{j_m,q-1}) \cap \supp(\psi_{m,i_m,j_m,q})$, and $p=\infty$. Technically, the presence of the $\psi_{m,i_m,j_m,q}^{- 1 + (N+K+M)/\Nfin }$ factor on the left side of \eqref{eq:sharp:Dt:psi:i:j:q:mixed} means that the bound doesn't follow from the statement of Lemma~\ref{lem:cooper:1}, but instead, it follows from its proof; the changes to the argument are minor and we ignore this distinction. First, we note that since $\Omega \subset \supp(\psi_{i,q})$, estimate \eqref{eq:nasty:D:wq} allows us to verify condition \eqref{eq:cooper:v} of Lemma~\ref{lem:cooper:1} with  $N_* = \sfrac{3\Nfin}{2}+1$, $\const_v = \Gamma_{q+1}^{i+1} \delta_{q}^{\sfrac 12}$, $\lambda_v = \Gamma_q \lambda_q, \tilde \lambda_v = \tilde \lambda_q, N_x = 2\Nindv \geq \Nindv , \mu_v =  \Gamma_{q+1}^{i+3}  \tau_{q-1}^{-1}, \tilde \mu_v = \Gamma_{q+1}^{-1} \tilde \tau_{q}^{-1} , N_t = \Nindt\geq \Nindt-\NcutSmall $. On the other hand, condition \eqref{eq:cooper:f} of Lemma~\ref{lem:cooper:1} holds in view of \eqref{eq:sharp:Dt:psi:i:j:q} with $\const_f = 1$, $\lambda_f = \Gamma_q \lambda_q, \tilde \lambda_f = \Gamma_q \tilde \lambda_q, N_x =  \Nindv, \mu_f =  \Gamma_{q+1}^{i+3}   \tau_{q-1}^{-1},\tilde \mu_f = \Gamma_{q+1}^{-1} \tilde \tau_{q}^{-1}, N_t = \Nindt-\NcutSmall$. As $|\alpha| = K$ and $|\beta|=M$, the bound \eqref{eq:sharp:Dt:psi:i:j:q:mixed}   is now a direct consequence of \eqref{eq:cooper:f:**} and the parameter inequality $\Gamma_{q+1}^{i+1} \delta_q^{\sfrac 12} \Gamma_q \tilde \lambda_q \leq \Gamma_{q+1}^{i-\cstar} \tau_{q}^{-1} \Leftarrow \Gamma_{q+1}^{\cstar+2} \delta_q^{\sfrac 12}  \tilde \lambda_q \leq   \tau_{q}^{-1}$, cf.~\eqref{eq:Lambda:q:x:1:NEW}.
\end{proof}

A direct consequence of Lemma~\ref{lem:sharp:Dt:psi:i:j:q} and identity \eqref{eq:orangutan} is that the inductive bounds \eqref{eq:sharp:Dt:psi:i:q:old} and \eqref{eq:sharp:Dt:psi:i:q:mixed:old} hold for $q' = q$, as is shown by the following Lemma. 

\begin{lemma}[Mixed spatial and material derivatives for velocity cutoffs]
\label{lem:sharp:Dt:psi:i:q}
Let $q\geq 1$, $0 \leq i \leq i_{\mathrm{max}}(q)$, $N,K,M,k \geq 0$, and let $\alpha,\beta \in {\mathbb N}^k$ be such that $|\alpha|=K$ and $|\beta|=M$.  Then we have
\begin{align}
&\frac{1}{\psi_{i,q}^{1- (K+M)/\Nfin}} \left|\left(\prod_{l=1}^k D^{\alpha_l} D_{t,q-1}^{\beta_l}\right) \psi_{i,q}\right| 
\notag\\
&\qquad \les \MM{K,\Nindv,\Gamma_{q} \lambda_q, \Gamma_{q} \tilde \lambda_q }
\MM{M,\Nindt-\NcutSmall,\Gamma_{q+1}^{i+3}  \tau_{q-1}^{-1}, \Gamma_{q+1}^{-1} \tilde \tau_q^{-1}}
\label{eq:sharp:Dt:psi:i:q}
\end{align}
for $K + M \leq \Nfin$, and 
\begin{align}
&\frac{1}{\psi_{i,q}^{1- (N+K+M)/\Nfin}} \left| D^N \left( \prod_{l=1}^k D_q^{\alpha_l} D_{t,q-1}^{\beta_l}\right)  \psi_{i,q} \right| \notag\\
&\qquad \les  \MM{N,\Nindv,\Gamma_{q} \lambda_q, \Gamma_{q} \tilde \lambda_q  }
(\Gamma_{q+1}^{i-\cstar} \tau_q^{-1})^K 
\MM{M,\Nindt-\NcutSmall,\Gamma_{q+1}^{i+3}  \tau_{q-1}^{-1}, \Gamma_{q+1}^{-1} \tilde \tau_q^{-1}}
\label{eq:sharp:Dt:psi:i:q:mixed}
\end{align}
holds for $N+ K+ M \leq \Nfin$.
\end{lemma}

\begin{remark}
As shown in Remark~\ref{rem:D:t:q':orangutan}, the bound \eqref{eq:sharp:Dt:psi:i:q:mixed} and identity~\eqref{eq:cooper:1} imply that estimate \eqref{eq:nasty:Dt:psi:i:q:orangutan} also holds with $q'=q$, namely that
\begin{align}
\frac{1}{\psi_{i,q}^{1- (N+M)/\Nfin}} \left| D^N  D_{t,q}^{M}  \psi_{i,q} \right|  \les  \MM{N,\NindLarge,  \Gamma_{q}  \lambda_{q},  \Gamma_{q} \tilde \lambda_{q}  } 
\MM{M,\Nindt-\NcutSmall, \Gamma_{q+1}^{i-\cstar} \tau_{q}^{-1}, \Gamma_{q+1}^{-1}  \tilde \tau_{q}^{-1}} \label{eq:dtq:psi:i:q:remark}
\end{align}
for $N+M\leq \Nfin$. Note that for all $M \geq 0$ we have 
\begin{align*}
&\MM{M,\Nindt-\NcutSmall, \Gamma_{q+1}^{i-\cstar} \tau_{q}^{-1}, \Gamma_{q+1}^{-1}  \tilde \tau_{q}^{-1}} \notag\\
&\leq \Gamma_{q+1}^{- (\Nindt-\NcutSmall)} \left(\tau_q \Gamma_{q+1}^{-1} \tilde \tau_q^{-1}\right)^{\Ncut} \MM{M,\Nindt , \Gamma_{q+1}^{i - \cstar +1} \tau_{q}^{-1}, \Gamma_{q+1}^{-1}  \tilde \tau_{q}^{-1}}\notag\\
&\leq \MM{M,\Nindt , \Gamma_{q+1}^{i- \cstar +1} \tau_{q}^{-1}, \Gamma_{q+1}^{-1}  \tilde \tau_{q}^{-1}}
\end{align*}
once $\Nindt$ is taken to be sufficiently large when compared to $\NcutSmall$ to ensure that 
\begin{align*}
 \left(\tau_q  \tilde \tau_q^{-1}\right)^{\Ncut}
\leq \Gamma_{q+1}^{ \Nindt }
\end{align*}
for all $q\geq 1$. This condition holds in view of \eqref{eq:Nind:cond:2}.  In summary, we have thus obtained
\begin{align}
&\frac{1}{\psi_{i,q}^{1- (N+M)/\Nfin}} \left| D^N  D_{t,q}^{M}  \psi_{i,q} \right| \notag\\
&  \les  \MM{N,\NindLarge,  \Gamma_{q}  \lambda_{q},  \Gamma_{q} \tilde \lambda_{q}  } 
\MM{M,\Nindt , \Gamma_{q+1}^{i-\cstar+1} \tau_{q}^{-1}, \Gamma_{q+1}^{-1}  \tilde \tau_{q}^{-1}}
\label{eq:nasty:Dt:psi:i:q:orangutan:redux}
\end{align}
for $N+M\leq \Nfin$.
\end{remark}

\begin{proof}[Proof of Lemma~\ref{lem:sharp:Dt:psi:i:q}]
Note that for $M = 0$ estimate \eqref{eq:sharp:Dt:psi:i:q} holds by \eqref{eq:sharp:D:psi:i:q}. The bound \eqref{eq:sharp:Dt:psi:i:q:mixed} holds for $M=0$, due to the expansion \eqref{eq:D:q:K:i}--\eqref{eq:D:q:K:ii},  the bound \eqref{eq:D:a:fj}  on the support of $\psi_{i,q}$,   to the bound \eqref{eq:sharp:Dt:psi:i:q:mixed} with $M=0$, and to the parameter inequality $\Gamma_{q+1}^{2+\cstar} \delta_q^{\sfrac 12} \tilde \lambda_q \leq \tau_{q}^{-1}$ (cf.~\eqref{eq:Lambda:q:x:1:NEW}). 
The rest of the proof is dedicated to the case $M \geq 1$. 

The argument is very similar to the proof of Lemma~\ref{lem:sharp:D:psi:i:q} and so we only emphasize the main differences. We start with the proof of \eqref{eq:sharp:Dt:psi:i:q}. We claim that in a the same way that \eqref{eq:sharp:D:psi:i:j:q} was shown to imply \eqref{eq:sharp:D:psi:im:q}, one may show that estimate \eqref{eq:sharp:Dt:psi:i:j:q} implies that for any $\Vec{i}$ and $0\leq m \leq \NcutSmall$ as on the right side of \eqref{eq:orangutan} (in particular, as in Lemma~\ref{lem:Dt:Dt:wq:psi:i:q:multi}), we have that 
\begin{align}
& \frac{ {\mathbf{1}}_{\supp(\psi_{\Vec{i},q})}}{\psi_{m,i_m,q}^{1- (K+M)/\Nfin}} \left| \left( \prod_{l=1}^k D^{\alpha_l} D_{t,q-1}^{\beta_l}\right)  \psi_{m,i_m,q} \right|  \notag\\
&\quad  \les
\MM{K,\Nindv,  \Gamma_{q}  \lambda_q, \tilde \lambda_q  \Gamma_{q} }
\MM{M,\Nindt-\NcutLarge, \Gamma_{q+1}^{i+3}  \tau_{q-1}^{-1}, \Gamma_{q+1}^{-1} \tilde \tau_{q}^{-1} }
\,.
\label{eq:5:129}
\end{align}
The proof of the above estimate is done by induction on $k$. For $k=1$, the first step in establishing \eqref{eq:5:129} is to use the Leibniz rule and induction on the number of material derivatives to reduce the problem to an estimate for $\psi_{m,i_m,q}^{-2 + (K+M)/\Nfin} D^K D_{t,q-1}^M (\psi_{m,i_m,q}^2)$; this is achieved in precisely the same way that \eqref{eq:cutoff:spatial:derivatives:1} was proven. The derivatives of $\psi_{m,i_m,q}^2$ are now bounded via the Leibniz rule and the definition \eqref{eq:psi:m:im:q:def}. Indeed, when $D^{K'} D_{t,q-1}^{M'}$ derivatives fall on $\psi_{m,i_m,j_m,q}^2$ the required bound is obtained from \eqref{eq:sharp:Dt:psi:i:j:q}, which gives the same upper bound as the one required by \eqref{eq:5:129}. On the other hand, if $D^{K-K'} D_{t,q-1}^{M-M'}$ derivatives fall on $\psi_{j_m,q-1}^2$, the required estimate is provided by \eqref{eq:nasty:Dt:psi:i:q:orangutan} with $q' = q-1$ and $i$ replaced by $j_m$; the resulting estimates are strictly better than what is required by \eqref{eq:5:129}. This shows that estimate \eqref{eq:5:129} holds for $k=1$. We then proceed inductively in $k\geq 1$, in  the same fashion as the proof of estimate \eqref{eq:nasty:D:wq} in Lemma~\ref{lem:Dt:Dt:wq:psi:i:q:multi}; the corresponding commutator bound is applicable because we work on $\supp(\psi_{m,i_m,q}) \cap \supp (\psi_{i,q})$.  In order to avoid redundancy we omit these details, and conclude the proof of \eqref{eq:5:129}. 

As in the proof of Lemma~\ref{lem:sharp:D:psi:i:q}, we are now able to show that \eqref{eq:sharp:Dt:psi:i:q} is a consequence of \eqref{eq:5:129}. As before, by induction on the number of material derivatives and the Leibniz rule we reduce the problem to an estimate for $\psi_{i,q}^{-2 + (K+M)/\Nfin} \prod_{l=1}^k D^{\alpha_l} D_{t,q-1}^{\beta_l} (\psi_{i,q}^2)$; see the proof of \eqref{eq:cutoff:spatial:derivatives:1} for details. In order to estimate derivatives of $\psi_{i,q}^2$, we use identities~\eqref{eq:fancy:cutoff:supp} and~\eqref{eq:cutoff:resummation}, which 
imply upon applying a differential operator, say $D_{t,q-1}$, that 
\begin{align}
D_{t,q-1} (\psi_{i,q}^2) 
&= D_{t,q-1} \left( \sum_{m=0}^{\NcutSmall}  \prod_{m'=0}^{m-1} \Psi_{m',i,q}^2 \cdot \psi_{m,i,q}^2 \cdot \prod_{m''=m+1}^{\NcutSmall} \Psi_{m'',i-1,q}^2  \right) \notag\\
&=  \sum_{m=0}^{\NcutSmall} \sum_{\bar m'=0}^{m-1} D_{t,q-1}(\psi_{\bar m',i,q}^2)  \prod_{\substack{0\leq m' \leq m-1 \\ m' \neq \bar m'}}  \Psi_{m',i,q}^2 \cdot \psi_{m,i,q}^2\cdot \prod_{m''=m+1}^{\NcutSmall} \Psi_{m'',i-1,q}^2 \notag\\
&\quad + \sum_{m=0}^{\NcutSmall} \sum_{\bar m''=m+1}^{\NcutSmall} \prod_{m'=0}^{m-1} \Psi_{m',i,q}^2 \cdot \psi_{m,i,q}^2 \cdot D_{t,q-1}(\Psi_{\bar m'',i-1,q}^2)  \prod_{\substack{m+1\leq m'' \leq \NcutSmall \\ m'' \neq \bar m''}} \Psi_{m'',i-1,q}^2 \notag\\
&\quad + \sum_{m=0}^{\NcutSmall}  \prod_{m'=0}^{m-1} \Psi_{m',i,q}^2 \cdot  D_{t,q-1} (\psi_{m,i,q}^2)  \cdot  \prod_{m''=m+1}^{\NcutSmall} \Psi_{m'',i-1,q}^2\,.
\label{eq:orangutan:1}
\end{align}
Higher order material derivatives of $\psi_{i,q}^2$, and mixtures of space and material derivatives are obtained similarly, by an application of the Leibniz rule. 
Equality \eqref{eq:orangutan:1} in particular justifies why we have only proven \eqref{eq:5:129} for $\Vec{i}$ and $0\leq m \leq \NcutSmall$ as on the right side of \eqref{eq:orangutan}! With \eqref{eq:5:129} and \eqref{eq:orangutan:1} in hand, we now repeat the argument from the proof of Lemma~\ref{lem:sharp:D:psi:i:q} (see the two displays below \eqref{eq:cutoff:spatial:derivatives:1}) and conclude that \eqref{eq:sharp:Dt:psi:i:q} holds. 

In order to conclude the proof of the Lemma, it remains to establish \eqref{eq:sharp:Dt:psi:i:q:mixed}. This bound follows now directly from \eqref{eq:sharp:Dt:psi:i:q} and an application of Lemma~\ref{lem:cooper:1} (to be more precise, we need to use the proof of this Lemma), in precisely the same way that \eqref{eq:sharp:Dt:psi:i:j:q} was shown earlier to imply \eqref{eq:sharp:Dt:psi:i:j:q:mixed}. As there are no changes to be made to this argument, we omit these details. 
\end{proof}

\subsubsection{\texorpdfstring{$L^1$}{L1} size of the velocity cutoffs}
The purpose of this section is to show that the inductive estimate \eqref{eq:psi:i:q:support:old} holds with $q'=q$. 
\begin{lemma}[Support estimate]
\label{lem:psi:i:q:support}
For all $0 \leq i \leq \imax(q)$ we have that 
\begin{align}
\norm{\psi_{i,q}}_{L^1}  \lesssim  \Gamma_{q+1}^{-2i+ \CLebesgue}
\label{eq:psi:i:q:support}
\end{align}
where $ \CLebesgue$ is defined in \eqref{eq:psi:i:q:support:old} and thus depends only on $b$. 
\end{lemma}
\begin{proof}[Proof of Lemma~\ref{lem:psi:i:q:support}]
If $i \leq ( \CLebesgue - 1)/2 $ then \eqref{eq:psi:i:q:support} trivially holds because  $0 \leq \psi_{i,q} \leq 1$, and $|\T^3| \leq \Gamma_{q+1}$ for all $q\geq 1$, once $a$ is chosen to be sufficiently large. Thus, we only need to be concerned with $i$ such that $( \CLebesgue + 1)/2 \leq i \leq \imax(q)$.

First, we note that Lemma~\ref{lem:partition:of:unity:psi:m} imply that the functions $\Psi_{m,i',q}$ defined in \eqref{eq:fancy:cutoff} satisfy $0 \leq \Psi_{m,i',q}^2 \leq 1$, and thus \eqref{eq:cutoff:resummation} implies that 
\begin{align}
\norm{\psi_{i,q}}_{L^1} \leq \sum_{m=0}^{\NcutSmall} \norm{\psi_{m,i,q}}_{L^1} \,. 
\label{eq:orange:orangutan:0}
\end{align}
Next, we let $j_*(i) = j_*(i,q)$ be the {\em maximal} index of $j_m$ appearing in \eqref{eq:psi:m:im:q:def}. In particular, recalling also \eqref{eq:max:j:i:q}, we have that 
\begin{align}
\Gamma_{q+1}^{i-1} < \Gamma_q^{j_*(i)} \leq \Gamma_{q+1}^i < \Gamma_q^{j_*(i)+1}  \,.
\label{eq:orange:orangutan:1}
\end{align}
Using \eqref{eq:psi:m:im:q:def}, in which we simply write $j$ instead of $j_m$,   the fact that $0\leq \psi_{j,q-1}^2, \psi_{m,i,j,q}^2 \leq 1$, and the inductive assumption \eqref{eq:psi:i:q:support:old} at level $q-1$, we may  deduce that 
\begin{align}
\norm{\psi_{m,i,q}}_{L^1} 
&\leq \norm{\psi_{j_*(i),q-1}}_{L^1} + \norm{\psi_{j_*(i)-1,q-1}}_{L^1} + \sum_{j=0}^{j_*(i)-2}  \norm{\psi_{j,q-1} \psi_{m,i,j,q}}_{L^1} \notag\\
&\leq \Gamma_q^{-2 j_*(i) + \CLebesgue} + \Gamma_q^{-2 j_*(i) + 2 + \CLebesgue} + \sum_{j=0}^{j_*(i)-2}  \abs{ \supp(\psi_{j,q-1} \psi_{m,i,j,q})}  
\,. \label{eq:orange:orangutan:2}
\end{align}
The second term on the right side of \eqref{eq:orange:orangutan:2} is estimated using the last inequality in \eqref{eq:orange:orangutan:1} as 
\begin{align}
\Gamma_q^{-2 j_*(i) + 2 + \CLebesgue} 
\leq  \Gamma_{q+1}^{-2 i} \Gamma_q^{4 + \CLebesgue}
\leq \Gamma_{q+1}^{-2 i + \CLebesgue -1}  \Gamma_q^{4 + \CLebesgue - b (\CLebesgue -1)} 
= \Gamma_{q+1}^{-2 i + \CLebesgue -1}
\label{eq:orange:orangutan:3}
\end{align}
where in the last equality we have used the definition of $\CLebesgue$ in \eqref{eq:psi:i:q:support:old}.
Clearly, the first term on the right side of \eqref{eq:orange:orangutan:2} is also bounded by the right side of \eqref{eq:orange:orangutan:3}. We are left to estimate the terms appearing in the sum on the right side of \eqref{eq:orange:orangutan:2}. The key fact is that for any $j \leq j_*(i)-2$ we have that $i \geq  i_*(j)+1$; this can be seen to hold because $b < 2$. Recalling the definition \eqref{eq:psi:i:j:def} and item~\ref{item:cutoff:2} of Lemma~\ref{lem:cutoff:construction:first:statement}, we obtain that for $j\leq j_*(i)-2$ we have
\begin{align}
\supp(\psi_{j,q-1} \psi_{m,i,j,q}) 
&\subseteq \left\{ (x,t) \in \supp(\psi_{j,q-1}) \colon h_{m,j,q}^2 \geq \frac 14 \Gamma_{q+1}^{ 2(m+1)(i-i_*(j))} \right\} \notag\\
&\subseteq \left\{ (x,t)   \colon \psi_{j\pm,q-1}^2 h_{m,j,q}^2 \geq \frac 14 \Gamma_{q+1}^{ 2(m+1)(i-i_*(j))} \right\} \,.
\label{eq:orange:orangutan:4}
\end{align}
In the second inclusion of \eqref{eq:orange:orangutan:4} we have appealed to \eqref{e:psi:i:q:overlap} at level $q-1$. By Chebyshev's inequality and the definition of $h_{m,j,q}$ in \eqref{eq:h:j:q:def} we deduce that 
\begin{align*}
&\abs{\supp(\psi_{j,q-1} \psi_{m,i,j,q})}\notag\\
&\leq 4 \Gamma_{q+1}^{-2(m+1)(i-i_*(j))} \sum_{n=0}^{\NcutLarge} \Gamma_{q+1}^{-2i_*(j)} \delta_q^{-1} (\lambda_q \Gamma_q)^{-2n} \left(\tau_{q-1}^{-1} \Gamma_{q+1}^{i_*(j)+2}\right)^{-2m} \norm{\psi_{j\pm,q-1} D^n D_{t,q-1}^m u_q}_{L^2}^2 \,.
\end{align*}
Since in the above display we have that $n\leq \NcutLarge \leq 2 \Nindv$ and $m\leq \NcutSmall \leq \Nindt$, we may combine the above estimate with \eqref{eq:mollified:velocity} and deduce that 
\begin{align}
 \abs{\supp(\psi_{j,q-1} \psi_{m,i,j,q})} 
&\leq 4 \Gamma_{q+1}^{-2(m+1)(i-i_*(j))}  \Gamma_{q+1}^{-2i_*(j)}   \left( \Gamma_q^{j+1} \Gamma_{q+1}^{-i_*(j)-2} \right)^{2m} \sum_{n=0}^{\NcutLarge} \Gamma_q^{-2n} \notag\\
&\leq 8 \Gamma_{q+1}^{-2 i }  \left( \Gamma_q^{j+1} \Gamma_{q+1}^{-i-2}\right)^{2m}   \notag\\
&\leq \Gamma_{q+1}^{-2 i +\CLebesgue - 1}
\,.
\label{eq:orange:orangutan:5}
\end{align}
In the last inequality we have used that $\Gamma_q^j \leq \Gamma_{q+1}^i$, that $m\geq 0$, and that $   \CLebesgue \geq 2$ (since $b\leq 6$). 

Combining \eqref{eq:orange:orangutan:0}, \eqref{eq:orange:orangutan:2}, \eqref{eq:orange:orangutan:3}, and \eqref{eq:orange:orangutan:5} we deduce that 
\begin{align*}
\norm{\psi_{i,q}}_{L^1}\leq \NcutSmall \,  j_*(i) \, \Gamma_{q+1}^{-2 i +\CLebesgue - 1} \,.
\end{align*}
In order to conclude the proof of the Lemma, we use that $\NcutSmall$ is a constant independent of $q$, and that by \eqref{eq:orange:orangutan:2} and \eqref{eq:imax:upper:lower} we have
\begin{align*}
j_*(i) \leq i \frac{\log \Gamma_{q+1}}{\log \Gamma_q} \leq \imax(q) b \leq \frac{4 b}{\eps_\Gamma (b-1)} \,.
\end{align*}
Thus $j_*(i)$ is also bounded from above by a constant independent of $q$ and upon taking $a$ sufficiently large we have 
\begin{align*}
 \NcutSmall \,  j_*(i) \, \Gamma_{q+1}^{-1} \leq \frac{4 \NcutSmall b}{\eps_\Gamma (b-1)} \Gamma_{q+1}^{-1} \leq 1
\end{align*}
which concludes the proof.
\end{proof}

\subsection{Definition of the temporal cutoff functions}
\label{sec:cutoff:temporal:definitions}
Let $\chi:(-1,1)\rightarrow[0,1]$ be a $C^\infty$ function which induces a partition of unity according to 
\begin{align}
 \sum_{k \in \Z} \chi^2(\cdot - k) \equiv 1.
\label{eq:chi:cut:partition:unity}
\end{align}
Consider the translated and rescaled function 
\begin{equation*}
    \chi\left(t \tau_{q}^{-1}\Gamma^{i-\cstar+2}_{q+1} - k\right) \, ,
\end{equation*}
which is supported in the set of times $t$ satisfying
\begin{equation}\label{eq:chi:support}
\left| t-\tau_q \Gamma_{q+1}^{-i+\cstar-2} k \right| \leq \tau_q\Gamma_{q+1}^{-i+\cstar-2} \quad \iff t\in \left[ (k-1)\tau_q \Gamma_{q+1}^{-i+\cstar-2}, (k+1)\tau_q \Gamma_{q+1}^{-i+\cstar-2} \right]  \, .
\end{equation}
We then define temporal cut-off functions 
\begin{align}
 \chi_{i,k,q}(t)=\chi_{(i)}(t) =  \chi\left(t \tau_{q}^{-1}\Gamma^{i-\cstar+2}_{q+1} - k\right) \, .
 \label{eq:chi:cut:def}
\end{align}
It is then clear that 
\begin{align}
{|\partial_t^m \chi_{i,k,q}| \les (\Gamma_{q+1}^{i-\cstar+2} \tau_{q}^{-1})^m}
\label{eq:chi:cut:dt}
\end{align}
for $m\geq 0$ and
\begin{equation}\label{e:chi:overlap}
    \chi_{i,k_1,q}(t)\chi_{i,k_2,q}(t) = 0
\end{equation}
for all $t\in\mathbb{R}$ unless $|k_1-k_2|\leq 1$. In analogy to $\psi_{i\pm,q}$, we define
\begin{equation}\label{e:chi:plus:minus:definition}
    \chi_{(i, k \pm, q)}(t) := \left( \chi_{(i,k-1,q)}^2(t) + \chi_{(i,k,q)}^2(t) + \chi_{(i,k+1,q)}^2(t)  \right)^\frac{1}{2},
\end{equation}
which are cutoffs with the property that
\begin{equation}\label{e:chi:overlaps}
    \chi_{(i,k\pm,q)} \equiv 1 \textnormal{ on } \supp{(\chi_{(i,k,q)})}.
\end{equation}

Next, we define the cutoffs $\tilde\chi_{i,k,q}$ by
\begin{equation}\label{eq:chi:tilde:cut:def}
\tilde\chi_{i,k,q}(t)=\tilde\chi_{(i)}(t) = \chi\left( t \tau_q^{-1}\Gamma_{q+1}^{i-\cstar} - \Gamma_{q+1}^{-\cstar}k \right).
\end{equation}
For comparison with \eqref{eq:chi:support}, we have that $\tilde\chi_{i,k,q}$ is supported in the set of times $t$ satisfying
\begin{equation}\label{eq:chi:tilde:support}
\left| t-\tau_q \Gamma_{q+1}^{-i+\cstar} k \right| \leq \tau_q\Gamma_{q+1}^{-i+\cstar}.
\end{equation}
As a consequence of these definitions and a sufficiently large choice of $\lambda_0$,  let $(i,k)$ and $(\istar,\kstar)$ be such that $\supp \chi_{i,k,q} \cap \supp \chi_{\istar,\kstar,q}\neq\emptyset$ and $\istar\in\{i-1,i,i+1\}$, then
\begin{equation}\label{eq:tilde:chi:contains}
\supp \chi_{i,k,q} \subset \supp \tilde\chi_{\istar,\kstar,q}.
\end{equation}

Finally, we shall require cutoffs $\overline{\chi}_{q,n,p}$ which satisfy the following three properties:
\begin{enumerate}[(1)]
    \item $\overline{\chi}_{q,n,p}(t)\equiv 1$ on $\supp_t \RR_{q,n,p}$
    \item $\overline{\chi}_{q,n,p}(t)=0$ if $\left\|\RR_{q,n,p}(\cdot, t')\right\|_{L^\infty(\mathbb{T}^3)}=0$ for all $|t-t'|\leq \left(\delta_q^{\sfrac 12}\lambda_q\Gamma_{q+1}^2\right)^{-1}$
    \item $\partial_t^m \overline{\chi}_{q,n,p}\lesssim \left(\delta_q^{\sfrac 12}\lambda_q\Gamma_{q+1}^2\right)^{m}$
\end{enumerate}
For the sake of specificity, recalling \eqref{eq:mollifiers}, we may set
\begin{equation}\label{def:chi:qnp}
    \overline{\chi}_{q,n,p} = \phi^{(t)}_{\left(\delta_q^{\sfrac 12}\lambda_q\Gamma_{q+1}^2\right)} \ast \mathbf{1}_{\left\{t:\left\| \RR_{q,n,p} \right\|_{L^\infty\left(\left[t-\left(\delta_q^{\sfrac 12}\lambda_q\Gamma_{q+1}^2\right)^{-1},t+\left(\delta_q^{\sfrac 12}\lambda_q\Gamma_{q+1}^2\right)^{-1}\right]\times\mathbb{T}^3\right)}>0\right\}} \, .
\end{equation}
It is then clear that $\overline{\chi}_{q,n,p}$ slightly expands and then mollifies the characteristic function of the time support of $\RR_{q,n,p}$ so that the inductive assumptions \eqref{eq:perturbation:time:support:redux:nn=0}, \eqref{eq:perturbation:time:support:redux:nn}, and \eqref{eq:perturbation:time:support:redux:nn=nmax} regarding the time support of $w_{q+1,n,p}$ may be verified.

\subsection{Estimates on flow maps}\label{s:deformation}
\label{sec:cutoff:flow:maps}
We can now make estimates regarding the flows of the vector field $\vlq$ on the support of a cutoff function.

\begin{lemma}[\textbf{Lagrangian paths don't jump many supports}]
\label{lem:dornfelder}
Let $q \geq 0$ and $(x_0,t_0)$ be given. Assume that the index $i$ is such that $\psi_{i,q}^2(x_0,t_0) \geq \kappa^2$, where $\kappa\in\left[\frac{1}{16},1\right]$. Then the forward flow $(X(t),t) := (X(x_0,t_0;t),t)$ of the velocity field $\vlq$ originating at $(x_0,t_0)$ has the property that $\psi_{i,q}^2(X(t),t) \geq\sfrac{\kappa^2}{2}$ for all $t$ be such that $|t - t_0|\leq \left(\delta_q^{\sfrac 12}\lambda_q \Gamma_{q+1}^{i+3} \right)^{-1}$, which by \eqref{eq:Lambda:q:x:1:NEW} and \eqref{eq:tilde:lambda:q:def} is satisfied for $|t-t_0|\leq \tau_q \Gamma_{q+1}^{-i+5+\cstar}$.
\end{lemma}
\begin{proof}[Proof of Lemma~\ref{lem:dornfelder}]
By the mean value theorem in time along the Lagrangian flow $(X(t),t)$ and \eqref{eq:dtq:psi:i:q:remark}, we have that
\begin{align*}
\left| \psi_{i,q}(X(t),t) - \psi_{i,q}(x_0,t_0) \right| 
&\leq |t-t_0| \norm{ D_{t,q} \psi_{i,q}}_{L^\infty} \notag\\
&\leq |t-t_0| \norm{ D_{t,q-1} \psi_{i,q}}_{L^\infty} + |t-t_0| \norm{u_{q} \cdot \nabla \psi_{i,q}}_{L^\infty}.
\end{align*}
From Lemma~\ref{lem:sharp:Dt:psi:i:q}, Lemma~\ref{lem:sharp:D:psi:i:q}, Lemma~\ref{lem:h:j:q:size}, and \eqref{eq:tau:qminusone:deltaq}, we have that 
\begin{align*}
 \norm{ D_{t,q-1} \psi_{i,q}}_{L^\infty} + \norm{u_{q} \cdot \nabla \psi_{i,q}}_{L^\infty} &\lesssim \Gamma_{q+1}^{i+3}\tau_{q-1}^{-1}+\delta_{q}^{\sfrac12} \Gamma_{q+1}^{i+1} \lambda_q\Gamma_q\\
 &\lesssim \delta_q^{\sfrac 12} \lambda_q \Gamma_{q+1}^{i+2},
\end{align*}
and hence, under the working assumption on $|t-t_0|$ we obtain
\begin{align}
\left| \psi_{i,q}(X(x_0,t_0;t),t) - \psi_{i,q}(x_0,t_0) \right| 
\lesssim \Gamma_{q+1}^{-1},
\label{eq:dornfelder:useful}
\end{align}
for some implicit constant $C>0$ which is independent of $q\geq 0$.
From the assumption of the lemma and \eqref{eq:dornfelder:useful} it follows that 
\begin{align*}
\psi_{i,q}(X(t),t) \geq  \kappa - C \Gamma_{q+1}^{-1} \geq \sfrac{\kappa}{\sqrt{2}}
\end{align*}
for all $q \geq 0$, since we have that $\kappa\geq\sfrac{1}{16}$ and $C \Gamma_{q+1}^{-1} \leq \sfrac{1}{100}$, which holds independently of $q$ once $\lambda_0$ is chosen sufficiently large.
\end{proof}

\begin{corollary}
\label{cor:dornfelder}
Suppose $(x,t)$ is such that $\psi^2_{i,q}(x,t)\geq \kappa^2$, where $\kappa\in\left[\sfrac{1}{16},1\right]$. For $t_0$ such that $\abs{t-t_0}\leq \left(\delta_q^{\sfrac 12}\lambda_q\Gamma_{q+1}^{i+4}\right)^{-1}$, which is in particular satisfied for $|t-t_0|\leq\tau_q\Gamma_{q+1}^{-i+4+\cstar}$, define $x_0$ to satisfy
\[
x=X(x_0,t_0;t).
\]
That is, the forward flow $X$ of the velocity field $\vlq$, originating at $x_0$ at time $t_0$, reaches the point $x$ at time $t$.
Then we have
\begin{equation*}
\psi_{i,q}(x_0,t_0)\neq 0 \,.
\end{equation*}
\end{corollary}
\begin{proof}[Proof of Corollary~\ref{cor:dornfelder}]
We proceed by contradiction and suppose that $\psi_{i,q}(x_0,t_0)=0$. Without loss of generality we can assume $t< t_0$. By continuity, there exists a minimal time  $t'\in(t,t_0]$ such that for $x'=x'(t')$ defined by
\[x=X(x',t';t),\]
we have
\begin{equation*}
\psi_{i,q}(x',t')= 0\,.
\end{equation*}
By minimality and \eqref{eq:lemma:partition:2}, there exists an $i'\in\{i-1,i+1\}$ such that
\begin{equation*}
\psi_{i',q}(x',t')= 1\,.
\end{equation*}
Applying Lemma~\ref{lem:dornfelder}, estimate \eqref{eq:dornfelder:useful}, we obtain
\begin{align}
\abs{\psi_{i',q}\left(X(x',t';t),t\right)-\psi_{i',q}(x',t')}=\abs{\psi_{i',q}(x,t)-\psi_{i',q}(x',t')}\lesssim \Gamma_{q+1}^{-1} \,. \label{eq:dornfelder:condition}
\end{align}
Here we have used that $|t'-t| \leq |t_0-t| \leq \left(\delta_q^{\sfrac 12}\lambda_q\Gamma_{q+1}^{i+4}\right)^{-1}\leq \left(\delta_q^{\sfrac 12}\lambda_q\Gamma_{q+1}^{i'+3}\right)^{-1}$, so that Lemma~\ref{lem:dornfelder} is applicable.
Since $\psi_{i',q}(x',t')=1$, from \eqref{eq:dornfelder:condition} we see that $\psi_{i',q}(x,t)>0$, and so $\psi^2_{i,q}(x,t)=1-\psi^2_{i',q}(x,t)$.  Then we obtain
\begin{align*}
\psi^2_{i,q}(x,t) &= 1-\psi_{i',q}^2(x,t) \notag\\
&= \left(1+\psi_{i',q}(x,t)\right) \left(1-\psi_{i',q}(x,t)\right) \notag\\
&= \left(1+\psi_{i',q}(x,t)\right) \left(\psi_{i',q}(x',t')-\psi_{i',q}(x,t)\right) \notag\\
&\lesssim \Gamma_{q+1}^{-1}
\end{align*}
which is a contradiction once $\lambda_0$ is chosen sufficiently large, since we assumed that $\psi^2_{i,q}(x,t)\geq \kappa^2$ and $\kappa\geq\sfrac{1}{16}$.
\end{proof}

\begin{definition}\label{def:transport:maps} We define $\Phi_{i,k,q}(x,t):=\Phi_{(i,k)}(x,t)$ to be the flows induced by $\vlq$ with initial datum at time $k {\tau_{q}}\Gamma_{q+1}^{-i}$ given by the identity, i.e.
\begin{equation}\label{e:Phi}
\left\{\begin{array}{l}
(\partial_t + \vlq \cdot\nabla) \Phi_{i,k,q} = 0 \\
\Phi_{i,k,q}(x,k{\tau_{q}}\Gamma_{q+1}^{-i})=x\, .
\end{array}\right.
\end{equation}
\end{definition}

We will use $D\Phi_{(i,k)}$ to denote the gradient of $\Phi_{(i,k)}$ (which is a thus matrix-valued function).  The inverse of the matrix $D\Phi_{(i,k)}$ is denoted by $\left(D\Phi_{(i,k)}\right)^{-1}$, in contrast to $D\Phi_{(i,k)}^{-1}$, which is the gradient of the inverse map $\Phi_{(i,k)}^{-1}$.

\begin{corollary}[\textbf{Deformation bounds}]
\label{cor:deformation}
For $k \in \Z$, $0 \leq i \leq  i_{\mathrm{max}}$, $q\geq 0$, and $2 \leq N \leq \sfrac{3\Nfin}{2}+1$, we have the following bounds on the support of $\psi_{i,q}(x,t){\tilde\chi_{i,k,q}(t)}$.\begin{align}
 \norm{D\Phi_{(i,k)} - {\mathrm{Id}}}_{L^\infty(\supp(\psi_{i,q} \tilde\chi_{i,k,q} ))} &\lesssim \Gamma_{q+1}^{-1}
 \label{eq:Lagrangian:Jacobian:1}\\
  \norm{D^N\Phi_{(i,k)} }_{L^\infty(\supp(\psi_{i,q} \tilde\chi_{i,k,q} ))} & \lesssim \Gamma_{q+1}^{-1} \MM{N-1, 2\Nindv, \Gamma_q\lambda_q,\tilde\lambda_q}\label{eq:Lagrangian:Jacobian:2}\\
  \norm{(D\Phi_{(i,k)})^{-1} - {\mathrm{Id}}}_{L^\infty(\supp(\psi_{i,q} \tilde\chi_{i,k,q} ))} & \lesssim \Gamma_{q+1}^{-1}\label{eq:Lagrangian:Jacobian:3}\\
  \norm{D^{N-1}\left((D\Phi_{(i,k)})^{-1}\right) }_{L^\infty(\supp(\psi_{i,q} \tilde\chi_{i,k,q} ))} & \lesssim \Gamma_{q+1}^{-1} \MM{N-1, 2\Nindv, \Gamma_q\lambda_q,\tilde\lambda_q} \label{eq:Lagrangian:Jacobian:4} \\
   \norm{D^N\Phi^{-1}_{(i,k)} }_{L^\infty(\supp(\psi_{i,q} \tilde\chi_{i,k,q} ))} & \lesssim \Gamma_{q+1}^{-1} \MM{N-1, 2\Nindv, \Gamma_q\lambda_q,\tilde\lambda_q}\label{eq:Lagrangian:Jacobian:7}
\end{align}
Furthermore, we have the following bounds for $1\leq N+M\leq \sfrac{3\Nfin}{2}$:
\begin{align}
    \left\| D^{N-N'} D_{t,q}^M D^{N'+1} \Phi_{(i,k)} \right\|_{L^\infty(\supp(\psi_{i,q}\tilde\chi_{i,k,q}))} &\leq  \tilde{\lambda}_q^{N} \MM{M,\NindSmall,\Gamma_{q+1}^{i-\cstar} \tau_q^{-1},\tilde{\tau}_q^{-1}\Gamma_{q+1}^{-1}}\label{eq:Lagrangian:Jacobian:5}\\
     \left\| D^{N-N'} D_{t,q}^M D^{N'} (D \Phi_{(i,k)})^{-1} \right\|_{L^\infty(\supp(\psi_{i,q}\tilde\chi_{i,k,q}))} &\leq \tilde{\lambda}_q^{N} \MM{M,\NindSmall,\Gamma_{q+1}^{i-\cstar} \tau_q^{-1},\tilde{\tau}_q^{-1}\Gamma_{q+1}^{-1}}\label{eq:Lagrangian:Jacobian:6}
\end{align}
for all $0\leq N'\leq N$.
\end{corollary}
\begin{proof}[Proof of Corollary~\ref{cor:deformation}] 
Let $t_k:=\tau_q\Gamma_{q+1}^{-i}k$. For $t$ is on the support of $\tilde\chi_{i,k,q}$, we may assume from \eqref{eq:chi:tilde:support} that $\abs{t-t_k}\leq {\tau_{q}}\Gamma_{q+1}^{-i+\cstar}$. Moreover, since the $\{\psi_{i',q}\}_{i'\geq 0}$ form a partition of unity, we know that there exists $i'$ such that $\psi_{i',q}^2(x,t) \geq \sfrac 12$ and $i' \in \{i-1,i,i+1\}$. Thus, we have that $\abs{t-t_k} \leq {\tau_{q}} \Gamma_{q+1}^{-i'+1+\cstar}$, and Corollary~\ref{cor:dornfelder} is applicable. For this purpose, let $x_0$ be defined by $X(x_0,t_k;t) = x$, where $X$ is the forward flow of the velocity field $\vlq$, which equals the identity at time $t_k$. Corollary~\ref{cor:dornfelder} guarantees that $(x_0,t_k) \in \supp(\psi_{i',q})$. 

The above argument shows that the flow $(X(x_0,t_k;t),t)$ remains in the support of $ \psi_{i',q}$ for all $t$ such that $|t-t_k| \leq \tau_{q} \Gamma_{q+1}^{-i+\cstar}$, where $i' \in \{i-1,i,i+1\}$.
In turn, using estimate \eqref{eq:nasty:D:vq}, this shows that 
\begin{align*}
 \sup_{|t-t_k| \leq {\tau_{q}}\Gamma_{q+1}^{-i+\cstar}} |D\vlq(X(x_0,t_k;t),t)| 
  \lesssim \norm{D\vlq}_{L^\infty(\supp (\psi_{i\pm,q} ))} 
  \lesssim \Gamma_{q+1}^{i+2} \delta_{q}^{\sfrac{1}{2}} \tilde\lambda_q.
\end{align*}
To conclude, using \eqref{item:transport:estimate:4} from Lemma~\ref{transport} and \eqref{eq:Lambda:q:x:1:NEW}, we obtain
\begin{align*}
 \norm{{D\Phi_{(i,k)}} - {\mathrm{Id}}}_{L^\infty(\supp(\psi_{i,q} \, \tilde\chi_{i,k,q}))} 
 \lesssim \tau_q \Gamma_{q+1}^{-i+\cstar} \Gamma_{q+1}^{i+2} \delta_q^{\sfrac 12} \tilde\lambda_q 
 \lesssim \Gamma_{q+1}^{-1}
\end{align*}
which implies the desired estimate in \eqref{eq:Lagrangian:Jacobian:1}.
Similarly, since the flow $(X(x_0,t_k;t),t)$ remains in the support of $ \psi_{i',q}$ for all $t$ such that $|t-t_k| \leq {\tau_{q}} \Gamma_{q+1}^{-i+\cstar}$, for $N\geq 2$ the estimates in \eqref{item:transport:estimate:3} from Lemma~\ref{transport} give that
\begin{align*}
 \norm{D^N\Phi_{(i,k)} }_{L^\infty(\supp(\psi_{i,q} \, \tilde\chi_{i,k,q}))} 
 &\lesssim {\tau_{q}} \Gamma_{q+1}^{-i+\cstar} \norm{D^N \vlq}_{L^\infty(\supp (\psi_{i\pm,q} ))}  \notag\\
 &\lesssim {\tau_{q}} \Gamma_{q+1}^{-i+\cstar} (\Gamma_{q+1}^{i+2} \delta_q^{\sfrac 12})\tilde\lambda_q \MM{N-1,2\NindLarge,\Gamma_{q} \lambda_q,\tilde \lambda_q} \notag\\
 &\lesssim \Gamma_{q+1}^{-1} \MM{N-1, 2\Nindv, \Gamma_q\lambda_q,\tilde\lambda_q}.
\end{align*}
Here we have used the bound \eqref{eq:nasty:D:vq} with $M=0$ and $K=N-1$ up to $N=\sfrac{3\Nfin}{2}+1$.

The first bound on the inverse matrix follows from the fact that matrix inversion is a smooth function in a neighborhood of the identity and fixes the identity. The second bound on the inverse matrix follows from the fact that $\det D\Phi_{(i,k)}=1$, so that we have the formula
$$  \textnormal{cof } D \Phi_{(i,k)}^T = (D\Phi_{(i,k)})^{-1}. $$
Then since the cofactor matrix is a $C^\infty$ function of the entries of $D\Phi$, we can apply Lemma~\ref{lem:Faa:di:Bruno} and the bound on $D^N \Phi_{(i,k)}$.  Note that in the application of Lemma~\ref{lem:Faa:di:Bruno}, we set $h=D\Phi_{(i,k)}-\Id$, $\Gamma=\Gamma_\psi=1$, $\const_h=\Gamma_{q+1}^{-1}$, and the cost of the spatial derivatives to be that given in \eqref{eq:Lagrangian:Jacobian:2}. The final bound on the inverse flow $\Phi_{(i,k)}^{-1}$ follows from the identity
\begin{equation}\label{eq:inverse:function:thm:bound}
D^N \left( \Phi_{(i,k)}^{-1} \right)(x) = D^{N-1} \left( \left(D\Phi_{(i,k)}\right)^{-1} \left(\Phi^{-1}(x)\right) \right),
\end{equation} 
the Faa di Bruno formula in Lemma~\ref{lem:Faa:di:Bruno}, induction on $N$, and the previously demonstrated bounds.  

The bound in \eqref{eq:Lagrangian:Jacobian:5} will be achieved by bounding 
\begin{equation*}
 D^{N-N'}\left[ \Dtq^M,D^{N'+1} \right] \Phiik \, ,
\end{equation*}
which after using that $\Dtq\Phiik=0$ will conclude the proof.  Towards this end, we apply Lemma~\ref{lem:cooper:2}, specifically Remark~\ref{rem:cooper:2} and Remark~\ref{rem:cooper:2:sum}, with $v=\vlq$ and $f=\Phi_{(i,k)}$.  The assumption \eqref{eq:cooper:2:v} (adjusted to fit Remark~\ref{rem:cooper:2:sum}) follows from \eqref{eq:nasty:D:vq} with $N_0=\sfrac{3\Nfin}{2}$, $\const_v=\Gamma_{q+1}^{i+1}\delta_q^{\sfrac{1}{2}}$, $\lambda_v=\tilde{\lambda_v}=\tilde{\lambda}_q$, $\mu_v=\Gamma_{q+1}^{i-\cstar}\tau_q^{-1}$, $\tilde{\mu}_v=\Gamma_{q+1}^{-1}\tilde{\tau}_q^{-1}$, and $N_t=\NindSmall$.  The assumption \eqref{eq:cooper:2:f} follows with $\const_f=\Gamma_{q+1}^{-1}$ from \eqref{eq:Lagrangian:Jacobian:2} and the fact that $\Dtq \Phi_{(i,k)}=0$. The desired bound then follows from the conclusion \eqref{eq:cooper:2:f:2:alt} from Remark~\ref{rem:cooper:2} after using $\Gamma_{q+1}^{-1}$ to absorb implicit constants.  The bound in \eqref{eq:Lagrangian:Jacobian:6} will follow again from Lemma~\ref{lem:Faa:di:Bruno:*} after using that $\left(D\Phiik\right)^{-1}$ is a smooth function of $D\Phiik$ in a neighborhood of the identity, which is guaranteed from \eqref{eq:Lagrangian:Jacobian:1}. As before, we set $\Gamma=\Gamma_\psi=1$ and $\const_h=\Gamma_{q+1}^{-1}$ in the application of Lemma~\ref{lem:Faa:di:Bruno:*}.  The derivative costs are precisely those   in \eqref{eq:Lagrangian:Jacobian:5}.
\end{proof}

\subsection{Stress estimates on the support of the new velocity cutoff functions}
\label{sec:cutoff:stress:bounds:0}
Before giving the definition of the stress cutoffs, we first note that the can upgrade the $L^1$ bounds for $\psi_{i,q-1} D^n D_{t,q-1}^m \RR_{\ell_q}$ available in \eqref{eq:mollified:stress:bounds}, to $L^1$ bounds for $\psi_{i,q} D^n D_{t,q}^m \RR_{\ell_q}$. We claim that:
\begin{lemma}[\textbf{$L^1$ estimates for zeroth order stress}]
\label{lem:inductive:rq:dtq}
Let $\RR_{\ell_q}$ be as defined in \eqref{eq:vlq:Rlq:def}. For $q\geq 1$ and $0 \leq i \leq \imax(q)$ we have the estimate
\begin{align}
\norm{D^k D_{t,q}^m \mathring R_{\ell_q}}_{L^1(\supp(\psi_{i,q}))}  
 \lessg \Gamma_q^{\shaq} \delta_{q+1} \MM{k,2\Nindv,\lambda_q\Gamma_q,\Tilde{\lambda}_q}\MM{m, \NindRt, \Gamma_{q+1}^{i-\cstar} \tau_q^{-1} , \Gamma_{q+1}^{-1} \Tilde{\tau}_{q}^{-1} }
\label{eq:Rn:inductive:dtq}
\end{align}
for all $k+m \leq \sfrac{3\Nfin}{2}$. 
\end{lemma}
\begin{proof}[Proof of Lemma~\ref{lem:inductive:rq:dtq}]
The first step is to apply Lemma~\ref{lem:cooper:2}, in fact Remark~\ref{rem:cooper:2:sum}, to the functions $v = v_{\ell_{q-1}}$, $f = \RR_{\ell_{q}}$, with $p=1$, and on the domain $\Omega = \supp (\psi_{i,q-1})$. The bound \eqref{eq:cooper:2:v} holds in view of the inductive assumption \eqref{eq:nasty:D:vq:old} with $q' = q-1$, for the parameters $\const_v = \Gamma_{q}^{i+1} \delta_{q-1}^{\sfrac 12} $, $\lambda_v = \tilde \lambda_v = \tilde \lambda_{q-1}$, $\mu_v = \Gamma_q^{i-\cstar} \tau_{q-1}^{-1}$, $\tilde \mu_v = \Gamma_q^{-1} \tilde \tau_{q-1}^{-1}$, $N_x = 2\Nindv$, $N_t = \Nindt$, and for $N_\circ = \sfrac{3\Nfin}{2}$. On the other hand, the assumption \eqref{eq:cooper:2:f} holds due to \eqref{eq:mollified:stress:bounds} and the fact that $\psi_{i\pm,q-1} \equiv 1$ on $\supp(\psi_{i,q-1})$, with the parameters $\const_f = \Gamma_q^\shaq \delta_{q+1}$, $\lambda_f = \lambda_q$, $\tilde \lambda_f = \tilde \lambda_q$, $N_x = 2\Nindv$, $\mu_f = \Gamma_q^{i+3} \tau_{q-1}^{-1}$, $\tilde \mu_f = \tilde \tau_{q-1}^{-1}$, $N_t = \Nindt$, and $N_\circ = 2\Nfin$. We thus conclude from \eqref{eq:cooper:2:f:2} that
\begin{align*}
&\norm{\left( \prod_{i=1}^{k} D^{\alpha_i} D_{t,q-1}^{\beta_i}\right) \RR_{\ell_q}}_{L^{1}(\supp(\psi_{i,q-1}))}   \les \Gamma_q^\shaq \delta_{q+1} \MM{|\alpha|,2\Nindv,\lambda_q,\tilde \lambda_q}   \MM{|\beta|,\Nindt,\Gamma_q^{i+3} \tau_{q-1}^{-1} ,\tilde \tau_{q-1}^{-1}}
\end{align*}
whenever $|\alpha|+|\beta| \leq \sfrac{3\Nfin}{2}$. Here we have used that $\tilde \lambda_{q-1} \leq \lambda_q$ and that $\Gamma_{q}^{i+1} \delta_{q-1}^{\sfrac 12} \tilde \lambda_{q-1} \leq \Gamma_q^{i+3} \tau_{q-1}^{-1} \leq \tilde \tau_{q-1}^{-1}$ (in view of \eqref{eq:Lambda:q:x:1:NEW}, \eqref{eq:Lambda:q:t:1}, and \eqref{eq:imax:old}). In particular, the definitions of $\psi_{i,q}$ in \eqref{eq:psi:i:q:recursive} and of $\psi_{m,i_m,q}$ in \eqref{eq:psi:m:im:q:def} imply that 
\begin{align}
&\norm{\left( \prod_{i=1}^{k} D^{\alpha_i} D_{t,q-1}^{\beta_i}\right) \RR_{\ell_q}}_{L^{1}(\supp(\psi_{i,q}))} \notag\\
&\qquad  \les \Gamma_q^\shaq \delta_{q+1} \MM{|\alpha|,2\Nindv,\lambda_q,\tilde \lambda_q}   \MM{|\beta|,\Nindt,\Gamma_{q+1}^{i+3} \tau_{q-1}^{-1} ,\tilde \tau_{q-1}^{-1}}
\label{eq:cooper:2:f:2:temp:1}
\end{align}
for all $|\alpha|+|\beta| \leq \sfrac{3\Nfin}{2}$.

The second step is to apply Lemma~\ref{lem:cooper:1} with $B = D_{t,q-1}$, $A = u_q \cdot \nabla$, $v = u_q$,  $f = \RR_{\ell_{q}}$, $p=1$, and $\Omega = \supp(\psi_{i,q})$. In this case $D^k (A+B)^m f = D^k D_{t,q}^m \RR_{\ell_q}$, which is exactly the object that we need to estimate in \eqref{eq:Rn:inductive:dtq}. The assumption \eqref{eq:cooper:v} holds due to \eqref{eq:nasty:D:wq} with $\const_v = \Gamma_{q+1}^{i+1} \delta_q^{\sfrac 12}$, $\lambda_v = \Gamma_q \lambda_q$, $\tilde \lambda_v = \tilde \lambda_q$, $N_x = 2\Nindv$, $\mu_v = \Gamma_{q+1}^{i+3} \tau_{q-1}^{-1}$, $\tilde \mu_v = \Gamma_{q+1}^{-1} \tilde \tau_q^{-1}$, $N_t = \Nindt$, and $N_* = \sfrac{3\Nfin}{2} + 1$. The assumption \eqref{eq:cooper:f} holds due to \eqref{eq:cooper:2:f:2:temp:1} with the parameters $\const_f = \Gamma_q^\shaq \delta_{q+1}$, $\lambda_f = \lambda_q$, $\tilde \lambda_f = \tilde \lambda_q$, $N_x = 2\Nindv$, $\mu_f = \Gamma_{q+1}^{i+3} \tau_{q-1}^{-1}$, $\tilde \mu_f = \tilde \tau_{q-1}^{-1}$, $N_t = \Nindt$, and $N_* = \sfrac{3\Nfin}{2}$. The bound \eqref{eq:cooper:f:*} and the parameter inequalities $\Gamma_{q+1}^{i+1} \delta_q^{\sfrac 12} \tilde \lambda_q \leq \Gamma_{q+1}^{i-\cstar-2} \tau_q^{-1} \leq \Gamma_{q+1}^{-1} \tilde \tau_q^{-1}$ and $\Gamma_{q+1}^{i+3} \tau_{q-1}^{-1} \leq \Gamma_{q+1}^{i-\cstar} \tau_q^{-1}$ (which hold due to \eqref{eq:Tau:q-1:q},  \eqref{eq:Lambda:q:x:1:NEW}, \eqref{eq:Lambda:q:t:1}, and \eqref{eq:imax:old}) then directly imply \eqref{eq:Rn:inductive:dtq}, concluding the proof.
\end{proof}
 
\begin{remark}[\textbf{$L^1$ estimates for higher order stresses}]
As discussed in Sections~\ref{ss:higherorderstresses} and \ref{ss:reynoldsheuristic}, in order to verify at level $q+1$ the inductive assumptions in \eqref{eq:inductive:assumption:derivative:q} for the new stress $\RR_{q+1}$, it will be necessary to consider a sequence of intermediate (in terms of the cost of a spatial derivative) objects $\RR_{q,n,p}$ indexed by $n$ for $1\leq n \leq \nmax$ and $1\leq p \leq \pmax$. For notational convenience, when $n=0$ and $p=1$, we define $\RR_{q,0,1}:=\RR_{\ell_q}$, and estimates on $\RR_{q,0}$ are already provided by Lemma~\ref{lem:inductive:rq:dtq}.  When $n=0$ and $p\geq 2$, $\RR_{q,0,p}=0$. For $1\leq n \leq\nmax$ and $1\leq p \leq \pmax$, the higher order stresses $\RR_{q,n,p}$ are defined in Section~\ref{ss:stress:definition}, specifically in \eqref{e:rqnp:definition}.  Note that the definition of $\RR_{q,n,p}$ is given as a finite sum of sub-objects $\HH_{q,n,p}^{n'}$ for $n'\leq n-1$ and thus requires induction on $n$.  The definition of $\HH_{q,n,p}^{n'}$ is contained in Section~\ref{ss:stress:error:identification}, specifically in \eqref{eq:Hqnp0:definition} and \eqref{eq:Hqnpnn:definition}.  Estimates on $\HH_{q,n,p}^{n'}$ on the support of $\psi_{i,q}$ are stated in \eqref{e:inductive:n:1:Hstress}, \eqref{e:inductive:n:2:Hstress}, and \eqref{e:inductive:n:3:Hstress} and proven in Section~\ref{ss:stress:oscillation:1}.  For the time being, we \emph{assume} that $\RR_{q,n,p}$ is {\em well-defined} and  satisfies  $L^1$ estimates  similar to those alluded to in \eqref{e:intro:type1:2}; more precisely, we assume that  
\begin{align}
&\norm{ D^k D_{t,q}^m \mathring R_{q,n,p}}_{L^1(\supp \psi_{i,q} )}  \les  \delta_{q+1,n,p}  \lambda\qnp^k \MM{m, \NindRt,\Gamma_{q+1}^{i-\cstarn} \tau_{q}^{-1},\Gamma_{q+1}^{-1} \Tilde{\tau}_{q}^{-1}}
\label{eq:Rn:inductive:assumption}
\end{align}
for all $0 \leq k+m \leq \Nfn$. For the purpose of defining the stress cutoff functions, the precise definitions of the $n$ and $p$-dependent parameters $\delta_{q+1,n,p}, \lambda\qnp$, $\Nfn$, and $\cstarn$ present in \eqref{eq:Rn:inductive:assumption} are not relevant. Note however that definitions for $\lambda\qnp$ for $n=0$ are given in \eqref{eq:lambda:q:0:1:def}, while for $1\leq n\leq\nmax$ and $1\leq p \leq \pmax$, the definitions are given in  \eqref{def:lambda:q:n:p}.  Similarly, when $n=0$, we let $\delta_{q+1,0,p}= \Gamma_{q}^{\shaq}\delta_{q+1}$ as is consistent with \eqref{eq:delta:q:0:def}, and when $1\leq n \leq \nmax$ and $1\leq p \leq \pmax$, $\delta_{q+1,n,p}$ is defined in \eqref{eq:delta:q:n:def}. Finally, note that there are losses in the sharpness and order of the available derivative estimates in \eqref{eq:Rn:inductive:assumption} relative to \eqref{eq:Rn:inductive:dtq}.  Specifically, the higher order estimates will only be proven up to $\Nfn$, which is a parameter that is decreasing with respect to $n$ and defined in \eqref{def:Nfn:formula}. For the moment it is only important to note  that $\Nfn \gg 14\Nindv$ for all $0\leq n \leq \nmax$, which is necessary in order to establish of \eqref{eq:inductive:assumption:derivative:q} and \eqref{eq:Rq:inductive:assumption} at level $q+1$.  Similarly, there is a loss in the cost of sharp material derivatives in \eqref{eq:Rn:inductive:assumption}, as $\cstarn$ will be a parameter which is decreasing with respect to $n$.  When $n=0$, we set $\cstarn=\cstar$ so that \eqref{eq:Rn:inductive:dtq} is consistent with \eqref{eq:Rn:inductive:assumption}. For $1\leq n \leq\nmax$, $\cstarn$ is defined in \eqref{def:cstarn:formula}.
\end{remark}

\subsection{Definition of the stress cutoff functions}
\label{sec:cutoff:stress:definitions}

For $q\geq 1$, $0 \leq i \leq i_{\mathrm{max}}$,   $0 \leq n \leq n_{\mathrm{max}}$,  and $1\leq p \leq \pmax$,  in analogy to the functions $h_{m,j_m,q}$ in \eqref{eq:h:j:q:def}, and keeping in mind the bound \eqref{eq:Rn:inductive:assumption}, we define
\begin{align}
g_{i,q,n,p}^2(x,t) =  1 + \sum_{k=0}^{\NcutLarge}\sum_{m=0}^{\NcutSmall} 
\delta_{q+1,n,p}^{-2} (\Gamma_{q+1} \lambda\qnp)^{-2k}  (\Gamma_{q+1}^{i-\cstarn+2} \tau_{q}^{-1})^{-2m} 
|D^k D_{t,q}^m \RR_{q,n,p}(x,t)|^2 .
\label{eq:g:i:q:n:def}
\end{align}
With this notation, for $j\geq 1$ the stress cut-off functions are defined by
\begin{align}
\omega_{i,j,q,n,p}(x,t) = \psi_{0,q+1} \Big( \Gamma_{q+1}^{-2 j} \, g_{i,q,n,p}(x,t) \Big)
\,,
\label{eq:omega:cut:def}
\end{align}
while for $j=0$ we let 
\begin{align}
\omega_{i,0,q,n,p}(x,t) = \tilde\psi_{0,q+1}\Big(g_{i,q,n,p}(x,t)\Big)
\,,
\label{eq:omega:cut:def:0}
\end{align}
where $\psi_{0,q+1}$ and $\tilde \psi_{0,q+1}$ are as in Lemma~\ref{lem:cutoff:construction:first:statement}.
The above defined cutoff functions $\omega_{i,j,q,n,p}$ will be shown to obey good estimates on the support of the velocity cutoffs $\psi_{i,q}$ defined earlier.

\subsection{Properties of the stress cutoff functions}
\label{sec:cutoff:stress:properties}
\subsubsection{Partition of unity}
An immediate consequence of \eqref{eq:tilde:partition} with $m=0$ is that for every fixed $i,n$, we have 
\begin{align}
\sum_{j\geq 0} \omega_{i,j,q,n,p}^2 = 1
\label{eq:omega:cut:partition:unity}
\end{align}
on $\T^3 \times \R$. Thus, $\{\omega_{i,j,q,n,p}^2 \}_{j\geq 0}$ is a partition of unity.

\subsubsection{$L^\infty$ estimates for the higher order stresses} 
We recall cf.~\eqref{eq:DN:psi:q} and \eqref{eq:DN:psi:q:gain} that the cutoff function $\psi_{0,q+1}$ appearing in the definition \eqref{eq:omega:cut:def} satisfies different derivative bounds according to the size of its argument. Accordingly, we introduce the following notation. 
\begin{definition}[\textbf{Left side of the cutoff function $\omega_{i,j,q,n,p}$}]
\label{def:omega:left:right}
For $j\geq 1$ we say that 
\begin{align}
\label{eq:omega:left:right}
(x,t) &\in \supp (\omega_{i,j,q,n,p}^{\mathsf{L}}) \qquad \mbox{if} \qquad \sfrac 14 \leq \Gamma_{q+1}^{-2j} g_{i,q,n,p}(x,t) \leq 1  \,. 
\end{align}
When $j=0$ we do not define the left side of the cutoff  function $\omega_{i,0,q,n,p}$.
\end{definition}

Directly from the definition \eqref{eq:g:i:q:n:def}--\eqref{eq:omega:cut:def:0}, the support properties of the functions $\psi_{0,q+1}$ and $\tilde \psi_{0,q+1}$ stated in Lemma~\ref{lem:cutoff:construction:first:statement}, and using Definition~\ref{def:omega:left:right}, it follows that:
\begin{lemma}
\label{lem:D:Dt:Rn:sharp}
For all $0 \leq m \leq \NcutSmall$, $0 \leq k \leq \NcutLarge$, and $j\geq 0$, we have that 
\begin{align*}
{\mathbf{1}}_{\supp (\omega_{i,j,q,n,p})} | D^k D_{t,q}^m \RR_{q,n,p} (x,t) | 
&\leq \Gamma_{q+1}^{2(j+1)}  \delta_{q+1,n,p}  (\Gamma_{q+1} \lambda\qnp)^{k}  (\Gamma_{q+1}^{i-\cstarn+2} \tau_{q}^{-1})^{m}
\,.
\end{align*}
In the above estimate, if we replace ${\mathbf{1}}_{\supp (\omega_{i,j,q,n,p})}$ with ${\mathbf{1}}_{\supp (\omega_{i,j,q,n,p}^{\mathsf{L}})}$ (cf.~Definition~\ref{def:omega:left:right}), then the factor $\Gamma_{q+1}^{2(j+1)}$ may be sharpened to $\Gamma_{q+1}^{2j}$. 
Moreover, if $j\geq 1$, then $g_{i,q,n,p}(x,t) \geq (\sfrac 14) \Gamma_{q+1}^{2j}$.
\end{lemma}

Lemma~\ref{lem:D:Dt:Rn:sharp} provides sharp $L^\infty$ bounds for the space and material derivatives of  $\RR_{q,n,p}$, at least when the number of space derivatives is less than $\NcutLarge$, and the number of material derivatives is less than $\NcutSmall$. If we are willing to pay a Sobolev-embedding loss, then~\eqref{eq:Rn:inductive:assumption} implies lossy $L^\infty$ bounds for large numbers of space and material derivatives. 
\begin{lemma}[\textbf{Derivative bounds with Sobolev loss}]
\label{lem:D:Rn:Sobolev:loss}
For $q\geq 1$, $n\geq 0$, and $0 \leq i \leq i_{\mathrm{max}}$, we have that:
\begin{align}
\norm{D^{k} D^m_{t,q} \RR_{q,n,p}}_{L^\infty(\supp \psi_{i,q})}
\les   
\delta_{q+1,n,p} \lambda\qnp^{k+3} 
\MM{m, \NindRt,\Gamma_{q+1}^{i-\cstarn+1} \tau_{q}^{-1},\Gamma_{q+1}^{-1} \Tilde{\tau}_{q}^{-1}}
\label{eq:D:Rn:Sobolev:loss:2}
\end{align}
for all $k + m \leq \Nfn-4$. 
\end{lemma}
\begin{proof}[Proof of Lemma~\ref{lem:D:Rn:Sobolev:loss}]
We apply Lemma~\ref{lem:Sobolev:cutoffs} to $f= \RR_{q,n,p}$, with $\psi_i = \psi_{i,q}$, and with $p=1$. 
Assumption \eqref{eq:L12:to:Linfty:psi} holds in view of \eqref{eq:sharp:D:psi:i:q}, with the parameter choice $\rho =  \Gamma_q \tilde \lambda_q < \Gamma_{q+1} \tilde\lambda_q = \lambda_{q,0,1} \leq \lambda\qnp$, where the inequalities follow immediately from \eqref{eq:lambda:q:0:1:def}-\eqref{def:lambda:q:n:p}. The assumption~\eqref{eq:Lp:vomit:1} holds due to~\eqref{eq:Rn:inductive:assumption}, with the parameter choices $\const_f =  \delta_{q+1,n,p}$, $\lambda  = \tilde \lambda = \lambda\qnp$,  $\mu_i = \Gamma_{q+1}^{i-\cstarn} \tau_q^{-1}$, $\tilde \mu_i = \Gamma_{q+1}^{-1} \tilde \tau_q^{-1}$, $N_t = \Nindt$, and $N_\circ = \Nfn$. The Lemma now directly follows from \eqref{eq:L2:to:Linfty} with $p=1$. 
\end{proof}

We note that Lemmas~\ref{lem:D:Dt:Rn:sharp} and~\ref{lem:D:Rn:Sobolev:loss} imply the following estimate:
\begin{corollary}[\textbf{$L^\infty$ bounds for the stress}]
\label{cor:D:Dt:Rn:sharp}
For $q \geq 0$, $0\leq i\leq \imax$, $0 \leq n \leq \nmax$, and $1\leq p \leq \pmax$ we have 
\begin{align}
&\norm{D^k D_{t,q}^m \RR_{q,n,p}}_{L^\infty(\supp \psi_{i,q} \cap \supp  \omega_{i,j,q,n,p})} \notag\\
&\qquad \les \Gamma_{q+1}^{2(j+1)} \delta_{q+1,n,p}  (\Gamma_{q+1} \lambda\qnp)^{k} \MM{m, \NindRt,\Gamma_{q+1}^{i-\cstarn+2} \tau_{q}^{-1},\Gamma_{q+1}^{-1} \Tilde{\tau}_{q}^{-1}}
\label{eq:D:Dt:Rn:sharp}
\end{align}
for all $k +m \leq \Nfn-4$. In the above estimate, if we replace $\supp (\omega_{i,j,q,n,p})$ with $\supp (\omega_{i,j,q,n,p}^{\mathsf{L}})$ (cf.~Definition~\ref{def:omega:left:right}), then the factor $\Gamma_{q+1}^{2(j+1)}$ may be sharpened to $\Gamma_{q+1}^{2j}$. 
\end{corollary}
\begin{proof}[Proof of Corollary~\ref{cor:D:Dt:Rn:sharp}]
For $m \leq \NcutSmall$ and $k\leq \NcutLarge$, the bound \eqref{eq:D:Dt:Rn:sharp} is already contained in Lemma~\ref{lem:D:Dt:Rn:sharp} (both for $\supp (\omega_{i,j,q,n,p})$, and the improved bound for $\supp (\omega_{i,j,q,n,p}^{\mathsf{L}})$). When either $k>\NcutLarge$ or $m> \NcutSmall$, we appeal to estimate~\eqref{eq:D:Rn:Sobolev:loss:2} and the parameter bound  
\begin{align*}
&\delta_{q+1,n,p} \lambda\qnp^{k+3}  \MM{m, \NindRt,\Gamma_{q+1}^{i-\cstarn+1} \tau_{q}^{-1},\Gamma_{q+1}^{-1} \Tilde{\tau}_{q}^{-1}} 
\notag\\
&  \leq \left( \Gamma_{q+1}^{- k - \min\{m,\Nindt\} } \lambda\qnp^3 \right) 
\delta_{q+1,n,p}  (\Gamma_{q+1} \lambda\qnp)^k 
\MM{m, \NindRt,\Gamma_{q+1}^{i-\cstarn+2} \tau_{q}^{-1},\Gamma_{q+1}^{-1} \Tilde{\tau}_{q}^{-1}}
\notag\\
&  \leq   \delta_{q+1,n,p}(\Gamma_{q+1} \lambda\qnp)^k 
\MM{m, \NindRt,\Gamma_{q+1}^{i-\cstarn+2} \tau_{q}^{-1},\Gamma_{q+1}^{-1} \Tilde{\tau}_{q}^{-1}}
\,.
\end{align*}
The second estimate in the above display is a consequence of the fact that when either $k>\NcutLarge$ or $m> \NcutSmall$, since $\NcutLarge\geq \NcutSmall$, we have
\begin{align}
\Gamma_{q+1}^{- k - \min\{m,\Nindt\} } \lambda\qnp^3  \leq \Gamma_{q+1}^{-\NcutSmall} \lambda_{q+1}^3  \leq 1  \,,
\label{eq:Nind:cond:8}
\end{align}
once $\NcutSmall$ (and hence $\NcutLarge$) are chosen large enough, as in \eqref{eq:Nind:cond:3}.
\end{proof}

In the proof of Lemma~\ref{lem:D:Dt:omega:sharp} below, we shall require one more $L^\infty$ bound for $\RR_{q,n,p}$, which is for iterates of space and material derivatives. It is convenient to record this bound now, as it follows directly from Corollary~\ref{cor:D:Dt:Rn:sharp}.
\begin{corollary}
\label{cor:D:Dt:Rn:sharp:new}
For $q \geq 0$, $0\leq i\leq \imax$, $0 \leq n \leq \nmax$, $1\leq p \leq \pmax$, and $\alpha,\beta \in \N_0^k$ we have 
\begin{align}
&\norm{\left(\prod_{\ell=1}^{k} D^{\alpha_\ell} D_{t,q}^{\beta_\ell}\right) \RR_{q,n,p}}_{L^\infty(\supp \psi_{i,q} \cap \supp \omega_{i,j,q,n,p})} \notag\\
&\qquad \les \Gamma_{q+1}^{2(j+1)} \delta_{q+1,n,p}  
(\Gamma_{q+1} \lambda\qnp)^{|\alpha|}  
\MM{|\beta|, \NindRt,\Gamma_{q+1}^{i-\cstarn+2} \tau_{q}^{-1},\Gamma_{q+1}^{-1} \Tilde{\tau}_{q}^{-1}}
\label{eq:D:Dt:Rn:sharp:new}
\end{align}
for all $|\alpha| + |\beta| \leq \Nfn-4$. In the above estimate, if we replace $\supp (\omega_{i,j,q,n,p})$ with $\supp (\omega_{i,j,q,n,p}^{\mathsf{L}})$ (cf.~Definition~\ref{def:omega:left:right}), then the factor $\Gamma_{q+1}^{2(j+1)}$ may be sharpened to $\Gamma_{q+1}^{2j}$. 
\end{corollary}
\begin{proof}[Proof of Corollary~\ref{cor:D:Dt:Rn:sharp:new}]
The proof follows from Corollary~\ref{cor:D:Dt:Rn:sharp} and Lemma~\ref{lem:cooper:2}. The bounds corresponding to $\supp \omega_{i,j,q,n,p}$ and $\supp \omega_{i,j,q,n,p}^{\mathsf{L}}$ are identical (except for the improvement $\Gamma_{q+1}^{2(j+1)} \mapsto \Gamma_{q+1}^{2j}$ in the later case), so we only give details for the former.  Since $D_{t,q} = \partial_t + v_{\ell_q} \cdot \nabla$, Lemma~\ref{lem:cooper:2} is applied with $v = v_{\ell_q}$, $f = \RR_{q,n,p}$, $\Omega = \supp \psi_{i,q} \cap \supp \omega_{i,j,q,n,p}$, and $p=\infty$. In view of estimate \eqref{eq:nasty:D:vq} and the fact that $\sfrac{3\Nfin}{2} \geq \Nfn$, the assumption \eqref{eq:cooper:2:v} holds with $\const_v = \Gamma_{q+1}^{i+1} \delta_q^{\sfrac 12}$, $\lambda_v = \Gamma_q \lambda_q$, $\tilde \lambda_v = \tilde \lambda_q$, $N_x = 2 \Nindv$, $\mu_v = \Gamma_{q+1}^{i-\cstarn} \tau_q^{-1}$, $\tilde \mu_v = \Gamma_{q+1}^{-1} \tilde \tau_q^{-1}$, and $N_t = \Nindt $. On the other hand, the bound \eqref{eq:D:Dt:Rn:sharp:new} implies assumption~\eqref{eq:cooper:2:f} with $\const_f = \Gamma_{q+1}^{2(j+1)} \delta_{q+1,n,p}$, $\lambda_f =  \tilde \lambda_f = \Gamma_{q+1} \lambda\qnp$,  $\mu_f = \Gamma_{q+1}^{i-\cstarn +2} \tau_q^{-1}$, $\tilde \mu_f = \Gamma_{q+1}^{-1} \tilde \tau_q^{-1}$, and $N_t = \Nindt$. Since $\lambda_v\leq   \lambda_f$, $\tilde \lambda_v \leq \tilde \lambda_f$, $\mu_v \leq \mu_f$, and $\tilde \mu_v = \tilde \mu_f$, we deduce from the bound \eqref{eq:cooper:2:f:2} (in fact, its version mentioned in Remark~\ref{rem:cooper:2:sum})  that \eqref{eq:D:Dt:Rn:sharp:new} holds, thereby concluding the proof. Here we are also implicitly using the parameter estimate $\const_v \tilde \lambda_v \leq \mu_f$, which holds due to \eqref{eq:Lambda:q:x:1:NEW}.
\end{proof}

\subsubsection{Maximal $j$ index in the stress cutoffs} 

\begin{lemma}[\textbf{Maximal $j$ index in the stress cutoffs}]
\label{lem:maximal:j}
Fix $q\geq 0$, $0 \leq n \leq n_{\mathrm{max}}$, and $1\leq p \leq \pmax$. There exists a $\jmax = \jmax(q,n,p) \geq 1$, determined by \eqref{eq:j:max:def} below, which is bounded independently of $q$, $n$, and $p$ as in \eqref{eq:jmax:bound}, such that for any $0 \leq i \leq \imax(q)$, we have 
\begin{align*}
\psi_{i,q} \, \omega_{i,j,q,n,p} \equiv 0 \qquad \mbox{for all} \qquad j > j_{\mathrm{max}}.
\end{align*}
Moreover, the bound
\begin{align*}
\Gamma_{q+1}^{2(j_{\mathrm{max}} -1) } \les \lambda\qnp^{3}
\end{align*}
holds, with an implicit constant that independent of $q$ and $n$. 
\end{lemma}
\begin{proof}[Proof of Lemma~\ref{lem:maximal:j}]
We define $j_{\mathrm{max}}$ by
\begin{align}
\jmax = \jmax(q,n,p)=  \frac 12  \left\lceil \frac{\log(M_b \sqrt{8 \NcutLarge\NcutSmall} \lambda\qnp^3)}{\log(\Gamma_{q+1})} \right\rceil
\label{eq:j:max:def}
\end{align}
where $M_b$ is the implicit $q$, $n$, $p$, and $i$-independent constant in \eqref{eq:D:Rn:Sobolev:loss:2}; that is we take the largest such constant among all values of $k$ and $m$ with $k+m\leq \Nfn-4$.
To see that $j_{\mathrm{max}}$ may be bounded independently of $q$, $n$, and $p$, we note that $\lambda\qnp \leq \lambda_{q+1}$, and thus 
\begin{align*}
2j_{\mathrm{max}} \leq 1 + \frac{\log(M_b \sqrt{8 \NcutLarge\NcutSmall}) + 3 \log(\lambda_{q+1})}{\log(\Gamma_{q+1})}  \to1 + \frac{3b}{\eps_\Gamma(b-1)} \quad \mbox{as} \quad q\to \infty.
\end{align*}
Thus, assuming that $a = \lambda_0$ is sufficiently large, we obtain that 
\begin{align}
2\jmax(q,n,p) \leq  \frac{4b}{\eps_\Gamma(b-1)}
\label{eq:jmax:bound}
\end{align}
for all $q\geq 0$, $0 \leq n \leq \nmax$, and $1\leq p\leq \pmax$.

To conclude the proof of the Lemma, let $j>j_{\mathrm{max}}$, as defined in \eqref{eq:j:max:def}, and assume by contradiction that there exists a point $(x,t) \in \supp (\psi_{i,q} \omega_{i,j,q,n,p}) \neq \emptyset$. In particular, $j\geq 1$. Then, by \eqref{eq:g:i:q:n:def}--\eqref{eq:omega:cut:def} and the pigeonhole principle, we see that there exists $0 \leq k \leq \NcutLarge$ and $0 \leq m \leq \NcutSmall$ such that 
\begin{align*}
|D^k D_{t,q}^m \RR_{q,n,p}(x,t)| \geq \frac{\Gamma_{q+1}^{2j}}{\sqrt{8 \NcutLarge \NcutSmall}} \delta_{q+1,n,p} (\Gamma_{q+1} \lambda\qnp)^k  (\Gamma_{q+1}^{i-\cstarn+2} \tau_{q}^{-1})^{m}.
\end{align*}
On the other hand, from \eqref{eq:D:Rn:Sobolev:loss:2}, we have that 
\begin{align*}
|D^k D_{t,q}^m \RR_{q,n,p}(x,t)| \leq M_b  \lambda\qnp^3 \delta_{q+1,n,p}  \lambda\qnp^k (\Gamma_{q+1}^{i-\cstarn+1} \tau_{q}^{-1})^{m} .
\end{align*}
The above two estimates imply that
\begin{align*}
\Gamma_{q+1}^{2(\jmax + 1)} \leq \Gamma_{q+1}^{2j} 
\leq M_b \sqrt{8 \NcutLarge \NcutSmall} \Gamma_{q+1}^{-k-m} \lambda\qnp^3,
\leq M_b \sqrt{8 \NcutLarge \NcutSmall} \lambda\qnp^3,
\end{align*}
which contradicts the fact that $j> j_{\mathrm{max}}$, as defined in \eqref{eq:j:max:def}.
\end{proof}

\subsubsection{Bounds for space and material derivatives of the stress cutoffs}
 
\begin{lemma}[\textbf{Derivative bounds for the stress cutoffs}]
\label{lem:D:Dt:omega:sharp}
For $q\geq 0$, $0 \leq n \leq \nmax$, $1\leq p \leq \pmax$, $0 \leq i \leq \imax$, and $0 \leq j \leq \jmax$, we have that
\begin{align}
\frac{{\mathbf{1}}_{\supp \psi_{i,q}} |D^N D_{t,q}^M \omega_{i,j,q,n,p}|}{\omega_{i,j,q,n,p}^{1-(N+M)/\Nfin}} 
\les (\Gamma_{q+1} \lambda\qnp)^N \MM{M, \Nindt-\NcutSmall,\Gamma_{q+1}^{i-\cstarn+2} \tau_{q}^{-1},\Gamma_{q+1}^{-1} \Tilde{\tau}_{q}^{-1}}
\label{eq:D:Dt:omega:sharp}
\end{align}
for all $N + M \leq \Nfn-\NcutLarge-\NcutSmall-4$.
\end{lemma}
\begin{remark}\label{rem:D:Dt:omega:sharp-ish}
Notice that the sharp derivative bounds in \eqref{eq:D:Dt:omega:sharp} are only up to $\NindSmall-\NcutSmall$.  In order to obtain bounds up to $\NindSmall$, we may argue exactly as in the string of inequalities which converted \eqref{eq:dtq:psi:i:q:remark} into \eqref{eq:nasty:Dt:psi:i:q:orangutan:redux}, resulting in the bound
\begin{align}
\frac{{\mathbf{1}}_{\supp \psi_{i,q}} |D^N D_{t,q}^M \omega_{i,j,q,n,p}|}{\omega_{i,j,q,n,p}^{1-(N+M)/\Nfin}} 
\les (\Gamma_{q+1} \lambda\qnp)^N \MM{M, \Nindt,\Gamma_{q+1}^{i-\cstarn+3} \tau_{q}^{-1},\Gamma_{q+1}^{-1} \Tilde{\tau}_{q}^{-1}} \, .
\label{eq:D:Dt:omega:sharp-ish}
\end{align}
\end{remark}

\begin{proof}[Proof of Lemma~\ref{lem:D:Dt:omega:sharp}]
For simplicity, we only treat here the case $j\geq 1$. Indeed, for  $j=0$ we simply replace $\psi_{0,q+1}$ with $\tilde \psi_{0,q+1}$, which by~Lemma~\ref{lem:cutoff:construction:first:statement} has similar properties to $\psi_{0,q+1}$.

The goal is to apply the Faa di Bruno Lemma~\ref{lem:Faa:di:Bruno:*} with $\psi = \psi_{0,q+1}$, $\Gamma = \Gamma_{q+1}^{-j}$, $D_t = D_{t,q}$,  and $h(x,t) = g_{i,q,n,p}(x,t)$, so that $g = \omega_{i,j,q,n,p}$.

Because the cutoff function $\psi = \psi_{0,q+1}$ satisfies slightly different estimates depending on whether we are in the case \eqref{eq:DN:psi:q} or \eqref{eq:DN:psi:q:gain}, assumption~\eqref{eq:Faa:di:Bruno:lem:1:*} holds with $\Gamma_\psi = 1$, and respectively $\Gamma_\psi = \Gamma_{q+1}^{-1}$, depending on whether we work on the set $\supp( \omega_{i,j,q,n,p}^{\mathsf{L}})$ or on the set $\supp (\omega_{i,j,q,n,p}) \setminus \supp( \omega_{i,j,q,n,p}^{\mathsf{L}})$ (cf.~Definition~\ref{def:omega:left:right}). We have in fact encountered this same issue in the proof of Lemmas~\ref{lem:sharp:D:psi:i:q} and~\ref{lem:sharp:Dt:psi:i:j:q}. The slightly worse value of $\Gamma_\psi$ for $(x,t) \in \supp( \omega_{i,j,q,n,p}^{\mathsf{L}})$ is however precisely balanced out by the fact that in Corollary~\ref{cor:D:Dt:Rn:sharp:new} the bound \eqref{eq:D:Dt:Rn:sharp:new} is improved by a factor for $\Gamma_{q+1}^2$ on $\supp( \omega_{i,j,q,n,p}^{\mathsf{L}})$. Since in the end these two factors of $\Gamma_{q+1}^2$ cancel out, as they did in Lemmas~\ref{lem:sharp:D:psi:i:q} and~\ref{lem:sharp:Dt:psi:i:j:q}, we only give the proof of the bound \eqref{eq:D:Dt:omega:sharp} for $(x,t) \in \supp (\omega_{i,j,q,n,p}) \setminus \supp( \omega_{i,j,q,n,p}^{\mathsf{L}})$, which is equivalent to the condition that $1 < \Gamma_{q+1}^{-2j} g_{i,q,n,p}(x,t)  \leq \Gamma_{q+1}^2$. Note moreover that we do not perform any estimates for $(x,t)$ such that $1 < \Gamma_{q+1}^{-2j}  g_{i,q,n,p}(x,t)  < (\sfrac 14) \Gamma_{q+1}^2$ since in this region $\psi_{0,q+1} \equiv 1$ (see item~\ref{item:cutoff:2}(b) in Lemma~\ref{lem:cutoff:construction:first:statement}) and so its derivatives equal to $0$. Therefore, for the remainder of the proof we work with the subset of $\supp \omega_{i,j,q,n,p}$ on which  we have 
\begin{align}
(\sfrac 14) \Gamma_{q+1}^2 \leq \Gamma_{q+1}^{-2j}  g_{i,q,n,p}(x,t) \leq \Gamma_{q+1}^2
\label{eq:omega:supp:subset}
\,.
\end{align}
This ensures that assumption~\eqref{eq:Faa:di:Bruno:lem:1:*} of Lemma~\ref{lem:Faa:di:Bruno:*} holds with $\Gamma_\psi = \Gamma_{q+1}^{-1}$.

 In order to verify condition \eqref{eq:Faa:di:Bruno:lem:2:*}, the main requirement is a supremum bound for $D^N D_{t,q}^M g_{i,q,n,p}$ in $L^\infty$ on the support of $\psi_{i,q} \omega_{i,j,q,n,p}$. In this direction, we claim that for all $(x,t)$ as in \eqref{eq:omega:supp:subset}, we have 
\begin{align}
 & {\mathbf{1}}_{\supp \psi_{i,q}} \abs{D^N D_{t,q}^M g_{i,q,n,p}(x,t)}   \les \Gamma_{q+1}^{2j+2} (\Gamma_{q+1} \lambda\qnp)^N \MM{M, \Nindt-\NcutSmall,\Gamma_{q+1}^{i-\cstarn+2} \tau_{q}^{-1},\Gamma_{q+1}^{-1} \Tilde{\tau}_{q}^{-1}}
\label{eq:Maradona:86}
\end{align}
for all $N+M\leq \Nfn-\NcutLarge-\NcutSmall-4$. Thus, assumption \eqref{eq:Faa:di:Bruno:lem:2:*} of Lemma~\ref{lem:Faa:di:Bruno:*} holds with $\const_h = \Gamma_{q+1}^{2j+2}$, $\lambda = \tilde \lambda = \Gamma_{q+1} \lambda\qnp$, $\mu = \Gamma_{q+1}^{i-\cstarn+2} \tau_{q}^{-1}$, $\tilde \mu = \Gamma_{q+1}^{-1} \Tilde{\tau}_{q}^{-1}$, and $N_t = \Nindt-\NcutSmall$. In particular, we note that $(\Gamma_\psi \Gamma)^{-2} \const_h = 1$, and estimate \eqref{eq:Faa:di:Bruno:lem:3:*} of Lemma~\ref{lem:Faa:di:Bruno:*} directly implies \eqref{eq:D:Dt:omega:sharp}.
 
Thus, in order to complete the proof of the lemma it remains to establish estimate \eqref{eq:Maradona:86}.
As in the proof of Lemma~\ref{lem:sharp:D:psi:i:q}, it is more convenient to first estimate $ D^N D_{t,q}^M (g_{i,q,n,p}(x,t)^2)$, as its definition (cf.~\eqref{eq:g:i:q:n:def}) makes it more amenable to the use of the Leibniz rule. Indeed, for all $N+M\leq \Nfn-\NcutLarge-\NcutSmall-4$ we have that 
\begin{align*}
D^N D_{t,q}^M g_{i,q,n,p}^2 
&= \sum_{N'=0}^N \sum_{M'=0}^M {N\choose N'} {M \choose M'} \notag\\
&\qquad \times \sum_{k=0}^{\NcutLarge}\sum_{m=0}^{\NcutSmall} \frac{D^{N'} D_{t,q}^{M'} D^k D_{t,q}^m \RR_{q,n,p} \, D^{N-N'} D^{M-M'}D^k D_{t,q}^m \RR_{q,n,p}}{\delta_{q+1,n,p}^{2} (\Gamma_{q+1} \lambda\qnp)^{2k}  (\Gamma_{q+1}^{i-\cstarn+2} \tau_{q}^{-1})^{2m}} .
\end{align*} 
Combining the above display with estimate \eqref{eq:D:Dt:Rn:sharp:new} and the fact that $k+m + N + M \leq \Nfn-4$, we deduce   
\begin{align}
& {\mathbf{1}}_{\supp \psi_{i,q} \cap \supp \omega_{i,j,q,n,p}} \abs{D^N D_{t,q}^M g_{i,q,n,p}^2} \notag\\
&\les \sum_{N'=0}^N \sum_{M'=0}^M  \sum_{k=0}^{\NcutLarge}\sum_{m=0}^{\NcutSmall} 
\frac{1}{\delta_{q+1,n,p}^{2} (\Gamma_{q+1} \lambda\qnp)^{2k}  (\Gamma_{q+1}^{i-\cstarn+2} \tau_{q}^{-1})^{2m}} \notag\\
&\qquad \times 
\Gamma_{q+1}^{2(j+1)} \delta_{q+1,n,p} (\Gamma_{q+1} \lambda\qnp)^{N'+k}
\MM{M'+m, \Nindt,\Gamma_{q+1}^{i-\cstarn+2} \tau_{q}^{-1},\Gamma_{q+1}^{-1} \Tilde{\tau}_{q}^{-1}}
\notag\\
&\qquad \times 
\Gamma_{q+1}^{2(j+1)} \delta_{q+1,n,p}  (\Gamma_{q+1} \lambda\qnp)^{N-N'+k}   
\MM{M-M'+m, \Nindt,\Gamma_{q+1}^{i-\cstarn+2} \tau_{q}^{-1},\Gamma_{q+1}^{-1} \Tilde{\tau}_{q}^{-1}}
\notag\\
&\les \Gamma_{q+1}^{4(j+1)} (\Gamma_{q+1} \lambda\qnp)^N
\MM{M, \Nindt-\NcutSmall,\Gamma_{q+1}^{i-\cstarn+2} \tau_{q}^{-1},\Gamma_{q+1}^{-1} \Tilde{\tau}_{q}^{-1}}
\,.
\label{eq:Maradona:90}
\end{align} 
Lastly, we show that the bound \eqref{eq:Maradona:90} and the fact that we work with $(x,t)$ such that \eqref{eq:omega:supp:subset} holds, implies \eqref{eq:Maradona:86}. This argument is the same as the one found earlier in \eqref{eq:sharp:D:psi:im:q}--\eqref{eq:cutoff:spatial:derivatives:1}. We establish \eqref{eq:Maradona:86}  inductively in $K$ for $N+M\leq K$. We know from \eqref{eq:omega:supp:subset} that \eqref{eq:Maradona:86} holds for $K=0$, i.e., for $N = M = 0$. So let us assume by induction that \eqref{eq:Maradona:86} was previously established for any pair $N'+M' \leq K-1$, and fix a new pair with $N+M = K$. Similarly to \eqref{eq:psi:m:i:q:Leibniz}, the Leibniz rule gives
\begin{align*}
&D^N D_{t,q}^M (g_{i,q,n,p}^2) - 2 g_{i,q,n,p} D^N D_{t,q}^M  g_{i,q,n,p} 
\notag\\
&= \sum_{\substack{0 \leq N' \leq N\\0\leq M' \leq M\\0<N'+M' < N+M}}   {N \choose N'} {M\choose M'} D^{N'} D_{t,q}^{M'} g_{i,q,n,p}\, D^{N-N'} D_{t,q}^{M - M'}  g_{i,q,n,p} \,.
\end{align*}
Since every term in the sum on the right side of the above display satisfies $1 \leq N'+M' \leq K-1$, these terms are bounded by our inductive assumption, and we deduce that 
\begin{align*}
{\mathbf{1}}_{\supp \psi_{i,q}} \abs{D^N D_{t,q}^M g_{i,q,n,p} }  
&\les \frac{ \abs{D^N D_{t,q}^M (g_{i,q,n,p}^2) }}{g_{i,q,n,p}} \notag\\
&+  \frac{\Gamma_{q+1}^{2(2j+2)} (\Gamma_{q+1} \lambda\qnp)^N \MM{M, \Nindt-\NcutSmall,\Gamma_{q+1}^{i-\cstarn+2} \tau_{q}^{-1},\Gamma_{q+1}^{-1} \Tilde{\tau}_{q}^{-1}}}{g_{i,q,n,p}}
\,.
\end{align*}
Thus, \eqref{eq:Maradona:86} also holds for $N+M =K$ by combining the above display with  \eqref{eq:omega:supp:subset} (which implies $g_{i,q,n,p} \geq \Gamma_{q+1}^{2j+2}$), and with estimate \eqref{eq:Maradona:90} (which gives the bounds for the derivatives of $g_{i,q,n,p}^2$). This concludes the proof of \eqref{eq:Maradona:86} and thus of the Lemma. 
\end{proof}

\subsubsection{\texorpdfstring{$L^r$}{Lr} norm of the stress cutoffs}

\begin{lemma}
\label{lem:omega:support}
Let $q \geq 0$. For $r\geq 1$ we have that
\begin{align}
\norm{\omega_{i,j,q,n,p}}_{L^r(\supp \psi_{i\pm,q})}  \lesssim  \Gamma_{q+1}^{-\sfrac{2j}{r}} 
\label{eq:omega:support}
\end{align}
holds for all $0\leq i \leq \imax$, $0 \leq j \leq\jmax$, $0\leq n \leq \nmax$, and $1\leq p \leq \pmax$. The implicit constant is independent of $i,j,q,n$ and $p$.
\end{lemma}
\begin{proof}[Proof of Lemma~\ref{lem:omega:support}]
The argument is similar to the proof of \eqref{eq:psi:i:q:support}. We begin with the case $r=1$. The other cases $r \in (1,\infty]$ follow from the fact that $\omega_{i,j,q,n,p}\leq 1$ and Lebesgue interpolation.

For $j=0$ we are done since by definition $0\leq \omega_{i,j,q,n,p} \leq 1$, thus we consider only $j\geq 1$. Since, $\psi_{i\pm2,q} \equiv 1$ on $\supp (\psi_{i\pm,q})$, and using Lemma~\ref{lem:D:Dt:Rn:sharp}, we see that for any $(x,t) \in \supp(\psi_{i\pm,q} \omega_{i,j,q,n,p})$ we have
\begin{align*}
\psi_{i\pm2,q}^2 g_{i,q,n,p}^2 
&= \psi_{i\pm2,q}^2 + \sum_{k=0}^{\NcutLarge}\sum_{m=0}^{\NcutSmall}  \frac{ |\psi_{i\pm2,q} D^k D_{t,q}^m \RR_{q,n,p}(x,t)|^2 }{\delta_{q+1,n,p}^{2} (\Gamma_{q+1} \lambda\qnp)^{2k}  (\Gamma_{q+1}^{i-\cstarn+2} \tau_{q}^{-1})^{2m}} \geq  \frac{1}{16} \Gamma_{q+1}^{4j}.
\end{align*}
Using that $a+b \geq \sqrt{a^2+b^2}$ for $a,b\geq 0$, and using $\Gamma_{q+1}^{4j} \geq 64$ for $j\geq 1$, we conclude  that
\begin{align*}
\sum_{k=0}^{\NcutLarge}\sum_{m=0}^{\NcutSmall}  \frac{ |\psi_{i\pm2,q} D^k D_{t,q}^m \RR_{q,n,p}(x,t)|  }{\delta_{q+1,n,p}  (\Gamma_{q+1} \lambda\qnp)^{k}   (\Gamma_{q+1}^{i-\cstarn+2} \tau_{q}^{-1})^{m}}  \geq \frac{1}{16} \Gamma_{q+1}^{2j}.
\end{align*}
Therefore, using Chebyshev's inequality and the inductive assumption \eqref{eq:Rn:inductive:assumption}, we obtain
\begin{align*}
&\left| \supp(\psi_{i\pm,q} \omega_{i,j,q,n,p}) \right| 
\notag\\
&\leq \left| \left\{ (x,t)\colon \psi_{i\pm2,q} g_{i,q,n,p} \geq (\sfrac{1}{16}) \Gamma_{q+1}^{2j} \right \}\right|
\notag\\
&\leq \left| \left\{ (x,t)\colon \sum_{k=0}^{\NcutLarge}\sum_{m=0}^{\NcutSmall} \frac{ |\psi_{i\pm2,q} D^k D_{t,q}^m \RR_{q,n,p}(x,t)|  }{\delta_{q+1,n,p}  (\Gamma_{q+1} \lambda\qnp)^{k}   (\Gamma_{q+1}^{i-\cstarn+2} \tau_{q}^{-1})^{m}} \geq (\sfrac{1}{16}) \Gamma_{q+1}^{2j} \right\}\right|
\notag\\
&\leq 16 \Gamma_{q+1}^{-2j} \sum_{k=0}^{\NcutLarge}\sum_{m=0}^{\NcutSmall}  \delta_{q+1,n,p}^{-1} (\Gamma_{q+1} \lambda\qnp)^{-k} (\Gamma_{q+1}^{i-\cstarn+2} \tau_{q}^{-1})^{-m} \norm{\psi_{i\pm2,q} D^k D_{t,q}^m \RR_{q,n,p}}_{L^1}
\notag\\
&\les 16 \Gamma_{q+1}^{-2j}\sum_{k=0}^{\NcutLarge}\sum_{m=0}^{\NcutSmall} \Gamma_{q+1}^{- k}
\notag\\
&\les \Gamma_{q+1}^{-2j}
\end{align*}
where the implicit constant only depends on $\NcutSmall$.
The proof is concluded since the $L^1$ norm of a function with range in $[0,1]$ is bounded by the measure of its support.
\end{proof}

\subsection{Definition and properties of the checkerboard cutoff functions}
\label{sec:cutoff:checkerboard:definitions}
For $0\leq n \leq \nmax$, consider all the $\frac{\T^3}{\lambda_{q,n,0}}$-periodic cells contained in $\mathbb{T}^3$, of which there are $\lambda_{q,n,0}^3$. Index these cells by integer triples $\vec{l}=(l,w,h)$ for $l,w,h\in\{0,...,\lambda_{q,n,0}-1\}$. Let $\mathcal{X}_{q,n,\vec{l}}$ be a partition of unity adapted to this checkerboard of periodic cells which satisfies
\begin{equation}\label{eq:checkerboard:useful:1}
\sum_{\vec{l}=(l,w,h)} \left(\mathcal{X}_{q,n,\vec{l}}\right)^2 = 1
\end{equation}
for any $q$ and $n$.  Furthermore, for $\vec{l}=(l,w,h),\vec{l}^*=(l^*,w^*,h^*)\in\{0,...,\lambda_{q,n,0}-1\}^3$ such that 
$$ |l-l^*| \geq 2, \qquad |w-w^*| \geq 2, \qquad |h-h^*| \geq 2,$$
we impose that
\begin{equation}\label{eq:checkerboard:useful:2}
\mathcal{X}_{q,n,\vec{l}} \mathcal{X}_{q,n,\vec{l}^*} = 0.
\end{equation}

\begin{definition}[\textbf{Checkerboard Cutoff Function}]\label{def:checkerboard}
Given $q$, $0\leq n\leq \nmax$, $i\leq \imax$, and $k\in\mathbb{Z}$, we define
\begin{equation}\label{eq:checkerboard:definition}
    \zeta_{q,i,k,n,\vec{l}}(x,t) = \mathcal{X}_{q,n,\vec{l}}\left(\Phi_{i,k,q}(x,t)\right).
\end{equation}
\end{definition}

\begin{lemma}\label{lem:checkerboard:estimates}
The cutoff functions $\left\{\zeta_{q,i,k,n,\vec{l}}\right\}_{\vec{l}}$ satisfy the following properties:
\begin{enumerate}[(1)]
    \item\label{item:check:1} The material derivative $\Dtq\left(\zeta_{q,i,k,n,\vec{l}}\right)$ vanishes.  
    \item\label{item:check:2} For each $t\in\mathbb{R}$ and all $x\in\mathbb{T}^3$, 
    \begin{equation}\label{eq:checkerboard:partition}
    \sum_{\vec{l}=(l,w,h)} \left(\zeta_{q,i,k,n,\vec{l}}(x,t)\right)^2 = 1.  
    \end{equation}
    \item\label{item:check:3} We have the spatial derivative estimate for all $m\leq\sfrac{3\Nfin}{2}+1$
    \begin{equation}\label{eq:checkerboard:derivatives}
        \left\| D^m \zeta_{q,i,k,n,\vec{l}} \right\|_{L^\infty\left(\supp \psi_{i,q}\tilde\chi_{i,k,q} \right)} \lesssim \lambda_{q,n,0}^m.
    \end{equation}
    \item\label{item:check:4} There exists an implicit dimensional constant independent of $q$, $n$, $k$, $i$, and $\vec{l}$ such that for all $(x,t)\in\supp\psi_{i,q}\tilde\chi_{i,k,q}$,
    \begin{equation}\label{eq:checkerboard:support}
        \textnormal{diam}\left(\supp\left(\zeta_{q,i,k,n,\vec{l}}(\cdot,t)\right)\right) \lesssim \left(\lambda_{q,n,0}\right)^{-1}.
    \end{equation}
\end{enumerate}
\end{lemma}

\begin{proof}[Proof of Lemma~\ref{lem:checkerboard:estimates}]
The proof of \eqref{item:check:1} is immediate given that $\zeta_{q,i,k,n,\vec{l}}$ is precomposed with the flow map $\Phi_{i,k,q}$. \eqref{eq:checkerboard:partition} follows from \eqref{item:check:1}, \eqref{eq:checkerboard:useful:1}, and the fact that for each $t\in\mathbb{R}$, $\Phi_{i,k,q}(t,\cdot)$ is a diffeomorphism of $\mathbb{T}^3$. The spatial derivative estimate in \eqref{eq:checkerboard:derivatives} follows from Lemma~\ref{lem:Faa:di:Bruno}, \eqref{eq:Lagrangian:Jacobian:2}, and the parameter definitions in \eqref{eq:tilde:lambda:q:def}, \eqref{eq:lambda:q:0:1:def}, and \eqref{def:lambda:q:n:p}. The property in \eqref{eq:checkerboard:support} follows from the construction of the $\mathcal{X}_{q,n,\vec{l}}$ functions (which can be taken simply as a dilation by a factor of $\lambda_{q,n,1}$ of a $q$-independent partition of unity on $\mathbb{R}^3$) and \eqref{eq:Lagrangian:Jacobian:1}.
\end{proof}

\subsection{Definition of the cumulative cutoff function}
\label{sec:cutoff:total:definitions}
Finally, combining the  cutoff functions defined in Definition~\ref{def:psi:i:q:def}, \eqref{eq:omega:cut:def}--\eqref{eq:omega:cut:def:0}, and \eqref{eq:chi:cut:def}, we define the cumulative cutoff function by
\begin{align*}
 {\eta_{i,j,k,q,n,p,\vec{l}}(x,t)=\psi_{i,q}(x,t) \omega_{i,j,q,n,p}(x,t) \chi_{i,k,q}(t)\overline{\chi}_{q,n,p}(t)\zeta_{q,i,k,n,\vec{l}}(x,t).}
\end{align*}
Since the values of $q$ and $n$ are clear from the context, the values in $\vec{l}$ are irrelevant in many arguments, and the time cutoffs $\overline{\chi}_{q,n,p}$ are only used in Section~\ref{ss:time:support}, we may abbreviate the above using any of
\begin{align*}
\eta_{i,j,k,q,n,p,\vec{l}} \, (x,t)=\eta_{i,j,k,q,n,p}(x,t)=\eta_{(i,j,k)}(x,t)  = \psi_{(i)}(x,t) \omega_{(i,j)}(x,t) \chi_{(i,k)}(t) \zeta_{(i,k)}(x,t).
\end{align*}
It follows from Lemma~\ref{lem:partition:of:unity:psi}, \eqref{eq:omega:cut:partition:unity}, \eqref{eq:chi:cut:partition:unity}, and \eqref{eq:checkerboard:partition} that for every ${(q,n,p)}$ fixed, we have a partition of unity
\begin{align}
 \sum_{i,j \geq 0} \sum_{k\in \Z} \sum_{\vec{l}} \eta_{i,j,k,q,n,p,\vec{l}}^2\,(x,t) = 1.
 \label{eq:eta:cut:partition:unity}
\end{align}
The sum in $i$ goes up to $i_{\mathrm{max}}$ (defined in \eqref{eq:imax:def}), while the sum in $j$ goes up to $j_{\mathrm{max}}$ (defined in \eqref{eq:j:max:def}).  In analogy with $\psi_{i \pm, q}$, we define
\begin{equation}\label{e:omega:plus:minus:definition}
    \omega_{(i,j\pm)}(x,t) := \left( \omega_{(i,j-1)}^2(x,t) + \omega_{(i,j)}^2(x,t) + \omega_{(i,j+1)}^2(x,t)  \right)^\frac{1}{2},
\end{equation}
which are cutoffs with the property that
\begin{equation}\label{e:omega:overlap}
    \omega_{(i,j\pm)} \equiv 1 \textnormal{ on } \supp{(\omega_{(i,j)})}.
\end{equation}
We then define  
\begin{equation}\label{e:eta:plus:minus:definition}
    \eta_{(i\pm,j\pm,k,\pm)}(x,t) := \psi_{i\pm,q}(x,t)\omega_{(i,j\pm)}(x,t)\tilde\chi_{i,k,q}(t)\zeta_{q,i,k,n,\vec{l}}(x,t),
\end{equation}
which are cutoffs with the property that
\begin{equation}\label{e:eta:overlap}
    \eta_{(i,\pm,j\pm,k\pm)} \equiv \zeta_{q,i,k,n,\vec{l}} \quad \textnormal{ on } \supp{\left( \psi_{(i)} \omega_{(i,j)}\chi_{(i,k)} \right)}.
\end{equation}
We conclude this section with estimates on the $L^p$ norms of the cumulative cutoff function $\eta_{(i,j,k)}$. 
\begin{lemma}\label{lemma:cumulative:cutoff:Lp}
For $r_1, r_2 \in [1,\infty]$ with $\frac{1}{r_1}+\frac{1}{r_2}=1$ we have
\begin{equation}
\sum_{\vec{l}} \left| \supp (\eta_{i,j,k,q,n,p,\vec{l}}) \right| 
\lessg 
\Gamma_{q+1}^{-2\left( \frac{i}{r_1} + \frac{j}{r_2} \right) + \frac{\CLebesgue}{r_1} +2 }  
\label{item:lebesgue:1} 
\end{equation}
\end{lemma}
\begin{proof}[Proof of Lemma~\ref{lemma:cumulative:cutoff:Lp}]
Applying Lemma~\ref{lem:psi:i:q:support}, Lemma~\ref{lem:omega:support}, H\"{o}lder's inequality, and interpolating yields
\begin{align*}
    \left| \supp(\psi_{i,q}) \cap \supp(\omega_{i,j,q,n,p}) \right| &\leq \left\| \psi_{i\pm,q} \omega_{(i,j\pm)} \right\|_{L^1}\\
    &\leq \left\| \psi_{i\pm,q} \right\|_{L^{r_1}} \left\| \omega_{(i,j\pm)} \right\|_{L^{r_2}}\\
    &\lessg \Gamma_{q+1}^{-\frac{2(i-1)-\CLebesgue}{r_1}-\frac{2(j-1)}{r_2}}.
\end{align*}
Using  $\frac{1}{r_1}+\frac{1}{r_2}=1$ and \eqref{eq:checkerboard:partition}, which gives that the $\zeta_{q,i,k,n,\vec{l}}$ form a partition of unity, yields \eqref{item:lebesgue:1}. 
\end{proof}

\section{From \texorpdfstring{$q$}{q} to \texorpdfstring{$q+1$}{qplusone}: breaking down the main inductive estimates}\label{section:statements}

The overarching goal of this section is to state several propositions which decompose the verification of the main inductive assumptions \eqref{eq:inductive:assumption:derivative:q} and \eqref{eq:perturbation:time:support} for the perturbation $w_{q+1}$ and \eqref{eq:Rq:inductive:assumption} for the stress $\RR_{q+1}$ into digestible components. We remind the reader, cf. Remark~\ref{remark:cutoffs:inductive}, that the rest of the inductive estimates stated in Section~\ref{sec:cutoff:inductive} are proven in Section~\ref{sec:cutoff}. We begin in Section~\ref{ss:induction:q} with Proposition~\ref{p:main:inductive:q}, which simply translates the main inductive assumptions into statements phrased at level $q+1$. At this point, we then introduce in Section~\ref{ss:notations:prep} a handful of notations which will be necessary in order to state the propositions which form the constituent parts of the proof of Proposition~\ref{p:main:inductive:q}. The next three propositions (\ref{p:inductive:n:1}, \ref{p:inductive:n:2}, and \ref{p:inductive:n:3}) are described and presented in Section~\ref{ss:induction:nn}.  They are significantly more detailed than Proposition~\ref{p:main:inductive:q}, as they contain the precise estimates that will be propagated throughout the construction and cancellation of the higher order stresses $\RR\qnn$.  These three propositions will be verified in Section~\ref{s:stress:estimates}.

\subsection{Induction on \texorpdfstring{$q$}{q}}\label{ss:induction:q}
The main claim of this section is an induction on $q$.

\begin{proposition}[\textbf{Inductive Step on $q$}]\label{p:main:inductive:q}
Given $\vlq$, $\RR_{\ell_q}$, and $\RR_{q}^{\textnormal{comm}}$ satisfying the Euler-Reynolds system 
\begin{subequations}\label{e:eulerreynolds:inductive:q}
\begin{align}
\partial_t \vlq + \div(\vlq\otimes\vlq) +\nabla p_{\ell_q} &= \div \RR_{\ell_q} + \div \RR_{q}^{\textnormal{comm}} \\
\div \vlq &= 0
\end{align}
\end{subequations}
with $\vlq$, $\RR_{\ell_q}$, and $\RR_{q}^{\textnormal{comm}}$ satisfying the conclusions of Lemma~\ref{lem:mollifying:ER} in addition to \eqref{eq:inductive:assumption:derivative}-\eqref{eq:nasty:Dt:wq:WEAK:old}, there exist $v_{q+1}=\vlq+w_{q+1}$ and $\RR_{q+1}$ which satisfy the following:
\begin{enumerate}[(1)]
\item $v_{q+1}$ and $\RR_{q+1}$ solve the Euler-Reynolds system
\begin{subequations}\label{e:eulerreynolds2:inductive:q}
\begin{align}
\partial_t v_{q+1} + \div(v_{q+1}\otimes v_{q+1}) +\nabla p_{q+1} &= \RR_{q+1} \\
\div v_{q+1} &= 0.
\end{align}
\end{subequations}
\item For all $k,m \leq 7\Nindv$,
\begin{equation}\label{e:main:inductive:q:velocity}
 \left\| \psi_{i,q} D^k \Dtq^m w_{q+1} \right\|_{L^2} \leq \Gamma_{q+1}^{-1}\delta_{q+1}^\frac{1}{2}\lambda_{q+1}^k \MM{m,\Nindt, \tau_q^{-1}\Gamma_{q+1}^{i-1}, \tilde{\tau}_q^{-1}\Gamma_{q+1}^{-1}}.
\end{equation}
Furthermore, we have that
\begin{align}
\supp_t (\RR_{q}) \subset [T_1,T_2] \quad \Rightarrow \quad \supp_t (w_{q+1})\subset 
\left[T_1 - (\lambda_{q} \delta_{q}^{\sfrac 12})^{-1},T_2 + (\lambda_{q} \delta_{q}^{\sfrac 12})^{-1} \right] 
\,.
\label{eq:perturbation:time:support:redux:0}
\end{align}
\item For all $k,m \leq 3\Nindv$, 
\begin{equation}\label{e:main:inductive:q:stress}
\left\| \psi_{i,q} D^k \Dtq^m \RR_{q+1} \right\|_{L^1} \leq \shaqqplusone  \delta_{q+2} \lambda_{q+1}^k \MM{m,\Nindt,\Gamma_{q+1}^{i+1} \tau_q^{-1},\Gamma_{q+1}^{-1}\tilde\tau_q^{-1}}
\end{equation}
\end{enumerate}
\end{proposition}

\begin{remark}
In achieving the conclusions \eqref{e:eulerreynolds2:inductive:q}, \eqref{e:main:inductive:q:velocity}, and \eqref{e:main:inductive:q:stress}, we have verified the inductive assumptions \eqref{eq:inductive:assumption:derivative:q}-\eqref{eq:Rq:inductive:assumption} at level $q+1$. The inductive assumption \eqref{eq:inductive:assumption:derivative} at levels $q'<q+1$ follows from Lemma~\eqref{lem:mollifying:ER}.  The proof of Proposition~\ref{p:main:inductive:q} will entail many estimates which are much more detailed than \eqref{e:main:inductive:q:velocity} and \eqref{e:main:inductive:q:stress}, but for the time being we record only the basic estimates which are direct translations of \eqref{eq:inductive:assumption:derivative:q}-\eqref{eq:Rq:inductive:assumption} at level $q+1$.
\end{remark}

\subsection{Notations}\label{ss:notations:prep}
The proof of Proposition~\ref{p:main:inductive:q} will be achieved through an induction with respect to $\tilde{n}$, where $0\leq \nn \leq \nmax$ corresponds to the addition of the perturbation $\displaystyle w_{q+1,\nn}=\sum\limits_{\pp=1}^{\pmax} w\qplusnnpp$.  The addition of each perturbation $w_{q+1,\nn}$ will move the minimum effective frequency present in the stress terms to $\lambda_{q,\nn+1,0}$.  This induction on $\nn$ requires three sub-propositions; the base case $\nn=0$, the inductive step from $\nn-1$ to $\nn$ for $\nn\leq\nmax-1$, and the final step from $\nmax-1$ to $\nmax$. Throughout these propositions, we shall employ the following notations.
\begin{enumerate}[(1)]
    \item \label{item:notation:1} $\boldsymbol{\nn}$ - An integer taking values $0\leq \nn \leq \nmax$ over which induction is performed. At every step in the induction, we add another component $w_{q+1,\nn}$ of the final perturbation 
    $$w_{q+1}=\sum\limits_{\nn=0}^\nmax \sum\limits_{\pp=1}^\pmax w\qplusnnpp.$$
    We emphasize that the use of $\nn$ at various points in statements and estimates means that we are \emph{currently} working on the inductive step at level $\nn$.
    \item \label{item:notation:2} $\boldsymbol{n}$ - An integer taking values $1\leq n \leq \nmax$ which correspond to the higher order stresses $\RR_{q,n}$.  Occasionally, we shall use the notation $\RR_{q,0}=\RR_{\ell_q}$ to streamline an argument.  We emphasize that $n$ will be used at various points in statements and estimates to reference \textit{higher order} objects in addition to those at level $\nn$, and so will satisfy the inequality $\nn\leq n$.  
    \item \label{item:notation:3} $\boldsymbol{{\HH}_{q,n,p}^{n'}}$ - The component of $\RR_{q,n,p}$ originating from an error term produced by the addition of $w_{q+1,n'}$.  The parameter $n'$ will always be a \emph{subsidiary} parameter used to reference objects created at or \emph{below} the level $\nn$ that we are currently working on, and so will satisfy $n'\leq \nn$.
    \item \label{item:notation:4} $\boldsymbol{\LPqnp}$ - We use the spatial Littlewood-Paley projectors $\LPqnp$ defined by
    \begin{align}
    \LPqnp=\begin{cases} \mathbb{P}_{\geq\lambda_{q,\nmax,\pmax}}& \mbox{if }n=\nmax, p=\pmax+1 \\
    \mathbb{P}_{\left[\lambda\qnpminus,\lambda\qnp\right)} &\mbox{if }1\leq n \leq \nmax, 1\leq p \leq \pmax
    \label{def:LPqnp}
    \end{cases}
    \end{align}
where $\Proj_{[\lambda_1,\lambda_2)}$ is defined in Section~\ref{sec:mollifiers:Fourier}  as $\Proj_{\geq \lambda_1} \Proj_{< \lambda_2}$. 
Note that for $n=\nmax$ and $p=\pmax+1$, $\mathbb{P}_{[q,\nmax,\pmax+1]}$ projects onto \emph{all} frequencies larger than $\lambda_{q,\nmax,\pmax}=\lambda_{q,\nmax+1,0}$. Errors which include the frequency projector $\mathbb{P}_{[q,\nmax,\pmax+1]}$ will be small enough to be absorbed into $\RR_{q+1}$. 

We shall frequently utilize sums of Littlewood-Paley projectors $\LPqnp$ to decompose products of intermittent pipe flows periodized to scale $\lambda\qnn^{-1}$.  These sums will be written in terms of three parameters - $n$, $p$, and $\nn$.  
As a consequence of \eqref{def:LPqnp}, \eqref{def:lambda:q:n:p},  \eqref{eq:rqn:perp:definition},  and \eqref{eq:lambda:q:n:def}, we have that $\lambda_{q,\nn+1,0}\leq \lambda\qnn$ for $0\leq \nn\leq\nmax$, so that
    \begin{equation}\label{e:Pqnp:identity:prep}
       \left(\Psum\right)   \mathbb{P}_{\geq\lambda_{q,\nn}} = \mathbb{P}_{\geq \lambda_{q,\nn+1,0}} \mathbb{P}_{\geq \lambda\qnn} = \mathbb{P}_{\geq\lambda_{q,\nn}}.
    \end{equation}
A consequence of \eqref{e:Pqnp:identity:prep} is that for $\frac{\mathbb{T}^3}{\lambda\qnn}$-periodic functions\footnote{We note that in the second equality in \eqref{e:Pqnp:identity}, such functions do not have active frequencies in between $\lambda_{q,\nn+1,0}$ and $\lambda\qnn$.} where $0\leq\nn\leq\nmax$, 
    \begin{align}
    f &= \dashint_{\mathbb{T}^3} f + \mathbb{P}_{\geq \lambda\qnn} f \notag\\
    &= \dashint_{\mathbb{T}^3} f + \Pqnn   \left(\Psum\right) f.  \label{e:Pqnp:identity}
    \end{align}
These equalities will be useful in the calculations in Section~\ref{ss:stress:error:identification}, and we will recall their significance when we estimate the Type 1 errors in Section~\ref{ss:stress:oscillation:1}.
    \item \label{item:notation:5} $\boldsymbol{\RR_{q+1}^\nn}$ - Any stress term which satisfies the estimates required of $\RR_{q+1}$ and which has already been estimated at the $\nn^{th}$ stage of the induction; that is, error terms arising from the addition of $w_{q+1,n'}$ for $n'\leq \nn$.  We \emph{exclude} $\RR_{q}^{\textnormal{comm}}$ from $\RR_{q+1}^\nn$, only absorbing it at the very end when we define $\RR_{q+1}$. Thus
    \begin{equation}\label{RR:q+1:n-1:to:n}
    \RR_{q+1}^{\nn+1} = \RR_{q+1}^{\nn} + \left(\textnormal{errors coming from }w_{q+1,\nn+1}\textnormal{ that also go into }\RR_{q+1}\right) \, .
    \end{equation}
\end{enumerate}

\subsection{Induction on \texorpdfstring{$\tilde{n}$}{tilden}}\label{ss:induction:nn}
The first proposition asserts that there exists a perturbation $w_{q+1,0}$ which we add to $\vlq$ so that $v_{q,0}:=\vlq+w_{q+1,0}$ satisfies the following.  First, $v_{q,0}$ solves the Euler-Reynolds system with a righthand side consisting of stresses $\RR_{q+1}^0$ and $\HH\qnp^0$ which belong respectively to $\RR_{q+1}$ and $\RR\qnp$ for $1\leq n \leq \nmax$ and $1\leq p \leq\pmax$. Secondly, $w_{q+1,0}$ satisfies estimates which in particular imply the inductive assumptions required of the velocity perturbation $w_{q+1}$ in \eqref{e:main:inductive:q:velocity}.\footnote{This is checked in Remark~\ref{rem:checking:inductive:velocity}.} Thirdly, $\RR_{q+1}^0$ satisfies the estimates required of $\RR_{q+1}$ in the inductive assumption \eqref{eq:Rn:inductive:assumption} (with an extra factor of smallness).  Finally, each $\HH\qnp^0$ satisfies the inductive assumptions required of $\RR\qnp$ in \eqref{eq:Rn:inductive:assumption}. 

\begin{proposition}[\textbf{Induction on $\nn$: The Base Case $\nn=0$}]\label{p:inductive:n:1}
Under the assumptions of Proposition~\ref{p:main:inductive:q} (equivalently the conclusions of Lemma~\ref{lem:mollifying:ER}), there exist $\displaystyle w_{q+1,0}=\sum\limits_{\pp=1}^\pmax w_{q+1,0,p}=w_{q+1,0,1}$, $\RR_{q+1}^0$, and $\HH\qnp^0$ for $1\leq n \leq \nmax$ and $1\leq p \leq\pmax$ such that the following hold.
\begin{enumerate}[(1)]
\item $v_{q,0}:=\vlq+w_{q+1,0}$ solves
\begin{subequations}\label{e:inductive:n:1:eulerreynolds}
\begin{align}
\partial_t v_{q,0} + \div(v_{q,0}\otimes v_{q,0}) +\nabla p_{q,0}  
&= \div\left(\RR_{q+1}^0\right) + \div\left( \sum\limits_{n=1}^{\nmax}\sum\limits_{p=1}^{\pmax} \HH\qnp^0 \right) + \div\RR_{q}^{\textnormal{comm}} \\
\div v_{q,0} &= 0.
\end{align}
\end{subequations}
\item For all $k+m \leq \NN{\textnormal{fin},0}-\NcutSmall-\NcutLarge-2\Ndec-9$ and $1\leq\pp\leq\pmax$ (although only $w_{q+1,0,1}$ is non-zero)
\begin{align}\label{e:inductive:n:1:velocity}
 \left\| D^k \Dtq^m w_{q+1,0,\pp} \right\|_{L^2\left(\supp\psi_{i,q}\right)} \lesssim \delta_{q+1,0,\pp}^\frac{1}{2} \Gamma_{q+1}^{3+\frac{\CLebesgue}{2}} \lambda_{q+1}^k \MM{m,\Nindt, \tau_q^{-1}\Gamma_{q+1}^{i-\cstar+4}, \tilde{\tau}_q^{-1}\Gamma_{q+1}^{-1}}.
\end{align}
Furthermore, we have that \begin{align}
\supp_t (\RR_{q}) \subset [T_1,T_2] \quad \Rightarrow \quad \supp_t (w_{q+1,0,\pp})\subset 
\left[T_1 - (\lambda_{q} \delta_{q}^{\sfrac 12}\Gamma_{q+1})^{-1},T_2 + (\lambda_{q} \delta_{q}^{\sfrac 12}\Gamma_{q+1})^{-1} \right] 
\,.
\label{eq:perturbation:time:support:redux:nn=0}
\end{align}
\item For all $k,m \leq 3\Nindv$, 
\begin{equation}\label{e:inductive:n:1:Rstress}
\left\| \psi_{i,q} D^k \Dtq^m \RR_{q+1}^0 \right\|_{L^1} \lesssim \shaqqplusone\Gamma_{q+1}^{-1} \delta_{q+2} \lambda_{q+1}^k \MM{m,\Nindt, \tau_q^{-1}\Gamma_{q+1}^{i+1}, \tilde{\tau}_q^{-1}\Gamma_{q+1}^{-1}}.
\end{equation}
Furthermore, we have that 
\begin{equation}\label{eq:Rqplus:time:0}
 \supp_t \RR_{q+1}^0 \subseteq \supp_t w_{q+1,0}  \, .  
\end{equation}
\item For all $k+m\leq \Nfn$ and $1\leq n \leq \nmax$, $1\leq p \leq\pmax$,
\begin{equation}\label{e:inductive:n:1:Hstress}
\left\| D^k \Dtq^m \HH\qnp^0 \right\|_{L^1\left(\supp\psi_{i,q}\right)} \lesssim \delta_{q+1,n,p} \lambda\qnp^k \MM{m,\Nindt, \tau_q^{-1}\Gamma_{q+1}^{i-\cstarn}, \tilde{\tau}_q^{-1}\Gamma_{q+1}^{-1}}.
\end{equation}
Furthermore, we have that 
\begin{equation}\label{eq:Hqnp:time:0}
 \supp_t \HH_{q,n,p}^0 \subseteq \supp_t w_{q+1,0}  \, .  
\end{equation}
\end{enumerate}
\end{proposition}

The second proposition assumes that perturbations $w_{q+1,n'}$ have been added for $n'\leq \nn-1$ while satisfying four criteria.  First, $v_{q,\nn-1}=\vlq+\sum\limits_{n'\leq \nn-1}w_{q+1,n'}$ solves an Euler-Reynolds system with stresses $\RR_{q+1}^{\nn-1}$ and $\HH\qnp^{n'}$. Secondly, the perturbations $w_{q+1,n'}$ satisfy the inductive assumptions required of $w_{q+1}$ in \eqref{e:main:inductive:q:velocity} for $n'\leq\nn-1$. Thirdly, $\RR_{q+1}^{\nn-1}$ satisfies the inductive assumption \eqref{e:main:inductive:q:stress} at level $q+1$.  Finally, $\HH\qnp^{n'}$ satisfies the assumption \eqref{eq:Rn:inductive:assumption} in the parameter regime $\nn\leq n\leq\nmax$, $n'\leq \nn-1$, $1\leq p \leq\pmax$. The conclusion of the proposition replaces each $\nn-1$ in the assumptions with $\nn$.

\begin{proposition}[\textbf{Induction on $\nn$: From $\nn-1$ to $\nn$ for $1\leq\nn\leq\nmax-1$}]\label{p:inductive:n:2}
Let $1\leq \nn \leq\nmax-1$ be given, and let 
$$v_{q,\nn-1}:=\vlq+\sum\limits_{n'=0}^{\nn-1}w_{q+1,n'}=\vlq + \sum\limits_{n'=0}^{\nn-1}\sum\limits_{p'=1}^\pmax w_{q+1,n',p'},$$
$\RR_{q+1}^{\nn-1}$, and $\HH\qnp^{n'}$ be given for $n'\leq \nn-1$, $\nn\leq n\leq\nmax$ and $1\leq p,p'\leq\pmax$ such that the following are satisfied.
\begin{enumerate}[(1)]
\item ${v_{q,\nn-1}}$ solves:
\begin{subequations}\label{e:inductive:n:2:eulerreynolds}
\begin{align}
\partial_t v_{q,\nn-1} &+ \div(v_{q,\nn-1}\otimes v_{q,\nn-1}) +\nabla p_{q,\nn-1} \notag \\ 
\qquad\qquad\qquad &= \div\left(\RR_{q+1}^{\nn-1}\right) + \div\left( \displaystyle\sum\limits_{n=\nn}^{\nmax}\sum\limits_{p=1}^\pmax\sum\limits_{n'=0}^{\nn-1} \HH\qnp^{n'} \right) + \div\RR_{q}^{\textnormal{comm}} \\
\div v_{q,\nn-1} &= 0 \,.
\end{align}
\end{subequations}
\item For all $k+m \leq \NN{\textnormal{fin},\textnormal{n}'}-\NcutSmall-\NcutLarge-2\Ndec-9$, $n'\leq\nn-1$, and $1\leq p'\leq\pmax$,
\begin{equation}\label{e:inductive:n:2:velocity}
 \left\| D^k \Dtq^m w_{q+1,n',p'} \right\|_{L^2\left(\supp\psi_{i,q}\right)} \lesssim \delta_{q+1,n',p'}^\frac{1}{2} \Gamma_{q+1}^{3+\frac{\CLebesgue}{2}}  \lambda_{q+1}^k \MM{m,\Nindt, \tau_q^{-1}\Gamma_{q+1}^{i-\cstarnprime+4}, \tilde{\tau}_q^{-1}\Gamma_{q+1}^{-1}}.
\end{equation}
Furthermore, we have that
\begin{align}
\supp_t (\RR_{q,n',p'}) &\subset [T_{1,n',p'},T_{2,n',p'}] \notag 
\\ \quad &\Rightarrow \supp_t  (w_{q+1,n',p'})\subset 
\left[T_{1,n',p'} - (\lambda_{q} \delta_{q}^{\sfrac 12}\Gamma_{q+1})^{-1},T_{2,n',p'} + (\lambda_{q} \delta_{q}^{\sfrac 12}\Gamma_{q+1})^{-1} \right] 
\,.
\label{eq:perturbation:time:support:redux:nn}
\end{align}
\item For all $k,m \leq 3\Nindv$,
\begin{equation}\label{e:inductive:n:2:Rstress}
\left\| \psi_{i,q} D^k \Dtq^m \RR_{q+1}^{\nn-1} \right\|_{L^1} \lesssim \shaqqplusone\Gamma_{q+1}^{-1} \delta_{q+2} \lambda_{q+1}^k \MM{m,\Nindt,\Gamma_{q+1}^{i+1} \tau_q^{-1},\Gamma_{q+1}^{-1}\tilde\tau_q^{-1}}.
\end{equation}
Furthermore, we have that
\begin{equation}\label{eq:Rqplus:time:nn}
 \supp_t \RR_{q+1}^{\nn-1} \subseteq \bigcup_{n'\leq \nn-1}\supp_t w_{q+1,n'}  \, .  
\end{equation}
\item For all $k+m\leq\Nfn$, $\nn\leq n \leq \nmax$, $n'\leq\nn-1$, and $1\leq p \leq\pmax$,
\begin{equation}\label{e:inductive:n:2:Hstress}
\left\| D^k \Dtq^m \HH\qnp^{n'} \right\|_{L^1\left(\supp\psi_{i,q}\right)} \lesssim \delta_{q+1,n,p} \lambda\qnp^k \MM{m,\Nindt, \tau_q^{-1}\Gamma_{q+1}^{i-\cstarn}, \tilde{\tau}_q^{-1}\Gamma_{q+1}^{-1}}.
\end{equation}
Furthermore, we have that
\begin{equation}\label{eq:Hqnp:time:nn}
 \supp_t \HH_{q,n,p}^{n'} \subseteq \supp_t w_{q+1,n'}  \, .  
\end{equation}
\end{enumerate}
Then there exists $w_{q+1,\nn}$ such that (1)-(4) are satisfied with $\nn-1$ replaced with $\nn$.
\end{proposition}

The final proposition considers the case $\tilde{n}=\nmax$ and shows that, under assumptions analogous to those in Proposition~\ref{p:inductive:n:2}, there exists $w_{q+1,\nmax}$ such that all remaining errors after the addition of $w_{q+1,\nmax}$ can be absorbed into $\RR_{q+1}$, thus verifying the conclusions of Proposition~\ref{p:main:inductive:q}.

\begin{proposition}[\textbf{Induction on $\nn$: The Final Case $\nn=\nmax$}]\label{p:inductive:n:3}
Let 
$$v_{q,\nmax-1}:=\vlq+\sum\limits_{n'=0}^{\nmax-1}w_{q+1,n'}=\vlq + \sum_{n'=0}^{\nmax-1}\sum_{p'=1}^\pmax w_{q+1,n',p'} $$
$\RR_{q+1}^{\nmax-1}$, and $\HH_{q,\nmax,p}^{n'}$ be given for $n'\leq \nmax-1$ and $1\leq p,p' \leq\pmax$ such that the following are satisfied.
\begin{enumerate}[(1)]
\item ${v_{q,\nmax-1}}$ solves:
\begin{subequations}\label{e:inductive:n:3:eulerreynolds}
\begin{align}
\partial_t v_{q,\nmax-1} &+ \div(v_{q,\nmax-1}\otimes v_{q,\nmax-1}) +\nabla p_{q,\nmax-1} \notag \\ 
&= \div\left(\RR_{q+1}^{\nmax-1}\right) + \div\left(\displaystyle\sum\limits_{n'=0}^{\nmax-1} \sum\limits_{p=1}^\pmax \HH^{n'}_{q,\nmax,p} \right) + \div\RR_{q}^{\textnormal{comm}} \\
\div v_{q,\nmax-1} &= 0 \, .
\end{align}
\end{subequations}
\item For all $k+m\leq \NN{\textnormal{fin},\textnormal{n}'}-\NcutSmall-\NcutLarge-2\Ndec-9$, $n'\leq \nmax-1$, and $1\leq p'\leq\pmax$,
\begin{equation}\label{e:inductive:n:3:velocity}
 \left\| D^k \Dtq^m w_{q+1,n',p'} \right\|_{L^2\left(\supp\psi_{i,q}\right)} \lesssim \delta_{q+1,n',p'}^\frac{1}{2}  \Gamma_{q+1}^{3+\frac{\CLebesgue}{2}}  \lambda_{q+1}^k \MM{m,\Nindt, \tau_q^{-1}\Gamma_{q+1}^{i-\cstarnprime+4}, \tilde{\tau}_q^{-1}\Gamma_{q+1}^{-1}}  . 
\end{equation}
Furthermore, we have that
\begin{align}
\supp_t (\RR_{q,n',p'}) &\subset [T_{1,n',p'},T_{2,n',p'}] \notag \\ 
&\Rightarrow \quad \supp_t  (w_{q+1,n',p'})\subset 
\left[T_{1,n',p'} - (\lambda_{q} \delta_{q}^{\sfrac 12}\Gamma_{q+1})^{-1},T_{2,n',p'} + (\lambda_{q} \delta_{q}^{\sfrac 12}\Gamma_{q+1})^{-1} \right] 
\,.
\label{eq:perturbation:time:support:redux:nn=nmax}
\end{align}
\item For all $k,m \leq 3\Nindv$, 
\begin{equation}\label{e:inductive:n:3:Rstress}
\left\| \psi_{i,q} D^k \Dtq^m \RR_{q+1}^{\nmax-1} \right\|_{L^1} \lesssim \shaqqplusone\Gamma_{q+1}^{-1} \delta_{q+2} \lambda_{q+1}^k\MM{m,\Nindt,\Gamma_{q+1}^{i+1} \tau_q^{-1},\Gamma_{q+1}^{-1}\tilde\tau_q^{-1}}.
\end{equation}
Furthermore, we have that
\begin{equation}\label{eq:Rqplus:time:nmax}
 \supp_t \RR_{q+1}^{\nmax-1} \subseteq \bigcup_{n'\leq \nmax-1}\supp_t w_{q+1,n'}  \, .  
\end{equation}
\item For all $k+m\leq\NN{\textnormal{fin},\textnormal{n}_{\textnormal{max}}}$, $n'\leq\nmax-1$, and $1\leq p \leq\pmax$
\begin{align}
\left\| D^k \Dtq^m \HH_{q,\nmax,p}^{n'} \right\|_{L^1\left(\supp\psi_{i,q} \right)} &\lesssim \delta_{q+1,\nmax,p} \lambda_{q,\nmax,p} \MM{m,\Nindt, \tau_q^{-1}\Gamma_{q+1}^{i-\cstarnmax}, \tilde{\tau}_q^{-1}\Gamma_{q+1}^{-1}}.\label{e:inductive:n:3:Hstress}
\end{align}
Furthermore, we have that
\begin{equation}\label{eq:Hqnp:time:nnmax}
 \supp_t \HH_{q,n,p}^{n'} \subseteq \supp_t w_{q+1,n'}  \, .  
\end{equation}
\end{enumerate}
Then there exists $w_{q+1,\nmax}$ and $\RR_{q+1}$ such that $v_{q+1}:=v_{q,\nmax-1}+w_{q+1,\nmax}$ and $\RR_{q+1}$ satisfy conclusions \eqref{e:eulerreynolds2:inductive:q}, \eqref{e:main:inductive:q:velocity}, \eqref{eq:perturbation:time:support:redux:0}, and \eqref{e:main:inductive:q:stress} from Proposition~\ref{p:main:inductive:q}.
\end{proposition}

\section{Proving the main inductive estimates}\label{s:stress:estimates}

Because the proofs of Propositions~\ref{p:inductive:n:1}, \ref{p:inductive:n:2}, and \ref{p:inductive:n:3} will be comprised of multiple arguments with many similarities, we divide up the proofs of the Propositions into sections corresponding to these arguments.\footnote{This organization of proof avoids having to alternate between the definitions of $w_{q+1,\nn,\pp}$ and $\RR_{q,\nn,\pp}$ for all $1\leq \nn \leq \nmax$ and $1\leq\pp\leq\pmax$. We judge that it is wiser to define all the perturbations simultaneously under the assumptions of Propositions~\ref{p:inductive:n:1}, \ref{p:inductive:n:2}, and \ref{p:inductive:n:3}.  Namely, we assume that each $\RR_{q,\nn,\pp}$ exists and satisfies the enumerated properties, some of which may not be verified until later.}  First, we define $\RR_{q,\nn,\pp}$ and $w\qplusnnpp$ in Section~\ref{ss:stress:definition} for each $0\leq\nn\leq\nmax$ and $1\leq \pp\leq\pmax$.  Then, Section~\ref{ss:stress:w:estimates} collects estimates on $w\qplusnnpp$, thus verifying \eqref{e:inductive:n:1:velocity} and \eqref{eq:perturbation:time:support:redux:nn=0}, \eqref{e:inductive:n:2:velocity} and \eqref{eq:perturbation:time:support:redux:nn}, and \eqref{e:inductive:n:3:velocity} and \eqref{eq:perturbation:time:support:redux:nn=nmax} at levels $\nn=0$, $1\leq\nn\leq\nmax-1$, and $\nn=\nmax$, respectively. Next, in Section~\ref{ss:stress:error:identification} we separate out the different types of error terms and write down the Euler-Reynolds system satisfied by $v_{q,\nn}$, which verifies \eqref{e:inductive:n:1:eulerreynolds}, \eqref{e:inductive:n:2:eulerreynolds}, and \eqref{e:inductive:n:3:eulerreynolds}, again at the respective values of $\nn$.

The error estimates are then divided into five sections.  We first estimate the transport and Nash errors in Sections~\ref{ss:stress:transport} and \ref{ss:stress:Nash}. The next section estimates the Type 1 oscillation errors (notated with $\HH\qnp^\nn$), which are obtained via Littlewood-Paley projectors $\LPqnp$.  In the parameter regime $1\leq n\leq\nmax$ and $1\leq p\leq\pmax$, Type 1 oscillation errors will satisfy the estimates \eqref{e:inductive:n:1:Hstress}, \eqref{e:inductive:n:2:Hstress}, and \eqref{e:inductive:n:3:Hstress} at respective parameter values $\nn=0$, $1\leq\nn\leq\nmax-1$, and $\nn=\nmax$.  Type 1 oscillation errors obtained from $\mathbb{P}_{[q,\nmax,\pmax+1]}$ have a sufficiently high minimum frequency (from \eqref{def:LPqnp} specifically $\lambda_{q,\nmax+1,0}$, which by a large choice of $\nmax$ is very close to $\lambda_{q+1}$) to be absorbed into $\RR_{q+1}$.  Then in Section~\ref{ss:stress:oscillation:2}, we use Proposition~\ref{prop:disjoint:support:simple:alternate} to show that on the support of a checkerboard cutoff function, Type 2 oscillation errors vanish. The divergence corrector errors are estimated in Sections~\ref{ss:stress:divergence:correctors}. The divergence corrector, Nash, and transport errors will always be absorbed into $\RR_{q+1}$ and thus must again satisfy one of \eqref{e:inductive:n:1:Rstress}, \eqref{e:inductive:n:2:Rstress}, or \eqref{e:inductive:n:3:Rstress}.  Finally, the conclusions \eqref{eq:perturbation:time:support:redux:nn=0}, \eqref{eq:Rqplus:time:0}, \eqref{eq:Hqnp:time:0}, \eqref{eq:perturbation:time:support:redux:nn}, \eqref{eq:Rqplus:time:nn}, \eqref{eq:Hqnp:time:nn}, \eqref{eq:perturbation:time:support:redux:nn=nmax}, \eqref{eq:Rqplus:time:nmax}, and \eqref{eq:Hqnp:time:nnmax},   concerning the time support will be verified in Section~\ref{ss:time:support}.

\subsection{Definition of \texorpdfstring{$\RR\qnnpp$}{rqnp} and \texorpdfstring{$w\qplusnnpp$}{wqnp}}\label{ss:stress:definition}

In this section we construct the perturbations $w\qplusnn$. Before doing so, we recall the significance of each parameter used to define the perturbations.
\begin{enumerate}[(a)]
    \item  $\xi$ is the vector direction of the axis of the pipe
    \item $i$ quantifies the amplitude of the velocity field $\vlq$ along which the pipe will be flowed
    \item  $j$ quantifies the amplitude of the Reynolds stress
    \item  $k$ describes which time cut-off $\chi_{i,k,q}$ is active
    \item $q+1$ is the stage of the overall convex integration scheme
    \item $\nn$ and $\pp$ signify which higher order stress $\RR_{q,\nn,\pp}$ is being corrected, and $\nn$ also denotes the intermittency parameter $\rqnperptilde$
    \item $\vecl=(l,w,h)$ is used to index the checkerboard cutoff functions.  Recall that the admissible values of $l$, $w$, and $h$ range from $0$ to $\lambda_{q,\nn,0}-1$ and thus depend on $\nn$.
\end{enumerate}

\subsubsection{The case \texorpdfstring{$\nn=0$}{nequalszero}}
To define $\displaystyle w_{q+1,0}=\sum_{\pp=1}^\pmax w_{q+1,0,p} = w_{q+1,0,1}$, we recall the notation $\RR_{\ell_q}=\RR_{q,0}$ and set
\begin{equation}\label{eq:Rq0j}
R_{q,0,1,j,i,k} = \nabla\Phi_{(i,k)} \left(\delta_{q+1,0,1}\Gamma^{2j+4}_{q+1}\Id- \mathring{R}_{q,0}\right) \nabla\Phi_{(i,k)}^T.
\end{equation}
For $\pp\geq 2$, we set $R_{q,0,\pp,j,i,k}=0$. Fix values of $i$, $j$, and $k$. Let $\xi\in\Xi$ be a vector from Proposition~\ref{p:split}. For all $\xi\in\Xi$, we define the coefficient function $a_{\xi,i,j,k,q,0,\pp,\vecl}$ by
\begin{equation}
a_{\xi,i,j,k,q,0,\pp,\vecl}:=a_{\xi,i,j,k,q,0,\pp}:=a_{(\xi)}=\delta_{q+1,0,\pp}^{\sfrac 12}\Gamma^{j+2}_{q+1}\eta_{i,j,k,q,0,\pp,\vecl} \gamma_{\xi}\left(\frac{R_{q,0,\pp,j,i,k}}{\delta_{q+1,0,\pp}\Gamma^{2j+4}_{q+1}}\right) \, .
\label{eq:a:xi:def}
\end{equation}
From Lemma~\ref{lem:D:Dt:Rn:sharp}, we see that on the support of $\eta_{(i,j,k)}$ we have $|\RR_{q,0,\pp}| \leq \Gamma_{q+1}^{2j+2} \delta_{q+1,0,\pp}$, and thus by estimate \eqref{eq:Lagrangian:Jacobian:1} from Corollary~\ref{cor:deformation}, for $\pp=1$ we have that
$$  \left| \frac{R_{q,0,\pp,j,i,k}}{\delta_{q+1,0,\pp}\Gamma^{2j+4}_{q+1}} - \Id \right| \leq \Gamma_{q+1}^{-1} < \frac 12  $$
once  $\lambda_0$ is sufficiently large. Thus we may apply Proposition~\ref{p:split}.

The coefficient function $a_{(\xi)}$ is then multiplied by an intermittent pipe flow
$$ \nabla \Phi_{(i,k)}^{-1}  \WW_{\xi,q+1,0} \circ \Phi_{(i,k)},  $$
where we have used the objects defined in Proposition~\ref{pipeconstruction} and the shorthand notation
\begin{align} 
\WW_{\xi,q+1,0} = \WW^{(i,j,k,0,\vecl)}_{\xi,q+1,0} = \WW_{\xi,q+1,0}^s = \WW^s_{\xi,\lambda_{q+1},r_{q+1,0}}. 
\label{eq:W:xi:q+1:0:def}
\end{align}
The superscript $s=(i,j,k,0,\vecl)$ indicates the placement of the intermittent pipe flow $\WW^{i,j,k,0,p,\vecl}_{\xi,q+1,0}$ (cf. \eqref{item:pipe:2} from Proposition~\ref{pipeconstruction}), which depends on $i$, $j$, $k$, $\nn=0$, and $\vecl$ and is only relevant in Section~\ref{ss:stress:oscillation:2}.\footnote{Note that for $\pp\geq 2$, $\delta_{q+1,0,\pp}=0$, so there is no need for the placement to depend on $\pp$ in this case, as $w_{q+1,0,\pp}$ will uniformly vanish.}  To ease notation, we will suppress the superscript except in Section~\ref{ss:stress:oscillation:2}. Furthermore, item \ref{item:pipe:1} from Proposition~\ref{pipeconstruction} gives that 
$$  \nabla \Phi_{(i,k)}^{-1}  \WW_{\xi,q+1,0} \circ \Phi_{(i,k)} = \curl \left( \nabla\Phi_{(i,k)}^T \mathbb{U}_{\xi,q+1,0} \circ \Phi_{(i,k)} \right). $$
We can now write the principal part of the first term of the perturbation as
\begin{equation}\label{wqplusoneonep}
    w_{q+1,0}^{(p)} = \sum_{i,j,k,\pp}\sum_{\vec{l}}\sum_{\xi} a_{(\xi)} \curl \left( \nabla\Phi_{(i,k)}^T \mathbb{U}_{\xi,q+1,0} \circ \Phi_{(i,k)} \right): = \sum_{i,j,k,\pp}\sum_{\vec{l}}\sum_{\xi} w_{(\xi)}.
\end{equation}
The notation $w_{(\xi)}$ implicitly encodes all indices and thus will be a useful shorthand for the principal part of the perturbation. To make the perturbation divergence free, we add
\begin{equation}\label{wqplusoneonec}
    w_{q+1,0}^{(c)} = \sum_{i,j,k,\pp}\sum_{\vec{l}}\sum_{\xi} \nabla a_{(\xi)} \times \left( \nabla\Phi_{(i,k)}^T \mathbb{U}_{\xi,q+1,0}\circ \Phi_{(i,k)} \right) = \sum_{i,j,k,\pp}\sum_{\vec{l}}\sum_\xi w_{(\xi)}^{(c)}
\end{equation}
so that
\begin{equation}\label{wqplusoneone}
    w_{q+1,0} = w_{q+1,0}^{(p)} + w_{q+1,0}^{(c)} = \sum_{i,j,k,\pp}\sum_{\vec{l}}\sum_{\xi} \curl \left( a_{(\xi)} \nabla\Phi_{(i,k)}^T \mathbb{U}_{\xi,q+1,0} \circ \Phi_{(i,k)} \right)
\end{equation}
is divergence-free and mean-zero.

\subsubsection{The case \texorpdfstring{$1\leq\nn\leq\nmax$}{onenmax}}
With $w_{q+1,0}$ constructed, we construct $\displaystyle w\qplusnn = \sum_{\pp=1}^\pmax w\qplusnnpp$ for $1\leq \nn \leq \nmax$. For $1\leq \pp\leq\pmax$, we define
\begin{equation}\label{e:rqnp:definition}
    \RR\qnnpp = \sum_{n'\leq \nn-1} \HH\qnnpp^{n'}.
\end{equation}
With this definition in hand, we set
\begin{equation}\label{eq:rqnpj}
R_{q,\nn,\pp,j,i,k}=\nabla\Phi_{(i,k)}\left(\delta_{q+1,\nn,\pp}\Gamma^{2j+4}_{q+1}\Id - \mathring{R}\qnnpp\right)\nabla\Phi_{(i,k)}^T,
\end{equation}
and define the coefficient function $a_{\xi,i,j,k,q,\nn,\pp,\vecl}$ by
\begin{equation}
a_{\xi,i,j,k,q,\nn,\pp,\vecl}=a_{\xi,i,j,k,q,\nn,\pp}=a_{(\xi)}=\delta_{q+1,\nn,\pp}^{\sfrac 12}\Gamma^{j+2}_{q+1}\eta_{i,j,k,q,\nn,\pp,\vecl} \gamma_\xi \left(\frac{R_{q,\nn,\pp,j,i,k}}{\delta_{q+1,\nn,\pp}\Gamma^{2j+4}_{q+1}}\right) \, .
\label{eq:a:oxi:def}
\end{equation}
By Lemma~\ref{lem:D:Dt:Rn:sharp} and Corollary~\ref{cor:deformation} as before, $R_{q,\nn,\pp,j,i,k}/(\delta_{q+1,\nn,\pp}\Gamma^{2j+4}_{q+1})$ lies in the domain of $\gamma_\xi$, as soon as $\lambda_0$ is sufficiently large (similarly to the display below \eqref{eq:a:xi:def}). The coefficient function is multiplied by an intermittent pipe flow
$$ \nabla\Phi_{(i,k)}^{-1} \WW_{\xi,q+1,\nn}\circ \Phi_{(i,k)} = \curl \left( \nabla\Phi_{(i,k)}^T \mathbb{U}_{\xi,q+1,\nn}\circ \Phi_{(i,k)} \right), $$
where we have used the shorthand notation
\begin{align}
\WW_{\xi,q+1,\nn} = \WW^{i,j,k,\nn,\pp,\vecl}_{\xi,q+1,\nn} = \WW_{\xi,q+1,\nn}^s = \WW_{\xi,\lambda_{q+1},\rqnperptilde}^s \,.
\label{eq:W:xi:q+1:nn:def}
\end{align}
As before, the superscript $s=(i,j,k,\nn,\pp,\vecl)$ refers to the placement of the pipe, depends on $i$, $j$, $k$, $\nn$, $\pp$, and $\vecl$, and will be chosen in Section~\ref{ss:stress:oscillation:2}. Thus the principal part of the perturbation is defined by
\begin{equation}\label{wqplusonenpp}
    w_{q+1,\nn,\pp}^{(p)} = \sum_{i,j,k}\sum_{\vecl}\sum_{\xi} a_{(\xi)} \curl \left( \nabla\Phi_{(i,k)}^T \mathbb{U}_{\xi,q+1,\nn}\circ \Phi_{(i,k)} \right)= \sum_{i,j,k}\sum_{\vecl}\sum_\xi w_{(\xi)}.
\end{equation}
As before, we add a corrector
\begin{equation}\label{wqplusonenpc}
    w_{q+1,\nn,\pp}^{(c)} = \sum_{i,j,k}\sum_{\vecl}\sum_{\xi} \nabla a_{(\xi)} \times \left( \nabla \Phi_{(i,k)}^T \mathbb{U}_{\xi,q+1,\nn} \circ \Phi_{(i,k)} \right)= \sum_{i,j,k}\sum_{\vecl}\sum_\xi w_{(\xi)}^{(c)},
\end{equation}
producing the divergence-free perturbation
\begin{align}
    w_{q+1,\nn}  = \sum_{\pp=1}^\pmax w\qplusnnpp     
    &= \sum_{\pp=1}^\pmax \left(w_{q+1,\nn,\pp}^{(p)} + w_{q+1,\nn,\pp}^{(c)} \right) \notag\\
    &=  \sum_{i,j,k,\pp}\sum_{\vecl}\sum_{\xi} \curl \left( a_{(\xi)} \nabla \Phi_{(i,k)}^T \mathbb{U}_{\xi,q+1,\nn} \circ \Phi_{(i,k)} \right)\,.\label{wqplusonenp}
\end{align}

\subsection{Estimates for \texorpdfstring{$w\qplusnnpp$}{wqn}}\label{ss:stress:w:estimates}

In this section, we verify \eqref{e:inductive:n:1:velocity}, \eqref{e:inductive:n:2:velocity}, and \eqref{e:inductive:n:3:velocity}. We first estimate the $L^r$ norms of the coefficient functions $a_{(\xi)}$. We have consolidated the proofs for each value of $\nn$ into the following lemma.
\begin{lemma}
\label{lem:a_master_est_p}
For  $N+M \leq \Nfnn-\NcutSmall-\NcutLarge-4$, $r\geq 1$, and $r_1,r_2\in[1,\infty]$ with $\frac{1}{r_1}+\frac{1}{r_2}=1$, we have
\begin{align}
&\norm{D^ND_{t,q}^M a_{\xi,i,j,k,q,\nn,\pp,\vec{l}}}_{L^r} 
\notag\\
&\qquad \lessg 
\bigl |\supp(\eta_{i,j,k,q,\nn,\pp,\vecl}) \bigr|^{\frac{1}{r}} \delta_{q+1,\nn,\pp}^{\sfrac 12}\Gamma^{j+2}_{q+1}
\left(\Gamma_{q+1}\lambda\qnnpp\right)^{N}
\MM{M, \NindSmall, \tau_{q}^{-1}\Gamma_{q+1}^{i-\cstarnn+3}, \tilde\tau_{q}^{-1}\Gamma_{q+1}^{-1}}
\label{e:a_master_est_p}.
\end{align}
\end{lemma}

\begin{proof}[Proof of Lemma~\ref{lem:a_master_est_p}]
We begin by considering the case $r=\infty$. The general case $r\geq 1$ will then follow from the size of the support of  $a_{(\xi)}$.
Recalling estimate \eqref{eq:D:Dt:Rn:sharp}, we have that for all $N+M\leq\Nfnn-4$,
\begin{align}
\norm{D^N D_{t,q}^M\mathring{R}_{q,\nn,\pp} }_{L^{\infty}(\supp \eta_{(i,j,k)})}
&\lessg \delta_{q+1,\nn,\pp}\Gamma^{2j+2}_{q+1}
 \left(\Gamma_{q+1}\lambda\qnnpp\right)^N
\MM{M, \NindSmall, \tau_{q}^{-1}\Gamma_{q+1}^{i-\cstarnn+2}, \tilde\tau_{q}^{-1}\Gamma_{q+1}^{-1} }.\nonumber
\end{align}
From Corollary~\ref{cor:deformation}, we have that for all $N+M\leq \sfrac{3\Nfin}{2}$,
\begin{align*}
\left\| D^N D_{t,q}^M D \Phi_{(i,k)} \right\|_{L^\infty(\supp(\psi_{i,q}\chi_{i,k,q}))} &\leq \tilde{\lambda}_q^{N} \MM{M,\NindSmall,\Gamma_{q+1}^{i-\cstar} \tau_q^{-1},\tilde{\tau}_q^{-1}\Gamma_{q+1}^{-1}}.
\end{align*}
Thus from the Leibniz rule and the definitions \eqref{eq:rqnpj}, \eqref{eq:Rq0j}, for $N+M\leq\Nfnn-4$,
\begin{align}
&\norm{D^N D_{t,q}^M R_{q,\nn,\pp,j,i,k} }_{L^{\infty}(\supp \eta_{(i,j,k)})}
\notag\\
&\qquad \lessg \delta_{q+1,\nn,\pp}\Gamma^{2j+4}_{q+1}\left(\Gamma_{q+1}\lambda\qnnpp\right)^N \MM{M, \NindSmall, \tau_{q}^{-1}\Gamma_{q+1}^{i-\cstarnn+2}, \tilde\tau_{q}^{-1}\Gamma_{q+1}^{-1}}\label{eq:davidc:1}
\,.\end{align}
The above estimates allow us to apply Lemma \ref{lem:Faa:di:Bruno:*} with $N=N'$, $M=M'$ so that $N+M\leq\Nfnn-4$, $\psi = \gamma_{\xi,}$ (which is allowable since by Proposition~\ref{p:split} we have that  $D^B \gamma_{\xi}$ is bounded uniformly with respect to $q$, and we have checked in Section~\ref{ss:stress:definition} that the argument of $\gamma_{\xi}$ remains strictly within a ball of radius $\varepsilon$ of the identity), $\Gamma_\psi=1$, $v = \vlq$, $D_t = D_{t,q}$, $h(x,t) = R_{q,\nn,\pp,j,i,k}(x,t)$, $C_h = \delta_{q+1,\nn,\pp}\Gamma_{q+1}^{2j+4} = \Gamma^2$, $\lambda=\tilde\lambda = \lambda\qnnpp\Gamma_{q+1}$, $\mu = \tau_{q}^{-1} \Gamma_{q+1}^{i-\cstarnn+2}$, $\tilde \mu = \tilde \tau_{q}^{-1}\Gamma_{q+1}^{-1}$, and $N_t=\Nindt$. We obtain that for all $N+M\leq\Nfnn-4$,
\begin{align*}
\norm{ D^ND_{t,q}^M \gamma_{\xi}\left(\frac{R_{q,\nn,\pp,j,i,k}}{\delta_{q+1,\nn,\pp}\Gamma^{2j+4}_{q+1}}\right)}_{L^{\infty}(\supp \eta_{(i,j,k)})} &\lesssim \left(\Gamma_{q+1}\lambda\qnnpp\right)^N \MM{M, \NindSmall, \tau_{q}^{-1}\Gamma_{q+1}^{i-\cstarnn+2}, \tilde\tau_{q}^{-1}\Gamma_{q+1}^{-1}} \,.
\end{align*}
From the above bound, definitions \eqref{eq:a:xi:def} and \eqref{eq:a:oxi:def}, the Leibniz rule, estimates \eqref{eq:nasty:Dt:psi:i:q:orangutan:redux}, \eqref{eq:chi:cut:dt}, \eqref{eq:D:Dt:omega:sharp-ish}, and Lemma~\ref{lem:checkerboard:estimates}, we obtain that for $N+M\leq\Nfnn-\NcutSmall-\NcutLarge-4$,\footnote{The limit on the number of derivatives comes from \eqref{eq:D:Dt:omega:sharp-ish} and \eqref{eq:davidc:1}.  The sharp cost of a material derivative comes from \eqref{eq:D:Dt:omega:sharp-ish}.}
\begin{align*}
&\norm{D^ND_{t,q}^M a_{(\xi)}}_{L^{\infty}(\supp \eta_{(i,j,k)})}\notag
\\
&\lessg  \delta_{q+1,\nn,\pp}^{\sfrac 12}\Gamma^{j+2}_{q+1}\sum_{\substack{N'+N''=N, \\M'+M''=M}} \norm{D^{N'}D_{t,q}^{M'}\eta_{(i,j,k)}}_{L^{\infty}} \norm{ D^{N''}D_{t,q}^{M''} \gamma_{\xi}\left(\frac{R_{q,\nn,\pp,j,i,k}}{\delta_{q+1,\nn,\pp}\Gamma^{2j+4}_{q+1}}\right)
}_{L^{\infty}(\supp \eta_{(i,j,k)})}\notag
\\
&\lessg \delta_{q+1,\nn,\pp}^{\sfrac 12}\Gamma^{j+2}_{q+1}\sum_{\substack{N'+N''=N, \\M'+M''=M}} \left(\Gamma_{q+1}\lambda\qnnpp\right)^{N'} \MM{M', \NindSmall, \tau_{q}^{-1}\Gamma_{q+1}^{i-\cstarnn+3}, \tilde\tau_{q}^{-1}\Gamma_{q+1}^{-1}} \notag\\
&\qquad \qquad \qquad \qquad \qquad \times \left(\Gamma_{q+1}\lambda\qnnpp\right)^{N''} \MM{M'', \NindSmall, \tau_{q}^{-1}\Gamma_{q+1}^{i-\cstarnn+2}, \tilde\tau_{q}^{-1}\Gamma_{q+1}^{-1}}
\notag\\
&\lessg \delta_{q+1,\nn,\pp}^{\sfrac 12}\Gamma^{j+2}_{q+1}
\left(\Gamma_{q+1}\lambda\qnnpp\right)^{N}
\MM{M, \NindSmall, \tau_{q}^{-1}\Gamma_{q+1}^{i-\cstarnn+3}, \tilde\tau_{q}^{-1}\Gamma_{q+1}^{-1}} \, .
\end{align*}
This concludes the proof of \eqref{e:a_master_est_p} when $r = \infty$. 
For general $r\geq 1$, we just note that $\supp(a_{(\xi)}) \subseteq \supp(\eta_{i,j,k,q,n,p,\vec{l}})$. 
\end{proof}

An immediate consequence of Lemma~\ref{lem:a_master_est_p} is that we have estimates for the velocity increments themselves. These are summarized in the following corollary.
\begin{corollary}
\label{cor:corrections:Lp}
For $N+M\leq \Nfnn-\NcutSmall-\NcutLarge-2\Ndec-8$  we have the estimate
\begin{align}
\norm{D^ND_{t,q}^M w_{(\xi)}}_{L^r} 
&\lessg\bigl | \supp(\eta_{i,j,k,q,\nn,\pp,\vecl}) \bigr|^{\frac{1}{r}} \delta_{q+1,\nn,\pp}^{\sfrac 12}\Gamma^{j+2}_{q+1} {\left(\rqnperptilde\right)}^{\sfrac2r-1} \notag\\
&\qquad \qquad \qquad \times\lambda_{q+1}^N  \MM{M, \NindSmall, \tau_{q}^{-1}\Gamma_{q+1}^{i-\cstarnn+3}, \tilde\tau_{q}^{-1}\Gamma_{q+1}^{-1}} \, .
\label{eq:w:oxi:est}
\end{align}
For $N+M\leq \Nfnn-\NcutSmall-\NcutLarge-2\Ndec-9$ and $(r,r_1,r_2)\in\left\{(1,2,2),(2,\infty,1)\right\}$, we have the estimates
\begin{align}
\norm{D^ND_{t,q}^M w_{(\xi)}^{(c)}}_{L^r} &\lessg
\frac{\Gamma_{q+1} \lambda\qnnpp}{\lambda_{q+1}}
\bigl | \supp(\eta_{i,j,k,q,\nn,\pp,\vecl}) \bigr|^{\frac{1}{r}} \delta_{q+1,\nn,\pp}^{\sfrac 12}\Gamma^{j+2}_{q+1}
{\left(r_{q+1,\nn}\right)}^{\sfrac2r-1}  \notag\\
&\qquad \qquad \qquad \times\lambda_{q+1}^N \MM{M, \NindSmall, \tau_{q}^{-1}\Gamma_{q+1}^{i-\cstarnn+3}, \tilde\tau_{q}^{-1}\Gamma_{q+1}^{-1}}
\label{eq:w:oxi:c:est} \\
\norm{D^N D_{t,q}^M w_{q+1,\nn,\pp}}_{L^r\left(\supp \psi_{i,q}\right)}
&\lesssim \delta_{q+1,\nn,\pp}^{\sfrac 12}\Gamma^{\frac{-2i+\CLebesgue}{r_1r}+2+\frac{2}{r}}_{q+1}  {\left(r_{q+1,\nn}\right)}^{\sfrac2r-1} \notag\\
&\qquad \qquad \qquad \times\lambda_{q+1}^N \MM{M, \NindSmall, \tau_{q}^{-1}\Gamma_{q+1}^{i-\cstarnn+4}, \tilde\tau_{q}^{-1}\Gamma_{q+1}^{-1}} \, .
\label{e:w_est}
\end{align}
Finally, we have that
\begin{align}
\supp_t (\RR_{q}) \subset [T_1,T_2] \quad \Rightarrow \quad \supp_t (w_{q+1,\nn,\pp})\subset 
\left[T_1 - (\lambda_{q} \delta_{q}^{\sfrac 12})^{-1},T_2 + (\lambda_{q} \delta_{q}^{\sfrac 12})^{-1} \right] 
\,.
\label{eq:perturbation:time:support:redux:proof}
\end{align}
\end{corollary}
\begin{remark}\label{rem:checking:inductive:velocity}
By choosing $r=2$, $r_2=1$, and $r_1=\infty$ in \eqref{e:w_est} and recalling that \eqref{eq:delta:q:nn:pp:ineq} and \eqref{eq:Nfinn:inequality} give that
$$  \delta_{q+1,\nn,\pp}^{\sfrac{1}{2}} \leq \Gamma_{q+1}^{-2} \delta_{q+1}^{\sfrac{1}{2}}, \qquad \Nfnn - \NcutSmall- \NcutLarge - 2\Ndec - 9 \geq 14 \NindLarge,  $$
we may sum over $\nn$ and $\pp$ in \eqref{e:w_est} and use the extra negative factor of $\Gamma_{q+1}$ to absorb any implicit constants. Finally, from \eqref{eq:cstarn:inequality}, we have that the cost of a sharp material derivative in \eqref{e:w_est} is sufficient to meet the bounds in \eqref{e:main:inductive:q:velocity}. Then we have verified \eqref{e:inductive:n:1:velocity}, \eqref{e:inductive:n:2:velocity}, and \eqref{e:inductive:n:3:velocity} at levels $\nn=0$, $1\leq \nn < \nmax$, and $\nn=\nmax$, respectively, and \eqref{e:main:inductive:q:velocity}.
\end{remark}

\begin{proof}[Proof of Corollary~\ref{cor:corrections:Lp}]
Recalling the definition of $w_{(\xi)}$ from \eqref{wqplusoneonep} and \eqref{wqplusonenp}, we aim to prove the first estimate by applying Remark~\ref{rem:slow:fast}, with $f=a_{(\xi)} \nabla\Phi_{(i,k)}^{-1}$, $\const_f=\bigl | \supp(\eta_{i,j,k,q,\nn,\pp,\vecl}) \bigr|^{\frac{1}{r}} \delta_{q+1,\nn,\pp}^{\sfrac 12}\Gamma^{j+2}_{q+1}$, $\Phi=\Phi_{(i,k)}$, $v=\vlq$, $\lambda=\Gamma_{q+1}\lambda\qnnpp$, $\zeta=\tilde\zeta=\lambda_{q+1}$, $\const_\varphi=r_{q+1,\nn}^{\sfrac2r-1}$, $\mu=\lambda\qnn=\lambda_{q+1}\rqnperptilde$, $\nu=\tau_q^{-1}\Gamma_{q+1}^{i-\cstarnn+3}$, $\tilde{\nu}=\tilde{\tau}_q^{-1}\Gamma_{q+1}^{-1}$, $g=\mathbb W_{\xi,q+1,\nn}$, $N_t=\NindSmall$, and $N_\circ=\Nfnn-\NcutSmall-\NcutLarge-4$. From \eqref{e:a_master_est_p} and Corollary~\ref{cor:deformation}, we have that for $N+M \leq \Nfnn-\NcutSmall-\NcutLarge-4$,
\begin{align}
&\norm{D^ND_{t,q}^M a_{(\xi)}}_{L^r} 
\lessg \bigl | \supp(\eta_{(i,j,k)})\bigr|^{\frac{1}{r}} \delta_{q+1,\nn,\pp}^{\sfrac 12}\Gamma^{j+2}_{q+1} \left(\Gamma_{q+1}\lambda\qnnpp\right)^N  \MM{M, \NindSmall, \tau_{q}^{-1}\Gamma_{q+1}^{i-\cstarnn+3}, \tilde\tau_{q}^{-1}\Gamma_{q+1}^{-1}}\label{e:a_master_est_p_repeat}\\
&\left\| D^N D_{t,q}^M (D \Phi_{(i,k)})^{-1} \right\|_{L^\infty(\supp(\psi_{i,q}\tilde\chi_{i,k,q}))} 
\leq \tilde{\lambda}_q^{N} \MM{M,\NindSmall,\Gamma_{q+1}^{i-\cstar} \tau_q^{-1},\tilde{\tau}_q^{-1}\Gamma_{q+1}^{-1}},
\label{eq:D:N+1:Phi}\\
&{\norm{D^N\Phi_{(i,k)} }_{L^\infty(\supp(\psi_{i,q}\tilde\chi_{i,k,q} ))} }
\lesssim \Gamma_{q+1}^{-1} \tilde \lambda_q^{N-1}\label{eq:Lagrangian:Jacobian:2:repeat}\\
&\norm{D^N\Phi^{-1}_{(i,k)} }_{L^\infty(\supp(\psi_{i,q}\tilde\chi_{i,k,q} ))} 
\lesssim \Gamma_{q+1}^{-1} \tilde \lambda_q^{N-1} \label{eq:Lagrangian:Jacobian:7:repeat}
\end{align}
showing that \eqref{eq:slow_fast_0}, \eqref{eq:slow_fast_1}, and \eqref{eq:slow_fast_2} are satisfied. Recall that $\mathbb W_{\xi,q+1,\nn}$ is periodic to scale
$$\displaystyle{\lambda\qnn^{-1}=\left(\lambda_{q+1}\rqnperptilde\right)^{-1}= \left(\lambda_q^{\ff^{\nn+1}}\lambda_{q+1}^{1-\ff^{\nn+1}}\right)^{-1}}.$$
By \eqref{eq:lambdaqn:identity:2} and \eqref{eq:lambdaqn:identity:3}, we have that for all $q$, $\nn$, and $\pp$,
\begin{align}
\lambda_{q+1}^4 \leq \left( \frac{\lambda\qnn}{2\pi\sqrt{3}\Gamma_{q+1}\lambda\qnnpp} \right)^{\Ndec}, \qquad 2\Ndec + 4 \leq \Nfnn-\NcutSmall-\NcutLarge-5 \label{eq:lambdaqn:identity:3:repeat}
\end{align}
and so the assumptions \eqref{eq:slow_fast_3} and \eqref{eq:slow_fast_3_a} from Lemma A.5 are satisfied.   From the estimates in Proposition~\ref{pipeconstruction}, we have that \eqref{eq:slow_fast_4} is satisfied with $\zeta=\tilde\zeta=\lambda_{q+1}$. We may thus apply Lemma \ref{l:slow_fast}, Remark~\ref{rem:slow:fast} to obtain that for both choices of $(r,r_1,r_2)$ and $N+M \leq \Nfnn-\NcutSmall-\NcutLarge-2\Ndec-8$,
\begin{align*}
&\norm{D^N\left(D_{t,q}^M \left(a_{(\xi)} \nabla \Phi_{(i,k)}^{-1}\right) \mathbb  W_{\xi,q+1,\nn}\circ \Phi_{(i,k)}
\right)}_{L^r}\notag\\
&  \lessg
\sum_{m=0}^{N} \bigl | \supp(\eta_{(i,j,k)}) \bigr|^{\frac{1}{r}} \delta_{q+1,\nn,\pp}^{\sfrac 12}\Gamma^{j+2}_{q+1}  \left(\Gamma_{q+1}\lambda\qnnpp\right)^{N-m} \MM{M, \NindSmall, \tau_{q}^{-1}\Gamma_{q+1}^{i-\cstarnn+3}, \tilde\tau_{q}^{-1}\Gamma_{q+1}^{-1}}  
\norm{D^{m} \mathbb  W_{\xi,q+1,\nn}}_{L^r}
\notag \\
& \lessg 
\sum_{m=0}^{N} \bigl | \supp(\eta_{(i,j,k)}) \bigr|^{\frac{1}{r}} \delta_{q+1,\nn,\pp}^{\sfrac 12}\Gamma^{j+2}_{q+1} 
 \left(\Gamma_{q+1}\lambda\qnnpp\right)^{N-m}\MM{M, \NindSmall, \tau_{q}^{-1}\Gamma_{q+1}^{i-\cstarnn+3}, \tilde\tau_{q}^{-1}\Gamma_{q+1}^{-1}}  \lambda_{q+1}^m \left(\rqnperp\right)^{\sfrac2r-1}
\notag \\
&  \lessg \bigl | \supp(\eta_{(i,j,k)}) \bigr|^{\frac{1}{r}} \delta_{q+1,\nn,\pp}^{\sfrac 12}\Gamma^{j+2}_{q+1}
\MM{M, \NindSmall, \tau_{q}^{-1}\Gamma_{q+1}^{i-\cstarnn+3}, \tilde\tau_{q}^{-1}\Gamma_{q+1}^{-1}} \lambda_{q+1}^N \left(\rqnperp\right)^{\sfrac2r-1}
\,.
\end{align*}
Here we have used that  ${\lambda_{q+1} \geq \Gamma_{q+1}\lambda\qnnpp}$ for all $0\leq n \leq \nmax$ and $1\leq \pp\leq\pmax$, and thus we have proven \eqref{eq:w:oxi:est}.

The argument for the corrector is similar, save for the fact that $D_{t,q}$ will land on $\nabla a_{(\xi)}$, and so we require an extra commutator estimate from Lemma~\ref{lem:cooper:2}, specifically Remark~\ref{rem:cooper:2:sum}. Note that $D_{t,q} \Phi_{(i,k)} = 0$ gives
\begin{align*}
D_{t,q}^M w_{(\xi)}^{(c)} &= D_{t,q}^M\left( \nabla a_{(\xi)} \times \left( \nabla \Phi_{(i,k)}^T \mathbb{U}_{\xi,q+1,\nn} \circ \Phi_{(i,k)} \right) \right) \\
&= \sum_{M'+M''=M} c(M',M) \left( D_{t,q}^{M'} \nabla a_{(\xi)} \right) \times \left( \left( D_{t,q}^{M''} \nabla \Phi_{(i,k)}^T\right) \mathbb{U}_{\xi,q+1,\nn} \circ \Phi_{(i,k)} \right). 
\end{align*}
Using \eqref{eq:nasty:D:vq} and \eqref{e:a_master_est_p_repeat} shows that \eqref{eq:cooper:2:v} and \eqref{eq:cooper:2:f} are satisfied with $f=\nabla a_{(\xi)}$, 
$$\const_f=\bigl | \supp(\eta_{i,j,k,q,\nn,\pp,\vecl}) \bigr|^{\frac{1}{r}} \delta_{q+1,\nn,\pp}^{\sfrac 12}\Gamma^{j+2}_{q+1} \Gamma_{q+1}\lambda_{q,\nn,\pp},$$
 $\const_v=\delta_{q}^\frac{1}{2}\Gamma_{q+1}^{i+1}$, $\lambda_v=\tilde\lambda_v=\tilde\lambda_q$, $\mu_v=\Gamma_{q+1}^{i-\cstar}\tau_q^{-1}$, $N_t=\Nindt$, $\tilde{\mu}_v=\tilde{\tau}_q^{-1}\Gamma_{q+1}^{-1}$, $\lambda_f=\tilde{\lambda}_f=\Gamma_{q+1}\lambda\qnnpp$, $\mu_f=\tau_q^{-1}\Gamma_{q+1}^{i-\cstarnn+3}$, and $\tilde{\mu}_f=\tilde{\tau}_q^{-1}\Gamma_{q+1}^{-1}$.   Applying Lemma~\ref{lem:cooper:2} (estimate \eqref{eq:cooper:2:f:2}) as before, we obtain that for $N+M\leq \Nfnn-\NcutSmall-\NcutLarge-5$,
\begin{align}
\norm{D^{N}D_{t,q}^M \nabla a_{(\xi)}}_{L^{r}}
\lessg \bigl | \supp(\eta_{(i,j,k)}) \bigr|^{\frac{1}{r}} \delta_{q+1,\nn,\pp}^{\sfrac 12}\Gamma^{j+2}_{q+1}(\Gamma_{q+1} \lambda\qnnpp)^{N+1}
\MM{M, \NindSmall, \tau_{q}^{-1}\Gamma_{q+1}^{i-\cstarnn+3}, \tilde\tau_{q}^{-1}\Gamma_{q+1}^{-1}}
\, . \label{eq:nabla:a:est}
\end{align}
In view of \eqref{eq:D:N+1:Phi} and \eqref{eq:lambdaqn:identity:3:repeat}, we may apply Lemma~\ref{l:slow_fast}, Remark~\ref{rem:slow:fast}, and the estimates from Proposition~\ref{pipeconstruction} to obtain that for $N+M\leq \Nfnn-\NcutSmall-\NcutLarge-2\Ndec-9$
\begin{align}
&\norm{D^N  D_{t,q}^M \left( \nabla a_{(\xi)} \times \left( \nabla \Phi_{(i,k)}^T \mathbb{U}_{\xi,q+1,\nn} \circ \Phi_{(i,k)} \right) \right) }_{L^r}
\notag\\
&\ \lessg \sum_{m=0}^{N} \bigl | \supp(\eta_{(i,j,k)}) \bigr|^{\frac{1}{r}} \delta_{q+1,\nn,\pp}^{\sfrac 12}\Gamma^{j+2}_{q+1} \Gamma_{q+1}
\lambda\qnnpp\lambda\qnnpp^{N-m}
\MM{M, \NindSmall, \tau_{q}^{-1}\Gamma_{q+1}^{i-\cstarnn+3}, \tilde\tau_{q}^{-1}\Gamma_{q+1}^{-1}}
 \norm{D^m \UU_{\xi,q+1,\nn}}_{L^r}
\notag \\
& \lessg \lambda_{q+1}^{m-1} \sum_{m=0}^{N} \bigl | \supp(\eta_{(i,j,k)}) \bigr|^{\frac{1}{r}} \delta_{q+1,\nn,\pp}^{\sfrac 12}\Gamma^{j+2}_{q+1} \Gamma_{q+1}
 \lambda\qnnpp\lambda\qnnpp^{N-m}
\MM{M, \NindSmall, \tau_{q}^{-1}\Gamma_{q+1}^{i-\cstarnn+3}, \tilde\tau_{q}^{-1}\Gamma_{q+1}^{-1}} \left(\rqnperp\right)^{\frac2r-1}
\notag \\
&\lessg 
\frac{\Gamma_{q+1}\lambda\qnnpp}{\lambda_{q+1}} \lambda_{q+1}^{N} \bigl | \supp(\eta_{(i,j,k)}) \bigr|^{\frac{1}{r}} \delta_{q+1,\nn,\pp}^{\sfrac 12}\Gamma^{j+2}_{q+1}
\MM{M, \NindSmall, \tau_{q}^{-1}\Gamma_{q+1}^{i-\cstarnn+3}, \tilde\tau_{q}^{-1}\Gamma_{q+1}^{-1}}
\left(\rqnperp\right)^{\sfrac2r-1} \,,
\label{eq:w:oxi:c:est:temp:1}
\end{align}
proving \eqref{eq:w:oxi:c:est}.

The final estimate \eqref{e:w_est} follows from the first two after recalling that $\psi_{i,q}$ may overlap with $\psi_{i+1,q}$, so that on the support of $\psi_{i,q}$, we will have to appeal to \eqref{e:a_master_est_p} at level $i+1$.  Then, we sum over $\vec{l}$ and appeal to the bound \eqref{item:lebesgue:1}.
Next, we may sum on $j$, index which we recall from Lemma~\ref{lem:maximal:j} is bounded independently of $q$, and $\pp$, $k$. The powers of $\Gamma_{q+1}^j$ cancel out since $r r_2 = 1$. Next, we sum over $\pp$, which is bounded independently of $q$, and recall that the parameter $k$, although not bounded independently of $q$, corresponds to a partition of unity, so that  the number of cutoff functions which may overlap at any fixed point \emph{is} finite and bounded independently of $q$.
\end{proof}

\subsection{Identification of error terms}\label{ss:stress:error:identification}

In this section, we identify the error terms arising from the addition of $\displaystyle w_{q+1,\nn}=\sum_{\pp=1}^\pmax w_{q+1,\nn,\pp}$.  After doing so, we can write down the Euler-Reynolds system satisfied by $v_{q,\nn}$, in turn verifying at level $\nn$ the conclusions \eqref{e:inductive:n:1:eulerreynolds}, \eqref{e:inductive:n:2:eulerreynolds}, and \eqref{e:inductive:n:3:eulerreynolds} of Propositions~\ref{p:inductive:n:1}, \ref{p:inductive:n:2}, and \ref{p:inductive:n:3}, respectively.

\subsubsection{The case \texorpdfstring{$\nn=0$}{nequalszero}}
By the inductive assumption of Proposition~\ref{p:inductive:n:1}, we have that $\div \vlq =0$, and
\begin{equation*}
    \partial_t{\vlq} + \div(\vlq \otimes \vlq) + \nabla p_{\ell_q} = \div \RR_{\ell_q} + \div \RR_{q}^{\textnormal{comm}}.
\end{equation*}
Adding $w_{q+1,0}$ as defined in \eqref{wqplusoneone}, we obtain that $v_{q,0}:=\vlq+w_{q+1,0}$ solves
\begin{align}
    &\partial_t v_{q,0} + \div (v_{q,0}\otimes v_{q,0}) + \nabla p_{\ell_q}  \notag\\
    &= (\partial_t + \vlq \cdot\nabla)w_{q+1,0}   +  w_{q+1,0} \cdot \nabla \vlq  + \div\left(w_{q+1,0}\otimes w_{q+1,0}\right) + \div \RR_{\ell_q}  + \div \RR_q^\textnormal{comm}\notag\\
    &:= \mathcal{T}_0 + \mathcal{N}_0 + \mathcal{O}_0 + \div \RR_{\ell_q}  +  \div\RR_q^\textnormal{comm}\label{e:euler:reynolds:0}.
\end{align}
For a fixed $\nn$, throughout this section we will consider sums over indices 
\begin{align*}
(\xi,i,j,k,\tilde p, \vecl) 
\end{align*}
where the direction vector $\xi$ takes on one of the finitely many values in Proposition~\ref{pipeconstruction}, $0 \leq i \leq \imax(q)$ indexes the velocity cutoffs (there are finitely many such values, cf.~\eqref{eq:imax:upper:bound:uniform}), $0\leq j \leq \jmax(q,\nn,\pp)$ indexes the stress cutoffs (there are finitely many such values, cf.~\eqref{eq:jmax:bound}), the parameter $k$ indexes the time cutoffs defined in \eqref{eq:chi:cut:def} (the number of values of $k$ is $q$-dependent, but this is irrelevant because they form a partition of unity cf.~\eqref{eq:chi:cut:partition:unity}), the parameter $1 \leq \tilde p \leq \pmax$ indexes which component of $\RR_{q+1,\nn,\pp}$ we are working with  (there are finitely many such values, cf.~\eqref{eq:pmax:DEF}), and lastly, $\vecl$ indexes the checkerboard cutoffs from Definition~\ref{def:checkerboard} (again, the number of such indexes is $q$-dependent, but this is acceptable because they form a partition of unity cf.~\eqref{eq:checkerboard:partition}).
For brevity of notation, we denote sums over such indexes as 
\begin{equation*}
\sum_{\xi,i,j,k,\pp,\vecl} \,.
\end{equation*}
Moreover, we shall denote as
\begin{equation}\label{e:sum:neq}
    \sum_{\neq\{\xi,i,j,k,\pp,\vecl\}} 
\end{equation}
the {\em double-summation} over indexes $(\xi,i,j,k,\tilde p, \vecl)$ and $(\xistar, \istar, \jstar, \kstar, \pstar, \veclstar)$ which belong to the set
\begin{equation}\label{e:xiijk}
    \left\{ (\xi,i,j,k,\pp,\vecl,\xistar), (\istar, \jstar, \kstar, \pstar, \veclstar): \xi\neq\xistar \lor i \neq \istar \lor j \neq \jstar \lor k \neq \kstar \lor \pp \neq \pstar \lor \vecl\neq\veclstar \right \},
\end{equation}
although we remind the reader that at the current stage $\nn=0$, the sum over $\pp$ is superfluous since $w_{q+1,0}=w_{q+1,0,1}$.  For the sake of consistency between $w_{q+1,0}$ and $w_{q+1,\nn}$ for $1\leq\nn\leq\nmax$, we shall include the index $\pp$ throughout this section. Expanding out the oscillation error $\mathcal{O}_0$, we have that
\begin{align}
    \mathcal{O}_0 &= \sum_{\xi,i,j,k,\pp,\vecl} \div \left( \curl\left(a_{(\xi)} \nabla\Phi_{(i,k)}^T \UU_{\xi,q+1,0} \circ \Phiik \right) \otimes \curl\left( a_{(\xi)}\nabla\Phi_{(i,k)}^T \UU_{\xi,q+1,0} \circ \Phiik \right) \right)\notag \\
    & + \sum_{\neq\{\xi,i,j,k,\pp,\vecl\}} \div \left( \curl\left(a_{(\xi)} \nabla\Phi_{(i,k)}^T \UU_{\xi,q+1,0} \circ \Phiik \right) \otimes \curl\left( a_{(\xistar)}\nabla\Phi_{(\istar,\kstar)}^T \UU_{\xistar,q+1,0} \circ \Phiikstar \right) \right)\notag\\
    &:= \div \mathcal{O}_{0,1} + \div \mathcal{O}_{0,2} \label{e:split:0:1}.
\end{align}
In Section~\ref{ss:stress:oscillation:2}, we will show that $\mathcal{O}_{0,2}$ is a Type 2 oscillation error so that
\begin{equation*}
\mathcal{O}_{0,2} = 0.
\end{equation*}
Recalling identity~\eqref{eq:pipes:flowed:1} and the notation~\eqref{eq:otimes:s}, we further split $\mathcal{O}_{0,1}$ as 
\begin{align}
    &\div \mathcal{O}_{0,1} = \sum_{\xi,i,j,k,\pp,\vecl} \div \left( \left( a_{(\xi)} \nabla\Phi_{(i,k)}^{-1} \WW_{\xi,q+1,0}\circ\Phi_{(i,k)} \right) \otimes \left( a_{(\xi)} \nabla\Phi_{(i,k)}^{-1} \WW_{\xi,q+1,0}\circ\Phi_{(i,k)} \right) \right) \notag \\
    &\qquad +2 \sum_{\xi,i,j,k,\pp,\vecl} \div \left( \left( a_{(\xi)} \nabla\Phi_{(i,k)}^{-1} \WW_{\xi,q+1,0}\circ\Phi_{(i,k)} \right) \otimes_{\mathrm{s}} \left( \nabla a_{(\xi)} \times \left( \nabla\Phi_{(i,k)}^{T} \UU_{\xi,q+1,0}\circ\Phi_{(i,k)}\right) \right) \right) \notag \\
    &\qquad + \sum_{\xi,i,j,k,\pp,\vecl} \div \left( \left( \nabla a_{(\xi)} \times \left( \nabla\Phi_{(i,k)}^{T} \UU_{\xi,q+1,0}\circ\Phi_{(i,k)}\right) \right) \otimes \left( \nabla a_{(\xi)} \times \left( \nabla\Phi_{(i,k)}^{T} \UU_{\xi,q+1,0}\circ\Phi_{(i,k)}\right) \right) \right)\notag\\
    &:= \div\left( \mathcal{O}_{0,1,1}+\mathcal{O}_{0,1,2}+\mathcal{O}_{0,1,3}\right).\label{e:split:0:2}
\end{align}
Aside from $\mathcal{O}_{0,1,1}$, each of these terms is a divergence corrector error and will therefore be estimated in Section~\ref{ss:stress:divergence:correctors}. 

Recall by~Propositions~\ref{prop:pipe:shifted},~\ref{pipeconstruction}, and by~\eqref{eq:W:xi:q+1:0:def} that $\WW_{\xi,q+1,0}$ is periodized to scale $\left(\lambda_{q+1}r_{q+1,0}\right)^{-1}=\lambda_{q,0}^{-1}$.
Using the definition of $\LPqnp$ and \eqref{e:Pqnp:identity}, we have that\footnote{The case $\nn=0$ is exceptional in the sense that the minimum frequency of $\mathbb{P}_{\geq\lambda_{q,0}}$ and the minimum frequency of $\mathbb{P}_{[q,1,0]}$ are in fact both equal to $\lambda_{q,0}=\lambda_{q,1,0}=\lambda_q^{\frac{4}{5}}\lambda_{q+1}^{\frac{1}{5}}$ from \eqref{def:lambda:q:1:0:def} and \eqref{eq:lambda:q:n:def}.  For the sake of consistency with the $\nn\geq 1$ cases, we will include the superfluous $\mathbb{P}_{\geq\lambda_{q,0}}$ in the calculations in this section.}
\begin{align*}
\WW_{\xi,q+1,0}\otimes\WW_{\xi,q+1,0} &= \dashint_{\mathbb{T}^3}{\WW_{\xi,q+1,0}\otimes\WW_{\xi,q+1,0}} \notag\\
&   + \mathbb{P}_{\geq \lambda_{q,0}}   \left(\sum\limits_{n=1}^{\nmax}\sum\limits_{p=1}^\pmax \LPqnp+\LPqnpmax\right) \left(\WW_{\xi,q+1,0}\otimes\WW_{\xi,q+1,0}\right)
\,.
\end{align*}
Combining this observation with identity \eqref{eq:pipes:flowed:2} from Proposition~\ref{pipeconstruction}, and with the definition of the $a_{(\xi)}$ in \eqref{eq:a:xi:def}, we further split $\mathcal{O}_{0,1,1}$ as
\begin{align}
    &\div\left(\mathcal{O}_{0,1,1}\right) = \sum_{\xi,i,j,k,\pp,\vecl} \div\left( a_{(\xi)}^2 \nabla\Phi_{(i,k)}^{-1} \left(\dashint_{\mathbb{T}^3}\WW_{\xi,q+1,0}\otimes\WW_{\xi,q+1,0}(\Phi_{(i,k)})\right) \nabla\Phi_{(i,k)}^{-T} \right) \notag \\
    & + \sum_{\xi,i,j,k,\pp,\vecl}\div\bigg{(} a_{(\xi)}^2 \nabla\Phi_{(i,k)}^{-1}  \notag\\
    &\qquad \qquad \times  \mathbb{P}_{\geq \lambda_{q,0}} \left(\sum\limits_{n=1}^{\nmax}\sum\limits_{p=1}^\pmax \LPqnp+\LPqnpmax\right) (\WW\otimes\WW)_{\xi,q+1,0}(\Phi_{(i,k)}) \nabla\Phi_{(i,k)}^{-T} \bigg{)}\notag\\
   &= \div \sum_{i,j,k,\pp,\vecl} \sum_\xi \delta_{q+1,0,\pp}\Gamma_{q+1}^{2j+4}\eta_{(i,j,k)}^2 \gamma_\xi^2\left(\frac{R_{q,0,\pp,j,i,k}}{\delta_{q+1,0,\pp}\Gamma_{q+1}^{2j+4}}\right)\nabla\Phi_{(i,k)}^{-1}\left(\xi\otimes\xi\right)\nabla\Phi_{(i,k)}^{-T} \notag\\
    & + \sum_{\xi,i,j,k,\pp,\vecl} \nabla a_{(\xi)}^2\nabla\Phi_{(i,k)}^{-1}  \notag\\
    &\qquad \times  \mathbb{P}_{\geq \lambda_{q,0}} \left(\sum\limits_{n=1}^{\nmax}\sum\limits_{p=1}^\pmax \LPqnp+\LPqnpmax\right) (\WW\otimes\WW)_{\xi,q+1,0}(\Phi_{(i,k)}) \nabla\Phi_{(i,k)}^{-T}\notag\\
    & + \sum_{\xi,i,j,k,\pp,\vecl} a_{(\xi)}^2(\nabla\Phiik^{-1})_{\alpha\theta} \notag\\
    &\qquad \times   \mathbb{P}_{\geq \lambda_{q,0}} \left(\sum\limits_{n=1}^{\nmax}\sum\limits_{p=1}^\pmax \LPqnp+\LPqnpmax\right) (\WW^\theta\WW^\gamma)_{\xi,q+1,0}(\Phiik) \partial_\alpha(\nabla\Phiik^{-1})_{\zeta\gamma}
    \,. 
    \label{eq:euler:reynolds:gross:0}
\end{align}
By Proposition~\ref{p:split}, equation \eqref{e:split}, and the definition \eqref{eq:Rq0j}, we may rewrite the first term on the right side of the above display as
\begin{align}
    &\div \sum_{i,j,k,\pp,\vecl} \sum_\xi \delta_{q+1,0,\pp}\Gamma_{q+1}^{2j+4}\eta_{(i,j,k)}^2 \gamma_\xi^2\left(\frac{R_{q,0,\pp,j,i,k}}{\delta_{q+1,0,\pp}\Gamma_{q+1}^{2j+4}}\right)\nabla\Phi_{(i,k)}^{-1}\left(\xi\otimes\xi\right)\nabla\Phi_{(i,k)}^{-T}\notag\\
    &\qquad = \div \sum_{i,j,k,\vecl}\eta_{(i,j,k)}^2 \left(\delta_{q+1,0,1}\Gamma_{q+1}^{2j+4}\Id-\RR_{\ell_q}\right)\notag\\
    &\qquad = -\div \sum_{i,j,k,\vecl} \eta^2_{(i,j,k)} \RR_{\ell_q} + \nabla \left(\sum_{i,j,k,\vecl}\eta_{(i,j,k)}^2 \delta_{q+1,0,1}\Gamma_{q+1}^{2j+4}\right)\notag\\
    &\qquad := -\div \left( \RR_{\ell_q} \right) + \nabla \pi \label{eq:cancellation:plus:pressure:0}
\end{align}
In the last equality of the above display we have used the fact that by \eqref{eq:eta:cut:partition:unity} we have
\begin{equation}\label{eq:cancellation:0}
  \RR_{\ell_q}   =  \sum_{i,j,k,\vecl} \eta_{(i,j,k)}^2 \RR_{\ell_q}  \, .
\end{equation}
We apply Proposition~\ref{prop:intermittent:inverse:div} to the remaining two terms from \eqref{eq:euler:reynolds:gross:0} to define for $1\leq n \leq \nmax$ and $1\leq p \leq\pmax$ \footnote{Recall that $\divH$ is the local portion of the inverse divergence operator.  The pressure and the nonlocal portion will be accounted for shortly. We will check in Section~\ref{ss:stress:oscillation:1} that these errors are of the form required by the inverse divergence operator as well as check the associated estimates.}
\begin{align}
    \HH_{q,n,p}^0 &:= \divH\bigg{(} \sum_{\xi,i,j,k,\pp,\vecl} \nabla a_{(\xi)}^2\nabla\Phi_{(i,k)}^{-1} \mathbb{P}_{\geq \lambda_{q,0}}   \LPqnp (\WW_{\xi,q+1,0}\otimes\WW_{\xi,q+1,0})(\Phi_{(i,k)}) \nabla\Phi_{(i,k)}^{-T}\notag\\
    &\qquad + \sum_{\xi,i,j,k,\pp,\vecl} a_{(\xi)}^2(\nabla\Phiik^{-1})_{\alpha\theta} \mathbb{P}_{\geq \lambda_{q,0}}   \LPqnp(\WW_{\xi,q+1,0}^\theta\WW_{\xi,q+1,0}^\gamma)(\Phiik) \partial_\alpha(\nabla\Phiik^{-1})_{\zeta\gamma} \bigg{)}.\label{eq:Hqnp0:definition}
\end{align}
The last terms from \eqref{eq:euler:reynolds:gross:0} with $\mathbb{P}_{[q,\nmax,\pmax+1]}$ will be absorbed into $\RR_{q+1}$, whereas the terms in \eqref{eq:Hqnp0:definition} correspond to the error terms in \eqref{e:inductive:n:1:Hstress}.

Before amalgamating the preceding calculations, we pause to calculate the means of various terms to which the inverse divergence operator from Proposition~\ref{prop:intermittent:inverse:div} will be applied.  Examining the equality
\begin{equation}\label{eq:id:zero:mean}
\partial_t v_{q,0} + \div \left( v_{q,0}\otimes v_{q,0} \right) + \nabla p_{\ell_q} = \mathcal{T}_0 + \mathcal{N}_0 + \mathcal{O}_0 + \div \RR_{\ell_q}  + \div \RR_q^{\textnormal{comm}}
\end{equation}
and recalling the definitions of $\mathcal{T}_0$, $\mathcal{N}_0$, and $\mathcal{O}_0$, we see immediately that every term can be written as the divergence of a tensor except for $\partial_t v_{q,0}$ and $\mathcal{T}_0$.  Note however that $v_{q,0}=\vlq+w_{q+1,0}$, that $\int_{\mathbb{T}^3} \partial_t \vlq =0$ (by integrating in space~\eqref{eq:vlq:equation}), and that $w_{q+1,0}$ is the curl of a vector field (cf. \eqref{wqplusonenp}). This shows that  
$\int_{\mathbb{T}^3} \partial_t v_{q,0} = 0$, and thus $\int_{\mathbb{T}^3} \mathcal{T}_0 = 0$ as well. Therefore, we may use \eqref{eq:inverse:div} and \eqref{eq:inverse:div:error:stress}  write 
$$ \mathcal{T}_0 =\div \left( \left( \divH + \divR \right) \mathcal{T}_0 \right) + \nabla P . $$
We can now combine the calculations of \eqref{e:euler:reynolds:0}, \eqref{e:split:0:1}, \eqref{e:split:0:2}, \eqref{eq:euler:reynolds:gross:0}, \eqref{eq:cancellation:plus:pressure:0} \eqref{eq:cancellation:0}, and \eqref{eq:Hqnp0:definition} and let the notation $\nabla\pi$ change from line to line to incorporate all the pressure terms to write that
\begin{align}
    &\partial_t v_{q,0} + \div \left( v_{q,0}\otimes v_{q,0} \right) + \nabla p_{\ell_q} \notag \\
    & = \mathcal{T}_0 + \mathcal{N}_0 + \mathcal{O}_0 + \div \RR_{\ell_q} + \div \RR_q^{\textnormal{comm}} \notag\\
    &= \mathcal{T}_0 +  \mathcal{N}_0 + \div\left(\mathcal{O}_{0,1}\right) + \div\left(\mathcal{O}_{0,2}\right) + \div \RR_{\ell_q}  + \div \RR_q^{\textnormal{comm}}\notag\\
    &= \mathcal{T}_0 +  \mathcal{N}_0 + \div \left(\RR_{\ell_q}+ \mathcal{O}_{0,1,1} \right)+ \div \left(\mathcal{O}_{0,1,2} + \mathcal{O}_{0,1,3}  + \mathcal{O}_{0,2}\right)  + \div \RR_q^{\textnormal{comm}}\notag \\
    &= \mathcal{T}_0 +  \mathcal{N}_0  + \nabla\pi \notag \\
    &\quad + \div (\divH+\divR) \bigg{[} \sum_{\xi,i,j,k,\pp,\vecl} \nabla a_{(\xi)}^2\nabla\Phi_{(i,k)}^{-1} \notag\\
    &\qquad\qquad \times \mathbb{P}_{\geq \lambda_{q,0}}   \left(\sum\limits_{n=1}^{\nmax}\sum\limits_{p=1}^\pmax \LPqnp+\LPqnpmax\right) (\WW_{\xi,q+1,0}\otimes\WW_{\xi,q+1,0})(\Phi_{(i,k)}) \nabla\Phi_{(i,k)}^{-T}\notag\\
    &\quad + \sum_{\xi,i,j,k,\pp,\vecl} a_{(\xi)}^2(\nabla\Phiik^{-1})_{\alpha\theta} \notag\\
    &\qquad\qquad \times \mathbb{P}_{\geq \lambda_{q,0}}   \left(\sum\limits_{n=1}^{\nmax}\sum\limits_{p=1}^\pmax \LPqnp+\LPqnpmax\right) (\WW^\theta\WW^\gamma)_{\xi,q+1,0}(\Phiik) \partial_\alpha(\nabla\Phiik^{-1})_{\zeta\gamma} \bigg{]} \label{eq:id:0:0} \\
    &\qquad + \div\left( \mathcal{O}_{0,1,2} + \mathcal{O}_{0,1,3}  + \mathcal{O}_{0,2} \right)   + \div \RR_q^{\textnormal{comm}}\notag \\
     &= \nabla\pi + \div \bigg{[}  \underbrace{\left(\divH + \divR \right) (\mathcal{T}_0)}_{\textnormal{transport}} + \underbrace{\left(\divH + \divR \right) (\mathcal{N}_0)}_{\textnormal{Nash}} + \RR_q^{\textnormal{comm}} \label{eq:id:0:1} \\
    &\qquad + \underbrace{\left( \divH + \divR \right)\bigg{(} \sum_{\xi,i,j,k,\pp,\vecl} \nabla a_{(\xi)}^2\nabla\Phi_{(i,k)}^{-1} \mathbb{P}_{[q,\nmax,\pmax+1]} (\WW_{\xi,q+1,0}\otimes\WW_{\xi,q+1,0})(\Phi_{(i,k)}) \nabla\Phi_{(i,k)}^{-T}}_{\textnormal{Type 1 - part of }\RR_{q+1}^0} \label{eq:id:0:2}\\
    &\qquad\qquad + \underbrace{\sum_{\xi,i,j,k,\pp,\vecl} a_{(\xi)}^2(\nabla\Phiik^{-1})_{\alpha\theta} \mathbb{P}_{[q,\nmax,\pmax+1]}(\WW_{\xi,q+1,0}^\theta\WW_{\xi,q+1,0}^\gamma)(\Phiik) \partial_\alpha(\nabla\Phiik^{-1})_{\zeta\gamma}}_{\textnormal{Type 1 - part of }\RR_{q+1}^0} \bigg{)} \label{eq:id:0:3}\\
    &\quad +\underbrace{ \divR \bigg{(} \sum_{\xi,i,j,k,\pp,\vecl} \nabla a_{(\xi)}^2\nabla\Phi_{(i,k)}^{-1} \mathbb{P}_{\geq \lambda_{q,0}}  \left(\sum\limits_{n=1}^{\nmax}\sum\limits_{p=1}^\pmax\LPqnp\right) (\WW_{\xi,q+1,0}\otimes\WW_{\xi,q+1,0})(\Phi_{(i,k)}) \nabla\Phi_{(i,k)}^{-T}}_{\textnormal{Type 1 - part of }\RR_{q+1}^0} \label{eq:id:0:4}\\
    &\qquad + \underbrace{\sum_{\xi,i,j,k,\pp,\vecl}  a_{(\xi)}^2(\nabla\Phiik^{-1})_{\alpha\theta} \mathbb{P}_{\geq \lambda_{q,0}}   \left(\sum\limits_{n=1}^{\nmax}\sum\limits_{p=1}^\pmax\LPqnp\right) (\WW_{\xi,q+1,0}^\theta\WW_{\xi,q+1,0}^\gamma)(\Phiik) \partial_\alpha(\nabla\Phiik^{-1})_{\zeta\gamma}}_{\textnormal{Type 1 - part of }\RR_{q+1}^0} \bigg{)} \label{eq:id:0:5}\\
    &\qquad\qquad + \underbrace{\mathcal{O}_{0,1,2} + \mathcal{O}_{0,1,3} }_{\textnormal{divergence corrector}} + \underbrace{\mathcal{O}_{0,2}}_{\textnormal{Type 2}} \bigg{]} \label{eq:id:0:6}\\
    &\quad + \div \divH \bigg{(} \underbrace{\sum_{\xi,i,j,k,\pp,\vecl} \nabla a_{(\xi)}^2\nabla\Phi_{(i,k)}^{-1} \mathbb{P}_{\geq \lambda_{q,0}}  \left(\sum\limits_{n=1}^{\nmax}\sum\limits_{p=1}^\pmax\LPqnp\right) (\WW_{\xi,q+1,0}\otimes\WW_{\xi,q+1,0})(\Phi_{(i,k)}) \nabla\Phi_{(i,k)}^{-T}}_{\textnormal{Type 1 - }\HH_{q,n,p}^0}  \label{eq:id:0:7}\\
    &\qquad + \underbrace{\sum_{\xi,i,j,k,\pp,\vecl} a_{(\xi)}^2(\nabla\Phiik^{-1})_{\alpha\theta} \mathbb{P}_{\geq \lambda_{q,0}}  \left(\sum\limits_{n=1}^{\nmax}\sum\limits_{p=1}^\pmax\LPqnp\right) (\WW_{\xi,q+1,0}^\theta\WW_{\xi,q+1,0}^\gamma)(\Phiik) \partial_\alpha(\nabla\Phiik^{-1})_{\zeta\gamma}}_{\textnormal{Type 1 - }\HH_{q,n,p}^0} \bigg{)} \label{eq:id:0:8}\\
    &:= \div(\RR_{q+1}^0) + \div \left( \sum\limits_{n=1}^{\nmax}\sum\limits_{p=1}^{\pmax} \HH_{q,n,p}^0 \right) + \nabla \pi + \div\RR_{q}^{\textnormal{comm}}, \notag
\end{align}
thus verifying \eqref{e:inductive:n:1:eulerreynolds} from Proposition~\ref{p:inductive:n:1} after condensing the labeled terms into $\RR_{q+1}^0$ and using \eqref{eq:Hqnp0:definition} on the pieces labeled $\HH_{q,n,p}^0$.

\subsubsection{The case \texorpdfstring{$1\leq\nn\leq\nmax-1$}{nequals}}

From \eqref{e:inductive:n:2:eulerreynolds}, we assume that $v_{q,\nn-1}$ is divergence-free and is a solution to
\begin{align*}
    \partial_t v_{q,\nn-1} + &\div\left( v_{q,\nn-1}\otimes v_{q,\nn-1} \right) + \nabla p_{q,\nn-1}\notag\\
    &= \div \left( \RR_{q+1}^{\nn-1} \right) + \div \left( \sum\limits_{n=\nn}^{\nmax} \sum_{p=1}^\pmax \sum\limits_{n'=0}^{\nn-1} \HH_{q,n,p}^{n'} \right) + \div\RR_{q}^{\textnormal{comm}}\, .
\end{align*}
Now using the definition of $\RR_{q,\nn,\pp}$ from \eqref{e:rqnp:definition} and adding $w_{q+1,\nn}$ as defined in \eqref{wqplusonenp}, we have that $v_{q,\nn}:= v_{q,\nn-1} + w_{q+1,\nn} = \vlq + \sum_{0\leq n' \leq \nn-1} w_{q+1,n'} + w_{q+1,\nn}$ solves
\begin{align}
    \partial_t v_{q,\nn} + &\div \left( v_{q,\nn}\otimes v_{q,\nn} \right) + \nabla p_{q,\nn-1} \notag\\
    &= \div\left(\RR_{q+1}^{\nn-1}\right) + \div \left( \sum\limits_{n=\nn+1}^{\nmax} \sum_{p=1}^\pmax \sum\limits_{n'=0}^{\nn-1} \HH_{q,n,p}^{n'} \right) + \div\RR_{q}^{\textnormal{comm}}\notag\\
    &\quad + (\partial_t + \vlq\cdot\nabla)w_{q+1,\nn} + w_{q+1,\nn}\cdot \nabla \vlq \notag\\
    &\quad + \sum_{n'\leq \nn-1} \div\left( w_{q+1,\nn}\otimes w_{q+1,n'} + w_{q+1,n'}\otimes w_{q+1,\nn}\right)\notag\\
    &\quad + \div \left( w_{q+1,\nn} \otimes w_{q+1,\nn}+\sum_{\pp=1}^\pmax\RR_{q,\nn,\pp} \right).\label{e:expand:nn}
\end{align}
The first term on the right hand side is $\RR_{q+1}^{\nn-1}$, which satisfies the same estimates as $\RR_{q+1}^\nn$ by \eqref{e:inductive:n:2:Rstress} and will thus be absorbed into $\RR_{q+1}^{\nn}$ (these estimates do not change in $\nn$ save for implicit constants).  The second term, save for the fact that the sum is over $n'\leq\nn-1$ rather than $n'\leq\nn$ and is therefore missing the terms $\HH_{q,n,p}^\nn$, matches \eqref{e:inductive:n:2:eulerreynolds} at level $\nn$ (i.e. replacing every instance of $\nn-1$ with $\nn$). As before, we apply the inverse divergence operators from Proposition~\ref{prop:intermittent:inverse:div} to the transport and Nash errors to obtain
$$  (\partial_t + \vlq\cdot\nabla)w_{q+1,\nn} + w_{q+1,\nn}\cdot\nabla\vlq = \div \left( (\divH + \divR) \left((\partial_t + \vlq\cdot\nabla)w_{q+1,\nn} + w_{q+1,\nn}\cdot\nabla\vlq\right)\right) + \nabla \pi,   $$
and these errors are absorbed into $\RR_{q+1}^\nn$ or the new pressure.  We will show in Section~\ref{ss:stress:oscillation:2} that the interaction of $w_{q+1,\nn}$ with previous terms $w_{q+1,n'}$ is a Type 2 oscillation error so that
\begin{equation}\label{nn:overlap:definition}
\sum_{n'\leq \nn-1} \left( w_{q+1,\nn}\otimes w_{q+1,n'} + w_{q+1,n'}\otimes w_{q+1,\nn}\right) = 0.
\end{equation}
So to verify \eqref{e:inductive:n:2:eulerreynolds} at level $\nn$, only the analysis of 
$$  \div \left( w_{q+1,\nn} \otimes w_{q+1,\nn} + \sum_{\pp=1}^\pmax \RR_{q,\nn,\pp} \right) $$
remains.  Reusing the notations from \eqref{e:sum:neq}\footnote{In a slight abuse of notation, notice that the admissible values of $\vecl$ have changed, since these parameters describe the checkerboard cutoff functions at scale $\lambda\qnnone^{-1}$ and thus depend on $\nn$.} and writing out the self-interaction of $w_{q+1,\nn}$ yields
\begin{align}
    \div\left(w_{q+1,\nn}\otimes w_{q+1,\nn}\right) &= \sum_{\xi,i,j,k,\pp,\vecl} \div \left( \curl\left(a_{(\xi)} \nabla\Phi_{(i,k)}^T \UU_{\xi,q+1,\nn} \right) \otimes \curl\left( a_{(\xi)}\nabla\Phi_{i,k}^T \UU_{\xi,q+1,\nn} \right) \right)\notag \\
    & + \sum_{\neq\{\xi,i,j,k,\pp,\vecl\}} \div \left( \curl\left(a_{(\xi)} \nabla\Phi_{(i,k)}^T \UU_{\xi,q+1,\nn} \right) \otimes \curl\left( a_{(\xi')}\nabla\Phi_{(i',k')}^T \UU_{\xi',q+1,\nn} \right) \right)\notag\\
    &:= \div \mathcal{O}_{\nn,1} + \div \mathcal{O}_{\nn,2} \label{e:split:nn:1}.
\end{align}
As before, we will show that $\mathcal{O}_{\nn,2}$ is a Type 2 oscillation error so that
\begin{equation*}
\mathcal{O}_{\nn,2} = 0.
\end{equation*}
Splitting $\mathcal{O}_{\nn,1}$ gives
\begin{align}
    &\div \mathcal{O}_{\nn,1} = \sum_{\xi,i,j,k,\pp,\vecl} \div \left( \left( a_{(\xi)} \nabla\Phi_{(i,k)}^{-1} \WW_{\xi,q+1,\nn}\circ\Phi_{(i,k)} \right) \otimes \left( a_{(\xi)} \nabla\Phi_{(i,k)}^{-1} \WW_{\xi,q+1,\nn}\circ\Phi_{(i,k)} \right) \right) \notag \\
    &\qquad + 2 \sum_{\xi,i,j,k,\pp,\vecl} \div \left( \left( a_{(\xi)} \nabla\Phi_{(i,k)}^{-1} \WW_{\xi,q+1,\nn}\circ\Phi_{(i,k)} \right) \otimes_{\mathrm{s}} \left( \nabla a_{(\xi)} \times \left( \nabla\Phi_{(i,k)}^{T} \UU_{\xi,q+1,\nn}\circ\Phi_{(i,k)}\right) \right) \right) \notag \\
    &\qquad + \sum_{\xi,i,j,k,\pp,\vecl} \div \left( \left( \nabla a_{(\xi)} \times \left( \nabla\Phi_{(i,k)}^{T} \UU_{\xi,q+1,\nn}\circ\Phi_{(i,k)}\right) \right) \otimes \left( \nabla a_{(\xi)} \times \left( \nabla\Phi_{(i,k)}^{T} \UU_{\xi,q+1,\nn}\circ\Phi_{(i,k)}\right) \right) \right)\notag\\
    &:= \div\left( \mathcal{O}_{\nn,1,1}+\mathcal{O}_{\nn,1,2}+\mathcal{O}_{\nn,1,3} \right).\label{e:split:nn:2}
\end{align}
The last two of these terms are again divergence corrector errors and will therefore be absorbed into $\RR_{q+1}^\nn$ and estimated in Section~\ref{ss:stress:divergence:correctors}.  So the only terms remaining from \eqref{e:expand:nn} are $\mathcal{O}_{\nn,1,1}$ and $\sum_{\pp=1}^\pmax\RR_{q,\nn,\pp}$, which are analyzed in a fashion similar to the $\nn=0$ case, save for the fact that summation over $\pp$ is now crucial.

Recall cf.~\eqref{eq:W:xi:q+1:nn:def} that $\WW_{\xi,q+1,\nn}$ is periodized to scale $\left(\lambda_{q+1}r_{q+1,\nn}\right)^{-1}=\lambda\qnn^{-1}$.
Using \eqref{e:Pqnp:identity}, we have that
\begin{align*}
\WW_{\xi,q+1,\nn}\otimes\WW_{\xi,q+1,\nn} &= \dashint_{\mathbb{T}^3}{\WW_{\xi,q+1,\nn}\otimes\WW_{\xi,q+1,\nn}}\notag\\
&   + \Pqnn   \left(\Psum\right) \left(\WW_{\xi,q+1,\nn}\otimes\WW_{\xi,q+1,\nn}\right).
\end{align*}
Combining this division with identity \eqref{eq:pipes:flowed:2} from Proposition~\ref{pipeconstruction}, we further split $\mathcal{O}_{\nn,1,1}$ as
\begin{align}
    &\div\left(\mathcal{O}_{\nn,1,1}\right) = \sum_{\xi,i,j,k,\pp,\vecl} \div\left( a_{(\xi)}^2 \nabla\Phi_{(i,k)}^{-1} \left(\dashint_{\mathbb{T}^3}\WW_{\xi,q+1,\nn}\otimes\WW_{\xi,q+1,\nn}(\Phi_{(i,k)})\right) \nabla\Phi_{(i,k)}^{-T} \right) \notag \\
    & + \sum_{\xi,i,j,k,\pp,\vecl} \div\bigg{(} a_{(\xi)}^2 \nabla\Phi_{(i,k)}^{-1} \notag\\
    &\qquad \qquad \times  \Pqnn   \left(\Psum\right) (\WW\otimes\WW)_{\xi,q+1,\nn}(\Phi_{(i,k)}) \nabla\Phi_{(i,k)}^{-T} \bigg{)}\notag\\
    &= \div \sum_{i,j,k,\pp,\vecl}\sum_\xi \delta_{q+1,\nn,\pp}\Gamma_{q+1}^{2j+4}\eta_{(i,j,k)}^2 \gamma_\xi^2\left(\frac{R_{q,\nn,\pp,j,i,k}}{\delta_{q+1,\nn,\pp}\Gamma_{q+1}^{2j+4}}\right)\nabla\Phi_{(i,k)}^{-1}\left(\xi\otimes\xi\right)\nabla\Phi_{(i,k)}^{-T} \notag\\
    & + \sum_{\xi,i,j,k,\pp,\vecl} \nabla a_{(\xi)}^2\nabla\Phi_{(i,k)}^{-1} \notag\\
    &\qquad \qquad \times \Pqnn   \left(\Psum\right) (\WW\otimes\WW)_{\xi,q+1,\nn}(\Phi_{(i,k)}) \nabla\Phi_{(i,k)}^{-T}\notag\\
    & + \sum_{\xi,i,j,k,\pp,\vecl} a_{(\xi)}^2(\nabla\Phiik^{-1})_{\alpha\theta} \notag\\
    &\qquad \qquad \times \Pqnn   \left(\Psum\right) (\WW^\theta\WW^\gamma)_{\xi,q+1,\nn}(\Phiik) \partial_\alpha(\nabla\Phiik^{-1})_{\zeta\gamma} \,.
\label{eq:euler:reynolds:gross:nn}
\end{align}
By Proposition~\ref{p:split}, equation~\eqref{e:split}, and identity \eqref{eq:rqnpj}, we obtain that
\begin{align}
    &\div \sum_{i,j,k,\pp,\vecl} \sum_\xi \delta_{q+1,\nn,\pp}\Gamma_{q+1}^{2j+4}\eta_{(i,j,k)}^2 \gamma_\xi^2\left(\frac{R_{q,\nn,\pp,j,i,k}}{\delta_{q+1,\nn,\pp}\Gamma_{q+1}^{2j+4}}\right)\nabla\Phi_{(i,k)}^{-1}\left(\xi\otimes\xi\right)\nabla\Phi_{(i,k)}^{-T}\notag\\
    &\qquad = \div \sum_{i,j,k,\pp,\vecl}\eta_{(i,j,k)}^2 \left( \delta_{q+1,\nn,\pp}\Gamma_{q+1}^{2j+4}\Id-\sum_{\pp=1}^\pmax\RR_{q,\nn,\pp}\right)\notag\\
    &\qquad = -\div \sum_{i,j,k,\vecl} \sum_{\pp=1}^\pmax \eta^2_{(i,j,k)} \RR_{q,\nn,\pp} + \nabla \left(\sum_{i,j,k,\vecl}\eta_{(i,j,k)}^2 \delta_{q+1,\nn,\pp}\Gamma_{q+1}^{2j+4}\right)\notag\\
    &\qquad := -\div  \sum_{\pp=1}^\pmax  \RR_{q,\nn,\pp} + \nabla \pi \,,\label{eq:cancellation:plus:pressure:nn}
\end{align}
where in the last equality we have appealed to \eqref{eq:eta:cut:partition:unity}. We can finally apply Proposition~\ref{prop:intermittent:inverse:div} to the remaining terms in \eqref{eq:euler:reynolds:gross:nn} for $\nn+1\leq n \leq\nmax$ and $1\leq p\leq\pmax$ to define 
\begin{align}
    &\HH_{q,n,p}^\nn := \divH \bigg{(} \sum_{\xi,i,j,k,\pp} \nabla a_{(\xi)}^2\nabla\Phi_{(i,k)}^{-1} \Pqnn  \LPqnp (\WW_{\xi,q+1,\nn}\otimes\WW_{\xi,q+1,\nn})(\Phi_{(i,k)}) \nabla\Phi_{(i,k)}^{-T}\notag\\
    &\qquad + \sum_{\xi,i,j,k,\pp} a_{(\xi)}^2(\nabla\Phiik^{-1})_{\alpha\theta}\Pqnn  \LPqnp(\WW_{\xi,q+1,\nn}^\theta\WW_{\xi,q+1,\nn}^\gamma)(\Phiik) \partial_\alpha(\nabla\Phiik^{-1})_{\zeta\gamma} \bigg{)}.\label{eq:Hqnpnn:definition}
\end{align}
As before, the terms from \eqref{eq:euler:reynolds:gross:nn} with $\LPqnpmax$ will be absorbed into $\RR_{q+1}^\nn$. We will show shortly that the terms $\HH_{q,n,p}^\nn$ in \eqref{eq:Hqnpnn:definition} are precisely the terms needed to make \eqref{e:expand:nn} match \eqref{e:inductive:n:2:eulerreynolds} at level $\nn$. As before, any nonlocal inverse divergence terms will be absorbed into $\RR_{q+1}^\nn$.

Recall from \eqref{RR:q+1:n-1:to:n} that $\RR_{q+1}^\nn$ will include $\RR_{q+1}^{\nn-1}$ in addition to error terms arising from the addition of $w_{q+1,\nn}$ which are small enough to be absorbed in $\RR_{q+1}$.  Then to check \eqref{e:inductive:n:2:eulerreynolds}, we return to \eqref{e:expand:nn} and use \eqref{e:split:nn:1}, \eqref{e:split:nn:2}, \eqref{eq:euler:reynolds:gross:nn}, \eqref{eq:cancellation:plus:pressure:nn}, and \eqref{eq:Hqnpnn:definition} to write 
\begin{align}
    &\partial_t v_{q,\nn} + \div \left( v_{q,\nn}\otimes v_{q,\nn} \right) + \nabla p_{q,\nn-1} \notag\\
    &= \div \left( \sum\limits_{n=\nn+1}^{\nmax} \sum_{p=1}^\pmax \sum\limits_{n'=0}^{\nn-1} \HH_{q,n,p}^{n'} \right) + \div\left(\RR_{q+1}^{\nn-1}\right) + \div\RR_{q}^{\textnormal{comm}}\notag\\
    &\quad + (\partial_t + \vlq\cdot\nabla)w_{q+1,\nn} + w_{q+1,\nn}\cdot \nabla \vlq \notag\\
    &\quad + \sum_{n'\leq \nn-1} \div\left( w_{q+1,\nn}\otimes w_{q+1,n'} + w_{q+1,n'}\otimes w_{q+1,\nn}\right)\notag\\
    &\quad + \div \left( w_{q+1,\nn} \otimes w_{q+1,\nn}+\sum_{\pp=1}^\pmax\RR_{q,\nn,\pp} \right) \notag\\
    &= \div\RR_{q}^{\textnormal{comm}} + \div \left( \sum\limits_{n=\nn+1}^{\nmax}\sum_{p=1}^\pmax\sum\limits_{n'=0}^{\nn-1} \HH_{q,n,p}^{n'} \right) + \div \bigg{(} \RR_{q+1}^{\nn-1} + (\divH + \divR)\left( \partial_t w_{q+1,\nn} + \vlq \cdot \nabla w_{q+1,\nn} \right) \notag\\
    &\qquad + (\divH + \divR)\left( w_{q+1,\nn}\cdot\nabla\vlq \right) + \sum_{n'\leq \nn-1}\left(w_{q+1,\nn}\otimes w_{q+1,n'} + w_{q+1,n'}\otimes w_{q+1,\nn}\right) \bigg{)}\notag\\
    &\qquad + \div\left(   \mathcal{O}_{\nn,1,2} + \mathcal{O}_{\nn,1,3}   + \mathcal{O}_{\nn,2}  \right) + \nabla \pi + \div \left( \mathcal{O}_{\nn,1,1} +  \sum_{\pp=1}^\pmax\RR\qnnpp \right) \notag\\
    &= \div\RR_{q}^{\textnormal{comm}} + \div \left( \sum\limits_{n=\nn+1}^{\nmax} \sum_{p=1}^\pmax \sum\limits_{n'=0}^{\nn-1} \HH_{q,n,p}^{n'} \right) + \div \bigg{(} \RR_{q+1}^{\nn-1} + (\divH + \divR)\left( \partial_t w_{q+1,\nn} + \vlq \cdot \nabla w_{q+1,\nn} \right) \notag\\
    &\quad + (\divH + \divR)\left( w_{q+1,\nn}\cdot\nabla\vlq \right) + \sum_{n'\leq \nn-1}\left(w_{q+1,\nn}\otimes w_{q+1,n'} + w_{q+1,n'}\otimes w_{q+1,\nn}\right) \bigg{)}\notag\\
    &\quad  + \div \left( \mathcal{O}_{\nn,1,2} + \mathcal{O}_{\nn,1,3}  + \mathcal{O}_{\nn,2} \right) + \nabla\pi \notag\\
    &\quad +\div\bigg{[}  \left(\divH+\divR\right) \bigg{(} \sum_{\xi,i,j,k,\pp,\vecl} \nabla a_{(\xi)}^2\nabla\Phi_{(i,k)}^{-1}  \notag\\
    &\qquad \qquad \qquad   \times \Pqnn  \left(\Psum\right) (\WW\otimes\WW)_{\xi,q+1,\nn}(\Phi_{(i,k)}) \nabla\Phi_{(i,k)}^{-T} \notag \\
    &\qquad\qquad + \sum_{\xi,i,j,k,\pp,\vecl} a_{(\xi)}^2(\nabla\Phiik^{-1})_{\alpha\theta}  \notag\\
    &\qquad \qquad \qquad\times \Pqnn   \left(\Psum\right) (\WW^\theta\WW^\gamma)_{\xi,q+1,\nn}(\Phiik) \partial_\alpha(\nabla\Phiik^{-1})_{\zeta\gamma} \bigg{)} \bigg{]} \label{eq:id:nn:0} \\
    &= \div\RR_{q}^{\textnormal{comm}} + \div \left( {\sum\limits_{n=\nn+1}^{\nmax} \sum_{p=1}^\pmax \sum\limits_{n'=0}^{\nn-1} \HH_{q,n,p}^{n'}} \right) + \div \bigg{(} \RR_{q+1}^{\nn-1} + \underbrace{(\divH + \divR)\left( \partial_t w_{q+1,\nn} + \vlq \cdot \nabla w_{q+1,\nn} \right)}_{\textnormal{transport}} \label{eq:id:nn:1}\\
    &\qquad + \underbrace{(\divH + \divR)\left( w_{q+1,\nn}\cdot\nabla\vlq \right)}_{\textnormal{Nash}} + \underbrace{\sum_{n'\leq \nn-1}\left(w_{q+1,\nn}\otimes w_{q+1,n'} + w_{q+1,n'}\otimes w_{q+1,\nn}\right)}_{\textnormal{Type 2}} \bigg{)} \label{eq:id:nn:2} \\   
    &\qquad + \div\bigg{(} \underbrace{\mathcal{O}_{\nn,1,2} + \mathcal{O}_{\nn,1,3}  }_{\textnormal{divergence corrector}} + \underbrace{\mathcal{O}_{\nn,2}}_{\textnormal{Type 2}} \bigg{)} + \nabla \pi \label{eq:id:nn:3} \\
    &\qquad + \div\bigg{[}  \underbrace{(\divH + \divR) \bigg{(} \sum_{\xi,i,j,k,\pp,\vecl} \nabla a_{(\xi)}^2\nabla\Phi_{(i,k)}^{-1} \mathbb{P}_{[q,\nmax,\pmax+1]} (\WW_{\xi,q+1,\nn}\otimes\WW_{\xi,q+1,\nn})(\Phi_{(i,k)}) \nabla\Phi_{(i,k)}^{-T}}_{\textnormal{Type 1 - part of }\RR_{q+1}^\nn} \label{eq:id:nn:4} \\
    &\qquad\qquad + \underbrace{\sum_{\xi,i,j,k,\pp,\vecl} a_{(\xi)}^2(\nabla\Phiik^{-1})_{\alpha\theta} \mathbb{P}_{[q,\nmax,\pmax+1]}(\WW_{\xi,q+1,\nn}^\theta\WW_{\xi,q+1,\nn}^\gamma)(\Phiik) \partial_\alpha(\nabla\Phiik^{-1})_{\zeta\gamma} \bigg{)}}_{\textnormal{Type 1 - part of }\RR_{q+1}^\nn} \label{eq:id:nn:5} \\
    & + \divR  \bigg{(} \sum_{\xi,i,j,k,\pp,\vecl} \nabla a_{(\xi)}^2\nabla\Phi_{(i,k)}^{-1} \notag\\ 
    &\qquad \qquad \times \underbrace{\Pqnn   \left(\Psum\right) (\WW\otimes\WW)_{\xi,q+1,\nn}(\Phi_{(i,k)}) \nabla\Phi_{(i,k)}^{-T}}_{\textnormal{Type 1 - part of }\RR_{q+1}^\nn} \label{eq:id:nn:6} \\
    & +\sum_{\xi,i,j,k,\pp,\vecl} a_{(\xi)}^2(\nabla\Phiik^{-1})_{\alpha\theta}\notag\\
    &\qquad \qquad \times \underbrace{ \Pqnn  \left(\Psum\right) (\WW^\theta\WW^\gamma)_{\xi,q+1,\nn}(\Phiik) \partial_\alpha(\nabla\Phiik^{-1})_{\zeta\gamma}}_{\textnormal{Type 1 - part of }\RR_{q+1}^\nn} \bigg{)} \bigg{]} \label{eq:id:nn:7} \\
    &\quad + \div   \underbrace{ \divH \bigg{[} \sum_{\xi,i,j,k,\pp,\vecl} \nabla a_{(\xi)}^2\nabla\Phi_{(i,k)}^{-1} \Pqnn   \left(\Psumchop\right) (\WW\otimes\WW)_{\xi,q+1,\nn}(\Phi_{(i,k)}) \nabla\Phi_{(i,k)}^{-T}}_{\textnormal{Type 1 - }\HH_{q,n,p}^\nn} \label{eq:id:nn:8} \\
    &\qquad + \underbrace{\sum_{\xi,i,j,k,\pp,\vecl} a_{(\xi)}^2(\nabla\Phiik^{-1})_{\alpha\theta} \Pqnn   \left(\Psumchop\right) (\WW^\theta\WW^\gamma)_{\xi,q+1,\nn}(\Phiik) \partial_\alpha(\nabla\Phiik^{-1})_{\zeta\gamma} \bigg{]}}_{\textnormal{Type 1 - }\HH_{q,n,p}^\nn} \label{eq:id:nn:9} \\
    &= \div\RR_{q}^{\textnormal{comm}} + \div \RR_{q+1}^{\nn} + \div \sum\limits_{n=\nn+1}^\nmax \sum_{p=1}^\pmax \sum\limits_{n'=0}^\nn \HH_{q,n,p}^{n'} + \nabla \pi, \notag
\end{align}
which concludes the proof after identifying the first seven lines (save for the double sum of $\HH_{q,n}^{n'}$ terms) of the second to last equality as $\RR_{q+1}^\nn$ and using \eqref{eq:Hqnpnn:definition} to incorporate the eighth and ninth lines into the new double sum of $\HH_{q,n}^{n'}$ terms.  Note that we have implicitly used in the above equalities that $\left(\partial_t + \vlq\cdot \nabla\right)w_{q+1,\nn}$ has zero mean, which can be deduced in the same fashion as for the case $\nn=0$.

\subsubsection{The case \texorpdfstring{$\nn=\nmax$}{nmax}}

From \eqref{e:inductive:n:3:eulerreynolds}, we assume that $v_{q,\nmax-1}$ is divergence-free and is a solution to
\begin{align*}
    \partial_t v_{q,\nmax-1} + &\div\left( v_{q,\nmax-1}\otimes v_{q,\nmax-1} \right) + \nabla p_{q,\nmax-1}\notag\\
    &= \div \left( \RR_{q+1}^{\nmax-1} \right) + \div \left( \sum\limits_{n'=0}^{\nmax-1}\sum_{p=1}^\pmax \HH_{q,\nmax,p}^{n'} \right) + \div\RR_{q}^{\textnormal{comm}}\, .
\end{align*}
Now using the definition of $\RR_{q,\nmax,p}$ from \eqref{e:rqnp:definition} and adding $w_{q+1,\nmax}$ as defined in \eqref{wqplusonenp}, we have that $v_{q+1}:= v_{q,\nmax-1} + w_{q+1,\nmax}$ solves
\begin{align}
    \partial_t v_{q+1} + &\div \left( v_{q+1}\otimes v_{q+1} \right) + \nabla p_{q,\nmax-1} \notag\\
    &= \div\RR_{q}^{\textnormal{comm}} + \div\left(\RR_{q+1}^{\nmax-1}\right) + (\partial_t + \vlq\cdot\nabla)w_{q+1,\nmax} + w_{q+1,\nmax}\cdot \nabla \vlq \notag\\
    &\quad + \sum_{n'\leq \nmax-1} \div\left( w_{q+1,\nmax}\otimes w_{q+1,n'} + w_{q+1,n'}\otimes w_{q+1,\nmax}\right)\notag\\
    &\quad + \div \left( w_{q+1,\nmax} \otimes w_{q+1,\nmax} + \sum_{p=1}^\pmax\RR_{q,\nmax,p} \right).\label{e:expand:nmax}
\end{align}
We absorb the term $\div\left(\RR_{q+1}^{\nmax-1}\right)$ into $\RR_{q+1}$ immediately. We will then show that up to a pressure term,
$$\left(\divH + \divR\right)\left(\left(\partial_t +  \vlq \cdot \nabla \right)w_{q+1,\nmax}\right), \qquad \left(\divH+\divR\right)\left(w_{q+1,\nmax}\cdot\nabla\vlq\right)  $$
can be absorbed into $\RR_{q+1}$ in Sections~\ref{ss:stress:transport} and \ref{ss:stress:Nash}, respectively.  
We will be show in \ref{ss:stress:oscillation:2} that the interaction of $w_{q+1,\nmax}$ with previous perturbations $w_{q+1,n'}$ will satisfy
\begin{equation}\label{nmax:overlap:definition}
\sum_{n'\leq \nmax-1} \left( w_{q+1,\nmax}\otimes w_{q+1,n'} + w_{q+1,n'}\otimes w_{q+1,\nmax}\right) = 0.
\end{equation}
Thus it remains to analyze
$$\div \left( w_{q+1,\nmax} \otimes w_{q+1,\nmax} + \sum_{p=1}^\pmax\RR_{q,\nmax} \right)$$
from \eqref{e:expand:nmax}. Reusing the notations from \eqref{e:sum:neq}--\eqref{e:xiijk}, we can write out the self-interaction of $w_{q+1,\nmax}$ as
\begin{align}
    \div&\left(w_{q+1,\nmax}\otimes w_{q+1,\nmax}\right)\notag\\
    & = \sum_{\xi,i,j,k,p,\vecl} \div \left( \curl\left(a_{(\xi)} \nabla\Phi_{(i,k)}^T \UU_{\xi,q+1,\nmax} \right) \otimes \curl\left( a_{(\xi)}\nabla\Phi_{i,k}^T \UU_{\xi,q+1,\nmax} \right) \right)\notag \\
    &\qquad + \sum_{\neq\{\xi,i,j,k,p,\vecl\}} \div \left( \curl\left(a_{(\xi)} \nabla\Phi_{(i,k)}^T \UU_{\xi,q+1,\nmax} \right) \otimes \curl\left( a_{(\xi')}\nabla\Phi_{(i',k')}^T \UU_{\xi',q+1,\nmax} \right) \right)\notag\\
    &:= \div \mathcal{O}_{\nmax,1} + \div \mathcal{O}_{\nmax,2} \label{e:split:nmax:1}.
\end{align}
As before, we will show in Section~\ref{ss:stress:oscillation:2} that $\mathcal{O}_{\nmax,2}$ is a Type 2 oscillation error and so
\begin{equation*}
\mathcal{O}_{\nmax,2}=0.
\end{equation*}
Splitting $\mathcal{O}_{\nmax,1}$ gives
\begin{align}
    &\div \mathcal{O}_{\nmax,1} = \sum_{\xi,i,j,k,p,\vecl} \div \left( \left( a_{(\xi)} \nabla\Phi_{(i,k)}^{-1} \WW_{\xi,q+1,\nmax}\circ\Phi_{(i,k)} \right) \otimes \left( a_{(\xi)} \nabla\Phi_{(i,k)}^{-1} \WW_{\xi,q+1,\nmax}\circ\Phi_{(i,k)} \right) \right) \notag \\
    &\qquad + 2 \sum_{\xi,i,j,k,p,\vecl} \div \left( \left( a_{(\xi)} \nabla\Phi_{(i,k)}^{-1} \WW_{\xi,q+1,\nmax}\circ\Phi_{(i,k)} \right) \otimes_{\mathrm{s}} \left( \nabla a_{(\xi)} \times \left( \nabla\Phi_{(i,k)}^{T} \UU_{\xi,q+1,\nmax}\circ\Phi_{(i,k)}\right) \right) \right) \notag \\
    &\qquad + \sum_{\xi,i,j,k,p,\vecl} \div \left( \left( \nabla a_{(\xi)} \times \left( \nabla\Phi_{(i,k)}^{T} \UU_{\xi,q+1,\nmax}\circ\Phi_{(i,k)}\right) \right) \otimes \left( \nabla a_{(\xi)} \times \left( \nabla\Phi_{(i,k)}^{T} \UU_{\xi,q+1,\nmax}\circ\Phi_{(i,k)}\right) \right) \right)\notag\\
    &:= \div\left( \mathcal{O}_{\nmax,1,1}+\mathcal{O}_{\nmax,1,2}+\mathcal{O}_{\nmax,1,3} \right).\label{e:split:nmax:2}
\end{align}
The last two of these three terms are again divergence corrector errors and will therefore be absorbed into $\RR_{q+1}$ and estimated in Section~\ref{ss:stress:divergence:correctors}. 

Recall cf.~\eqref{eq:W:xi:q+1:0:def} that $\WW_{\xi,q+1,\nmax}$ is periodized to scale $\left(\lambda_{q+1}r_{q+1,\nmax}\right)^{-1}=\lambda_{q,\nmax}^{-1}$.  
Combining this observation with \eqref{eq:pipes:flowed:2} from Proposition~\ref{pipeconstruction} and \eqref{e:Pqnp:identity}, we further split $\mathcal{O}_{\nmax,1,1}$ as\footnote{In this case, $\mathbb{P}_{\geq\lambda_{q,\nmax}}$ has a greater minimum frequency than $\LPqnpmax$, cf. \eqref{eq:lambda:q:n:0:def}, \eqref{eq:lambda:q:n:def}, and \eqref{def:LPqnp}.  For the sake of consistency, we write $\mathbb{P}_{\geq\lambda_{q,\nmax}} \LPqnpmax$ throughout this section.}
\begin{align}
    &\div\left(\mathcal{O}_{\nmax,1,1}\right) = \sum_{\xi,i,j,k,p,\vecl} \div\left( a_{(\xi)}^2 \nabla\Phi_{(i,k)}^{-1} \left(\dashint_{\mathbb{T}^3}\WW_{\xi,q+1,\nmax}\otimes\WW_{\xi,q+1,\nmax}(\Phi_{(i,k)})\right) \nabla\Phi_{(i,k)}^{-T} \right) \notag \\
    & \quad+ \sum_{\xi,i,j,k,p,\vecl} \div\left( a_{(\xi)}^2 \nabla\Phi_{(i,k)}^{-1}\mathbb{P}_{\geq\lambda_{q,\nmax}} \LPqnpmax (\WW_{\xi,q+1,\nmax}\otimes\WW_{\xi,q+1,\nmax})(\Phi_{(i,k)}) \nabla\Phi_{(i,k)}^{-T} \right)\notag\\
    &= \div \sum_{i,j,k,p,\vecl} \sum_\xi \delta_{q+1,\nmax,p}\Gamma_{q+1}^{2j+4} \eta_{(i,j,k)}^2 \gamma_\xi^2\left(\frac{R_{q,\nmax,p,j,i,k}}{\delta_{q+1,\nmax,p}\Gamma_{q+1}^{2j+4}}\right) \nabla\Phi_{(i,k)}^{-1}\left(\xi\otimes\xi\right)\nabla\Phi_{(i,k)}^{-T} \notag\\
    &\quad + \sum_{\xi,i,j,k,p,\vecl}  \nabla a_{(\xi)}^2\nabla\Phi_{(i,k)}^{-1} \mathbb{P}_{\geq\lambda_{q,\nmax}} \LPqnpmax (\WW_{\xi,q+1,\nmax}\otimes\WW_{\xi,q+1,\nmax})(\Phi_{(i,k)}) \nabla\Phi_{(i,k)}^{-T}\notag\\
    &\quad + \sum_{\xi,i,j,k,p,\vecl}  a_{(\xi)}^2(\nabla\Phiik^{-1})_{\alpha\theta} \mathbb{P}_{\geq\lambda_{q,\nmax}} \LPqnpmax (\WW_{\xi,q+1,\nmax}^\theta\WW_{\xi,q+1,\nmax}^\gamma)(\Phiik) \partial_\alpha(\nabla\Phiik^{-1})_{\zeta\gamma}. \label{eq:euler:reynolds:gross:nmax}
\end{align}
By \eqref{e:split} from Proposition~\ref{p:split} and \eqref{eq:rqnpj}, we obtain that
\begin{align}
    &\div \sum_{i,j,k,p,\vecl} \sum_\xi \delta_{q+1,\nmax,p}\Gamma_{q+1}^{2j+4}\eta_{(i,j,k)}^2 \gamma_\xi^2\left(\frac{R_{q,\nmax,p,j,i,k}}{\delta_{q+1,\nmax,p}\Gamma_{q+1}^{2j+4}}\right)\nabla\Phi_{(i,k)}^{-1}\left(\xi\otimes\xi\right)\nabla\Phi_{(i,k)}^{-T}\notag\\
    &\qquad = \div \sum_{i,j,k,p,\vecl}\eta_{(i,j,k)}^2 \left( \delta_{q+1,\nmax,p}\Gamma_{q+1}^{2j+4}\Id-\RR_{q,\nmax,p}\right)\notag\\
    &\qquad = -\div \sum_{i,j,k,\vecl} \sum_{p=1}^\pmax \eta^2_{(i,j,k)} \RR_{q,\nmax,p} + \nabla \left(\sum_{i,j,k,p,\vecl}\eta_{(i,j,k)}^2 \delta_{q+1,\nmax,p}\Gamma_{q+1}^{2j+4}\right)\notag\\
    &\qquad := -\div  \sum_{p=1}^\pmax   \RR_{q,\nmax,p} + \nabla \pi \,,\label{eq:cancellation:plus:pressure:nmax}
\end{align}
where in the last line we have used \eqref{eq:eta:cut:partition:unity}.   We can apply Proposition~\ref{prop:intermittent:inverse:div} to the remaining two terms in \eqref{eq:euler:reynolds:gross:nmax} to produce the terms
\begin{align}
    &\left(\divH+\divR\right)\bigg{(} \sum_{\xi,i,j,k,p,\vecl} \nabla a_{(\xi)}^2\nabla\Phi_{(i,k)}^{-1} \mathbb{P}_{\geq\lambda_{q,\nmax}} \LPqnpmax (\WW\otimes\WW)_{\xi,q+1,\nmax}(\Phi_{(i,k)}) \nabla\Phi_{(i,k)}^{-T}\notag\\
    &\qquad + \sum_{\xi,i,j,k,p,\vecl} a_{(\xi)}^2(\nabla\Phiik^{-1})_{\alpha\theta} \mathbb{P}_{\geq\lambda_{q,\nmax}} \LPqnpmax (\WW^\theta\WW^\gamma)_{\xi,q+1,\nmax}(\Phiik) \partial_\alpha(\nabla\Phiik^{-1})_{\zeta\gamma}\bigg{)},\label{eq:Hqnpnmax:definition}
\end{align}
which will be absorbed into $\RR_{q+1}$ and estimated in Section~\ref{ss:stress:oscillation:1}.  

Before combining the previous steps, we remind the reader that at this point, $\RR_{q+1}$ will be fully defined, and will include $\RR_{q+1}^{\nmax-1}$, all the error terms arising from the addition of $w_{q+1,\nmax}$, and $\RR_{q}^\textnormal{comm}$.  Then from \eqref{e:expand:nmax}, \eqref{nmax:overlap:definition}, \eqref{e:split:nmax:1}, \eqref{e:split:nmax:2}, \eqref{eq:euler:reynolds:gross:nmax}, \eqref{eq:cancellation:plus:pressure:nmax},  and \eqref{eq:Hqnpnmax:definition}, we can finally write that
\begin{align}
    \partial_t &v_{q+1} + \div \left( v_{q+1}\otimes v_{q+1} \right) + \nabla p_{q,\nmax-1} \notag\\
    &= \div\RR_{q}^{\textnormal{comm}} + \div\left(\RR_{q+1}^{\nmax-1}\right) + (\partial_t + \vlq\cdot\nabla)w_{q+1,\nmax} + w_{q+1,\nmax}\cdot \nabla \vlq \notag\\
    &\quad + \sum_{n'\leq \nmax-1} \div\left( w_{q+1,\nmax}\otimes w_{q+1,n'} + w_{q+1,n'}\otimes w_{q+1,\nmax}\right)\notag\\
    &\quad + \div \left( w_{q+1,\nmax} \otimes w_{q+1,\nmax} + \sum_{p=1}^\pmax\RR_{q,\nmax,p} \right)\notag\\
    &= \div\RR_{q}^{\textnormal{comm}} + \div\bigg{(}\RR_{q+1}^{\nmax-1} +  \left(\divH+\divR\right) \left( \partial_t w_{q+1,\nmax} + \vlq\cdot\nabla w_{q+1,\nmax} \right)\notag\\
    &\qquad  + (\divH+\divR)\left(w_{q+1,\nmax}\cdot \nabla \vlq\right) +  \sum_{n'\leq \nmax-1} \left( w_{q+1,\nmax}\otimes w_{q+1,n'} + w_{q+1,n'}\otimes w_{q+1,\nmax}\right) \bigg{)}\notag\\
    &\qquad + \div \left( \mathcal{O}_{\nmax,1,1} + \mathcal{O}_{\nmax,1,2} + \mathcal{O}_{\nmax,1,3}   + \mathcal{O}_{\nmax,2} + \sum_{p=1}^\pmax\RR_{q,\nmax,p} \right) + \nabla \pi \notag\\
    &=\div\RR_{q}^{\textnormal{comm}} + \div\bigg{(}\RR_{q+1}^{\nmax-1} +  \underbrace{\left(\divH+\divR\right) \left( \partial_t w_{q+1,\nmax} + \vlq\cdot\nabla w_{q+1,\nmax} \right)}_{\textnormal{transport}} \label{eq:id:nmax:1} \\
    &\qquad  + \underbrace{(\divH+\divR)\left(w_{q+1,\nmax}\cdot\nabla \vlq\right)}_{\textnormal{Nash}} +  \underbrace{\sum_{n'\leq \nmax-1} \left( w_{q+1,\nmax}\otimes w_{q+1,n'} + w_{q+1,n'}\otimes w_{q+1,\nmax}\right)}_{\textnormal{Type 2}} \bigg{)} \label{eq:id:nmax:2} \\
    &\qquad + \div \bigg{(}   \underbrace{\mathcal{O}_{\nmax,1,2} + \mathcal{O}_{\nmax,1,3}  s}_{\textnormal{divergence corrector}} + \underbrace{\mathcal{O}_{\nmax,2}}_{\textnormal{Type 2}} \bigg{)} + \nabla \pi \label{eq:id:nmax:3} \\
    & + \div   \underbrace{\left(\divH+\divR\right)\bigg{(} \sum_{\xi,i,j,k,p,\vecl} \nabla a_{(\xi)}^2\nabla\Phi_{(i,k)}^{-1} \mathbb{P}_{\geq\lambda_{q,\nmax}} \LPqnpmax (\WW\otimes\WW)_{\xi,q+1,\nmax}(\Phi_{(i,k)}) \nabla\Phi_{(i,k)}^{-T}}_{\textnormal{Type 1}} \label{eq:id:nmax:4} \\
    &\qquad\qquad + \underbrace{\sum_{\xi,i,j,k,p} a_{(\xi)}^2(\nabla\Phiik^{-1})_{\alpha\theta} \mathbb{P}_{\geq\lambda_{q,\nmax}} \LPqnpmax (\WW^\theta\WW^\gamma)_{\xi,q+1,\nmax}(\Phiik) \partial_\alpha(\nabla\Phiik^{-1})_{\zeta\gamma}\bigg{)}}_{\textnormal{Type 1}} \label{eq:id:nmax:5} \\
    &= \div(\RR_{q+1}) + \nabla \pi, \notag
\end{align}
concluding the proof after again noting that $\left(\partial_t + \vlq\cdot\nabla\right)w_{q+1,\nn}$ has zero mean.

\subsection{Transport errors}\label{ss:stress:transport}

\begin{lemma}\label{l:transport:error}
For all $0\leq \nn \leq \nmax$, the transport errors satisfy
$$ \Dtq w_{q+1,\nn} = \partial_t w_{q+1,\nn} + \vlq \cdot \nabla w_{q+1,\nn} = \div   \left( \divH + \divR \right)\left(  \partial_t w_{q+1,\nn} + \vlq \cdot \nabla w_{q+1,\nn} \right) + \nabla p_\nn $$
with the estimates
\begin{align*}
&\left\| \psi_{i,q} D^N \Dtq^M \left( \left( \divH + \divR \right)\left(  \partial_t w_{q+1,\nn} + \vlq \cdot \nabla w_{q+1,\nn} \right) \right) \right\|_{L^1} \notag\\
&\qquad \qquad  \lesssim \delta_{q+2}\Gamma_{q+1}^{\shaq-1} \lambda_{q+1}^N \MM{M,\Nindt, \tau_q^{-1}\Gamma_{q+1}^{i+1},\Gamma_{q+1}^{-1}\tilde\tau_q^{-1}}
\end{align*}
for all $N,M\leq 3\NindLarge$.
\end{lemma}

\begin{proof}[Proof of Lemma~\ref{l:transport:error}]
The transport errors are given in \eqref{eq:id:0:1}, \eqref{eq:id:nn:1}, and \eqref{eq:id:nmax:1}. Writing out the transport error, we have that
\begin{align}
    \left(\partial_t + \vlq \cdot \nabla \right) w_{q+1,\nn} &=\left( \partial_t + \vlq \cdot \nabla \right) \left(\sum_{i,j,k,\pp,\vecl,\xi} \curl \left( a_{\xi,i,j,k,q,\nn,\pp,\vecl} \nabla\Phi_{(i,k)}^{T}\UU_{\xi,q+1,\nn} \circ \Phi_{(i,k)} \right) \right) \nonumber\\
    &= \sum_{i,j,k,\pp,\vecl,\xi}\left( \partial_t + \vlq \cdot \nabla \right)\left(a_{(\xi)} \nabla\Phi_{(i,k)}^{-1}\right) \WW_{\xi,q+1,\nn} \circ \Phi_{(i,k)} \notag\\
    &\qquad + \sum_{i,j,k,\pp,\vecl,\xi} \left( \left(\partial_t+\vlq\cdot\nabla\right) \nabla a_{(\xi)} \right) \times \left( \nabla\Phiik \UU_{\xi,q+1,\nn}\circ \Phiik \right) \notag\\
    &\qquad + \sum_{i,j,k,\pp,\vecl,\xi} \nabla a_{(\xi)} \times \left( \left( \left(\partial_t+\vlq\cdot\nabla\right) \nabla\Phiik \right)\UU_{\xi,q+1,\nn}\circ \Phiik \right) \label{eq:transport:estimate:1}
\end{align}
Due to the fact that the second two terms arise from the addition of the corrector defined in \eqref{wqplusoneonec} and \eqref{wqplusonenpc}, and the fact that the bounds for the corrector in \eqref{eq:w:oxi:c:est} are stronger than that of the principal part of the perturbation, we shall completely estimate only the first term and simply indicate the set-up for the second and third. Before applying Proposition~\ref{prop:intermittent:inverse:div}, recall that the inverse divergence of \eqref{eq:transport:estimate:1} needs to be estimated on the support of a cutoff $\psi_{i,q}$ in order to verify \eqref{e:inductive:n:1:Rstress}, and \eqref{e:inductive:n:2:Rstress}, and \eqref{e:inductive:n:3:Rstress}. Recall from the identification of the error terms (cf. \eqref{eq:id:zero:mean} and the subsequent argument) that for all $\nn$, $\left(\partial_t+\vlq\cdot\nabla\right)w_{q+1,\nn}$ has zero mean.  Thus, although each individual term in the final equality in \eqref{eq:transport:estimate:1} may not have zero mean, we can safely apply $\divH$ and $\divR$ to each term and estimate the outputs while ignoring the last term in \eqref{eq:inverse:div:error:stress}.

We will apply Proposition~\ref{prop:intermittent:inverse:div}, specifically Remark~\ref{rem:div:derivative:bounds}, to each summand in the first term on the right side of \eqref{eq:transport:estimate:1}, with the following choices.  We set $v=\vlq$, and $D_t=\Dtq=\partial_t+\vlq\cdot\nabla$ as usual.  We set $N_*=M_*=\lfloor \sfrac{1}{2}\left(\Nfnn-\NcutSmall-\NcutLarge-5\right) \rfloor$, with $\Ndec$ and $\dpot$ satisfying \eqref{eq:lambdaqn:identity:3}.  We define
$$  G = (\partial_t + \vlq \cdot \nabla)(a_{(\xi)}\nabla\Phi_{(i,k)}^{-1}) \xi, $$
with $\lambda=\Gamma_{q+1}\lambda\qnnpp$, $\nu=\tau_q^{-1}\Gamma_{q+1}^{i-\cstarnn+3}$, $M_t=\Nindt$, $\tilde\nu=\tilde\tau_{q}^{-1}\Gamma_{q+1}^{-1}$, and 
$$\const_G = \bigl|\supp( \eta_{i,j,k,q,\nn,\pp,\vec{l}}) \bigr| 
\delta_{q+1,\nn,1}^{\sfrac{1}{2}} \Gamma_{q+1}^{i-\cstarnn+j+5} \tau_q^{-1} \,, $$
which is the correct amplitude in view of \eqref{e:a_master_est_p} with $r=1$ and $r_1=r_2=2$,  and \eqref{eq:Lagrangian:Jacobian:6}.
Thus, we have that 
\begin{align}
 \left\| D^N \Dtq^M G \right\|_{L^1}
 \lessg \const_G \left(\lambda\qnnpp\Gamma_{q+1}\right)^N \MM{M,\Nindt-1,\tau_q^{-1}\Gamma_{q+1}^{i-\cstarnn+3},\tilde\tau_q^{-1}\Gamma_{q+1}^{-1}}  ,\label{eq:david:transport:0}
\end{align}
for all $N,M \leq \lfloor \sfrac{1}{2}\left(\Nfnn-\NcutSmall-\NcutLarge-5\right) \rfloor $ after using \eqref{eq:cstarn:inequality} and \eqref{eq:Nind:cond:2},
and so \eqref{eq:inverse:div:DN:G} is satisfied.  We set $\Phi=\Phi_{i,k}$  and $\lambda'=\tilde\lambda_q$.  Appealing as usual to Corollary~\ref{cor:deformation} and \eqref{eq:nasty:D:vq}, we have that \eqref{eq:DDpsi} and \eqref{eq:DDv} are satisfied.

Referring to \eqref{item:pipe:1} from Proposition~\ref{pipeconstruction}, we set $\varrho=\varrho_{\xi,\lambda_{q+1},r_{q+1,\nn}}$ and $\vartheta=\vartheta_{\xi,\lambda_{q+1},r_{q+1,\nn}}$.  Setting $\zeta=\lambda_{q+1}$, we have that \eqref{item:inverse:i} is satisfied.  Setting $\mu=\lambda_{q+1}r_{q+1,\nn}=\lambda\qnn$ and referring to \eqref{item:pipe:2} from Proposition~\ref{pipeconstruction}, we have that \eqref{item:inverse:ii} is satisfied.  Setting $\Lambda=\zeta=\lambda_{q+1}$ and $C_*=r_{q+1,\nn}$ and referring to \eqref{e:pipe:estimates:1} and \eqref{e:pipe:estimates:2} from Proposition~\ref{pipeconstruction}, we have that \eqref{eq:DN:Mikado:density} is satisfied. \eqref{eq:inverse:div:parameters:0} is immediate from the definitions. Referring to \eqref{eq:lambdaqn:identity:2}, we have that \eqref{eq:inverse:div:parameters:1} is satisfied. Thus, we conclude from  \eqref{eq:inverse:div:stress:1} with $\alpha_{\mathsf R}$ as in \eqref{eq:alpha:equation:1},  that for $N,M \leq\lfloor \sfrac{1}{2}\left(\Nfnn-\NcutSmall-\NcutLarge-5\right) \rfloor -\dpot$,
\begin{align*}
&\left\| D^N \Dtq^M \left( \divH \left( (\partial_t + \vlq \cdot \nabla)(a_{(\xi)}\nabla\Phi_{(i,k)}^{-1}) \xi \right) \right) \right\|_{L^1} =\left\|  D^N \Dtq^M \left( \divH \left( G \varrho\circ\Phi \right)\right) \right\|_{L^1} \notag\\
&\qquad \lessg 
\bigl|\supp( \eta_{i,j,k,q,\nn,\pp,\vec{l}}) \bigr| 
\delta_{q+1,\nn,1}^{\sfrac{1}{2}} \Gamma_{q+1}^{i-\cstarnn+j+6} \tau_q^{-1}
r_{q+1,\nn} \lambda_{q+1}^{-1} \lambda_{q+1}^N \MM{M,\Nindt,\tau_q^{-1}\Gamma_{q+1}^{i},\tilde\tau_q^{-1}\Gamma_{q+1}^{-1}} \,,
\end{align*}
after appealing to  \eqref{eq:cstarn:inequality}.
From \eqref{eq:nfnn:mess}, these bounds are valid for all $N,M\leq 3\NindLarge$.
The bound obtained above is next summed over $(i,j,k,\pp,\nn,\vecl)$. First, we treat the sum over $\vecl$. By noting that \eqref{item:lebesgue:1} with $r_1= 2$ and $r_2 = 2$, and \eqref{eq:cstarn:inequality} imply 
\begin{equation*}
\sum_{\vecl} \bigl|\supp( \eta_{i,j,k,q,\nn,\pp,\vec{l}}) \bigr|  \Gamma_{q+1}^{i-\cstarnn+j+6}
\leq \Gamma_{q+1}^{-2\left( \frac{i}{2} + \frac{j}{2} \right) + \frac{\CLebesgue}{2} +2 }   \Gamma_{q+1}^{i-\cstarnn+j+6}
=  \Gamma_{q+1}^{ \frac{\CLebesgue}{2} +3 }  
\,,
\end{equation*}
we conclude that 
\begin{align}
&\left\| D^N \Dtq^M \left( \divH \left( \partial_t w_{q+1,\nn} + \vlq \cdot \nabla w_{q+1,\nn} \right) \right) \right\|_{L^1\left(\supp \psi_{i,q}\right)} \notag\\
&\qquad 
\lessg \sum_{i'=i-1}^{i+1} \sum_{j,k,\pp,\xi}  
\Gamma_{q+1}^{ \frac{\CLebesgue}{2} +3 }
\delta_{q+1,\nn,1}^{\sfrac{1}{2}} \tau_q^{-1}
r_{q+1,\nn} \lambda_{q+1}^{-1} \lambda_{q+1}^N \MM{M,\Nindt,\tau_q^{-1}\Gamma_{q+1}^{i'},\tilde\tau_q^{-1}\Gamma_{q+1}^{-1}} 
\notag\\
&\qquad \lessg \Gamma_{q+1}^{4+\frac{\CLebesgue}{2}} \delta_{q+1,\nn,1}^\frac{1}{2}\tau_q^{-1} r_{q+1,\nn} \lambda_{q+1}^{-1} \lambda_{q+1}^N \MM{M,\Nindt,\tau_q^{-1}\Gamma_{q+1}^{i+1},\tilde\tau_q^{-1}\Gamma_{q+1}^{-1}} \notag\\
&\qquad \lesssim \Gamma_{q+1}^{\shaq-1} \delta_{q+2} \lambda_{q+1}^N  \MM{M,\Nindt,\tau_q^{-1}\Gamma_{q+1}^{i+1},\tilde\tau_q^{-1}\Gamma_{q+1}^{-1}}  \label{eq:david:transport:2}
\end{align}
after also using \eqref{eq:drq:identity}.

To finish the proof for the first term in \eqref{eq:transport:estimate:1}, we must provide a matching estimate for the $\divR$ portion.  Following again the parameter choices in Remark~\ref{rem:div:derivative:bounds}, we set $N_\circ=M_\circ=3\NindLarge$.  As in the argument from Lemma~\ref{lem:oscillation:general:estimate}, we have that \eqref{eq:inverse:div:v:global}, \eqref{eq:inverse:div:v:global:parameters}, and \eqref{eq:riots:4} are satisfied, this time with $\zeta=\lambda_{q+1}$.  Thus we achieve the estimate in \eqref{eq:inverse:div:error:stress:bound}. Summing over $\vec{l}$ loses a factor less than $\lambda_{q+1}^3$, while summing over the other indices costs a constant independent of $q$. This completes the estimate for the first term from \eqref{eq:transport:estimate:1}.

For the second and third terms, we explain how to identify $G$ and $\varrho$ in order to give an idea of how to obtain similar estimates. Using \ref{item:pipe:1} from Proposition~\ref{pipeconstruction} and the vector calculus identity $\curl \curl = \nabla   \div - \Delta$, we obtain that
\begin{equation}\label{eq:transport:u}
\UU_{\xi,q+1,\nn} = \curl \left( \xi \lambda_{q+1}^{-2\dpot} \Delta^{\dpot-1} \left( \vartheta_{\xi,\lambda_{q+1},r_{q+1,\nn}} \right) \right) = \lambda_{q+1}^{-2\dpot} \xi \times \nabla \left( \Delta^{\dpot-1} \left( \vartheta_{\xi,\lambda_{q+1},r_{q+1,\nn}} \right) \right).
\end{equation}
With a little massaging, one can now rewrite the second and third terms in \eqref{eq:transport:estimate:1} in the form $G \varrho\circ\Phiik$.  Since both terms have traded a spatial derivative on $\UU_{\xi,q+1,\nn}$ for a spatial derivative on $a_{(\xi)}$, inducing a gain, one can easily show that the estimates for these terms will be even stronger than those for the first term.  Notice that we have set $N_*=M_*= \lfloor \sfrac{1}{2}\left(\Nfnn-\NcutSmall-\NcutLarge-7\right) \rfloor $ since we have lost a spatial derivative on $a_{(\xi)}$. We omit the rest of the details.
\end{proof}

\subsection{Nash errors}\label{ss:stress:Nash}

\begin{lemma}\label{l:Nash:error}
For all $0\leq \nn \leq \nmax$, the Nash errors satisfy
$$ w_{q+1,\nn}\cdot\nabla\vlq = \div \left( \left( \divH + \divR \right) w_{q+1,\nn} \cdot \nabla \vlq \right) + \nabla p_\nn $$
with
\begin{align*}
\left\| \psi_{i,q} D^k \Dtq^m \left( \left( \divH + \divR \right) w_{q+1,\nn} \cdot \nabla \vlq \right) \right\|_{L^1} \lesssim \delta_{q+2}\Gamma_{q+1}^{\shaq-1} \lambda_{q+1}^N \MM{M,\Nindt, \tau_q^{-1}\Gamma_{q+1}^{i+1},\Gamma_{q+1}^{-1}\tilde\tau_q^{-1}}
\end{align*}
for all $N,M\leq 3\NindLarge$.
\end{lemma}
\begin{proof}[Proof of Lemma~\ref{l:Nash:error}]
The estimates are similar to those in Lemma~\ref{l:transport:error}. Writing out the Nash error, we have that
\begin{align}
    w_{q+1,\nn}\cdot\nabla\vlq &= \sum_{i-1 \leq i' \leq i+1} \sum_{j,k,\pp,\vecl,\xi} \notag \curl\left( a_{\xi,i,j,k,q,\nn} \nabla\Phiik^T \UU_{\xi,q+1,\nn} \circ \Phiik \right) \\
    &= \left(\sum_{i,j,k,\pp,\vecl,\xi} \nabla a_{(\xi)} \times \left( \Phiik^T \UU_{\xi,q+1,\nn} \circ \Phiik \right) \right) \cdot \nabla \vlq \notag\\
    &\qquad + \left( \sum_{i,j,k,\pp,\vecl,\xi} a_{(\xi)} \nabla\Phi_{(i,k)}^{-1}\WW_{\xi,q+1,\nn} \circ \Phi_{(i,k)} \right) \cdot \nabla\vlq. \label{eq:Nash:estimate:1}
\end{align}
Due to the fact that the first term arises from the addition of the corrector defined in \eqref{wqplusoneonec} and \eqref{wqplusonenpc}, and the fact that the bounds for the corrector in \eqref{eq:w:oxi:c:est} are stronger than that of the principal part of the perturbation, we shall completely estimate only the second term and simply indicate the set-up for the first. Before applying Proposition~\ref{prop:intermittent:inverse:div}, recall that the inverse divergence of \eqref{eq:transport:estimate:1} needs to be estimated on the support of a cutoff $\psi_{i,q}$ in order to verify \eqref{e:main:inductive:q:stress}, \eqref{e:inductive:n:1:Rstress}, and \eqref{e:inductive:n:2:Rstress}. Note that the Nash error can be written as $\div \left( w_{q+1,\nn}\cdot\vlq \right)$ and so has zero mean.  Thus, although each individual term in the final equality in \eqref{eq:Nash:estimate:1} may not have zero mean, we can safely apply $\divH$ and $\divR$ to each term and estimate the outputs while ignoring the last term in \eqref{eq:inverse:div:error:stress}.

We will apply Proposition~\ref{prop:intermittent:inverse:div} to the second term with the following choices.  We set $v=\vlq$, and $D_t=\Dtq=\partial_t + \vlq\cdot\nabla$ as usual.  We set $N_*=M_*= \lfloor \sfrac{1}{2}\left(\Nfnn-\NcutLarge-\NcutSmall-4\right)\rfloor $, with $\Ndec$ and $\dpot$ satisfying \eqref{eq:lambdaqn:identity:3}.  We define
\begin{equation*}
G= a_{(\xi)} \nabla\Phi_{(i,k)}^{-1} \xi \cdot \nabla \vlq
\end{equation*}
and set 
$$\const_G =  \bigl|\supp( \eta_{i,j,k,q,\nn,\pp,\vec{l}}) \bigr| 
\delta_{q+1,\nn,1}^{\sfrac{1}{2}} \Gamma_{q+1}^{i-\cstarnn+j+5} \tau_q^{-1} \, ,$$
 $\lambda=\Gamma_{q+1}\lambda\qnnpp$, $\nu=\tau_q^{-1}\Gamma_{q+1}^{i-\cstarnn+3}$, $M_t=\Nindt$, and $\tilde\nu=\tilde\tau_q^{-1}\Gamma_{q+1}^{-1}$.  From \eqref{e:a_master_est_p} with $r=1$ and $r_1=r_2=2$, \eqref{eq:Lagrangian:Jacobian:6}, and \eqref{eq:nasty:D:vq}, we have that for $N,M\leq \lfloor \sfrac{1}{2}\left(\Nfnn-\NcutLarge-\NcutSmall-4\right)\rfloor $
\begin{align}
\left\| D^N \Dtq^M G \right\|_{L^1}  \lessg \const_G \left(\Gamma_{q+1}\lambda\qnnpp\right)^N \MM{M,\Nindt,\tau_q^{-1}\Gamma_{q+1}^{i-\cstarnn+3},\tilde\tau_q^{-1}\Gamma_{q+1}^{-1}} , \label{eq:david:nash:1}
\end{align}
and so \eqref{eq:inverse:div:DN:G} is satisfied.  Note that we have used \eqref{eq:Lambda:q:x:1:NEW} when converting the $\delta_q^{\sfrac{1}{2}} \tilde\lambda_q$ to a $\tau_q^{-1}$.  Setting $\Phi=\Phi_{(i,k)}$ and $\lambda'=\tilde\lambda_q$, we have that \eqref{eq:DDpsi} and \eqref{eq:DDv} are satisfied as usual.  The choices of $\varrho$, $\vartheta$, $\zeta$, $\mu$, $\Lambda$, and $\const_*$ are identical to those of the transport error (both terms contain $\WW_{\xi,q+1,\nn}\circ\Phiik$), and so we have that \eqref{item:inverse:i}-\eqref{item:inverse:ii}, \eqref{eq:DN:Mikado:density}, \eqref{eq:inverse:div:parameters:0}, and \eqref{eq:inverse:div:parameters:1} are satisfied as well.  Since the bound \eqref{eq:david:nash:1} is identical to that of \eqref{eq:david:transport:0}, we obtain an estimate identical to \eqref{eq:david:transport:2}.  The argument for the $\divR$ portion follows analogously to that for the first term from the transport error.  Finally, after using \eqref{eq:transport:u} again, one may obtain similar estimates for the first term in \eqref{eq:Nash:estimate:1}, concluding the proof. 
\end{proof}

\subsection{Type 1 oscillation errors}\label{ss:stress:oscillation:1}

The Type 1 oscillation errors are defined in the three parameter regimes $\nn=0$, $1\leq \nn \leq \nmax-1$, and $\nn=\nmax$.  In the case $\nn=0$, Type 1 oscillation errors stem from the term identified in \eqref{eq:id:0:0}, which we recall is
\begin{align}
 &(\divH+\divR) \bigg{(} \sum_{\xi,i,j,k,\pp,\vecl} \nabla a_{(\xi)}^2\nabla\Phi_{(i,k)}^{-1} \mathbb{P}_{\geq\lambda_{q,0}}   \Bigl (\sum\limits_{n=1}^{\nmax}\sum\limits_{p=1}^\pmax \LPqnp+\LPqnpmax\Bigr) \notag\\
 &\quad \qquad \qquad \qquad \times  (\WW_{\xi,q+1,0}\otimes\WW_{\xi,q+1,0})(\Phi_{(i,k)}) \nabla\Phi_{(i,k)}^{-T}\notag\\
    &\qquad\qquad + \sum_{\xi,i,j,k,\pp,\vecl} a_{(\xi)}^2(\nabla\Phiik^{-1})_{\alpha\theta} \mathbb{P}_{\geq\lambda_{q,0}}   \Bigl(\sum\limits_{n=1}^{\nmax}\sum\limits_{p=1}^\pmax \LPqnp+\LPqnpmax\Bigr) \notag\\
    &\quad \qquad \qquad \qquad \times  (\WW_{\xi,q+1,0}^\theta\WW_{\xi,q+1,0}^\gamma)(\Phiik) \partial_\alpha(\nabla\Phiik^{-1})_{\zeta\gamma} \bigg{)} \label{eq:type:1:0}.
\end{align}
This sum is divided into the terms identified in \eqref{eq:id:0:2}, \eqref{eq:id:0:3}, \eqref{eq:id:0:4}, \eqref{eq:id:0:5}, \eqref{eq:id:0:7}, and \eqref{eq:id:0:8}.  The errors defined in \eqref{eq:id:0:7} and \eqref{eq:id:0:8} are $\HH_{q,n,p}^0$ errors and will be corrected by later perturbations $w_{q+1,n,p}$, while the others will be immediately absorbed into $\RR_{q+1}^0$.

In the case $1\leq \nn\leq \nmax-1$, Type 1 oscillation errors stem from the term identified in \eqref{eq:id:nn:0}
\begin{align}
&  \left(\divH+\divR\right)  \Biggl{(} \sum_{\xi,i,j,k,\pp,\vecl} \nabla a_{(\xi)}^2\nabla\Phi_{(i,k)}^{-1} \Pqnn   \Bigl(\sum\limits_{n=\nn+1}^{\nmax}\sum\limits_{p=1}^\pmax \LPqnp+\LPqnpmax\Bigr) \notag\\
&\quad \qquad\qquad \qquad \times  (\WW_{\xi,q+1,\nn}\otimes\WW_{\xi,q+1,\nn})(\Phi_{(i,k)}) \nabla\Phi_{(i,k)}^{-T} \notag \\
    &\qquad\qquad   + \sum_{\xi,i,j,k,\pp,\vecl} a_{(\xi)}^2(\nabla\Phiik^{-1})_{\alpha\theta} \Pqnn   \Bigl(\sum\limits_{n=\nn+1}^{\nmax}\sum\limits_{p=1}^\pmax \LPqnp+\LPqnpmax\Bigr)  \notag\\
    & \quad  \qquad\qquad \qquad\times (\WW_{\xi,q+1,\nn}^\theta\WW_{\xi,q+1,\nn}^\gamma)(\Phiik) \partial_\alpha(\nabla\Phiik^{-1})_{\zeta\gamma} \Biggr{)} \label{eq:type:1:nn}.
\end{align}
This sum is divided into the terms identified in \eqref{eq:id:nn:4}, \eqref{eq:id:nn:5}, \eqref{eq:id:nn:6}, \eqref{eq:id:nn:7}, \eqref{eq:id:nn:8}, and \eqref{eq:id:nn:9}.  As before, the last two terms are $\HH\qnp^{\nn}$ errors and will be corrected by later perturbations, while the others are absorbed into $\RR_{q+1}^\nn$.  

In the case $\nn=\nmax$, Type 1 oscillation errors are identified in \eqref{eq:id:nmax:4} and \eqref{eq:id:nmax:5} as
\begin{align}
 &{\left(\divH+\divR\right)\Bigg{(} \sum_{\xi,i,j,k,p,\vecl} \nabla a_{(\xi)}^2\nabla\Phi_{(i,k)}^{-1} \mathbb{P}_{\geq\lambda_{q,\nmax}} \LPqnpmax (\WW_{\xi,q+1,\nmax}\otimes\WW_{\xi,q+1,\nmax})(\Phi_{(i,k)}) \nabla\Phi_{(i,k)}^{-T}} \notag \\
    &\quad + {\sum_{\xi,i,j,k,p,\vecl} a_{(\xi)}^2(\nabla\Phiik^{-1})_{\alpha\theta} \mathbb{P}_{\geq\lambda_{q,\nmax}} \LPqnpmax (\WW_{\xi,q+1,\nmax}^\theta\WW_{\xi,q+1,\nmax}^\gamma)(\Phiik) \partial_\alpha(\nabla\Phiik^{-1})_{\zeta\gamma}\Bigg{)}}\label{eq:type:1:nmax}.
\end{align}
These errors are completely absorbed into $\RR_{q+1}$.

To prove the desired estimates on these error terms, we will first analyze a single term of the form
\begin{align}
 &(\divH+\divR) \Bigg{(} \sum_{\xi,i,j,k,\pp,\vecl} \nabla a_{(\xi)}^2\nabla\Phi_{(i,k)}^{-1} \Pqnn  \LPqnp (\WW_{\xi,q+1,\nn}\otimes\WW_{\xi,q+1,\nn})(\Phi_{(i,k)}) \nabla\Phi_{(i,k)}^{-T}\notag\\
    &\qquad\qquad\qquad + \sum_{\xi,i,j,k,\pp,\vecl} a_{(\xi)}^2(\nabla\Phiik^{-1})_{\alpha\theta} \Pqnn \LPqnp (\WW_{\xi,q+1,\nn}^\theta\WW_{\xi,q+1,\nn}^\gamma)(\Phiik) \partial_\alpha(\nabla\Phiik^{-1})_{\zeta\gamma} \Bigg{)} \notag\\
    &\qquad =: \left(\divH + \divR \right) \Onpnp \, . \label{eq:type:1:general}
\end{align}
The estimates in Lemma~\ref{lem:oscillation:general:estimate} for this term on the support of a cutoff function $\psi_{i,q}$ will depend on $\nn$ and $\pp$, which range from $0\leq \nn \leq \nmax$ and $1\leq\pp\leq\pmax$, respectively, and $n$ and $p$, which range from $\nn+1\leq n \leq \nmax$ and $1\leq p\leq\pmax$, with the additional endpoint case $n=\nmax$, $p=\pmax+1$.  We then use this general estimate to specify in Remark~\ref{rem:oscillation:general} how the terms corresponding to various values of $n$, $\nn$, $p$, and $\pp$ are absorbed into either higher order stresses $\HH\qnp^\nn$ or $\RR_{q+1}^\nn$, and eventually $\RR_{q+1}$.

\begin{lemma}\label{lem:oscillation:general:estimate}
The terms $\Onpnp$ defined in \eqref{eq:type:1:general} satisfy the following.
\begin{enumerate}[(1)]
\item\label{item:Onpnp:1}  For the special case $n=\nmax$, $p=\pmax+1$ and all $0\leq \nn\leq \nmax$, $1\leq \pp \leq \pmax$, as well as for all cases $0 \leq \nn<n \leq \nmax$, $1\leq p,\pp \leq\pmax$, the nonlocal portion of the inverse divergence satisfies
\begin{equation}\label{eq:Onpnp:estimate:1}
\left\| D^N \Dtq^M \left(\divR \Onpnp \right)\right\|_{L^1\left(\mathbb{T}^3\right)} \leq \frac{\delta_{q+2}}{\lambda_{q+1}} \lambda_{q+1}^N \tau_q^{-M} \,
\end{equation}
for all $N,M\leq 3\NindLarge$.
\item\label{item:Onpnp:2}  For $n=\nmax$, $p=\pmax+1$, all $0 \leq \nn\leq \nmax$ and $1\leq \pp\leq \pmax$, the high frequency, local portion of the inverse divergence satisfies
\begin{align}
&\left\| D^N \Dtq^M \left(\divH \mathcal{O}_{\nn,\pp,\nmax,\pmax+1}\right) \right\|_{L^1\left(\supp\psi_{i,q}\right)} \notag\\
&\qquad \qquad \lesssim \Gamma_{q+1}^\shaq \Gamma_{q+1}^{-1} \delta_{q+2} \lambda_{q+1}^N \MM{M, \Nindt, \tau_q^{-1}\Gamma_{q+1}^{i-\cstarnn+4}, \Gamma_{q+1}^{-1}\tilde\tau_q^{-1}}\label{eq:Onpnp:estimate:2}
\end{align}
for all $N,M\leq 3\NindLarge$.
\item\label{item:Onpnp:3} For $0 \leq \nn<  n\leq\nmax$ and $1\leq p, \pp \leq\pmax$, the medium frequency, local portion of the inverse divergence satisfies
\begin{align}
\left\| D^N \Dtq^M \left( \divH \Onpnp \right) \right\|_{L^1\supp\left(\psi_{i,q}\right)} & \lesssim \delta_{q+1,n,p} \lambda\qnp^N \MM{M, \Nindt, \tau_q^{-1}\Gamma_{q+1}^{i-\cstarnn+4}, \Gamma_{q+1}^{-1}\tilde\tau_q^{-1}}    \label{eq:Onpnp:estimate:3}
\end{align}
for all $N+M \leq \Nfn$.
\end{enumerate}
\end{lemma}
\begin{remark}\label{rem:oscillation:general}
Note that after appealing to $\nn \leq n-1$, \eqref{def:cstarn:formula}, and \eqref{eq:cstarn:inequality}, \eqref{eq:Onpnp:estimate:3} matches \eqref{e:inductive:n:1:Hstress}, \eqref{e:inductive:n:2:Hstress}, and \eqref{e:inductive:n:3:Hstress}, or equivalently \eqref{eq:Rn:inductive:assumption}.  In addition, after appealing again to $\nn \leq n-1$, \eqref{def:cstarn:formula}, and \eqref{eq:cstarn:inequality}, \eqref{eq:Onpnp:estimate:1} and \eqref{eq:Onpnp:estimate:2} are sufficient to meet \eqref{e:inductive:n:1:Rstress}, \eqref{e:inductive:n:2:Rstress}, and \eqref{e:inductive:n:3:Rstress}.
\end{remark}
\begin{proof}[Proof of Lemma~\ref{lem:oscillation:general:estimate}]
The first step is to use item \eqref{item:pipe:1} and \eqref{eq:pipes:flowed:2} from Proposition~\ref{pipeconstruction} to rewrite \eqref{eq:type:1:general} as
\begin{align}
 &(\divH+\divR) \Bigg{(} \sum_{\xi,i,j,k,\pp,\vecl} \nabla a_{(\xi)}^2\nabla\Phi_{(i,k)}^{-1} \Pqnn  \LPqnp (\WW_{\xi,q+1,\nn}\otimes\WW_{\xi,q+1,\nn})(\Phi_{(i,k)}) \nabla\Phi_{(i,k)}^{-T}\notag\\
    &\qquad +\sum_{\xi,i,j,k,\pp,\vecl} a_{(\xi)}^2(\nabla\Phiik^{-1})_{\theta\alpha} \Pqnn \LPqnp (\WW_{\xi,q+1,\nn}^\theta\WW_{\xi,q+1,\nn}^\gamma)(\Phiik) \partial_\alpha(\nabla\Phiik^{-1})_{\gamma\kappa} \Bigg{)} \notag\\
    &= \left(\divH + \divR \right) \Bigg{(}  \sum_{\xi,i,j,k,\pp,\vecl} \Pqnn   \LPqnp \left( \left(\varrho_{\xi,\lambda_{q+1},r_{q+1,\nn}}\right)^2 \right)(\Phi_{(i,k)}) \notag \\
    &\qquad \times \bigg{(}  \partial_\alpha a_{(\xi)}^2 \left(\nabla\Phi_{(i,k)}^{-1}\right)_{\gamma\kappa} \xi^\theta \xi^\gamma \left(\nabla\Phi_{(i,k)}^{-T}\right)_{\theta\alpha} + a_{(\xi)}^2 \left( \nabla \Phi_{(i,k)} \right)^{-1})_{\theta\alpha} \xi^\theta \xi^\gamma \partial_\alpha\left( \nabla\Phi_{(i,k)}^{-1} \right)_{\gamma\kappa} \bigg{)} \Bigg{)} \label{eq:type:1:general:a}.
\end{align}
Next, we must identify the functions and the values of the parameters which will be used in the application of Proposition~\ref{prop:intermittent:inverse:div}, specifically Remark~\ref{rem:div:derivative:bounds}.  We first address the bounds required in \eqref{eq:inverse:div:DN:G}, \eqref{eq:DDpsi}, and \eqref{eq:DDv}, which we can treat simultaneously for items \eqref{item:Onpnp:1}, \eqref{item:Onpnp:2}, and \eqref{item:Onpnp:3}.  Afterwards, we split the proof into two parts. First, we set $n=\nmax$, $p=\pmax+1$ and prove \eqref{eq:Onpnp:estimate:1} for \emph{only} these specific values of $n$ and $p$, as we simultaneously prove \eqref{eq:Onpnp:estimate:2}.  Next, we consider $n<\nmax$ and prove \eqref{eq:Onpnp:estimate:1} in the remaining cases, as we simultaneously prove \eqref{eq:Onpnp:estimate:3}.

Returning to \eqref{eq:inverse:div:DN:G}, we will verify that this inequality holds with  $v=\vlq$, $D_t=\Dtq=\partial_t + \vlq \cdot \nabla$, and $\displaystyle N_* = M_* = \lfloor \sfrac{\Nsharp}{2} \rfloor$, where $\Nsharp=\Nfnn-\NcutSmall-\NcutLarge-5$.  In order to verify the assumption $N_*- \dpot \geq 2\Ndec + 4$, we use that $\Ndec$ and $\dpot$ satisfy \eqref{eq:lambdaqn:identity:3}, which gives that
\begin{equation}\label{eq:d:Ndec:inequality}
2\Ndec+4 \leq \lfloor \sfrac{1}{2}\left( \Nfnn-\NcutSmall-\NcutLarge-5 \right) - \dpot \rfloor \,.
\end{equation}
Denoting the $\kappa^\textnormal{th}$ component of the below vector field $G$ by $G_\kappa$, we \emph{fix} a value of $(\xi,i,j,k,\pp,\vecl)$ and set
\begin{align}
    G_\kappa =  \partial_\alpha a_{(\xi)}^2 \left(\nabla\Phi_{(i,k)}^{-1}\right)_{\gamma\kappa} \xi^\theta \xi^\gamma \left(\nabla\Phi_{(i,k)}^{-T}\right)_{\alpha\theta}   + a_{(\xi)}^2 \left( \nabla \Phi_{(i,k)} \right)^{-1})_{\alpha\theta} \xi^\theta \xi^\gamma \partial_\alpha\left( \nabla\Phi_{(i,k)}^{-1} \right)_{\kappa \gamma} \, . \label{eq:f:zeta:1}
\end{align}
We now establish \eqref{eq:inverse:div:DN:G}--\eqref{eq:DDv} with the parameter choices
\begin{align}
\const_G= |\supp(\eta_{i,j,k,q,\nn,\pp,\vecl}) \bigr|    \Gamma^{2j-3-\CLebesgue}_{q+1}  \Gamma_{q}^{\shaq} \delta_{q+1} \tilde\lambda_q \prod_{n'\leq\nn} \left( f_{q,n'} \Gamma_{q+1}^{8+\CLebesgue} \right)\, ,
\label{eq:crazy:const:G}
\end{align}
 $\lambda=\lambda\qnnpp\Gamma_{q+1}$, $M_t=\Nindt$, $\nu=\tau_q^{-1}\Gamma_{q+1}^{i-\cstarnn+4}$, $\tilde\nu=\tilde\tau_q^{-1}\Gamma_{q+1}^{-1}$, and $\lambda' = \tilde\lambda_q$. Applying Lemma~\ref{lem:a_master_est_p} and estimate \eqref{eq:nabla:a:est} with $r=2$, $r_2=1$, $r_1=\infty$, and the bounds \eqref{eq:Lagrangian:Jacobian:5} and \eqref{eq:Lagrangian:Jacobian:6}, we see that 
\begin{align}
&\left\| D^N \Dtq^M \Bigl(\partial_\alpha a_{(\xi)}^2 \left(\nabla\Phi_{(i,k)}^{-1}\right)_{\gamma\kappa} \xi^\theta \xi^\gamma \left(\nabla\Phi_{(i,k)}^{-T}\right)_{\alpha\theta} \Bigr) \right\|_{L^1} 
\notag\\
&\lessg |\supp(\eta_{i,j,k,q,\nn,\pp,\vecl}) \bigr|   \Gamma^{2j+5}_{q+1} \lambda\qnnpp \delta_{q+1,\nn,\pp}
(\Gamma_{q+1} \lambda\qnnpp)^{N}
\MM{M, \NindSmall, \tau_{q}^{-1}\Gamma_{q+1}^{i-\cstarnn+3}, \tilde\tau_{q}^{-1}\Gamma_{q+1}^{-1}}
\notag\\
&\lesssim |\supp(\eta_{i,j,k,q,\nn,\pp,\vecl}) \bigr|    \Gamma^{2j-2-\CLebesgue}_{q+1} \notag\\
&\qquad \times \Gamma_{q+1}^{-1} \Gamma_{q}^{\shaq} \delta_{q+1} \tilde\lambda_q \prod_{n'\leq\nn} \left( f_{q,n'} \Gamma_{q+1}^{8+\CLebesgue} \right) \left(\Gamma_{q+1}\lambda\qnnpp\right)^{N} \MM{M,\Nindt,\tau_q^{-1}\Gamma_{q+1}^{i-\cstarnn+3},\tilde\tau_q^{-1}\Gamma_{q+1}^{-1}}  \label{eq:G:estimate:1}
\end{align}
holds for all $N,M\leq \lfloor \sfrac{1}{2}\left( \Nfnn-\NcutSmall-\NcutLarge-5 \right) \rfloor $.
To achieve the last inequality, we have used the definition of $\delta_{q+1,\nn,\pp}$ in 
\eqref{eq:delta:q:n:def} and the definition of $f_{q,\nn}$ in \eqref{def:f:q:n} to rewrite 
$$ 
\delta_{q+1,\nn,\pp} \lambda\qnnpp \Gamma_{q+1}^{7+\CLebesgue} = \Gamma_{q+1}^{-1} \Gamma_q^{\shaq} \delta_{q+1} \tilde\lambda_q \prod_{n' \leq \nn} \left( f_{q,n'} \Gamma_{q+1}^{8+\CLebesgue} \right) \, .
$$
For the second half of $G_\kappa$, we can appeal to \eqref{eq:Lagrangian:Jacobian:5} and \eqref{eq:Lagrangian:Jacobian:6}, and use that $\tilde\lambda_q\leq\lambda\qnnpp$ for all $\nn$ and $\pp$ to deduce that for $N,M\leq \lfloor \sfrac{1}{2}\left( \Nfnn-\NcutSmall-\NcutLarge-5 \right) \rfloor $ we have
\begin{align*}
\left\|  D^N \Dtq^M  \partial_\alpha \left( \nabla\Phi_{i,k}^{-1} \right)_{\gamma\kappa} \right\|_{L^\infty(\supp\psi_{i,q}\tilde\chi_{i,k,q})}  \leq \left(\Gamma_{q+1}\lambda\qnnpp\right)^{N+1} \MM{M,\Nindt,\tau_q^{-1}\Gamma_{q+1}^{i-\cstar},\tilde\tau_q^{-1}\Gamma_{q+1}^{-1}}.
\end{align*}
Combining these estimates shows that 
\begin{align}
&\left\| D^N \Dtq^M G_{\kappa} \right\|_{L^1} \lessg \const_{G}  \left(\Gamma_{q+1}\lambda\qnnpp\right)^{N} \MM{M,\Nindt,\tau_q^{-1}\Gamma_{q+1}^{i-\cstarnn+3},\tilde\tau_q^{-1}\Gamma_{q+1}^{-1}} \, \label{eq:G:estimate:2}
\end{align}for $N,M\leq \lfloor \sfrac{1}{2}\left( \Nfnn-\NcutSmall-\NcutLarge-5 \right) \rfloor $,  showing that \eqref{eq:inverse:div:DN:G} has been satisfied.

We set the flow in Proposition~\ref{prop:intermittent:inverse:div} as  $\Phi=\Phi_{i,k}$, which by definition satisfies $\Dtq \Phi_{i,k}=0$. Appealing to \eqref{eq:Lagrangian:Jacobian:2} and \eqref{eq:Lagrangian:Jacobian:7}, we have that \eqref{eq:DDpsi} is satisfied. From \eqref{eq:nasty:D:vq}, the choice of $\nu$ from earlier, and \eqref{eq:Lambda:q:x:1:NEW}, we have that $Dv = D \vlq$ satisfies the bound \eqref{eq:DDv}.

\textbf{Proof of item \eqref{item:Onpnp:2} and of item \eqref{item:Onpnp:1} when $n=\nmax$, $p=\pmax+1$.\,}
We first assume that $\nn<\nmax$.  In this case, we have that the minimum frequency $\lambda_{q,\nmax+1,0}$ of $\LPqnpmax$ is larger than the minimum frequency $\lambda\qnn$ of $\Pqnn$ from \eqref{eq:lambda:q:n:0:def} and \eqref{eq:lambda:q:n:def}.  We therefore can discard $\Pqnn$ from \eqref{eq:type:1:general:a} and with the goal of satisfying verifying~\eqref{item:inverse:i}--\eqref{item:inverse:iii} of Proposition~\ref{prop:intermittent:inverse:div}, we set 
\begin{align}
\zeta =\lambda_{q,\nmax+1,0} , &\qquad \mu=\lambda\qnn, \qquad \Lambda=\lambda_{q+1}, \label{eq:case:1:1}
\end{align}
and 
\begin{subequations}\label{eq:case:1:1:1}
\begin{align}
\varrho &= \LPqnpmax \left(\left(\varrho_{\xi,\lambda_{q+1},r_{q+1,\nn}}\right)^2 \right), \\
\vartheta &= \lambda_{q,\nmax+1,0}^{2\dpot}\Delta^{-\dpot} \LPqnpmax\left( \varrho^2_{\xi,\lambda_{q+1},r_{q+1,\nn}}\right) \, , 
\end{align}
\end{subequations}
where we recall that $\varrho_{\xi,\lambda,r}$ is defined via~Propositions~\ref{prop:pipe:shifted} and~\ref{pipeconstruction}.
We then have immediately that 
\begin{align}
\varrho &= \LPqnpmax \left(\left(\varrho_{\xi,\lambda_{q+1},r_{q+1,\nn}}\right)^2 \right) \notag\\
&= \lambda_{q,\nmax+1,0}^{-2\dpot} \Delta^\dpot \lambda_{q,\nmax+1,0}^{2\dpot}\Delta^{-\dpot}\left(\LPqnpmax\left(\varrho^2_{\xi,\lambda_{q+1},r_{q+1,\nn}}\right)\right) \notag\\
&= \lambda_{q,\nmax+1,0}^{-2\dpot} \Delta^\dpot \vartheta \,,
\label{eq:case:1:2}
\end{align}
and so \eqref{item:inverse:i} from Proposition~\ref{prop:intermittent:inverse:div} is satisfied. By property~\eqref{item:point:1} of Proposition~\ref{prop:pipe:shifted}, we have that the functions $\varrho$ and $\vartheta$ defined in \eqref{eq:case:1:1:1}
are both periodic to scale $\left(\lambda_{q+1}r_{q+1,\nn}\right)^{-1}=\lambda\qnn^{-1}$, and so \eqref{item:inverse:ii} is satisfied.  The estimates in \eqref{eq:DN:Mikado:density} follow with $\const_*=1$ from standard Littlewood-Paley arguments (see also the discussion in part (b) of Remark~\ref{rem:div:usage}) and item \eqref{item:pipe:5} from Proposition~\ref{pipeconstruction}. Note that in the case $N=2\dpot$ in \eqref{eq:DN:Mikado:density}, the inequality is weakened by a factor of $\lambda_{q+1}^{\alpha_{\mathsf{R}}}$, for an arbitrary $\alpha_{\mathsf{R}}>0$; thus, \eqref{item:inverse:ii} is satisfied. At this stage let us fix a value for this parameter $\alpha_{\mathsf{R}}$: we choose it to be sufficiently small (with respect to $b$ and $\eps_\Gamma$) to ensure that the loss $\lambda_{q+1}^{\alpha_{\mathsf{R}}}$ may be absorbed by the spare negative factor of $\Gamma_{q+1}$ in the definition of $\const_G$, as is postulated in \eqref{eq:alpha:equation:1}. From \eqref{eq:tilde:lambda:q:def}, \eqref{eq:lambda:q:n:def}, \eqref{eq:lambda:q:0:1:def}, and \eqref{def:lambda:q:n:p}, we have that
 $$  \tilde\lambda_q \leq \lambda\qnnpp \ll \lambda\qnn \leq \lambda_{q,\nmax+1,0} \leq \lambda_{q+1}, $$
and so \eqref{eq:inverse:div:parameters:0} is satisfied. From \eqref{eq:lambdaqn:identity:2} we have that
$$  \lambda_{q+1}^4 \leq \left( \frac{\lambda\qnn}{2\pi \sqrt{3} \Gamma_{q+1} \lambda\qnnpp} \right)^\Ndec  $$
if $\Ndec$ is chosen large enough, and so \eqref{eq:inverse:div:parameters:1} is satisfied. 
Applying the estimate \eqref{eq:inverse:div:stress:1} with $\alpha$ as in  \eqref{eq:alpha:equation:1}, recalling the value for $\const_G$ in \eqref{eq:crazy:const:G},    using \eqref{eq:lemma:partition:2} and \eqref{item:lebesgue:1}  
with $r_1=\infty$ and $r_2 = 1$, we obtain that 
\begin{align}
&\left\| D^N \Dtq^M \left( \divH \mathcal{O}_{\nn,\pp,\nmax,\pmax+1} \right) \right\|_{L^1\left(\supp\psi_{i,q}\right)} \notag\\
&\quad \lesssim \sum_{i'=i-1}^{i+1}\sum_{\xi,j,k,\vecl} \Lambda^{\alpha_{\mathsf R}} |\supp(\eta_{i,j,k,q,\nn,\pp,\vecl}) \bigr|    \Gamma^{2j-3-\CLebesgue}_{q+1}  \Gamma_{q}^{\shaq} 
\notag\\
&\qquad \qquad \qquad \times \delta_{q+1} \tilde\lambda_q \prod_{n'\leq\nn} \left( f_{q,n'} \Gamma_{q+1}^{8+\CLebesgue} \right) \const_* \zeta^{-1} \MM{N,1,\zeta,\Lambda}\MM{M,M_t,\nu,\tilde\nu}\notag\\
 &\quad \lessg \Gamma_{q+1} \Big (\Gamma_{q+1}^{-1} \Gamma_q^{\shaq} \delta_{q+1} \tilde\lambda_q \prod_{n'\leq \nn} \big(f_{q,n'}\Gamma_{q+1}^{8+\CLebesgue}\big) \Big)\lambda_{q,\nmax+1,0}^{-1} \lambda_{q+1}^N \MM{M,\Nindt,\tau_q^{-1}\Gamma_{q+1}^{i-\cstarnn+4},\tilde\tau_q^{-1}\Gamma_{q+1}^{-1}}\notag\\
 &\quad \lessg \Gamma_{q+1}^\shaq \Gamma_{q+1}^{-1} \delta_{q+2} \lambda_{q+1}^N \MM{M,\NindSmall,\tau_q^{-1}\Gamma_{q+1}^{i-\cstarnn+4}, \tilde\tau_q^{-1}\Gamma_{q+1}^{-1}} \,,\label{eq:H:estimate:1}
\end{align}
for $N,M\leq \lfloor \sfrac{1}{2}\left(\Nfnn-\NcutSmall-\NcutLarge-5\right) \rfloor - \dpot$.
In the last inequality, we have used the parameter estimate  \eqref{eq:hopeless:mess:new}, which directly implies
\begin{equation}\label{eq:hopeless:mess:new:1}
\Gamma_q^{\shaq} \delta_{q+1} \tilde\lambda_q \prod_{n'\leq \nn} \left(f_{q,n'}\Gamma_{q+1}^{8+\CLebesgue}\right) \lambda_{q,\nmax+1,0}^{-1} \leq \Gamma_{q+1}^\shaq \Gamma_{q+1}^{-1} \delta_{q+2} \,.
\end{equation}
Then, after using \eqref{eq:nfnn:mess}, which gives that for all $\nn$ we have
\begin{equation}\label{eq:nfnn:mess:1}
\lfloor \sfrac{1}{2}\left(\Nfnn-\NcutSmall-\NcutLarge-5\right) \rfloor - \dpot \geq 3\NindLarge,
\end{equation}
and thus the range of derivatives allowed in \eqref{eq:H:estimate:1} is exactly as needed in \eqref{eq:Onpnp:estimate:2}, thereby proving this bound.

Continuing to follow the parameter choices in Remark~\ref{rem:div:derivative:bounds}, we set $N_\circ=M_\circ=3\NindLarge$, and as before $\Nsharp = \Nfnn-\NcutSmall-\NcutLarge-5$.  From \eqref{eq:nfnn:mess:2}, we have that the condition $N_\circ \leq \sfrac{\Nsharp}{4}$ is satisfied. 
The inequalities \eqref{eq:inverse:div:v:global} and \eqref{eq:inverse:div:v:global:parameters} follow from the discussion in Remark~\ref{rem:div:derivative:bounds}.  The inequality in \eqref{eq:riots:4} follows from \eqref{eq:Lambda:q:t:1}, \eqref{eq:CF:new},   the fact that $\lambda = \Gamma_{q+1} \lambda_{q,\nn,\pp} \leq \Gamma_{q+1} \lambda_{q,\nn,\pmax}$, and $\zeta=\lambda_{q,\nmax+1,0} > \lambda_{q,\nmax-1} \geq \lambda_{q,\nn}$, as in the discussion in Remark~\ref{rem:div:derivative:bounds}.  Having satisfied these assumptions, we may now appeal to estimate in \eqref{eq:inverse:div:error:stress:bound}, which gives \eqref{eq:Onpnp:estimate:1} for the case $\nn<n=\nmax$, $p=\pmax+1$, and any value of $\pp$.

Recall we began this case by assuming that $\nn<\nmax$.  In the case $\nn=\nmax$ and $1\leq \pp \leq \pmax$, we have from \eqref{eq:lambda:q:n:def} and \eqref{def:lambda:q:n:p} that $\lambda_{q,\nmax}>\lambda_{q,\nmax+1,0}$, and so
$$  \LPqnpmax   \Pqnn = \mathbb{P}_{\geq \lambda_{q,\nmax}} \, .  $$
Then we can set $\zeta=\mu=\lambda_{q,\nmax}$.  The only change is that \eqref{eq:hopeless:mess:new:1} becomes stronger, since $\lambda_{q,\nmax} > \lambda_{q,\nmax+1,0}$, and so the desired estimates follow by arguing as before.  We omit further details.  

\textbf{Proof of  item~\eqref{item:Onpnp:3} and of item \eqref{item:Onpnp:1} when $p\neq\pmax+1$ and $n\leq \nmax$.\,}
Note that in both of these cases we have $\nn < n$. 
We first point that that we may assume that $n$ and $p$ are such that $\lambda\qnn<\lambda\qnp$.  If not, then $\Pqnn \LPqnp=0$, and so the estimate is trivially satisfied.  We then set 
\begin{align}
\zeta=\max\left\{\lambda\qnn,\lambda\qnpminus\right\} , &\qquad \mu=\lambda\qnn, \qquad \Lambda=\lambda\qnp, \label{eq:case:1:1:redux}
\end{align}
and
\begin{subequations}
\label{eq:case:1:1:1:redux}
\begin{align}
\varrho &= \Pqnn \LPqnp \left(\left(\varrho_{\xi,\lambda_{q+1},r_{q+1,\nn}}\right)^2 \right), \\
\vartheta &= \zeta^{2\dpot}\Delta^{-\dpot} \Pqnn \LPqnp \left( \varrho^2_{\xi,\lambda_{q+1},r_{q+1,\nn}}\right) \, . 
\end{align}
\end{subequations}
We then have from the discussion part (b) of Remark~\ref{rem:div:usage} that
\begin{align}
\varrho &= \Pqnn \LPqnp \left(\varrho_{\xi,\lambda_{q+1},r_{q+1,\nn}}^2 \right) \notag\\
&= \zeta^{-2\dpot} \Delta^\dpot \zeta^{2\dpot}\Delta^{-\dpot}\left(\Pqnn \LPqnp\left(\varrho^2_{\xi,\lambda_{q+1},r_{q+1,\nn}}\right)\right), \notag\\
&=\zeta^{-2\dpot} \Delta^\dpot \vartheta \,,
\label{eq:case:1:2:redux}
\end{align}
and so \eqref{item:inverse:i} from Proposition~\ref{prop:intermittent:inverse:div} is satisfied. 
By property~\eqref{item:point:1} of Proposition~\ref{prop:pipe:shifted}, $\varrho$ and $\vartheta$ 
 are both periodic to scale $\left(\lambda_{q+1}r_{q+1,\nn}\right)^{-1}=\lambda\qnn^{-1}$, and so \eqref{item:inverse:ii} is satisfied.  The estimates in \eqref{eq:DN:Mikado:density} follow with $\const_*=1$ from the discussion in part (b) of Remark~\ref{rem:div:usage}. Note that in the case $N=2\dpot$ in \eqref{eq:DN:Mikado:density}, the inequality is weakened by a factor of $\lambda_{q+1}^{\alpha_{\mathsf R}}$, and so \eqref{item:inverse:ii} is satisfied. Here we again use $\alpha_{\mathsf R}$ as in \eqref{eq:alpha:equation:1}, so this loss will be absorbed using a factor of $\Gamma_{q+1}$. From \eqref{eq:tilde:lambda:q:def}, \eqref{eq:lambda:q:0:1:def},  \eqref{def:lambda:q:n:p} and \eqref{eq:lambda:q:n:def}, and the assumption that $\lambda\qnn<\lambda\qnp$, we have that
$$  
\tilde\lambda_q \leq \lambda\qnnpp \ll \lambda\qnn \leq \max\left\{\lambda\qnn,\lambda\qnpminus\right\} \leq \lambda\qnp,
$$
and so, since $\Lambda \leq \lambda_{q+1}$, \eqref{eq:inverse:div:parameters:0} is satisfied. From \eqref{eq:lambdaqn:identity:2} we have that
$$  \lambda_{q+1}^4 \leq \left( \frac{\lambda\qnn}{2\pi \sqrt{3} \Gamma_{q+1} \lambda\qnnpp} \right)^\Ndec\, ,  $$
and so \eqref{eq:inverse:div:parameters:1} is satisfied. 
Applying the estimate \eqref{eq:inverse:div:stress:1} for the parameter range in Remark~\ref{rem:div:derivative:bounds},  recalling that \eqref{eq:f:zeta:1} includes the indicator function of  $\supp\left(\psi_{i,q}\right)$,   recalling the definition of $\const_G$ in \eqref{eq:crazy:const:G}, using \eqref{eq:lemma:partition:2} and \eqref{item:lebesgue:1}  
with $r_1=\infty$ and $r_2 = 1$, and using   $\zeta^{-1} \leq \lambda_{q,n,p-1}^{-1}$, we have that 
\begin{align}
&\left\| D^N \Dtq^M \left( \divH \mathcal{O}_{\nn,\pp,n,p} \right) \right\|_{L^1\left(\supp\psi_{i,q}\right)} \notag\\
&\qquad \lesssim \sum_{i'=i-1}^{i+1} \sum_{\xi,j,k,\vecl} \Lambda^{\alpha_{\mathsf{R}}}  |\supp(\eta_{i,j,k,q,\nn,\pp,\vecl}) \bigr|    \Gamma^{2j-3-\CLebesgue}_{q+1}  \Gamma_{q}^{\shaq} \notag\\
&\qquad \qquad \times \delta_{q+1} \tilde\lambda_q \prod_{n'\leq\nn} \left( f_{q,n'} \Gamma_{q+1}^{8+\CLebesgue} \right) \const_* \zeta^{-1} \MM{N,1,\zeta,\Lambda}\MM{M,M_t,\nu,\tilde\nu}\notag\\
 &\qquad \lessg \Gamma_{q+1} \Gamma_{q+1}^{-1} \Gamma_q^{\shaq} \delta_{q+1} \tilde\lambda_q \prod_{n'\leq \nn} \left(f_{q,n'}\Gamma_{q+1}^{8+\CLebesgue}\right) \lambda\qnpminus^{-1} \lambda\qnp^N \MM{M,\Nindt,\tau_q^{-1}\Gamma_{q+1}^{i-\cstarnn+4},\tilde\tau_q^{-1}\Gamma_{q+1}^{-1}}\notag\\
 &\qquad \lessg \delta_{q+1,n,p} \lambda\qnp^N \MM{M,\NindSmall,\tau_q^{-1}\Gamma_{q+1}^{i-\cstarnn+4}, \tilde\tau_q^{-1}\Gamma_{q+1}^{-1}} \, . \label{eq:H:estimate:1:redux}
\end{align}
In the last inequality, we have used that since $n < \nn$, by \eqref{eq:delta:q:n:def} we have
\begin{equation}\label{eq:hopeless:mess:new:1:redux}
\Gamma_q^{\shaq} \delta_{q+1} \tilde\lambda_q \prod_{n'\leq \nn} \left(f_{q,n'}\Gamma_{q+1}^{8+\CLebesgue}\right) \lambda\qnpminus^{-1}
\leq \delta_{q+1,n,p}
\end{equation}
for all $N,M\leq \lfloor \sfrac{1}{2}\left(\Nfnn-\NcutSmall-\NcutLarge-5\right) \rfloor - \dpot$.
Then after using \eqref{eq:nfnn:nfn:mess}, which gives that for all $\nn<n$ that
\begin{equation}\label{eq:nfnn:mess:1:redux}
\lfloor \sfrac{1}{2}\left(\Nfnn-\NcutSmall-\NcutLarge-5\right) \rfloor - \dpot \geq \Nfn,
\end{equation}
we have achieved \eqref{eq:Onpnp:estimate:3}.

Continuing to follow the parameter choices in Remark~\ref{rem:div:derivative:bounds}, we set $N_\circ=M_\circ=3\NindLarge$, and as before $\Nsharp = \Nfnn-\NcutSmall-\NcutLarge-5$.  From \eqref{eq:nfnn:mess:2}, we have that the condition $N_\circ \leq \sfrac{\Nsharp}{4}$ is satisfied.  The inequalities \eqref{eq:inverse:div:v:global} and \eqref{eq:inverse:div:v:global:parameters} follow from the discussion in Remark~\ref{rem:div:derivative:bounds}.  The inequality in \eqref{eq:riots:4} follows from \eqref{eq:CF:new} and the fact that $\lambda = \Gamma_{q+1}\lambda_{q,\nn,\pp} \leq \Gamma_{q+1} \lambda_{q,\nn,\pmax}$ and $\zeta = \max\{\lambda_{q,\nn},\lambda_{q,n,p-1}\} \geq \lambda_{q,\nn}$.  We then achieve the concluded estimate in \eqref{eq:inverse:div:error:stress:bound}, which gives \eqref{eq:Onpnp:estimate:1} for the case $p\neq\pmax+1$, $n\leq\nmax$ and any values of $\nn$, $\pp$ with $\nn<n$.
\end{proof}

\subsection{Type 2 oscillation errors}\label{ss:stress:oscillation:2}

In order to show that the Type 2 errors (previously identified in~\eqref{eq:id:0:6}, \eqref{eq:id:nn:2},  \eqref{eq:id:nn:3}, \eqref{eq:id:nmax:2}, \eqref{eq:id:nmax:3}) vanish, we will apply Proposition~\ref{prop:disjoint:support:simple:alternate} on the support of a specific cutoff function
$$  \eta =  \eta_{i,j,k,q,n,p,\vecl}=\psi_{i,q}\chi_{i,k,q}\overline{\chi}_{q,n,p}\omega_{i,j,q,n,p}\zeta_{i,q,k,n,\vecl} \, .  $$
Before we may apply the proposition, we first estimate in Lemma~\ref{l:overlap} the number of cutoff functions $\eta^*$ which may overlap with $\eta$, with an eye towards keeping track of all the pipes that we will have to dodge in order to successfully place pipes on $\eta$.  The next three Lemmas (\eqref{lem:overlap:1}-\eqref{lem:overlap:3}) are technical in nature and are necessary in order to apply Lemma~\ref{lem:axis:control}.  Specifically, we show that given $\eta$, $\eta^*$ and a fixed time $t^*$, one may find a convex set which contains the intersection of the supports of $\eta$ and $\eta^*$ at $t^*$.  The time $t^*$ will be the time at which the pipes on $\eta^*$ are \emph{straight}, and combined with the convexity, Lemma~\ref{lem:axis:control} may be applied.  The upshot of this is that the pipes belonging to $\eta^*$ only undergo mild deformations on the support of $\eta$.  This allows us to finally apply Proposition~\ref{prop:disjoint:support:simple:alternate} to place pipes on $\eta$ which dodge all pipes originating from overlapping cutoff functions $\eta^*$.  We remark that since $\overline{\chi}_{q,n,p}$ depends only on $n$ and $p$, which are indices already encoded in $\omega_{i,j,q,n,p}$, throughout this section we will suppress the dependence of the cumulative cutoff function $\eta$ on $\overline{\chi}_{q,n,p}$ (defined in \eqref{def:chi:qnp}), as it does not affect any of the estimates.

\subsubsection{Preliminary estimates}

\begin{lemma}[\textbf{Keeping Track of Overlap}]\label{l:overlap}
Given a cutoff function $\eta_{i,j,k,q,n,p,\vecl}$, consider the set of all tuples $\left(\istar,\jstar,\kstar,\nstar,\pstar,\veclstar\right)$ such that the cutoff function $\eta_{\istar,\jstar,\kstar,q,\nstar,\pstar,\veclstar}$ satisfies:
\begin{enumerate}[(1)]
\item $\nstar \leq n$
\item There exists $(x,t)$ such that
\begin{equation}\label{eq:overlap:1}
\eta_{i,j,k,q,n,p,\vecl}(x,t) \eta_{\istar,\jstar,\kstar,q,\nstar,\pstar,\veclstar}(x,t) \neq 0.
\end{equation}
\end{enumerate}
Then the cardinality of the set of all such tuples is bounded above by $\const_\eta \Gamma_{q+1}$, where the constant $\const_\eta$ depends only on $\nmax$, $\pmax$, $\jmax$, and dimensional constants.  In particular, due to \eqref{eq:nmax:DEF}, \eqref{eq:pmax:DEF}, and \eqref{eq:jmax:bound}, $\const_\eta$ is independent of $q$ and the values of the other parameters indexing the cutoff functions.  
\end{lemma}
\begin{proof}[Proof of Lemma~\ref{l:overlap}]
Recall that the cutoff functions are defined by
\begin{equation}\label{eq:eta:def:counting}
\eta_{i,j,k,q,n,p,\vecl}(x,t) = \psi_{i,q}(x,t) \chi_{i,k,q}(t)\overline{\chi}_{q,n,p}(t) \omega_{i,j,q,n,p}(x,t) \zeta_{i,q,k,n,\vecl}(x,t).
\end{equation}
As noted in the outline of this section, we will suppress the dependence on $\overline{\chi}_{q,n,p}$, since the $n$ and $p$ indices are already accounted for in $\omega_{i,j,q,n,p}$. The proof proceeds by first counting the number of combinations $(\istar,\kstar)$ for which it is possible that there exists $(x,t)$ such that
\begin{equation}\label{eq:psi:overlap:counting}
\psi_{i,q}(x,t)\chi_{i,k,q}(t) \psi_{\istar,q}(x,t)\chi_{\istar,\kstar,q}(t) \neq 0.
\end{equation}
Next, for a given $(\istar,\kstar)$, we count the number of values of $(\jstar,\nstar,\pstar)$ such that there exists $(x,t)$ such that
\begin{equation}\label{eq:omega:overlap:counting}
\omega_{i,j,q,n,p}(x,t) \omega_{\istar,\jstar,q,\nstar,\pstar}(x,t) \neq 0.
\end{equation}
Finally, for a given $(\istar,\kstar,\jstar,\nstar,\pstar)$, we count the number of triples $(\lstar,\wstar,\hstar)$ such that $\nstar\leq n$ and there exists $(x,t)$ such that
\begin{equation}\label{eq:checkerboard:overlap:counting}
\zeta_{i,q,k,n,p,\vecl}(x,t) \zeta_{\istar,q,\kstar,\nstar,\pstar,\veclstar}(x,t) \neq 0.
\end{equation}

Recalling the definition of $\chi_{i,k,q}$ from \eqref{eq:chi:cut:def} and \eqref{e:chi:overlap}, we see that $\psi_{i,q}\chi_{i,\kstar,q}$ may have non-empty overlap with $\psi_{i,q}\chi_{i,k,q}$ if and only if $\kstar \in \{k-1,k,k+1\}$.  Next, from \eqref{eq:lemma:partition:2}, we have that only $\psi_{i-1,q}$ and $\psi_{i+1,q}$ may overlap with $\psi_{i,q}$.  Now, let $(x,t)\in\supp \psi_{i,q}\chi_{i,k,q}$ be given such that there exists $k_{i-1}$ such that 
$$  \psi_{i,q}(x,t)\chi_{i,k,q}(t) \psi_{i-1,q}(x,t)\chi_{i-1,k_{i-1},q}(t) \neq 0.  $$
From the definition of $\chi_{i-1,k_{i-1},q}$, it is immediate that the diameter of the support of $\chi_{i-1,k_{i-1},q}$ is \emph{larger} than the diameter of the support of $\chi_{i,k,q}$.  It follows that there can be at most three values of $\kstar$ (one of which is $k_{i-1}$) such that $\chi_{i-1,\kstar,q}$ has non-empty overlap with $\chi_{i,k,q}$.  Finally, let $(x,t)\in\supp\psi_{i,q}\chi_{i,k,q}$ be given such that there exists $k_{i+1}$ such that 
$$  \psi_{i,q}(x,t)\chi_{i,k,q}(t) \psi_{i+1,q}(x,t)\chi_{i+1,k_{i+1},q}(t) \neq 0.  $$
From the definition of $\chi_{i+1,\kstar,q}$, there exists a constant $\const_{\chi}$ depending on $\chi$ but not $i$, $q$, or $\kstar$ such that for all $|k'| \geq \const_\chi \Gamma_{q+1}$
$$  \chi_{i+1,k_{i+1}+k',q }(t) \chi_{i,k,q}(t) = 0  $$
for all $t\in\mathbb{R}$.  Therefore, the number of $\kstar$ such that $\chi_{i+1,\kstar,q}$ may have non-empty overlap with $\chi_{i,k,q}$ is no more than $2\const_\chi\Gamma_{q+1}+1$.  In summary, the number of pairs $(\istar,\kstar)$ such that \eqref{eq:psi:overlap:counting} holds for some $(x,t)$ is bounded above by
\begin{equation}\label{eq:counting:i:k}
3+3+2\const_\chi\Gamma_{q+1}+1 \leq 3\const_\chi\Gamma_{q+1}
\end{equation}
if $\lambda_0$ is sufficiently large, where the implicit constant is independent of $q$ or any other parameters which index the cutoff functions.  

Now let $(\istar,\kstar)$ be given such that $\psi_{\istar,q}\chi_{\istar,\kstar,q}$ has nonempty overlap with $\psi_{i,q}\chi_{i,k,q}$.  Once values of $\nstar$, $\pstar$, and $\jstar$ are chosen, these three parameters along with the value of $\istar$ uniquely determine a stress cutoff function $\omega_{\istar,\jstar,q,\nstar,\pstar}$.   Since $\istar$ was fixed, we may let $\jstar$, $\nstar$, and $\pstar$ vary.  Using that $\jstar\leq\jmax \leq 4 b/(\eps_\Gamma (b-1))$ from \eqref{eq:jmax:bound}, $\nstar\leq\nmax$, $\pstar\leq\pmax$ where $\nmax$, and $\pmax$ are independent of $q$, the number of tuples $(\istar,\kstar,\jstar,\nstar,\pstar)$ such that there exists $(x,t)$ with
\begin{equation}\label{eq:overlap:psi:chi:omega}
\psi_{i,q}(x,t)\chi_{i,k,q}(x,t)\omega_{i,j,q,n,p}(x,t) \psi_{\istar,q}(x,t)\chi_{\istar,\kstar,q}(x,t)\omega_{\istar,\jstar,q,\nstar,\pstar}(x,t) \neq 0
\end{equation}
is bounded by a dimensional constant multiplied by $\Gamma_{q+1}\nmax\pmax 4 b/(\eps_\Gamma (b-1))$.

Finally, fix a tuple $(\istar,\kstar,\jstar,\nstar,\pstar)$ such that \eqref{eq:overlap:psi:chi:omega} holds at $(x,t)$.  From \eqref{eq:checkerboard:partition}, there exists $\veclstar=(\lstar,\wstar,\hstar)$ such that $\zeta_{\istar,q,\kstar,\nstar,\veclstar}(x,t)\neq 0$.  From \eqref{eq:checkerboard:support}, \eqref{eq:Lagrangian:Jacobian:1}, and the fact that $\nstar\leq n$, there exists a dimensional constant $\const_\zeta$ such at most $\const_\zeta$ of the checkerboard cutoffs neighboring $\zeta_{\istar,q,\kstar,\nstar,\veclstar}$ can intersect the support of $\zeta_{i,q,k,n,\vecl}$.  Since all Lagrangian trajectories originating at $(x,t)$ follow the same velocity field $\vlq$ and the checkerboard cutoffs are precomposed with Lagrangian flows, this property is preserved in time.  Thus we have shown that for each tuple $(\istar,\kstar,\jstar,\nstar,\pstar)$, the number of associated tuples $(\lstar,\wstar,\hstar)$ such that $\zeta_{\istar,q,\kstar,\nstar,\veclstar}$ can have nonempty intersection with $\zeta_{i,q,k,n,\vecl}$ is bounded by a dimensional constant independent of $q$. 

Combining the preceding arguments, we obtain that the number of cutoff functions $\eta_{\istar,\jstar,\kstar,q,\nstar,\pstar,\veclstar}$ which may overlap nontrivially with $\eta_{i,j,k,q,n,p,\vecl}$ is bounded by at most a dimensional constant multiplied by $\Gamma_{q+1} \nmax\pmax 4 b/(\eps_\Gamma (b-1))$, finishing the proof.  
\end{proof}

\begin{lemma}\label{lem:overlap:1}
Let $(x,t),(y,t)\in\supp\psi_{i,q}$ be such that  $\psi_{i,q}^2(x,t) \geq \sfrac{1}{4}$ and $\psi_{i,q}^2(y,t) \leq \sfrac{1}{8}$.  Then there exists a geometric constant $\const_\ast > 1$ such that
\begin{equation}\label{eq:xy:distance}
|x-y| \geq \const_* \left(\Gamma_q\lambda_q\right)^{-1}.
\end{equation} 
\end{lemma}
\begin{proof}[Proof Lemma~\ref{lem:overlap:1}]
Let $L(x,y)$ be the line segment connecting $x$ and $y$.  From \eqref{eq:sharp:D:psi:i:q}, we have that for $z\in L(x,y)$ (in fact for all $z\in\mathbb{T}^3$),
\begin{equation}\label{eq:overlap:1:1}
\left| \nabla \psi_{i,q}(z) \right| \lesssim \psi_{i,q}^{1-\frac{1}{\Nfin}}(z) \lambda_q\Gamma_{q}.
\end{equation}
Thus we can write
\begin{align*}
\frac{1}{8}  \leq \left| \psi_{i,q}^2(x,t)-\psi_{i,q}^2(y,t) \right|  
&\leq 2 \left| \psi_{i,q}(x)-\psi_{i,q}(y) \right| \notag\\
& \leq 2 \left| \int_0^1 \nabla\psi_{i,q}(x+t(y-x))\cdot(y-x)\,dt \right| \notag\\
& \leq 2 |x-y| \left\|\nabla\psi_{i,q}\right\|_{L^\infty}\notag\\
&\lesssim \Gamma_q\lambda_q |x-y|,
\end{align*}
and \eqref{eq:xy:distance} follows.
\end{proof}

\begin{lemma}\label{lem:overlap:2}
Consider cutoff functions
$$\eta:=\eta_{i,j,k,q,n,p,\vecl} = \psi_{i,q}\chi_{i,k,q}\omega_{i,j,q,n,p}\zeta_{i,k,q,n,\vecl}, $$
$$\eta^*:= \eta_{\istar,\jstar,\kstar,q,\nstar,\pstar,\vecl^*} = \psi_{\istar,q}\chi_{\istar,\kstar,q}\omega_{\istar,\jstar,q,\nstar,\pstar} \zeta_{\istar,\kstar,q,\nstar,\vecl^*}, $$
where $n^*\leq n$ and $\eta$ and $\eta^*$ overlap as in Lemma~\ref{l:overlap}. Let $t^*\in\supp\chi_{\istar,\kstar,q}$ be given.  Then there exists a convex set $\Omega:=\Omega(\eta,\eta^*,t^*)$ with diameter $\lambda_{q,n,0}^{-1}\Gamma_{q+1}$ such that
\begin{equation}\label{eq:overlap:2:1}
\left(\supp \zeta_{i,k,q,n,\vecl} \cap \{t=t^*\}\right) \subset \Omega \subset \supp \psi_{i\pm,q}.
\end{equation}
\end{lemma}
\begin{proof}[Proof of Lemma~\ref{lem:overlap:2}]
Let $(x,t_0)\in\supp\left(\eta\eta^*\right)$.  Then there exists $i'\in\{i-1,i,i+1\}$ such that $\psi_{i',q}^2(x,t_0)\geq\frac{1}{2}$.  Consider the flow $X(x,t)$ originating from $(x,t_0)$.  Then for any $t$ such that $|t-t_0|\leq\tau_q\Gamma_{q+1}^{-i+5+\cstar}$, we can apply Lemma~\ref{lem:dornfelder} to deduce that $\psi_{i',q}^2(t,X(x,t))\geq\frac{1}{4}$.  By the definition of $\chi_{\istar,\kstar,q}$, the fact that $\istar\in\{i-1,i,i+1\}$, the existence of $(x,t_0)\in\supp(\chi_{i,k,q}\chi_{\istar,\kstar,q})$, and the fact that $t^*\in\supp\chi_{\istar,\kstar,q}$, we in particular deduce that $\psi_{i',q}^2(t^*,X(x,t^*))\geq\frac{1}{4}$.  Now, let $y$ be such that 
$$|X(x,t^*)-y|\leq \lambda_{q,n,0}^{-1}\Gamma_{q+1} \leq \tilde\lambda_q^{-1} <  \const_*\tilde\lambda_q^{-1}$$
 for $\const_*$ given in \eqref{eq:xy:distance}, where we have used the definitions of $\lambda_{q,n,0}$ in \eqref{eq:lambda:q:0:1:def}, \eqref{def:lambda:q:1:0:def}, and \eqref{eq:lambda:q:n:0:def}. Then from Lemma~\ref{lem:overlap:1}, it cannot be the case that $\psi_{i',q}^2(t^*,y)\leq\frac{1}{8}$, and so 
\begin{align}
y\in\supp\psi_{i',q}\cap\{t=t^*\}\subset \supp\psi_{i\pm,q}\cap\{t=t^*\} \, .
\label{eq:overlap:2:2a}
\end{align}
 Since $y$ is arbitrary, we conclude that the ball of radius $\Gamma_{q+1}\lambda_{q,n,0}^{-1}$ is contained in $\supp\psi_{i\pm,q}\cap\{t=t^*\}$.  We let $\Omega(\eta,\eta^*,t^*)$ to be precisely this ball (hence a convex set).  Since $\Dtq\zeta_{i,k,q,n,\vecl}=0$ and $(x,t_0)\in\supp\zeta_{i,k,q,n,\vecl}$, we have that $X(x,t^*)\in\supp\zeta_{i,k,q,n,\vecl} \cap \{ t = t^*\}$.  Then, recalling that the support of $\zeta_{i,k,q,n,\vecl}$ must obey the diameter bound in \eqref{eq:checkerboard:support} on the support of $\tilde\chi_{i,k,q}$, which contains the support of $\chi_{\istar,\kstar,q}$ by \eqref{eq:tilde:chi:contains},  we conclude that
\begin{equation}\label{eq:overlap:2:2}
\supp \zeta_{i,k,q,n,\vecl} \cap\{t=t^*\} \subset \Omega \, .
\end{equation}
Combining \eqref{eq:overlap:2:2a} and \eqref{eq:overlap:2:2} concludes the proof of the lemma.
\end{proof}

\begin{lemma}\label{lem:overlap:3}
As in Lemma~\ref{l:overlap}, consider cutoff functions
$$\eta:=\eta_{i,j,k,q,n,p,\vecl} = \psi_{i,q}\chi_{i,k,q}\omega_{i,j,q,n,p}\zeta_{i,k,q,n,\vecl},$$
$$ \eta^*:= \eta_{\istar,\jstar,\kstar,q,\nstar,\pstar,\vecl^*} = \psi_{\istar,q}\chi_{\istar,\kstar,q}\omega_{\istar,\jstar,q,\nstar,\pstar} \zeta_{\istar,\kstar,q,\nstar,\vecl^*}. $$
Let $t^*\in\supp\chi_{\istar,\kstar,q}$ be such that $\Phi^*:=\Phiikstar$ is the identity at time $t^*$. Using Lemma~\ref{lem:overlap:2}, define $\Omega:=\Omega(\eta,\eta^*,t^*)$.  Define $\Omega(t):= \Omega(\eta,\eta^*,t^*,t) := X(\Omega,t)$, where $X(\cdot,t^*)$ is the identity.
\begin{enumerate}[(1)]
\item For $t\in\supp\chi_{i,k,q}$,
\begin{equation}\label{eq:overlap:3:1}
\supp \eta(\cdot,t) \subset \Omega(t) \subset \supp \psi_{i\pm,q} .
\end{equation}
\item Let $\WW^* \circ \Phi^* :=\WW_{\xi*,q+1,n*}^{\istar,\jstar,\kstar,\nstar,\vecl^*} \circ \Phiikstar$ be an intermittent pipe flow supported on $\eta^*$.  Then there exists a geometric constant $\const_{\textnormal{pipe}}$ such that 
$$ \left( \supp \WW^* \circ \Phi^*  \cap \{t=t^*\} \cap \Omega \right) \subset \bigcup_{n=1}^{N} S_n, $$
where the sets $S_n$ are cylinders concentrated around line segments $A_n$ for $n\in\{1,...,N\}$ with 
\begin{equation}\label{eq:counting:support}
N \leq \const_{\textnormal{pipe}}  \left( \frac{\lambda\qn }{\lambda_{q,n,0}  \Gamma_{q+1}^{-1}} \right)^2.
\end{equation}
\item $\WW^*\circ \Phi^* (\cdot,t)$ and the associated axes $A_n(t)$ and sets $S_n(t)$ satisfy the conclusions of Lemma~\ref{lem:axis:control} on the set $\Omega(t)$ for $t\in\supp\chi_{i,k,q}$.
\end{enumerate}
\end{lemma}
\begin{proof}[Proof of Lemma~\ref{lem:overlap:3}]
From the previous lemma, we have that for all $y\in\Omega$, $\psi_{i\pm,q}^2(y,t^*)\geq\sfrac 18$.  Applying Lemma~\ref{lem:dornfelder}, we have that for all $t$ with $|t-t^*|\leq\tau_q\Gamma_{q+1}^{-i+5+\cstar}$, the Lagrangian flow originating from $(y,t^*)$ has the property that
\begin{equation}\label{eq:dornfelder:needed}
\psi_{i\pm,q}^2(t,X(y,t)) \geq \sfrac{1}{16} \,.
\end{equation}
Recalling from \eqref{eq:chi:tilde:support} that the diameter of the support of $\tilde\chi_{\istar,\kstar,q}$ is $\tau_q\Gamma_{q+1}^{-\istar+\cstar}$ and that $i-1\leq\istar\leq i+1$, we have that in particular the Lagrangian flow originating at $(y,t^*)$ satisfies \eqref{eq:dornfelder:needed} for all $t\in\supp\tilde\chi_{\istar,\kstar,q}$.  From \eqref{eq:tilde:chi:contains}, \eqref{eq:dornfelder:needed} is then satisfied in particular for all $t\in\supp\chi_{i,k,q}$, thus proving the second inclusion from \eqref{eq:overlap:3:1}.  To prove the first inclusion, we use \eqref{eq:overlap:2:1}, the definition of $\Omega(t)$, and the equality $\Dtq\zeta_{i,k,q,n,\vecl}=0$ to deduce that 
$$ \supp \zeta_{i,k,q,n,\vecl}(\cdot,t) \subset \Omega(t), $$
finishing the proof of \eqref{eq:overlap:3:1}.

To prove the second claim, recall that $\WW^*\circ \Phi^*$ at $t=t^*$ is periodic to scale $\lambda_{q,\nstar}^{-1}$ for $\nstar\leq n$, and the diameter of $\Omega$ is $2\lambda_{q,n,0}^{-1}\Gamma_{q+1}$ (in fact $\Omega$ is a ball). Considering the quotient of the respective  diameters  squared, the claim then follows   after absorbing the geometric constant $n_\xi^*$ from Proposition~\ref{prop:pipe:shifted} into $\const_{\textnormal{pipe}}$.  

To see that we may apply Lemma~\ref{lem:axis:control}, first note that $\Omega = \Omega(t^*)$ is convex by construction, and so the first assumption of Lemma~\ref{lem:axis:control} is met.  We choose $v=\vlq$ and $X$ and $\Phi$ to be the associated backwards and forwards flows originating from $t_0=t^*$.  From \eqref{eq:nasty:D:vq}, \eqref{eq:overlap:3:1}, and \eqref{eq:tilde:lambda:q:def}, we have that for $t\in\supp\chi_{i,k,q}$ and $x \in\Omega(t) $, 
\begin{equation}\label{eq:finding:322}
\left| \nabla \vlq (x,t) \right| \lesssim \delta_q^{\sfrac{1}{2}} \tilde\lambda_q \Gamma_{q+1}^{i+2} = \delta_q^{\sfrac{1}{2}} \lambda_q \Gamma_{q+1}^{i+7},
\end{equation}
and so \eqref{e:axis:derivative:bounds} is satisfied with $C=i+7$.  Recall again from \eqref{eq:tilde:chi:contains} that $\supp\tilde\chi_{\istar,\kstar,q}$ contains the support of $\chi_{i,k,q}$, and that from \eqref{eq:chi:tilde:support} the support of $\tilde\chi_{\istar,\kstar,q}$ has diameter $\tau_q\Gamma_{q+1}^{-\istar+\cstar}$.  We then use \eqref{eq:Lambda:q:x:1:NEW} and \eqref{eq:tilde:lambda:q:def} to write that for any $t \in \supp\tilde\chi_{\istar,\kstar,q}$ we have
\begin{align*}
|t-t^*|  \leq \tau_q \Gamma_{q+1}^{-\istar+\cstar+1}  
&\leq \tau_q \Gamma_{q+1}^{-i + \cstar +2} \notag\\
&\leq \left(\delta_q^{\sfrac{1}{2}} \tilde\lambda_q \Gamma_{q+1}^{\cstar+6} \right)^{-1} \Gamma_{q+1}^{-i+\cstar+2} \notag\\
&= \left(\delta_q^{\sfrac{1}{2}} \lambda_q \Gamma_{q+1}^{\cstar+11} \right)^{-1} \Gamma_{q+1}^{-i+\cstar+2} \notag\\
& \leq \left(\delta_q^{\sfrac{1}{2}} \lambda_q \Gamma_{q+1}^{i+9} \right)^{-1},
\end{align*}
so that \eqref{eq:tau:axis:support} is satisfied since $C+2=i+9$. We can now apply  Lemma~\ref{lem:axis:control}, concluding the proof of the Lemma.
\end{proof}

\subsubsection{Applying Proposition~\ref{prop:disjoint:support:simple:alternate}}

\begin{lemma}\label{lem:osc:2}
The Type 2 oscillation errors vanish.  More specifically, 
\begin{enumerate}[(1)]
\item\label{osc:est:2:1} When $\nn=0$, the Type 2 errors identified in \eqref{eq:id:0:6} vanish.
\item\label{osc:est:2:2} When $1\leq \nn \leq \nmax-1$, the Type 2 errors identified in \eqref{eq:id:nn:2} and \eqref{eq:id:nn:3} vanish.
\item\label{osc:est:2:3} When $\nn=\nmax$, the Type 2 errors identified in \eqref{eq:id:nmax:2} and \eqref{eq:id:nmax:3} vanish.
\end{enumerate}
\end{lemma}

\begin{proof}[Proof of Lemma~\ref{lem:osc:2}]
We first recall what the Type $2$ oscillation errors are.
When $\nn=0$, the errors identified in \eqref{eq:id:0:6} can be written using \eqref{e:split:0:1} as
\begin{equation}\label{eq:osc:2:proof:1}
\mathcal{O}_{0,2} = \sum_{\neq\{\xi,i,j,k,\pp,\vecl\}}   \curl\left(a_{(\xi)} \nabla\Phi_{(i,k)}^T \UU_{\xi,q+1,0} \circ \Phiik \right) \otimes \curl\left( a_{(\xistar)}\nabla\Phi_{(\istar,\kstar)}^T \UU_{\xistar,q+1,0} \circ \Phiikstar \right) \,,
\end{equation}
where the notation $\neq\{\xi,i,j,k,\pp,\vecl\}$ is defined in \eqref{e:xiijk} and denotes summation over all pairs of cutoff function indices for which at least one parameter differs between the two pairs. When $1\leq\nn\leq\nmax$, the Type 2 errors identified in \eqref{eq:id:nn:2} and \eqref{eq:id:nmax:2} can be written as
\begin{align}\label{eq:osc:davidc:1}
2\sum_{n'\leq\nn-1} w_{q+1,\nn} \otimes_{\textnormal{s}} w_{q+1,n'} 
= 2&\sum_{\nstar\leq\nn-1} \sum_{\xi,i,j,k,\pp,\vecl}\sum_{\xistar,\istar,\jstar,\kstar,\pstar,\vecl^*} \curl\left( a_{(\xi)} \nabla\Phiik^T \UU_{\xi,q+1,\nn} \circ \Phiik \right) \notag\\
&\qquad \qquad  \otimes_{\textnormal{s}} \curl \left( a_{(\xistar)} \nabla\Phiikstar^T \UU_{\xistar,q+1,\nstar} \circ \Phiikstar \right) \,.
\end{align}
When $1\leq\nn\leq\nmax$, the Type 2 errors identified in \eqref{eq:id:nn:3} and \eqref{eq:id:nmax:3} can be written as
\begin{equation}\label{eq:osc:davidc:2}
\sum_{\neq\{\xi,i,j,k,\pp,\vecl\}}   \curl\left(a_{(\xi)} \nabla\Phi_{(i,k)}^T \UU_{\xi,q+1,\nn} \right) \otimes \curl\left( a_{(\xistar)}\nabla\Phi_{(\istar,\kstar)}^T \UU_{\xistar,q+1,\nn} \right) \,,
\end{equation}
where the notation $\neq\{\xi,i,j,k,\pp,\vecl\}$ has been reused from \eqref{e:xiijk}.  To show that the errors defined in \eqref{eq:osc:2:proof:1}, \eqref{eq:osc:davidc:1}, and \eqref{eq:osc:davidc:2} vanish, it suffices to show the following.  For pairs of cutoff functions $\eta_{i,j,k,q,\nn,\pp,\vecl}$ and $\eta_{\istar,\jstar,\kstar,q,\nstar,\pstar,\vecl^*}$ satisfying the two conditions in Lemma~\ref{l:overlap}, and vectors $\xi,\xistar\in\Xi$,
\begin{align}\label{eq:osc:davidc:3}
&\supp \left(\WW_{\xi,q+1,\nn}^{i,j,k,\nn,\pp,\vecl}\circ\Phiik\right) \cap \supp \eta_{i,j,k,q,\nn,\pp,\vecl} \notag\\
&\qquad \qquad \qquad  \cap \supp \left(\WW_{\xistar,q+1,\nstar}^{\istar,\jstar,\kstar,\nstar,\pstar,\vecl^*}\circ\Phiikstar\right) \cap \supp \eta_{\istar,\jstar,\kstar,q,\nstar,\pstar,\vecl^*} = \emptyset.
\end{align}
The proof of this claim will proceed by fixing $\nn$, using the preliminary estimates, and applying Proposition~\ref{prop:disjoint:support:simple:alternate}.

Let $\nn$ be fixed and assume that  $w_{q+1,n'}$ for $n'<\nn$ has been defined (when $\nn=0$, this assumption is vacuous). In particular, placements have been chosen for all intermittent pipe flows indexed by $n'$.  Now, consider all the cutoff functions $\eta_{i,j,k,q,\nn,\pp,\vecl}$ utilized at stage $\nn$. Since the parameters indexing the cutoff functions are countable, we may choose \emph{any} ordering of the tuples $(i,j,k,\pp,\vecl)$ at level $\nn$.  Combined with an ordering of the direction vectors $\xi\in\Xi$, we thus have an ordering of the cutoff functions $\eta_{i,j,k,q,\nn,\pp,\vecl}$ and the associated intermittent pipe flows $\WW_{\xi,q+1,\nn}^{i,j,k,\nn,\pp,\vecl}\circ\Phiik$. 

To ease notation, we will abbreviate the cutoff functions as $\eta_z$ and the associated intermittent pipe flows as $(\WW\circ\Phi)_z$, where $z\in\mathbb{N}$ corresponds to the ordering. We will apply Proposition~\ref{prop:disjoint:support:simple:alternate} inductively on $z$ such that the following two conditions hold.  Our goal is to place the pipe flow $(\WW\circ\Phi)_z$ such that 
\begin{equation}\label{eq:osc:2:proof:3}
\supp (\WW\circ\Phi)_{z'}  \cap \supp(\WW\circ\Phi)_z \cap \supp \eta_{z} = \emptyset \,,
\end{equation}
for all $z'<z$, 
and such that 
\begin{equation}\label{eq:osc:2:proof:4}
\supp w_{q+1,n'} \cap \supp (\WW\circ\Phi)_z \cap \supp \eta_{z} = \emptyset \,,
\end{equation}
for all $n'<\nn$.
The first condition shows that all Type 2 errors such as \eqref{eq:osc:2:proof:1} and \eqref{eq:osc:davidc:2} which arise from two sets of pipes both indexed by $\nn$ vanish, while the second condition shows that the Type 2 errors which arise from pipes indexed by $n'<\nn$ interacting with pipes indexed by $\nn$ vanish, such as \eqref{eq:osc:davidc:1}.

Throughout the rest of the proof, $z'$ will only ever denote an integer less than $z$ such that $\eta_{z}$ and $\eta_{z'}$ overlap.  Although we have suppressed the indices, note that $\eta_{z'}$ and $\eta_z$  both correspond to the index $\nn$. Conversely, let $\eta_{z''}$ denote a generic cutoff function indexed by $n'$ which overlaps with $\eta_z$. By Lemma~\ref{l:overlap}, there exists a geometric constant $\const_{\eta}$ such that the number of cutoff functions $\eta_{z'}$ or $\eta_{z''}$ which overlap with $\eta_z$ is bounded above by $\const_\eta \Gamma_{q+1}$.  Let $t_{z'}\in\supp\chi_{i_{z'},k_{z'},q}$ be the time for which $\Phi_{i_{z'},k_{z'},q}$ is the identity, and let $\Omega\left(\eta_z,\eta_{z'},t_{z'}\right)$ be the convex set constructed in Lemma~\ref{lem:overlap:2}, where we have set $t^*=t_{z'}$.  Let $\Omega\left(\eta_z,\eta_{z'},t_{z'},t\right)$ denote the image of $\Omega\left(\eta_z,\eta_{z'},t_{z'}\right)$ under this flow, as defined in Lemma~\ref{lem:overlap:3}.  We then have that the set 
\begin{equation}\label{eq:pipe:counting:1}
 \supp (\WW\circ\Phi)_{z'} \cap \supp \Omega\left(\eta_z,\eta_{z'},t_{z'}\right) \cap \{t=t_{z'}\}
\end{equation}
is contained in the union of sets $S_n^{z'}$ concentrated around axes $A_n^{z'}$ for 
$$
n\leq \const_{\textnormal{pipe}}\Gamma_{q+1}^2 \frac{\lambda_{q,\nn}^2}{\lambda_{q,\nn,0}^2} \,,
$$
and the flowed axes $A_n^{z'}$ and pipes of $(\WW\circ\Phi)_{z'}$ satisfy the conclusions of Lemma~\ref{lem:axis:control}.  Furthermore, substituting $z''$ for $z'$ in the preceding discussion, all the analogous definitions and conclusions can be made for cutoff functions $\eta_{z''}$ and pipe flows $(\WW\circ\Phi)_{z''}$.  

We will apply Proposition~\ref{prop:disjoint:support:simple:alternate} with the following choices.  Let $t_z$ be the time at which the flow map $\Phi_{i,k,q}$ corresponding to $\eta_z$ is the identity.  Set
\begin{equation}\label{def:the:real:omega}
\Omega = \left( \bigcup_{z'<z} \Omega\left(\eta_z,\eta_{z'},t_{z'},t_z\right) \right) \bigcup \left( \bigcup_{n'<\nn} \Omega\left(\eta_z,\eta_{z''},t_{z''},t_z\right) \right)
\end{equation}
and set 
\begin{align}\label{eq:choosing:r1}
r_1=\Gamma_{q+1}^{-1}\frac{\lambda_{q,\nn,0}}{\lambda_{q+1}}=\begin{cases}\left(\frac{\lambda_q}{\lambda_{q+1}}\right)^{\ff^{\nn-1}\cdot\frac{5}{6}}\Gamma_{q+1}^{-1}&\mbox{if }\nn\geq2\\ 
\left(\frac{\lambda_q}{\lambda_{q+1}}\right)^\frac{4}{5}\Gamma_{q+1}^{-1}& \mbox{if } \nn = 1 \\
 \frac{\tilde\lambda_q}{\lambda_{q+1}}  &\mbox{if } \nn=0.
\end{cases}
\end{align}
We have used here the definitions of $\lambda_{q,\nn,0}$ given in \eqref{def:lambda:q:1:0:def}, \eqref{eq:lambda:q:0:1:def}, and \eqref{eq:lambda:q:n:0:def}.  Note that by \eqref{eq:overlap:3:1}, $\supp\eta_z(\cdot,t_z)\subset \Omega\left(\eta_z,\eta_{z'},t_{z'},t_z\right)$ for each $z'<z$, with the analogous inclusion holding when $z'$ is replaced by $z''$.  In particular, we have that $\supp\eta_z(\cdot,t_z)\subset\Omega$.  Furthermore, we have additionally from Lemma~\ref{lem:overlap:3} that Lemma~\ref{lem:axis:control} may be applied on $\Omega(t)$ for all $t\in\chi_{i,k,q}$. Thus, the diameter of $\Omega(\eta_z,\eta_{z'},t_{z'},t_z)$ satisfies
\begin{equation}\label{eq:davidc:diameter:omega}
\textnormal{diam}\left(\Omega\left(\eta_z,\eta_{z'},t_{z'},t_z\right)\right) \leq (1+\Gamma_{q+1}^{-1})\textnormal{diam}\left(\Omega(\eta_z,\eta_{z'},t_{z'})\right) = 2(1+\Gamma_{q+1}^{-1})\lambda_{q,\nn,0}^{-1}\Gamma_{q+1}
\end{equation}
Using that the diameter of the support of $\eta_z(\cdot,t_z)$ is bounded by a dimensional constant times $\lambda_{q,\nn,0}^{-1}$ from \eqref{eq:checkerboard:support} and recalling that $\supp\eta_z(\cdot,t_z)\subset \Omega\left(\eta_z,\eta_{z'},t_{z'},t_z\right)$ with the analogous conclusion holding for $z''$, we have that
\begin{align}
\textnormal{diam}(\Omega) &\leq 4(1+\Gamma_{q+1}^{-1})\lambda_{q,\nn,0}^{-1}\Gamma_{q+1} + \Gamma_{q+1}\lambda_{q,\nn,0}^{-1} \notag\\
&\leq 6(1+\Gamma_{q+1}^{-1}) \Gamma_{q+1} \left(\lambda_{q,\nn,0}\right)^{-1}\notag\\
&\leq 16(\lambda_{q+1}r_1)^{-1}\notag
\end{align}
for each value of $\nn$ from \eqref{eq:choosing:r1}, and so \eqref{eq:Omega:diameter:alt} is satisfied. 

Now set
$$\const_A=\const_{\textnormal{pipe}}\const_{\eta}\Gamma_{q+1} ,  \qquad r_2= r_{q+1,n} \approx \left( \frac{\lambda_q}{\lambda_{q+1}} \right)^{\ff^{\nn+1}}, $$
where above we have appealed to \eqref{eq:rqn:perp:definition} and \eqref{eq:r:q+1:n:bounds}. By \eqref{eq:counting:support} and Lemma~\ref{l:overlap}, the total number of pipes contained in $\Omega$ is no more than
$$  \const_{\textnormal{pipe}}\const_{\eta} \Gamma_{q+1}^3\frac{\lambda\qnn^2}{\lambda\qnnone^2}.  $$
Then we can write
\begin{align*}
\const_{\textnormal{pipe}}\const_{\eta} \Gamma_{q+1}^3\frac{\lambda\qnn^2}{\lambda_{q,\nn,0}^2} = \const_A \frac{r_2^2}{r_1^2},\notag
\end{align*}
and so \eqref{eq:Npipe:bound} is satisfied.  Furthermore, the assumptions on the axes and the neighborhoods of the axes required by Proposition~\ref{prop:disjoint:support:simple:alternate} follow from Lemma~\ref{lem:overlap:3}, which allows us to appeal to the conclusions of Lemma~\ref{lem:axis:control}.  Finally, from \eqref{eq:parameter:relative:intermittency}, we have that for $\nn\geq 2$,
\begin{align}
C_* \const_A r_2^4 \leq 16 C_* \const_{\textnormal{pipe}} \const_\eta \Gamma_{q+1} \left( \frac{\lambda_q}{\lambda_{q+1}} \right)^{\ff^{\nn+1}\cdot 4} \leq \left( \frac{\lambda_q}{\lambda_{q+1}}  \right)^{\ff^{\nn-1}\cdot\frac{5}{6}\cdot3}\Gamma_{q+1}^{-3} = r_1^3,
\end{align}
showing that \eqref{eq:r1:r2:condition:alt} is satisfied for $\nn\geq 2$.  In the cases $\nn=0$ and $\nn=1$, the desired inequalities follow from \eqref{eq:choosing:r1} and \eqref{eq:parameter:relative:intermittency:0} and \eqref{eq:parameter:relative:intermittency:1}, and so we have checked that \eqref{eq:r1:r2:condition:alt} is satisfied for all $0\leq\nn\leq\nmax$.  Then from the conclusion \eqref{e:disjoint:conclusion} of Proposition~\ref{prop:disjoint:support:simple:alternate}, we have that on the support of $\Omega$, which in particular contains the support of $\eta_z(\cdot,t_z)$ from \eqref{eq:overlap:3:1}, we can choose the support of $(\WW\circ\Phi)_z$ to be disjoint from the support of $(\WW\circ\Phi)_{z'}$ and $(\WW\circ\Phi)_{z''}$ for all overlapping $z''$ and $z'$. Then since $\Dtq(\WW\circ\Phi)_z=\Dtq(\WW\circ\Phi)_{z'}=\Dtq(\WW\circ\Phi)_{z''}=0$, \eqref{eq:osc:2:proof:3} and \eqref{eq:osc:2:proof:4} are satisfied, concluding the proof.
\end{proof}

\subsection{Divergence corrector errors}\label{ss:stress:divergence:correctors}
\begin{lemma}\label{l:divergence:corrector:error}
For all $0\leq \nn \leq \nmax$, $1\leq \pp \leq \pmax$, and $j \in \{2,3\}$, the divergence corrector errors $\mathcal{O}_{\nn,1,j}$ satisfy
$$ \left\| \psi_{i,q} D^k \Dtq^m \mathcal{O}_{\nn,1,j} \right\|_{L^1} \lesssim \Gamma_{q+1}^{\shaq-1} \delta_{q+2}\lambda_{q+1}^k  \MM{k,\Nindt,\Gamma_{q+1}^{i+1} \tau_q^{-1}, \Gamma_{q+1}^{-1}\tilde\tau_q^{-1} } $$
for all $k,m\leq 3\NindLarge$.
\end{lemma}

\begin{proof}[Proof of Lemma~\ref{l:divergence:corrector:error}]
The divergence corrector errors are given in \eqref{e:split:0:2}, \eqref{e:split:nn:2}, and \eqref{e:split:nmax:2}. The estimates for $j= \{ 2,3\}$ are each similar, and so we shall only prove the case $j=2$. Thus we estimate
\begin{equation}\label{e:div:corrector:1}
\left\| \psi_{i,q} D^k \Dtq^m \sum_{\xi,i',j,k,\pp,\vecl} \left( \left( a_{(\xi)} \nabla\Phi_{(i',k)}^{-1}\WW_{\xi,q+1,\nn} \circ \Phi_{(i',k)} \right) \otimes \left( \nabla a_{(\xi)}\times \left( \nabla\Phi_{(i',k)}^T \UU_{\xi,q+1,\nn} \circ \Phi_{(i',k)} \right) \right) \right) \right\|_{L^1} . 
\end{equation}
Recall that $\xi$ takes only six distinct values and that $j\leq\jmax$, $\pp\leq\pmax$ are bounded independently of $q$.  Furthermore, on the support of $\psi_{i,q}$, only $\psi_{i-1,q}$, $\psi_{i,q}$, and $\psi_{i+1,q}$ are non-zero from \eqref{eq:lemma:partition:2}.  As a result, only time cutoffs $\chi_{i-1,k,q}$, $\chi_{i,k,q}$, and $\chi_{i+1,k,q}$ may be non-zero.  Since for each $i$ the $\chi_{i,k,q}$'s form a partition of unity in time for which only two cutoff functions are non-zero at any fixed time, for every time, the sum in \eqref{e:div:corrector:1} is a finite sum for which the number of non-zero terms in the summand is bounded independently of $q$.  Similarly, the sum over $\vecl$ forms a partition of unity which only finitely many cutoff functions overlap at any fixed point in space and time. Therefore we may absorb the effects of $\xi$, $j$, $k$, $\pp$, and $\vecl$ in the implicit constant in the inequality. 

Using H\"{o}lder's inequality and estimates \eqref{eq:w:oxi:est} and \eqref{eq:w:oxi:c:est} from Corollary~\ref{cor:corrections:Lp} with $r=2$, $r_2=1$, and $r_1=\infty$, we have that for $N,M \leq \lfloor \sfrac{1}{2}\left(\Nfnn-\NcutSmall-\NcutLarge-2\Ndec-9\right) \rfloor $,
\begin{align}
    & \sum_{\xi,i',j,k,\pp,\vecl} \left\| \psi_{i,q} D^k \Dtq^m  \left( \left( a_{(\xi)} \nabla \Phi_{(i',k)}^{-1}\WW_{\xi,q+1,\nn} \circ \Phi_{(i',k)} \right) \otimes \left( \nabla a_{(\xi)}\times \left( \nabla\Phi_{(i',k)}^T \UU_{\xi,q+1,\nn} \circ \Phi_{(i',k)} \right) \right) \right) \right\|_{L^1} \notag\\
    &\qquad \lessg \Gamma_{q+1}^{8+\CLebesgue} \delta_{q+1,\nn,\pp} \lambda_{q+1}^k \MM{m,\Nindt,\tau_q^{-1}\Gamma_{q+1}^{i-\cstarnn+4},\tilde\tau_q^{-1}\Gamma_{q+1}^{-1}} \frac{\lambda\qnnpp}{\lambda_{q+1}}\notag\\
    &\qquad \lesssim \Gamma_{q+1}^{\shaq-1} \delta_{q+2} \lambda_{q+1}^k \MM{m,\Nindt,\tau_q^{-1}\Gamma_{q+1}^{i+1} , \tilde\tau_q^{-1}\Gamma_{q+1}^{-1} }  \notag,
\end{align}
which proves the desired estimate after recalling that for all $\nn$,
\begin{align}
\lfloor \sfrac{1}{2}\left(\Nfnn-\NcutSmall-\NcutLarge-2\Ndec-9\right) \rfloor &\geq 3\NindLarge \notag\\
\Gamma_{q+1}^{8+\CLebesgue} \frac{\delta_{q+1,\nn,\pp}\lambda\qnnpp}{\lambda_{q+1}} &\leq \delta_{q+2}\Gamma_{q+1}^{\shaq-1} \notag\\
-\cstarnn+4 &\leq 1 \, , \notag
\end{align}
which follow from \eqref{eq:Nfinn:inequality}, \eqref{eq:delta:q:n:def} and \eqref{eq:hopeless:mess:new}, and \eqref{eq:cstarn:inequality}, respectively.
\end{proof}

\subsection{Time support of perturbations and stresses}\label{ss:time:support}

First, we prove \eqref{eq:perturbation:time:support:redux:nn=0}.  Indeed, appealing to \eqref{eq:vlq:Rlq:def}, which defines $\RR_{\ell_q}$ in terms of a mollifier applied to $\RR_q$, \eqref{eq:tilde:tau:q:def}, which defines the scale at which $\RR_q$ is mollified, and \eqref{def:chi:qnp}, which ensures that the time support of $w_{q+1,0}$ is only enlarged relative to the time support of $\RR_{\ell_q}$ by $2\left(\delta_q^{\sfrac 12} \lambda_q \Gamma_{q+1}^2\right)^{-1}$, we achieve \eqref{eq:perturbation:time:support:redux:nn=0}. To prove \eqref{eq:Rqplus:time:0} and \eqref{eq:Hqnp:time:0}, first note that application of the inverse divergence operators $\divH$ and $\divR$ \emph{commutes} with multiplication by $\overline{\chi}_{q,n,p}$.\footnote{This is simple to check from the formula given in Proposition~\ref{prop:Celtics:suck} and the formula for the standard nonlocal inverse divergence operator given in \eqref{eq:RSZ}, both of which involve operations which are purely spatial, such as differentiation and application of Fourier multipliers.}  Then by the definition of $\RR_{q+1}^0$ and $\HH_{q,n,p}^0$ in Section~\ref{ss:stress:error:identification}, we achieve \eqref{eq:Rqplus:time:0} and \eqref{eq:Hqnp:time:0}.  Proving the inclusions in \eqref{eq:perturbation:time:support:redux:nn}, \eqref{eq:Rqplus:time:nn}, \eqref{eq:Hqnp:time:nn}, \eqref{eq:perturbation:time:support:redux:nn=nmax}, \eqref{eq:Rqplus:time:nmax}, and \eqref{eq:Hqnp:time:nnmax}, follows similarly from \eqref{def:chi:qnp}, the properties of $\divH$ and $\divR$, and the definitions of $\RR_{q+1}^\nn$ and $\HH_{q,n,p}^\nn$ in Section~\ref{ss:stress:error:identification}.  Finally, to see that \eqref{eq:perturbation:time:support:redux:0} follows from the inclusions already demonstrated, notice that the threshold in \eqref{eq:perturbation:time:support:redux:0} is weaker than any of the previous inclusions by a factor of $\Gamma_{q+1}$, and so we may allow the time support of $\RR_{q+1}^\nn$ to expand slightly as $\nn$ increases from $0$ to $\nmax$ while still meeting the desired inclusion.

\section{Parameters}
\label{sec:parameters}

The purpose of this section is to provide an exhaustive delineation of the many parameters, inequalities, and notations which arise throughout the bulk of the paper. In Section~\ref{sec:parameters:DEF}, we define the $q$-independent parameters \emph{in order}, beginning with the regularity index $\beta$, and ending with the number $a_*$, which will be used to absorb every implicit constant throughout the paper.  Then in Section~\ref{ss:q:dependent:parameters}, we define the parameters which depend on $q$, as well as the parameters which depend in addition on $n$ and $p$.  The definitions of both the $q$-independent and $q$-dependent parameters will appear rather arbitrary, but are justified in Section~\ref{ss:many:inequalities}.  This section contains, in no particular order, consequences of the definitions made in the previous two sections which are necessary to close the estimates in the proof.  Finally, Sections~\ref{sec:mollifiers:Fourier} and \ref{ss:notation} contain the definitions of a few operators and some notations that are used throughout the paper.

\subsection{Definitions and hierarchy of the parameters}
\label{sec:parameters:DEF}
The parameters in our construction are chosen as follows:
\begin{enumerate}[(i)]
\item \label{item:beta:DEF} Choose an arbitrary regularity parameter $\beta \in [\sfrac 13 ,\sfrac 12)$. In light of~\cite{BDLSV17,Isett2018}, there is no reason to consider the regime $\beta < \sfrac 13$.
\item Choose $b \in (1, \sfrac 32)$ sufficiently small such that 
\begin{align}
2\beta b &< 1
\label{eq:b:DEF} \,.
\end{align}
The heuristic reason for \eqref{eq:b:DEF} is given by~\eqref{eq:onehalf}. Note that \eqref{eq:b:DEF} and the inequality $\beta< \sfrac 12$ imply that $\beta(2b+1) < \sfrac 32$, which is a required inequality for the heuristic estimate~\eqref{eq:Nash:transport:heuristic}.
\item \label{item:nmax:pmax:DEF}  With $\beta$ and $b$ chosen, we may now designate a number of parameters:
\begin{enumerate}  
\item The parameter $\nmax$, which per Section~\ref{ss:higher:order:stress:details} denotes the total number of higher order stresses $\RR\qn$ and thus primary frequency divisions in between $\lambda_q$ and $\lambda_{q+1}$, is defined as the smallest integer for which 
\begin{align}
1- 2\beta b > \frac 56 \left( \frac 45\right)^{\nmax - 1} 
\,.
\label{eq:nmax:DEF} 
\end{align}
\item The parameter $\pmax$, which per Section~\ref{ss:higher:order:stress:details} denotes the total number of subdivided components $\RR\qnp$ of a higher order stress $\RR\qn$ and thus secondary frequency divisions in between $\lambda_q$ and $\lambda_{q+1}$, is defined as the smallest integer for which 
\begin{align}
\frac{1}{\pmax} < \frac{1-2\beta b}{10}
\,.
\label{eq:pmax:DEF} 
\end{align}
\item The parameter $\CLebesgue$ appearing in \eqref{eq:psi:i:q:support:old} is use to quantify the $L^1$ norm of the velocity cutoff functions $\psi_{i,q}$. It is defined as
\begin{align}
\CLebesgue = \frac{b+4}{b-1}
\,.
\label{eq:CLebesgue:DEF}
\end{align}
\item The exponent $\mathsf{C_R}$ is used in order to define a small parameter in the estimate for the Reynolds stress, cf.~\eqref{eq:Rq:inductive:assumption}. This parameter is then used in the proof to absorb geometric constants in the construction. It is defined as
\begin{align}
 \mathsf{C_R} = 4 b + 1
 \,.
 \label{eq:shaq:DEF}
\end{align}
 \end{enumerate}
 \item The parameter $\cstar$, which is first introduced in \eqref{eq:sharp:Dt:psi:i:q:mixed:old} and utilized in Sections~\ref{section:statements} and \ref{s:stress:estimates} to control small losses in the sharp material derivative estimates, is defined in terms of $\nmax$ as 
\begin{align}
 \cstar = 4 \nmax + 5\,.
 \label{eq:cstar:DEF}
\end{align}
\item The parameter $\eps_\Gamma > 0 $, which is used in \eqref{def:Gamma:q:actual} to quantify the \emph{finest} frequency scale between $\lambda_q$ and $\lambda_{q+1}$ utilized throughout the scheme, is defined as the greatest real number for which the following inequalities hold  
\begin{subequations}
\label{eq:eps:Gamma:DEF} 
\begin{align}
\eps_\Gamma \Big( 7 + \mathsf{C_R} + \nmax (8 + \CLebesgue)    ) \Big)  
&< \frac{1-2\beta}{10} \label{eq:eps:gamma:1} \\
\eps_\Gamma 
&< \frac{1}{100} \left( \frac{4}{5} \right)^{\nmax-1} \label{eq:eps:gamma:2} \\
\eps_\Gamma
&<     \frac{b}{9 (b-1)} \label{eq:eps:gamma:3}  \\
2 b \eps_\Gamma  ( \cstar + 7  ) 
&<1-\beta 
\,. 
\label{eq:eps:gamma:4}
\end{align}
\end{subequations}
\item The parameter $\alpha_{\mathsf{R}} > 0$ from the $L^1$ loss of the inverse divergence operator  is now defined as 
\begin{align}
\alpha_{\mathsf{R}} = \frac{\eps_\Gamma (b-1)}{2b}  
\,.
\label{eq:alpha:DEF} 
\end{align}
\item The parameters $\NcutSmall$ and $\NcutLarge$ are used in Section~\ref{sec:cutoff} in order to define the velocity and stress cutoff functions. $\NcutLarge$ is the number of space derivatives which are embedded into the definitions of these cutoff functions, while $\NcutSmall$ is the number of material derivatives. See~\eqref{eq:h:j:q:def}, \eqref{eq:psi:i:q:recursive}, and \eqref{eq:g:i:q:n:def}. These large parameters are chosen solely in terms of $b$ and $\eps_\Gamma$ as 
\begin{align}
\frac 12 \NcutLarge = \NcutSmall = \left\lceil \frac{3b}{ \eps_\Gamma (b-1)} + \frac{15 b}{2} \right\rceil
\,.
\label{eq:Ncut:DEF}
\end{align}
\item The parameter $\Nindt$, which is the number of sharp material derivatives propagated on stresses and velocities in Sections~\ref{section:inductive:assumptions} through \ref{s:stress:estimates}, is chosen as the smallest integer for which we have
\begin{align}
\Nindt  =  \left \lceil\frac{4}{\eps_\Gamma (b-1)} \right \rceil  \NcutSmall
\,.
\label{eq:Nind:t:DEF}
\end{align}
\item The parameter $\Nindv$, whose primary role is to quantify the number of sharp space derivatives propagated on the velocity increments and stresses, cf.~\eqref{eq:inductive:assumption:derivative} and~\eqref{eq:Rq:inductive:assumption}, is chosen as the smallest integer for which we have the bounds
\begin{align}
4 b \Nindt  + 8 +  b (\mathsf{C_R}+3) \eps_\Gamma (b-1) + 2\beta (b^3-1) 
&< \eps_\Gamma (b-1)  \Nindv 
\,.
\label{eq:Nind:v:DEF}
\end{align}
\item The value of the decoupling parameter $\Ndec$, which is used in the $L^p$ decorrelation Lemma~\ref{lem:Lp:independence}, is chosen as the smallest integer for which we have
\begin{align}\label{eq:Ndec:DEF}
\Ndec  \left(  \frac{1}{30}\left(\frac 45\right)^{\nmax}  - \eps_\Gamma\right)
> \frac{4b}{b-1} 
\,.
\end{align}
\item The value of the parameter $\dpot$, which in essence is used in the inverse divergence operator of Proposition~\ref{prop:intermittent:inverse:div} to count the order of a parametrix expansion, is chosen as the smallest integer for which we have 
\begin{align}\label{eq:dpot:DEF}
(\dpot-1)   \left(  \frac{1}{30}\left(\frac 45\right)^{\nmax}  - \eps_\Gamma\right)
>  \frac{(12 \Nindv + 7)b}{b-1}  
\,.
\end{align}

\item\label{item:Nfin:DEF} The value of $\Nfin$, which is introduced in Section~\ref{section:inductive:assumptions} and used to quantify the highest order derivative estimates utilized throughout the scheme is chosen as the smallest integer such that 
\begin{align}
\label{eq:Nfin:DEF}
\frac{3}{2} \Nfin >  (2 \NcutSmall + \NcutLarge + 14 \Nindv + 2\dpot + 2\Ndec + 12) 2^{\nmax+1}
\,.
\end{align} 

\item\label{item:astar:DEF} Having chosen all the previous parameters in items \eqref{item:beta:DEF}--\eqref{item:Nfin:DEF}, there exits a {\em sufficiently large} parameter $a_* \geq 1 $, which depends on all the parameters listed above (which recursively means that $a_* = a_*(\beta ,b)$), and which allows us to choose $a$ an {\em arbitrary number} in the interval $[a_*,\infty)$. While we do not give a formula for $a_*$ explicitly, it is chosen so that $a_*^{(b-1)\eps_\Gamma}$ is at least twice larger than {\em all the implicit constants in the $\les$ symbols throughout the paper}; note that these constants only depend on the parameters in items \eqref{item:beta:DEF}--\eqref{item:Nfin:DEF} --- never on $q$ --- which justifies the existence of $a_*$.
\end{enumerate}

Having made the choices in items \eqref{item:beta:DEF}--\eqref{item:astar:DEF} above, we are now ready to define the $q$-dependent parameters which appear in the proof.

\subsection{Definitions of the $q$-dependent parameters}\label{ss:q:dependent:parameters}

\subsubsection{Parameters which depend on $q$}
For $q\geq 0$, we define the fundamental frequency parameter used in this paper as
\begin{align}
\lambda_q &= 2^{\big{\lceil} (b^q) \log_2 a \big{\rceil}} \,. \label{def:lambda:q:actual} 
\end{align}
Definition~\eqref{def:lambda:q:actual} gives that $\lambda_q$ is an integer power of $2$, and that we have the bounds
\begin{align}
a^{(b^q)} \leq \lambda_q \leq 2 a^{(b^q)}
\qquad\mbox{and}\qquad
\frac 13 \lambda_q^b \leq \lambda_{q+1} \leq 2
\lambda_q^b
\label{eq:lambda:q:to:q+1}
\end{align}
for all $q\geq 0$. Throughout the paper the above two inequalities are used by putting the factors of $\sfrac 13$ and $2$ into the implicit constants of $\les$ symbols. 
In terms of $\lambda_q$, the fundamental amplitude parameter used in the paper is
\begin{align}
\delta_q &=  \lambda_1^{(b+1)\beta}  \lambda_q^{-2\beta} \label{def:delta:q:actual}
\,.
\end{align}
In terms of the parameter $\eps_\Gamma$ from \eqref{eq:eps:Gamma:DEF}, we introduce a parameter which is used repeatedly throughout the paper to mean ``a tiny power of the frequency parameter'':
\begin{align}
\Gamma_{q+1} &= \left( \frac{\lambda_{q+1}}{\lambda_q} \right)^{\varepsilon_\Gamma} 
\,.
\label{def:Gamma:q:actual}
\end{align}
In order to cap off our derivative losses, we need to mollify in space and time using the operators described in Section~\ref{sec:mollifiers:Fourier} below. This is done in terms of the following space and time parameters:
\begin{align}
\tilde\lambda_q &= \lambda_q\Gamma_{q+1}^5   \label{eq:tilde:lambda:q:def} \\
\tilde\tau_q^{-1} &= \tau_q^{-1} \tilde\lambda_q^3 \tilde\lambda_{q+1} \label{eq:tilde:tau:q:def}
\,.
\end{align}
While $\tilde\tau_q$ is used for mollification and thus for rough material derivative bounds, the fundamental time parameter used in the paper for sharp material derivative bounds is
\begin{align}
\tau_q &= \left( \delta_q^{\sfrac 12}\tilde\lambda_q \Gamma_{q+1}^{\cstar+6} \right)^{-1} \label{def:tau:q:actual}
\,. 
\end{align}
Note that besides depending on the parameters introduced in \eqref{item:beta:DEF}--\eqref{item:astar:DEF}, the parameters introduced above only depend on $q$, but are independent of $n$ and $p$. 

\subsubsection{Parameters which depend also on $n$ and $p$}
The rest of the parameters depend on $n \in \{0, \ldots, \nmax\}$ and on $p\in \{0,\ldots,\pmax\}$.  We start by defining the frequency parameter $\lambda_{q,n}$ and the intermittency parameter $r_{q+1,n}$ by
\begin{align}
\lambda\qn  &=    2^{\left\lceil \left(\frac{4}{5}\right)^{n+1} \log_2 \lambda_q + \left(1-\left(\frac{4}{5}\right)^{n+1} \right) \log_2 \lambda_{q+1} \right\rceil} \label{eq:lambda:q:n:def} \\
r_{q+1,n} &= \frac{\lambda\qn}{\lambda_{q+1}} \label{eq:rqn:perp:definition} 
\end{align}
for $0\leq n \leq \nmax$.
In particular, \eqref{eq:lambda:q:n:def} shows that $\lambda_{q+1} r_{q+1,n}$ is an integer power of $2$, and we have the bound
\begin{align}
\lambda_q^{\left(\frac 45\right)^{n+1}} \lambda_{q+1}^{1 - \left(\frac 45\right)^{n+1}}
\leq 
\lambda\qn
\leq 
2 
\lambda_q^{\left(\frac 45\right)^{n+1}} \lambda_{q+1}^{1 - \left(\frac 45\right)^{n+1}} \,,
\label{eq:lambda:q:n:bounds}
\end{align}
while \eqref{eq:rqn:perp:definition} implies that $r_{q+1}^{-1}$ is an integer power of $2$, and we have the estimates
\begin{align}
\left(\frac{\lambda_q}{\lambda_{q+1}}\right)^{\ff^{n+1}}
\leq r_{q+1,n} 
\leq 
2 \left(\frac{\lambda_q}{\lambda_{q+1}}\right)^{\ff^{n+1}}\,.
\label{eq:r:q+1:n:bounds}
\end{align}
As with \eqref{eq:lambda:q:to:q+1} we absorb the factors of $2$ in \eqref{eq:lambda:q:n:bounds} and \eqref{eq:r:q+1:n:bounds} into the implicit constants in $\les$ symbols.

We also define the frequency parameters $\lambda_{q,n,p}$ by
\begin{align}
\lambda_{q,0,p} &= \Gamma_{q+1}\tilde\lambda_q & n=0, 0\leq p \leq \pmax \label{eq:lambda:q:0:1:def}\\
\lambda_{q,1,0} &= \lambda_q^{\frac{4}{5}}\lambda_{q+1}^{\frac{1}{5}} &n=1, p=0 \label{def:lambda:q:1:0:def} \\
\lambda_{q,n,0} &= \lambda_{q}^{\ff^{n-1}\cdot\frac{5}{6}} \lambda_{q+1}^{1-\ff^{n-1}\cdot\frac{5}{6}} & 2 \leq n \leq \nmax+1 \label{eq:lambda:q:n:0:def} \\
\lambda_{q,n,p} &= \lambda_{q,n,0}^{1-\sfrac{p}{\pmax}}\lambda_{q,n+1,0}^{\sfrac{p}{\pmax}} & 1\leq n \leq\nmax, 0\leq p \leq \pmax. \label{def:lambda:q:n:p}
\end{align}
For $0\leq n \leq \nmax$, we define
\begin{align}
f_{q,0} &=1 &n=0 \label{def:f:q:0} \\
f\qn &= \left(\frac{\lambda_{q,n+1,0}}{\lambda_{q,n,0}}\right)^{\sfrac{1}{\pmax}} &1\leq n \leq \nmax.    \label{def:f:q:n} 
\end{align}
We define $\delta_{q+1,0,p}$ by 
\begin{align}
\delta_{q+1,0,1} &= \Gamma_{q}^{\shaq}  \delta_{q+1} &p=1 \label{eq:delta:q:0:def}\\
\delta_{q+1,0,p}&=0 & 2\leq p \leq \pmax. \label{eq:delta:0:p:convention}
\end{align}
When $1\leq n \leq \nmax$ and $1\leq p \leq \pmax$, we define $\delta_{q+1,n,p}$ by 
\begin{equation}
\delta_{q+1,n,p} = \Gamma_{q}^{\shaq} \delta_{q+1} \cdot \left(\frac{\tilde\lambda_q}{\lambda\qnpminus}\right) \cdot \prod_{n'<n} \left( f_{q,n'} \Gamma_{q+1}^{8+\CLebesgue} \right)  \, . \label{eq:delta:q:n:def}
\end{equation}
We remark that by the definition of $\lambda_{q,1,0}$ given in \eqref{def:lambda:q:1:0:def}, and more generally $\lambda\qnp$ in \eqref{def:lambda:q:n:p}, the fact that $n\geq 1$, and a large choice of $\pmax$ which makes $f_{q,n}$ (defined in \eqref{def:f:q:n}) small, $\delta\qplusnp$ is significantly smaller than $\Gamma_q^\shaq\delta_{q+1}$.

For $1\leq n \leq \nmax$, we define $\cstarn$ in terms of $\cstar$ by 
\begin{align}
\label{def:cstarn:formula}
\cstarn &= \cstar - 4n  \,.
\end{align}

For $n=0$, we set
\begin{equation}\label{eq:Nfn0:def}
\NN{\textnormal{fin},0} = \frac{3}{2}\Nfin,
\end{equation}
while for $1\leq n \leq \nmax$, we define $\Nfn$ inductively on $n$ by using \eqref{eq:Nfn0:def} and the formula
\begin{equation}\label{def:Nfn:formula}
\Nfn =  \left\lfloor \frac 12 \left( \NN{\textnormal{fin},\textnormal{n}-1}  - \NcutSmall - \NcutLarge - 6 \right) - \dpot \right\rfloor  \, .
\end{equation}

\subsection{Inequalities and consequences of the parameter definitions}\label{ss:many:inequalities}
The definitions made in the previous two sections have the following consequences, which will be used frequently throughout the paper.  

Due to \eqref{def:lambda:q:actual} we have that $\Gamma_{q+1} \geq (\sfrac 12)^{b \eps_\Gamma} \lambda_q^{(b-1)\eps_\Gamma} \geq  (\sfrac 12)^{b \eps_\Gamma} \lambda_0^{(b-1)\eps_\Gamma} \geq (\sfrac 12) a_*^{(b-1)\eps_\Gamma}$. As was already mentioned in item~\eqref{item:astar:DEF}, we have chosen $a_*$ to be sufficiently large so that $a_*^{(b-1)\eps_\Gamma}$ is at least twice larger than all the implicit constants appearing in all $\les$ symbols throughout the paper. Therefore, for any $q\geq 0$, we may use a single power of $\Gamma_{q+1}$ to absorb any implicit constant in the paper: an inequality of the type $A \les B$ may be rewritten as $A \leq \Gamma_{q+1} B$.  

From \eqref{def:Gamma:q:actual}, \eqref{eq:tilde:lambda:q:def}, and \eqref{eq:eps:gamma:3}, we have that 
\begin{align}
\Gamma_{q+1}^4 \tilde \lambda_q  \leq  \lambda_{q+1}  \, . \label{eq:Lambda:q:x:1} 
\end{align}
From the definition \eqref{def:tau:q:actual} of $\tau_q$ and \eqref{def:cstarn:formula}, which gives that $\cstarn$ is decreasing with respect to $n$, we have that for all $0\leq n \leq \nmax$
\begin{align}
\Gamma_{q+1}^{\cstarn+6}  \delta_q^{\sfrac 12} {\tilde \lambda_q} &\leq \tau_q^{-1} \, . \label{eq:Lambda:q:x:1:NEW}
\end{align}

Using the definitions \eqref{def:delta:q:actual}, \eqref{def:Gamma:q:actual}, \eqref{eq:tilde:lambda:q:def}, and \eqref{def:tau:q:actual}, writing out everything in terms of $\lambda_{q-1}$, and appealing to \eqref{eq:eps:gamma:4}, we have that
\begin{align}
\tau_{q-1}^{ -1 }\Gamma_{q+1}^{3 + \cstar} &\leq \tau_q^{-1} \label{eq:Tau:q-1:q} \\
\tau_{q-1}^{-1}\Gamma_{q+1} &\leq \delta_q^{\sfrac 12} \lambda_q \, . \label{eq:tau:qminusone:deltaq}
\end{align}
From the definitions \eqref{eq:cstar:DEF} of $\cstar$ and \eqref{def:cstarn:formula} of $\cstarn$, we have that for all $0\leq n \leq \nmax$,
\begin{equation}\label{eq:cstarn:inequality}
-\cstarn + 4 \leq -1.
\end{equation}

From the definition of $\tilde\tau_q$, it is immediate that
\begin{align}
\tau_{q}^{-1} {\tilde \lambda_{q}^4}  &\leq {\tilde \tau_q^{-1}} \leq  \tau_{q}^{-1} {\tilde \lambda_{q}^3} {\tilde \lambda_{q+1}}  \, .
\label{eq:Lambda:q:t:1}
\end{align}
From \eqref{eq:eps:gamma:4}, the assumption that $\beta\geq\sfrac{1}{3}$, and the assumption $b\leq\sfrac{3}{2}$, we can write everything out in terms of $\lambda_q$ to deduce that
\begin{align}
\tau_q^{-1}\Gamma_{q+1}^{9} &\leq \tau_{q+1}^{-1} 
\,. \label{eq:Lambda:q:t:3}
\end{align}

From the definitions \eqref{eq:lambda:q:n:def} and \eqref{eq:lambda:q:0:1:def}--\eqref{def:lambda:q:n:p}, for all $0\leq n \leq \nmax$ and $0\leq p \leq \pmax$ we have   
\begin{equation*}
\frac{\lambda\qnp}{\lambda\qn} \ll 1 \, .
\end{equation*}
More precisely, when $n=0$ we have that 
\begin{align}
\frac{\Gamma_{q+1} \lambda\qnp}{\lambda\qn} = \frac{\Gamma_{q+1}^2 \tilde \lambda_q}{\lambda_{q,0}} = \frac{\Gamma_{q+1}^7 \lambda_q}{\lambda_{q,0}} = \left(\frac{\lambda_{q+1}}{\lambda_q}\right)^{- \frac 15 + 7 \eps_\Gamma}
\label{eq:lambdaqn:identity:2+}
\end{align}
while for $n\geq 1$ it holds that 
\begin{align}
\frac{\Gamma_{q+1} \lambda\qnp}{\lambda\qn} \leq \frac{\Gamma_{q+1} \lambda_{q,n+1,0}}{\lambda\qn}  
= \left(\frac{\lambda_{q+1}}{\lambda_q}\right)^{(\frac 45)^n (\frac 45 - \frac 56) + \eps_\Gamma} 
\leq \left(\frac{\lambda_{q+1}}{\lambda_q}\right)^{- \frac{1}{30} (\frac 45)^{\nmax} +\eps_\Gamma}
\label{eq:lambdaqn:identity:2++}
\end{align}
as it is clear that the quotient on the left hand side is largest when $n=\nmax$.  
Note that due to \eqref{eq:nmax:DEF} we have $\frac{1}{30}\left(\frac 45\right)^{\nmax}  - \eps_\Gamma < \frac{1-2\beta b}{30} - \eps_\Gamma \leq \frac{1}{5} - 7 \eps_\Gamma$; here we also used that $\eps_\Gamma \leq \frac{1}{36}$, which handily follows from \eqref{eq:eps:gamma:2}. Combining \eqref{eq:lambdaqn:identity:2+} and \eqref{eq:lambdaqn:identity:2++} we thus arrive at 
\begin{align}
\frac{\Gamma_{q+1} \lambda\qnp}{\lambda\qn} \leq   \left(\frac{\lambda_{q+1}}{\lambda_q}\right)^{- \frac{1}{30} (\frac 45)^{\nmax} +\eps_\Gamma} \leq \left(2 \lambda_q^{b-1}\right)^{- \frac{1}{30} (\frac 45)^{\nmax} +\eps_\Gamma}
\label{eq:lambdaqn:identity:2+++}
\end{align}
for all $0\leq n \leq \nmax$ and $0\leq p\leq \pmax$. Combining the above estimate with our choice of $\Ndec$ in \eqref{eq:Ndec:DEF}, we thus arrive at 
\begin{align}
   \lambda_{q+1}^4 \leq \left( \frac{\lambda\qnn}{2\pi\sqrt{3}\Gamma_{q+1}\lambda\qnnpp} \right)^{\Ndec} \, . \label{eq:lambdaqn:identity:2}
\end{align}
for all $0\leq\nn\leq\nmax$ and $1\leq\pp\leq\pmax$.

Next, we a list a few consequences of the fact that $\Nindv \gg \Nindt$, as specified in \eqref{eq:Nind:v:DEF}. First, we note from \eqref{eq:Lambda:q:t:1} that 
\begin{align}
 \tilde \tau_{q-1}^{-1} \tau_{q-1} \leq \tilde \lambda_{q-1}^3 \tilde \lambda_q \leq \lambda_q^4  
 \label{eq:trickery:trickery}
\end{align}
where in the second inequality we have used that $\eps_\Gamma \leq \frac{3}{20 b}$. In turn, the above inequality combined with \eqref{eq:Nind:v:DEF} implies the following estimates, all of which are used for the first time in Section~\ref{sec:mollification:stuff}:
\begin{subequations}
\begin{align}
 \lambda_{q-1}^8 \Gamma_{q+1}^{1+  \mathsf{C_R}  } \frac{\delta_{q-1}}{\delta_{q+2}}  
 \left( \tilde \tau_{q-1}^{-1}   \tau_{q-1} \right)^{\Nindt} 
&\leq     \Gamma_{q}^{\Nindv-2}
\label{eq:N:c:condition:2} \\
\tilde \lambda_q^2 \left( \Tilde{\tau}_{q-1}^{-1} \tau_{q-1}  \right)^{\Nindt} 
&\leq \Gamma_{q+1}^{5 \Nindv}  
\label{eq:N:c:condition:2:new}\\
 \lambda_{q-1}^4 \delta_{q-1}^{\sfrac 12} \Gamma_q^2 \delta_q^{-\sfrac 12}  ( \tilde \tau_{q-1}^{-1}\tau_{q-1} )^{\Nindt}  &\leq \Gamma_q^{\Nindv}
 \,.
\label{eq:N:c:condition:2:also:new} 
\end{align}
\end{subequations}
 
Next, as a consequence of our choice of $\NcutSmall$ and $\NcutLarge$ in \eqref{eq:Ncut:DEF}, we obtain the following bounds, which are used in Section~\ref{sec:cutoff}  
\begin{align}
\tilde\lambda_{q}^{\sfrac 32} \Gamma_{q}^{- \NcutSmall} 
\leq \lambda_q^3 \Gamma_q^{-\NcutSmall}   \leq  1 \,.
   \label{eq:Nind:cond:3}
\end{align}
for all $q\geq 0$. The fact that $\Nindt$ is taken to be much larger than $\NcutSmall$, as expressed in \eqref{eq:Nind:t:DEF}, implies when combined with \eqref{eq:trickery:trickery} the following bound, which is also used in Section~\ref{sec:cutoff}:
\begin{align}
 \left(\tau_q  \tilde \tau_q^{-1}\right)^{\Ncut}
 \leq \lambda_{q+1}^{4\Ncut} 
\leq \Gamma_{q+1}^{ \Nindt }
\label{eq:Nind:cond:2}
\end{align}
for all $q\geq 1$.  

The parameter $\alpha_{\mathsf R}$ is chosen in \eqref{eq:alpha:DEF} in order to ensure the inequality
\begin{equation}\label{eq:alpha:equation:1}
\lambda_{q+1}^{\alpha_{\mathsf{R}}} \leq \Gamma_{q+1}.
\end{equation}
for all $q\geq 0$. This fact is used in Section~\ref{s:stress:estimates}.
Several other, much more hideous, parameter inequalities are used in Section~\ref{s:stress:estimates}, and for the readers' convenience we list them next. First, we claim that 
\begin{equation}
\label{eq:hopeless:mess:new}
\Gamma_{q+1} \Gamma_q^{\shaq} \delta_{q+1} \tilde\lambda_q \prod_{n'\leq \nmax} \left(f_{q,n'}\Gamma_{q+1}^{8+\CLebesgue}\right) \lambda_{q,\nmax+1,0}^{-1} 
\leq \Gamma_{q+1}^\shaq \Gamma_{q+1}^{-1} \delta_{q+2}
\,.
\end{equation}
In order to verify the above bound, we  appeal  to to the choices made in \eqref{eq:b:DEF}, \eqref{eq:nmax:DEF}, and \eqref{eq:pmax:DEF}, to the definitions \eqref{eq:tilde:lambda:q:def}, \eqref{def:lambda:q:1:0:def}, \eqref{eq:lambda:q:n:0:def}, \eqref{def:f:q:n}, and the fact that $\nn \leq \nmax$, to deduce that the left side of \eqref{eq:hopeless:mess:new} is bounded from above by
\begin{align*}
&\delta_{q+1} \Gamma_{q+1}^{6 + \nmax (8+\CLebesgue)} \frac{\lambda_q }{\lambda_{q,\nmax+1,0}} \left( \frac{\lambda_{q,\nmax+1,0}}{\lambda_{q,1,0}}\right)^{\frac{1}{\pmax}} \notag\\
&= \delta_{q+1} \Gamma_{q+1}^{6 + \nmax (8+\CLebesgue)} \left( \frac{\lambda_{q}}{\lambda_{q+1}}\right)^{ \left( 1 - (\frac 45)^{\nmax} \frac 56\right)}  \left( \frac{\lambda_{q+1}}{\lambda_{q}}\right)^{\frac{1}{\pmax} \left( \frac 45 - (\frac 45)^{\nmax} \frac 56\right)}\notag\\
&\leq \frac{\lambda_q \delta_{q+1}}{\lambda_{q+1}} \Gamma_{q+1}^{6 + \nmax (8+\CLebesgue)}  
\left( \frac{\lambda_{q+1}}{\lambda_{q}}\right)^{(1-2\beta b)\frac 45 }
\left( \frac{\lambda_{q+1}}{\lambda_{q}}\right)^{\frac{1-2\beta b}{10} \frac{4}{5} }\notag\\
&\leq \left(\Gamma_{q+1}^\shaq \Gamma_{q+1}^{-1} \delta_{q+2}\right) \frac{\lambda_q \delta_{q+1}}{\lambda_{q+1}\delta_{q+2}} \Gamma_{q+1}^{7 + \mathsf{C_R} + \nmax (8+\CLebesgue)}  
\left( \frac{\lambda_{q+1}}{\lambda_{q}}\right)^{(1-2\beta b) \frac{22}{25} } \notag\\
&\leq \left(\Gamma_{q+1}^\shaq \Gamma_{q+1}^{-1} \delta_{q+2}\right)   \Gamma_{q+1}^{7 + \mathsf{C_R} + \nmax (8+\CLebesgue)}  
\left( \frac{\lambda_{q+1}}{\lambda_{q}}\right)^{- (1-2\beta b) \frac{3}{25} }
\end{align*}
The proof of \eqref{eq:hopeless:mess:new} is now completed by appealing to \eqref{eq:eps:gamma:1}, which ensures that $\Gamma_{q+1}$ represents a sufficiently small power of $\sfrac{\lambda_{q+1}}{\lambda_q}$.

Next, we claim that due to our choice of $\dpot$, we have
\begin{equation}
\label{eq:CF:new}
\Gamma_q^\shaq \delta_{q+1}\tilde\lambda_q \prod_{n'\leq\nmax} \left( f_{q,n'} \Gamma_{q+1}^{8+\CLebesgue} \right)   \lambda_{q+1}  \left( \frac{\Gamma_{q+1} \lambda_{q,\nn,\pmax}}{\lambda_{q,\nn}} \right)^{\dpot -1} \left( \lambda_{q+1}^{4} \right)^{3\NindLarge} \leq \frac{\delta_{q+2}}{\lambda_{q+1}^5}.
\end{equation}
In order to verify the above bound we use the previously established estimate \eqref{eq:hopeless:mess:new}  in conjunction with \eqref{eq:lambdaqn:identity:2+++}; after dropping the helpful factor of $\Gamma_{q+1}^{-2 - \mathsf{C_R}}$, we deduce that the left side of \eqref{eq:CF:new} is bounded from above by 
\begin{align*}
& \delta_{q+2} \lambda_{q,\nmax+1,0} \lambda_{q+1}  \left( \frac{\Gamma_{q+1} \lambda_{q,\nn,\pmax}}{\lambda_{q,\nn}} \right)^{\dpot -1} \left( \lambda_{q+1}^{4} \right)^{3\NindLarge} \notag\\
&\leq \frac{\delta_{q+2}}{\lambda_{q+1}^5}  \lambda_{q+1}^3 \left(2 \lambda_q^{b-1}\right)^{- (\dpot-1)\left(\frac{1}{30} (\frac 45)^{\nmax} - \eps_\Gamma\right)}   \lambda_{q+1}^{12 \NindLarge}
\end{align*}
The choice of $\dpot$ in \eqref{eq:dpot:DEF} shows that the above estimate directly implies \eqref{eq:CF:new}.

The amplitudes of the higher order corrections $w\qplusnp$ must meet the inductive assumptions stated in \eqref{eq:inductive:assumption:derivative:q}.  In order to meet the satisfactory bound in Remark~\ref{rem:checking:inductive:velocity}, from \eqref{eq:delta:q:0:def}--\eqref{eq:delta:q:n:def}, we deduce the bound
\begin{equation}\label{eq:delta:q:nn:pp:ineq}
\delta_{q+1,\nn,\pp}^{\sfrac{1}{2}} \leq \Gamma_{q+1}^{-2} \delta_{q+1}^{\sfrac{1}{2}}.
\end{equation}
Indeed, the case $\nn=0$ follows from the definition of $\mathsf{C_R}$ in \eqref{eq:shaq:DEF}, while the case $\nn\geq 1$ is a consequence of the definition \eqref{eq:delta:q:n:def}, which implies that  $\delta_{q,\nn,\pp} \leq \delta_{q,0,1}$, for any $\nn\geq 1$ and any $\pp \geq 1$. 

Another parameter inequality which is necessary to estimate the transport and Nash errors in Sections~\ref{ss:stress:transport} and \ref{ss:stress:Nash}, is
\begin{equation}\label{eq:drq:identity}
\Gamma_{q+1}^{4+\frac{\CLebesgue}{2}} \delta_{q+1,\nn,1}^\frac{1}{2}\tau_q^{-1} r_{q+1,\nn} \lambda_{q+1}^{-1} \leq \Gamma_{q+1}^{\shaq-1} \delta_{q+2}
\end{equation}
for all $0 \leq \nn \leq \nmax$.  When $\nn=0$, this inequality may be deduced by writing everything out in terms of $\lambda_q$, appealing to the appropriate definitions, and then using that $\beta<\sfrac{1}{2}$ from item~\ref{item:beta:DEF}, \eqref{eq:b:DEF}, \eqref{eq:CLebesgue:DEF}, \eqref{eq:shaq:DEF}, \eqref{eq:cstar:DEF}, \eqref{eq:eps:gamma:2}, after which one arrives at
\begin{align*}
    \varepsilon_\Gamma\left( 4 + \frac{b-4}{2} + \frac{1}{2}\shaq + \cstar + 12 \right) + \beta(2b+1) < \frac{1}{100} + \frac{3}{2} < \frac{9}{5}.
\end{align*}
It is clear there is quite a bit of room in the above inequality, and similarly, \eqref{eq:drq:identity} becomes \emph{most} restrictive when $\nn=\nmax$.  In this case, one may again write everything out in terms of $\lambda_q$, move everything to the left hand side, and appeal to the most of the same referenced inequalities as before to see that
\begin{align*}
    \varepsilon_\Gamma\left(22+4\nmax\right) + \beta(2b+1) - \frac{3}{2} &\leq  \varepsilon_\Gamma\left(22+4\nmax\right) + \beta -\frac{1}{2} < 0 \, ,
\end{align*}
where in the last inequality we have instead appealed to \eqref{eq:eps:gamma:1} rather than \eqref{eq:eps:gamma:2}, proving \eqref{eq:drq:identity} in the remaining cases $1\leq \nn \leq \nmax$.

Parameter inequalities which play a crucial role in showing that the Oscillation~2 type errors vanish, see~Section~\ref{ss:stress:oscillation:2}, are:
\begin{subequations}
\label{eq:Super:Mario:inequalities}
\begin{align}
16 C_* \const_{\textnormal{pipe}} \const_\eta \Gamma_{q+1} \left( \frac{\lambda_q}{\lambda_{q+1}} \right)^{\ff^{\nn+1}\cdot 4} 
&< \left( \frac{\lambda_q}{\lambda_{q+1}}  \right)^{\ff^{\nn-1}\cdot\frac{5}{6}\cdot3}\Gamma_{q+1}^{-3}
\,, \qquad \mbox{for} \qquad \nn \geq 2 \,,\label{eq:parameter:relative:intermittency} \\
16 C_* \const_{\textnormal{pipe}} \const_\eta \Gamma_{q+1} \left(\frac{\lambda_q}{\lambda_{q+1}}\right)^{\frac{4}{5}\cdot 4} 
&< \left(\frac{\tilde\lambda_q}{\lambda_{q+1}}\right)^3 \,  \label{eq:parameter:relative:intermittency:0} \,,\\
16 C_* \const_{\textnormal{pipe}} \const_\eta \Gamma_{q+1}^4 \left(\frac{\lambda_q}{\lambda_{q+1}}\right)^{\left(\frac{4}{5}\right)^2\cdot 4} 
&< \left(\frac{\lambda_q}{\lambda_{q+1}}\right)^{\frac{4}{5}\cdot 3} \, . \label{eq:parameter:relative:intermittency:1}
\end{align}
\end{subequations}
where $C_*$ is the geometric constant from Lemma~\ref{prop:disjoint:support:simple:alternate}--estimate~\eqref{eq:r1:r2:condition:alt}, $\const_{\textnormal{pipe}}$ is a geometric  constant which appears in Lemma~\ref{lem:overlap:3}--estimate~\eqref{eq:counting:support}, and $\const_\eta$ is the constant from Lemma~\ref{l:overlap}. In order to verify \eqref{eq:Super:Mario:inequalities}, we first note that $C_* \const_{\textnormal{pipe}} \const_\eta \leq \Gamma_{q+1}$, since $a_*$ was chosen to be sufficiently large. Inequality \eqref{eq:parameter:relative:intermittency:0} is then an immediate consequence of the fact that $\sfrac{16}{5} > 3$. The bound \eqref{eq:parameter:relative:intermittency} follows from 
\begin{equation}\label{eq:algebra:and:arithmetic}
\Gamma_{q+1}^5 
< \left( \frac{\lambda_{q+1}}{\lambda_{q}} \right)^{\ff^{\nmax-1} \left( \frac{64}{25}   - \frac{5}{2}  \right)} 
\leq \left( \frac{\lambda_{q+1}}{\lambda_{q}} \right)^{\ff^{\nn+1}\cdot 4 - \ff^{\nn-1}\cdot\frac{5}{6}\cdot3} \,.
\end{equation}
The second inequality in the above display is a consequence of $\nn \leq \nmax$, while the first one follows from \eqref{eq:eps:gamma:2}.  Finally, inequality \eqref{eq:parameter:relative:intermittency:1} is a consequence of the fact that $\sfrac{64}{25}-\sfrac{12}{5}>\sfrac{64}{25}-\sfrac{5}{2}$ and the first inequality in \eqref{eq:algebra:and:arithmetic}, which bounds $\Gamma_{q+1}^5$.

We conclude this section by verifying a few inequalities concerning the parameter $\Nfn$, which counts the number of available space-plus-material derivative for the residual stress $\RR_{q,n}$. For all $0 \leq n \leq \nmax$ we require that  
\begin{subequations}
\label{eq:Luigi:is:Mario:s:brother}
\begin{align}
\Nindt, 2\Ndec + 4 
&\leq \lfloor \sfrac{1}{2}\left(\Nfn-\NcutSmall-\NcutLarge-5\right) \rfloor - \dpot \,,
\label{eq:lambdaqn:identity:3} \\
14 \NindLarge
&\leq \NN{\textnormal{fin},\textnormal{n}} - \NcutSmall-\NcutLarge - 2\Ndec - 9 \,,
\label{eq:Nfinn:inequality} \\
6\NindLarge  
&\leq \lfloor \sfrac{1}{2}\left(\Nfn-\NcutSmall-\NcutLarge-6\right) \rfloor - \dpot  \,,
\label{eq:nfnn:mess} \\
6\NindLarge 
&\leq \lfloor \sfrac{1}{4}\left(\Nfn-\NcutSmall-\NcutLarge-7\right) \rfloor\,.
\label{eq:nfnn:mess:2}
\end{align}
\end{subequations}
for all $0\leq n\leq\nmax$.  Additionally for  $0\leq \nn<n \leq \nmax$, we require that 
\begin{align}
\lfloor \sfrac{1}{2}\left(\Nfnn-\NcutSmall-\NcutLarge-6\right) \rfloor - \dpot \geq \Nfn
\label{eq:nfnn:nfn:mess}
\end{align}
holds. The inequality \eqref{eq:nfnn:nfn:mess} is a direct consequence of the recursive formula~\eqref{def:Nfn:formula}  and of the fact that the sequence $\Nfn$ is monotone decreasing with respect to $n$. Using \eqref{eq:Nfn0:def} and \eqref{def:Nfn:formula} one may show that
$$
\Nfn \geq 2^{-n} \NN{\textnormal{fin},0} - (2 \dpot + \NcutSmall + \NcutLarge + 8)
\,.
$$
Noting that the bounds \eqref{eq:Luigi:is:Mario:s:brother} are most restrictive for $n = \nmax$, they now readily follow from our choice \eqref{eq:Nfin:DEF}.

\subsection{Mollifiers and Fourier projectors}
\label{sec:mollifiers:Fourier}
Let $\phi(\zeta):\mathbb{R}\rightarrow \mathbb{R}$ be a smooth, $C^\infty$ function compactly supported in the set $\{\zeta: |\zeta|\leq 1 \}$ which in addition satisfies
\begin{equation}\label{eq:phi}
\int \phi(\zeta) \,d\zeta = 1, \qquad \int \phi(\zeta) \zeta^n =0 \quad \forall n=1,2,...,\Nindv.
\end{equation}
Let $\tilde{\phi}(x):\mathbb{R}^3\rightarrow \mathbb{R}$ be defined by $\tilde{\phi}(x)=\phi(|x|)$. For $\lambda,\mu\in\mathbb{R}$, define
\begin{equation}\label{eq:mollifiers}
    \phi_{\lambda}^{(x)}(x) = {{\lambda}^3} \tilde{\phi}\left( \lambda x \right), \qquad \phi_\mu^{(t)}(t) = \mu \phi(\mu t).
\end{equation}
For $q\in\mathbb{N}$, we will define the spatial and temporal convolution operators
\begin{equation}\label{mollifier:operators}
   \Pqx := \phi_{\tilde{\lambda}_q}^{(x)} \ast, \qquad \Pqt := \phi_{\tilde{\tau}_{q-1}^{-1}}^{(t)} \ast , \qquad \Pqxt := \Pqx   \Pqt.
\end{equation}

We will use the notation $\Proj_{\leq   \lambda}$ to denote the standard (Littlewood-Paley) Fourier projection operators onto spatial frequencies which are less than or equal to $\lambda$, $\Proj_{\geq \lambda}$ to denote the standard Littlewood-Paley projection operators onto spatial frequencies which are greater than or equal to $\lambda$, and the notation
$$  \Proj_{[\lambda_1,\lambda_2)}  $$
to denote the Fourier projection operator onto spatial frequencies $\xi$ such that $\lambda_1 \leq |\xi| <\lambda_2$.
If $\lambda_1=\lambda_2$, we adopt the convention that $\Proj_{[\lambda_1,\lambda_2)}f=0$ for any $f$.

\subsection{Notation}\label{ss:notation}
\begin{align}
\MM{n,N,\lambda,\Lambda} &= \lambda^{\min\{n,N\}} \Lambda^{\max\{n-N,0\}} \notag  \\
a \otimes_{\mathrm{s}} b &= \frac 12 (a   \otimes b + b   \otimes a)
\label{eq:otimes:s} \\
a \symring b &= \frac 12 (a \,\mathring \otimes \, b + b \, \mathring  \otimes \, a)
\label{eq:otimes:symm} \\
\supp_t f &= \overline{\{t: f|_{\mathbb{T}^3\times \{t\}} \not\equiv 0 \}} \label{eq:time:support}
\end{align}
We will use repeatedly the notation (noted in the introduction in \eqref{e:intro_v_est} and \eqref{e:intro_R_est} and in Remark~\ref{rem:norms:are:uniform:inductive}) 
\begin{equation}\label{eq:time:support:def}
\norm{f}_{L^p}:=\norm{f}_{L^\infty_t (L^p(\T^3))} \, .
\end{equation}
That is, all $L^p$ norms stand for {\em $L^p$ norms in  space, uniformly in time}. Similarly, when we wish to emphasize a set dependence on $\Omega\subset \R \times \T^3$ of an $L^p$ norm, we write 
\begin{equation}\label{eq:time:support:def:2}
\norm{f}_{L^p(\Omega)}:=\norm{{\mathbf{1}}_{\Omega}\; f}_{L^\infty_t (L^p(\T^3))}\, .
\end{equation}

\appendix

\section{Useful lemmas}\label{sec:appendix}

This appendix contains a collection of auxiliary lemmas which are used throughout the paper:
\begin{itemize}
    \item Section~\ref{sec:transport:est} recalls the classical $C^N$ estimates for solutions of the transport equation. This is for instance used in Section~\ref{s:deformation}.
    \item Section~\ref{app:lemma:5:1} gives the detailed construction of the basic cutoff functions $\tilde \psi_{m,q}$ and $\psi_{m,q}$, which are used in Section~\ref{sec:cutoff} to construct the velocity and the stress cutoff functions.
    \item Section~\ref{sec:Lp:decorrelation} recalls the fundamental fact that the $L^p$ norm of the product of a slowly oscillating function and a fast periodic function is essentially bounded by the product of their $L^p$ norms. 
    \item Section~\ref{sec:Sobolev:inequality:cutoffs} contains a version of the Sobolev inequality which takes into account the support of the velocity cutoff functions. 
    \item Section~\ref{sec:Faa:di:Bruno} contains a number of consequences of the multivariate Faa di Bruno formula. Most of the results here are used for bounding the space and material derivatives of the cutoff functions in Section~\ref{sec:cutoff}. We also present here, cf.~Lemma~\ref{l:slow_fast}, a version of the $L^p$ decorrelation lemma from Section~\ref{sec:Lp:decorrelation} in which the fast periodic function is composed with a volume-preserving flow map. Lemma~\ref{l:slow_fast} plays a crucial role in estimating the $L^2$ norms of the velocity increments in Section~\ref{ss:stress:w:estimates}.
    \item Sections~\ref{sec:operator:iterates} and~\ref{sec:operator:commutators} contain a number of lemmas 
    which allow us to go back and forth between information for (arbitrarily) high order derivative bounds in Eulerian and Lagrangian variables. These lemmas concerning sums of operators and commutators with material derivatives are frequently used throughout the paper to overcome the fact that material derivatives and spatial/temporal derivatives do not commute. 
    \item Section~\ref{sec:inverse:divergence} introduces in~Proposition~\ref{prop:intermittent:inverse:div} the inverse divergence operator used in this paper. We call this operator ``intermittency friendly'' because it is composed of a principal part which precisely maintains the spatial support of the vector field it is applied to, plus a secondary part which is nonlocal, but whose amplitude is incredibly small. It is here that the definition~\eqref{def:pipe:flow:main} for the density of our pipe flows plays an important role, as the high order $\dpot$ of the Laplacian present in~\eqref{def:pipe:flow:main} allows us to perform a parametric expansion which maintains (to leading order) the support of pipes, and also takes into account deformations due to the flow map.
\end{itemize}

\subsection{Transport estimates}
\label{sec:transport:est}
We shall require the following estimates for smooth solutions of transport equations.  For proofs we refer the reader to~\cite[Appendix~D]{BDLISZ15}. 

\begin{lemma}[\textbf{Transport Estimates}]\label{transport}
Consider the transport equation
$$ \partial_t f + u \cdot \nabla f = g, \qquad \qquad f|_{t_0} = f_0 $$
where $f,g:\mathbb{T}^n\rightarrow \mathbb{R}$ and $u:\mathbb{T}^n\rightarrow \mathbb{R}^n$ are smooth functions.  Let $X$ be the flow of $u$, defined by 
$$ \frac{d}{dt}X = u(X,t) , \qquad X(x,t_0)=x,$$
and let $\Phi$ be the inverse of the flow of $X$, which is the identity at time $t_0$.
Then the following hold:    
\begin{enumerate}[(1)]
    \item\label{item:transport:estimate:1} $\displaystyle \| f(t) \|_{C^0} \leq \|f_0\|_{C^0} + \int_{t_0}^t \|g(s)\|_{C^0} \,ds $
    \item\label{item:transport:estimate:2} $\displaystyle \| Df(t) \|_{C^0} \leq \| Df_0 \|_{C^0} e^{(t-t_0)\|Du\|_{C^0}} + \int_{t_0}^t e^{(t-s)\|Du\|_{C^0}} \|Dg(s)\|_{C^0} \,ds $
    \item\label{item:transport:estimate:3} For any $N\geq 2$, there exists a constant $C=C(N)$ such that
    \begin{align*}
        \qquad\quad \|D^N f(t)\|_{C^0} \leq  &\left( \|D^N f_0\|_{C^0} + C(t-t_0)\|D^n u\|_{C^0} \| Df \|_{C^0} \right) e^{C(t-t_0)\|Du\|_{C^0}}\\
        & \quad + \int_{t_0}^t e^{C(t-s)\|Du\|_{C^0}} \left( \|D^N g(s)\|_{C^0} + C(t-s) \|D^N u \|_{C^0} \|Dg(s)\|_{C^0} \right) \,ds
    \end{align*}
    \item\label{item:transport:estimate:4} $\displaystyle \|D\Phi(t)-\Id \|_{C^0} \leq e^{(t-t_0)\|Du\|_{C^0}} - 1 \leq (t-t_0)\|Du\|_{C^0} e^{(t-t_0)\|Du\|_{C^0}} $
    \item\label{item:transport:estimate:5} For $N\geq 2$ and a constant $C=C(N)$, $$ \| D^N \Phi(t)\|_{C^0} \leq C(t-t_0)\| D^N u \|_{C^0} e^{C(t-t_0)\|Du\|_{C^0}} $$
\end{enumerate}
\end{lemma}

\subsection{Proof of Lemma~\ref{lem:cutoff:construction:first:statement}}
\label{app:lemma:5:1}
We first consider the function
\begin{align}
f(x)=\begin{cases}0&\mbox{if }x\leq 0\\ e^{-\frac{1}{x^2}}&\mbox{if }x>0.
\label{def:cutoff:f}
\end{cases}
\end{align}
We claim that for all $0\leq N \leq \Nfin$ and $x>0$,
\begin{equation}\label{eq:f:derivative:estimates}
    \frac{|D^N f(x)|}{(f(x))^{1-\frac{N}{\Nfin}}} \lesssim 1.
\end{equation}
The proof of this is achieved in two steps; first, one can show by induction that for all $0\leq N \leq \Nfin$, there exist constants $K_N$ and $c_k$ for $0\leq k \leq K_N$ such that
\begin{equation}\label{eq:f:quotient:estimates}
    D^N \left( e^{-\frac{1}{x^2}}\right) = \sum\limits_{k=0}^{K_N} \frac{c_k}{x^k} e^{-\frac{1}{x^2}}.
\end{equation}
Next, one may also check that for any powers $p,q>0$, 
\begin{equation}\label{eq:p:limit:0}
 \lim_{x\rightarrow 0^+} e^{-\frac{q}{x^2}} \frac{1}{x^p} =0.   
\end{equation}
Then for $1\leq N \leq \Nfin$, we see that $0\leq 1-\frac{N}{\Nfin} < 1$, and so using \eqref{eq:f:quotient:estimates}, we have that the left-hand side of \eqref{eq:f:derivative:estimates} may be split into a finite linear combination of terms of the form in \eqref{eq:p:limit:0}, showing that \eqref{eq:f:derivative:estimates} is valid. 

We now glue together two versions of $f$ as follows with the goal of forming a prototypical cutoff function $\psi$. First, let $x_0=\sqrt{\frac{1}{\ln(2)}}$ so that $f(x_0)=\frac{1}{2}$.  Now consider the function $\tilde f(x)=f(2x_0-x)$, and set
\begin{align}
F(x)=\begin{cases}f(x)&\mbox{if }x\leq x_0\\ 1-f(2x_0-x)&\mbox{if }x>x_0.
\label{def:cutoff:F}
\end{cases}
\end{align}
Then $F(x)$ is continuous everywhere, and $C^\infty$ everywhere except $x_0$, where it is not necessarily differentiable.  Furthermore, one can check that by the definition of $F$ and \eqref{eq:f:derivative:estimates}, for all $0\leq N \leq\Nfin$,
\begin{equation}\label{eq:F:derivative:estimates}
    \frac{|D^N F(x)|}{(F(x))^{1-\frac{N}{\Nfin}}} \lesssim 1 \mbox{ for all } 0<x<x_0, \qquad \frac{ |D^N \left(1-(F(x))^2\right)^\frac{1}{2} |}{\left(1-(F(x))^2\right)^{\frac{1}{2}\left(1-\frac{N}{\Nfin}\right)}} \lesssim 1 \mbox{ for all } x_0<x<2x_0.
\end{equation}
The latter inequality follows from noticing that for $x$ close to $2x_0$,
\begin{align*}
    \left(1-(F(x))^2\right)^\frac{1}{2} &= \left( (1+F(x))(1-F(x)) \right)^\frac{1}{2} 
    = \left(1+F(x)\right)^\frac{1}{2}\left(f(2x_0-x)\right)^\frac{1}{2} \,.
\end{align*}
Since multiplying by a smooth function strictly larger than $1$, rescaling $f$ by a fixed parameter, and raising $f$ to a positive power preserves the estimate \eqref{eq:f:derivative:estimates} up to implicit constants (in fact raising $f$ to a power is equivalent to rescaling), \eqref{eq:F:derivative:estimates} is verified.

Towards the goal of adjusting $F$ to be differentiable at $x_0$, let $E$ be the set $\left(\frac{x_0}{2},\frac{3x_0}{2}\right)$, and let $\phi$ be a compactly, $C^\infty$ mollifier such that the support of the mollified characteristic function $\mathcal{X}_{E}\ast \phi(x)$ is contained in $\left(\frac{x_0}{4},\frac{7x_0}{4}\right)$.  Setting 
\begin{equation}\label{def:psi:base:construction}
    \psi(x) = \left(\mathcal{X}_{E}\ast \phi(x)\right) \phi \ast F(x) + \left(1-\mathcal{X}_{E}\ast \phi(x)\right)F(x),
\end{equation}
one may check that $\psi$ is $C^\infty$ and has the following properties:
\begin{align}
    \psi(x)&=0 \mbox{ for } x \leq 0 \label{eq:psi:property:1}\\
    0<\psi(x)&<1 \mbox{ for } 0<x<2x_0 \label{eq:psi:property:2}\\
    \psi(x) &= 1 \mbox{ for } x \geq 2x_0 \label{eq:psi:property:3}\\
    \frac{|D^N \psi(x)|}{(\psi(x))^{1-\frac{N}{\Nfin}}} &\lesssim 1 \mbox{ for all } 0<x \label{eq:psi:property:4}\\
    \frac{ |D^N \left(1-(\psi(x))^2\right)^\frac{1}{2} |}{\left(1-(\psi(x))^2\right)^{\frac{1}{2}\left(1-\frac{N}{\Nfin}\right)}} &\lesssim 1 \mbox{ for all } 0<x<2x_0. \label{eq:psi:property:5}
\end{align}

We can now build $\tilde\psi_{m,q}$. By rescaling and translating $\psi$ and using \eqref{eq:psi:property:1}-\eqref{eq:psi:property:3}, one can check that
\begin{equation}\label{eq:psi:m:q:tilde:def}
    \tilde\psi_{m,q}(x)=\psi\left( \frac{x-\Gamma_{q}^{2(m+1)}}{\frac{1}{2x_0}\left(\frac{1}{4}-1\right)\Gamma_q^{2(m+1)}}\right)
\end{equation}
satisfies all components of \eqref{item:cutoff:1}. Notice that this rescaling involves a factor proportional to $\Gamma_q^{-2(m+1)}$. Then using \eqref{eq:psi:property:4} and the fact that every derivative $\psi_{m,q}$ introduces another factor of $\Gamma_q^{-2(m+1)}$, we have that \eqref{eq:DN:psi:q:0} is satisfied. 

We now outline how to construct $\psi_{m,q} (\Gamma_{q}^{-2(m+1)}y )$, which is the first term in the series in \eqref{eq:tilde:partition} and will define $\psi_{m,q}(y)$.  The basic idea is that the region $ (\frac{1}{4}\Gamma_{q+1}^{2(m+1)},\Gamma_{q+1}^{2(m+1)} )$ where $\tilde\psi_{m,q}$ decreases from $1$ to $0$ will be the region where $\psi_{m,q} (\Gamma_{q+1}^{-2(m+1)}y )$ increases from $0$ to $1$, and furthermore in order to satisfy \eqref{eq:tilde:partition}, we have a formula for $\psi_{m,q} (\Gamma_{q+1}^{-2(m+1)}y )$ for these $y$-values. Specifically, in order to ensure \eqref{eq:tilde:partition} for $y\in (\frac{1}{4}\Gamma_{q+1}^{2(m+1)},\Gamma_{q+1}^{2(m+1)} )$, we define 
$$\psi_{m,q}^2 \left(\Gamma_{q+1}^{-2(m+1)}y \right)=1-\tilde\psi_{m,q}^2(y)$$
in this range of $y$-values.  Then by adjusting \eqref{eq:psi:property:5} to reflect the rescalings present in the definition of $\tilde\psi_{m,q}$ and $\psi_{m,q} (\Gamma_q^{-2(m+1)}y )$, we have that for $y\in\left(\frac{1}{4},1\right)$, $\psi_{m,q}$ is well-defined and \eqref{eq:DN:psi:q} holds. To define $\psi_{m,q} (\Gamma_q^{-2(m+1)}y )$ for $y\in [\frac{1}{4}\Gamma_q^{4(m+1)},\Gamma_q^{4(m+1)} ]$ and thus $\psi_{m,q}\left(y\right)$ for $y\in [\frac{1}{4}\Gamma_q^{2(m+1)},\Gamma_q^{2(m+1)} ]$, we can use that for $y\in [\frac{1}{4}\Gamma_q^{4(m+1)},\Gamma_q^{4(m+1)} ]$, the rescaled function $\psi_{m,q} (\Gamma_{q+1}^{-4(m+1)}y )$ (i.e. the term in \eqref{eq:tilde:partition} with $i=2$) is now well-defined. Then we can set
\begin{equation*}
    \psi_{m,q}^2\left(\Gamma_{q+1}^{-2(m+1)}y\right)=1-\psi^2_{m,q}\left(\Gamma_{q+1}^{-4(m+1)}y\right)
\end{equation*}
so that $\psi_{m,q}$ is well-defined for $y\in [\frac{1}{4}\Gamma_q^{2(m+1)},\Gamma_q^{2(m+1)} ]$ and \eqref{eq:tilde:partition} holds in this range of $y$-values. Appealing again to \eqref{eq:psi:property:4} and \eqref{eq:psi:property:5}, we have that \eqref{eq:DN:psi:q:gain} is satisfied in the claimed range of $y$-values.   Finally, in the missing interval $ [1,\frac{1}{4}\Gamma_q^{2(m+1)} ]$, we set $\psi_{m,q}\equiv 1$.  One can now check that \eqref{eq:tilde:partition} holds for all $y\geq 0$, and that \eqref{eq:psi:support:base:case} follows from \eqref{item:cutoff:1}, \eqref{item:cutoff:2}, and \eqref{eq:tilde:partition}, concluding the proof.

\subsection{$L^p$ decorrelation}
\label{sec:Lp:decorrelation}
The following lemma may be found in~\cite[Lemma 3.7]{BV19}.

\begin{lemma}[\textbf{$L^p$ de-correlation estimate}]
\label{lem:Lp:independence}
Fix integers $\Ndec  \geq 1$, $\mu \geq \lambda \geq 1$ and assume that these integers obey
\begin{align} 
 \lambda^{\Ndec +4}    \leq \left(\frac{\mu}{2\pi\sqrt{3}}  \right)^{\Ndec} 
\,.
\label{eq:Lp:independence:assumption}
\end{align}
Let $p \in \{1,2\}$, and let $f$ be a $\T^3$-periodic function such that 
\begin{align}
\max_{0 \leq N \leq \Ndec+4} \lambda^{-N} \|D^N f\|_{L^p} \leq \const_f
\label{eq:Lp:independence:assumption:2}
\end{align}
for a constant  $\const_f > 0$.\footnote{For instance, if $f$ has frequency support in the ball of radius $\lambda$ around the origin, we have that $\const_f \approx  \norm{f}_{L^p}$.}
 Then, for any $(\T/\mu)^{3}$-periodic function $g$, we have that 
\begin{align}\notag
 \|f g \|_{L^p} \lesssim \const_f \|g\|_{L^p} \,,
\end{align}
where the implicit constant is universal (in particular, independent of $\mu$ and $\lambda$).
\end{lemma}

\subsection{Sobolev inequality with cutoffs}
\label{sec:Sobolev:inequality:cutoffs}

\begin{lemma}
\label{lem:Sobolev:cutoffs}
Let $0 \leq \psi_i \leq 1$ be cutoff functions such that $\psi_{i\pm}  = (\psi_{i-1}^2 + \psi_i^2 + \psi_{i+1}^2)^{\sfrac 12} = 1$ on $\supp(\psi_i)$, and such that for some $\rho>0$ we have
\begin{align}
|D^K \psi_i(x)| \les  \psi_i^{1-K/\Nfin}(x) \rho^K
\label{eq:L12:to:Linfty:psi}
\end{align} 
for all $K \leq 4$. 
Fix parameters $p \in [1,\infty]$, $0<\lambda\leq \tilde \lambda$, $0<\mu_i \leq \tilde \mu_i$, $N_x,N_t \geq 0$, and assume that the sequences $\{\mu_i\}_{i\geq0}$ and $\{\tilde \mu_i\}_{i\geq 0}$ are nondecreasing. 
Assume that there exist $N_*,M_* \geq 0$ such that the function $f \colon \T^3 \to \R$ obeys the estimate 
\begin{align}
\norm{\psi_i D^N D_t^M f}_{L^p} \les \const_f \MM{N,N_x,\lambda,\tilde \lambda} \MM{M,N_t,\mu_i,\tilde \mu_i}
\label{eq:Lp:vomit:1}
\end{align}
for all $N \leq N_*$ and $M\leq M_*$. Then, we have that 
\begin{subequations}
\begin{align}
\norm{\psi_i^2 D^N D_t^M f}_{L^\infty} &\les \const_f (\max\{1,\rho,\tilde \lambda\})^{\sfrac 3p} \MM{N,N_x,\lambda,\tilde \lambda} \MM{M,N_t,\mu_i,\tilde \mu_i}
\label{eq:L2:to:Linfty:*}
\\
\norm{D^N D_t^M f}_{L^\infty(\supp \psi_i)} &\les \const_f  (\max\{1,\rho,\tilde \lambda\})^{\sfrac 3p}  \MM{N,N_x,\lambda,\tilde \lambda} \MM{M,N_t,\mu_{i+1},\tilde \mu_{i+1}}
\label{eq:L2:to:Linfty}
\end{align}
\end{subequations}
for all $N \leq N_*- \lfloor \sfrac{3}{p}\rfloor - 1$ and $M\leq M_*$. 

Lastly, if the inequality \eqref{eq:Lp:vomit:1} holds for all $N+M \leq N_\circ$ for some $N_\circ \geq 0$ (instead of $N\leq N_*$ and $M\leq M_*$), then the bounds \eqref{eq:L2:to:Linfty:*} and \eqref{eq:L2:to:Linfty} hold for $N+ M \leq N_\circ - \lfloor \sfrac{3}{p} \rfloor - 1$. 
\end{lemma}
\begin{proof}[Proof of Lemma~\ref{lem:Sobolev:cutoffs}]
The proof uses that $\lfloor \sfrac{3}{p}\rfloor +1 > \sfrac{3}{p}$ for all $p\in [1,\infty]$, and that $W^{s,p} \subset L^\infty$ for $s > \sfrac{3}{p}$.  Moreover, the proof of \eqref{eq:L2:to:Linfty:*} is nearly identical to that of \eqref{eq:L2:to:Linfty}, and thus we only give the proof of \eqref{eq:L2:to:Linfty}; moreover, for simplicity we only give the proof for $p=2$, as all the other Lebesgue exponents are treated in the same way. By Gagliardo-Nirenberg-Sobolev interpolation we have
\begin{align*}
 \norm{D^N D_t^M f}_{L^\infty(\supp \psi_i)} 
 &\leq \norm{\psi_{i\pm}^2 D^N D_t^M f}_{L^\infty (\T^3)} \notag\\
 & \les \norm{\psi_{i\pm}^2 D^N D_t^M f}_{L^2 (\T^3)}^{\sfrac 14} \norm{\psi_{i\pm}^2 D^N D_t^M f}_{\dot{H}^2 (\T^3)}^{\sfrac 34} +   \norm{\psi_{i\pm}^2 D^N D_t^M f}_{L^2 (\T^3)} \,.
\end{align*}
Using \eqref{eq:L12:to:Linfty:psi}, \eqref{eq:Lp:vomit:1}, and the monotonicity of the $\mu_i$ and $\tilde \mu_i$, we obtain
\begin{align*}
&\norm{\psi_{i\pm}^2 D^N D_t^M f}_{\dot{H}^2 (\T^3)}\notag\\
&\les 
\norm{\psi_{i\pm} D^{N+2} D_t^M f}_{L^2}  + \norm{D \psi_{i\pm}}_{L^\infty} 
\norm{\psi_{i\pm} D^{N+1} D_t^M f}_{L^2} +  \norm{\frac{D^2 (\psi_{i\pm}^2)}{\psi_{i\pm}}}_{L^\infty} \norm{\psi_{i\pm} D^{N} D_t^M f}_{L^2} 
\notag\\
&\les
\norm{\psi_{i\pm} D^{N+2} D_t^M f}_{L^2}  + \rho 
\norm{\psi_{i\pm} D^{N+1} D_t^M f}_{L^2} +  \rho^2 \norm{\psi_{i\pm} D^{N} D_t^M f}_{L^2}  
\notag\\
&\les (\max\{\tilde \lambda, \rho\})^2  \const_f \MM{N,N_x,\lambda,\tilde \lambda} \MM{M,N_t,\mu_{i+1},\tilde \mu_{i+1}}
\,,
\end{align*}
for all $N\leq N_*-2$ and $M\leq M_*$.
In the second inequality above we have used that $|D^2(\psi_{i\pm}^2)| \les \rho^2 \psi_{i\pm}(x)$, which follows from \eqref{eq:L12:to:Linfty:psi}. Combining the above two displays  proves \eqref{eq:L2:to:Linfty}. 

Note that for $p=1$ we require that $|D^4(\psi_{i\pm}^2)| \les \rho^4 \psi_{i\pm}(x)$, which also follows from \eqref{eq:L12:to:Linfty:psi} since $\Nfin \geq 4$, and this is why we have assumed this inequality to hold for all $K\leq 4$.

Lastly, assume that \eqref{eq:Lp:vomit:1} holds for all $N+M \leq N_\circ$, and fix any $N',M'\geq 0$ such that $N' + M' \leq N_\circ - \lfloor \sfrac{3}{p} \rfloor - 1$. Let $N_* = N' + \lfloor \sfrac{3}{p} \rfloor + 1$ and $M_* = M'$. Then \eqref{eq:Lp:vomit:1} gives a bound for $\|\psi_i D^{N''} D_t^{M''} f\|_{L^p}$ for all $N'' \leq N_*$ and $M'' \leq M_*$. The bounds \eqref{eq:L2:to:Linfty:*} and \eqref{eq:L2:to:Linfty} thus give an estimate for  $\|\psi_i D^{N'} D_t^{M'} f\|_{L^p}$, which concludes the proof.
\end{proof}

\subsection{Consequences of the Faa di Bruno formula}
\label{sec:Faa:di:Bruno}
We are using the following version of the multivariable Faa di Bruno formula \cite[Theorem 2.1]{ConstantineSavits96}. Let $g= g(x_1,\ldots,x_d) = f ( h(x_1,\ldots,x_d))$, where $f \colon \R^m \to \R$, and $h \colon \R^d \to \R^m$ are $C^n$ smooth functions of their respective variables. Let $\alpha \in {\mathbb N}_0^d$ be s.t. $|\alpha|=n$, and let $\beta \in \N_0^m$ be such that $1\leq |\beta|\leq n$. We then define  
\begin{align} 
p(\alpha,\beta) 
&= \Bigg\{ (k_1,\ldots,k_n; \ell_1,\ldots,\ell_n) \in (\N_0^m)^n \times (\N_0^d)^n \colon \exists s \mbox{ with }1\leq s \leq n \mbox{ s.t. } \notag \\
&\qquad   \qquad  |k_j|, |\ell_j| > 0 \Leftrightarrow 1 \leq j \leq s, \, 0 \prec \ell_{1} \prec \ldots \prec \ell_s,  \sum_{j=1}^s k_j =  \beta, \sum_{j=1}^s |k_j| \ell_j = \alpha \Bigg\}.
\label{eq:Faa:di:Bruno:1}
\end{align} 
Then the multivariable Faa di Bruno formula states that we have the equality
\begin{align}
\partial^\alpha g (x) = \alpha! \sum_{|\beta|=1}^n (\partial^\beta f)(h(x)) \sum_{p(\alpha,\beta)} \prod_{j=1}^n \frac{(\partial^{\ell_j} h(x))^{k_j}}{k_j! (\ell_j!)^{k_j}}.
\label{eq:Faa:di:Bruno:2}
\end{align}
Note that in \eqref{eq:Faa:di:Bruno:1} we have that $k_j = 0 \in \N_0^m$ and $\ell_j = 0 \in N_0^d$ for $j\geq s+1$. Therefore, we could write the sums and products with $j \in \{1,\ldots,s\}$ as sums for $j\in \{1,\ldots,n\}$. Keeping in mind this convention, we importantly note that in \eqref{eq:Faa:di:Bruno:2} we can have $|\ell_j| = 0$ only if $|k_j|= 0$, and in this case the entire term in the product is equal to $1$. That is, the product in \eqref{eq:Faa:di:Bruno:2} only goes from $1$ to $s$, and in this case $|\ell_j|\geq 1$ for $j\in \{1,\ldots,s\}$. This fact will be used frequently. 

For applications to cutoff functions we apply this formula for scalar functions $h$, i.e. $m=1$, while for applications to the perturbation or Reynolds stress sections we apply this formula for vector fields $h$, i.e. $m=3$.

Since throughout this manuscript the number of derivatives that we need to estimate is uniformly bounded (say by $\Nfin$), we may ignore the factorial terms in \eqref{eq:Faa:di:Bruno:2} and include them in the implicit constant of $\les$. Using this convention, we summarize in the following lemma a useful consequence of the Fa\'a di Bruno formula above.

\begin{lemma}[\textbf{Fa\'a di Bruno}]
\label{lem:Faa:di:Bruno}
Fix $N \leq \Nfin$.
Let $\psi \colon [0,\infty) \to [0,1]$ be a smooth function obeying 
\begin{align}
|D^B \psi| \les \Gamma_\psi^{-2B} \psi^{1 - B/\Nfin}
\label{eq:Faa:di:Bruno:lem:1}
\end{align}
for all $B \leq N$, and some $\Gamma_\psi>0$. Let $\Gamma, \lambda, \Lambda >0$ and $N_* \leq N$. Furthermore, let $h \colon \T^3\times \R \to \R$ and denote \[ g(x) = \psi(\Gamma^{-2} h(x)).\] Assume the function $h$ obeys
\begin{align}
\norm{D^B h}_{L^\infty(\supp g)} \les \const_h \MM{B,N_*,\lambda,\Lambda}
\label{eq:Faa:di:Bruno:lem:2}
\end{align}
for all $B \leq N$, where  the implicit constant is independent of $\lambda, \Lambda, \Gamma, \const_h>0$. 
Then, we have that for all points $(x,t) \in \supp h$, the bound
\begin{align}
\frac{|D^N g|}{g^{1 - N/\Nfin}} \les \MM{N,N_*,\lambda,\Lambda} \max\{(\Gamma_\psi  \Gamma)^{-2} \const_h , (\Gamma_\psi \Gamma)^{-2N} \const_h^N \}
\label{eq:Faa:di:Bruno:lem:3}
\end{align}
holds. If the $\psi^{1-B/\Nfin}$ factor on the right side of \eqref{eq:Faa:di:Bruno:lem:1} is replaced by $1$, then the $g^{1 - N/\Nfin}$ factor on the left side of \eqref{eq:Faa:di:Bruno:lem:3} also has to be replaced by $1$.
\end{lemma}

\begin{proof}[Proof of Lemma~\ref{lem:Faa:di:Bruno}]
The goal is to apply \eqref{eq:Faa:di:Bruno:1}--\eqref{eq:Faa:di:Bruno:2} with $f(x) = \psi(\Gamma^{-2} x)$. For $(x,t) \in \supp (g)$ we obtain from  \eqref{eq:min:max:exponents:prod}, \eqref{eq:Faa:di:Bruno:lem:1}, and \eqref{eq:Faa:di:Bruno:lem:3} that
\begin{align*}
\frac{|D^N g|}{g^{1-N/\Nfin}} 
&\les \sum_{B=1}^N \frac{|D^B \psi|}{\psi^{1-B/\Nfin}} \psi^{(N-B)/\Nfin}  \Gamma^{-2B} \sum_{p(\alpha,B)} \prod_{j=1}^n \norm{\partial^{\ell_j} h}_{L^\infty(\supp g)}^{k_j} \notag\\
&\les \sum_{B=1}^N (\Gamma_\psi \Gamma)^{-2B} \sum_{p(\alpha,B)} \prod_{j=1}^n \left(\const_h \MM{\ell_j,N_*,\lambda,\Lambda} \right)^{k_j} \notag\\
&\les   \sum_{B=1}^N (\Gamma_\psi \Gamma)^{-2B} \const_h^B \MM{N,N_*,\lambda,\Lambda}
\end{align*}
for any $1 \leq B \leq N$. The conclusion of the lemma follows upon bounding the geometric sum.
\end{proof}

Frequently in the paper, we need a version of Lemma~\ref{lem:Faa:di:Bruno} which also deals with mixed spatial and material derivatives. A convenient statement is:
\begin{lemma}[\textbf{Mixed derivative Fa\'a di Bruno}]
\label{lem:Faa:di:Bruno:*}
Fix $N, M\in \N$ such that $N+M \leq \Nfin$.
Let $\psi \colon [0,\infty) \to [0,1]$ be a smooth function obeying 
\begin{align}
|D^B \psi| \les \Gamma_\psi^{-2B} \psi^{1 - B/\Nfin}
\label{eq:Faa:di:Bruno:lem:1:*}
\end{align}
for all $B \leq N$ and a constant $\Gamma_\psi > 0$. 
Let $v$ be a fixed vector field, and denote $D_t = \partial_t + v\cdot \nabla$, which is a scalar differential operator.
Let $\Gamma, \lambda, \tilde \lambda, \mu, \tilde \mu \geq 1$ and $N_x,N_t\leq N$. Furthermore, let $h \colon \T^3\times \R \to \R$ and denote 
\[
g(x,t) = \psi(\Gamma^{-2} h(x,t)).
\] 
Assume the function $h$ obeys
\begin{align}
\norm{D^{N'} D_t^{M'} h}_{L^\infty(\supp g)} \les \const_h \MM{N',N_x,\lambda,\tilde \lambda} \MM{M',N_t,\mu,\tilde \mu}
\label{eq:Faa:di:Bruno:lem:2:*}
\end{align}
for all $N'\leq N$, and $M' \leq M$,  
where  the implicit constant is independent of $\lambda, \tilde \lambda, \mu, \tilde \mu, \Gamma$, and $\const_h$. 
Then, we have that for all points $(x,t) \in \supp h$, the bound
\begin{align}
\frac{|D^N D_t^M g|}{g^{1 - (N+M)/\Nfin}} \les \MM{N,N_x,\lambda,\tilde \lambda} \MM{M,N_t,\mu,\tilde \mu}  \max\left\{(\Gamma_\psi \Gamma)^{-2} \const_h , ( (\Gamma_\psi  \Gamma)^{-2} \const_h)^{N+M} \right\}
\label{eq:Faa:di:Bruno:lem:3:*}
\end{align}
holds. If the $\psi^{1-B/\Nfin}$ factor on the right side of \eqref{eq:Faa:di:Bruno:lem:1:*} is replaced by $1$, then the $g^{1 - (N+M)/\Nfin}$ factor on the left side of \eqref{eq:Faa:di:Bruno:lem:3:*} also has to be replaced by $1$.
\end{lemma}

\begin{proof}[Proof of Lemma~\ref{lem:Faa:di:Bruno:*}]
Let $X(a,t)$ be the flow induced by the vector field $v$, with initial condition $X(a,t) = x$. Denote by $a = X^{-1}(x,t)$ the inverse of the map $X$. We then note that 
\begin{align*}
D_t^M g(x,t) = \left(\partial_t^M ( (g \circ X)(a,t)) \right)\vert_{a = X^{-1}(x,t)}.
\end{align*}
We wish to apply the above with the function $g(x,t) = \psi(\Gamma^{-2}h(x,t))$. 
We apply the Faa di Bruno formula \eqref{eq:Faa:di:Bruno:1}--\eqref{eq:Faa:di:Bruno:2} with the one dimensional differential operator $\partial_t^M$ to the composition $g\circ X$, note that $\partial_t^{\beta_i} (h(X(a,t),t)) = (D_t^{\beta_i} h)(X(a,t),t)$, and then evaluate the resulting expression at $a = X^{-1}(x,t)$, to obtain
\begin{align*}
D_t^M g(x,t) = M! \sum_{B=1}^{M} \Gamma^{-2B} \psi^{(B)}(\Gamma^{-2} h(x,t)) \sum_{\substack{\{\kappa , \beta \in \N^M \colon \\ |\kappa|= B, \kappa \cdot \beta = M \} }} \prod_{i=1}^{M} \frac{ \left( (D_t^{\beta_i} h)(x,t) \right)^{\kappa_i}}{\kappa_i! (\beta_i !)^{\kappa_i}}.
\end{align*}
We now apply $D^N$ to the above expression, use the Leibniz rule, and then appeal again to the Faa di Bruno formula \eqref{eq:Faa:di:Bruno:1}--\eqref{eq:Faa:di:Bruno:2}, this time for spatial derivatives.
We obtain
\begin{align}
D^N D_t^M g(x,t) 
&= M! N! \sum_{B=1}^{M}  \sum_{K=0}^{N} \sum_{B'=0}^{K} \Gamma^{-2(B+B')} \psi^{(B+B')}(\Gamma^{-2} h(x,t)) \sum_{p(K,B')} \prod_{j=1}^{K} \frac{(D^{\ell_j} h(x,t))^{k_j}}{k_j! (\ell_j!)^{k_j}} \notag\\
&\quad \times \sum_{\substack{\{ \alpha \in \N^M \colon \\|\alpha|=N-K\}}} \sum_{\substack{\{\kappa , \beta \in \N^M \colon \\ |\kappa|= B, \kappa \cdot \beta = M \} }} \prod_{i=1}^{M} \frac{D^{\alpha_i}  (  ( (D_t^{\beta_i} h)(x,t)  )^{\kappa_i} )}{\alpha_i! \kappa_i! (\beta_i !)^{\kappa_i}}.
\label{eq:Faa:di:Bruno:lem:4:*}
\end{align}
Upon dividing by $g^{1-(N+M)/\Nfin}$ and noting that $B+B' \leq M+N$, from \eqref{eq:Faa:di:Bruno:lem:1:*}, identity \eqref{eq:Faa:di:Bruno:lem:4:*}, the Leibniz rule, and assumption \eqref{eq:Faa:di:Bruno:lem:2:*}, we obtain
\begin{align*}
\frac{|D^N D_t^M g|}{g^{1-(N+M)/\Nfin}} 
&\les \sum_{B=1}^M \sum_{K=0}^{N}\sum_{B'=0}^{K}(\Gamma_\psi \Gamma)^{-2(B+B')} \const_h^{B'} \MM{K,N_x,\lambda,\tilde \lambda} \const_h^B \MM{N-K,N_x,\lambda,\tilde \lambda} \MM{M,N_t,\mu,\tilde \mu}
\notag\\
&\les \MM{N,N_x,\lambda,\tilde \lambda}  \MM{M,N_t,\mu,\tilde \mu} \sum_{B=1}^M  \sum_{B'=0}^{N}(\Gamma_\psi \Gamma)^{-2(B+B')} \const_h^{B'+B}
\end{align*}
from which \eqref{eq:Faa:di:Bruno:lem:3:*} follows by summing the geometric series.
\end{proof}

\begin{lemma}\label{l:useful_ests}
Given a smooth function $f \colon \R^3 \times \R \to \R$, suppose that for $\lambda \geq 1$ the vector field $\Phi \colon \R^3 \times \R \to \R^3$ satisfies the estimate
\begin{align}
\norm{D^{N+1} \Phi}_{L^{\infty}(\supp f)}&\les \lambda^{N}
\label{e:Faa_di_Bruno_0}
\end{align}
for $0 \leq N \leq N_*$.  Then for any $1 \leq N\leq N_*$ we have
\begin{align}
\abs{D^N\left( f \circ \Phi \right)(x,t)} \lesssim& \sum_{m=1}^N \lambda^{N-m}\abs{(D^m f)\circ \Phi (x,t)}
\label{e:Faa_di_Bruno_1} 
\end{align}
and thus trivially we obtain
\begin{align*}
\abs{D^N\left( f \circ \Phi \right)(x,t)} \lesssim& \sum_{m=0}^N \lambda^{N-m}\abs{(D^m f)\circ \Phi (x,t)}.
\end{align*}
for any $0 \leq N\leq N_*$.
\end{lemma}
\begin{proof}[Proof of Lemma~\ref{l:useful_ests}]
Applying \eqref{eq:Faa:di:Bruno:2}, noting that $|\ell_j|= 0$ implies $|k_j|=0$, and assumption \eqref{e:Faa_di_Bruno_0},  we have that for any multi index $\alpha \in \N_0^3$ with $|\alpha| = N$, 
\begin{align*}
\abs{\partial^{\alpha}\left( f \circ \Phi \right)(x,t)}&\les \sum_{|\beta|=1}^N \abs{((\partial^\beta f)\circ \Phi )(x,t)} \prod_{j=1}^N  \sum_{p(\alpha,\beta)} \abs{\left( \partial^{\ell_j} \Phi(x,t)\right)^{k_j}}\\
&\les \sum_{|\beta|=1}^N \abs{(\partial^\beta f)\circ \Phi} \prod_{j=1}^N  \sum_{p(\alpha,\beta)}\lambda^{({\abs{\ell_j}-1})|k_j|}\\
&\les \sum_{m=1}^N \lambda^{N-m}\abs{(D^m f)\circ \Phi}
\end{align*}
by the definition \eqref{eq:Faa:di:Bruno:1}. Thus we obtain \eqref{e:Faa_di_Bruno_1}.  
\end{proof}

In order to estimate the perturbation in $L^p$ spaces as well as terms appearing in the Reynolds stress we will need the following abstract lemma, which follows from Lemmas~\ref{lem:Lp:independence}  and~\ref{l:useful_ests}.
\begin{lemma}
\label{l:slow_fast}
Let $p\in\{1,2\}$, and fix integers $N_* \geq M_* \geq \Ndec\geq 1$. 
Suppose $f \colon \R^3 \times \R \to \R$ and let $\Phi \colon \R^3\times \R \to \R^3$ be a vector field advected by an incompressible velocity field $v$, i.e. $D_{t}\Phi = (\partial_t + v\cdot \nabla) \Phi=0$. Denote by $\Phi^{-1}$ the inverse of the flow $\Phi$,  which is the identity at a time slice which intersects the support of $f$. Assume that for some $\lambda , \nu, \tilde \nu \geq 1$ and $\const_f>0$  the functions $f$ satisfies the estimates
\begin{align}
\norm{D^N D_{t}^M f}_{L^{p}}&\les \const_f \lambda^N\MM{M,N_{t},\nu,\tilde\nu} \label{eq:slow_fast_0}
\end{align}
for all $N\leq N_*$ and $M\leq M_*$, and that $\Phi$, and $\Phi^{-1}$ are bounded as
\begin{align}
\norm{D^{N+1} \Phi}_{L^{\infty}(\supp f)}&\les \lambda^{N}\label{eq:slow_fast_1} \\
\norm{D^{N+1} \Phi^{-1}}_{L^{\infty}(\supp f)} &\les  \lambda^{N}\label{eq:slow_fast_2}
\end{align}
for all $N\leq N_*$.
Lastly, suppose that $\varphi$ is $(\T/ \mu)^3$-periodic, and that there exist parameters  $\tilde\zeta \geq \zeta \geq \mu$  and $\const_\varphi>0$ such that
\begin{align}
\norm{D^N \varphi}_{L^p} \les \const_\varphi \MM{N, N_x, \zeta, \tilde\zeta} \, \label{eq:slow_fast_4}
\end{align}
for all $0\leq N \leq N_*$.  If the parameters $$\lambda\leq \mu \leq \zeta \leq \tilde\zeta$$ satisfy
\begin{align}
 \tilde\zeta^4  \leq  \left(\frac{\mu}{2 \pi \sqrt{3} \lambda}\right)^{\Ndec}    \label{eq:slow_fast_3}
\,,
\end{align}
and we have
\begin{equation}\label{eq:slow_fast_3_a}
2 \Ndec + 4 \leq N_* \,,
\end{equation}
then the bound
\begin{align}
\norm{D^N  D_{t }^M \left( f\; \varphi\circ \Phi
\right)}_{L^p}
&\les \const_f \const_\varphi \MM{N, N_x, \zeta, \tilde\zeta}  \MM{M,M_{t},\nu,\tilde\nu}
\label{eq:slow_fast_5} 
\end{align}
holds for $N\leq N_*$ and $M\leq M_*$.  
\end{lemma} 

\begin{remark}
We emphasize that \eqref{eq:slow_fast_5} holds for   the same range of $N$ and $M$ that \eqref{eq:slow_fast_0} holds, as soon as $N_*$ is sufficiently large compared to $\Ndec$ so that \eqref{eq:slow_fast_3_a} holds. 
\end{remark}

\begin{remark}
\label{rem:slow:fast}
 We note that if estimate \eqref{eq:slow_fast_0} is known to hold for $N+M \leq N_\circ$ for some $N_\circ \geq 2\Ndec + 4$ (instead of, for $N\leq N_*$ and $M\leq M_*$), and if the bounds \eqref{eq:slow_fast_1}--\eqref{eq:slow_fast_2} hold for all $N\leq N_\circ$, then it follows from the below proof that the bound \eqref{eq:slow_fast_5} holds for $N+ M \leq N_\circ$ and $M \leq N_\circ - 2 \Ndec - 4$. The only modification required to the proof (given below) is that instead of considering the cases $N' \leq N_* - \Ndec - 4$ and $N' > N_* - \Ndec - 4$, we now have to split according to $N' + M \leq N_\circ - \Ndec - 4$ and $N' + M > N_\circ  - \Ndec - 4$. In the second case we use that $N-N'' \geq N_\circ - M - \Ndec - 4 \geq \Ndec$, which holds exactly because $M \leq   N_\circ - 2 \Ndec - 4$.
\end{remark}

\begin{proof}[Proof of Lemma~\ref{l:slow_fast}]
Since $D_{t}\Phi = 0$ we have $D_{t}^M (\varphi\circ \Phi) = 0$. 
Using that $\div v\equiv 0$, so that $\Phi$ and $\Phi^{-1}$ preserve  volume, and Lemma~\ref{l:useful_ests}, which we may apply due to \eqref{eq:slow_fast_1}, we have
\begin{align}
\norm{D^ND_{t}^M\left( f \; \varphi\circ \Phi \right)}_{L^p} 
& \lesssim \sum_{N'=0}^N \norm{D^{N'} D_{t}^M f \; D^{N-N'}\left(  \varphi \circ \Phi \right)}_{L^p}
\notag \\
&
\lesssim \sum_{N'=0}^N \sum_{N''=0}^{N-N'}\lambda^{N-N'-N''}\norm{D^{N'} D_{t}^Mf \;(D^{N''}\varphi)\circ \Phi}_{L^p}
\notag \\
&
\lesssim \sum_{N'=0}^N \sum_{N''=0}^{N-N'}\lambda^{N-N'-N''} \norm{\left(D^{N'} D_{t}^Mf\right) \circ \Phi^{-1} D^{N''}\varphi}_{L^p}.
\label{eq:slow_fast_temp_1}
\end{align}

In \eqref{eq:slow_fast_temp_1} let us first consider the case   $N' \leq N_* - \Ndec - 4$, so that $N'+M \leq N_* + M_* - \Ndec - 4$.  Under assumption \eqref{eq:slow_fast_2} we may apply Lemma \ref{l:useful_ests}, and using \eqref{eq:slow_fast_0} we have
\begin{align}
\norm{D^{n}\left( (D^{N'} D_{t}^Mf) \circ (\Phi^{-1},t)\right)}_{L^{p}}&\les \sum_{n'=0}^{n}\lambda^{n-n'}\norm{(D^{n'+N'} D_{t}^M f)\circ \Phi^{-1}  }_{L^{p}} \notag \\
&\les  \const_f \sum_{n'=0}^{n}\lambda^{n-n'}\lambda^{n'+N'}\MM{M,N_{t},\nu,\tilde\nu} \notag \\
&\les  \left(\const_f \lambda^{N'}\MM{M,N_{t},\nu,\tilde\nu}\right) \lambda^n \,,
\label{eq:slow_fast_temp_2}
\end{align}
for all $n\leq \Ndec+4$. This bound matches \eqref{eq:Lp:independence:assumption:2}, with $\const_f$ replaced by $\const_f \lambda^{N'} \MM{M,N_t,\nu,\tilde \nu}$. 
Since like $\varphi$, the function $D^{N''}\varphi$ is $(\T/\mu)^3$-periodic, due to \eqref{eq:slow_fast_temp_2}, the fact that $\lambda \leq \tilde \zeta$, and assumption \eqref{eq:slow_fast_3}, we may apply Lemma~\ref{lem:Lp:independence} to conclude
\begin{align*}
\norm{\left(D^{N'} D_{t}^M f\right) \circ \Phi^{-1} D^{N''}\varphi}_{L^p}\les \const_f \lambda^{N'} \MM{M,N_t,\nu,\tilde \nu} \norm{D^{N''}\varphi}_{L^p}.
\end{align*}
Inserting this bound back into \eqref{eq:slow_fast_temp_1} and using \eqref{eq:slow_fast_4} concludes the proof of \eqref{eq:slow_fast_5} for the values of $N'$ considered in this case.
 
Next, let us consider the case  $N' > N_* - \Ndec -4$. Since $0\leq N' \leq N$, in particular this means that $N>N_* - \Ndec -4$, and  since $N'' \leq N - N'$ we also obtain that $N-N'' \geq N' > N_* - \Ndec -4 \geq \Ndec$. Here we have used \eqref{eq:slow_fast_3_a}. Then the H\"older inequality, the fact that $\Phi^{-1}$ is volume preserving, the Sobolev embedding $W^{4,p} \subset L^\infty$, the ordering $\tilde\zeta\geq\zeta\geq \mu \geq 1$ and assumption \eqref{eq:slow_fast_3}, imply that
\begin{align*}
\lambda^{N - N'-N''} \norm{\left(D^{N'} D_{t}^Mf\right) \circ \Phi^{-1} D^{N''}\varphi}_{L^p} 
&\les \lambda^{N - N'-N''} \norm{D^{N'} D_{t}^M f }_{L^p}  \norm{D^{N''}\varphi}_{L^\infty} 
\notag\\
&\les \lambda^{N - N'-N''}  \const_f \lambda^{N'}  \MM{M,N_{t},\nu,\tilde\nu} \const_\varphi  \MM{N'' + 4 ,N_x,\zeta,\tilde\zeta} 
\notag\\
&\les \const_f \const_\varphi \MM{N,N_x,\zeta,\tilde\zeta} \MM{M,N_{t},\nu,\tilde\nu} \tilde\zeta^4  \left(\frac{\lambda}{\zeta}\right)^{N-N''}  
\notag\\ 
&\les \const_f \const_\varphi \MM{N,N_x,\zeta,\tilde\zeta} \MM{M,N_{t},\nu,\tilde\nu} \tilde\zeta^4  \left(\frac{\lambda}{\mu}\right)^{\Ndec}   
\notag\\ 
&\les \const_f \const_\varphi \MM{N,N_x,\zeta,\tilde\zeta} \MM{M,N_{t},\nu,\tilde\nu} \,.
\end{align*}
Combining the above estimate with \eqref{eq:slow_fast_temp_1}, we deduce that also for  $N' > N_* - \Ndec -4$, the bound \eqref{eq:slow_fast_5} holds, concluding the proof of the lemma.
\end{proof}

\subsection{Bounds for sums and iterates of operators}
\label{sec:operator:iterates}
For two differential operators $A$ and $B$ we have the expansion
\begin{align}
(A+B)^m = \sum_{k=1}^{m} \sum_{\substack{\alpha,\beta \in {\N}^k \\ |\alpha| + |\beta| = m}} \left( \prod_{i=1}^{k} A^{\alpha_i} B^{\beta_i} \right).
\label{eq:cooper:1} 
\end{align}
Clearly \eqref{eq:cooper:1} simplifies if $[A,B]=0$.
A lot of times we need to apply the above formula with 
\[
A = v \cdot \nabla,
\] 
for some vector field $v$.  The question we would like to address in this section is the following:  {\em Assume that we have already established estimates on $(\prod_i D^{\alpha_i} B^{\beta_i}) v$, for $|\alpha|+|\beta|\leq m$. Can we deduce  estimates for the operator $(A+B)^m = (v\cdot \nabla + B)^m$?} The answer is ``yes'', and is summarized in the following lemma: 
\begin{lemma}
\label{lem:cooper:1}
Fix $N_x,N_t,N_*  \in \N$, $\Omega \in \T^3 \times \R$ a space-time domain, and let $v$ be a vector field. For  $k\geq 1$ and $\alpha,\beta \in \N^k$ such that $|\alpha|+ |\beta| \leq N_*$, we assume that we have the bounds
\begin{align}
\norm{ \left(\prod_{i=1}^k D^{\alpha_i} B^{\beta_i} \right) v}_{L^\infty(\Omega)} \les \const_v \MM{|\alpha|,N_x,\lambda_v,\tilde \lambda_v} \MM{|\beta|,N_t,\mu_v,\tilde \mu_v}
\label{eq:cooper:v}
\end{align}
for some $\const_v\geq 0$, $1\leq \lambda_v \leq \tilde \lambda_v$, and $1\leq \mu_v \leq \tilde \mu_v$.
With the same notation and restrictions on $|\alpha|,|\beta|$, let $f$ be a function which for some  $p \in [1,\infty]$  obeys 
\begin{align}
\norm{ \left(\prod_{i=1}^k D^{\alpha_i} B^{\beta_i} \right) f}_{ L^p (\Omega)} \les \const_f \MM{|\alpha|,N_x,\lambda_f,\tilde \lambda_f}  \MM{|\beta|,N_t,\mu_f,\tilde \mu_f}
\label{eq:cooper:f}
\end{align}
for some $\const_f\geq 0$, $1\leq \lambda_f \leq \tilde \lambda_f$, and $1\leq \mu_f \leq \tilde \mu_f$. 
Denote
\begin{align*}
\lambda = \max\{ \lambda_f,\lambda_v\}, \quad \tilde \lambda= \max\{\tilde \lambda_f,\tilde \lambda_v\}, \quad \mu = \max\{\mu_f,\mu_v\}, \quad \tilde \mu = \max\{\tilde \mu_f,\tilde \mu_v\}.
\end{align*}
Then, for \[ A = v\cdot \nabla \] we have the bounds 
\begin{align}
\norm{D^n \left( \prod_{i=1}^{k} A^{\alpha_i} B^{\beta_i} \right) f}_{ L^p (\Omega)} 
&\les \const_f \const_v^{|\alpha|} \MM{n+|\alpha|,N_x,\lambda,\tilde \lambda}  \MM{|\beta|,N_t,\mu,\tilde \mu} \label{eq:cooper:f:**} \\
&\les \const_f  \MM{n,N_x,\lambda,\tilde \lambda} (\const_v \tilde \lambda)^{|\alpha|} \MM{|\beta|,N_t,\mu,\tilde \mu} 
\notag\\ 
&\les \const_f  \MM{n,N_x,\lambda,\tilde \lambda}  \MM{|\alpha|+|\beta|,N_t, \max\{\mu, \const_v \tilde \lambda\}, \max\{\tilde \mu,\const_v \tilde \lambda\} } 
\label{eq:cooper:f:***}
\end{align}
as long as $n+|\alpha|+|\beta|\leq N_*$.
As a consequence,  if $k=m$ then \eqref{eq:cooper:1} and \eqref{eq:cooper:f:***} imply the bound
\begin{align}
\norm{D^n (A + B)^m f}_{ L^p (\Omega)} 
\les \const_f  \MM{n,N_x,\lambda,\tilde \lambda}  \MM{m,N_t, \max\{\mu, \const_v \tilde \lambda\}, \max\{\tilde \mu,\const_v \tilde \lambda\} }
\label{eq:cooper:f:*}
\end{align}
for $n+m \leq N_*$.  
\end{lemma}
\begin{remark}
\label{rem:cooper:1}
The previous lemma is applied typically with $v= u_q$  and $B = D_{t,q-1}$ in order to obtain estimates for $D^n (\prod_i D_q^{\alpha_i} D_{t,q-1}^{\beta_i}) f$, and hence for $D^n D_q^m f$. A more non-standard application of this lemma uses $v = - v_{q-1}$ and $B=  D_{t,q-1}$ in order to obtain estimates for time derivatives via $D^n \partial_t^m f =D^n( -v_{q-1} \cdot \nabla +  D_{t,q-1})^m f$.
\end{remark}
\begin{proof}[Proof of Lemma~\ref{lem:cooper:1}]
We recall \eqref{eq:D:q:K:i}--\eqref{eq:D:q:K:ii} and note that we may write (ignoring the way in which tensors are contracted)
\begin{align}
A^n = (v \cdot \nabla)^n = \sum_{j=1}^{n} f_{j,n} D^j \quad \mbox{where} \quad f_{j,n} = \sum_{\substack{\zeta \in \N^n \\ |\zeta|=n-j}} c_{n,j,\zeta} \prod_{\ell=1}^{n} (D^{\zeta_\ell} v)
\label{eq:cooper:2}
\end{align}
where the $c_{n,j,\zeta} $ are certain combinatorial coefficients (tensors) with the dependence given in the subindex, and $D^a$ represents $\partial^\alpha$ for some multi-index $\alpha$ with $|\alpha|=a$. Inserting \eqref{eq:cooper:2} into the product of operators in \eqref{eq:cooper:1}, we see that 
\begin{align}
D^n \prod_{i=1}^{k} A^{\alpha_i} B^{\beta_i} 
&= \sum_{\substack{\gamma\in \N^k \\1^k \leq \gamma \leq \alpha}}
D^n \prod_{i=1}^{k} (f_{\gamma_i,\alpha_i} D^{\gamma_i} B^{\beta_i} )
\notag\\
&= \sum_{\substack{\gamma\in \N^k \\1^k \leq \gamma \leq \alpha}} \sum_{\substack{0\leq n'\leq n + |\gamma|  \\ 0 \leq m' \leq |\beta|}} 
\sum_{\substack{\delta,\kappa \in \N^k \\ |\delta|=n+|\gamma|-n' \\ |\kappa| = |\beta|-m'}}
\left( \prod_{i=1}^{k} \left( \sum_{\substack{\delta_i', \kappa_i' \in \N^k \\ |\delta_i'| = \delta_i \\ |\kappa_i'| = \kappa_i}} \tilde c_{(\ldots)}
\left(\prod_{\ell_i=1}^{k} D^{\delta_{i,\ell_i}'} B^{\kappa_{i,\ell_i}'} \right) f_{\gamma_i,\alpha_i} \right) \right)\notag\\
&\qquad \qquad \times
\left(\sum_{\substack{\eta, \rho \in \N^k \\ |\eta| = n' \\ |\rho| = m'}}  \bar c_{(\ldots)} \prod_{s=1}^k D^{\eta_s } B^{\rho_s}  \right)
\label{eq:cooper:3}
\end{align}
where the $\tilde c_{(\dots)} ,\bar c_{(\dots)}\geq 0$ are certain combinatorial coefficients (tensors). Combining \eqref{eq:cooper:1}--\eqref{eq:cooper:3}, we obtain that 
\begin{align}
D^n \left(\prod_{i=1}^{k} A^{\alpha_i} B^{\beta_i}  \right) f 
&=
\sum_{\substack{\gamma\in \N^k \\1^k \leq \gamma \leq \alpha}} \sum_{\substack{0\leq n'\leq n + |\gamma|  \\ 0 \leq m' \leq |\beta|}} 
\sum_{\substack{\delta,\kappa \in \N^k \\ |\delta|= n + |\gamma|-n' \\ |\kappa| = |\beta|-m'}}
  \left(\sum_{\substack{\eta, \rho \in \N^k \\ |\eta| = n' \\ |\rho| = m'}}  \bar c_{(\dots)} \left(\prod_{s=1}^k D^{\eta_s } B^{\rho_s}  \right) f\right)
\notag\\
&\ \times
\left( \prod_{i=1}^{k} \left( \sum_{\substack{\delta_i', \kappa_i' \in \N^k \\ |\delta_i'| = \delta_i \\ |\kappa_i'| = \kappa_i}} \tilde c_{(\dots)} \left(
\prod_{\ell_i=1}^{k} D^{\delta_{i,\ell_i}'} B^{\kappa_{i,\ell_i}'} \right) \left(\sum_{\substack{\zeta_i \in \N^{\alpha_i} \\ |\zeta_i|=\alpha_i-\gamma_i}}  c_{(\dots)} \prod_{r_i=1}^{\alpha_i} (D^{\zeta_{i,r_i}} v)  \right) \right)   \right)
\label{eq:cooper:4}
\end{align}
where the $c_{(\dots)}, \tilde c_{(\dots)}, \bar c_{(\dots)} \geq 0$ are certain combinatorial coefficients (tensors) whose dependence is omitted for simplicity (it may depends on all the parameters in the sums and products). The above expansion combined with the Leibniz rule, the bound \eqref{eq:min:max:exponents:prod}, and assumptions \eqref{eq:cooper:v}--\eqref{eq:cooper:f}, implies
\begin{align*}
\norm{D^n \left(\prod_{i=1}^{k} A^{\alpha_i} B^{\beta_i}  \right) f }_{ L^p (\Omega)}
&\les
\sum_{\substack{\gamma\in \N^k \\1^k \leq \gamma \leq \alpha}} \sum_{\substack{0\leq n'\leq n+|\gamma|  \\ 0 \leq m' \leq |\beta|}} 
\sum_{\substack{\delta,\kappa \in \N^k \\ |\delta|=n+|\gamma|-n' \\ |\kappa| = |\beta|-m'}}
\left(\sum_{\substack{\eta, \rho \in \N^k \\ |\eta| = n' \\ |\rho| = m'}} \norm{\left( \prod_{s=1}^k D^{\eta_s } B^{\rho_s} \right) f}_{L^p(\Omega)} \right)
\notag\\
&\quad  \times
\left( \prod_{i=1}^{k}  \left(\sum_{\substack{\zeta_i \in \N^{\alpha_i} \\ |\zeta_i|=\alpha_i-\gamma_i}}  \sum_{\substack{\delta_i', \kappa_i' \in \N^k \\ |\delta_i'| = \delta_i \\ |\kappa_i'| = \kappa_i}}  
 \norm{\left( \prod_{\ell_i=1}^{k} D^{\delta_{i,\ell_i}'} B^{\kappa_{i,\ell_i}'} \right) \left(  \prod_{r_i=1}^{\alpha_i} (D^{\zeta_{i,r_i}} v)  \right) }_{L^\infty(\Omega)}   \right) \right)
\notag\\
&\les
\sum_{\substack{\gamma\in \N^k \\1^k \leq \gamma \leq \alpha}} \sum_{\substack{0\leq n'\leq n+ |\gamma|  \\ 0 \leq m' \leq |\beta|}} 
\sum_{\substack{\delta,\kappa \in \N^k \\ |\delta|=n+|\gamma|-n' \\ |\kappa| = |\beta|-m'}} \left(\const_f \MM{n',N_x,\lambda,\tilde \lambda} \MM{m',N_t,\mu,\tilde \mu}\right) \notag\\
&\quad \times \left( \prod_{i=1}^{k}    \const_v^{\alpha_i} \MM{\alpha_i-\gamma_i+\delta_{i},N_x,\lambda,\tilde \lambda} \MM{\kappa_i,N_t,\mu,\tilde \mu}  \right)
\notag\\
&\les
\const_f  
\sum_{\substack{0\leq n'\leq n+ |\alpha|  \\ 0 \leq m' \leq |\beta|}} \left(\const_f \MM{n',N_x,\lambda,\tilde \lambda} \MM{m',N_t,\mu,\tilde \mu}\right) \notag\\
&\quad \times \left(\const_v^{|\alpha|} \MM{|\alpha|+n-n',N_x,\lambda,\tilde \lambda} \MM{|\beta|-m',N_t,\mu,\tilde \mu}  \right)
\notag\\
&\les 
\const_f \const_v^{|\alpha|} \MM{|\alpha|+n,N_x,\lambda,\tilde \lambda} \MM{|\beta|,N_t,\mu,\tilde \mu}   
\end{align*}
which is precisely the bound claimed in \eqref{eq:cooper:f:**}. Estimate \eqref{eq:cooper:f:***} follows immediately, while the bound \eqref{eq:cooper:f:*} is a consequence of the above and \eqref{eq:cooper:1}.
\end{proof}

\subsection{Commutators with material derivatives}
\label{sec:operator:commutators}
Let $D$ represent a pure spatial derivative and let 
\[
D_t = \partial_t + v\cdot \nabla
\] 
denote a material derivative along the smooth (incompressible) vector field $v$. This vector field $v$ is fixed throughout this section. The question we would like to address in this section is the following:  {\em Assume that for the vector field $v$ we have  $D^aD_t^b Dv$ estimates available. Can we then bound the operator norm of $D_t^b D^a$ in terms of the operator norm of $D^a D_t^b$?} 

Following Komatsu~\cite[Lemma 5.2]{Komatsu79}, a useful ingredient for bounding commutators of Eulerian and material derivatives is the following lemma. We use the following commutator notation:
\begin{align*} 
(\ad D_t)^0(D) &= D \\ 
(\ad D_t)^1(D) &= [D_t,D] = - Dv \cdot \nabla \\
(\ad D_t)^a (D) = (\ad D_t) ((\ad D_t)^{a-1} (D)) &= [ D_t , (\ad D_t)^{a-1} (D)]
\end{align*}
for all $a \geq 2$. Note that for any  $a \geq 0$,  $(\ad D_t)^a (D)$ is a differential operator of order $1$.

\begin{lemma}
\label{lem:Komatsu}
Let $m, n \geq 0$. Then we have that the commutator of $D_t^m$ and $D^n$ is given by
\begin{align}
 \left[ D_t^m, D^n \right] = \sum_{\{ \alpha \in \N^n \colon 1 \leq |\alpha|\leq m \} } \frac{m!}{\alpha! (m-|\alpha|)!} \left( \prod_{\ell = 1}^{n} (\ad D_t)^{\alpha_\ell}(D) \right) D_t^{m-|\alpha|}.
 \label{eq:Komatsu}
\end{align}
By the product in \eqref{eq:Komatsu} we mean the product/composition of operators 
\[
\prod_{\ell = 1}^{n} (\ad D_t)^{\alpha_\ell}(D) = (\ad D_t)^{\alpha_n}(D) (\ad D_t)^{\alpha_{n-1}}(D) \ldots (\ad D_t)^{\alpha_1}(D)
\,,
\]
so that on the right side of \eqref{eq:Komatsu} we have a sum of differential operators of order at most $n$.
\end{lemma}

For the above lemma to be useful, we need to be able to characterize the operator $(\ad D_t)^a (D)$. 

\begin{lemma}
\label{lem:ad:Dt:a:D}
Let $a \in \N$. Then the order $1$ differential operator $(\ad D_t)^a (D)$ may be expressed as
\begin{align}
(\ad D_t)^a (D) =  \sum_{k=1}^{a}  \sum_{\{ \beta \in \N^k \colon |\beta| = a-k \} } c_{a,k,\beta}   \prod_{j=1}^{k} (D_t^{\beta_j} D v) \cdot \nabla
\label{eq:ad:Dt:a:D}
\end{align}
where the $\prod$ in \eqref{eq:ad:Dt:a:D} denotes the product of matrices, $c_{a,k,\beta}$ are coefficients which depend only on $a,k$, $\beta$.
\end{lemma}

\begin{proof}[Proof of Lemma~\ref{lem:ad:Dt:a:D}]
When $a=1$ we know that $(\ad D_t)(D) =  - Dv \cdot \nabla$, so that the lemma trivially holds. We proceed by induction on $a$. 
Using  the fact that $[D_t,\nabla] = - D v \cdot \nabla$, we obtain
\begin{align*}
(\ad D_t)^{a+1} (D) 
&=   D_t \left( \sum_{k=1}^{a}   \sum_{\beta \in \pi(k,a)}  c_{a,k,\beta} \prod_{j=1}^{k} (D_t^{\beta_j} D v)\right)  \cdot \nabla   + \sum_{k=1}^{a}  \sum_{\beta \in \pi(k,a)}  c_{a,k,\beta} \prod_{j=1}^{k} (D_t^{\beta_j} D v) \cdot [D_t,\nabla]
\notag\\
&=   D_t \left( \sum_{k=1}^{a}  \sum_{\beta \in \pi(k,a)}  c_{a,k,\beta} \prod_{j=1}^{k} (D_t^{\beta_j} D v)\right)   \cdot \nabla  - \sum_{k=1}^{a}  \sum_{\beta \in \pi(k,a)}  c_{a,k,\beta} \prod_{j=1}^{k} (D_t^{\beta_j} D v) Dv \cdot \nabla
\end{align*}
where we have denoted by 
\begin{align*}
\pi(k,a) = \left\{ \beta \in \N^k \colon |\beta| = a-k \right\}
\end{align*}
the set of all partitions of a set of $a-k$ elements into $k$  sets.
For the first term we use the Leibniz rule for $D_t$, so that for any $\beta \in \pi(k,a)$, we obtain an element $\beta + e_j \in \pi(k,a+1)$, with $e_j = (0,\ldots,0,1,0,\ldots,0) \in \N^k$, and the $1$ lies in the $j^{th}$ coordinate. For $1 \leq k \leq a$, this in fact lists all the elements in $\pi(k,a+1)$. For the second sum, we identify $\beta\in \pi(k,a)$ with $\beta \in \pi(k+1,a+1)$, upon padding it with a $0$ in the $k+1^{st}$ entry. Changing variables $k+1 \to k$, then recovers an element $\beta \in \pi(k,a+1)$, including the case $k = a+1$, which was missing from the first sum.
\end{proof}

From Lemma~\ref{lem:Komatsu} and Lemma~\ref{lem:ad:Dt:a:D} we deduce the following: 
\begin{lemma}
\label{lem:cooper:2}
 Let $p\in [1,\infty]$. 
Fix $N_x,N_t,N_*,M_* \in \N$, let $v$ be a vector field, let $D_t = \partial_t + v\cdot \nabla$ be the associated material derivative, and let $\Omega$ be a space-time domain. Assume that the vector field $v$ obeys 
\begin{align}
\norm{D^N D_t^M D v}_{L^\infty(\Omega)} \les \const_v \MM{N+1,N_x,\lambda_v,\tilde \lambda_v} \MM{M,N_t,\mu_v,\tilde \mu_v}
\label{eq:cooper:2:v}
\end{align}
for $N \leq N_*$  and $M \leq M_*$.
Moreover, let $f$ be a function which obeys
\begin{align}
\norm{D^N D_t^M f}_{L^p(\Omega)} \les \const_f \MM{N,N_x,\lambda_f,\tilde \lambda_f} \MM{M,N_t,\mu_f,\tilde \mu_f}
\label{eq:cooper:2:f}
\end{align}
for all $N\leq N_*$ and $M \leq M_*$. 
Denote
\begin{align*}
\lambda = \max\{ \lambda_f,\lambda_v\}, \quad \tilde \lambda= \max\{\tilde \lambda_f,\tilde \lambda_v\}, \quad \mu = \max\{\mu_f,\mu_v\}, \quad \tilde \mu = \max\{\tilde \mu_f,\tilde \mu_v\}.
\end{align*}
Let $m,n,\ell \geq 0$ be such that $n+\ell \leq N_*$ and $m\leq M_*$. 
Then, we have that the commutator $[D_{t}^m,D^n]$ is bounded  as
\begin{align}
\norm{D^\ell \left[ D_t^m,D^n \right] f}_{L^{p}(\Omega)} 
&\les \const_f  \const_v \tilde \lambda_v \MM{\ell+n,N_x,\lambda,\tilde \lambda}   \MM{m-1,N_t,\max\{\mu,\const_v \tilde \lambda_v\},\max\{\tilde \mu,\const_v \tilde \lambda_v\}} 
\label{eq:cooper:2:f:1:*}
\\
&\les \const_f \MM{\ell+n,N_x,\lambda,\tilde \lambda}   \MM{m ,N_t,\max\{\mu,\const_v \tilde \lambda_v\},\max\{\tilde \mu,\const_v \tilde \lambda_v\}}.
\label{eq:cooper:2:f:1}
\end{align}
Moreover, we have that for $k\geq 2$, and any $\alpha,\beta\in \N^k$  with $|\alpha|\leq N_*$ and $|\beta|\leq M_*$, the estimate
\begin{align}
\norm{\left( \prod_{i=1}^{k} D^{\alpha_i} D_t^{\beta_i}\right) f}_{L^{p}(\Omega)} 
&\les \const_f \MM{|\alpha|,N_x,\lambda,\tilde \lambda}   \MM{|\beta|,N_t,\max\{\mu,\const_v \tilde \lambda_v\},\max\{\tilde \mu,\const_v \tilde \lambda_v\}}
\label{eq:cooper:2:f:2}
\end{align}
holds.
\end{lemma}

\begin{remark}
\label{rem:cooper:2:sum}
If instead of \eqref{eq:cooper:2:v} and \eqref{eq:cooper:2:f} holding for $N\leq N_*$ and $M\leq M_*$, we know that both of these inequalities hold for all $N+M \leq N_\circ$ for some $N_\circ \geq 1$, then the conclusions of the Lemma hold as follows: the bounds \eqref{eq:cooper:2:f:1:*} and \eqref{eq:cooper:2:f:1} hold for $\ell+n+m\leq N_\circ$, while \eqref{eq:cooper:2:f:2} holds for $|\alpha|+|\beta| \leq N_\circ$. This fact follows immediately from the proof of the Lemma, but may alternatively also be derived from its statement (see also Lemma~\ref{lem:Sobolev:cutoffs}).
\end{remark}

\begin{remark}
\label{rem:cooper:2}
In Lemma~\ref{lem:cooper:2}, if the assumption \eqref{eq:cooper:2:f} is replaced by 
\begin{align}
\norm{D^{N} D_t^M f}_{L^{p}(\Omega)} \les \const_f \MM{N-1,N_x,\lambda_f,\tilde \lambda_f} \MM{M,N_t,\mu_f,\tilde \mu_f}
\,,
\label{eq:cooper:2:f:alt}
\end{align}
whenever $1 \leq N\leq N_*$, then the conclusion \eqref{eq:cooper:2:f:2} changes, and it instead becomes
\begin{align}
\norm{\left( \prod_{i=1}^{k} D^{\alpha_i} D_t^{\beta_i}\right) f}_{L^{p}(\Omega)} 
&\les \const_f \MM{|\alpha|-1,N_x,\lambda,\tilde \lambda}   \MM{|\beta|,N_t,\max\{\mu,\const_v \tilde \lambda_v\},\max\{\tilde \mu,\const_v \tilde \lambda_v\}}
\label{eq:cooper:2:f:2:alt}
\end{align}
whenever $|\alpha|\geq 1$. This follows for instance by noting that the sum on the second line of \eqref{eq:cooper:2:f:1:temp} only contains terms with $j\geq 1$, so that \eqref{eq:cooper:2:f:alt} is not required when $N=0$.
\end{remark}

\begin{proof}[Proof of Lemma~\ref{lem:cooper:2}]
First, we deduce from \eqref{eq:ad:Dt:a:D} that for any $\alpha_i\geq 1$ and $1\leq i \leq n$, we have
\begin{align}
(\ad D_t)^{\alpha_i}(D) = \sum_{\kappa_i=1}^{\alpha_i}  f_{\kappa_i,\alpha_i} \cdot \nabla
\label{eq:cooper:2:proof:1}
\end{align}
where the functions $f_{\kappa_i,\alpha_i}$ are computed as
\begin{align*}
f_{\kappa_i,\alpha_i} = \sum_{\{ \beta \in \N^{\kappa_i} \colon |\beta| = \alpha_i - \kappa_i \}} c_{(\dots)} \prod_{j=1}^{\kappa_i} (D_t^{\beta_j} Dv)
\end{align*}
for suitable combinatorial coefficients (tensors) $c_{(\dots)}$ which depend on $\kappa_i,\alpha_i$, and $\beta$.
In particular, in view of assumption \eqref{eq:cooper:2:v}, and the Leibniz rule, we have that 
\begin{align}
\norm{D^\ell f_{\kappa_i,\alpha_i} }_{L^\infty(\Omega)} \les \const_v^{\kappa_i} \MM{\kappa_i+\ell,N_x,\lambda_v,\tilde \lambda_v} \MM{\alpha_i-\kappa_i,N_t,\mu_v,\tilde \mu_v}.
\label{eq:cooper:2:proof:2}
\end{align}
Next, from \eqref{eq:cooper:2:proof:1} we deduce that for any $\alpha \in \N^n$ with $|\alpha|\geq 1$, one may write 
\begin{align}
\prod_{i=1}^{n} (\ad D_t)^{\alpha_i}(D) = \sum_{j=1}^{n} g_{j,\alpha} D^j
\label{eq:cooper:2:proof:3}
\end{align}
where
\begin{align*}
g_{j,\alpha} = \sum_{\{\kappa \in \N^n \colon 1^n \leq \kappa \leq \alpha \}} \sum_{\{\gamma \in \N^n \colon |\gamma| = n-j\}}  \tilde c_{(\dots)} \prod_{i=1}^n D^{\gamma_i} f_{\kappa_i,\alpha_i}.
\end{align*}
As a consequence of \eqref{eq:cooper:2:proof:2} we see that 
\begin{align}
\norm{D^\ell g_{j,\alpha}}_{L^{\infty}(\Omega)}  \les \sum_{|\kappa|=1}^{|\alpha|} \const_v^{|\kappa|} \MM{\ell+n-j+|\kappa|,N_x,\lambda_v,\tilde \lambda_v} \MM{|\alpha|-|\kappa|,N_t,\mu_v,\tilde \mu_v}.
\label{eq:cooper:2:proof:4}
\end{align}
From \eqref{eq:Komatsu}, assumption~\eqref{eq:cooper:2:f}, identity \eqref{eq:cooper:2:proof:3}, and bound \eqref{eq:cooper:2:proof:4}, we see that 
\begin{align}
\norm{D^\ell \left[ D_t^m,D^n \right] f}_{L^{p}(\Omega)} 
&  \les \sum_{|\alpha|= 1}^{m} \sum_{j=1}^n \norm{D^\ell \left( g_{j,\alpha} D^j D_t^{m-|\alpha|}\right) f}_{L^{p}(\Omega)}
\notag\\
&  \les \sum_{|\alpha|= 1}^{m} \sum_{j=1}^n \norm{D^\ell g_{j,\alpha}}_{L^\infty(\Omega)} \norm{D^j D_t^{m-|\alpha|}  f}_{L^{p}(\Omega)} + \norm{g_{j,\alpha}}_{L^\infty(\Omega)} \norm{D^{\ell+j} D_t^{m-|\alpha|}  f}_{L^{p}(\Omega)}
\notag\\
&  \les \sum_{k= 1}^{m} \sum_{j=1}^n    \const_f \const_v^{k} \MM{\ell+n-j+k,N_x,\lambda_v,\tilde \lambda_v}  \MM{j,N_x\lambda ,\tilde \lambda } \MM{m-k,N_t,\mu ,\tilde \mu } \notag\\
&\qquad \qquad   + \const_f \const_v^{k} \MM{n-j+k,N_x,\lambda_v,\tilde \lambda_v}   \MM{j+\ell,N_x\lambda ,\tilde \lambda } \MM{m-k,N_t,\mu ,\tilde \mu }
\notag\\
&  \les \const_f  \MM{\ell+n,N_x,\lambda ,\tilde \lambda }  \sum_{k= 1}^{m}      (\const_v \tilde \lambda_v)^{k} \MM{m-k,N_t,\mu ,\tilde \mu } 
\label{eq:cooper:2:f:1:temp}
\end{align}
from which \eqref{eq:cooper:2:f:1} follows directly. 

In order to prove \eqref{eq:cooper:2:f:2} we proceed by induction on $k$. For $k=1$ the statement holds in view of \eqref{eq:cooper:2:f}. We assume that \eqref{eq:cooper:2:f:2} holds for $k' \leq k-1$, and denote
\begin{align*}
P_{k'} = \left( \prod_{i=1}^{k'} D^{\alpha_i} D_t^{\beta_i}\right) f.
\end{align*}
With this notation we have
\begin{align*}
P_k 
&= D^{\alpha_k} D_t^{\beta_k} D^{\alpha_{k-1}} D_t^{\beta_{k-1}} P_{k-2} \notag\\
&= D^{\alpha_k+\alpha_{k-1}} D_t^{\beta_k+\beta_{k-1}}   P_{k-2} + D^{\alpha_k} \left[ D_t^{\beta_k} , D^{\alpha_{k-1}}\right] D_t^{\beta_{k-1}} P_{k-2}.
\end{align*}
Using \eqref{eq:cooper:2:f:2} with $k-1$ gives the desired estimate for the first term above. For the second term, we appeal to the commutator bound \eqref{eq:cooper:2:f:1}, applied to $D_t^{\beta_{k-1}} P_{k-2}$, which obeys condition \eqref{eq:cooper:2:f} in view of \eqref{eq:cooper:2:f:2} at level $k-1$. This concludes the proof of \eqref{eq:cooper:2:f:2} at level $k$.
\end{proof}

\subsection{Intermittency-friendly inversion of the divergence}
\label{sec:inverse:divergence}
Given a vector field $G^i $, a zero mean periodic function $\varrho$ and an incompressible flow $\Phi$,  our goal in this section is to write $G^{i}(x) \varrho(\Phi(x))$  as the divergence of a symmetric tensor. 
\begin{proposition}[\textbf{Inverse divergence iteration step}]
\label{prop:Celtics:suck}
Fix  two zero-mean $\T^3$-periodic functions $\varrho$ and $\vartheta$, which are related by $\varrho  =  \Delta \vartheta $. Let $\Phi$ be a volume preserving transformation of $\T^3$, such that $\norm{\nabla \Phi - \Id}_{L^\infty(\T^3)} \leq \sfrac 12$. Define the matrix $A = (\nabla \Phi)^{-1}$. Given a vector field $G^i$, we have  
\begin{align}
G^i  \varrho\circ \Phi = \partial_n \RR^{in} + \partial_i P + E^i
\label{eq:Celtics:suck:total}
\end{align}
where the traceless symmetric stress $R^{in}$ is given by
\begin{align}
\RR^{in}
&= \left(  G^i A^n_\ell  + G^n A^i_\ell  -A^i_k A^n_k   G^p \partial_p \Phi^\ell \right) (\partial_\ell \vartheta) \circ \Phi  -  P \delta_{in}
\,,
\label{eq:Celtics:suck:stress}
\end{align}
where the pressure term is given by 
\begin{align}
P
&= \left(2   G^n A^n_\ell   -A^n_k A^n_k   G^p \partial_p \Phi^\ell \right) (\partial_\ell \vartheta) \circ \Phi   \, ,
\label{eq:Celtics:suck:pressure}
\end{align}
and the error term $E^i$ is given by 
\begin{align}
E^i
&=  \left(  \partial_n \left(G^p A^i_k A^n_k - G^n A^i_k A^p_k\right)  \partial_p \Phi^\ell
- \partial_n G^i A^n_\ell  \right) (\partial_\ell \vartheta) \circ \Phi 
\, .
\label{eq:Celtics:suck:error}
\end{align}
\end{proposition}
\begin{proof}[Proof of Proposition~\ref{prop:Celtics:suck}]  Note that by definition we have $A^{k}_\ell \partial_j \Phi^\ell = \delta_{kj}$. Since $\Phi$ is volume preserving, $\det(\nabla \Phi) = 1$, and so each entry of the matrix $A$ equals the corresponding cofactor of $\nabla \Phi$, which in three dimensions is a quadratic function of entries of $\nabla \Phi$ given explicitly by $A^{i}_j = \frac 12 \eps_{ipq}\eps_{jk\ell} \partial_k \Phi^p \partial_\ell \Phi^q$. In two dimensions $A$ is a linear map in $\nabla \Phi$. Moreover, since $\Phi$ is volume preserving,  the Piola identity $\partial_j A^{j}_i = 0$ holds for every $i \in \{1,2,3\}$. 
The main identity that we use in the proof is that for any scalar function $\varphi$ we have $(\partial_i \varphi) \circ \Phi = A^{m}_i \partial_m (\varphi \circ \Phi) = \partial_m (A^m_i \varphi \circ \Phi)$. 

Starting from $\varrho =  \Delta \vartheta$, we have
\begin{align*}
G^i \varrho\circ \Phi
&=  G^i (\partial_{kk} \vartheta)\circ \Phi
\\
&= G^i A^n_k \partial_n (\partial_k \vartheta) \circ \Phi
\\
&= \partial_n \left(G^i A^n_k (\partial_k \vartheta) \circ \Phi \right) - \partial_n G^i A^n_k  (\partial_k \vartheta) \circ \Phi
\\
&= \partial_n \left(G^i A^n_k (\partial_k \vartheta) \circ \Phi + G^n A^i_k (\partial_k \vartheta) \circ \Phi\right) 
-  \partial_n \left( G^n A^i_k (\partial_k \vartheta) \circ \Phi \right)
- \partial_n G^i A^n_k  (\partial_k \vartheta) \circ \Phi
\,.
\end{align*}
Next, we have
\begin{align*}
\partial_n \left( G^n A^i_k (\partial_k \vartheta) \circ \Phi \right)
&=  \partial_n \left( G^n A^i_k A^p_k \partial_p ( \vartheta \circ \Phi) \right)
\\
&= \partial_p \partial_n \left( G^n A^i_k A^p_k   \vartheta \circ \Phi  \right) - \partial_n \left( \partial_p( G^n A^i_k A^p_k  ) \vartheta \circ \Phi  \right)
\\
&= \partial_p \left( G^n A^i_k A^p_k   \partial_n (\vartheta \circ \Phi ) \right)
+ \partial_p  \left(\partial_n( G^n A^i_k A^p_k )  \vartheta \circ \Phi  \right)  
- \partial_n \left( \partial_p( G^n A^i_k A^p_k)    \vartheta \circ \Phi  \right)
\\
&=\partial_n \left( G^p A^i_k A^n_k   \partial_p (\vartheta \circ \Phi ) \right)
+ \partial_n  \left(\partial_p( G^p A^i_k A^n_k)  \vartheta \circ \Phi  \right) 
- \partial_n  \left(\partial_p(G^n A^i_k A^p_k)  \vartheta \circ \Phi  \right) 
\end{align*}
where in the last equality we have just switched the letters of summation $n$ and $p$.
We further massage the last term in the above equality. 
\begin{align*}
\partial_n  \left(\partial_p(G^n A^i_k A^p_k)  \vartheta \circ \Phi  \right)  
&= \partial_p \left(G^n A^i_k A^p_k\right)  \partial_n(\vartheta \circ \Phi  )  
+  \partial_{np} \left(G^n A^i_k A^p_k \right)  \vartheta \circ \Phi    
\\
&= \partial_p \left(G^n A^i_k A^p_k\right)  \partial_n(\vartheta \circ \Phi  )  
+ \partial_p \left( \partial_{n} \left(G^n A^i_k A^p_k \right)  \vartheta \circ \Phi   \right) -  \partial_{n} \left(G^n A^i_k A^p_k \right) \partial_p( \vartheta \circ \Phi )
\end{align*}
Combining the above three equalities, we arrive at 
\begin{align*}
G^i \varrho\circ \Phi
&= \partial_n \left( (G^i A^n_k  + G^n A^i_k) (\partial_k \vartheta) \circ \Phi - A^i_k A^n_k   G^p  \partial_p (\vartheta \circ \Phi )  \right) 
\\
&\quad 
+ \partial_n \left(G^p A^i_k A^n_k - G^n A^i_k A^p_k\right)  \partial_p(\vartheta \circ \Phi  ) 
- \partial_n G^i A^n_k  (\partial_k \vartheta) \circ \Phi
\\
&= \partial_n \left( (G^i A^n_k  + G^n A^i_k) (\partial_k \vartheta) \circ \Phi - A^i_k A^n_k   G^p  \partial_p \Phi^\ell (\partial_\ell \vartheta) \circ \Phi    \right) 
\\
&\quad 
+ \partial_n \left(G^p A^i_k A^n_k - G^n A^i_k A^p_k\right)  \partial_p \Phi^\ell (\partial_\ell \vartheta) \circ \Phi  
- \partial_n G^i A^n_\ell  (\partial_\ell \vartheta) \circ \Phi
\end{align*}
In the last equality we have exchanged the order of summation. Identities \eqref{eq:Celtics:suck:total}--\eqref{eq:Celtics:suck:error}  follow upon declaring that the trace part of the symmetric stress is the pressure.
\end{proof}

Proposition~\ref{prop:Celtics:suck} allows us to obtain the following result, which is the main conclusion of this section. 
\begin{proposition}[\textbf{Inverse divergence with error term}]
\label{prop:intermittent:inverse:div}
Fix an incompressible vector field $v$ and denote its material derivative by $D_t = \partial_t + v\cdot\nabla$.  Fix integers $N_* \geq M_* \geq   1$. Also fix $\Ndec, \dpot \geq 1$ such that  $N_* - \dpot \geq 2\Ndec + 4$.

Let $G$ be a vector field and assume there exists a constant $\const_{G} > 0$ and parameters $\lambda, \nu\geq 1$ such that 
\begin{align}
\norm{D^N D_{t}^M G}_{L^{1}}\lesssim \const_{G} \lambda^N\MM{M,M_{t},\nu,\tilde\nu}
\label{eq:inverse:div:DN:G}
\end{align}
for all $N \leq N_*$ and $M \leq M_*$.

Let $\Phi$ be a volume preserving transformation of $\T^3$, such that 
\[
D_t \Phi = 0 \,
\qquad \mbox{and} \qquad
\norm{\nabla \Phi - \Id}_{L^\infty(\supp G)} \leq \sfrac 12 \,.
\] 
Denote by $\Phi^{-1}$ the inverse of the flow $\Phi$,  which is the identity at a time slice which intersects the support of $G$.
Assume that  the velocity field $v$ and the flow functions $\Phi$ and $\Phi^{-1}$ satisfy the following bounds 
\begin{align}
\norm{D^{N+1}   \Phi}_{L^{\infty}(\supp G)} + \norm{D^{N+1}   \Phi^{-1}}_{L^{\infty}(\supp G)} 
&\les \lambda'^{N}
\label{eq:DDpsi}\\
\norm{D^ND_t^M D v}_{L^{\infty}(\supp G)}
&\les \nu \lambda'^{N}\MM{M,M_{t},\nu,\tilde\nu}
\label{eq:DDv}
\,,
\end{align}
for all $N \leq N_*$, $M\leq M_*$, and some $\lambda'>0$. 

Lastly, let $\varrho,\vartheta \colon \T^3 \to \R$ be two zero mean functions with  the following properties:
\begin{enumerate}[(i)]
\item \label{item:inverse:i} there exists $\dpot \geq 1$ and a parameter $\zeta\geq 1$ such that $\varrho (x) = \zeta^{-2\dpot } \Delta^\dpot \vartheta(x)$
\item \label{item:inverse:ii} there exists a parameter $\mu\geq 1$ such that $\varrho$ and $\vartheta$ are $(\sfrac{\T}{\mu})^3$-periodic
\item \label{item:inverse:iii} there exists  parameters $\Lambda\geq \zeta$ and $\const_{*} \geq 1$ such that 
\begin{align}
\norm{D^N \varrho}_{L^1} \les \const_{*} \Lambda^{N}
\qquad \mbox{and} \qquad
\norm{D^N \vartheta}_{L^1} \les \const_{*} \MM{N,2d,\zeta,\Lambda} 
\label{eq:DN:Mikado:density}
\end{align} 
for all $0\leq N \leq \Nfin$, except for the case $N = 2\dpot$  when the Calder\'on-Zygmund inequality fails. In this exceptional case, the second inequality in \eqref{eq:DN:Mikado:density} is allowed to be weaker by a factor of $\Lambda^\alpha$, for an arbitrary $\alpha \in (0,1]$; that is, we only require that $\norm{D^{2\dpot} \vartheta}_{L^1} \les \const_{*} \Lambda^\alpha\zeta^{2\dpot} $.
\end{enumerate}

If the above parameters satisfy
\begin{align}
 \lambda' \leq \lambda \ll \mu \leq \zeta \leq \Lambda  \,,
 \label{eq:inverse:div:parameters:0}
\end{align}
where by the second inequality in \eqref{eq:inverse:div:parameters:0} we mean that 
\begin{align}
 \Lambda^4 \left(\frac{\mu}{2\pi \sqrt{3} \lambda}\right)^{-\Ndec} \leq 1
 \,,
 \label{eq:inverse:div:parameters:1}
\end{align}
then, we have that 
\begin{align}
G \; \varrho\circ \Phi  &=  \div \RR + \nabla P + E  
=: \div\left( \divH \left( G \varrho \circ \Phi \right) \right) + \nabla P + E. \label{eq:inverse:div}
\end{align}
where the traceless symmetric stress $\RR=\divH( G \varrho \circ \Phi)$ and the scalar pressure $P$ are supported in $\supp G$, and for any fixed $\alpha\in (0,1)$ they satisfy 
\begin{align}
\norm{D^N D_{t}^M \RR}_{L^{1}} + \norm{D^N D_{t}^M P}_{L^{1}}
 &\les \Lambda^\alpha \const_{G} \const_{*}   \zeta^{-1} \MM{N,1,\zeta,\Lambda} \MM{M,M_{t},\nu,\tilde\nu} 
\label{eq:inverse:div:stress:1}
\end{align}
for all $N \leq N_* - \dpot $ and $M\leq M_*$. The implicit constants  depend on $N,M,\alpha$ but not $G$, $\varrho$, or $\Phi$. Lastly, for $N \leq N_* - \dpot $ and $M\leq M_*$ the error term $E$  in \eqref{eq:inverse:div} satisfies
\begin{align}
\norm{D^N D_{t}^M E}_{L^{1}}  
\les \const_{G} \const_{*} \Lambda^\alpha \lambda^\dpot  \zeta^{-\dpot } \Lambda^N \MM{M,M_{t},\nu,\tilde\nu} 
\,.
\label{eq:inverse:div:error:1}
\end{align}
We emphasize that the range of $M$ in \eqref{eq:inverse:div:stress:1} and \eqref{eq:inverse:div:error:1} is exactly the same as the one in \eqref{eq:inverse:div:DN:G}, while the range of permissible values for $N$ shrank from $N_*$ to $N_* - \dpot $.

Lastly, let $N_\circ, M_\circ$ be integers such that $1 \leq M_\circ \leq N_\circ \leq M_*/2$.
Assume that in addition to the bound \eqref{eq:DDv} we have the following global lossy estimates
\begin{align}
\norm{D^N \partial_t^M v}_{L^\infty(\T^3)}\les  \const_v \tilde \lambda_q^N \tilde \tau_q^{-M}  
\label{eq:inverse:div:v:global}
\end{align}
for all  $M \leq M_\circ$ and $N+M \leq N_\circ + M_\circ$, where 
\begin{align}
\const_v \tilde \lambda_q \les \tilde \tau_q^{-1}, \qquad \mbox{and} \qquad  \lambda' \leq \tilde \lambda_q \leq \Lambda \leq \lambda_{q+1}  \,.
\label{eq:inverse:div:v:global:parameters}
\end{align}
If $\dpot $ is chosen {\em large enough} so that  
\begin{align}
\const_G \const_* \Lambda \left(\frac{\lambda}{\zeta}\right)^{\dpot -1}   \left(1 + \frac{\max\{\tilde \tau_q^{-1}, \tilde \nu, \const_v \Lambda \}}{\tau_{q}^{-1}}\right)^{M_\circ}
\leq \frac{\delta_{q+2}}{\lambda_{q+1}^5}
\,,
\label{eq:riots:4}
\end{align}
then we may write  
\begin{align}
E = \div \RR_{\mathrm{nonlocal}} + \fint_{\T^3} G \varrho \circ \Phi dx =: \div \left(\divR(G \varrho \circ \Phi)\right) + \fint_{\T^3} G \varrho \circ \Phi dx\,,
\label{eq:inverse:div:error:stress}
\end{align}
where $\RR_{\mathrm{nonlocal}} = \divR(G \varrho \circ \Phi)$ is a traceless symmetric stress which satisfies
\begin{align}
\norm{D^N D_{t}^M \RR_{\mathrm{nonlocal}} }_{L^{1}}  
\leq \frac{\delta_{q+2}}{\lambda_{q+1}^5} \lambda_{q+1}^{N} \tau_{q}^{-M} 
\label{eq:inverse:div:error:stress:bound}
\end{align}
for  $N \leq N_\circ$ and $M\leq M_\circ$.
\end{proposition}

Before turning to the proof of Lemma~\ref{prop:intermittent:inverse:div}, let us make three remarks.  First, we highlight certain parameter values which will occur commonly in applications of the proposition. Second, we comment on a technical aspect of the application of the Proposition in  Section~\ref{ss:stress:error:identification}.  Finally, we comment on the assumptions \eqref{item:inverse:i}--\eqref{item:inverse:iii} and \eqref{eq:inverse:div:parameters:1} and \eqref{eq:riots:4} for the functions $\varrho$ and $\vartheta$, which in applications are related to the transversal densities of the pipe flows. 

\begin{remark}\label{rem:div:derivative:bounds}
Frequently, $G$ will come with derivative bounds which are satisfied for $N+M\leq \Nsharp$.  In this case, we set $N_*=M_*=\sfrac{\Nsharp}{2}$, so that \eqref{eq:inverse:div:DN:G} is satisfied. The bounds in \eqref{eq:DDpsi} and \eqref{eq:DDv} will be true (due to Corollary~\ref{cor:deformation} and estimate~\eqref{eq:nasty:D:vq}) for much higher order derivatives than $\sfrac{\Nsharp}{2}$, and so we ignore them. The bounds in \eqref{eq:DN:Mikado:density} are given by construction in Proposition~\ref{pipeconstruction}. Then the bounds  \eqref{eq:inverse:div:stress:1} and \eqref{eq:inverse:div:error:1} are satisfied for $N\leq\sfrac{\Nsharp}{2}-\dpot$ and $M\leq\sfrac{\Nsharp}{2}$, and in particular for $N+M\leq \sfrac{\Nsharp}{2}-\dpot$.  In \eqref{eq:inverse:div:v:global} we will then set $N_\circ=M_\circ\leq\sfrac{\Nsharp}{4}$, which in practice will give $N_\circ=M_\circ = 3\NindLarge$.  Arguing in the same way used to produce the bound \eqref{eq:space:time:v:ell:q-1:rough} shows that for $N+M\leq\Nfin$,
\begin{equation}\label{eq:inverse:div:v:global:proof}
\left\| D^N \partial_t^M \vlq \right\|_{L^\infty} \lesssim \left( \lambda_q^4 \delta_q^{\sfrac 12} \right) \tilde\lambda_q^N \tilde\tau_q^{-M}
\end{equation}
and so \eqref{eq:inverse:div:v:global} is satisfied with $\const_v =\lambda_q^4 \delta_q^{\sfrac 12}$ up to $N+M\leq 2\Nfin$ (which will in fact be far beyond anything required for the inverse divergence). The inequalities in \eqref{eq:inverse:div:v:global:parameters} follow from \eqref{eq:Lambda:q:t:1}, \eqref{eq:Lambda:q:x:1:NEW}, and the definitions of $\lambda'=\tilde\lambda_q$ and $\Lambda=\lambda_{q+1}$.  In applications, $\tilde\nu=\tilde\tau_q^{-1}\Gamma_{q+1}^{-1}$, so that from \eqref{eq:Lambda:q:x:1:NEW} and \eqref{eq:Lambda:q:t:1}, we have that
$$  
\max\{\tilde\tau_q^{-1},\tilde\nu,\const_v \Lambda \} \leq \tau_q^{-1} \tilde \lambda_q^3 \tilde \lambda_{q+1} \leq \tau_q^{-1}  \lambda_{q+1}^4 \,,
$$
which holds as soon as $\eps_\Gamma$ is taken to be sufficiently small.
Then, \eqref{eq:riots:4} will follow from \eqref{eq:CF:new}. Finally, \eqref{eq:inverse:div:error:stress:bound} will hold for all $N,M\leq\sfrac{\Nsharp}{4}$, which will be taken larger than $3\NindLarge$. In summary, if \eqref{eq:inverse:div:DN:G} is known to hold for $N+M \leq N^\sharp$, then \eqref{eq:inverse:div:stress:1} holds for $N\leq \sfrac{N^\sharp}{2} - \dpot$ and $M\leq \sfrac{N^\sharp}{2}$, while \eqref{eq:inverse:div:error:stress:bound} is valid for $N,M \leq \sfrac{N^\sharp}{4}$.
\end{remark}
\begin{remark}
In the identification of the error terms in  Section~\ref{ss:stress:error:identification}, we apply Proposition~\ref{prop:intermittent:inverse:div} to write
$$   G \varrho \circ \Phi = \div \left( \divH (G \varrho \circ \Phi) \right) + \nabla P + \div \left( \divR \left( G \varrho \circ \Phi \right) \right) + \fint_{\T^3} G \varrho \circ \Phi dx . $$
The estimates on $G$, $\varrho$, and $\Phi$, and then the right hand side of the above equality will be checked in later sections.  We emphasize that $\divH$ is a \emph{local} operator and is thus well-suited to working with estimates on the support of a cutoff function.  Conversely, $\divR$ is non-local but will always produce extremely small errors which can be absorbed into $\RR_{q+1}$ and for which the cutoff functions are not relevant.
\end{remark}
\begin{remark}\label{rem:div:usage}
We consider examples of functions $\vartheta$ and $\varrho$ with which Proposition~\ref{prop:intermittent:inverse:div} is used.
\begin{enumerate}[(a)]
\item {\em This is the case corresponding to the density of a pipe flow.}
Recalling the construction of pipe flows from Proposition~\ref{pipeconstruction}, we let $\varrho=\varrho_{\xi,\lambda,r}^k$ and $\vartheta=\vartheta_{\xi,\lambda,r}^k$.  Set $\zeta=\Lambda=\lambda$ (where the $\lambda$ refers to Proposition~\ref{pipeconstruction}, not the $\lambda$ from Proposition~\ref{prop:intermittent:inverse:div}) and $\mu=\lambda r$. To verify \eqref{item:inverse:i}, we appeal to item \eqref{item:pipe:1} from Proposition~\ref{pipeconstruction} and our choice of $\Lambda$ and $\mu$. The periodicity requirement in \eqref{item:inverse:ii} follows from item \eqref{item:pipe:2} from Proposition~\ref{pipeconstruction} and referring back to item \eqref{item:point:1} from Proposition~\ref{prop:pipe:shifted}. Next, \eqref{eq:DN:Mikado:density} is satisfied with $\const_\ast=r $ using \eqref{e:pipe:estimates:1}.  Finally, \eqref{eq:inverse:div:parameters:1} and \eqref{eq:riots:4} will follow from large choice of $\Ndec$ and $\dpot $ and the fact that our choice of $\lambda$ can always be related to $\zeta$ and $\mu$ by a power strictly less than $1$ (see \eqref{eq:lambdaqn:identity:2} and \eqref{eq:CF:new}).  


\item {\em This is the case corresponding to the Littlewood-Paley projection for the square of the density of a pipe flow.}
Fix $1\leq \mu \leq \zeta < \Lambda$, and a constant $\const_*>0$. Let $\eta(x)$ be any $(\sfrac{\T}{\mu})^3$-periodic function (which need not have zero-mean), with $\norm{\eta}_{L^p(\T^3)} \leq \const_*$. In applications, we shall refer to \eqref{eq:pipes:flowed:2} from Proposition~\ref{pipeconstruction} and set $\eta=\left(\varrho_{\xi,\lambda,r}^k\right)^2$ and $\mu=\lambda r$. This means that we may write
$\eta(x) = \eta_\mu( \mu x)$ where $\eta_\mu$ is $\T^3$-periodic, with $\norm{\eta_\mu }_{L^1(\T^3)} \leq \const_*$. Following \eqref{eq:pipes:flowed:2} from Proposition~\ref{pipeconstruction} with $\lambda_1=\zeta$, $\lambda_2=\Lambda$, we may define 
\[
\varrho(x) = \left(\Proj_{[ \zeta, \Lambda]} \eta\right)(x) = \left(\Proj_{\left[\frac{\zeta}{\mu},\frac{\Lambda}{\mu}\right]} \eta_\mu\right)(\mu x)
\,, 
\]
a function which is  $(\sfrac{\T}{\mu})^3$-periodic and has zero mean (since $\zeta\geq \mu > 0$), and clearly
$$\norm{D^N \varrho}_{L^1(\T^3)} \leq \const_* \Lambda^N.$$
We now define the associated function $\vartheta$ by first defining the zero mean $\T^3$-periodic function  
\[
\vartheta_\mu = \left(\frac{\zeta}{\mu}\right)^{2\dpot } \Delta^{-\dpot } \Proj_{\left[\frac{\zeta}{\mu},\frac{\Lambda}{\mu}\right]} \eta_\mu 
\, ,
\]
where the negative powers of the Laplacian are defined simply as a Fourier multiplier (since the periodic function we apply it to has zero mean). Then we let 
\[
\vartheta(x) = \vartheta_\mu(\mu x)
\]
which has zero mean, is $(\sfrac{\T}{\mu})^3$-periodic, and clearly satisfies $\Delta^\dpot  \vartheta = \zeta^{2\dpot } \varrho$, as required. It only remains to estimate the $\dot{W}^{N,1}$ norms of $\vartheta$, which up to paying a factor of $\mu^N$ is equivalent to  estimating the $\dot{W}^{N,1}$ norms of $\vartheta_\mu$. When $0 \leq N < 2\dpot $, the operator 
$$\displaystyle D^N \Delta^{-\dpot } \Proj_{\left[\frac{\zeta}{\mu},\frac{\Lambda}{\mu}\right]}$$
is a bounded operator on $L^1$, whose operator norm is $\les (\sfrac{\zeta}{\mu})^{N-2\dpot }$. This may be verified via a standard Littlewood-Paley argument. The exceptional case $N=2\dpot $  leads to a logarithmic loss since there are roughly $\log(\Lambda/\mu)$-many Littlewood-Paley shells to estimate; we absorb this loss into a factor of $\Lambda^\alpha$, with $\alpha>0$ arbitrarily small. Since $\norm{\eta_\mu}_{L^1} \leq \const_*$, the second estimate in \eqref{item:inverse:iii} above clearly follows, at least when $N\leq 2\dpot $. The case $N> 2\dpot $ follows similarly, except that now $D^N \Delta^{-\dpot }$ is a positive order operator, and thus the  $L^1$ operator norm of $D^N \Delta^{-\dpot } \Proj_{\left[\frac{\zeta}{\mu},\frac{\Lambda}{\mu}\right]}$ is bounded by 
$\approx(\sfrac{\Lambda}{\mu})^{N-2\dpot }$.  We remark that  as in the previous case, \eqref{eq:inverse:div:parameters:1} and \eqref{eq:riots:4} will follow from large choices of $\Ndec$ and $\dpot $ and the fact that our choice of $\lambda$ can always be related to $\zeta$ and $\mu$ by a power strictly less than $1$.  
\end{enumerate}
\end{remark}

\begin{proof}[Proof of Proposition~\ref{prop:intermittent:inverse:div}]
Since $D_{t}\Phi\equiv 0$, we have that $D^N D_t^m \nabla \Phi = D^N [D_t^M, \nabla] \Phi$. We may now appeal to Lemma~\ref{lem:cooper:2}, more precisely, to Remark~\ref{rem:cooper:2}. Let $\Omega= \supp G$, and $f =\Phi$, so that \eqref{eq:DDpsi} implies that \eqref{eq:cooper:2:f:alt} holds with $\const_f = 1$, $\lambda_f = \tilde \lambda_f = \lambda'$, and $\mu_f = \tilde \mu_f = 1$ (in fact, whenever $M\geq 1$ we may replace the right side of \eqref{eq:cooper:2:f:alt} by $0$). Moreover, \eqref{eq:DDv} implies that \eqref{eq:cooper:2:v} holds with $\const_v = \nu/\lambda'$, $\lambda_v = \tilde \lambda_v = \lambda'$, $N_t = M_t$, $\mu_v = \nu$ and $\tilde \mu_v = \tilde \nu$. We deduce from \eqref{eq:cooper:2:f:2:alt} that 
\begin{align}
 \norm{D^{N''} D_t^M D^{N' }D \Phi}_{L^\infty(\supp G)} \les \lambda'^{N'+N''} \MM{M,M_t,\nu, \tilde \nu}
 \label{eq:derivatives:phase:gradient}
\end{align}
whenever  $N'+N'' \leq N_*$  and $M\leq M_*$. 
Similarly, we use Lemma~\ref{lem:cooper:2} with $f=G$, so that due to \eqref{eq:inverse:div:DN:G} we know that \eqref{eq:cooper:2:f} holds with $\const_f = \const_G$, $\lambda_f = \tilde \lambda_f = \lambda$, $\mu_f = \nu$, $\tilde \mu_f = \tilde \nu$, and $N_t = M_t$. With $\Omega = \supp G$, since $\lambda' \leq \lambda$, as before we have that \eqref{eq:DDv} implies that \eqref{eq:cooper:2:v} holds with $\const_v = \nu/\lambda$, $\lambda_v = \tilde \lambda_v = \lambda$, $N_t = M_t$, $\mu_v = \nu$ and $\tilde \mu_v = \tilde \nu$. Therefore, from \eqref{eq:cooper:2:f:2} we obtain that 
\begin{align}
 \norm{D^{N''} D_t^M D^{N'}G}_{L^1} \les \const_G \lambda^{N'+N''} \MM{M,M_t,\nu, \tilde \nu}
 \label{eq:derivatives:mixed:G}
\end{align}
whenever $N'+N''\leq N_*$ and $M \leq M_*$. With \eqref{eq:derivatives:phase:gradient} and \eqref{eq:derivatives:mixed:G}, we turn to the proof of \eqref{eq:inverse:div:stress:1}. 

Instead of defining $\RR$ and $P$ separately, we shall simply construct a symmetric stress $R$ with a prescribed divergence, and use the convention that $P = \tr(R)$ and $\RR = R - \tr(R) \Id$. The construction is based on iterating Proposition~\ref{prop:Celtics:suck}, $d$ times.
For notational purposes, let $\varrho_{(0)} = \varrho$, and for $1 \leq k \leq \dpot $ we let $\varrho_{(k)} = (\zeta^{-2}\Delta)^{\dpot -k} \vartheta$.  Then $\varrho_{(k-1)} = \zeta^{-2} \Delta \varrho_{(k)}$, and $\varrho_{(\dpot )} = \vartheta$. We also define $G_{(0)} = G$.

Since $\rho_{(0)} = \zeta^{-2} \Delta \rho_{(1)}$, we deduce from Proposition~\ref{prop:Celtics:suck}, identities \eqref{eq:Celtics:suck:total}--\eqref{eq:Celtics:suck:error} that 
\begin{align}
G_{(0)}^i \varrho_{(0)}\circ \Phi =  \partial_n R_{(0)}^{in} +  G_{(1)}^{i\ell} (\zeta^{-1}  \partial_\ell \varrho_{(1)}) \circ \Phi
\label{eq:inverse:div:G0}
\end{align}
where the symmetric stress $R_{(0)}$ is given by
\begin{align}
R_{(0)}^{in} = \zeta^{-1}   \underbrace{\left(  G_{(0)}^i A^n_\ell  + G_{(0)}^n A^i_\ell  -A^i_k A^n_k   G_{(0)}^p \partial_p \Phi^\ell \right)}_{=:S_{(0)}^{in\ell} }   (\zeta^{-1} \partial_\ell \varrho_{(1)})\circ \Phi  \,,
\label{eq:inverse:div:R0}
\end{align}
the error terms are computed as
\begin{align}
G_{(1)}^{i\ell} =  \zeta^{-1} \left(  \partial_n \left(G_{(0)}^p A^i_k A^n_k - G_{(0)}^n A^i_k A^p_k\right)  \partial_p \Phi^\ell\right)
- \partial_n G_{(0)}^i A^n_\ell  
\label{eq:inverse:div:G1}
\,,
\end{align}
where as before we denote $(\nabla \Phi)^{-1} = A$. We first show that the symmetric stress $R_{(0)}$ defined in \eqref{eq:inverse:div:R0} satisfies the estimate \eqref{eq:inverse:div:stress:1}. First, we note that the $\zeta^{-1}$ factor has already been accounted for explicitly in \eqref{eq:inverse:div:R0}. Second, we note that since $D_t \Phi =0$, material derivatives may only land on the components of the $3$-tensor $S_{(0)}$. Third, the function $\zeta^{-1} D \varrho_{(1)}$ has zero mean, is $(\sfrac{\T}\mu)^3$ periodic, and satisfies  
\begin{align}
\norm{D^N (\zeta^{-1} D \varrho_{(1)}) }_{L^1} \les \const_{*} \MM{N,1,\zeta,\Lambda}
\label{eq:zero:order:stress:1}
\end{align} 
for $1\neq N\leq \Nfin$, in view of \eqref{eq:DN:Mikado:density}. For  $N=1$, the above estimate  incurs a logarithmic loss of $\Lambda$, which we can absorb with $\Lambda^\alpha$ for any $\alpha>0$ to produce the estimate
\begin{align}
\norm{D  (\zeta^{-1} D \varrho_{(1)}) }_{L^1} \les \Lambda^\alpha \const_{*} \MM{N,1,\zeta,\Lambda}.
\label{eq:zero:order:stress:1:1}
\end{align}
The implicit constants depend on $\alpha$ and degenerate as $\alpha\rightarrow 0$. Fourth, the components of the $3$-tensor $S_{(0)}$ are sums of terms of two kinds: $G_{(0)} \otimes A$  is a linear function of $G_{(0)}$ multiplied by a homogeneous quadratic polynomial in $D \Phi$, while $G \otimes A \otimes A \otimes D\Phi$ is a linear function of $G$ multiplied by a homogeneous polynomial of degree $5$ in the entries of $D\Phi$. In particular, due to our assumption \eqref{eq:inverse:div:DN:G} and the previously established bound \eqref{eq:derivatives:phase:gradient}, upon applying the Leibniz rule and using that $\lambda'\leq \lambda$, we obtain that 
\begin{align}
\norm{D^N D_t^M S_{(0)}}_{L^1} \les \const_{G} \lambda^N  \MM{M,M_t,\nu, \tilde \nu} 
\label{eq:zero:order:stress:2}
\end{align} 
for $N\leq N_*$ and $M\leq M_*$.
Having collected these estimates, the $L^1$ norm of the space-material derivatives of $R_{(0)}$ is obtained from Lemma~\ref{l:slow_fast}. As dictated by \eqref{eq:inverse:div:R0} we apply this lemma with $f = \zeta^{-1} S_{(0)}$, and $\varphi = \zeta^{-1} \nabla \varrho_{(1)}$. Due to \eqref{eq:zero:order:stress:2}, the bound \eqref{eq:slow_fast_0} holds with $\const_f = \const_G \zeta^{-1}$. Due to \eqref{eq:DDpsi} and $\lambda'\leq \lambda$, the assumptions \eqref{eq:slow_fast_1} and \eqref{eq:slow_fast_2} are verified. Next, due to \eqref{eq:zero:order:stress:1} and \eqref{eq:zero:order:stress:1:1}, the assumption \eqref{eq:slow_fast_4} is verified, with $N_x=1$, $\tilde \zeta = \Lambda$,  and $\const_{\varphi}=\const_* \Lambda^\alpha$. Lastly, assumption \eqref{eq:inverse:div:parameters:1} verifies the condition \eqref{eq:slow_fast_3} of Lemma~\ref{l:slow_fast}. 
Thus, applying estimate \eqref{eq:slow_fast_5} we deduce that 
\begin{align}
\norm{D^N D_t^M R_{(0)}}_{L^1} \les \const_G \const_* \Lambda^\alpha \zeta^{-1} \MM{N,1,\zeta,\Lambda} \MM{M,M_t,\nu,\tilde \nu}
\label{eq:zero:order:stress:2a}
\end{align}
for all $N\leq N_*$ and $M\leq M_*$, which is precisely the bound stated in \eqref{eq:inverse:div:stress:1}.   Here we have used that $N_* \geq 2 \Ndec +4$, which was required due to \eqref{eq:slow_fast_3_a}.

Next we analyze the second term in \eqref{eq:inverse:div:G0}. The point is that this term has the {\em same structure} as what we started with; for every fixed $\ell \in \{1,2,3\}$, we may replace $G_{(0)}^i$ by $G_{(1)}^{i\ell}$, and we replace $\varrho_{(0)}$ with $\zeta^{-1} \partial_\ell \varrho_{(1)}$; the only difference is that the bounds for this term are better. Indeed, from \eqref{eq:inverse:div:G1} we see that the $2$-tensor $G_{(1)}$ is the sum of entries in $\zeta^{-1} DG_{(0)} \otimes A$, $\zeta^{-1} DG_{(0)} \otimes A \otimes A \otimes D\Phi$, and $\zeta^{-1} G_{(0)} \otimes D A \otimes A \otimes D\Phi$. Recalling that the entries of $A$ are homogeneous quadratic polynomials in the entries of $D\Phi$, from \eqref{eq:derivatives:phase:gradient}, \eqref{eq:derivatives:mixed:G}, $\lambda'\leq \lambda$, and the Leibniz rule we deduce that
\begin{align}
 \norm{D^{N''} D_t^M D^{N'}G_{(1)}^{i\ell}}_{L^1} \les \const_G (\lambda \zeta^{-1}) \lambda^{N'+N''} \MM{M,M_t,\nu, \tilde \nu}\,.
 \label{eq:derivatives:mixed:G:1}
\end{align}
for $N' + N'' \leq N_*-1$, and $M\leq M_*$. 
Compare the above estimate with \eqref{eq:derivatives:mixed:G}, and we notice that since $\lambda \zeta^{-1} \ll 1$, the bounds for $G_{(1)}$ are indeed better than those for $G_{(0)}$; the only caveat is the the bounds hold for one fewer spatial derivatives. In order to iterate  Proposition~\ref{prop:Celtics:suck}, for simplicity we ignore the $\ell$ index, since the argument works in exactly the same way for all values of $\ell$, we write $G_{(1)}^{i\ell}$ simply as $G_{(1)}^{i}$, and $\partial_\ell \varrho_{(1)}$ as $D \varrho_{(1)}$.  We start by noting that $\zeta^{-1} D\varrho_{(1)} = \zeta^{-2} \Delta (\zeta^{-1} D\varrho_{(2)})$. Thus, using identities \eqref{eq:Celtics:suck:total}--\eqref{eq:Celtics:suck:error} we obtain that the second term in \eqref{eq:inverse:div:G0} may be written as 
\begin{align}
G_{(1)}^i (\zeta^{-1} D\varrho_{(1)}) \circ \Phi =  \partial_n R_{(1)}^{in} +  G_{(2)}^{i\ell} (\zeta^{-2}  \partial_\ell D \varrho_{(2)}) \circ \Phi
\label{eq:inverse:div:G0:new}
\end{align}
where the symmetric stress $R_{(1)}$ is given by
\begin{align}
R_{(1)}^{in} = \zeta^{-1}   \underbrace{\left(  G_{(1)}^i A^n_\ell  + G_{(1)}^n A^i_\ell  -A^i_k A^n_k   G_{(1)}^p \partial_p \Phi^\ell \right)}_{=:S_{(1)}^{in\ell} } (\zeta^{-2}  \partial_\ell D \varrho_{(2)}) \circ \Phi  \,,
\label{eq:inverse:div:R1}
\end{align}
the error terms are computed as
\begin{align}
G_{(2)}^{i\ell} =  \zeta^{-1} \left(  \partial_n \left(G_{(1)}^p A^i_k A^n_k - G_{(1)}^n A^i_k A^p_k\right)  \partial_p \Phi^\ell\right) - \partial_n G_{(1)}^i A^n_\ell  
\label{eq:inverse:div:G2}
\,,
\end{align}
We emphasize that by combining \eqref{eq:inverse:div:G1} with \eqref{eq:inverse:div:R1} and \eqref{eq:inverse:div:G2}, we may compute the $3$-tensor $S_{(1)}$ and the $2$-tensor $G_{(2)}$ {\em explicitly in terms of just space derivatives} of $G$ and $D\Phi$. Using a similar argument to the one which was used to prove \eqref{eq:zero:order:stress:2}, but by appealing to \eqref{eq:derivatives:mixed:G:1} instead of \eqref{eq:derivatives:mixed:G}, we deduce that  for $N\leq N_*-1$ and $M\leq M_*$, 
\begin{align}
\norm{D^N D_t^M S_{(1)}}_{L^1} \les \const_{G} (\lambda \zeta^{-1}) \lambda^N  \MM{M,M_t,\nu, \tilde \nu} \,.
\label{eq:zero:order:stress:3}
\end{align} 
Using the bound \eqref{eq:zero:order:stress:3} and the estimate 
\begin{align*}
\norm{D^N (\zeta^{-2}  \partial_\ell D \varrho_{(2)}) }_{L^1} \les \const_{*} \MM{N,2,\zeta,\Lambda}
\,,
\end{align*} 
which is a consequence of \eqref{eq:DN:Mikado:density} --- in the case $N=2$  as before we may weaken the bound by a factor of $\Lambda^\alpha$ ---  we may deduce from Lemma~\ref{l:slow_fast} that  
\begin{align}
\norm{D^N D_t^M R_{(1)}}_{L^1} \les \const_G \const_* \Lambda^\alpha (\lambda \zeta^{-2})  \MM{N,2,\zeta,\Lambda}  \MM{M,M_t,\nu,\tilde \nu}
\label{eq:zero:order:stress:2b}
\end{align}
for $N\leq N_*-1$ and $M\leq M_*$, which is an estimate that is even better than \eqref{eq:zero:order:stress:2a}, since $\lambda \ll \zeta\leq \Lambda$.   This shows that the first term in \eqref{eq:inverse:div:G0:new} satisfies the expected bound. The second term in \eqref{eq:inverse:div:G0:new} may in turn be shown to satisfy
\begin{align}
 \norm{D^{N''} D_t^M D^{N'}G_{(2)}^{i\ell}}_{L^1} \les \const_G (\lambda^2 \zeta^{-2}) \lambda^{N'+N''} \MM{M,M_t,\nu, \tilde \nu}\,.
 \label{eq:derivatives:mixed:G:2}
\end{align}
for $N'+N''\leq N_*-2$ and $M\leq M_*$, and it is clear that this procedure may be iterated $\dpot$ times. 

Without spelling out these details, the iteration procedure described above produces
\begin{align}
G_{(0)} \varrho_{(0)}\circ \Phi =  \sum_{k=0}^{\dpot-1} \div R_{(k)}  +  \underbrace{G_{(\dpot)} \otimes (\zeta^{-\dpot}  D^\dpot \vartheta) \circ \Phi}_{=:E}
\label{eq:inverse:div:final}
\end{align}
where each of the $\dpot$ symmetric stresses satisfies 
\begin{align}
\norm{D^N D_t^M R_{(k)}}_{L^1} \les \const_G \const_* \Lambda^{\alpha} (\lambda^{k} \zeta^{-k+1}) \MM{N,1,\zeta,\Lambda} \MM{M,M_t,\nu,\tilde \nu}
\,,
\label{eq:zero:order:stress:final}
\end{align}
for $N\leq N_*-k$, and $M\leq M_*$.  Each component of the the error tensor $G_{(\dpot)}$ in \eqref{eq:inverse:div:final} is recursively computable solely in terms of $G$ and $D\Phi$ and their spatial derivatives, and satisfies 
\begin{align}
 \norm{D^{N''} D_t^M D^{N'}G_{(\dpot)}}_{L^1} \les \const_G (\lambda^\dpot \zeta^{-\dpot}) \lambda^{N'+N''} \MM{M,M_t,\nu, \tilde \nu}
 \label{eq:derivatives:mixed:ginal}
\end{align}
for $N'+N''\leq N_*-\dpot$ and $M\leq M_*$.
Lastly, since $\norm{D^N (\zeta^{-\dpot}  D^\dpot \vartheta)}_{L^1} \les \const_{*} \Lambda^\alpha \MM{N,\dpot,\zeta,\Lambda}$ and $D^\dpot \vartheta$ is $(\sfrac{\T}{\mu})^3$-periodic, a final application of  Lemma~\ref{l:slow_fast} combined with \eqref{eq:derivatives:mixed:ginal} and the assumption that $N_* -\dpot \geq 2\Ndec +4$, shows that estimate \eqref{eq:inverse:div:error:1} holds. 

Next, we turn to the proof of \eqref{eq:inverse:div:error:stress} and \eqref{eq:inverse:div:error:stress:bound}. Recall that $E$ is defined by the second term in \eqref{eq:inverse:div:final}, and thus $\fint_{\T^3} G \varrho \circ \Phi dx = \fint_{\T^3} E dx$.
Using the standard nonlocal inverse-divergence operator
\begin{align}
 \RSZ v  = \Delta^{-1} \left(\nabla v +  (\nabla v)^T \right)   - \frac{1}{2} \left( \Id + \nabla \nabla \Delta^{-1}\right) \Delta^{-1}  \div v  \,.
\label{eq:RSZ}
\end{align} 
we may define 
\begin{align*}
\RR_{\mathrm{nonlocal}} = \RSZ E  \,.
\end{align*}
By the definition of $\RSZ$ we have that $\RR_{\mathrm{nonlocal}}$ is traceless, symmetric, and satisfies $\div \RR_{\mathrm{nonlocal}} = E - \fint_{\T^3} E dx$, i.e. \eqref{eq:inverse:div:error:stress} holds. In the last equality we have used that by assumption $G \varrho \circ \Phi$ has zero mean, and thus so does $E$. The idea here is very simple: because $\dpot$ is very large, the gain of $(\lambda \zeta^{-1})^\dpot$ present in the $E$ estimate \eqref{eq:inverse:div:error:1} is so strong, that we may simply convert $D$ and $D_t$ bounds on $E$ to (terrible) $\partial_t$ bounds, which commute with $\RSZ$, and we can still get away with it.

Using the formulas \eqref{eq:RNC} and \eqref{eq:DNC} and the assumption \eqref{eq:inverse:div:v:global}, since $D$ and $\partial_t$ commute with $\RSZ$, we deduce that for every $N\leq N_\circ$ and $M\leq M_\circ$ we have 
\begin{align}
\norm{D^N D_t^M \RR_{\mathrm{nonlocal}}}_{L^1} 
&\les \sum_{\substack{M'\leq M \\ N'+M' \leq N+M}} \sum_{K=0}^{M-M'}  \const_v^K \tilde \lambda_q^{N-N'+K} \tilde \tau_{q}^{-(M-M'-K)} \norm{D^{N'} \partial_t^{M'} \RSZ E}_{L^1}
\notag \\
&\les \sum_{\substack{M'\leq M \\ N'+M' \leq N+M}}  \tilde \lambda_q^{N-N'} \tilde \tau_{q}^{-(M-M')} \norm{D^{N'} \partial_t^{M'}  E}_{L^p}
\label{eq:riots:1}
\end{align}
for any $p \in (1,\sfrac{3}{2})$, 
where in the last inequality we have used that by assumption $\const_v \tilde \lambda_q \les \tilde \tau_q^{-1}$, and that  $\RSZ \colon L^p(\T^3) \to L^1(\T^3)$ is a bounded operator.

Our goal is to appeal to estimate~\eqref{eq:cooper:f:*} in Lemma~\ref{lem:cooper:1}, with $A = - v\cdot \nabla$, $B = D_t$ and $f=E$ in order to estimate the $L^p$ norm of $D^{N'} \partial_t^{M'}  E = D^{N'} (A+B)^{M'} E$. 

First, we claim that $v$ satisfies the lossy estimate
\begin{align}
\norm{D^ND_t^M v}_{L^{\infty}}
\les \const_v \tilde \lambda_q^{N} \tilde \tau_q^{-M}
\label{eq:DDv2} 
\end{align}
for $M\leq M_\circ$ and $N+M\leq N_\circ + M_\circ$. This estimate does not follow from \eqref{eq:DDv}, which only provides bounds for $Dv$, instead of $v$. For this purpose, we apply Lemma~\ref{lem:cooper:1} with $f = v$, $B= \partial_t$, $A = v\cdot \nabla$, and $p=\infty$.  Using \eqref{eq:inverse:div:v:global}, and the fact that $B = \partial_t$ and $D$ commute, we obtain that bounds \eqref{eq:cooper:v} and \eqref{eq:cooper:f} hold with $\const_f = \const_v$, $\lambda_v = \tilde \lambda_v = \lambda_f = \tilde \lambda_f = \tilde \lambda_q$, and $\mu_v = \tilde \mu_v = \mu_f = \tilde \mu_f = \tilde \tau_q^{-1}$. Since $A+ B = D_t$, we obtain from the bound \eqref{eq:cooper:f:*} and our assumption $\const_v \tilde \lambda_q \les \tilde \tau_q^{-1}$ that \eqref{eq:DDv2} holds.

Second, we claim that for any $k\geq 1$ we have
\begin{align}
\norm{ \left(\prod_{i=1}^k D^{\alpha_i} D_{t}^{\beta_i}\right) v}_{L^\infty(\supp G)} \les \const_v \tilde \lambda_q^{|\alpha|} (\max\{\tilde \nu, \tilde \tau_q^{-1}\})^{|\beta|}
\label{eq:riots:5}
\end{align}
whenever $|\beta|\leq M_\circ$ and $|\alpha| +|\beta| \leq N_\circ + M_\circ$.
To see this, we use Lemma~\ref{lem:cooper:2} with $f =v$, $p=\infty$, and $\Omega = \supp G$. From \eqref{eq:DDv} we have that \eqref{eq:cooper:2:v} holds with  $\const_v = \nu/ \lambda' $, $\lambda_v = \tilde \lambda_v = \lambda'$, $\mu_v=\nu$, and $\tilde \mu_v = \tilde \nu$. On the other hand, from \eqref{eq:DDv2} we have that \eqref{eq:cooper:2:f} holds with $\const_f = \const_v$, $\lambda_f = \tilde \lambda_f = \tilde \lambda_q$, and $\mu_f =\tilde \mu_f = \tilde \tau_q^{-1}$. Since $\tilde \lambda_q \geq \lambda'$, we deduce from \eqref{eq:cooper:2:f:2} that \eqref{eq:riots:5} holds. 

Third, we claim that 
\begin{align}
\norm{ \left(\prod_{i=1}^k D^{\alpha_i} D_{t}^{\beta_i}\right) E}_{L^p(\supp G)} \les \const_G \const_* (\lambda \zeta^{-1})^\dpot  \Lambda^{|\alpha|+1} \MM{|\beta|,M_t,\nu, \tilde \nu}
\label{eq:riots:6}
\end{align}
holds  whenever $|\alpha|\leq N_* - d$ and $|\beta|\leq M_*$. This estimate again follows from Lemma~\ref{lem:cooper:2}, this time with $f= E$, by appealing to the previously established bound \eqref{eq:inverse:div:error:1} and the Sobolev embedding $W^{1,1}(\T^3) \subset L^p(\T^3)$ for $p\in(1,\sfrac 32)$.

At last, we are in the position to apply Lemma~\ref{lem:cooper:1}. The bound 
\eqref{eq:riots:5} implies that assumption \eqref{eq:cooper:v} holds with $B= D_t$, $\lambda_v = \tilde \lambda_v = \tilde \lambda_q$, and $\mu_v= \mu_v = \max\{\tilde \tau_q^{-1}, \tilde \nu\}$. The bound
\eqref{eq:riots:6} implies that assumption \eqref{eq:cooper:f} of Lemma~\ref{lem:cooper:1} holds with $\const_f = \const_G \const_* (\lambda \zeta^{-1})^\dpot \Lambda $, $\lambda_f = \tilde \lambda_f = \Lambda$, $\mu_f=\nu$, and $\tilde \mu_f = \tilde \nu$. We may now use estimate \eqref{eq:cooper:f:*}, and the assumption that $\Lambda \geq \tilde \lambda_q $ to deduce that 
\begin{align}
 \norm{D^{N'} \partial_t^{M'}  E}_{L^p} \les \const_G \const_* (\lambda \zeta^{-1})^\dpot \Lambda^{N'+1} (\max\{\const_v \Lambda, \tilde \nu, \tilde \tau_q^{-1}\})^{M'}
 \label{eq:riots:2}
\end{align}
holds whenever $M' \leq M_\circ$ and $N'+M' \leq N_\circ + M_\circ$.
Combining \eqref{eq:riots:1} and \eqref{eq:riots:2} we deduce that 
\begin{align}
\norm{D^N D_t^M \RR_{\mathrm{nonlocal}}}_{L^1} 
&\les \const_G \const_* (\lambda \zeta^{-1})^\dpot
\sum_{\substack{M'\leq M \\ N'+M' \leq N+M}}  \tilde \lambda_q^{N-N'} \tilde \tau_{q}^{-(M-M')} \Lambda^{N'+1} (\max\{\const_v \Lambda, \tilde \nu, \tilde \tau_q^{-1}\})^{M'}
\notag\\
&\les \const_G \const_* (\lambda \zeta^{-1})^\dpot \Lambda^{N+1} (\max\{\const_v \Lambda, \tilde \nu, \tilde \tau_q^{-1}\})^{M}
\label{eq:riots:3}
\end{align}
whenever $N\leq N_\circ$ and $M\leq M_\circ$.  Estimate \eqref{eq:inverse:div:error:stress:bound} follows by appealing to the assumption \eqref{eq:riots:4}, which ensures that the gain from $(\lambda \zeta^{-1})^{\dpot-1}$ is already a sufficiently strong amplitude gain, and we use the leftover factor of $\lambda \zeta^{-1}$ to absorb implicit constants.
\end{proof}


\bibliographystyle{abbrv}

\begin{thebibliography}{10}

\bibitem{AnselmetEtAl84}
F.~Anselmet, Y.~Gagne, E.~J. Hopfinger, and R.~A. Antonia.
\newblock High-order velocity structure functions in turbulent shear flows.
\newblock {\em J. Fluid Mech.}, 140:63--89, 1984.

\bibitem{BBV19}
R.~Beekie, T.~Buckmaster, and V.~Vicol.
\newblock Weak solutions of ideal mhd which do not conserve magnetic helicity.
\newblock {\em Annals of PDE}, 6(1):Paper No. 1, 40 pp., 2020.

\bibitem{RajMatt}
R.~Beekie and M.~Novack.
\newblock Non-conservative solutions of the Euler-$\alpha$ equations.
\newblock {\em 	arXiv:2111.01027}, 2021.

\bibitem{BrueColombo21}
E.~Bru{\`e} and M.~Colombo.
\newblock Nonuniqueness of solutions to the Euler equations with vorticity in a Lorentz space.
\newblock {\em 	arXiv:2108.09469}, 2021.

\bibitem{brue2020positive}
E.~Bru{\`e}, M.~Colombo, and C.~De~Lellis.
\newblock Positive solutions of transport equations and classical nonuniqueness
  of characteristic curves.
\newblock {\em Arch. Rational Mech. Anal.}, 240:1055--1090, 2021.

\bibitem{Buckmaster15}
T.~Buckmaster.
\newblock Onsager's conjecture almost everywhere in time.
\newblock {\em Comm. Math. Phys.}, 333(3):1175--1198, 2015.

\bibitem{BCV18}
T.~Buckmaster, M.~Colombo, and V.~Vicol.
\newblock Wild solutions of the {N}avier-{S}tokes equations whose singular sets
  in time have {H}ausdorff dimension strictly less than $1$.
\newblock {\em arXiv:1809.00600}, 2018.

\bibitem{BDLISZ15}
T.~Buckmaster, C.~De~Lellis, P.~Isett, and L.~Sz{{\'e}}kelyhidi, Jr.
\newblock Anomalous dissipation for $1/5$-{H}\"older {E}uler flows.
\newblock {\em Ann. of Math.}, 182(1):127--172, 2015.

\bibitem{BDLSZ16}
T.~Buckmaster, C.~De~Lellis, and L.~Sz{{\'e}}kelyhidi, Jr.
\newblock Dissipative {E}uler flows with {O}nsager-critical spatial regularity.
\newblock {\em Comm. Pure Appl. Math.}, 69(9):1613--1670, 2016.

\bibitem{buckmaster2013transporting}
T.~Buckmaster, C.~De~Lellis, and L.~Sz{\'{e}}kelyhidi~Jr.
\newblock Transporting microstructure and dissipative {E}uler flows.
\newblock {\em arXiv:1302.2815}, 2013.

\bibitem{BDLSV17}
T.~Buckmaster, C.~De~Lellis, L.~Sz{\'{e}}kelyhidi~Jr., and V.~Vicol.
\newblock Onsager{\textquotesingle}s conjecture for admissible weak solutions.
\newblock {\em Comm. Pure Appl. Math.}, 72(2):229--274, July 2018.

\bibitem{BV_EMS19}
T.~Buckmaster and V.~Vicol.
\newblock Convex integration and phenomenologies in turbulence.
\newblock {\em EMS Surv. Math. Sci.}, 6(1):173--263, 2019.

\bibitem{BV19}
T.~Buckmaster and V.~Vicol.
\newblock Nonuniqueness of weak solutions to the {N}avier-{S}tokes equation.
\newblock {\em Ann. of Math.}, 189(1):101--144, 2019.

\bibitem{BV20}
T.~Buckmaster and V.~Vicol.
\newblock Convex integration constructions in hydrodynamics.
\newblock {\em Bull. Amer. Math. Soc.}, 58(1):1--44, 2020.

\bibitem{ChenEtAl05}
S.~Chen, B.~Dhruva, S.~Kurien, K.~Sreenivasan, and M.~Taylor.
\newblock Anomalous scaling of low-order structure functions of turbulent
  velocity.
\newblock {\em J. Fluid Mech.}, 533:183--192, 2005.

\bibitem{CCFS08}
A.~Cheskidov, P.~Constantin, S.~Friedlander, and R.~Shvydkoy.
\newblock Energy conservation and {O}nsager's conjecture for the {E}uler
  equations.
\newblock {\em Nonlinearity}, 21(6):1233--1252, 2008.

\bibitem{cheskidov2020nonuniqueness}
A.~Cheskidov and X.~Luo.
\newblock Nonuniqueness of weak solutions for the transport equation at critical space regularity.
\newblock {\em arXiv:2004.09538}, 2020.

\bibitem{cheskidov2020sharp}
A.~Cheskidov and X.~Luo.
\newblock Sharp nonuniqueness for the {N}avier-{S}tokes equations.
\newblock {\em arXiv:2009.06596}, 2020.

\bibitem{cheskidov2021}
A.~Cheskidov and X.~Luo.
\newblock $L^2$-critical nonuniqueness for the 2D Navier-Stokes equations.
\newblock {\em arXiv:2105.12117}, 2021.

\bibitem{CS2014}
A.~Cheskidov and R.~Shvydkoy.
\newblock {E}uler equations and turbulence: analytical approach to intermittency.
\newblock {\em SIAM J. Math. Anal.}, 46(1):353--374, 2014.

\bibitem{Constantin07}
P.~Constantin.
\newblock On the {E}uler equations of incompressible fluids.
\newblock {\em Bull. Amer. Math. Soc. (N.S.)}, 44(4):603--621, 2007.

\bibitem{ConstantinETiti94}
P.~Constantin, W.~E, and E.~Titi.
\newblock Onsager's conjecture on the energy conservation for solutions of
  {E}uler's equation.
\newblock {\em Comm. Math. Phys.}, 165(1):207--209, 1994.

\bibitem{ConstantinFefferman94}
P.~Constantin and C.~Fefferman.
\newblock Scaling exponents in fluid turbulence: some analytic results.
\newblock {\em Nonlinearity}, 7(1):41, 1994.

\bibitem{ConstantinNieTanveer99}
P.~Constantin, Q.~Nie, and S.~Tanveer.
\newblock Bounds for second order structure functions and energy spectrum in
  turbulence.
\newblock {\em Phys. Fluids}, 11(8):2251--2256, 1999.
\newblock The International Conference on Turbulence (Los Alamos, NM, 1998).

\bibitem{ConstantineSavits96}
G.~Constantine and T.~Savits.
\newblock A multivariate {F}a{\`a} di {B}runo formula with applications.
\newblock {\em Trans. Amer. Math. Soc.}, 348(2):503--520, 1996.

\bibitem{Dai18}
M.~Dai.
\newblock Non-uniqueness of {L}eray-{H}opf weak solutions of the 3d
  {H}all-{MHD} system.
\newblock {\em arXiv:1812.11311}, 2018.

\bibitem{DaneriSzekelyhidi17}
S.~Daneri and L.~Sz{{\'e}}kelyhidi, Jr.
\newblock Non-uniqueness and h-principle for {H}\"older-continuous weak
  solutions of the {E}uler equations.
\newblock {\em Arch. Rational Mech. Anal.}, 224(2):471--514, 2017.

\bibitem{delellis2020nonuniqueness}
C.~De~Lellis and H.~Kwon.
\newblock On non-uniqueness of {H}\"{o}lder continuous globally dissipative
  {E}uler flows.
\newblock {\em arXiv:2006.06482}, 2020.

\bibitem{DeLellisSzekelyhidi09}
C.~De~Lellis and L.~Sz{\'e}kelyhidi, Jr.
\newblock The {E}uler equations as a differential inclusion.
\newblock {\em Ann. of Math. (2)}, 170(3):1417--1436, 2009.

\bibitem{DLSZ12}
C.~De~Lellis and L.~Sz\'{e}kelyhidi, Jr.
\newblock The {$h$}-principle and the equations of fluid dynamics.
\newblock {\em Bull. Amer. Math. Soc. (N.S.)}, 49(3):347--375, 2012.

\bibitem{DeLellisSzekelyhidi13}
C.~De~Lellis and L.~Sz{{\'e}}kelyhidi, Jr.
\newblock Dissipative continuous {E}uler flows.
\newblock {\em Invent. Math.}, 193(2):377--407, 2013.

\bibitem{DLSZ17}
C.~De~Lellis and L.~Sz{\'e}kelyhidi~Jr.
\newblock High dimensionality and h-principle in {PDE}.
\newblock {\em Bull. Amer. Math. Soc.}, 54(2):247--282, 2017.

\bibitem{DLSZ19}
C.~De~Lellis and L.~Sz{\'e}kelyhidi~Jr.
\newblock On turbulence and geometry: from {N}ash to {O}nsager.
\newblock {\em arXiv:1901.02318}, 01 2019.

\bibitem{DrivasEyink19}
T.~Drivas and G.~Eyink.
\newblock An {O}nsager singularity theorem for {L}eray solutions of
  incompressible {N}avier--{S}tokes.
\newblock {\em Nonlinearity}, 32(11):4465, 2019.

\bibitem{DuchonRobert00}
J.~Duchon and R.~Robert.
\newblock Inertial energy dissipation for weak solutions of incompressible
  euler and navier-stokes equations.
\newblock {\em Nonlinearity}, 13(1):249, 2000.

\bibitem{Eyink94}
G.~Eyink.
\newblock Energy dissipation without viscosity in ideal hydrodynamics {{I}}.
  {{Fourier}} analysis and local energy transfer.
\newblock {\em Physica D: Nonlinear Phenomena}, 78(3--4):222--240, 94.

\bibitem{EyinkSreeniviasan06}
G.~Eyink and K.~Sreenivasan.
\newblock Onsager and the theory of hydrodynamic turbulence.
\newblock {\em Rev. Modern Phys.}, 78(1):87--135, 2006.

\bibitem{FoiasFrischTemam75}
C.~Foias, U.~Frisch, and R.~Temam.
\newblock Existence de solutions ${C}^\infty$ des \'equations d'euler.
\newblock {\em C. R. Acad. Sci. Paris, Ser. A}, 280:505--508, 1975.

\bibitem{Frisch95}
U.~Frisch.
\newblock {\em Turbulence}.
\newblock Cambridge University Press, Cambridge, 1995.
\newblock The legacy of A. N. Kolmogorov.

\bibitem{Gromov86}
M.~Gromov.
\newblock {\em Partial differential relations}, volume~9.
\newblock Springer Science \& Business Media, 1986.

\bibitem{Isett12}
P.~Isett.
\newblock H{\"o}lder continuous {E}uler flows in three dimensions with compact
  support in time.
\newblock {\em arXiv:1211.4065}, 11 2012.

\bibitem{Isett17}
P.~Isett.
\newblock On the endpoint regularity in {O}nsager's conjecture.
\newblock {\em arXiv preprint arXiv:1706.01549}, 2017.

\bibitem{Isett2018}
P.~Isett.
\newblock A proof of Onsager{\textquotesingle}s conjecture.
\newblock {\em Annals of Mathematics}, 188(3):871, 2018.

\bibitem{IshiharaEtAl09}
T.~Ishihara, T.~Gotoh, and Y.~Kaneda.
\newblock Study of high--{R}eynolds number isotropic turbulence by direct
  numerical simulation.
\newblock {\em Annual Review of Fluid Mechanics}, 41:165--180, 2009.

\bibitem{KanedaEtAl03}
Y.~Kaneda, T.~Ishihara, M.~Yokokawa, K.~Itakura, and A.~Uno.
\newblock Energy dissipation rate and energy spectrum in high resolution direct
  numerical simulations of turbulence in a periodic box.
\newblock {\em Physics of Fluids}, 15(2):L21--L24, 2003.

\bibitem{Kolmogorov41a}
A.~Kolmogorov.
\newblock Local structure of turbulence in an incompressible fluid at very high
  reynolds number.
\newblock {\em Dokl. Acad. Nauk SSSR}, 30(4):299--303, 1941.
\newblock Translated from the Russian by V. Levin, Turbulence and stochastic
  processes: Kolmogorov's ideas 50 years on.

\bibitem{Komatsu79}
G.~Komatsu.
\newblock Analyticity up to the boundary of solutions of nonlinear parabolic
  equations.
\newblock {\em Comm. Pure Appl. Math.}, 32(5):669--720, 1979.

\bibitem{Lichtenstein25}
L.~Lichtenstein.
\newblock {\"U}ber einige {E}xistenzprobleme der {H}ydrodynamik homogener,
  unzusammendr{\"u}ckbarer, reibungsloser {F}l\"ussigkeiten und die
  {H}elmholtzschen {W}irbels\"atze.
\newblock {\em Mathematische Zeitschrift}, 23(1):89--154, 1925.

\bibitem{Luo18}
X.~Luo.
\newblock Stationary solutions and nonuniqueness of weak solutions for the
  {N}avier-{S}tokes equations in high dimensions.
\newblock {\em Arch. Ration. Mech. Anal.}, 233(2):701--747, 2019.

\bibitem{MeneveauSreenivasan91}
C.~Meneveau and K.~Sreenivasan.
\newblock The multifractal nature of turbulent energy dissipation.
\newblock {\em J. Fluid Mech.}, 224:429--484, 1991.

\bibitem{ModenaSZ18}
S.~Modena and L.~Sz{{\'e}}kelyhidi, Jr.
\newblock Non-renormalized solutions to the continuity equation.
\newblock {\em arXiv:1806.09145}, 2018.

\bibitem{ModenaSZ17}
S.~Modena and L.~Sz\'{e}kelyhidi, Jr.
\newblock Non-uniqueness for the transport equation with {S}obolev vector
  fields.
\newblock {\em Ann. PDE}, 4(2):Paper No. 18, 38, 2018.

\bibitem{MullerSverak03}
S.~M{\"u}ller and V.~{\v S}ver{\'a}k.
\newblock {Convex Integration for Lipschitz Mappings and Counterexamples to
  Regularity}.
\newblock {\em Annals of Mathematics}, 157(3):pp. 715--742, 2003.

\bibitem{Nash54}
J.~Nash.
\newblock ${C}^1$ isometric imbeddings.
\newblock {\em Ann. of Math.}, 60(3):383--396, 1954.

\bibitem{NguyenEtAl19}
F.~Nguyen, J.-P. Laval, P.~Kestener, A.~Cheskidov, R.~Shvydkoy, and
  B.~Dubrulle.
\newblock Local estimates of {H}\"older exponents in turbulent vector fields.
\newblock {\em Physical Review E}, 99(5):053114, 2019.

\bibitem{MattVlad}
M.~Novack and V.~Vicol.
\newblock An intermittent Onsager theorem.
\newblock {\em arXiv:2203.13115}, 2022.

\bibitem{Onsager49}
L.~Onsager.
\newblock Statistical hydrodynamics.
\newblock {\em Nuovo Cimento (9)}, 6(Supplemento, 2(Convegno Internazionale di
  Meccanica Statistica)):279--287, 1949.

\bibitem{Scheffer93}
V.~Scheffer.
\newblock An inviscid flow with compact support in space-time.
\newblock {\em J. Geom. Anal.}, 3(4):343--401, 1993.

\bibitem{Shnirelman00}
A.~Shnirelman.
\newblock Weak solutions with decreasing energy of incompressible {E}uler
  equations.
\newblock {\em Comm. Math. Phys.}, 210(3):541--603, 2000.

\bibitem{Shvydkoy10}
R.~Shvydkoy.
\newblock Lectures on the {O}nsager conjecture.
\newblock {\em Discrete Contin. Dyn. Syst. Ser. S}, 3(3):473--496, 2010.

\bibitem{SreenivasanEtAl96}
K.~Sreenivasan, S.~Vainshtein, R.~Bhiladvala, I.~San~Gil, S.~Chen, and N.~Cao.
\newblock Asymmetry of {{Velocity Increments}} in {{Fully Developed
  Turbulence}} and the {{Scaling}} of {{Low-Order Moments}}.
\newblock {\em Phys. Rev. Lett.}, 77(8):1488--1491, 1996.

\bibitem{SulemFrisch75}
P.-L. Sulem and U.~Frisch.
\newblock Bounds on energy flux for finite energy turbulence.
\newblock {\em J. Fluid Mech.}, 72(3):417--423, 1975.

\bibitem{Szekelyhidi12}
L.~Sz{\'e}kelyhidi~Jr.
\newblock From isometric embeddings to turbulence.
\newblock {\em HCDTE lecture notes. Part II. Nonlinear hyperbolic PDEs,
  dispersive and transport equations}, 7:63, 2012.

\bibitem{tao19}
T.~Tao.
\newblock 255b, notes 2: Onsager's conjecture, 2019.

\bibitem{Wiedemann17}
E.~Wiedemann.
\newblock Weak-strong uniqueness in fluid dynamics.
\newblock {\em arXiv:1705.04220}, 2017.

\end{thebibliography}


\end{document}